\documentclass[reqno]{amsart} %
\usepackage{amsmath,amsfonts,amssymb,comment,amscd}
\usepackage{graphicx}
\usepackage{enumerate}
\usepackage{times}
\usepackage[normalem]{ulem}
\usepackage{xypic}
\usepackage[usenames,dvipsnames]{color}
\usepackage{stmaryrd}

\usepackage{etoolbox}

\usepackage{appendix}

\newcommand{\warning}[1]{{\color{blue}  [#1] }}
\newcommand{\av}[1]{{\color{red} \sf AV: [#1]}}
\newcommand{\pn}[1]{{\color{ForestGreen} \sf PN: [#1]}}

 \usepackage{pdfsync}
\usepackage{hyperref}

\newcommand{\so}{\mathfrak{so}}

\newcommand{\LG}{{}^L G}

\def\h{\operatorname{h}}

\def\sym{\operatorname{sym}}
\def\End{\operatorname{End}}
\def\diag{\operatorname{diag}}

\renewcommand{\geq}{\geqslant}
\renewcommand{\leq}{\leqslant}
\newcommand{\Opp}{\Op}
\newcommand{\temp}{\mathrm{temp}}

\newcommand{\Spec}{\mathrm{Spec}}

\newcommand{\Op}{\mathrm{Op}}

\newcommand{\dist}{\mathrm{dist}}
\newcommand{\kbar}{\overline{k}}
\newcommand{\trace}{\mathrm{tr}}

\newcommand{\ad}{\mathrm{ad}}

\DeclareMathOperator{\Gal}{Gal}

\DeclareMathOperator{\Hom}{Hom}

\DeclareMathOperator{\GL}{GL}
\DeclareMathOperator{\Res}{Res}

\DeclareMathOperator{\res}{res}

\newcommand{\GG}{\mathbf{G}}
\newcommand{\HH}{\mathbf{H}}

\newcommand{\SO}{\mathrm{SO}}

\newcommand{\nt}{\mathrm{nt}}

\newcommand{\ggC}{\mathfrak{g}_{\C}}

\newcommand{\stab}{\mathrm{stab}}

\newcommand{\reg}{{\mathrm{reg}}}

\newcommand{\supp}{{\mathrm{supp}}}

\newcommand{\eps}{\varepsilon}

\numberwithin{equation}{section}
\newcommand{\tr}{\mathrm{tr}}

\newcommand{\unit}{\mathrm{unit}}

\renewcommand{\sc}{\mathrm{sc}}

\newcommand{\pr}{\mathrm{pr}}

\newcommand{\p}{\mathfrak{p}}

\newcommand{\Ad}{\mathrm{Ad}}

\def\boldGL{\mathbf{G}\mathbf{L}}

\def\boldAut{\mathbf{A}\mathbf{u}\mathbf{t}}

\newcommand{\Lie}{\mathrm{Lie}}

\newcommand{\C}{\mathbb{C}}

\newcommand{\vol}{\mathrm{vol}}

\newcommand{\G}{\mathbf{G}}

\newcommand{\ev}{\mathrm{ev}}
\newcommand{\pf}{\mathrm{pf}}

\newcommand{\PGL}{\mathrm{PGL}}
\newcommand{\alg}{\mathrm{alg}}
\newcommand{\image}{\mathrm{image}}

\newcommand{\NTC}{\mathrm{NTC}}

\newcommand{\Ind}{\mathrm{Ind}}

\newcommand{\Sym}{\mathrm{Sym}}

\newcommand{\R}{\mathbb{R}}
\renewcommand{\H}{\mathbf{H}}

\newcommand{\Q}{\mathbb{Q}}

\newcommand{\adele}{\mathbb{A}}

\DeclareFontFamily{OT1}{rsfs}{}
\DeclareFontShape{OT1}{rsfs}{n}{it}{<-> rsfs10}{}
\DeclareMathAlphabet{\mathscr}{OT1}{rsfs}{n}{it}

\newcommand{\beq}{\begin{displaymath}}
  \newcommand{\eeq}{\end{displaymath}}
\newcommand{\bC}{ \mathcal{L}(\Pi,\Sigma) }
\newcommand{\bc}{ \mathcal{L}(\pi,\sigma) }
\newcommand{\rsL}{ \frac{ L^{(R)}(\frac{1}{2}, -,  \Pi \times \Sigma^{\vee})}{L^{(R)}(1, \Ad, \Pi \times {\Sigma}^{\vee})}  }
\def\rsLplain_s{ L(s, -,  \Pi \times {\Sigma}^\vee)  }

\newcommand{\be}{\begin{equation}}
  \newcommand{\ee}{\end{equation}}

\def\O{\operatorname{O}}

\newtheorem{theorem}{Theorem}{\it}{\it}

{\it}{\rm}
{\it}{\it}
\newtheorem*{Lemma*}{Lemma}{\it}{\it}
{\it}{\it}
{\it}{\it}
{\it}{\it}
{\it}{\it}
\newtheorem*{Theorem*}{Theorem}{\it}{\it}
\newtheorem*{Corollary*}{Corollary}{\it}{\it}
\newtheorem*{Question*}{Question}{\it}{\it}
\newtheorem*{Proposition*}{Proposition}{\it}{\it}

\theoremstyle{plain} %

\newtheoremstyle{itplain} %
{6pt}                    %
{5pt\topsep}                    %
{\itshape}                   %
{}                           %
{\itshape}                   %
{.}                          %
{5pt plus 1pt minus 1pt}                       %
{}  %

\theoremstyle{itplain} %
\newtheorem{lemma}{Lemma}

\newtheorem*{lemma*}{Lemma}
\newtheorem*{hypothesis*}  {Hypothesis}
\newtheorem*{claim*}{Claim}
\newtheorem*{subclaim*}{Subclaim}
\newtheorem*{proposition*}{Proposition}
\newtheorem*{conjecture*}{Conjecture}
\newtheorem{proposition}{Proposition}

\counterwithin*{lemma}{subsubsection}

\newtheorem{corollary} {Corollary} 
\newtheorem*{corollary*} {Corollary} 

\theoremstyle{remark}

\newtheorem{remark}             {Remark}
\newtheorem{example}              {Example}
\counterwithin*{example}{subsubsection}
\newtheorem*{remark*}            {Remark}
\counterwithin*{remark}{subsubsection}
\newtheorem{definition}     {Definition}
\counterwithin*{definition}{subsubsection}
\newtheorem*{definition*}{Definition}
\newtheorem*{notation*}{Notation}
\newtheorem*{example*}{Example}

\usepackage{xparse}

\ExplSyntaxOn
\NewDocumentCommand{\giit}{O{}}
 {
  \str_case:nn { #1 }
   {
    {}{/\mkern-6mu/}
    {\big}{\big/\mkern-7mu\big/}
    {\Big}{\Big/\mkern-10mu\Big/}
    {\bigg}{\bigg/\mkern-14mu\bigg/}
    {\Bigg}{\Bigg/\mkern-18mu\Bigg/}
   }
 }
\ExplSyntaxOff

\newcommand{\git}{\giit}
\makeindex

\begin{document}

\title{The orbit method and analysis of automorphic forms}

\author{Paul D. Nelson and Akshay Venkatesh}

\setcounter{tocdepth}{1}

\maketitle
\begin{abstract}
  We develop the orbit method in a   quantitative form, along the
  lines of microlocal analysis,
  and apply it to the analytic
  theory of automorphic forms.

  Our main global application is an asymptotic formula for
  averages of Gan--Gross--Prasad periods in arbitrary rank.  The
  automorphic form on the larger group is held fixed, while that
  on the smaller group varies over a family of size roughly the
  fourth root of the conductors of the corresponding
  $L$-functions.
  Ratner's results on measure classification
  provide an important input to the proof.
  
  Our
  local results include asymptotic expansions for certain
  special functions arising
  from representations of higher rank
  Lie groups, such as the relative characters defined by
  matrix coefficient integrals as in the Ichino--Ikeda
  conjecture.

\end{abstract}

\tableofcontents

\newtoggle{Prt0}
\newtoggle{Prt1}
\newtoggle{Prt2}
\newtoggle{Prt3}
\newtoggle{Prt4}
\newtoggle{Prt5}
\newtoggle{easyproof}
\newtoggle{cleanpart}

\toggletrue{Prt0}
\toggletrue{Prt1}
\toggletrue{Prt2}
\toggletrue{Prt3}
\toggletrue{Prt4}
\toggletrue{Prt5}
\togglefalse{easyproof}
\togglefalse{cleanpart}

  \section{Introduction}\label{sec:introduction}

The problem of giving asymptotic formulas for moments of
large degree $L$-functions has proved challenging.
In approaching this problem, one
encounters difficult analytic questions
in the
representation theory of reductive Lie groups,
involving complicated multi-dimensional
oscillatory 
integrals.

The orbit method
(see, e.g., \cite{MR1701415, MR1737729})
is a philosophy for, among other purposes,
reducing difficult problems
in the representation theory of Lie groups
to simpler problems in symplectic geometry.
It has been
widely applied in the algebraic side of that theory.

This paper
develops the orbit method in a quantitative analytic form.
We combine the tools thus developed
with an indirect application of Ratner's theorem
to study moments of automorphic $L$-functions
on higher rank groups.

\subsection{Overview of results}
We refer the reader who is not familiar with automorphic forms
to \S\ref{sec:outline and ingredients} and onwards
for an 
introduction, in explicit terms,
to 
the main ideas of this paper.

Let $\HH \hookrightarrow \GG$ be an inclusion of reductive
groups over a number field $F$.  Let $\Pi$ and $\Sigma$ be
cuspidal automorphic representations of $\GG$ and $\HH$, respectively.
Assuming that $\HH \hookrightarrow \GG$
is a strong Gelfand pair, and under a temperedness assumption, 
one may define an ``automorphic
branching coefficient'' $\bC \geq 0$ which
quantifies how vectors in $\Pi$ correlate with
$\Sigma$.  We recall this definition in a simple setting in
\S\ref{sec:branch-coef-overview} and more fully in
\S\ref{sec:branch-coefs}.

We focus on the ``Gan--Gross--Prasad'' case
(\S\ref{sec:gross-prasad-pairs-inv-theory}) in which
\[
  \text{$(\GG, \HH)$ is a form of either $(\SO_{n}, \SO_{n-1})$ or
    $(\mathrm{U}_n, \mathrm{U}_{n-1})$.}
\]
The definition of $\bC$
then applies, at least for tempered $\Pi$ and $\Sigma$, and one
expects $\bC$ to be related to special values of
$L$-functions: Ichino--Ikeda \cite{MR2585578} and N. Harris
\cite{MR3159075} conjecture
the formula
\begin{equation} \label{II} \bC =   2^{-\beta}  \rsL \Delta_G^{(R)},
\end{equation}
whose terms are as follows (see \emph{loc. cit.} for details):
\begin{itemize}
\item $R$ is a fixed set of places outside of which $\Pi$ and
  $\Sigma$
  are spherical.
  (Thus $\mathcal{L}(\Pi,\Sigma)$ depends upon
  $R$.)
\item $\Sigma^{\vee}$ denotes the contragredient of the unitary
  representation $\Sigma$; it is isomorphic to the conjugate representation
  $\overline{\Sigma}$, and we will occasionally use the latter notation. 
\item $L^{(R)}$ denotes an $L$-function
  without Euler factors
  in $R$.
\item $2^{\beta}$\index{$2^\beta$} is the size of the Arthur component group of
  $\Pi \boxtimes \Sigma^{\vee}$ on $\GG \times \HH$.
\item   $\Delta_G$\index{$\Delta_G$} is the $L$-function
  whose local factor at almost every prime $p$
  equals $\frac{p^{\dim(G)}}{\# \mathbf{G}(\mathbf{F}_p)}$
   (see \cite{MR1474159});  e.g., $\Delta_G = \zeta(2) \zeta(4) \dots \zeta(2n)$ for $G = \mathrm{SO}_{2n+1}$;
  $\Delta_G^{(R)}$ omits factors at $R$. 
\end{itemize}
The formula \eqref{II} has been proved in the unitary case,
under
local assumptions
which allow one to use a simple form of
the
trace formula, by W. Zhang \cite{MR3245011}
(see also \cite[\S2.2]{2017arXiv171208844Z} and
\cite{2016arXiv160206538B}).
 
%   since it's clear
%   from the main theorem
% In any case, we deal directly with $\bC$.

Fix one such $\Pi$.  What are the asymptotic statistics of
$\mathcal{L}(\Pi,\Sigma)$, as $\Sigma$ varies over a large
family?  For example, what are the moments?
Predictions for these may be obtained via \eqref{II} and random matrix
heuristics (\S\ref{sec:comp-with-rand})
for families of $L$-functions.
To verify such predictions rigorously
has proved an interesting challenge,
testing our understanding of families of automorphic
forms and $L$-functions.
It
has been successfully
undertaken in many low-degree cases,
where obtaining strong error estimates
remains an active area of research
(see, e.g., \cite{MR3650231,
  MR3334233, MR3702671, 2018arXiv180401450B}
and references).

We aim here to explore
some first
cases of arbitrarily large degree.
Our main result
(theorem \ref{thm:main-subconvex})
may be summarized informally as follows:
\begin{Theorem*}
  Assume certain local conditions,
  including the compactness of the quotients $\mathbf{G}(F) \backslash
  \mathbf{G}(\mathbb{A})$
  and
  $\mathbf{H}(F) \backslash
  \mathbf{H}(\mathbb{A})$.
  For each sufficiently small positive
  real $\h$,
  let $\mathcal{F}_{\h}$
  be the family of
  all $\Sigma$ as above
  which
  are
  locally distinguished by $\Pi$,
  have Satake parameters at some fixed archimedean place
  inside the rescaling $\h^{-1} \Omega$ of
  some nice fixed compact
  set $\Omega$, and
  have ``fixed level'' at the remaining places in $R$.
  Then the branching
  coefficient $\mathcal{L}(\Pi,\Sigma)$,
  averaged over $\Sigma \in \mathcal{F}_{\h}$,
  is asymptotic to
  $1/2$:
  \begin{equation} \label{approx theorem}
    \lim_{\h \rightarrow 0}
    \frac{1}{|\mathcal{F}_{\h}|} \sum_{\Sigma \in
      \mathcal{F}_{\h}} \bC 
    = \frac{1}{2}.
  \end{equation}
\end{Theorem*}

 For ``typical'' $\Pi$
and $\Sigma$,
we expect that \eqref{II}
holds
with
$\beta = 2$
% We expect that
% generically $\beta = 2$
(see  \S\ref{sec-18-5} for further explanation).
Our result should thus translate,
under  \eqref{II},
as follows:
\begin{equation}\label{approx theorem 2}
  \text{The average value of $\rsL \Delta_G^{(R)}$ is 
    $2$}.
\end{equation}
We outline in \S\ref{sec:comp-with-rand}
why \eqref{approx theorem 2} agrees with random matrix theory
heuristics for orthogonal families of $L$-functions with
positive root number.

One way to normalize the strength of \eqref{approx theorem}
is to note
(\S\ref{analytic-conductors})
that
the size of the family
$\mathcal{F}_h$ is roughly the fourth power of the analytic
conductor of the relevant $L$-function.
By ignoring all but one term and slowly shrinking the family,
we obtain a ``weakly subconvex'' bound
\begin{equation} \label{weaksubconvexity}
  \bC =
  o\left(\mathrm{cond}(\Pi \times \Sigma)^{1/4} \right)
\end{equation}
(compare with \cite{soundararajan-2008, 2018arXiv180403654S}).
The
hypotheses
relevant for
\eqref{weaksubconvexity}
are that $\Pi$ is fixed,
while $\Sigma$ traverses a
sequence whose archimedean Satake parameters all tend off to
$\infty$ at the same rate.

The new ideas used to obtain \eqref{approx theorem} are based on the orbit
method, applied in two ways:
\begin{itemize}
\item Firstly, to determine the
  asymptotics of complicated
  oscillatory integrals on higher rank
  groups.  For instance,
  theorem
  \ref{thm:sph-char-main-export} gives general and uniform asymptotic
  expansions of relative characters away from the
  conductor-dropping locus.
  This analysis gives a robust supply of analytic test vectors for
  the local matrix coefficient integrals as in Ichino--Ikeda.
  We hope these to be of general use in analytic problems
  involving families of automorphic forms in higher rank.
\item Secondly, to obtain invariant measures towards which we
  can apply measure-theoretic techniques.  Indeed, a major
  global ingredient for \eqref{approx theorem} is an application
  of Ratner's theorem to the case of measures invariant by the
  centralizer of a regular nilpotent element in $\mathbf{G}(F)$.
  The estimate \eqref{approx theorem} is ineffective, and the
  application of Ratner is solely responsible for the
  ineffectivity.  We expect that an effective version of Ratner's
  theorem for the case at hand would lead to a subconvex
  estimate for $\mathcal{L}(\Pi,\Sigma)$; this perhaps contributes interest
  to the problem of effectivization.
  % and contributes
  % interest to
  % the problem of effectivizing
  % such cases of Ratner's results.
%
%
%
\end{itemize}
 To implement these, we develop a microlocal calculus for
Lie group representations,
which may be understood as a quantitative,
analytic form of the
orbit method
and the
philosophy of geometric quantization.
% To implement these, we need a microlocal calculus for
% representations with reasonable analytic estimates.  This
% calculus may be understood as a quantitative incarnation of the
% orbit method, or equivalently as an analytic incarnation of the
% philosophy of geometric quantization.

These basic ideas do not
 interact with
%use essentially
the arithmetic nature of the setting; in particular, we do not
use Hecke operators.  However,
the problem of averaging and
bounding $L$-functions seems to be the most interesting source
of applications at the moment.

\subsection{Compatibility with random matrix heuristics}\label{sec:comp-with-rand}
We briefly outline why our result \eqref{approx theorem 2}
should be compatible with the standard heuristics.

Random matrix theory heuristics
 (see, e.g., \cite{MR2396123} and
references, such as \cite{KS99, MR1784410, MR2149530,
  KeSn00a, KeSn00b}) suggest, for a family of $L$-functions
$L(\frac{1}{2}, f)$ parameterized by elements $f$ of some nice
enough family $\mathcal{F}$, that
\[
\frac{1}{|\mathcal{F}|}
\sum_{\mathcal{F}} L^{(R)}(\tfrac{1}{2},f) \approx
\left(\mbox{global factor $g$}\right) \times \prod_{p}  \langle
a_p \rangle \]
where
\begin{itemize}
\item $p$ runs over the finite primes of $F$ outside $R$,
\item $\langle a_p \rangle$ is the expectation
  \[ \langle a_p \rangle :=  \frac{1}{|\mathcal{F}|} \sum_{\mathcal{F}} L_p(\tfrac{1}{2}, f)\]
of  the central value of the $p$th Euler
factor,
and
\item the ``global factor''
  $g$ is described by random matrix theory,
  and given here by the following limit of integrals
  taken with respect to probability Haar measures:
  \[g = \lim_{N \rightarrow \infty} \int_{x \in \mathrm{SO}(2N)}
  \det(1-x).
  \]
  Indeed,
the  symmetry type \cite{KS99} of the family of $L$-functions implicit in \eqref{approx theorem}
is $O(2N)$,
i.e., it is an orthogonal family with positive root numbers.
One may see this by, for instance, considering the analogous situation
with the number field replaced by a function field
(see also \cite{MR3437869, MR3675175}).
\end{itemize}

Each integral $\int_{\SO(2 N)} \det(1-x)$ 
equals $2$, independent of $N$.
Indeed, $\det(1-x)$ is the (super-)trace
of $x \in \SO(2N)$ on $\wedge^{*} (\C^{2N})$.
The average trace of $x$ on $\wedge^j (\C^{2 N})$
computes the dimension of invariants, which
is $1$ for $j \in \{0,2N\}$ and $0$ otherwise.
Thus  $g=2$.

We sketch here 
why
$\langle a_p \rangle = 1$ for every $p$.
We expect
the families $\mathcal{F}_{\h}$
considered here
to
have the (provable) property that
the local component $\sigma_p$ at $p$
of a uniformly random element $\Sigma \in \mathcal{F}_{\h}$
becomes distributed,
as $\h \rightarrow 0$,
according to
the Plancherel measure
$d  \sigma_p$,
thus
\[ \langle a_p \rangle = \int_{\sigma_p}
\underbrace{ \frac{ L_p(\frac{1}{2}, -, \pi_p
    \times \sigma_p^\vee )}{L(1, \Ad, \pi_p \times  \sigma_p^\vee) }
  \Delta_{G,p}
}_{=:
  I(\pi_p, \sigma_p)}
\, d \sigma_p.\]
Ichino and Ikeda \cite[Theorem
1.2]{MR2585578} have shown in the orthogonal case, and N. Harris
\cite[Theorem 2.12]{MR3159075} in the unitary case, that this
integral can be rewritten as a matrix coefficient integral:
$I(\pi_p, \sigma_p) = \int_{h \in H_p} \langle h v, v
\rangle \langle u, h u \rangle dh$, where $v \in \pi_p$
and $u \in \sigma_p$
are 
spherical unit vectors and
$d h$ is
the Haar measure assigning unit volume to a hyperspecial
maximal compact subgroup.
Using
the
Plancherel formula,
one can show (\S\ref{sec:local-disintegration})
that $\int_{\sigma_p} I(\pi_p, \sigma_p) d\sigma_p =1$, as
required.

We have not discussed yet
the quantity  $\beta$ as in \eqref{II}. As mentioned above,
we expect for ``typical'' $\Pi$ and $\Sigma$ 
that $\beta=2$. 
However,
for ``atypical''
$\Pi$ (i.e., endoscopic lifts), one can have $\beta > 2$
in the entire family.  In that case, the limit of \eqref{approx  theorem 2} 
is instead a larger power of $2$.
Correspondingly,  the $L$-function in question factors.
It seems to us that our result continues
to match with $L$-function heuristics after appropriately accounting for this factorization and variation of root numbers, but we have not checked all details.

The random matrix theory predictions for higher moments involve somewhat
more complicated coefficients than the quantity $g = 2$
appearing above.  It may be possible to
study higher moments by adapting our method to
periods of Eisenstein series, and
would be interesting to obtain in that way
some geometric
perspective on those predictions.

\subsection{Basic setup}\label{sec:outline and ingredients}

We now simultaneously outline the contents of this paper
and sketch the main arguments.

Suppose given a pair of unimodular Lie groups $G$ and $H$, with
$H \leq G$.  An example relevant for our main theorem is when
\begin{center}
  $G$ and $H$ are the real points of the split forms of
  $\SO_{n+1}$ and $\SO_n$.
\end{center}
We assume that representations of $G$
have ``multiplicity-free restrictions'' to $H$, as happens in
the indicated example (see \cite{MR2874638} for a precise
statement and proof).

Suppose also given a lattice $\Gamma$ in $G$
for which $\Gamma_H := \Gamma \cap H$ is a
lattice in $H$.
We assume that both quotients
\[[G] := \Gamma \backslash G,
  \quad
  [H] := \Gamma_H \backslash H
\]
are compact.
We equip them with Haar measures.

For the motivating applications to $L$-functions, we must
consider adelic quotients.  This entails some additional
work at ``auxiliary places'' (see \S\ref{localBernstein})
which we do not discuss
further in this overview.

\subsection{Branching coefficients:
  comparing global and local restrictions}\label{sec:branch-coef-overview}
Let
\[
\pi \subseteq L^2([G])
\]
be an irreducible unitary
subrepresentation,
with the group $G$ acting by right translation.
We assume that $\pi$ is tempered.
The branching coefficients of interest
arise from comparing the \emph{two} natural ways to restrict $\pi$ to
$H$:
\begin{itemize}
\item (Globally) Take a smooth vector $v \in \pi$.  It defines a
  function on $[G]$.  We may restrict it to obtain a function
  $v|_{[H]}$ on $[H]$.
  The $L^2$-norm of that restriction may be decomposed as
  \begin{equation}\label{eqn:overview-global-decomp}
    \int_{[H]} |v|^2
    = \sum _{
      \text{irreducible } \sigma  \subseteq L^2([H])
    }
    \|
    \mbox{projection of  $v|_{[H]}$  to $\sigma$}
    \|^2.
  \end{equation}
\item (Locally) Consider $\pi$ as an abstract unitary
  representation of $G$.  We may restrict it to obtain an
  abstract unitary representation $\pi|_H$ of $H$.  
  We verify in the examples of interest
  (see \S \ref{sec:spher-char-disint})
  that this restriction decomposes as a direct integral
  \[\pi|_H
  = \oint_{\sigma} m_{\sigma} \sigma,
  \]
  weighted by multiplities $m_{\sigma} \in \{0,1\}$,
  and taken over  
  tempered irreducible unitary representation $\sigma$
  of $H$ with respect to Plancherel measure.  
  We may
  define the components $\pr_\sigma(v) \in \sigma$ of a
  smooth vector $v \in \pi$ with respect
  to this decomposition, 
  and we have
  \begin{equation}\label{eqn:v-norm-via-sigma-integral-whee}
      \|v\|^2
      = \int_\sigma \|\pr_\sigma(v)\|^2.
    \end{equation}
{\em A priori}, $\|\pr_{\sigma}(v)\|$ is defined only as a measurable
function of $\sigma$, but there is a natural way to define it pointwise in the cases of 
interest. 
\end{itemize}

Let $\sigma \subseteq L^2([H])$ be an irreducible tempered
subrepresentation
for
which $m_\sigma = 1$;
we refer to such $\sigma$ as \emph{$\pi$-distinguished}.
By the multiplicity one property,
we may define a proportionality constant
$\bc \geq 0$ by requiring that
for every smooth vector $v \in \pi$,
\begin{equation}\label{eqn:overview-branching-coef}
  \|
  \mbox{projection of  $v|_{[H]}$  to $\sigma$}
  \|^2
  =
  \bc
  \cdot \|\pr_{\sigma}(v)\|^2.
\end{equation}
We note that $\mathcal{L}(\pi,\sigma)$
depends upon the choices of Haar measure.

\subsection{Objective}
Let $\h$ traverse a sequence of positive reals tending to zero,
and let $\mathcal{F}_{\h}$
be a corresponding sequence of families
consisting of irreducible
tempered representations $\sigma \subseteq L^2([H])$ for which
$m_{\sigma} = 1$.

We assume that each family $\mathcal{F}_{\h}$ arises
from some nice subset $\tilde{\mathcal{F}_{\h}}$
of the $\pi$-distinguished tempered dual of $H$
as the set of all
irreducible $\sigma \subseteq L^2([H])$
for which the class of $\sigma$ belongs to $\tilde{\mathcal{F}_{\h}}$.
We assume that
$|\mathcal{F}_{\h}| \rightarrow \infty$ as $\h \rightarrow 0$. 
(In our main theorem, we
take for $\tilde{\mathcal{F}_{\h}}$ the set of $\sigma$ whose infinitesimal
character belongs to the rescaling $\h^{-1} \Omega$ of some fixed
nonempty bounded open set $\Omega$.)
Our aim is to determine the asymptotics
of the sums
\[
  \sum_{\sigma \in \mathcal{F}_{\h}}
  \bc.
\]
In our motivating examples, these are (in some cases conjecturally)
proportional to sums of special values of $L$-functions.

We drop the subscript $\h$ in what follows for notational simplicity.

\subsection{Strategy}\label{sec:rough-idea-proof}
The rough idea of our proof is
to find a vector
$v \in \pi$
which simultaneously
``picks out the family $\mathcal{F}$''
in that
\begin{equation}\label{eqn:overview-vector-desiderata}
  \|\pr_{\sigma}(v)\|^2
  \approx \begin{cases}  1 & \text{ if } \sigma
    \in \tilde{\mathcal{F}}, \\ 0 & \text{ if } \sigma \notin \tilde{\mathcal{F}}
  \end{cases}
\end{equation}
and ``becomes equidistributed''
in that
\begin{equation} \label{equid}
  \frac{\int_{[H]} |v|^2}{\vol([H])}    \sim
  \frac{\int_{[G]} |v|^2}{\vol([G])}.
\end{equation}
 
 The vector $v$ will of course depend
upon the family $\mathcal{F} = \mathcal{F}_{\h}$,
hence upon the asymptotic parameter $\h > 0$,
and the above estimates
are to be understood as holding
asymptotically
in the $\h \rightarrow 0$ limit.
Note that \eqref{eqn:overview-vector-desiderata} is a purely local problem of harmonic analysis
in the representation $\pi$, whereas \eqref{equid} is a global problem: it relates
to the specific way in which $\pi$ is embedded in $L^2([G])$. 

 The Weyl law on $[H]$ says that
the cardinality of
$\mathcal{F}$ is approximately the
volume of $[H]$ times the Plancherel measure of
$\tilde{\mathcal{F}}$,
thus by
\eqref{eqn:overview-vector-desiderata} and
\eqref{eqn:v-norm-via-sigma-integral-whee},
\begin{equation}\label{eqn:overview-weyl-law}
  |\mathcal{F}| \approx \vol([H]) \|v\|^2.
\end{equation}
Comparing \eqref{eqn:overview-global-decomp},
\eqref{eqn:overview-branching-coef}
and
\eqref{eqn:overview-vector-desiderata}
with
\eqref{eqn:overview-weyl-law}
yields the required asymptotic formula
\begin{equation}\label{eqn:intro-strat-reqd-asymp-form}
  \frac{1}{|\mathcal{F}|} \sum_{\sigma \in \mathcal{F}}
  \bc \sim \frac{1}{\vol([G])}.
\end{equation}
In applications, the volumes are
defined using Tamagawa
measure
and given then by
$\vol([H]) = \vol([G])= 2$.
Thus \eqref{eqn:intro-strat-reqd-asymp-form} leads to \eqref{approx
  theorem}.

Observe,
finally,
that it suffices to produce
\emph{families} of vectors $v_i \ (i \in I)$ for which the
analogues of \eqref{eqn:overview-vector-desiderata} and
\eqref{equid} hold on average over the index set $I$.  

We note that this basic strategy has appeared (sometimes
implicitly) in several antecedent works (see, e.g., \cite{MR780071, MR2726097, venkatesh-2005, michel-2009}).
We would like to note,  in particular, the influence of the ideas of Bernstein and Reznikov in exploiting the uniqueness of an invariant functional. 
The novelty here is that we execute
this strategy successfully in
arbitrarily large rank.

\subsection{Microlocal calculus for Lie group
  representations}\label{sec:intro-op-calc}
To implement this
strategy, we need some way to produce and analyze (families of)
vectors in the representation $\pi$.
Our approach, inspired
by microlocal analysis,
is to work with vectors implicitly through their
symmetry properties under group elements $g \in G$ within a
suitable shrinking neighborhood of the identity element.
Since we do not expect our readers to have extensive prior
familiarity with microlocal analysis, we describe here its
content in our setting.
The discussion in this section
is rather informal, but we hope that the reader will find
it helpful
in navigating the many technical estimates of the text. 

The shrinking neighborhood in question depends upon an
infinitesimal scaling parameter $\h \rightarrow 0$.
Let
$\mathfrak{g}$ denote the Lie algebra of $G$,
and
$\mathfrak{g}^\wedge$
the Pontryagin dual of $\mathfrak{g}$;
we identify $\mathfrak{g}^\wedge$ with the imaginary dual
$i \mathfrak{g}^*$,
and denote by $e^{x \xi} \in \mathbb{C}^{(1)}$
the image of $(x,\xi) \in \mathfrak{g} \times
\mathfrak{g}^\wedge$
under the natural pairing.
Informally, we say that a vector $v \in \pi$ is \index{microlocalized vector}
\emph{microlocalized at $\xi \in \mathfrak{g}^\wedge$} if
\begin{equation}\label{eq:microlocalized-vector-approx-defn}
  \pi(\exp(x)) v \approx e^{x \xi/\h}
  v
\end{equation}
for all $x \in \mathfrak{g}$ of size $|x| = \O(\h)$. 
In practice, we choose $\h$ small enough in terms of $\pi$ that such vectors
exist.  For instance, if $\pi$ is generated by a Maass form
$\varphi$ of Laplace eigenvalue $\lambda$ on a locally symmetric
space $\Gamma \backslash G / K$, then we choose $\h$ comparable to or
smaller than the ``wavelength scale'' $|\lambda|^{-1/2}$ of
$\varphi$;
similar considerations apply more generally,
with the role of wavelength scale
played by the inverse root norm $|\lambda_\pi|^{-1/2}$ of the
infinitesimal
character of $\pi$ (see \S\ref{ss:norms}).
We quantify the heuristic
\eqref{eq:microlocalized-vector-approx-defn} in several
different ways throughout this paper, typically by working with
sequences of representations $\pi = \pi_{\h}$ and vectors
$v = v_{\h}$ that vary with $\h$ and asking that
the difference between the left and right hand sides of
\eqref{eq:microlocalized-vector-approx-defn}
decay at some specified rate as $\h \rightarrow 0$.
% If we choose $\h$ comparable to $|\lambda|^{-1/2}$,
% then the solutions \eqref{eq:microlocalized-vector-approx-defn}
% will include microlocal lifts of $\varphi$ as in \ref{MR916129,
%   MR1838659, MR2346281}. 

The group elements $g = \exp(x)$
with $|x| = \O(\h)$
approximately commute as $\h \rightarrow 0$, and so the
operators $\pi(g)$ may be approximately simultaneously
diagonalized; their common approximate eigenvectors are the
microlocalized vectors.
We might thus hope
for an approximate decomposition
\begin{equation}\label{eq:approximate-basis-microlocal-intro}
  \pi \approx \bigoplus_{\xi}
  \mathbb{C} v_\xi,
\end{equation}
where $\xi$ traverses a subset of $\mathfrak{g}^\wedge$ and
$v_\xi \in \pi$ is microlocalized at $\xi$,
and to have an approximate
functional
calculus
\begin{equation}\label{eq:opp-intro}
  \Opp_{\h} : \{\text{symbols $\mathfrak{g}^\wedge \rightarrow \mathbb{C}$}\}
  \rightarrow \{\text{operators on $\pi$}\},
\end{equation}
\begin{equation}\label{eq:opp-intro-acts-v}
  \Opp_{\h}(a) v_\xi \approx a(\xi) v_\xi,
\end{equation}
where ``symbol'' refers
to a class of functions with suitable
regularity.

To implement such ideas rigorously,
we
write down
an operator assignment
\eqref{eq:opp-intro}
and verify that it has
the properties
suggested
by
the heuristics
\eqref{eq:approximate-basis-microlocal-intro}
and
\eqref{eq:opp-intro-acts-v}.
The assignment is similar to
the classical Weyl calculus in the theory of pseudodifferential
operators, which may be recovered (more-or-less) by specializing
our discussion to standard representations of Heisenberg
groups.
The definition
and basic properties of this assignment
are philosophically unsurprising.  If
$a : \mathfrak{g}^\wedge \rightarrow \mathbb{C}$ is the Fourier
mode corresponding to a small enough element $x$ in the Lie
algebra of $G$, then $\Opp_{\h}(a) = \pi(\exp(\h x))$.
If $a$ is real-valued, then $\Opp_{\h}(a)$ is self-adjoint.  If
$a$ is positive, then $\Opp_{\h}(a)$ is asymptotically positive
as $\h \rightarrow 0$.  The association $a \mapsto \Opp_{\h}(a)$
is nearly $G$-equivariant.
One has composition formulas
relating $\Opp_{\h}(a) \Opp_{\h}(b)$ to $\Opp_{\h}(a \star_{\h} b)$ for
a suitable star product $a \star_{\h} b$,
bounds for operator and trace norms, and so on; see
theorems \ref{thm:star-prod-basic},
\ref{thm:main-properties-opp-symbols},
\ref{thm:rescaled-operator-memb},
\ref{thm:trace-estimates-i},
\ref{thm:star-prod-asymp-general},
\ref{thm:comp-gen-subsp}
and
\ref{trace bounds for Op}.

Our main input concerning $\pi$ is the Kirillov character
formula, due in this case to
Rossmann, which asserts roughly
(see
\S\ref{sec:kirillov-frmula} for a precise statement)
that
\begin{equation}\label{eqn:kirillov-formula-intro}
  \trace(\Opp_{\h}(a))
\approx
\int_{\h \mathcal{O}_\pi} a \, d \omega,
\end{equation}
where $\mathcal{O}_{\pi} \subseteq \mathfrak{g}^{\wedge}$
is a $G$-orbit (or finite union thereof),
called the {\em coadjoint orbit attached to $\pi$}. 
In the metaphor of microlocal analysis, the coadjoint orbit
$\mathcal{O}_\pi$
is the ``phase space'' underlying $\pi$;
it bears the same relationship
to $\pi$ as the cotangent bundle $T^* M$ of a manifold $M$
does
to $L^2(M)$. 
The coadjoint orbit is equipped with a canonical symplectic structure, and
$d \omega$ is the associated symplectic volume form.

For each real-valued symbol
$a$, we may informally identify the self-adjoint operator
$\Opp_{\h}(a)$ with a family of vectors $(v_i)_{i \in I}$ by
writing it in the form
$\sum_{i \in I} v_i \otimes \overline{v_i}$ for some
$v_i \in \pi$;
here we have identified operators
on $\pi$ with elements
of (a suitable completion of)
$\pi \otimes \overline{\pi}$.
If the symbol $a$ is taken to be suitably
concentrated near a regular point
$\xi \in \mathfrak{g}^\wedge$, then the corresponding family is
essentially a singleton (i.e., of cardinality $\O(\h^{-\eps})$),
consisting of vectors microlocalized
at $\xi$.
By decomposing the constant symbol $a = 1$ associated to the
identity operator $\Opp_{\h}(a) = 1$ and appealing to
\eqref{eqn:kirillov-formula-intro}, we may write any reasonable
vector as a linear combination of microlocalized vectors
$v_{\xi}$, taken over representatives $\xi/\h$ for a partition
of the coadjoint orbit $\mathcal{O}_\pi$ into pieces of unit
symplectic volume.  We may thus regard the microlocalized
vectors as giving an approximate basis
\eqref{eq:approximate-basis-microlocal-intro},
with respect to which the $\Opp_{\h}(a)$
act as the approximate
multipliers
\eqref{eq:opp-intro-acts-v};
in other words,
\[
\Opp_{\h}(a) \approx \sum_{\xi \in \h \mathcal{O}_\pi} a(\xi) v_\xi \otimes
\overline{v_\xi}.
\]
In this way the $\Opp_{\h}$-calculus parametrizes weighted
families of microlocalized vectors.
These
considerations apply uniformly across the various classes of
tempered representations of $G$ (principal series, discrete
series, ...)  and without reference to any explicit model.

Informally,
if
the dominant contribution
to the decomposition of a vector
$v \in \pi$
as a sum $\sum v_\xi$ of microlocalized
vectors comes from those $\xi$ belonging to some nice
subset of $\h \mathcal{O}_\pi$,
then we refer to that subset as \index{microlocal support}
the \emph{microlocal support}
of $v$;
equivalently,
it describes
where the distribution
on $\mathfrak{g}^\wedge$
given by $a \mapsto \langle \Opp_{\h}(a) v, v \rangle$
is concentrated.
A vector is then
microlocalized if its microlocal support
is concentrated near a specific point.

These notions
from microlocal analysis
play a central role
throughout the paper,
so
we illustrate their content in a couple basic examples.
Figure \ref{fig:1} depicts
coadjoint orbits for
certain representations $\pi$ of the groups $G = \SO(3)$
and $G = \SO(1,2) \cong \PGL_2(\mathbb{R})$, respectively;
in the latter case, $\pi$ belongs to the holomorphic discrete
series.\footnote{The pictures
  were created using the online graphing calculator
  GeoGebra \cite{geogebra}.
}
\begin{figure}[h]
  \includegraphics[width=4cm,height=4cm]{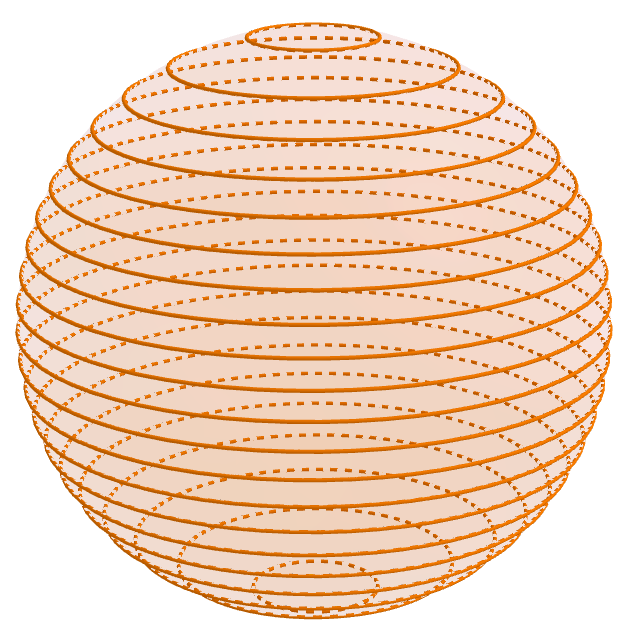}
  \includegraphics[width=4cm,height=4cm]{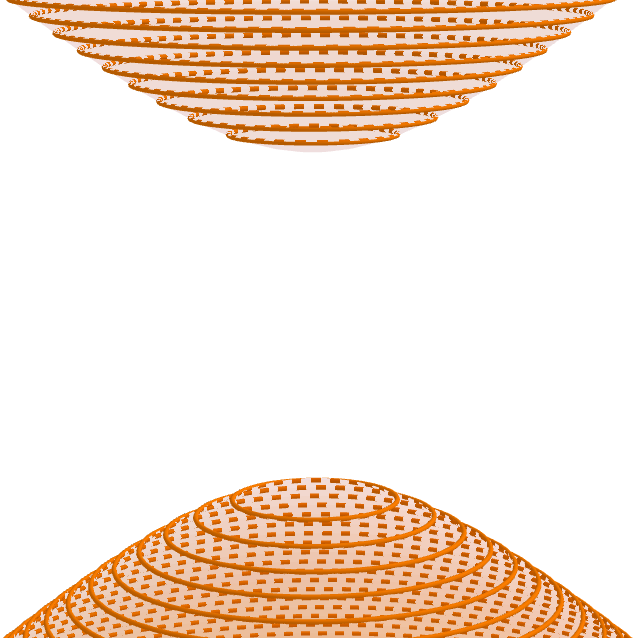}
  \caption{Coadjoint orbits for $(\SO_3,\SO_2)$\label{fig:1}}
\end{figure}
Each of these groups contains the compact subgroup $H = \SO(2)$.
The circles
drawn
on the orbits are level sets for the
``$z$-coordinate''
projection $\mathfrak{g}^* \rightarrow \mathfrak{h}^*$ dual to
the inclusion $\mathfrak{h} \hookrightarrow \mathfrak{g}$ of Lie
algebras.  They divide the orbits into strips of equal
symplectic volume, say of volume $1$, which
correspond under the
orbit method philosophy to the basis of $\pi$
given by $H$-isotypic weight vectors
$e_n$;
the strip should be regarded as the
microlocal support
of the corresponding weight vector.
We may normalize the weights $n$
to be even integers
lying in
\begin{equation}\label{eq:range-of-weights-so3-so12}
    \begin{cases}
    [-k,k] & \text{ if } G = \SO(3), \\
    (-\infty,-k-2] \cup [k+2,\infty)
    & \text{ if } G = \SO(1,2) \cong \PGL_2(\mathbb{R})
    \\
  \end{cases}
\end{equation}
for some nonnegative even integer $k$.

Let us
now take $k$ tending off to $\infty$, but simultaneously
zoom out our camera by
the factor $k$, so that the above picture of the coadjoint
orbit $\mathcal{O}_\pi$ remains constant.  Which weight vectors
should we then consider to be ``microlocalized''?
That is to say, for which vectors
does the ``zoomed out''
microlocal support concentrate
within some small distance of
a specific point in the picture as $k
\rightarrow \infty$?
The strength of this notion
depends
upon the
definition of ``small distance,''
which can sensibly range from the weakest scale $o(1)$
to the Planck scale $\O(k^{-1/2})$.

Vectors of highest or lowest weight ($e_{\pm k}$ or $e_{\pm (k+2)}$)
{\em are} microlocalized,
even at the Planck scale,
since the corresponding strips are
concentrated near
the ``caps'' of the coadjoint orbit (i.e.,
the regions of extremal $z$-coordinate).
By contrast,
``typical'' weight vectors
--
such as the weight zero vector $e_0$
in the representation of $\SO(3)$,
corresponding to the equatorial strip
--
are \emph{not} microlocalized, even at the weakest scale.
In particular, weight vectors
do \emph{not}
give an approximate basis
of microlocalized vectors as in
\eqref{eq:approximate-basis-microlocal-intro}.
The partition of the coadjoint orbit
corresponding to a microlocalized
basis would
instead have every partition element
concentrated near a specific point.

Microlocalized vectors occasionally go in the literature by
other names, such as ``coherent states.''  They are extremely
useful for the sort of asymptotic analysis pursued in this
paper.  Among other desirable properties, their matrix
coefficients behave simply near the identity, and are as
concentrated as possible;
we discuss this phenomenon further
in \S\ref{sec:intro-local-issues}.

 In the body of this paper, we do not often refer
explicitly to microlocalized vectors.  We instead work with them
implicitly through their approximate projectors $\Opp_{\h}(a)$.
We hope that by phrasing the introduction in terms of
microlocalized vectors, it may serve as a useful guide
to the ideas behind the arguments executed in the body.

\subsection{Measure classification;
  equidistribution}\label{sec:intro-meas-class-equid}
In the discussion starting in \S\ref{sec:intro-op-calc}, we allowed both $\pi$ and
$\h$ to vary simultaneously.
Let us suppose now that the representation $\pi$ is held fixed,
independent
of the scaling parameter $\h$.
As $\h \rightarrow 0$,
the rescaled coadjoint orbit $\h \mathcal{O}_\pi$
then
tends to a subset of the nilcone $\mathcal{N} \subseteq
\mathfrak{g}^\wedge$.
Figure \ref{fig:2} depicts this for the coadjoint
orbits
corresponding to
fixed principal series
representations of $G = \SO(1,2) \cong \PGL_2(\mathbb{R})$.
\begin{figure}[h]
  \includegraphics[width=3cm,height=3cm]{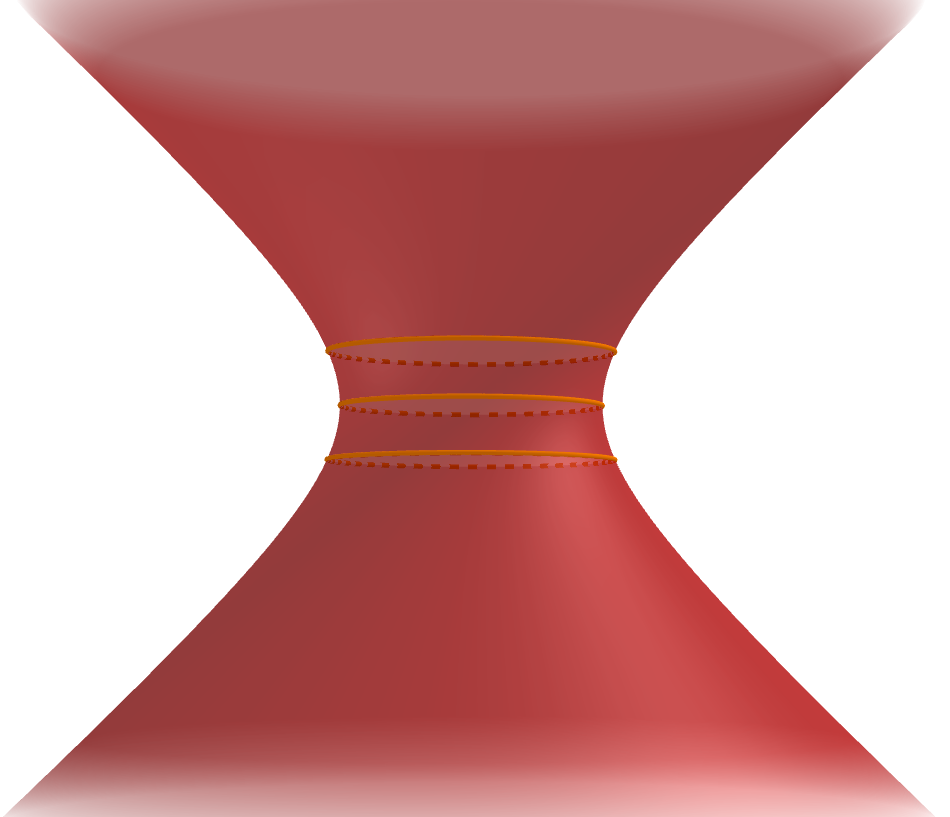}
  \quad
  \includegraphics[width=3cm,height=3cm]{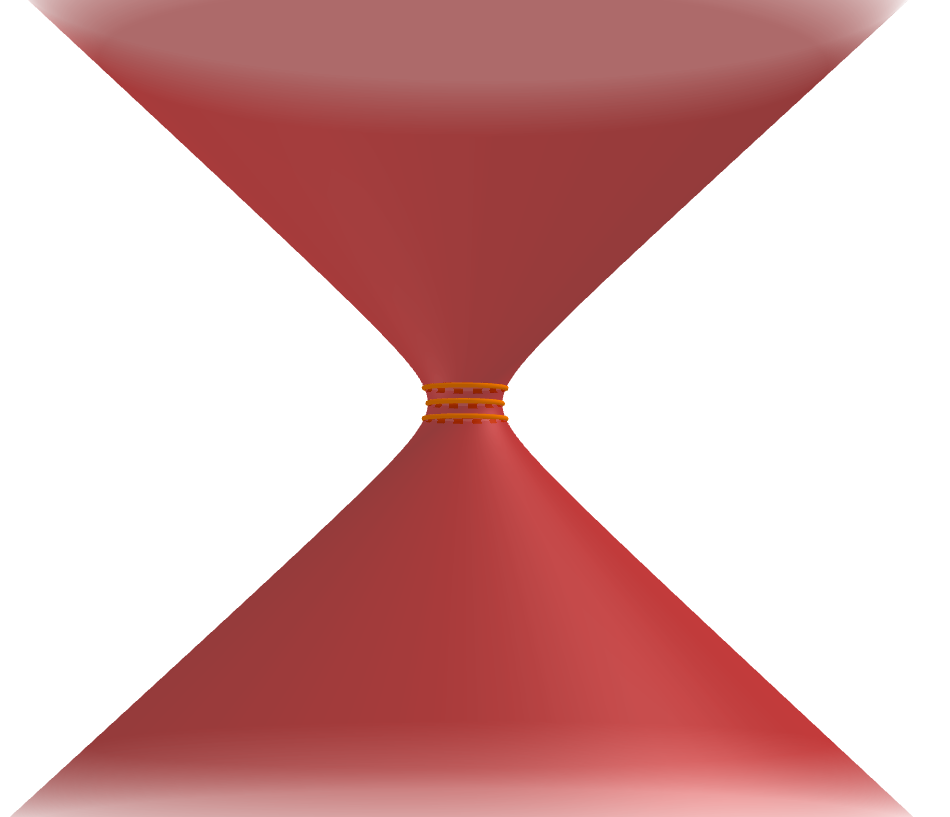}
  \quad
  \includegraphics[width=3cm,height=3cm]{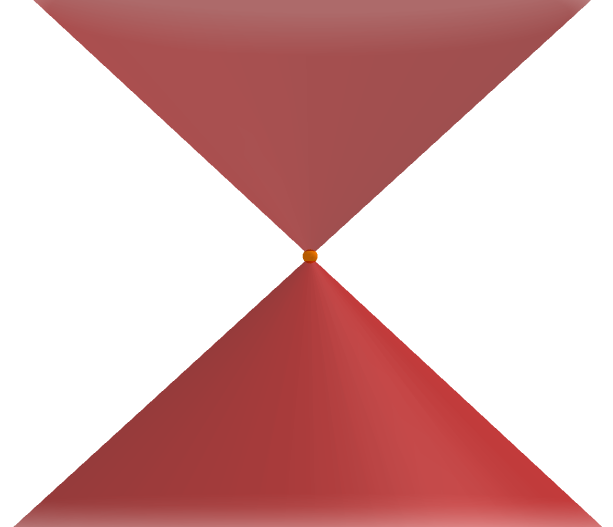}
  \caption{Rescaling to the nilcone\label{fig:2}}
\end{figure}
The operators $\Opp_{\h}(a)$
are
negligible unless $a$ is supported near $\mathcal{N}$.
We thereby obtain
in the $\h \rightarrow 0$ limit
a sequence
\begin{equation}\label{eqn:overview-epic-sequence}
  \{\text{functions $\mathcal{N} \rightarrow \mathbb{R}$}\}
  \xrightarrow{\Opp_{\h}}
  \{\text{self-adjoint $T \in \End(\pi)$}\}
  \rightarrow 
  \{\text{measures on $[G]$}\}
\end{equation}
where the final map sends
$T = \sum v_i \otimes \overline{v_i} \in \End(\pi) \cong \pi
\otimes \overline{\pi}$ to
$\sum |v_i|^2(x) \, d x$,
with $d x$ the chosen Haar measure on $[G]$.
% The $\h \rightarrow 0$ limit of the composition \eqref{eqn:overview-epic-sequence}
% is $G$-equivariant.
% For details, see \S\ref{sec:limit-stat-attach-symb}.
In \S\ref{sec:limit-stat-attach-symb}, we implement this
informal discussion rigorously and construct a $G$-equivariant
limit map from functions on $\mathcal{N}$ to measures on $[G]$.
We emphasize that this limit map is insensitive to the shape of the
unscaled coadjoint orbit $\mathcal{O}_\pi$,
whose role becomes replaced in the limit
by a subset of the nilcone $\mathcal{N}$.

A key
observation is now that the limits of the measures on $[G]$
obtained via \eqref{eqn:overview-epic-sequence} may be
understood using Ratner's theorem.
 The application of Ratner is
indirect, because these measures themselves do not acquire
any obvious additional invariance; rather, they may be
decomposed into measures having unipotent invariance. 

 Indeed, after suitably rescaling, we may describe the limiting
behavior of the sequence \eqref{eqn:overview-epic-sequence} in
terms of a $G$-invariant measure $\mu$ on
the product space
$\mathcal{N} \times [G]$.  Let
$\mathcal{N}_{\reg} \subseteq \mathcal{N}$ denote the regular
subset;
it is the union of the open $G$-orbits on
the nilcone $\mathcal{N}$,
and its complement $\mathcal{N} - \mathcal{N}_{\reg}$ has lower
dimension.  We assume that $\pi$
is generic,
or equivalently, has maximal Gelfand--Kirillov
dimension
(cf. \S\ref{sec:limit-coadj-orbits-h-indep});
for instance, the principal series
satisfy this assumption.
Then the support
of $\mu$ intersects $\mathcal{N}_{\reg} \times [G]$.
By
disintegrating the restriction of $\mu$ to
$\mathcal{N}_{\reg} \times [G]$ over the projection to
$\mathcal{N}_{\reg}$, we obtain a nontrivial family of measures
$\mu_\xi$ on $[G]$ indexed by
$\xi \in \mathcal{N}_{\reg}$.
Speaking informally, we may
regard $\mu_\xi$ as the $\h \rightarrow 0$ limit
of an average of $L^2$-masses
$|v_\xi|^2(x) \, d x$
taken over all vectors $v_\xi$ microlocalized at $\xi$.
In any case, each such measure $\mu_\xi$ is invariant by the
centralizer of the regular nilpotent element $\xi$.  In
favorable situations, an application of Ratner's theorem then
forces the $\mu_{\xi}$ and hence $\mu$ itself to be uniform.

This last paragraph mimics, in the context of Lie group
representations, some of the semiclassical ideas behind the
construction of the microlocal lift.  We discuss this connection
at more length in a sequel to this paper.

The argument just described gives
a rich supply of families
$(v_i)_{i \in I}$
of vectors $v_i \in \pi$
for which
\begin{equation}
  \int_{[G]} |v_i|^2 \Psi
  \sim
  \frac{
    \int_{[G]} |v_i|^2}{\vol([G])} 
  \int_{[G]} \Psi
  \quad \text{ on average over $i \in I$}
\end{equation}
for fixed continuous
functions $\Psi$ on $[G]$.  Although the
characteristic function of
$[H] \subseteq [G]$ is not itself continuous, we may
deduce an averaged form of \eqref{equid} by applying a similar
argument to the derivatives of the functions obtained via
\eqref{eqn:overview-epic-sequence} and appealing to the Sobolev
lemma.  This approach
may be understood as
an infinitesimal variant of
the ``period
thickening'' technique of Bernstein--Reznikov \cite{MR1930758}.

\subsection{Branching and stability} \label{sec:intro-branch-stab}
Having indicated how we achieve the objective \eqref{equid}, we
turn now to the problem of producing (families of)
vectors $v$ which
pick out the family $\mathcal{F}$ as in
\eqref{eqn:overview-vector-desiderata}. 

This is   a {\em quantitative}
version of the branching problem in representation theory: we   wish to understand
not only how a representation of $G$ restricts to $H$, but in fact the behavior of individual
vectors under the restriction process. 
 Our approach is inspired by the following basic principle of the
orbit method: restricting $\pi$ to $H$ should correspond to
\begin{equation} \label{gh} \mbox{disintegrating the coadjoint orbit $\mathcal{O}_\pi$ along the
projection
$\mathfrak{g}^\wedge \rightarrow \mathfrak{h}^\wedge$}.\end{equation}
For example, the distinction of $\sigma$ by $\pi$ should be
equivalent, at least asymptotically, to the existence of a
solution to the equation
\begin{equation}\label{eqn:omega-restricts-to-eta}
  \xi |_{\mathfrak{h}} = \eta
\end{equation}
with $(\xi,\eta) \in \mathcal{O}_\pi \times \mathcal{O}_\sigma
\subseteq \mathfrak{g}^\wedge \times \mathfrak{h}^\wedge$.

 The geometry of the projection \eqref{gh} plays an important role
in our argument, so
we devote a fair amount of space
to studying it in a purely algebraic context
(see \S\ref{sec:gross-prasad-pairs-inv-theory},
\S\ref{sec:stability}, \S\ref{sec:volume-forms}).
Of particular significance is the branch locus of \eqref{gh}, i.e., the 
locus where
the map $\mathcal{O}_\pi \rightarrow \mathfrak{h}^\wedge$
induced by
\eqref{gh} fails to have surjective differential. 
Many features of our analysis break down near this branch locus.

We recall, from geometric invariant theory,
that
$\xi \in \mathfrak{g}^\wedge$ is called \emph{stable}
if the following conditions are satisfied
(see \S\ref{sec:stability} for details):
\begin{itemize}
\item $\xi$ does not lie in 
  the branch locus; equivalently,
  $\xi$ has trivial $\mathfrak{h}$-stabilizer.
\item $\xi$
  has closed $\mathbf{H}$-orbit,
  where
  $\mathbf{H}$ is the algebraic group underlying $H$.
\end{itemize}
This notion plays a fundamental role in our paper,
and appears to be analytically significant:
the complement of the stable locus is where the analytic conductor 
$C(\pi \times \overline{\sigma})$
of the Rankin--Selberg $L$-function $L(\pi \times \overline{\sigma}, s)$ drops
(see \S\ref{sec:stab-regul-terms}). 

For instance, in the basic examples
depicted in
Figure \ref{fig:1},
with
$G \in \{\SO(3),\SO(1,2)\}$ and $H = \SO(2)$,
the coadjoint orbits for $H$ are just
singletons $\{\eta\}$, corresponding to one-dimensional
representations, and the equation
\eqref{eqn:omega-restricts-to-eta} says that $\xi$ should have
$z$-coordinate given by $\eta$; in other words, the solutions to \eqref{eqn:omega-restricts-to-eta}
are the horizontal slices shown in the diagram.

The stable case, in Figure \ref{fig:1},  is when
$\xi$ is not at the north or south poles of the sphere.
Note that the the group $H = \SO(2)$ of rotations fixing the
$z$-axis acts  transitively on any circle in
$\mathcal{O}_\pi$ with $z$-coordinate $\eta$,
 and freely
away from the poles.  This is a general pattern:
the set of solutions to \eqref{eqn:omega-restricts-to-eta} -- if nonempty --
admits a diagonal action by $H$, which is simply transitive
in the stable case (cf. \S\ref{sec:fibers-x-mapsto},
\S\ref{sec:gross-prasad-pairs-over-R-continuity-etc}).

Figure \ref{fig:3}
depicts the
coadjoint orbit of a principal series representation $\pi$ of
$G = \SO(1,2) \cong \PGL_2(\mathbb{R})$, with $H$ the diagonal
subgroup $\SO(1,1) \cong \GL_1(\mathbb{R})$.
\begin{figure}[h]
  \includegraphics[width=6.6cm,height=4.4cm]{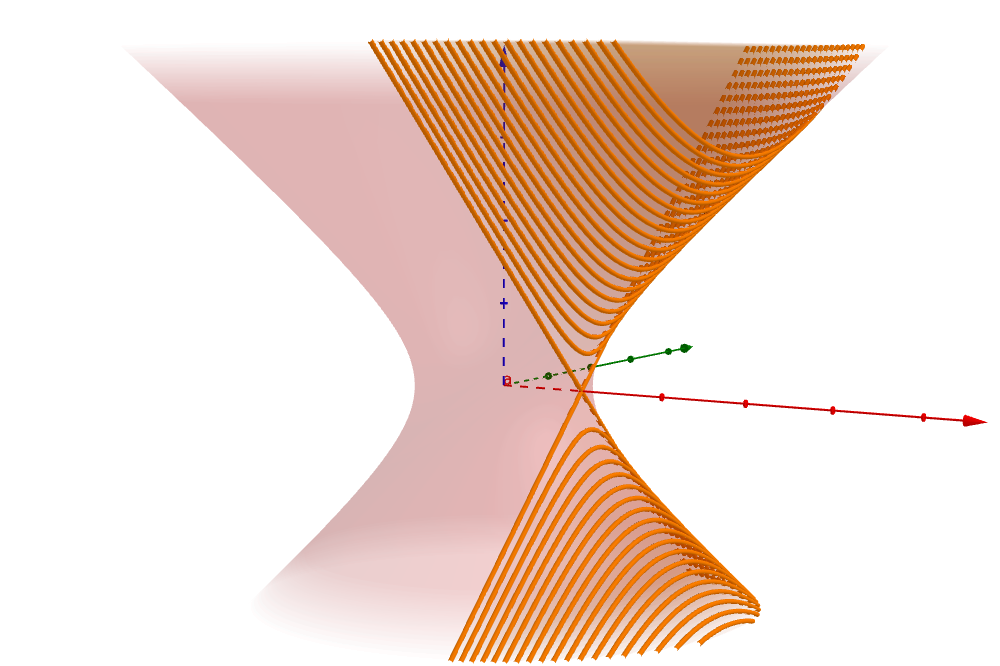}
  \caption{$H$-orbits on the one-sheeted hyperboloid\label{fig:3}}
\end{figure}
One sees again
that $H$ acts simply transitively on generic fibers
(i.e., away from the ``cross'').

\subsection{Analysis of
  matrix coefficient integrals; inverse branching}\label{sec:intro-local-issues}
  
Having set up the necessary preliminaries regarding the geometry of orbits, 
we
 return to the problem of producing (families of) 
vectors $v$ which
pick out the family $\mathcal{F}$ as in
\eqref{eqn:overview-vector-desiderata}.
The solution involves several steps
which may be of independent interest:
\begin{enumerate}[(i)]
\item We prove in \S\ref{sec:sph-char-2} some
  asymptotic formulas for
  $\|\pr_{\sigma}(v)\|^2$, on average over $v$, when the pair
  $(\pi,\sigma)$ is tempered and the infinitesimal characters of
  $\pi$ and $\sigma$ satisfy a stability condition (see
  \S\ref{sec:stab-terms-spectra} for the exact definition).
In more technical terms,  we compute asymptotically
the Fourier transform of relative characters in a small
neighbourhood of the identity.
The asymptotic formulae
readily give a solution (\S\ref{sec:inverse-branching-dist-place})
to the inverse problem of producing (families of) $v$ approximating
a given function $\sigma \mapsto \|\pr_{\sigma}(v)\|^2$.
The proofs depend heavily upon the operator calculus
discussed in \S\ref{sec:intro-op-calc}.
  \item
    We must estimate the
    ``undesirable'' contributions to
    \eqref{eqn:overview-global-decomp}
    coming from $\sigma$ which are either non-tempered,
    non-stable or of ``high frequency.''
    The proofs
    (see, e.g.,  \S\ref{sec:main-result-inv-branch-arch}
    and
    \S\ref{sec:trunc-H-expn-overview})
    apply the operator calculus.
    We refer to \S\ref{sec:recap-overview-proof}
    for some discussion
    of the relevant subtleties.
\end{enumerate}

As illustration, let us indicate how the methods
underlying steps (i) and (ii) address the \emph{analytic test vector
  problem} for a tempered pair $(\pi,\sigma)$ in the stable
case, i.e., away from conductor-dropping.
Informally speaking,
that problem 
is to find
\[\mbox{``nice''  vectors
$v \in \pi$ and $u \in \sigma$ for which
the local period
$|\langle \pr_\sigma(v), u \rangle|^2$
is ``large''.}
\]
Solutions to this problem have appeared in period-based
approaches to the subconvexity period (see, e.g., 
\cite[\S3.6.1]{michel-2009},
\cite[\S2.17.1]{nelson-subconvex-reduction-eisenstein},
\cite{MR2726097}, \cite{MR2373356}, \cite{MR2729264},
\cite{MR2972602}, \cite[Rmk 50]{nelson-padic-que}, \cite[Thm
1.2]{HNSminimal}); the point is that the period formula
\eqref{eqn:overview-branching-coef} and a ``trivial'' estimate
for the global period $\int_{[H]} v \overline{u}$ often suffice
to recover the convexity bound for $L$-function, suggesting the
possibility for improvement in arithmetic settings via Hecke
amplification.

The problem reduces to a robust understanding of the
asymptotics of the local periods, which may be expressed
in terms of matrix coefficient integrals as
follows (see
\S\ref{sec:spher-char-disint}):
\begin{equation}\label{eqn:intro-projection-vs-matrix-coeff-int}
  |\langle \pr_\sigma(v), u \rangle|^2
=
\int_{s \in H}
\langle \pi(s) v, v  \rangle
\langle u, \sigma(s) u \rangle.
\end{equation}
Such integrals had previously been estimated in some low-rank examples
(see, e.g.,  \cite{watson-2008, MR2982336, MR2726097,
  PDN-AP-AS-que, 2017arXiv170605167M, 2017arXiv170900935M}) via
explicit calculation.

To analyze  \eqref{eqn:intro-projection-vs-matrix-coeff-int}, first decompose as in
\S\ref{sec:intro-op-calc}
to reduce to the case
that $v$ and $u$ are microlocalized
at some elements $\xi \in \h \mathcal{O}_\pi$ and
$\eta \in \h \mathcal{O}_\sigma$ of the rescaled coadjoint
orbits.  If $\xi, \eta$ do not
satisfy
\eqref{eqn:omega-restricts-to-eta},
  at least up to some  small error,
then
the resulting integral
 is readily seen to be negligible.

We thereby reduce to giving an 
asymptotic expansion
for \eqref{eqn:intro-projection-vs-matrix-coeff-int}
when $v$ and $u$ are microlocalized
at a stable solution $(\xi,\eta)$
to \eqref{eqn:omega-restricts-to-eta}.
One of the main local results
of this paper
(\S\ref{sec:sph-char-2})
establishes such expansions
after averaging over
small
families of microlocalized vectors,
of cardinality
$\O(\h^{-\eps})$ for any fixed $\eps > 0$ .
Such averages are harmless
for the applications pursued here;
we note that
the basic technique
applies
also in $p$-adic settings,
where 
such averages
may often be omitted  (see
\cite[\S3]{HNSminimal} for illustration
in a basic example).
Our expansions give in particular
a solution to the test vector problem
% indicated above
(modulo taking small averages),
which
appears to be new even in some low-rank
cases (e.g., for discrete series representations in the triple
product setting
$H = \GL_2(\mathbb{R}) \leq G = \GL_2(\mathbb{R})^2$).
 
We summarize the proof of the local result just described.
The stability hypothesis
ensures that
if a group
element $s \in H$ is not too close to the identity,
then it moves
any solution to \eqref{eqn:omega-restricts-to-eta}
a fair bit (cf. \S\ref{sph-chr-sec-5} for details),
so
the microlocal support of
the pair of vectors
$\pi(s) v$ and $\sigma(s) v$
is \emph{disjoint}
% \footnote{To achieve this disjointnesss, it is necessary to ensure
%   that $v$ and $u$ are microlocalized at a small enough scale,
%   which is
%    achieved readily
%   %
%   via the symbol calculus.}
from that of $v$ and $u$,
and so the 
matrix coefficient integrand in \eqref{eqn:intro-projection-vs-matrix-coeff-int}
is negligible. 
We may thus truncate the integral
\eqref{eqn:intro-projection-vs-matrix-coeff-int}
to the range where $s$ is small.
We evaluate
the contribution from this range
asymptotically (\S\ref{sph-chr-sec-3}),
after a small average over $v$ and $u$,
by applying
the Kirillov formula
for $\pi$
and $\sigma$.
The application involves
some pleasant book-keeping
(\S\ref{sec:volume-forms}, \S\ref{sec:gross-prasad-pairs-over-R-continuity-etc})
concerning disintegration
of volume forms (e.g., of $d \omega$ on
$\mathcal{O}_\pi$ under
$\mathcal{O}_\pi \rightarrow \mathfrak{h}^\wedge \git H$).

The methods described here are more than adequate to produce families
of $v$ which pick out the fairly coarse families $\mathcal{F}$ required by our
motivating application; see
\S\ref{sec:inverse-branching-dist-place}.  We hope that they will be useful
in many other problems involving asymptotic analysis of special
functions arising from representations of higher rank groups.
There are several avenues for extending
our methods further.  For instance, it
seems to us an intriguing
problem to obtain analogous asymptotic expansions in non-stable
cases.
 % In another direction,
% while we do not foresee major difficulties in adapting
% our techniques to the ``depth aspect'' over a fixed
% non-archimedean local field (rather than over $\mathbb{R}$),
% it may be interesting to
% investigate the ``horizontal aspect'' involving representations
% of bounded depth over local fields of increasing residue
% characteristic.

\subsection{Comparison with other  work}
We briefly discuss the relationship of the ideas of this paper
to some others:
\begin{enumerate}[(i)]
\item The general philosophy that equidistribution results lead
  to mean value theorems for $L$-functions was advocated in
  \cite{venkatesh-2005}, but the source of the equidistribution
  in this paper is quite different from that in
  \cite{venkatesh-2005}.  The latter dealt exclusively with
  translates of a given vector by a large group element,
  which do not pick out sufficiently flexible families to be
  useful for
  the higher-rank applications pursued here.
\item[(ii)] Bernstein
  and Reznikov initiated the systematic application of
  identities such as \eqref{eqn:overview-branching-coef},
  with a carefully chosen vector, to the problem of
  estimating branching coefficients (on average or
  individually).
  Their ideas have significantly influenced our
  basic strategy, as described in \S\ref{sec:rough-idea-proof}.

\item[(iii)]
  The orbit method \cite{MR1701415} is very well-known in the
  representation theory of Lie groups and widely used in the
  algebraic side of that theory. It is particularly satisfactory in the nilpotent
  case \cite{Kirillovnilpotent, ShalikaTschinkel}. 
 However it does not seem to
  have been used much in the way of this paper -- that is to
  say, developed quantitatively along the lines of microlocal
  analysis and used as an analytical tool.   One
  example where it has been applied in such a context is the
  asymptotic analysis of Wigner $6j$ symbols; see \cite{JR}.
  We hope it will be useful in many other
  contexts like this.
  
 % \item[(iv)]
 %   Simon Marshall
 %   and
 %   Ruixiang Zhang
 %   have recently announced a
 %   subconvexity result, also for the Gross-Prasad family,
 %   in the $p$-adic depth aspect, 
 %   with both forms (on the small and
 %   large groups) varying at the same rate.
 %   Their argument uses  ``microlocalized'' vectors (see the discussion of
 %   \S\ref{sec:intro-local-issues}), and they derive upper bounds
 %   via the Hecke-amplified pretrace formula
 %   rather than asymptotic formulas via measure-classification.
\end{enumerate}

Although we have focused on applications to the averaged
Gan--Gross--Prasad period, we hope that the ideas and methods of
this paper, especially those related to the orbit method, should
be helpful in a broader class of problems involving analysis of
automorphic forms in higher rank.

\subsection{Reading suggestions}
The reader
might wish to peruse Part \ref{part:micr-analys-lie}
and then to skip directly to Part \ref{part:appl-aver-gan},
referring back to earlier results as needed.

\subsection{Acknowledgments}
We would like to thank Joseph Bernstein, Philippe Michel, and Andre Reznikov 
for many inspiring discussions, Emmanuel Kowalski for helpful
feedback on an earlier draft, Davide Matasci for some
corrections
and the anonymous referee for a thorough reading and many
helpful suggestions
that have improved this paper.

During the completion of this paper,
A.V. was an Infosys member at the Institute for Advanced Study. He would like
to thank the IAS for  providing wonderful working conditions. Over the course of writing
the paper he received support from the Packard foundation and
the National Science Foundation.

Parts of this work were carried out during visits of P.N. to
Stanford in 2013--2016, at a 2014 summer school
at the Institut des Hautes {\'E}tudes Scientifiques,
while in residence in Spring 2017 at the
Mathematical Sciences Research Institute,
and during a couple short-term visits in Spring 2018
to the Institute for Advanced Study.  He thanks all of these
institutions for their support and excellent working conditions.

\subsection{General notation and conventions}\label{sec:general-conventions}
We collect
these here together for the reader's convenience.
We have attempted to include reminders where
appropriate throughout the text.

\subsubsection{Asymptotic notation}\label{sec:asymptotic-notation}
We write $X = \O(Y)$ or $X \ll Y$ or $Y \gg X$ to denote that
$|X| \leq c |Y|$ for some quantity $c$ which is ``constant''
or ``fixed'', with precise meanings
clarified
in each situation.
We write $X \asymp Y$
to denote that $X \ll Y \ll X$.
``Implied constants'' are allowed
to depend freely upon any local fields, groups, norms, or bases
under consideration.
 
\subsubsection{The asymptotic parameter $\h$}\label{sec:asympt-param-h}
Throughout this paper, the symbol $\h$ will denote a parameter in
$(0,1]$.
Whenever we use this notation, there will always be an
implicit subset
$\mathcal{H} \subseteq (0,1]$
in which the parameter $\h$ takes values; this subset will usually (but not always)
be a discrete set with $0$ as its only limit point, reflecting, e.g.,
the wavelength scale of a family of Laplace eigenfunctions. 

We always require
the
implied constants in any asymptotic notation % (\S\ref{sec:asymptotic-notation})
to be independent of $\h$.
We use notation
such as $\O(\h^\infty)$
to denote a quantity
of the form
$\O(\h^N)$
for any fixed $N$.

By an ``$\h$-dependent element'' $s$ of some
\index{$\h$-dependent elements}
set $S$, we
mean an element that depends, perhaps implicitly,
upon the parameter $\h \in \mathcal{H}$,
thus $s = s(\h)$.
We might understand
$s$ more formally as an assignment $\mathcal{H} \rightarrow S$,
$\h \mapsto s(\h)$.
For instance, an $\h$-dependent positive
real $c = c(\h) \in \mathbb{R}^\times_+$ is a
map $\mathcal{H} \rightarrow \mathbb{R}^\times_+$.
The parameter $\h$ itself
may be understood as an $\h$-dependent positive real,
corresponding to the identity map.
The terminology
applies even if the set itself varies with $\h$, thus
$S = S(\h)$;
an $\h$-dependent element of $S$
is then an element of the Cartesian product
$\prod_{\h} S(\h)$.

\subsubsection{Translating between $\h$-dependent and absolute
  formulations }
The results stated in this paper in terms of
$\h$-dependent elements can often be reformulated quantitatively
for individual $\h$.  Namely, a theorem formulated in terms of
an arbitrary $\h$-dependent sequence $\pi_{\h}$ is formally
equivalent to an ``absolute'' statement valid for all $\pi$ and
all $\h \in (0,1)$.  More frequently, we encounter theorems
involving an arbitrary $\h$-dependent sequence $\pi_{\h}$ and an
$\h$-dependent sequence $f_{\h}$ of functions which remain
bounded with respect to certain norms, where the norms themselves
are typically $\h$-dependent.
Such theorems may likewise be
reformulated as quantitative bounds valid for all
$\pi, f$ and $\h \in (0,1)$.  We illustrate this translation in
Remark \ref{rmk:h-dependent-explication} of
\S\ref{sec:trace-estimates-i}.
As one sees in this
example, the resulting absolute statements are more explicit but
less succinct.  For this reason, we have usually preferred the
$\h$-dependent formulation in this paper.
% We illustrate this
% principle in Remark \ref{rmk:h-dependent-explication} of
% \S\ref{sec:trace-estimates-i}.

\subsubsection{Groups}\label{sec:intro-groups}
Let $k$ be a field of characteristic zero.
We denote by $\kbar$ an algebraic closure of $k$.

We generally denote by
$X$ the set of $k$-points of a $k$-variety $\mathbf{X}$ (and
vice-versa, if $\mathbf{X}$ is clear by context).  We identify
$\mathbf{X}$ with $X$ when $k = \kbar$.
\index{$\mathbf{G}$ vs. $G$}
\index{reductive groups}

By a \emph{reductive group $\mathbf{G}$ over $k$},
we will always
mean a \emph{connected reductive algebraic} $k$-group $\mathbf{G}$.
The
set $G$ of $k$-points of $\mathbf{G}$ is then Zariski dense
(see \cite[Corollary 18.3]{MR1102012}).
 
We will frequently
use
restriction of scalars
to regard reductive groups over $\mathbb{C}$
also as reductive groups
over $\mathbb{R}$.

For an algebraic $k$-group $\mathbf{G}$, we denote by $G$,
$\mathfrak{g}$ and $\mathfrak{g}^*$ the respective sets of
$k$-points of $\mathbf{G}$, of its Lie algebra, and of the
$k$-dual of its Lie algebra.
The group $G$
acts on $\mathfrak{g}$ and $\mathfrak{g}^*$.
Recall that an element of
one of
the latter spaces
is \emph{regular}
if its orbit
has maximal dimension, or equivalently,
if its centralizer has minimal dimension.
We use a
subscripted ``$\reg$'' to denote subsets of regular elements.

Suppose now that $\mathbf{G}$ is an algebraic
group over an archimedean local field $F$,
thus $F =\mathbb{R}$ or $\mathbb{C}$.
We may then regard $G$ as a \emph{real} Lie group;
complex Lie groups
do not play a role in this paper.

For a Lie group $G$,
we write $\mathfrak{g}$ for its Lie algebra
and
$\mathfrak{g}^\wedge$\index{$\mathfrak{g}^\wedge$}
for the Pontryagin dual of $\mathfrak{g}$.
We identify $\mathfrak{g}^\wedge$ with $i \mathfrak{g}^*$
(see \S\ref{sec:measures-et-al-G-g-g-star}
for details).

% \subsubsection{Lie groups}
% Let $G$ be a Lie group over
% $\mathbb{R}$.
% We denote by $\mathfrak{g}^\wedge$ the
% Pontryagin dual of $\mathfrak{g}$, i.e., the group of continuous
% homomorphisms from $\mathfrak{g}$ to the circle group
% $\mathbb{C}^{(1)} := \{z \in \mathbb{C} : |z| = 1\}$.  The group
% $\mathfrak{g}^\wedge$ is a vector space of the same dimension
% as $\mathfrak{g}$,
% and is non-canonically isomorphic to the
% linear dual $\mathfrak{g}^*$; the isomorphism
% depends
% upon the choice of a nontrivial character
% $k \rightarrow \mathbb{C}^{(1)}$.

\subsubsection{Unitary representations}
Let $\mathbf{G}$ be a reductive group over a local field
$k$ 
of characteristic zero.
As above, we denote by $G$ the group of
$k$-points of $\mathbf{G}$.
By definition, a unitary representation
$\pi$ of $G$ is a Hilbert space $V$ equipped with a
strongly continuous homomorphism $\pi$ from $G$ to the group of
unitary operators on $V$.
In Part \ref{part:micr-analys-lie},
Part \ref{part:micr-analys-lie2} and Appendix
\ref{sec:prel-repr-reduct}, we commit the standard abuse of
notation by writing $\pi$ for the Hilbert space $V$; we then
denote by $\pi^{\infty}$ the subspace of smooth vectors
(where in the non-archimedean case,
``smooth'' means ``open stabilizer'').
In
\S\ref{sec:spher-char-disint},
\S\ref{sec:sph-char-2}, Part
\ref{sec:inv-branch} and Part \ref{part:appl-aver-gan}, it will
be convenient instead to write $\pi$ for the subspace of smooth
vectors in $V$.  It will occasionally be useful also to confuse
``unitary'' with ``unitarizable.''  The reader will be reminded
locally of these conventions.

\subsubsection{Topologies
  on vector spaces}
\label{sec:prim-topol-vect}
When given a vector space $V$ defined
as the subset of some larger space $\tilde{V}$ on which some
$[0,+\infty]$-valued
seminorms take finite values,
we define $I(V)$ to be the set of restrictions
to $V$ of those seminorms,
and equip $V$ with its \emph{evident topology}:
that for which
an open base at $v_0 \in V$ is given by the sets
$\{v \in V : \|v-v_0\|_i < \eps \text{ for all } i \in M \}$,
where $\eps$ runs over the positive reals and $M$ over the
finite subsets of $I(V)$.
In all examples we consider,
the seminorms will separate points,
so $V$ will be a Hausdorff locally convex space.
\begin{example}\label{exa-C-infty-U}
  This discussion applies to $V = C^\infty(U)$, for $U$ an open
  subset of $\mathbb{R}^n$, taking for $I(V)$ the set of pairs
  $(K,D)$, where $K \subseteq U$ is compact and $D$ is a
   differential operator (with smooth coefficients, say)
  on $\mathbb{R}^n$, with
  $\|f\|_{(K,D)} := \sup_{x \in K} |D f(x)|$.
  It applies also to the Sobolev spaces $\pi^s$ ($s \in
  \mathbb{Z} \cup \{\infty \}$)
  defined below (\S\ref{sec:sobolev-spaces})
  for a unitary representation $\pi$ of a Lie group.
\end{example}

Unwinding the definitions, a linear map $T : V \rightarrow W$
between two spaces
so-equipped
is \emph{continuous} if for each $j \in I(W)$ there is a finite
subset $M \subseteq I(V)$ and a scalar $C \geq 0$ so that
\begin{equation}\label{eq:continuity-between-graded-pre-frechet-spaces}
  \|T v\|_{j} \leq C \sum_{i \in M} \|v\|_{i}
  \text{ for all } v \in V,
\end{equation}
while a family of such linear maps $T_\alpha : V \rightarrow W$
is \emph{equicontinuous}
if for each $j$, we may choose $M$ and $C$ uniformly with respect
to the family's indexing parameter $\alpha$.
We might also describe
the latter situation by saying
that \emph{$T_\alpha$ is continuous, uniformly in $\alpha$.}

We can make an analogous definition \emph{even if the spaces
  themselves vary}, provided that the indexing sets for their
seminorms admit natural identifications.  Thus, suppose
$T_\alpha : V_\alpha \rightarrow W_\alpha$ is a family of linear
maps between spaces as above, and suppose also given
identifications $I(V_\alpha) = I(V_\beta)$ and
$I(W_\alpha) = I(W_\beta)$ for all $\alpha,\beta$.  We may then
speak as above of $T_\alpha$ being \emph{continuous, uniformly
  in $\alpha$.}
\begin{example}
  Fix a Lie group $G$ and an element $x$ of its Lie algebra.
  As $\pi$ varies over the unitary representations of
  $G$, the family of operators
  $\pi^\infty \rightarrow \pi^\infty$
  defined by $x$ is
  continuous, uniformly in $\pi$.
\end{example}

Let $V$ be a vector space arising as the increasing union of a
sequence of topological vector subspaces $V_m$ as above,
with 
continuous inclusions.
\begin{example}
  $C_c^\infty(U)$, for $U$ as in Example \ref{exa-C-infty-U},
  can be described in this way,
  as can the space $\pi^{-\infty} = \cup_{s \in \mathbb{Z}} \pi^s$ of distributional vectors
  defined below (\S\ref{sec:sobolev-spaces}).
\end{example}
By the
\emph{evident topology} on $V$, we then mean the
``locally convex inductive limit''
%or ``direct limit''
topology,
which may be characterized
in either of the following ways:
\begin{enumerate}[(i)]
\item It is the finest for
  which the inclusions $V_m \hookrightarrow V$ are continuous.
\item For every
  locally convex space $T$, the continuous maps $V \rightarrow T$ are
  precisely those that restrict to continuous maps
  $V_m \rightarrow T$ for each $m$.
\end{enumerate}
Given another space $W$ as above, 
we say that a family
of maps $T_\alpha : V \rightarrow W$ is \emph{continuous,
  uniformly in $\alpha$} if for each $m$, the
family of restrictions $T_\alpha : V_m \rightarrow W$ has the
property explained above.

We consider several examples
(\S\ref{sec:symb-basic-propz},
\S\ref{sec:variation-of-op-wrt-chi}, \S\ref{sec:more-spaces-ops})
of vector spaces $V$ consisting of
\index{$\h$-dependent elements}
``$\h$-dependent
vectors in a varying family of vector spaces,
equipped with the evident topology.''
More formally,
suppose given:
\begin{itemize}
\item 
$\h$-dependent vector spaces $V(\h)$; 
\item An indexing set 
$\mathcal{I}$ for $\h$-dependent 
 seminorms on $V(\h)$;
thus, for each $i \in \mathcal{I}$, 
we are given an $\h$-dependent seminorm $i(\h): V(\h) \rightarrow [0,\infty]$. 
\end{itemize}
Let $V$ denote the subspace of the Cartesian product $\prod_{\h}
V(\h)$
consisting of 
$\h$-dependent vectors $v = v(\h)$
for which the seminorm
  \begin{equation}
\label{vsemi}
\|v\|_{V,i} := \sup_{\h} \|v(\h)\|_{i(\h)}
\end{equation}
is finite for each $i \in \mathcal{I}$. 
We then topologize $V$ by means of these seminorms. 

Still in the situation just described, suppose given 
an $\h$-dependent positive real $c > 0$.
We denote by $c V$ the image of $V \subseteq \prod_{\h} V(\h)$
under the bijection $v \mapsto c v = c(\h) v(\h)$,
equipped with the seminorms and topology transported
from $V$ via this bijection.
Thus an $\h$-dependent vector $v = v(\h) \in V(\h)$ belongs
to $c v$
if every seminorm
\[
  \|v\|_{c V,i}
  :=
  \sup_{\h \in (0,1]} \|c(\h)^{-1} v(\h)\|_{i(\h)}
\]
is finite.
For instance, we may speak of the space $\h^{10} V$;
it is the image of $V$ under multiplication by $\h^{10}$.
We denote by $\h^\infty V$
the intersection $\cap_{\eta} \h^{\eta} V$,
topologized in the evident way.
In practice, $\h^\infty V$
consists of $\h$-dependent vectors which are ``negligible''
in the $\h \rightarrow 0$ limit.

\begin{example}\label{example:h-eta-S-V}
  Let $V$ be a finite-dimensional real vector space.  We
  topologize the Schwartz space $\mathcal{S}(V)$, as usual,
  by means of the seminorms $f \mapsto \|D f\|_{L^\infty}$
  indexed by the polynomial-coefficient differential operators
  $D$ 
  on
  $V$.  Per the above conventions, $\h^\eta \mathcal{S}(V)$ is the
  space of $\h$-dependent Schwartz functions $f = f(\h)$ whose
  seminorms 
  satisfy,
  for each $D$ as above,
  \[
  \sup_{\h} \h^{-\eta} \|D f(\h)\|_{L^\infty} < \infty.
  \]
  We apply this notation
  even
  when $\eta = 0$;
  elements of $\h^0 \mathcal{S}(V)$
  are then $\h$-dependent Schwartz
  functions whose seminorms are
  bounded uniformly with respect to $\h$.
\end{example}

\subsubsection{Miscellaneous}
For an element $\xi$ of a normed space,
we often write
\[
  \langle \xi \rangle := (1 + |\xi|^2)^{1/2}.
\]
This quantifies the size of $\xi$,
but is never smaller than $1$.

When we have equipped some space $X$ (e.g., a group $G$ as
above) with a ``standard'' measure $\mu$ (e.g., a Haar measure),
we often write $\int_{x \in X} f(x)$ as shorthand for
the integral
$\int_{x \in X} f(x) \, d \mu(x)$
taken with respect to the standard measure.

We extend addition to a binary operation $+$ on the extended
real line $\mathbb{R} \cup \{\pm \infty\}$,
given by
\[
\infty + (-\infty) = (-\infty) + \infty := -\infty,
\]
and in other cases
in the obvious way;
this extension is used starting in \S\ref{sec:symb-basic-propz}.

\iftoggle{cleanpart}
{
  \newpage
}
\part{Microlocal analysis on Lie group representations I: definitions and basic properties}\label{part:micr-analys-lie}
Let $G$ be a unimodular Lie group over $\mathbb{R}$, with Lie
algebra $\mathfrak{g}$.  Let $\pi$ be a unitary representation
of $G$.  This part (Part \ref{part:micr-analys-lie})
and its sequel (Part \ref{part:micr-analys-lie2})  will study the basic quantitative properties of an
assignment $\Opp$ from functions on the dual of $\mathfrak{g}$
to operators on $\pi$. We will describe the contents of both parts here;
we suggest that the reader peruse Part \ref{part:micr-analys-lie}
and consult Part \ref{part:micr-analys-lie2}  as needed. 

\S\ref{sec-1-2}--\S\ref{sec:oper-attach-symb-1}
and their sequels
\S\ref{sec:star-prod-expn}--\S\ref{sec:oper-attach-symb}
concern aspects of
this assignment which apply to any unitary representation $\pi$,
such as
$\pi = L^2(G)$.\footnote{ It may be possible to reduce from the
  general case to this particular one and then to appeal to the
  pseudodifferential calculus as in \cite[Prop 1.1]{MR764508}, but doing
  so does not seem to yield an overall simplification, so we
  have opted instead for a direct and self-contained treatment.
}
Their results may be equivalently formulated in terms of the
convolution structure on $L^1(G)$, and more generally on spaces
of distributions supported near and singular only at the
identity element.  In particular, all estimates in these
sections are uniform in $\pi$.

In \S\ref{sec:infin-char}, we record 
some preliminaries concerning the relationship between
representations of real reductive groups and their infinitesimal
characters.  These are relevant for us because our main result
concerns averages over automorphic representations having
infinitesimal character in a prescribed region.

\S\ref{sec:kirillov-frmula-baby}
and its sequel
\S\ref{sec:kirillov-frmula} establish finer properties of
$\Opp$ for $G$ reductive and $\pi$ irreducible and tempered.
The relevant consequence of these assumptions is the
Kirillov-type formula.

We note that \S\ref{sec:sph-char-2} of Part
\ref{part:gross-prasad-pairs} applies the results of Part
\ref{part:micr-analys-lie} to determine (under certain
assumptions) the ``asymptotic decomposition'' of $\Opp$ under
restriction to certain subgroups.

The contents of Part
\ref{part:micr-analys-lie}
% as philosophically unsurprising
% generalizations of
may be understood as generalizations of standard results in the
theory of pseudodifferential operators (see \cite{MR2304165,
  MR0367730} and references), which corresponds roughly to the
case in which $G$ is $2$-step nilpotent.
We note some minor differences:
\begin{enumerate}[(i)]
\item The exponential map is neither injective nor surjective
  for the groups $G$ of interest to us, so we work within a
  fixed small enough neighborhood of the identity element of
  $G$.
\item The Baker--Campbell--Hausdorff formula is much simpler
  when $G$ is $2$-step nilpotent; for the $G$ of interest, it
  contains arbitrarily nested commutators.  It was not
  obvious to us at the outset whether this would present an
  obstruction.
\item The appropriate way to generalize the standard operator
  classes in the theory of pseudodifferential operators was not
  obvious to us; the definition given in
  \S\ref{sec:operator-classes} took some work to identify.
\end{enumerate}

  We note that the $\star$-product implicit in our operator calculus on
  the Lie algebra has been previously considered by Rieffel
  \cite{MR1064995}, who observed that it gives a deformation
  quantization of the Poisson structure.  As Rieffel observes in
  \cite[p658]{MR1064995}: ``At the heuristic level, our results
  also mesh nicely with Kirillov's orbit method for the theory
  of group representations, and perhaps the connection can be
  made stronger.''  The results of this paper concerning the
  operator calculi may be seen as steps in this direction.
  We mention also the work of B. Cahen
  and B. Harris (see for instance \cite{MR3994322, MR3812114}).

\subsubsection*{Notation}
We denote by $\pi^\infty \leq \pi$ the (dense) subspace of smooth
vectors, by $\mathfrak{U}$ the universal enveloping algebra of
$\mathfrak{g}_\mathbb{C}$, and by
$\pi : \mathfrak{U} \rightarrow \End(\pi^{\infty})$ the induced
map.  For $u \in \mathfrak{U}$, we occasionally write simply $u$
instead of $\pi(u)$
when it is clear from context that $u$ is acting on $\pi$.

\section{Operators attached to Schwartz functions}
\label{sec-1-2}
We define here the basic construct of the microlocal calculus
on representations:  an assignment from Schwartz functions on the dual
of the Lie algebra to operators. See \eqref{eq:opp-intro} and surrounding discussion for motivation. 
We extend and refine this assignment in later sections.

\subsection{Measures, Fourier transforms, etc.}
\label{sec:measures-et-al-G-g-g-star}
We fix an open neighborhood
$\mathcal{G}$ of the
origin
in the Lie algebra $\mathfrak{g}$,
 taken sufficiently small.
In particular,
$\exp : \mathcal{G} \rightarrow G$
is an analytic isomorphism onto its image.

We choose Haar measures
$d g$ on $G$
and $d x$ on $\mathfrak{g}$
satisfying the following compatibility:
for $x \in \mathcal{G}$ and $g = \exp(x)$,
we have \index{$j(x)$}
$d g = j(x) \, d x$,
where
the analytic function
$j : \mathcal{G} \rightarrow \mathbb{R}_{>0}$
satisfies $j(0) = 1$.

\index{$\mathcal{G}$}
Since $\mathcal{G}$ is
small, the image $j(\mathcal{G})$ is a precompact subset of
$\mathbb{R}_{>0}$.
We assume also that $-x \in \mathcal{G}$ whenever
$x \in \mathcal{G}$.

We use the letters
$x,y,z$ for elements of $\mathfrak{g}$
and $\xi,\eta,\zeta$ for elements of
its imaginary dual
$i \mathfrak{g}^* := \Hom_{\mathbb{R}}(\mathfrak{g},i
\mathbb{R})$.
We denote the natural pairing
by $x \xi := \xi x := \xi(x) \in i \mathbb{R}$.

Recall that $\mathfrak{g}^\wedge =
\Hom(\mathfrak{g},\mathbb{C}^{(1)})$
denotes the Pontryagin dual.
We identify $i \mathfrak{g}^*$
with $\mathfrak{g}^\wedge$
via the canonical isomorphism
$\xi \mapsto [x \mapsto e^{x \xi}]$;
we will find it clearer to work
with $\mathfrak{g}^\wedge$ for analytic purposes,
$i \mathfrak{g}^*$ for
algebraic purposes.
In particular,
we identify
$\mathfrak{g}$ 
with the space of $i\mathbb{R}$-valued
linear functions on $\mathfrak{g}^\wedge$.

We write $\mathcal{S}(\dotsb)$ for the Schwartz space.\index{$\mathcal{S}$}
We equip $\mathfrak{g}^\wedge$
with the Haar measure $d \xi$
for which
the Fourier transforms
\index{Fourier transforms $\wedge$, $\vee$}
$\mathcal{S}(\mathfrak{g}) \ni \phi \mapsto \phi^\wedge \in
\mathcal{S}(\mathfrak{g}^\wedge)$
and $\mathcal{S}(\mathfrak{g}^\wedge) \ni a \mapsto a^\vee \in
\mathcal{S}(\mathfrak{g})$
defined by
\[\phi^\wedge(\xi) := \int_{x \in \mathfrak{g}}
  \phi(x)
  e^{x \xi}
  \, d x,
  \quad
  a^\vee(x) := \int_{\xi \in \mathfrak{g}^\wedge}
  a(\xi)
  e^{-x \xi} \, d \xi
\]
are mutually inverse.

We let $G$ act on $\mathfrak{g}$
by the adjoint action $g \cdot x := \Ad(g) x$,
on $\mathfrak{g}^\wedge = i \mathfrak{g}^*$
by the coadjoint action $x(g \cdot \xi) :=
(g^{-1} \cdot x) \xi$,
and on functions $f$ on either space
by $g \cdot f :=
f(g^{-1} \cdot -)$.

\subsection{The basic operator map\label{sec:op-first}}
For a neighborhood $\mathcal{G}$ of the origin in
\index{cutoff $\chi$}
$\mathfrak{g}$ as above,
we let $\mathcal{X}(\mathcal{G})$ denote the set
of all ``cutoffs'' $\chi \in C_c^\infty(\mathcal{G})$
with the following properties:
\begin{enumerate}[(i)]
\item $\chi$ is even: $\chi(x) = \chi(-x)$ for all $x$.
\item $\chi(x) \in [0,1]$ for all $x$.
  In particular, $\chi$ is real-valued.
\item $\chi(x) = 1$ for all $x$ in
  some neighborhood of the origin.
\end{enumerate}
For $\mathcal{G}$ as above, $\chi \in
\mathcal{X}(\mathcal{G})$
and $a \in \mathcal{S}(\mathfrak{g}^\wedge)$,
we define
the bounded operator\index{$\Opp$}
$\Opp(a,\chi:\pi)
\in \End(\pi)$
by the formula
\begin{equation}\label{eq:defn-quantization-schwartz}
  \Opp(a,\chi:\pi) v := \int_{x \in \mathfrak{g}} \chi(x) a^\vee(x) \pi(\exp(x)) v \, d x.
\end{equation}
We abbreviate $\Opp(a,\chi) := \Opp(a,\chi:\pi)$
when $\pi$ is clear from context.  
Starting
in \S\ref{sec:variation-of-op-wrt-chi},
we further abbreviate
$\Opp(a) := \Opp(a,\chi)$.

%\pn{bumped up} \av{here I would vote for putting back in its
%  prevoius place: let people absorb version 0 before version1}
%\pn{Hmm, I bumped it up primarily so that I could
%  insert a meaningful comparison with the heuristic
%  discussion from \S\ref{sec:intro-op-calc}.
%  I also figured
%  that we work throughout the article primarily with the rescaled
%  operator map
%  (rather than with symbols at fixed $\h$ whose
%  supports tend off to infinity), so it makes sense
%  to introduce it fairly early,
%  and also it makes an instructive first exercise
%  to verify \eqref{eq:defn-quantization-schwartz-rescaled}.
%  I agree though that we shouldn't overload the reader.
%  Thoughts on balancing these competing priorities?
%}

Recall that $\h$ denotes a parameter in
$(0,1]$.
For $a$ as above,
we define the rescaled function
$a_{\h}(\xi) := a(\h \xi)$
and the correspondingly rescaled operator
\index{$\Opp_{\h}$}
\begin{equation}\label{eq:defn-quantization-schwartz-rescaled}
  \Opp_{\h}(a,\chi) := \Opp(a_{\h},\chi)
  =
  \int_{x \in \mathfrak{g}} \chi(\h x) a^\vee(x) \pi(\exp(\h x))
  \, d x.
\end{equation}
 As informal motivation for this definition,
recall from
\eqref{eq:opp-intro-acts-v}
the desiderata
that $\Opp_{\h}(a) v \approx a(\xi) v$
for $v \in \pi$ microlocalized
at $\xi \in \mathfrak{g}^\wedge$,
i.e., for which $\pi(\exp(\h x)) v \approx e^{x \xi} v$
whenever $|x| \ll 1$.
Indeed,
since $a^\vee(x)$ is small unless $|x| \ll 1$,
in which case $\chi(\h x) = 1$,
we see from \eqref{eq:defn-quantization-schwartz-rescaled} that
$\Opp_{\h}(a,\chi) v
\approx
(\int_{x \in \mathfrak{g}} a^\vee(x) e^{x \xi} \, d x) v
= a(\xi) v$.

\subsection{Adjoints\label{sec:self-adj-of-op-schw}}
Our assumptions on $\chi$ imply that 
the adjoint $\Opp(a,\chi)^*$
of $\Opp(a,\chi)$, regarded as a bounded operator on $\pi$,
is given by $\Opp(\overline{a}, \chi)$.

\subsection{Equivariance\label{sec:equivariance-prelim}}
\label{sec-47-7-5}
Let $a \in \mathcal{S}(\mathfrak{g}^\wedge)$,
$g \in G$,
$\chi \in \mathcal{X}(\mathcal{G})$.
Assume that $\Ad(g) \supp(\chi) \subseteq \mathcal{G}$.
Then
$g \cdot \chi \in \mathcal{X}(\mathcal{G})$,
and we verify readily that
$\pi(g) \Opp(a,\chi) \pi(g)^{-1}
= \Opp(g \cdot a, g \cdot \chi)$.
\subsection{Composition: preliminary
  discussion\label{sec:comp-prelim}}
For $f \in C_c^\infty(G)$, we define $\pi(f) \in \End(\pi)$ by
$\pi(f) v := \int_{g \in G} f(g) \pi(g) v \, d g$.  Then
\begin{equation}\label{eqn:composition-convolution-f}
  \pi(f_1) \pi(f_2) = \pi(f_1 \ast f_2),
\end{equation}
where $\ast$ denotes
convolution.
This fact unwinds
to a composition formula for $\Opp(a,\chi)$:

Suppose given a pair of open neighborhoods
$\mathcal{G}, \mathcal{G}'$
of the origin as above
and cutoffs $\chi \in \mathcal{X}(\mathcal{G}), \chi'
\in \mathcal{X}(\mathcal{G}')$
for which:
\begin{enumerate}[(i)]
\item $\exp(\mathcal{G}) \exp(\mathcal{G}) \subseteq \exp(\mathcal{G}')$, so that we may
  define
  $\ast : \mathcal{G} \times \mathcal{G} \rightarrow \mathcal{G}
  '$
  by \[
  x \ast y := \log(\exp(x) \exp(y)).
  \]
\item $\chi '(x \ast y) = 1$ for all $x,y \in \supp(\chi)$.
\end{enumerate}
Let us denote temporarily by $\phi \mapsto f$
the topological isomorphism $C_c^\infty(\mathcal{G}) \rightarrow
C_c^\infty(\exp(\mathcal{G}))$
defined by the rule $f(\exp(x)) := \phi(x) j(x)^{-1}$.
Then $f(g) \, d g$ is the pushforward
of $\phi(x) \, d x$,
and so
$\pi(f) = \int_{x \in \mathfrak{g}}
\phi(x) \pi(\exp(x)) \, d x$.
We may define continuous bilinear operators $\star$
on $C_c^\infty(\mathcal{G})$
and then on $\mathcal{S}(\mathfrak{g}^\wedge)$
by requiring that the diagram
\begin{equation*}
  \begin{CD}         
    C_c^\infty(\exp(\mathcal{G}))^2 @>\ast>> C_c^\infty(\exp(\mathcal{G}'))\\
    @A\cong A \phi \mapsto f A  @V f \mapsto \phi V \cong V \\    
    C_c^\infty(\mathcal{G})^2     @>\star>>
    C_c^\infty(\mathcal{G}')
    \\
    @AA a \mapsto \chi a^\vee A
    @V \phi \mapsto \phi^\wedge VV
    \\
    \mathcal{S}(\mathfrak{g}^\wedge)^2     @>\star>>  \mathcal{S}(\mathfrak{g}^\wedge)
  \end{CD}
\end{equation*}
commute.
The top arrow is defined thanks to (i).
The middle arrow
may be described conveniently
in terms of the
Fourier transform:
for $\phi_1, \phi_2 \in C_c^\infty(\mathcal{G})$
and $\zeta \in \mathfrak{g}^\wedge$,
\begin{equation}\label{eqn:future-reference-star-product-fourier}
  \int_{z \in \mathfrak{g}}\phi_1 \star \phi_2(z) e^{z \zeta}
  \, d z
  = \int_{x,y \in \mathfrak{g}}
  \phi_1(x) \phi_2(y) e^{(x \ast y) \zeta} \, d x \, d y.
\end{equation}
(Indeed, by definition, $\phi_1 \star \phi_2(z) \, d z$
is obtained
from
$\phi_1(x) \phi_2(y)  \, d x \, d y$
by pushing forward to the group, convolving,
and then pulling back to the Lie algebra;
testing this definition against $z \mapsto e^{z \zeta}$ gives
\eqref{eqn:future-reference-star-product-fourier}.)
The bottom arrow,
which we refer to as
the \emph{star product},\index{star product $\star$}
is
given by
$a_1 \star a_2 = (\chi a_1^\vee \star \chi a_2^\vee)^\wedge$.
Note that $\star$ depends upon $\chi$.

\begin{lemma*}
  For $a_1,a_2 \in \mathcal{S}(\mathfrak{g}^\wedge)$,
  \begin{equation}\label{eqn:composition-quantized-schwartz}
    \Opp(a_1,\chi) \Opp(a_2,\chi) 
    =
    \Opp(a_1 \star a_2, \chi ').
  \end{equation}
\end{lemma*}
\begin{proof}
  Set $\phi_i := \chi a_i^\vee
  \in C_c^\infty(\mathcal{G})$,
  and let $f_i \in C_c^\infty(\exp(\mathcal{G}))$
  be as associated above.
  The identity \eqref{eqn:composition-quantized-schwartz}
  follows
  from \eqref{eqn:composition-convolution-f}
  upon unwinding the definitions and noting
  that $\chi ' \equiv 1$ on the support of
  $\chi a_1^\vee \star \chi a_2^\vee$.
\end{proof}

We also have the rescaled composition formula
$\Opp_{\h}(a,\chi)\Opp_{\h}(b,\chi) = \Opp(a \star_{\h} b,
\chi ')$,
where the rescaled star product $a \star_{\h} b$
is defined by requiring that
\[
(a \star_{\h} b)_{\h} = a_{\h} \star b_{\h}.
\]
We note that $\star$ is recovered from $\star_{\h}$ by taking $\h =
1$.

\section{Operator classes}\label{sec:operator-classes}
The operator map 
\eqref{eq:defn-quantization-schwartz} has been defined
for functions $a$ belonging to the Schwartz space of $\mathfrak{g}^{\wedge}$. 
We want to extend it to other functions, e.g., functions $a$ that are permitted polynomial growth at $\infty$. In that case, $\Op(a)$ is no longer a bounded map $\pi \rightarrow \pi$;
rather, it behaves more like the (densely defined) operator on $\pi$ 
induced by an element of the universal enveloping algebra.  

Motivated by this, 
we define certain classes of densely defined operators on $\pi$. These classes will eventually serve as the target of $\Op$,
after we extend its definition to various symbol classes.

\subsection{The operator $\Delta$ and its inverse}
\label{sec:the-operator-delta}
Let us fix a basis $\mathcal{B} := \mathcal{B}(\mathfrak{g})$ of
$\mathfrak{g}$,
and set \index{$\Delta$}
\[
\Delta := \Delta_G := 1 - \sum_{x \in \mathcal{B}} x^2
\in \mathfrak{U}.
\]
We will often confuse
$\Delta$ with
its image $\pi(\Delta) \in \End(\pi^\infty)$.
It has the following properties:
\begin{enumerate}[(i)]
\item It induces a densely-defined self-adjoint positive
  operator on $\pi$ with bounded inverse $\Delta^{-1}$
  (see \cite{MR0110024}).
  The operator norm of $\Delta^{-1}$ is $\leq 1$.
\item
  For $n \geq 0$,
  let $D(\Delta^n) \subseteq \pi$ denote
  the domain of the densely-defined self-adjoint operator
  extending $\pi(\Delta^n)$.
  Then $\pi^{\infty} = \cap_{n \geq 0} D(\Delta^n)$  (see \cite[Cor 9.3]{MR0107176}).
  Consequently,
  $\Delta^{-1}$ acts on $\pi^\infty$.
\end{enumerate}

\subsection{Sobolev spaces}
\label{sec-1-5-2}
\label{sec:sobolev-spaces}
\index{$\pi^s$}
\index{$\pi^\infty$}
For $s \in \mathbb{Z}$,  we define an inner product
$\langle -, - \rangle_{\pi^s}$
on $\pi^\infty$ by
the rule
\[
\langle v_1, v_2 \rangle_{\pi^s} := \langle \Delta^s
v_1, v_2 \rangle.
\]
We denote by $\pi^s$
the Hilbert space completion of $\pi^\infty$ with respect
to the associated norm
$\|v\|_{\pi^s} := \langle v, v \rangle_{\pi^s}^{1/2}$.
These norms increase with $s$,
so up to natural identifications,
\[
\pi^\infty = \cap \pi^s \leq \dotsb \leq \pi^{s+1} \leq \pi^s
\leq \pi^{s-1} \leq \dotsb \leq \pi^{-\infty} :=
\cup \pi^s.
\]
These spaces come with evident topologies
(\S\ref{sec:prim-topol-vect}).

The inner product on $\pi$ induces
a duality between $\pi^s$ and $\pi^{-s}$.

We note  
that for each $s \in \mathbb{Z}_{\geq 0}$
there exist $C_s > c_s > 0$, depending upon $\mathcal{B}$,
so that
\begin{equation}\label{eq:pi-s-concretized}
  c_s \|v\|_{\pi^s}^2
  \leq
  \sum_{r=0}^s
  \sum_{x_1,\dotsc,x_r \in \mathcal{B}}
  \|x_{1} \dotsb x_{r} v\|^2
  \leq
  C_s  \|v\|_{\pi^s}^2
\end{equation}
(see, e.g.,  \cite[proof of Lem 6.3]{MR0107176}).

\subsection{Definition of operator classes}
\label{sec-1-5-3}
\label{sec:operator-spaces-defn-etc}
\index{operator space $\Psi^m$}
By an \emph{operator on $\pi$}, we mean simply a linear map
$T : \pi^\infty \rightarrow \pi^{-\infty}$.
For each $x \in \mathfrak{g}$,
the commutator
\[
  \theta_x(T) := [\pi(x), T]
\]
is likewise an operator on $\pi$.
Indeed,
by
\eqref{eq:pi-s-concretized},
we have for $s \geq 1$
that
$\pi(x) : \pi^s \rightarrow \pi^{s-1}$.
By duality, it follows that $\pi(x) : \pi^{1-s} \rightarrow
\pi^{-s}$.
In particular, $\pi(x)$ acts on both $\pi^{\infty}$ and
$\pi^{-\infty}$,
so both compositions $\pi(x) \circ T$ and $T \circ \pi(x)$ are defined.

The map $x \mapsto \theta_x$ extends to an
algebra morphism
$\mathfrak{U} \rightarrow \End (\{\text{operators on $\pi$}\})$,
denoted $u \mapsto \theta_u$.
For example,
for $x_1,\dotsc,x_n \in \mathfrak{g}$,
\[
\theta_{x_1 \dotsb x_n}(T) = [\pi(x_1), [\pi(x_2),
\dotsc, [\pi(x_n),T]]].
\]

\begin{definition*}
  For $m \in \mathbb{Z}$,
  we say that an operator $T$ on $\pi$ has \emph{order $\leq m$}
  if for each $s \in \mathbb{Z}$ and $u \in \mathfrak{U}$,
  the operator $\theta_u(T)$
  induces a bounded map
  \[
  \theta_u(T) : \pi^s \rightarrow \pi^{s-m}.
  \]
\end{definition*}
\begin{remark*}
  When $\pi$ is a standard representation of a Heisenberg group,
  this definition is closely related to Beals's characterization
  \cite{MR0435933} of pseudodifferential operators.
\end{remark*}

We denote by $\Psi^m := \Psi^m(\pi)$ the space of operators on
$\pi$ of order $\leq m$,
by
$\Psi^{-\infty} := \cap_m \Psi^m$
the space of ``smoothing operators,''
and by
$\Psi^{\infty} := \cup_m \Psi^m$
the space of ``finite-order operators.''
Then
\begin{equation}\label{eqn:inclusions-of-operator-classes}
  \Psi^{-\infty}
  \subseteq \dotsb
  \subseteq \Psi^{m-1}
  \subseteq \Psi^m \subseteq \Psi^{m+1}
  \subseteq \dotsb \subseteq \Psi^\infty.
\end{equation}
These
spaces come with
evident topologies
(\S\ref{sec:prim-topol-vect});
thus
for $m \in \mathbb{Z}$,
the relevant seminorms on $\Psi^m$
are $T \mapsto \|\theta_u(T)\|_{\pi^s \rightarrow
  \pi^{s-m}}$,
taken over $s \in \mathbb{Z}$ and $u \in \mathfrak{U}$.
The inclusions \eqref{eqn:inclusions-of-operator-classes}
are continuous.

We note that $\Psi^m$ depends,
implicitly,
upon the group $G$
which we regard as acting on
$\pi$.

\subsection{Composition}\label{sec:operator-classes-composition}
Observe
that finite-order operators act on the space of smooth vectors,
i.e., $\Psi^\infty \subseteq \End(\pi^\infty)$.
We may thus compose such operators.
\begin{lemma*}
  For $m_1, m_2 \in \mathbb{Z} \cup \{\pm \infty \}$,
  composition induces continuous maps
  $\Psi^{m_1} \times \Psi^{m_2} \rightarrow \Psi^{m_1 + m_2}$,
  where as usual
  $\infty + (-\infty) := -\infty$.
\end{lemma*}  
\begin{proof}
  For $T_1 \in \Psi^{m_1}$, $T_2 \in \Psi^{m_2}$ and
  $u \in \mathfrak{U}$, we may write $\theta_u(T_1 T_2)$ as
  a sum of expressions $\theta_{u_1}(T_1) \theta_{u_2}(T_2)$
  with $u_1,u_2 \in \mathfrak{U}$.
  For $s \in \mathbb{Z}$,
  the compositions of bounded maps
  \[
  \pi^{s}
  \xrightarrow{\theta_{u_2}(T_2)}
  \pi^{s - m_2}
  \xrightarrow{\theta_{u_1}(T_1)}
  \pi^{s-m_1 - m_2}
  \]
  are bounded,
  and so $\theta_u(T_1 T_2)$ induces a bounded map
  $\pi^s \rightarrow \pi^{s-m_1-m_2}$, as required.
\end{proof}

\subsection{Differential
  operators}\label{sec:sort-of-obvious-operator-class-memberships}
It is easy to see that $\Psi^0$
contains the identity operator.  We record some further examples:
\begin{lemma*}
  Let $m \in \mathbb{Z}$.  Then
  \begin{itemize}
  \item  $\pi(\Delta)^m \in \Psi^{2 m}$, and
  \item  $\pi(x_1 \dotsb x_m) \in \Psi^{m}$ if $m \geq 0$
    and $x_1,\dotsc,x_m \in \mathfrak{g}$.
  \end{itemize}    
\end{lemma*}
The
proof
is elementary but somewhat tedious,
hence postponed to \S\ref{sec:membership-criteria-technical-stuff}.

\subsection{Smoothing operators\label{sec:smoothing-ops}}
\begin{lemma}
   An operator $T$ on $\pi$
    belongs to $\Psi^{-\infty}$
    if and only if
    for each
    $N \in \mathbb{Z}_{\geq 0}$, the operator $\Delta^N T \Delta^N$
    induces a bounded map $\pi \rightarrow \pi$.
    The corresponding seminorms
    $T \mapsto \|\Delta^N T \Delta^N\|_{\pi \rightarrow \pi}$
    describe the topology on $\Psi^{-\infty}$.
  \end{lemma}
  \begin{proof}
    This follows readily from \eqref{eq:pi-s-concretized}
    and the definitions.
  \end{proof}
      \begin{lemma}      For $f \in C_c^\infty(G)$,
  one has $\pi(f) \in \Psi^{-\infty}$,
  and the induced map
  \[
  C_c^\infty(G) \rightarrow \Psi^{-\infty}
  \]
  is continuous.
\end{lemma}
\begin{proof}
  Set $f' := \Delta^N \ast f \ast \Delta^N$.
  We note that $\Delta^N \pi(f) \Delta^N = \pi(f')$,
  that
  $\|\pi(f')\|_{\pi \rightarrow \pi} \leq \|f'\|_{L^1}$,
  and that the map
  $f \mapsto f'$ on $C_c^\infty(G)$
  is continuous.
\end{proof}

\section{Symbol classes}\label{sec:symbol-classes-i}

As promised, we now define
various enlargements of the Schwartz space;
we will later extend $\Op$
to % be valued in
these
spaces.  

We also study the behavior of the star products
 $\star$ and $\star_{\h}$, defined above,
on these enlarged spaces. Eventually, under $\Op$,
these products will be intertwined with operator composition.

\subsection{Multi-index notation}
\label{sec-1-3-1}
Temporarily denote by $n := \dim(G)$ the dimension of the
underlying Lie group.
For convenience, we choose a basis
for
$\mathfrak{g}$.
This choice defines coordinates
$\mathfrak{g}
\ni x \mapsto (x_1,\dotsc,x_n)
\in \mathbb{R}^n$
and
$\mathfrak{g}^\wedge
\ni \xi \mapsto (\xi_1,\dotsc,\xi_n) \in i
\mathbb{R}^n$
for which $x \xi = x_1 \xi_1 + \dotsb + x_n \xi_n$.

For each ``multi-index''
$\alpha \in \mathbb{Z}_{\geq 0}^n$
we set $|\alpha| := \alpha_1 + \dotsb + \alpha_n$
and
$\alpha! := \alpha_1! \dotsb \alpha_n!$
and
$x^\alpha :=
x_1^{\alpha_1} \dotsb x_{n}^{\alpha_n}
\in \mathbb{R}$
and
$\xi^\alpha
:=
\xi_1^{\alpha_1} \dotsb \xi_n^{\alpha_n}
\in i^{|\alpha|} \mathbb{R}$.
We define the differential operators
\index{differential operators $\partial^\alpha$}
$\partial^\alpha$
on $C^\infty(\mathfrak{g}^\wedge)$
and
on
$C^\infty(\mathfrak{g})$
by
requiring that
\begin{equation}
  (\partial^\alpha a)^\vee (x) = x^\alpha a^\vee(x),
  \quad
  (\partial^\alpha \phi)^\wedge  (\xi) = (-\xi)^\alpha \phi^\wedge(\xi)
\end{equation}
for
$a \in C^\infty(\mathfrak{g}^\wedge)$
and
$\phi \in C^\infty(\mathfrak{g})$;
the formal Taylor expansions then
read
\[
a(\zeta + \xi)
=
\sum_{\alpha}
\frac{\xi^\alpha}{\alpha !}
\partial^\alpha a(\zeta),
\quad
\phi(z + x)
=
\sum_{\alpha}
\frac{x^\alpha}{\alpha !}
\partial^\alpha \phi(z).
\]

We fix norms $|.|$ on $\mathfrak{g}$ and $\mathfrak{g}^\wedge$.
Recall that we abbreviate $\langle \xi \rangle := (1 + |\xi|^2)^{1/2}$.

\subsection{Formal expansion of the star product}
\label{sec-1-3-3}
\label{sec:lead-homog-comp-star}
\label{sec:BCHD}
Let $\ast$ and $\star_h$ be as in 
\S\ref{sec:comp-prelim}.
We define the analytic map $\{\cdot,\cdot\} : \mathcal{G} \times \mathcal{G}
\rightarrow \mathfrak{g}$
to be the ``remainder term''
$\{x,y\} := x \ast y - x - y$
in the Baker--Campbell--Hausdorff formula.
It follows then from
\eqref{eqn:future-reference-star-product-fourier}
that for $a, b \in \mathcal{S}(\mathfrak{g}^\wedge)$,
\begin{equation}\label{eqn:formula-for-a-star-b-suggesting-answer}
  a \star b(\zeta)
  =
  \int_{x,y \in \mathfrak{g}}
  a^\vee(x)
  b^\vee(y)
  e^{x \zeta} e^{y \zeta} e^{\{x, y\} \zeta}
  \chi(x) 
  \chi(y)
  \, d x \, d y.
\end{equation}
The factor $e^{\{x,y\} \zeta} \in \mathbb{C}^{(1)}$
defines an analytic
function of
$(x,y,\zeta)  \in \mathcal{G} \times \mathcal{G} \times \mathfrak{g}^\wedge$,
and so admits a power series expansion
\begin{equation}\label{eqn:power-series-for-exp-of-stuff}
  e^{\{x, y\}\zeta}
  = 
  \sum_{\alpha,\beta,\gamma}
  c_{\alpha \beta \gamma} x^\alpha
  y^\beta \zeta^\gamma.
\end{equation}

The estimate $\{x,y\} = \O(|x| \, |y|)$
controls which monomials
$x^{\alpha} y^{\beta} \zeta^{\gamma}$ actually appear in this
expansion.
Using superscript to indicate degree,
the terms appearing
are $1$, then
$x^r y^s \zeta^1  (r, s \geq 1)$, then
$x^r y^s \zeta^2  (r, s \geq 2)$, and so on.
In particular,
$x^a y^b \zeta^{c}$ appears
only if $j = a + b - c$ is a nonnegative integer,
and each $j$ corresponds to finitely many terms.

By grouping the RHS of
\eqref{eqn:power-series-for-exp-of-stuff}
in this way,
substituting
into
\eqref{eqn:formula-for-a-star-b-suggesting-answer}
and casually discarding the truncations $\chi$,
we arrive
at the formal asymptotic expansion
\begin{equation}\label{eq:formal-star-0}
  a \star b \sim \sum_{j \geq 0} a \star^j
  b,
\end{equation}
where $\star^j$
is
the finite bidifferential operator
on $C^\infty(\mathfrak{g}^\wedge)$
defined
by
\begin{equation}\label{eqn:a-star-i-b-defn} 
  a \star^j b(\zeta)
  :=
  \sum _{
    \substack{
      \alpha,\beta,\gamma :
      \\
      |\alpha| + |\beta| - |\gamma| = j,
      \\
      |\gamma| \leq \min(|\alpha|,|\beta|),
      \\
      \max(|\alpha|,|\beta|) \leq j
      \\
    }
  }
  c_{\alpha \beta \gamma}
  \zeta^\gamma \partial^\alpha a(\zeta) \partial^\beta b(\zeta).
\end{equation}
(We have introduced some redundant summation
conditions
for emphasis.)
The operator $\star^j$
has order $\leq j$ with respect
to both variables
and is homogeneous of degree $j$
with respect to dilation,
i.e.,
$a_{\h} \star^j b_{\h} = \h^j (a \star^j b)_{\h}$,
thus \eqref{eq:formal-star-0}
and \eqref{eqn:a-star-i-b-defn}
suggest
the rescaled formal expansion
\begin{equation}\label{eq:formal-star}
  a \star_{\h} b \sim \sum_{j \geq 0} \h^j a \star^j
  b.
\end{equation}
The leading term
is $a \star^0 b = a b$
(pointwise multiplication of functions),
while
the next term
$a \star^1 b$
is a multiple of the Poisson bracket
(cf. Gutt \cite{gutt1983explicit}
for explicit formulas for general $j$).

\subsection{Basic symbol classes}\label{sec:defin-basic-symb}
\begin{definition*}
  Let
  $m \in \mathbb{R}$.
  We write $S^m$ for the space of \index{symbol class $S^m$}
  smooth functions
  $a : \mathfrak{g}^\wedge \rightarrow \mathbb{C}$
  so that
  % We say that a smooth function
  % $a : \mathfrak{g}^\wedge \rightarrow \mathbb{C}$
  % is a \emph{symbol of order $\leq m$},
  % and write
  % \[
  % a \in S^m,
  % \]
  % if
  for each multi-index $\alpha$ there exists
  $C_\alpha \geq 0$
  (depending upon $a$)
  so that
  for all $\xi \in \mathfrak{g}^\wedge$,
  \begin{equation}\label{eqn:estimate-defining-S-m}
    \left\lvert \partial^\alpha a(\xi) \right\rvert \leq
    C_\alpha \langle \xi  \rangle^{m-|\alpha|},
  \end{equation}
  where as usual $\langle \xi  \rangle = (1 + |\xi|^2)^{1/2}$.
  We extend this definition to $m = -\infty$ by taking
  intersections and to $m = +\infty$ by taking unions.
  We refer to elements $a \in S^m$ as \emph{symbols of order
    $\leq m$}.
\end{definition*}
We equip
$S^m$ with its evident topology
(\S\ref{sec:prim-topol-vect}).
\begin{example*}
If $m \in \mathbb{Z}_{\geq 0}$, then a polynomial of degree
$\leq m$ defines an element of $S^m$.  The space $S^{-\infty}$
coincides with the Schwartz space
$\mathcal{S}(\mathfrak{g}^\wedge)$.
\end{example*}

For finite $m$, we may characterize the elements $a \in S^m$
informally as those which oscillate dyadically and are bounded
by a multiple of $\langle \xi \rangle^m$.  To see why this is
the case, take $m=0$ for concreteness.  For $R > 1$, the
estimate \eqref{eqn:estimate-defining-S-m} says that the
rescaled function $\xi \mapsto a(R\xi)$ is smooth, with fixed
bounds for all derivatives inside the dyadic annulus
$1 < |\xi| < 2$. In fact, it is bounded by $C_0$, its
derivatives bounded by $\max_{|\alpha| = 1} C_\alpha$, its
second derivatives bounded by $\max_{|\alpha|=2} C_\alpha$, and
so on.

\subsection{$\h$-dependent symbol classes}\label{sec:h-dependent-symbol}
We encourage the reader to skim
this section, and all further discussion of $\h$-dependent symbol classes, 
on a first reading;
these classes will be exploited
in the \emph{proofs} of our core
technical results (Parts \ref{part:micr-analys-lie2}
and
\ref{part:gross-prasad-pairs}),
but not in their \emph{applications}
(Parts \ref{sec:inv-branch} and \ref{part:appl-aver-gan}).

Recall that $\h$ denotes
a small positive parameter, and that symbols in $S^m$ oscillate
on dyadic ranges.  We will occasionally need to consider symbols
that vary in a controlled manner with $\h$ and oscillate on
slightly smaller than dyadic ranges.

Recall (\S\ref{sec:asympt-param-h}) that
an \emph{$\h$-dependent function}
$a : \mathfrak{g}^\wedge \rightarrow \mathbb{C}$
is a function
which depends --
perhaps implicitly -- upon $\h$:
\[
  a(\xi) := a(\xi;\h).
\]
We still denote, as before,
by $a_{\h}$ the rescaled $\h$-dependent function
\[
a_{\h}(\xi) := a(\h \xi) = a(\h \xi; \h).
\]

\begin{definition*}
  Let $m \in \mathbb{R}$, $\delta \in [0,1)$.
  Let $a : \mathfrak{g}^\wedge \rightarrow \mathbb{C}$
  be a smooth $\h$-dependent function.
  We write\index{symbol class $S^m_{\delta}$}
  \[
    a \in S^m_{\delta}
  \]
  if for each multi-index $\alpha$
  there exists $C_\alpha \geq 0$
  so that for all $\xi \in \mathfrak{g}^\wedge$ and $\h \in
  (0,1]$,
  \begin{equation}\label{eqn:defining-inequality-for-S-m-h-delta}
    |\partial^\alpha a(\xi)| \leq C_\alpha \h^{-\delta |\alpha|}
    \langle \xi \rangle^{m-|\alpha|}.
  \end{equation}
\end{definition*}
For instance,
$a \in S^m_0$ means that
\begin{itemize}
\item $a(\cdot;\h)$ belongs to
  $S^m$ for each $\h$, and
   \item the constants $C_\alpha$ 
     defining the membership
     $a(\cdot;\h) \in S^m$
     may be taken \emph{independent of
     $\h$.}
\end{itemize}
Informally,
$S^m_\delta$
consists of elements
which oscillate
at the scale
$\xi + \O(\h^{\delta} \langle \xi  \rangle)$
(see \S\ref{sec:localized-functions}
for details);
it
is obtained from $S^m_0$
by
``adjoining smoothened
characteristic functions of balls of
rescaled radius $\h^\delta$.''
As explained below,
$\delta = 1/2$ corresponds
to the Planck scale.
The most important range for us is
when
$\delta \in [0,1/2)$;
taking $\delta =0$
means we work at dyadic scales,
while taking $\delta$ close to $1/2$
means we work ``just above the Planck scale.''

We write
$S^m_\delta(\mathfrak{g}^\wedge)$ when
we wish to indicate
explicitly which Lie algebra $\mathfrak{g}$
is being considered.
We extend the definition to
$m = \pm\infty$ as before, and equip
these spaces with their evident topologies.
We note that as $m,\delta$ increase,
the spaces $S^m_\delta$ are related by continuous inclusions.

We may combine the notation
$S^m_\delta$
introduced here with the notation
$\h^\eta V$
introduced in
\S\ref{sec:prim-topol-vect}
to obtain the space
\[
  \h^{\eta} S^m_\delta,
\]
consisting of  smooth $\h$-dependent functions
$a : \mathfrak{g}^\wedge \rightarrow \mathbb{C}$ satisfying
\[
    |\partial^\alpha a(\xi)| \leq \h^{\eta} C_\alpha \h^{-\delta
    |\alpha|} \langle \xi \rangle^{m-\alpha}.
\]

% (The notation
% on the RHS of \eqref{eqn:differentiation-of-symbol-class},
% which is likely clear by context,
% was defined
% precisely in \S\ref{sec:h-dependent-symbol} and in Example
% \ref{example:h-eta-S-V}
% of \S\ref{sec:prim-topol-vect}.)

\begin{example*}
  Fix $f \in
  C_c^\infty(\mathbb{R}_{>0})$
  (independent of $\h$)
  and define the $\h$-dependent elements
  $a,b \in C_c^\infty(\mathfrak{g}^\wedge - \{0\})$
  by the formulas
  \[
    a(\xi) := f \left(
      \frac{|\xi| - 1}{\h^\delta }
    \right),
    \quad
    b(\xi) :=
    \h^{\delta} f(|\xi|).
  \]
  Then $a$ belongs to $S^{-\infty}_\delta$
  and $b$ belongs to $\h^{\delta} S^{-\infty}_0$,
  but not in general the other way around.
\end{example*}

\subsection{Basic properties}\label{sec:symb-basic-propz}
Pointwise
multiplication
defines continuous maps
\[
(- \cdot - ) : S^{m_1} \times S^{m_2} \rightarrow S^{m_1 + m_2},
\]
where
as usual $\infty + (-\infty) := -\infty$.
For any multi-index $\alpha$,
differentiation and monomial-multiplication
give continuous maps
\[
\partial^\alpha : S^m \rightarrow  S^{m- |\alpha|},
\]
\[
\xi^\alpha : S^m \rightarrow
S^{m+|\alpha|}.
\]
For the $\h$-dependent classes,
we have analogously
\[
(- \cdot - ) : S^{m_1}_{\delta_1} \times S^{m_2}_{\delta_2}
\rightarrow S^{m_1 + m_2}_{\max(\delta_1,\delta_2)},
\]
\begin{equation}\label{eqn:differentiation-of-symbol-class}
  \partial^\alpha : S^m_{\delta} \rightarrow \h^{-\delta|\alpha|} S^{m- |\alpha|}_{\delta},
\end{equation}
\[
\xi^\alpha : S^m_{\delta} \rightarrow
S^{m+|\alpha|}_{\delta},
\]
with notation as in \S\ref{sec:h-dependent-symbol}.

For a treatment of $\h$-dependent classes
in the setting of microlocal analysis on manifolds,
we mention \cite{MR3669792}.

\subsection{Star product asymptotics}\label{sec:star-prod-asympt}
 
% Recall
% (from \S\ref{sec:lead-homog-comp-star})
% that the differential operator $\star^j$
% has order $j$ in both
% variables; more precisely, it contains contributions from
% $(|\alpha|,|\beta|,|\zeta|) = (j,1,1)$ and $(1,j,1)$ and
% $(j,j,j)$,
% but no contributions of larger order.
 
Using the properties of $\star^j$
indicated in \S\ref{sec:lead-homog-comp-star},
we verify
readily
that
\begin{equation}\label{eqn:star-prod-motivate-1}
  \star^j :
  S^{m_1} \times  S^{m_2}
  \rightarrow S^{m_1 + m_2 - j},
\end{equation}
\begin{equation}\label{eqn:star-prod-motivate-2}
  \star^j : S^{m_1}_\delta
  \times  S^{m_2}_{\delta}
  \rightarrow  \h^{-2 \delta j} S^{m_1 + m_2 - j}_{\delta}.
\end{equation}
From \eqref{eqn:star-prod-motivate-1},
we see that the expansion \eqref{eq:formal-star} of the
star product
converges formally with respect to the symbol classes $S^m$.
From \eqref{eqn:star-prod-motivate-2}
and its proof,
we see that the same holds
for $S^m_{\delta}$
if $\delta < 1/2$,
but not if $\delta > 1/2$.
Indeed, if the symbols $a$ and $b$ oscillate substantially
at scales finer than about $\h^{1/2}$, then the summands
in the formal expansion
$\sum_{j \geq 0} \h^j a \star^{j} b$ do not decay as
$\h \rightarrow 0$ and $j \rightarrow \infty$.  The scale
$\h^{1/2}$ is thus the natural limit of our calculus,
corresponding to the ``Planck scale,'' or in the language of
\S\ref{sec:intro-op-calc}, to projections onto individual
vectors;
we discuss the latter point further
in remark \ref{rmk:planck-after-kirillov}
of \S\ref{sec:trace-estimates-i}.

These observations may motivate the following:
\begin{theorem}\label{thm:star-prod-basic}
  We have the following:
  \begin{enumerate}[(i)]
    \item
    There is a unique continuous bilinear extension
    \begin{equation}\label{eqn:star-on-S-infi}
      \star : S^\infty \times S^\infty \rightarrow S^\infty 
    \end{equation}
    of
    the star product $\star$,
    defined initially
    on Schwartz spaces
    as in \S\ref{sec:comp-prelim}.
    It induces continuous bilinear maps
    \begin{equation}
      \star : S^{m_1} \times S^{m_2} \rightarrow
      S^{m_1 + m_2}
    \end{equation}
    and,
    for
    $\delta \in [0,1/2)$,
    \[
    \star_{\h} : 
    S^{m_1}_{\delta} \times S^{m_2}_{\delta}
    \rightarrow S^{m_1 + m_2}_{\delta}
    \]
  \item
    Fix
    $J \in \mathbb{Z}_{\geq 0}$.
    If $a \in S^{m_1}$ and $b \in S^{m_2}$,
    then
    \[
    a \star b \equiv \sum_{0 \leq j < J}
    a \star^j b \mod{S^{m_1 + m_2 - J}},
    \]
    where the remainder term\index{star product homogeneous component $\star^j$}
    $r := a \star b - \sum_{0 \leq j < J}
    a \star^j b \in S^{m_1 + m_2 - J}$
    varies continuously with $a$ and $b$.
    In particular,
    \[
    a \star_{\h} b
    \equiv a b \mod{\h S^{m_1 + m_2 - 1}},
    \]
    \[
    a \star_{\h} b
    \equiv a b + \h a \star^{1} b \mod{\h^2 S^{m_1 + m_2 - 2}}.
    \]
    
    Similarly,
    for $\delta \in [0,1/2)$,
    $a \in S^{m_1}_{\delta}$ and $b \in S^{m_2}_{\delta}$,
    \[
    a \star_{\h}
    b \equiv \sum_{0 \leq j < J} \h^j a \star^j b
    \mod{\h^{(1 - 2 \delta) J} S^{m_1 + m_2 -J}_{\delta}},
    \]
    with continuously-varying remainder.
  \end{enumerate}
\end{theorem}
We verify this (in a slightly more general form) in
\S\ref{sec:star-prod-expn} below.  The proof is an application
of integration by parts and Taylor's theorem to the integral
representation
\eqref{eqn:formula-for-a-star-b-suggesting-answer}.

\begin{remark*}
  Our discussion applies with minor modifications
  to slightly more general symbol
  classes, e.g., for $m \in \mathbb{R}$, $\rho \in (0,1]$ and
  $\delta \in [0,1)$, to the class $S^m_{\delta,\rho}$
  defined by the condition
  $|\partial^\alpha a(\xi)| \leq C_\alpha \h^{-\delta |\alpha|}
  \langle \xi \rangle^{m-\rho|\alpha|}$,
  with the most important range being
  when $\rho \in (1/2,1]$ and $\delta \in [0,1/2)$.
  One could also
  rescale in more general ways than we do here, or work with
  symbols that are substantially rougher in directions
  transverse to the foliation of $\mathfrak{g}^\wedge$ by
  coadjoint orbits.  We are content here to develop the minimal
  machinery required for our motivating applications, leaving
  such extensions to the interested reader.
\end{remark*}

\section{Operators attached to symbols}\label{sec:oper-attach-symb-1}

\subsection{Weak definition of the operator
  map}\label{sec:weak-defin-oper}
 
Let
$\chi \in \mathcal{X}(\mathcal{G})$
(cf. \S\ref{sec:op-first}).
By \S\ref{sec:smoothing-ops},
the assignment\index{$\Opp, \Opp_{\h}$}
$\Opp(\cdot,\chi) : S^{-\infty} \rightarrow \Psi^{-\infty}$
is defined and continuous.
We calculate that
\begin{equation}\label{eqn:weak-extn}
  \langle \Opp(a, \chi) u, v \rangle
  =
  \int_{\xi \in \mathfrak{g}^\wedge} a(\xi)
  \Bigl(\int_{x \in \mathfrak{g}}
  e^{-x \xi}
  \chi(x) \langle \pi(\exp(x)) u, v \rangle \, d x \Bigr) \, d \xi
\end{equation}
for all $u,v \in \pi^\infty$.

We observe that
for any tempered
distribution $a$ on $\mathfrak{g}^\wedge$,
the formula \eqref{eqn:weak-extn} defines
an operator
$\Opp(a,\chi) : \pi^{\infty} \rightarrow \pi^{-\infty}$,
sending
smooth vectors to distributional vectors.
(To see this, we need only note that
the function $\mathfrak{g} \ni x \mapsto \chi(x) \langle
\pi(\exp(x)) u,v \rangle$
belongs to the Schwartz space.)
This observation
applies in particular
to $a \in S^{\infty}$.
We may similarly extend
the definition of the rescaled analogue $\Opp_{\h}(a,\chi)$.

\subsection{Polynomial symbols}
\label{sec:quantize-polynomials}
 
The identification of $\mathfrak{g}$ with the space of
$i\mathbb{R}$-valued linear functions on $\mathfrak{g}^\wedge$
extends to identify $\Sym(\mathfrak{g}_\mathbb{C})$ with the
space of polynomial symbols.
For instance, if $p = y_1 \dotsb y_n$
with each $y_i \in \mathfrak{g}$, then
$p(\xi) := y_1(\xi) \dotsb y_n(\xi)$.
Which operators
arise from such symbols?
\begin{lemma*}
  \label{lem:opp-of-polynomials}
  For $p \in \Sym(\mathfrak{g}_\mathbb{C})$,
  we have
  \[
  \Opp(p,\chi) = \pi(\sym(p)).
  \]
\end{lemma*}
Here
$\sym : \Sym(\mathfrak{g}_\mathbb{C}) \rightarrow \mathfrak{U}$
denote the symmetrization map,
i.e., the
linear isomorphism
that sends a monomial
to the average of its permutations.
The proof
is given
in \S\ref{sec:quantize-polynomials-pf}.

We assume henceforth
when working with $\Opp$
that the norm $|.|$
is chosen so that $\mathcal{B}(\mathfrak{g})$ is an orthonormal basis.
Then for $p(\xi) := 1 + |\xi|^2 = \langle \xi  \rangle^2$,
we have $\Opp(p) = \Delta$ (cf. \S\ref{sec:the-operator-delta}).

\subsection{$\h$-dependence}\label{sec:h-dependence-Psi-m}
When working
with the rescaled operator assignment
$\Opp_{\h}$,
we allow the representation $\pi$ to be $\h$-dependent,
i.e., to vary implicitly with the
small parameter $\h \in (0,1]$: $\pi = \pi(\h)$.
We emphasize that the variation of $\pi$ with $\h$ is arbitrary
(without, e.g., continuity or measurability requirements);
indeed, in our applications, we consider only those $\h$
belonging to some discrete subset of $(0,1]$.

\begin{definition*}
  Let $\pi$ be an $\h$-dependent unitary representation
  of $G$. Fix $\delta \in [0,1)$.
  We denote by $\Psi^m_\delta$
  the space of $\h$-dependent operators
  $T = T(\h)$ on $\pi = \pi(\h)$
  with the property that
  for each $u \in \mathfrak{U}$
  and $s \in \mathbb{Z}$,
  there exists $C_{u,s} \geq 0$  (independent of $\h$)
  so that for all $\h \in (0,1]$,
  \[
    \|\theta_u^\delta(T)\|_{\pi^s \rightarrow \pi^{s-m}}
    \leq C_{u,s},
  \]
  where $u \mapsto \theta_u^\delta$
  denotes the linear map
  given for $u = x_1 \dotsb x_n$ ($x_1,\dotsc,x_n \in
  \mathfrak{g}$)
  by
  $\theta_u^\delta := \h^{n \delta} \theta_u$,
  with $\theta^0_u = \theta_u$ as in
  \S\ref{sec:operator-spaces-defn-etc}.
  We extend the definition to $m = \pm \infty$
  by taking intersections or unions.
\end{definition*}
For example, in the (most important) special case $\delta =0$,
the space $\Psi^m_0$ consists of $\h$-dependent elements of
$\Psi^m$ whose seminorms are uniformly bounded with respect to
$\h$.  In general, $\Psi^m_\delta$ consists of $\h$-dependent
elements of $\Psi^m$ whose seminorms vary with $\h$ in the
indicated manner, depending upon the exponent $\delta$.  The
results of \S\ref{sec:operator-classes-composition} remain valid
for $\Psi_\delta^m$ (with the same proof), while the results of
\S\ref{sec:sort-of-obvious-operator-class-memberships} hold for
$\Psi_0^m$, hence for any $\Psi_\delta^m$.

The
classes $\Psi^m_\delta$ enter into our applications only in a
crude way, in the estimation of remainder terms and the proofs
of \emph{a priori} bounds.  For such purposes, it is useful
to note that estimates involving $\Psi^m$ yield slightly
modified estimates for $\Psi^m_\delta$:
% We note
% For
\begin{lemma*}
  Let $\pi$ and $\delta$ be as in
  the definition above.
  \begin{enumerate}[(i)]
  \item   Fix $m \in \mathbb{Z}$.
  Let $\nu : \Psi^m \rightarrow \mathbb{R}_{\geq 0}$
  be an $\h$-dependent $\h$-uniformly continuous seminorm.
  Then
  for large enough fixed $M \geq 0$
  and all $T \in \Psi_\delta^m$,
  we have $\nu(T) \ll \h^{-M}$.
\item Fix $N \geq 0$.  Let
  $\ell : \Psi^{-N} \rightarrow \mathbb{C}$ be an $\h$-uniformly
  continuous linear map.  Then for some fixed $N' \geq N$, the
  restriction of $\ell$ to $\Psi^{-N'}$ induces a
  continuous map
  $\ell : \Psi^{-N'}_\delta \rightarrow \mathbb{C}$.
  \end{enumerate}
\end{lemma*}
\begin{proof}
  (i): By definition, we
  may find a finite family of pairs $(u,s)$ so that for each
  $T \in \Psi^m$, the quantity $\nu(T)$ is bounded by a constant
  multiple of $\|\theta_u(T)\|_{\pi^s \rightarrow \pi^{s-m}}$
  for some such pair.  We may assume that $u = x_1 \dotsb x_n$
  with $x_j \in \mathfrak{g}$.  The conclusion then holds for
  any $M$ larger than the maximum value of $n \delta$.

  (ii): We may bound $|\ell|$ in terms of finitely many pairs
  $(u,s)$ as in the proof of (i), with $u = x_1 \dotsb x_n$.
  We take for $N' - N$ the maximum value of $n$
  and use that each $\pi(x_j) \in \Psi^1$.
\end{proof}

For an $\h$-dependent positive scalar $c = c(\h)$, we denote as
in \S\ref{sec:asympt-param-h} by $c \Psi^m_\delta$ the image of
$\Psi^m_\delta$ under the map $T \mapsto c T = c(\h) T(\h)$.

We
define
\[
  \h^\infty \Psi^m_\delta := \cap_{\eta} \h^\eta \Psi^m
\]
and topologize it as in \S\ref{sec:asympt-param-h}.  This space
is independent of $\delta$; for instance, for finite $m$, it
consists of $\h$-dependent operators $T$ on $\pi$ that induce
continuous maps $T : \pi^s \rightarrow \pi^{s-m}$ having
operator norms $\O(\h^{N})$ for all fixed $s$ and $N$.  We often
drop the subscript $\delta$ and write simply $\h^\infty \Psi^m$.

% For an $\h$-dependent positive scalar $c$,
% we denote
% as in \S\ref{sec:asympt-param-h}
% by
% $c \Psi^m$ the space of $\h$-dependent operators
% $T \in \Psi^m = \Psi^m(\pi(\h))$
% for which each $\Psi^m$-seminorm of
% $c^{-1} T$ is bounded uniformly in $\h$.
% We define the space
% \[
%   \h^\infty \Psi^m := \cap_{\eta} \h^{\eta} \Psi^m
% \]
% and topologize it
% as in \S\ref{sec:asympt-param-h}.
% In words, this space consists of $\h$-dependent operators
% $T$ on $\pi$
% that induce continuous maps
% $T : \pi^s \rightarrow \pi^{s-m}$
% having operator norms $\O(\h^{N})$
% for all fixed $s$ and $N$.

\subsection{Variation with respect to the cutoff}
\label{sec:variation-of-op-wrt-chi}
Our operator assignment
is not particularly
sensitive to the choice of cutoff:
\begin{lemma*}
  Fix $\chi_1, \chi_2 \in \mathcal{X}(\mathcal{G})$.
  \begin{enumerate}[(i)]
  \item For $a \in S^\infty$,
    \begin{equation}\label{eqn:chi-insensitive-1}
      \Opp(a,\chi_1) \equiv \Opp(a,\chi_2) \mod{\Psi^{-\infty}},
    \end{equation}
    with remainder $R :=
    \Opp(a,\chi_1) - \Opp(a,\chi_2) \in \Psi^{-\infty}$
    varying continuously with $a$.
  \item More generally, for $a \in S^\infty_{\delta}$,
    \begin{equation}\label{eqn:chi-insensitive-2}
      \Opp_{\h}(a,\chi_1) \equiv \Opp_{\h}(a,\chi_2)
      \mod{\h^\infty \Psi^{-\infty}},
    \end{equation}
    with continuously-varying remainder.
  \end{enumerate}
\end{lemma*}
The latter continuity means explicitly
that for all $M,N \geq 0$, we may write
$\Opp_{\h}(a,\chi_1) - \Opp_{\h}(a,\chi_2) = \h^N R$, with
$R \in \Psi^{-M}$, and the induced map $a \mapsto R$ is
continuous, uniformly in $\h$.

The proof,
given in \S\ref{sec:variation-of-op-wrt-chi-pf},
amounts to noting
that the Fourier transforms of our symbols
are represented away
from the origin in $\mathfrak{g}$ by smooth
functions of rapid decay.
The singularity at the origin
is related to the order of the symbol.
These observations are
the analogue,
in our setup,
of the fact that kernels
of pseudodifferential operators
are smooth away from the diagonal.

We henceforth
fix $\chi, \chi'$
as in \S\ref{sec:comp-prelim}
and abbreviate
\[
\Opp(a) := \Opp(a:\pi) := \Opp(a,\chi) = \Opp(a,\chi:\pi),
\]
and similarly for $\Opp_{\h}$.

\subsection{Equivariance}\label{sec:equivariance}
By combining \S\ref{sec:equivariance-prelim}
and \S\ref{sec:variation-of-op-wrt-chi},
we see
that $\Opp$ is nearly $G$-equivariant.
Indeed,
for each fixed group element $g \in G$,
we may find a cutoff $\chi_1 \in \mathcal{X}(G)$
so that $g^{-1} \cdot \chi_1 \in \mathcal{X}(G)$;
then, with all congruences taken modulo $\Psi^{-\infty}$,
\[
  \Opp(g \cdot a)
  :=
  \Opp(g \cdot a, \chi)
  \equiv 
  \Opp(g \cdot a, \chi_1)
  =
  \Opp(a, g^{-1} \cdot \chi_1)
  \equiv
  \Opp(a).
\]
This estimate remains valid for $g$ in any fixed compact
subset of $G$ modulo the center.
A finer assertion holds for the $\h$-dependent classes:
\begin{lemma*}
  Fix $\delta \in [0,1)$ and $\eps > 0$.
  Let $g \in G$
  be an $\h$-dependent group element
  satisfying
  \begin{equation}\label{eqn:adjoint-bound-for-G-equivariance}
    \|\Ad(g)\| \ll \h^{-1 + \delta  + \eps}.
  \end{equation}
  Then for $a \in S^\infty_\delta$,
  the $\h$-dependent symbol $g \cdot a$ satisfies
  \[
    \Opp_{\h}(g \cdot a)
    \equiv \pi(g) \Opp_{\h}(a) \pi(g)^{-1}
    \mod{ \h^\infty \Psi^{-\infty} }.
  \]
\end{lemma*}
The proof is given in \S\ref{sec:equivariance-pf}.

\subsection{Operator class memberships}\label{sec:oper-class-memb-0}
The symbol and operator classes
have been defined so as to interact nicely:
\begin{theorem}\label{thm:main-properties-opp-symbols}
  For any $m \in \mathbb{Z} \cup \{\pm \infty \}$,
  we have
  \[
    \Opp(S^m) \subseteq \Psi^m,
  \]
  and the induced map is
  continuous.  In particular,
  elements of
  $\Opp(S^\infty)$
  act on $\pi^\infty$,
  and so may be composed.  Their
  compositions satisfy
  $\Opp(S^{m_1}) \Opp(S^{m_2}) \subseteq \Opp(S^{m_1+m_2}) +
  \Psi^{-\infty}$; more precisely,
  $\Opp(a) \Opp(b) = \Opp(a \star b, \chi ')
  \equiv \Opp(a \star b) \mod{\Psi^{-\infty}}$,
  with continuously-varying remainder.
\end{theorem}
The proof is given in \S\ref{sec-1-6-4}.
By combining with the asymptotic expansion
of the star product (theorem \ref{thm:star-prod-basic}),
we obtain:
\begin{corollary}
  Fix $m_1,m_2 < \infty$ and $J \in \mathbb{Z}_{\geq 0}$.  Then
  for $a \in S^{m_1}$ and $b \in S^{m_2}$, we have
  \[
    \Opp(a) \Opp(b) \equiv \sum_{0 \leq j < J} \Opp(a \star^j b)
    \mod{\Psi^{m_1+m_2-J}},
  \]
  with continuously-varying remainder.
\end{corollary}

We also have the rescaled analogues:
\begin{theorem}\label{thm:rescaled-operator-memb}
  Fix $\delta \in [0,1/2)$.  For $m \in \mathbb{Z}$, we
  have
  \[
    \Opp_{\h}(S^m_{\delta}) \subseteq \h^{\min(0,m)}
    \Psi_\delta^m.
    % \quad
    % \Opp_{\h}(S^{-m}_{\delta}) \subseteq \h^{-m}
    % \Psi_\delta^{-m}.
  \]
  For $a,b \in S^{\infty}_{\delta}$, we have
  \[
    \Opp_{\h}(a) \Opp_{\h}(b) = \Opp_{\h}(a \star_{\h} b, \chi ')
    \equiv \Opp_{\h}(a \star_{\h} b) \mod{\h^\infty
      \Psi^{-\infty}}.
  \]
\end{theorem}
The proof is given in \S\ref{sec:oper-class-memb-rescl}.
The following consequence
will be very useful:
\begin{corollary}\label{cor:epic-comp-law}
  Let $m_1,m_2 < \infty$ and $\delta \in [0,1/2)$.
  For $M,N \in \mathbb{Z}_{\geq 0}$
  there exists $J \in \mathbb{Z}_{\geq 0}$,
  depending only upon $(m_1,m_2,M,N,\delta)$,
  so that for all
  $a \in S^{m_1}_\delta, b \in
  S^{m_2}_\delta$,
  \begin{equation}\label{eqn:comp-with-remainder-J}
    \Opp_{\h}(a) \Opp_{\h}(b)
    \equiv 
    \sum_{0 \leq j < J}
    \h^j \Opp_{\h}(a \star^j b)
    \mod{\h^N \Psi_\delta^{-M}}.
  \end{equation}
\end{corollary}
\begin{proof}
  By theorems \ref{thm:star-prod-basic}
  and \ref{thm:rescaled-operator-memb},
  the two sides of \eqref{eqn:comp-with-remainder-J}
  agree modulo
  \[
    \mathcal{E} := 
    \h^\infty \Psi^\infty  +
    \Opp_{\h}(\h^{(1- 2\delta) J} S^{m_1+m_2-J}_{\delta}).
  \]
  We choose $M',N' \in \mathbb{Z}_{\geq 0}$ with
  $m_1 + m_2 - M' \leq -M$ and
  $(1 - 2 \delta) M' \geq M$ and
  $(1 - 2 \delta) N' \geq N$.
  We take
  $J := M' + N'$.
  Then $(1 - 2 \delta) J \geq M + N$
  and $m_1 + m_2 - J \leq -M - N' \leq - M$,
  so
  \[
    \h^{(1- 2\delta) J} S^{m_1+m_2-J}_{\delta}
    \subseteq 
    \h^{M+N} S^{-M}_{\delta}.
  \]
  By another application of theorem
  \ref{thm:rescaled-operator-memb},
  we conclude that
  $\mathcal{E} \subseteq \h^N \Psi_\delta^{-M}$.
\end{proof}
 A very useful generalization, involving proper subgroups
$G_1, G_2$ of $G$,
will be given in \S\ref{sec:gener-prop-subsp-2}.

\section{The Kirillov
  formula\label{sec:kirillov-frmula-baby}}
 
Our discussion thus far has been quite general: $\pi$ was any
unitary representation of a unimodular Lie group $G$.
Conversely,
the only control
we have established over the operators
we have constructed on $\pi$
is through their
operator norms.
We now consider a more restrictive situation
and derive correspondingly stronger control.

Let $\mathbf{G}$ be a reductive algebraic group
over an archimedean local field $F$.
(Our discussion applies
somewhat more broadly, e.g., to nilpotent groups,
but we focus on the case relevant for our applications.) 
By restriction of scalars, we may suppose $F = \mathbb{R}$.
Per our general conventions, $G$ denotes
the group of real points of $\mathbf{G}$.

Let $\pi$ be a \emph{tempered irreducible} unitary
representation of $G$.  We can
then apply the character formula for $\pi$ (\S\ref{ss:
  Kiriillov}), expressed in Kirillov form in terms of coadjoint
orbits (\S\ref{sec:canonical-symplectic-form}), to study the
traces and trace norms of the operators constructed via our
calculus.  

\subsection{Coadjoint
  orbits\label{sec:canonical-symplectic-form}}
The survey article \cite{MR1701415} is a useful
reference for the following discussion.
A \emph{coadjoint orbit} $\mathcal{O}$ is an orbit\index{coadjoint orbit}
of $G$ on $\mathfrak{g}^\wedge \cong i \mathfrak{g}^*$.
Being an orbit of a Lie group, it is a smooth manifold.

Each such orbit carries moreover a canonical $G$-invariant
symplectic structure $\sigma := \sigma_\mathcal{O}$, given at
each $\xi \in \mathcal{O}$ by the alternating form
\[
  \sigma_\xi(\ad_x^* \xi, \ad_y^* \xi) := [x,y]\xi/i
\]
on the tangent space
$T_\xi(\mathcal{O}) = \{ \ad_x^* \xi :  x \in \mathfrak{g} \}$.

In particular, $\mathcal{O}$ is even-dimensional.
We denote by
\index{$d$}
$d := d(\mathcal{O}) := (1/2) \dim_\mathbb{R}(\mathcal{O})$ half
the real dimension of $\mathcal{O}$; it is an integer with
$2 d \leq \dim \mathfrak{g}$.

The $d$-fold wedge product
$\sigma^{d}$ defines a volume form on $\mathcal{O}$.  We refer
to the measure induced by
the volume form\index{symplectic measure}
\index{volume form, symplectic}
\[
\omega := \omega_{\mathcal{O}} := \frac{1}{d!}  (\frac{\sigma
}{2 \pi })^d
\]
as the \emph{normalized symplectic measure} on $\mathcal{O}$.
If we choose local
coordinates $x_i,\xi_j$ on $\mathcal{O}$
($1 \leq i,j \leq d$)
with respect to which $\sigma = \sum_j d x_j \wedge d \xi_j$,
then $\omega = \prod_j (d x_j \wedge \frac{d \xi_j}{2 \pi})$.
 Integration with respect to $\omega$ defines a
measure on $\mathfrak{g}^\wedge$
\cite{MR0320232}.
One verifies readily
that
these measures enjoy the homogeneity property: if $t \in \mathbb{R}^\times_+$
and $f \in C^{\infty}_c(\mathfrak{g})$, then
\begin{equation} \label{homogeneity of symplectic measure}
  \int_{x \in \mathcal{O}} f (x) d\omega_{\mathcal{O}}(x) =
  t^{-d(\mathcal{O})} \int_{x \in t
    \mathcal{O}}
  f(t^{-1} x) d\omega_{t \mathcal{O}}(x).
\end{equation}

A coadjoint orbit is called \emph{regular}
if it consists of regular elements,
or equivalently,
has maximal dimension among all coadjoint orbits.
 
To each coadjoint orbit
$\mathcal{O}$ 
we may assign an \emph{infinitesimal character}
$[\mathcal{O}]$
in the GIT quotient $[\mathfrak{g}^\wedge]$ of $\mathfrak{g}^\wedge$
by $G$;
we recall the details
below in \S\ref{sec:infin-char}, \S\ref{sec:kirillov-frmula}.

By a \emph{coadjoint multiorbit},
we will mean a
\index{coadjoint multiorbit}
finite union
of coadjoint orbits
sharing the same infinitesimal character.
A \emph{regular coadjoint multiorbit}
is a coadjoint multiorbit whose elements are regular.
% we then define \emph{regular coadjoint multiorbits}
% in the evident way.
The notation and terminology introduced above carries over
with minor modifications;
for instance,
given
a nonempty coadjoint multiorbit,
we may speak of its infinitesimal
character or its normalized symplectic measure.
We note that the number of coadjoint orbits
having a given infinitesimal character
is bounded by a constant depending only upon $G$
(see \cite[Thm 3, Rmk 16]{MR0158024} and \cite[\S3]{MR0095844}),
so a coadjoint multiorbit
consists of a uniformly bounded number of orbits.

\subsection{The Kirillov character formula} \label{ss: Kiriillov}
Harish--Chandra
showed
(see, e.g.,  \cite[\S X]{MR855239})
that there is a locally $L^1$-function
$\chi_\pi : G \rightarrow \mathbb{C}$,
the \emph{character} of $\pi$,
with the following property.
Having fixed a Haar measure $d g$ on $G$,
we may associate
to each $f \in C_c^\infty(G)$
a
smooth compactly-supported
measure $f(g) \, d g$
on $G$
and an operator $\pi(f)
= \int_{g \in G} \pi(g) f(g) \, d g$.
Then $\pi(f)$ is trace class,
with trace given by
$\chi_\pi(f)
=  \int_{g \in G} \chi_\pi(g) f(g) \, d g$.

The facts recalled thus far apply to any irreducible admissible
representation.
Recall now that $\pi$ is assumed tempered.
A fundamental theorem of Rossmann \cite{MR508985, MR650379, MR587333}
 gives the validity of 
the Kirillov formula for $\pi$,
i.e.,  gives an exact formula for
the ``Fourier transform'' of $\chi_\pi$
in a neighborhood of the identity element.
Recall from \S\ref{sec:measures-et-al-G-g-g-star}
the definition of $j$.
\begin{theorem}
  There is a unique nonempty regular
  \index{coadjoint orbit $\mathcal{O}_\pi$}
  coadjoint multiorbit $\mathcal{O}_{\pi} \subseteq
  \mathfrak{g}^\wedge$
  so that for all $x$ in some fixed neighborhood
  of the origin in $\mathfrak{g}$,
  we have the identity of distributions
  \begin{align}\label{eqn:kir-form-defn}
    \chi_{\pi}(e^x) \sqrt{j(x)}
    &= \mbox{Fourier transform of  
      normalized symplectic measure on $ \mathcal{O}_{\pi}$}
    \\ \nonumber
    &:=   \int_{\xi \in \mathcal{O}_\pi}
      e^{x \xi } \, d \omega_{\mathcal{O}_{\pi}}(\xi).
  \end{align}
  Moreover:
  \begin{enumerate}[(i)]
  \item
    The infinitesimal character
    of $\mathcal{O}_\pi$
    is the infinitesimal character of $\pi$
    (cf. \S\ref{sec:infin-char}).
  \item If
    the infinitesimal character of $\pi$ is regular, then
    $\mathcal{O}_{\pi}$
    is a coadjoint orbit.
  \end{enumerate}
\end{theorem}

We emphasize the following:
\begin{itemize}
\item The statement of \eqref{eqn:kir-form-defn} is
  independent of choices of Haar measure.  
\item In general, $G$ may have several orbits with the same
  infinitesimal character as $\pi$.  It is remarkable that only
  a single one contributes in the case of regular infinitesimal
  character.
  (We do not directly use this fact in this paper.)
\item We have defined
  $\mathcal{O}_{\pi}$ only for tempered $\pi$.
\end{itemize}

We henceforth
denote by
$d$ the maximal dimension of any coadjoint orbit, so that
for each $\pi$  as above,
every orbit in $\mathcal{O}_\pi$ is $2 d$-dimensional.

\begin{remark*}
  In the $p$-adic case, the Harish-Chandra/Howe local character
  expansion gives a result of similar spirit but less
  precise. It describes the Fourier transform of the character,
  but only in a neighbourhood of the identity that depends on
  the representation $\pi$. As a result, it detects only the
  geometry of the coadjoint orbit ``at infinity.''
\end{remark*}

\subsection{Trace estimates}
\label{sec:trace-estimates-i}
The Kirillov formula
implies that
\begin{equation}\label{eq:assumed-kirillov-formula}
  \trace(\Opp(a))
  =
  \int_{\mathcal{O}_{\pi}}
  (j^{-1/2} \chi a^\vee)^\wedge \, d \omega_{\mathcal{O} _\pi}
\end{equation}
for all $a \in S^{-\infty}(\mathfrak{g}^\wedge)$.
By simple estimates
(see \S\ref{sec:smooth-away-from-origin}, \S\ref{sec:plancherel-formula}),
this conclusion
extends continuously to $a \in S^{-N}$ for large enough
$N \in \mathbb{Z}_{\geq 0}$.
By appeal to the homogeneity property
\eqref{homogeneity of symplectic measure},
we see more generally that
\begin{equation}\label{eq:assumed-kirillov-formula-rescaled}
  \trace(\Opp_{\h}(a))
  =
  \h^{-d}
  \int_{\h \mathcal{O}_{\pi}}
  (j_{\h}^{-1/2} \chi_{\h} a^\vee)^\wedge \, d \omega_{\h \mathcal{O} _\pi},
\end{equation}
where $j_{\h}(x) := j(\h x)$
and $\chi_{\h}(x) := \chi(\h x)$.
Using these formulas, we will establish in
\S\ref{sec:appl-kir-type-formulas} some refined and generalized
forms of the following:
\begin{theorem}\label{thm:trace-estimates-i}
  Fix $N$ sufficiently large in terms of the reductive
  Lie group $G$.
    \begin{enumerate}[(i)]
    \item
      Let $a \in S^{-N}$,
      let $\pi$ be a tempered irreducible unitary representation
      of $G$, and let $\h \in (0,1)$.
      Then
      \begin{equation}\label{eqn:baby-limit-kirillov}
        \trace(\Opp_{\h}(a)) = \h^{-d}
        (\int_{ \h \mathcal{O}_\pi} a \, d \omega_{\h
          \mathcal{O}_\pi}
        + \O(\h)).
      \end{equation}
      The implied constant is independent of $\pi$ and $\h$, and
      may be taken to depend continuously upon $a$.      
    \item Let $\pi$ be a tempered irreducible unitary
      representation of $G$.
      Any
      $T \in \Psi^{-N} := \Psi^{-N}(\pi)$ (see
      \S\ref{sec:operator-spaces-defn-etc})
      defines a trace class operator on $\pi$.
      The trace norm depends continuously upon $T$, uniformly in
      $\pi$ (in the sense of \S\ref{sec:prim-topol-vect}).
  \item Let $a \in S^{-N}_{\delta}$
    for some fixed $\delta \in [0,1/2)$.
    Let $\pi$ be an $\h$-dependent tempered irreducible
    unitary representation of $G$.
    Then the trace norms of
    $\h^d \Opp_{\h}(a)$ are bounded, uniformly in $\pi$ and
    $\h$, and continuously with respect to $a$.
  \end{enumerate}
  %   Let $\h \in (0,1]$ be a positive parameter,
  % and let $\pi$ be an $\h$-dependent tempered irreducible unitary
  % representation of the reductive Lie group $G$.

  % \begin{enumerate}[(i)]
  % \item
  %   For $a \in S^{-N}$,
  %   \begin{equation}\label{eqn:baby-limit-kirillov}
  %         \trace(\Opp_{\h}(a)) = \h^{-d}
  %   (\int_{ \h \mathcal{O}_\pi} a \, d \omega_{\h
  %     \mathcal{O}_\pi}
  %   + \O(\h)).
  % \end{equation}
  %   The implied constant
  %   is
  %   uniform in $\pi$ and $\h$,
  %   and
  %   may be taken to depend
  %   continuously upon $a$.
  % \item
  %   Any $T \in \Psi^{-N}$
  %   defines a trace class operator on $\pi$;
  %   the trace norm depends continuously upon $T$,
  %   uniformly in $\pi$.
  % \item Let $a \in S^{-N}_{\delta}$
  %   for some fixed $\delta \in [0,1/2)$.  Then the trace norms of
  %   $\h^d \Opp_{\h}(a)$ are bounded, uniformly in $\pi$ and
  %   $\h$, and continuously with respect to $a$.
  % \end{enumerate}
\end{theorem}
We note that part (i) follows readily from
\eqref{eq:assumed-kirillov-formula}, a Taylor expansion of
$j_{\h}, \chi_{\h}$, and some \emph{a priori} bounds for
integrals over coadjoint orbits,
while
part (ii) follows from the
uniform trace class property of $\Delta^{-N}$.
Part (iii)
is established using the symbol/operator calculi.

\begin{remark}\label{rmk:h-dependent-explication}
  To illustrate the content of our ``$\h$-dependent'' notation,
  we record an equivalent formulation of part (iii) of Theorem
  \ref{thm:trace-estimates-i}.
  Let
  $N \in \mathbb{Z}_{\geq 0}$ be large enough in terms of
  $G$.  Let $\delta \in [0,1/2)$.  There exist $C \geq 0$ and
  $J \in \mathbb{Z}_{\geq 0}$ so that for each tempered
  irreducible unitary representation $\pi$ of $G$, each symbol
  $a \in S^{-N}$ and each scaling parameter $\h \in (0,1]$, the
  trace norm of the operator $\h^d \Opp_{\h}(a)$ on $\pi$ is
  bounded by $C \sum_{|\alpha| \leq J} \nu_{\alpha,\h}(a)$, where
  \[
    \nu_{\alpha,\h}(a)
    :=
    \sup_{\xi \in \mathfrak{g}^\wedge}
    \frac{
      |\partial^\alpha a(\xi)|
    }{
      \h^{-\delta |\alpha|}
      \langle \xi \rangle^{-N-|\alpha|}
    }
  \]
  denotes the infimum of all scalars $C_\alpha \geq 0$
  for which the specialization to $m := -N$ of the inequality
  \eqref{eqn:defining-inequality-for-S-m-h-delta} holds for all
  $\xi \in \mathfrak{g}^\wedge$.
\end{remark}

\begin{remark}
  In our applications of theorem \ref{thm:trace-estimates-i}, it
is important that $\pi$ and $\h$ may vary simultaneously, but
our calculus is interesting only if the rescaled orbit
$\h \mathcal{O}_\pi$ does not ``escape to $\infty$'' as
$\h \rightarrow 0$.
In many examples
of interest,
it happens that
$(\h \mathcal{O}_{\pi}, d \omega_{\h \mathcal{O}_{\pi}})$
converges
to some ``limit orbit''
$(\mathcal{O}, d \omega)$ (\S\ref{sec:limit-coadjoint});
studying the limiting behavior of the calculus
will be a major concern of the later parts of the paper
(\S\ref{sec:limit-stat-attach-symb},
\S\ref{sec:appl-meas-class}).
A special case
relevant (but not sufficient) for our aims is when $\pi$ is
\emph{independent} of $\h$ and generic; the limit orbit
$\mathcal{O}$ is then contained
in the regular subset of the nilcone
(\S\ref{sec:limit-coadj-orbits-h-indep}).
\end{remark}

\begin{remark}
\label{rmk:planck-after-kirillov}
We have noted already (in \S\ref{sec:star-prod-asympt}) that the
$\h^{1/2}$ scale is a natural limit to our calculus.  Using
\eqref{eqn:baby-limit-kirillov}, this may now be understood as
follows: If the symbol $a$ is a smooth approximation to the
characteristic of a ball, with origin some regular element
$\xi \in \h \mathcal{O}_\pi$ and with radius $C \h^{1/2}$, then
\S\ref{sec:self-adj-of-op-schw} and
\eqref{eqn:comp-with-remainder-J} suggest that
$\Opp_{\h}(a)$ should approximate a self-adjoint idempotent,
i.e., an orthogonal projector onto a subspace $V$ of $\pi$,
consisting of vectors
``microlocalized within $C \h^{-1/2}$
of $\h^{-1} \xi$''
(cf. \S\ref{sec:intro-op-calc}).
But
\eqref{eqn:baby-limit-kirillov} suggests that
$\dim(V) \asymp C^{2 d}$, which makes sense only if $C$ is not
too small.  Conversely,
our calculus
allows one to work with such symbols
provided that $C \gg \h^{-\eps}$ for fixed but arbitrarily small
$\eps > 0$, hence to construct and manipulate approximate
projectors onto subspaces of size $\h^{-o(1)}$; in other words,
to work with an approximate orthonormal basis of $\pi$
consisting of vectors microlocalized at elements
$\xi \in \h \mathcal{O}_\pi$.
\end{remark}

\iftoggle{cleanpart}
{
  \newpage
}
\part{Microlocal analysis on Lie group representations II: proofs and refinements}\label{part:micr-analys-lie2}

We now give proofs of results in Part
\ref{part:micr-analys-lie}
as well as certain refinements
that will be useful at localized points of the later treatment. 
The reader might wish
to skim or skip Part \ref{part:micr-analys-lie2}
on a first reading.

\section{Star product asymptotics}
\label{sec:star-prod-expn}
The aim of this section
is to prove theorem \ref{thm:star-prod-basic}
in a generalized form
(theorem \ref{thm:star-prod-asymp-general})
that will be very convenient
in applications.

\subsection{Setup}\label{sec:setup-star-prod-asymp}
Let $\mathfrak{g}_1, \mathfrak{g}_2$
be subalgebras of $\mathfrak{g}$.
We assume that they arise as the Lie algebras
of some unimodular Lie subgroups $G_1, G_2$ of $G$.
The most important example
is when
$\mathfrak{g}_1 = \mathfrak{g}_2 = \mathfrak{g}$.

We fix some sufficiently small even precompact
open neighborhoods of
$\mathcal{G} \subseteq \mathfrak{g}$ and
$\mathcal{G}_1 \subset \mathfrak{g}_1$ and $\mathcal{G}_2 \subset
\mathfrak{g}_2$
of the respective origins,
with $\mathcal{G}_1, \mathcal{G}_2$ small enough in terms of $\mathcal{G}$.
In particular,
the maps
$\ast, \{,\} : \mathcal{G}_1 \times \mathcal{G}_2 \rightarrow
\mathcal{G}$
given 
(as in
\S\ref{sec:comp-prelim},
\S\ref{sec:lead-homog-comp-star})
by $x \ast y := \log(\exp(x) \exp(y))$
and $\{x, y\} := x \ast y - x - y$
are defined and analytic.

We assume that
$\mathfrak{g}_1 + \mathfrak{g}_2 = \mathfrak{g}$,
or equivalently, that the multiplication map $G_1 \times G_2
\rightarrow G$
is submersive near the identity element.
The restriction map
\[
\mathfrak{g}^\wedge \ni \zeta \mapsto (\zeta_1,\zeta_2)
\in \mathfrak{g}_1^\wedge \times \mathfrak{g}_2^\wedge
\]
is then injective.

Convolution defines a continuous map
$C_c^\infty(\exp(\mathcal{G}_1))
\times
C_c^\infty(\exp(\mathcal{G}_2))
\rightarrow
C_c^\infty(\exp(\mathcal{G}))$.
By fixing cutoffs $\chi_j \in \mathcal{X}(\mathcal{G}_j)$ as in
\S\ref{sec:op-first}
and arguing as in \S\ref{sec:comp-prelim},
we may define a ``star product''
\[
\star : \mathcal{S}(\mathfrak{g}_1^\wedge)
\times \mathcal{S}(\mathfrak{g}_2^\wedge)
\rightarrow
\mathcal{S}(\mathfrak{g}^\wedge)
\]
by the formula
$a \star b :=
(a^\vee \chi_1 \star b^\vee \chi_2)^\wedge$;
it admits the integral representation
\[
  a \star b(\zeta) :=
  \int_{x,y}
  a^\vee(x)
  b^\vee(y)
  e^{x \zeta_1}
  e^{y \zeta_2}
  e^{ \{x , y\} \zeta }
  \chi_1(x)
  \chi_2(y).
\]
Here the integral
is over $(x,y) \in \mathfrak{g}_1 \times \mathfrak{g}_2$
with respect to some fixed Haar measures.

For a unitary representation $\pi$ of $G$,
we may define
$\Opp(a,\chi_1)$ and $\Opp(b,\chi_2)$ as in \S\ref{sec-1-2},
acting on $\pi$ via its restrictions to $G_1$ and $G_2$.
We then have
\[
  \Opp(a,\chi_1)
  \Opp(b,\chi_2)
  = 
  \Opp(a \star b,\chi')
\]
for any cutoff $\chi'$ on $\mathfrak{g}$ taking the value $1$ on
$\mathcal{G}$.

We wish to extend the domain of definition of $\star$
and understand its asymptotic behavior.
As in \S\ref{sec:BCHD},  we may expand into homogeneous components,
namely
\begin{equation}\label{eqn:taylor-expn-of-Omega}
  \Omega(x,y,\zeta) := e^{\{x,y\}\zeta}
  = \sum _{\substack{
      \alpha,\beta,\gamma: \\
      |\gamma| \leq \min(|\alpha|,|\beta|)
    }
  }
  c_{\alpha \beta \gamma}
  x^\alpha y^\beta \zeta^\gamma
  = \sum_{j \geq 0} \Omega_j(x,y,\zeta),
\end{equation}
where $\Omega_j$ denotes the contribution
from $|\alpha| + |\beta| - |\gamma| = j$.
This again suggests the formal asymptotic expansion
$a \star b \sim \sum_{j \geq 0} a \star^j b$,
where
\begin{align*}
  a \star^j b(\zeta)
  &:=
    \int_{x,y}
    a^\vee(x)
    b^\vee(y)
    e^{x \zeta_1}
    e^{y \zeta_2}
    \Omega_j(x,y,\zeta)
  \\
  &=
    \sum _{
    \substack{
    \alpha,\beta,\gamma: \\
  |\gamma| \leq \min(|\alpha|,|\beta|), \\
  |\alpha| + |\beta| - |\gamma| = j
  }
  }
  c_{\alpha \beta \gamma}
  \zeta^\gamma \partial^\alpha a(\zeta)
  \partial^\beta  b(\zeta).
\end{align*}

\subsection{When should the star product map symbols to symbols?}
\label{sec:when-should-star}
Let
$m_1, m_2 \in \mathbb{Z} \cup \{\pm\}$.
For $j=1,2$,
we introduce
the temporary
abbreviation $S^{m_j} := S^{m_j}(\mathfrak{g}_j^\wedge)$
and similarly
$S^{m_j}_{\delta} :=
S^{m_j}_{\delta}(\mathfrak{g}_j^\wedge)$.

When does $\star$ extend naturally to a continuous map with
domain $S^{m_1} \times S^{m_2}$
and codomain one of our symbol classes?
We might ask first
whether the finite-order differential operators $\star^j$
admit such an extension,
starting with the simplest case $j = 0$,
for which
\[
a \star^0 b(\zeta) = a(\zeta_1)
b(\zeta_2).
\]
\begin{example*}
  Suppose 
  $\mathfrak{g}_2 \neq \mathfrak{g}$.  Take for
  $a \in S^0(\mathfrak{g}_1^\wedge)$ a constant symbol and for
  $b \in S^{-\infty}(\mathfrak{g}_2^\wedge)$ a
  compactly-supported symbol, both nonzero.
  One can verify then that
  $a \star^0  b \in C^\infty(\mathfrak{g}^\wedge)$
  does not belong to $S^\infty(\mathfrak{g}^\wedge)$.
\end{example*}
\begin{definition}
  We say that the pair $(m_1, m_2)$
  is \emph{admissible}
  (relative to  $\mathfrak{g}_1,\mathfrak{g}_2 \subseteq
  \mathfrak{g}$)
  if
  both of the following implications are satisfied:
  \begin{itemize}
  \item If $\mathfrak{g}_1 \neq \mathfrak{g}$,
    then $m_2 = -\infty$.
  \item If $\mathfrak{g}_2 \neq \mathfrak{g}$,
    then $m_1 = -\infty$.
  \end{itemize}
\end{definition}
For instance, if $\mathfrak{g}_1 = \mathfrak{g}_2 = \mathfrak{g}$,
then any pair is admissible,
while if $\mathfrak{g}_1 \neq \mathfrak{g}$ and $\mathfrak{g}_2
\neq \mathfrak{g}$,
then $(-\infty,-\infty)$ is the only admissible pair.
One verifies readily that if
$(m_1,m_2)$ is admissible,
then
\[
\star^j : S^{m_1} \times S^{m_2}
\rightarrow S^{m_1 + m_2 - j}(\mathfrak{g}^\wedge)
\]
and
more generally
\[
\star^j : S^{m_1}_{\delta} \times S^{m_2}_{\delta}
\rightarrow \h^{-2 \delta j} S^{m_1 + m_2 - j}_{\delta} (\mathfrak{g}^\wedge)
\]
are defined and continuous.
In fact, the converse
holds as well:
\begin{lemma*}
  For $(m_1,m_2)$ as above,
  the following are equivalent:
  \begin{enumerate}[(i)]
  \item $a \star^0 b \in S^{\infty}(\mathfrak{g})$
    for all $(a,b) \in S^{m_1} \times S^{m_2}$.
  \item $a \star^j b \in S^{\infty}(\mathfrak{g})$
    for all $(a,b) \in S^{m_1} \times S^{m_2}$ and $j \in
    \mathbb{Z}_{\geq 0}$.
  \item The pair $(m_1,m_2)$ is admissible.
  \end{enumerate}
\end{lemma*}
We have included this lemma for motivational
purposes only;
the proof is left to the reader.

\subsection{Main result}
\label{sec-1-3}
 \begin{theorem}\label{thm:star-prod-asymp-general}
  The star product $\star$ extends
  uniquely to a compatible family of continuous maps
  $\star : S^{m_1}(\mathfrak{g}_1^\wedge) \times
  S^{m_2}(\mathfrak{g}_2^\wedge) \rightarrow S^{m_1 +
    m_2}(\mathfrak{g}^\wedge)$
  taken over the admissible pairs
  $(m_1,m_2)$.

  Fix $\delta_1, \delta_2 \in [0,1)$
  with $\delta_1 + \delta_2 < 1$,
  $J \in \mathbb{Z}_{\geq
    0}$ and an admissible pair $(m_1,m_2)$.
  Then for all
  $a \in S^{m_1}_{\delta_1}(\mathfrak{g}_1^\wedge)$
  and
  $b \in S^{m_2}_{\delta_2}(\mathfrak{g}_2^\wedge)$,
  we have
  the asymptotic expansion
  \begin{equation}\label{eq:}
    a \star_{\h} b \equiv \sum_{0 \leq j < J}
    \h^j a \star^j b
    \mod {
      \h^{(1 - \delta_1 - \delta_2) J}
      S^{m_1+m_2-J}_{\max(\delta_1,\delta_2)} (\mathfrak{g}^\wedge)
    },
  \end{equation}
  with continuously-varying remainder.
\end{theorem}

The proof is given in \S\ref{sec:star-prod-pf-red},
after some preliminaries.

\begin{remark*}
  In all applications
  of theorem \ref{thm:star-prod-asymp-general}
  except that in \S\ref{sph-chr-sec-3},
  we take
  $\delta_1 = \delta_2 \in [0,1/2)$.
\end{remark*}

\subsection{Taylor's theorem\label{sec:taylor}}
\label{sec-1-5}
\begin{lemma*}
  Fix $J \in \mathbb{Z}_{\geq 0}$
  and a multi-index $\delta \in \mathbb{Z}_{\geq 0}^{\dim(\mathfrak{g})}$.
  For $(x,y,\zeta) \in \mathcal{G}_1 \times \mathcal{G}_2 \times \mathfrak{g}^\wedge$,
  abbreviate
  \[
  \rho := \max(|x|,|y|,|x| \, |y| \, |\zeta|).
  \]
  Then
  \begin{equation}\label{eq:claimed-taylor-expn-Omega}
    \partial_\zeta^\delta \Omega(x,y,\zeta) -
    \sum_{0 \leq j < J}
    \partial_\zeta^\delta \Omega_j(x,y,\zeta)
    \ll
    \rho^{J},
  \end{equation}
  where the implied constant may depend upon $(J,\delta)$
  but not upon $(x,y,\zeta)$.
\end{lemma*}
\begin{proof}
  Recall \eqref{eqn:taylor-expn-of-Omega}.
  We note first that
  if
  $|\gamma| \leq \min(|\alpha|,|\beta|)$
  and $|\alpha| + |\beta| - |\gamma| = j$,
  then
  \[
  |x^{\alpha} y^{\beta} \zeta^{\gamma}|
  \leq
  |x|^{|\alpha|} |y|^{|\beta|} |\zeta|^{|\gamma|}
  \leq \rho^{j},
  \]
  as follows readily
  by induction on $|\gamma|$.
  Thus
  \begin{equation}\label{eq:bound-for-partial-zeta-monomial}
    \partial_\zeta^{\delta} x^{\alpha} y^{\beta} \zeta^{\gamma}
    \ll
    (|x| \, |y|)^{|\delta|}
    \rho^{j} \ll \rho^j.
  \end{equation}
  We observe next, by the analyticity of $\Omega$,
  that there is a constant
  $R >  0$
  so that
  $|c_{\alpha \beta \gamma}| \ll  R^{|\alpha| + |\beta| -
    |\gamma|}$.
  From this and \eqref{eq:bound-for-partial-zeta-monomial},
  we obtain
  \begin{equation}\label{eq:bound-for-Omega-j}
    \partial_\zeta^{\delta} \Omega_j(x,y,\zeta)
    \ll
    (1 + j)^{O(1)} (R \rho)^{j}
  \end{equation}
  for all $j \geq 0$.
  By \eqref{eq:bound-for-Omega-j} (applied with $j < J$)
  and the trivial estimate
  $\partial_\zeta^{\delta} \Omega(x,y,\zeta)
  \ll (|x| \, |y|)^{|\delta|} \ll 1$,
  we deduce the claim \eqref{eq:claimed-taylor-expn-Omega}
  in the special case that $\rho$ is bounded uniformly from below,
  say by $1/2 R$.
  In the remaining case $\rho \leq 1/2 R$,
  we deduce \eqref{eq:claimed-taylor-expn-Omega}
  by summing \eqref{eq:bound-for-Omega-j}
  over $j \geq J$.
\end{proof}
\subsection{Integration by parts\label{sec:IBP}}
\label{sec-1-6}

For $(\xi,\eta,\zeta) \in \mathfrak{g}_1^\wedge \times
\mathfrak{g}_2^\wedge
\times \mathfrak{g}^\wedge$, we set
\begin{equation}\label{eq:defn-F-via-Omega}
  F(\xi,\eta,\zeta)
  :=
  \int_{x,y} e^{x \xi + y \eta} \chi_1(x) \chi_2(y)
  \Omega(x,y,\zeta).
\end{equation}
Recall that $\mathcal{G}_1$ and $\mathcal{G}_2$
have been taken sufficiently small,
and that $\mathfrak{g} = \mathfrak{g}_1 + \mathfrak{g}_2$.
It will be convenient now to normalize
the norms $|.|$ on the dual spaces $\mathfrak{g}^\wedge,
\mathfrak{g}_1^\wedge,
\mathfrak{g}_2^\wedge$
to be Euclidean norms
with the property that
for $\zeta \in \mathfrak{g}^\wedge$,
\begin{equation}\label{eq:}
  |\zeta|^2 = |\zeta_1|^2 + |\zeta_2|^2.
\end{equation}
  
\begin{lemma}\label{lem:good-for-xi-or-eta-nearly-zeta}
  Fix $N \in \mathbb{Z}_{\geq 0}$ and a multi-index $\gamma$.
  Set $t := \sqrt{|\xi|^2 + |\eta|^2}$.
  Then
  \begin{equation}
    t \geq \tfrac{1}{2} |\zeta|
    \implies
    \partial_\zeta^{\gamma} F(\xi,\eta,\zeta) \ll t^{-N}.
  \end{equation}
  The implied
  constant may depend upon $(N,\gamma)$,
  but not upon $(\xi,\eta,\zeta)$.
\end{lemma}
The conclusion holds with $\tfrac{1}{2}$ replaced
by any fixed fraction,
provided that $\mathcal{G}_1, \mathcal{G}_2$ are taken
sufficiently small.
The basic idea is that the hypothesis $t \geq \tfrac{1}{2}
|\zeta|$
implies that the integral \eqref{eq:defn-F-via-Omega}
has no stationary point.
\begin{proof}
  We may write
  \[
  \partial_\zeta^\gamma 
  F(\xi,\eta,\zeta)
  =
  \int_{x,y}
  f(x,y)
  e^{t \phi(x,y)},
  \]
  where
  $f(x,y) := \{x,y\}^{\gamma} \chi_1(x) \chi_2(y)$
  and
  $\phi : \mathcal{G}_1 \times \mathcal{G}_2 \rightarrow i \mathbb{R}$
  is given by
  \[
  \phi(x,y) :=
  \frac{x \xi + y \eta }{t}
  +
  \{x,y\}
  \frac{\zeta}{t}.
  \]
  Since $\mathcal{G}_1, \mathcal{G}_2$ are 
  small
  and $|\zeta|/t \leq 2$,
  the total derivative
  $\partial \phi : \mathcal{G}_1 \times \mathcal{G}_2
  \rightarrow i \mathfrak{g}_1^* \times i \mathfrak{g}_2^*$
  approximates the unit vector $(\xi/t,\eta/t)$.
  In particular,
  the Euclidean norm $|\partial \phi|(x,y)$
  of the total derivative
  is bounded from below
  by (say) $1/2$
  for all $(x,y) \in \mathcal{G}_1 \times \mathcal{G}_2$.
  Moreover, $\phi$ lies in a fixed
  bounded subset of $C^\infty(\Omega)$.
  The required estimate
  follows by ``partial integration,''
  as summarized by the following lemma.
\end{proof}

\begin{lemma}
  Fix $n,N \in \mathbb{Z}_{\geq 0}$ and $\eps > 0$.
  Let $\Omega$ be an open subset of $\mathbb{R}^n$.
  Let $\phi : \Omega \rightarrow i \mathbb{R}$
  be smooth.
  Assume that the total derivative $\partial \phi : \Omega \rightarrow i
  \mathbb{R}^n$
  has Euclidean norm $|\partial \phi | : \Omega \rightarrow
  \mathbb{R}_{\geq 0}$
  bounded from below by $\eps$.
  Then for all $f \in C_c^\infty(\Omega)$
  and $t > 0$,
  \[
  \int_{\mathbb{R}^n} f e^{t \phi}
  \ll
  t^{-N}
  \sum_{|\alpha| \leq N} \|\partial^\alpha f\|_{L^1},
  \]
  where the implied constant
  is independent of $(f,t)$
  and depends continuously upon $\phi \in C^\infty(\Omega)$.
\end{lemma}
\begin{proof}
  We may assume that $N$ is even,
  say $N = 2 r$.
  Let $\Delta$ denote the multiple of the standard Laplacian
  for which $e^{\phi} = |\partial \phi|^{-2} \Delta(e^\phi)$,
  and set $D := |\partial \phi|^{-2} \Delta$.
  Integrating by parts repeatedly,
  we obtain
  $I =
  t^{-N}
  \int
  f
  D^{r}(e^{t \phi})
  =
  t^{-N}
  \int
  D^{r} (f)
  e^{\phi}$,
  so that $|I| \leq t^{-N} \int |D^{r} f|$.
  Set $b := |\partial \phi|^{-2}$.
  We may expand
  \[
    D^r f
    =
    \sum 
    C(\alpha,\beta^{(1)},\dotsc,\beta^{(r)})
    \,
    (\partial^\alpha f)
    \, (\partial^{\beta^{(1)}} b)
    \,
    \dotsb
    \,(\partial^{\beta^{(r)}} b),
  \]
  with the sum taken over multi-indices
  $\alpha, \beta^{(1)}, \dotsc, \beta^{(r)}$
  satisfying
  $|\alpha| +  |\beta^{(1)}| + \dotsb  +  |\beta^{(r)}|
    = N$.
  By the quotient rule for derivatives,
  we have
  $\|\partial^\beta b\|_{L^\infty(\Omega)} \ll 1$.
  The required estimate
  follows.
\end{proof}

\subsection{Decomposition into localized
  symbols}\label{sec:localized-functions}
Let $m \in \mathbb{Z}$,
$\delta \in [0,1)$
and
$\omega \in \mathfrak{g}^\wedge$.
Observe that each $a \in S^m_{\delta}$
varies mildly over
the ball
\[
U_\omega := \{\xi \in \mathfrak{g}^\wedge
: |\xi-\omega| \leq \tfrac{1}{2} \h^\delta \langle \omega
\rangle\}.
\]
\begin{definition*}
  We say that the symbol
  $a \in S^m_{\delta}$ is \emph{localized at $\omega$} if it is
  supported on $U_\omega$.
\end{definition*}
Note that this terminology depends
implicitly upon $\delta$.
\begin{lemma}\label{lem:localized-fn-fourier}
  If $a \in S^m_{\delta}$ is localized
  at $\omega \in \mathfrak{g}^\wedge$,
  then
  \[
  a(\xi) = \langle \omega  \rangle^m
  \phi \left( \frac{\xi - \omega }{\h^\delta \langle \omega
      \rangle} \right),
  \]
  where $\phi \in C_c^\infty(\mathfrak{g}^\wedge)$
  depends continuously upon $a$.
  In particular:
  \begin{itemize}
  \item The rescaled Fourier transform
    of $a$ has the form
    \begin{equation}\label{eqn:}
      a_{\h}^\vee(x) = \langle \omega  \rangle^m e^{-x \omega/\h}
      A^{\dim(\mathfrak{g})} \phi^\vee(A x),
      \quad
      A := \h^{\delta-1} \langle \omega \rangle,
    \end{equation}
    where $\phi^\vee \in \mathcal{S}(\mathfrak{g})$
    depends continuously upon $a$.
  \item For fixed $n \in \mathbb{Z}_{\geq 0}$,
    \begin{equation}\label{eqn:fourier-bound-localized-integral-monomial}
      \int_{x \in \mathfrak{g}}
      \left\lvert a_{\h}^\vee(x) \right\rvert
      \, |x|^n
      \ll A^{-n},
    \end{equation}
    with continuous dependence upon $a$.
  \end{itemize}
\end{lemma}
\begin{proof}
  Each assertion follows readily
  from the definition of $S^m_{\delta}$.
\end{proof}

It is not difficult to decompose any symbol
into localized symbols.
To that end, the following partition of unity
is convenient:
\begin{lemma}\label{lem:ghetto-partition-of-unity}
  Fix $\delta \in (0,1]$.
  There is an $\h$-dependent countable
  collection
  $\Omega = \Omega_{\delta,\h} \subseteq \mathfrak{g}^\wedge$
  of points $\omega \in \mathfrak{g}^\wedge$
  with the following properties:
  \begin{enumerate}[(i)]
  \item The balls $U_\omega$,
    for $\omega \in \Omega$,
    cover $\mathfrak{g}^\wedge$.
  \item For $X \geq 1$,
    we have $\# \{\omega \in \Omega : |\omega| \leq X\}
    \ll \h^{-\O(1)} X^{\O(1)}$.
  \item $\sup_{\omega_1 \in \Omega} \# \{\omega_2 \in \Omega : U_{\omega_1} \cap
    U_{\omega_2} \neq \emptyset \} \ll 1$.
  \item
    We have
    \[
    \sum_{\omega \in \Omega}
    \phi_{\omega}
    \left( \frac{\xi - \omega }{\h^\delta \langle \omega
        \rangle}\right)
    = 1
    \text{ for all } \xi \in \mathfrak{g}^\wedge,
    \]
    where $\phi_\omega$
    belongs to a fixed
    bounded subset
    of $C_c^\infty(\mathfrak{g}^\wedge)$
    and is supported on $\{\xi  :  |\xi| \leq \tfrac{1}{2}\})$.
  \end{enumerate}
\end{lemma}
\begin{proof}
  We construct
  $\Omega$ and $\phi_\omega$,
  leaving the remaining verifications
  to the reader.
  Fix an element $q \in [2,3]$
  which is generic in a sense to be clarified below.
  Fix a dyadic partition
  of unity
  $1 = \psi_0(\xi) + \sum_{n \geq 1}
  \psi_1(q^{-n} \xi)$,
  where $\psi_0 \in C_c^\infty(\mathfrak{g}^\wedge)$
  and $\psi_1 \in C_c^\infty(\mathfrak{g}^\wedge - \{0\})$.
  For $n \geq 1$, write $\psi_n(\xi) := \psi_1(q^{-n} \xi)$.
  Fix a sufficiently dense lattice $L \subseteq \mathfrak{g}^\wedge$
  and an additive partition of unity
  $1 = \sum_{\ell \in L}
  \rho(\xi-\ell)$,
  with $\rho \in C_c^\infty(\mathfrak{g}^\wedge)$.
  Then, for $n \geq 0$,
  \begin{equation}\label{eq:ghetto-partition-of-unity}
    \psi_n(\xi)
    =
    \sum_{\ell \in L}
    \psi_n(\xi)
    \rho(\frac{\xi}{ \h^{\delta } q^{n} } - \ell).
  \end{equation}
  Take for $\Omega$ the set
  consisting of all
  $\omega = \h^{\delta} q^{n} \ell$
  for which the corresponding summand
  in \eqref{eq:ghetto-partition-of-unity}
  is nonzero;
  the genericity assumption
  on
  $q$ 
  implies that these
  elements are pairwise distinct
  as $n$ varies.
  Take for $\phi_\omega$ the corresponding summand.
\end{proof}

\begin{lemma}
  Each $a \in S^m_{\delta}$
  may be decomposed
  as $a = \sum_{\omega \in \Omega} \langle \omega  \rangle^m
  a_\omega$,
  with $\Omega$ as above,
  where
  $a_\omega \in S^0_{\delta}$
  is localized at $\omega$
  and depends continuously upon $a$.
\end{lemma}
\begin{proof}
  Take $a_\omega := \langle \omega  \rangle^{-m}
  a(\xi) \phi_\omega \left( \frac{\xi - \omega }{\h^\delta
      \langle \omega  \rangle} \right)$.
\end{proof}

\subsection{Proof
  of theorem  \ref{thm:star-prod-asymp-general}}
\label{sec:star-prod-pf-red}
The claimed uniqueness follows from the fact that
$C_c^\infty$
has dense image in $S^{-\infty}$
and also in $S^{\infty}$
(note that $C^{m} \subseteq S^{m}$,
the closure of the image of
$C_c^\infty$,
contains $S^{m'}$ whenever $m' < m$).
The existence
follows, via a limiting procedure, from the continuity
established below in the course of the proof of
the asymptotic expansion.
For the latter,
we may assume that $m_1, m_2 < \infty$.
It suffices to consider the following cases:
\begin{enumerate}[(a)]
\item \label{item:neg-neg}
  $m_1 = m_2 = - \infty$.
\item \label{item:neg-pos}
  $m_1 = -\infty$, $m_2 \in \mathbb{Z}$
  and $\mathfrak{g}_1 = \mathfrak{g}$.
\item \label{item:pos-neg}
  $m_1 \in \mathbb{Z}$,   $m_2 = -\infty$
  and $\mathfrak{g}_2 = \mathfrak{g}$.
\item \label{item:pos-pos}
  $m_1 \in \mathbb{Z}$,   $m_2 \in \mathbb{Z}$
  and $\mathfrak{g}_1 = \mathfrak{g}_2 = \mathfrak{g}$.
\end{enumerate}
Abbreviate $\delta := \max(\delta_1,\delta_2)$.
We must verify then
that
\[r := a \star_{\h} b - \sum_{0 \leq j < J}
\h^j a \star^j b\]
belongs to $\h^{(1-\delta_1-\delta_2) J}
S^{m_1+m_2-J}_{\delta}(\mathfrak{g}^\wedge)$,
i.e., that
for each fixed
multi-index $\gamma \in \mathbb{Z}_{\geq 0}^{\dim(G)}$,
\begin{equation}\label{eq:required-estimate-for-partial-gamma-r-zeta}
  \partial^{\gamma} r(\zeta)
  \ll \h^{(1 - \delta_1-\delta_2) J -\delta |\gamma|}
  \langle \zeta  \rangle^{m_1+m_2-J-|\gamma|},
\end{equation}
where the implied constant 
may depend upon $(m_1,m_2,J,\delta_1,\delta_2)$
and continuously upon $a$ and $b$,
but not upon $\h, \zeta$.
If either $m_1$ or $m_2$ is $-\infty$,
then the meaning of
\eqref{eq:required-estimate-for-partial-gamma-r-zeta}
is that
$\partial^{\gamma} r(\zeta)
\ll \h^{(1 - \delta_1 - \delta_2) J -\delta |\gamma|}
\langle \zeta  \rangle^{-N}$
holds for each fixed $N$.

In fact, since the spaces
$\h^{(1-\delta_1-\delta_2) j} S^{m+n-j}_{\delta}$ decrease as $j$
increases, the terms $\h^j a \star^j b$, for fixed $j \geq J$,
satisfy the analogue of the estimate
\eqref{eq:required-estimate-for-partial-gamma-r-zeta} required
by $r$.
The proof of theorem \ref{thm:star-prod-asymp-general}
thereby reduces to that of the following assertion:
for each fixed
$N \in \mathbb{Z}_{\geq 0}$
and
multi-index $\gamma$,
one has for large enough $J \in \mathbb{Z}_{\geq 0}$
that
\begin{equation}\label{eqn:required-estimate-partial-gamma-r-zeta-involving-N}
  \partial^\gamma r(\zeta)
  \ll \h^N \langle \zeta  \rangle^{-N}.
\end{equation}
If $m_k = -\infty$ for some $k=1,2$, then it will suffice to
show that
\eqref{eqn:required-estimate-partial-gamma-r-zeta-involving-N}
holds under the weaker assumption that $m_k$ is any fixed
(negative) integer taken
sufficiently small in terms of $N$.
In particular,
we may assume that $m_1,m_2 \in \mathbb{Z}$.

We may
decompose $a = \sum_{\omega_1 \in \Omega_1} \langle \omega
\rangle^{m_1} a_{\omega_1}$
and
$b = \sum_{\omega_2 \in \Omega_2} \langle \omega
\rangle^{m_2} b_{\omega_2}$
as in \S\ref{sec:localized-functions},
where $a_{\omega_1} \in S^0_{\delta_1}(\mathfrak{g}_1^\wedge)$
and
$b_{\omega_2} \in S^0_{\delta_2}(\mathfrak{g}_2^\wedge)$
are localized at $\omega_1, \omega_2$, respectively.
We may assume $N$ chosen large enough that
\[
\h^{N/3} \sum_{\omega_j \in \Omega_j}
\langle \omega_j \rangle^{m_j-N} \ll 1 \quad (j=1,2),
\]
say.
The proof
of
\eqref{eqn:required-estimate-partial-gamma-r-zeta-involving-N}
thereby reduces to that of the following:
\begin{proposition*}
  Fix $\delta_1,\delta_2 \in [0,1)$ with $\delta_1 + \delta_2 <
  1$,
  $N \in \mathbb{Z}_{\geq 0}$
  and a multi-index $\gamma \in \mathbb{Z}_{\geq 0}^{\dim(G)}$.
  Fix $J, M \in \mathbb{Z}_{\geq 0}$
  sufficiently large in terms of $N$ and $\gamma$.
  Let $a \in S^0_{\delta_1}(\mathfrak{g}_1^\wedge)$
  and $b \in S^0_{\delta_2}(\mathfrak{g}_2^\wedge)$
  be localized at $\omega_1 \in \mathfrak{g}_1^\wedge$
  and $\omega_2 \in \mathfrak{g}_2^\wedge$,
  respectively.
  Set $\delta := \max(\delta_1,\delta_2)$
  and
  $r := a \star_{\h} b - \sum_{0 \leq j < J}
  \h^j a \star^j b$.
  We then
  have
  the following estimates,
  in which implied
  constants may depend upon $(\delta_1,\delta_2,N,\gamma)$
  and continuously upon $a$ and $b$,
  but not upon $(\h,\zeta,\omega_1,\omega_2)$.
  \begin{enumerate}[(a)]
  \item
    $\partial^\gamma r(\zeta)
    \ll \h^{N} \langle \zeta  \rangle^{-N}
    \langle \omega_1 \rangle^{M} \langle \omega_2 \rangle^{M}$.
  \item
    If $\mathfrak{g}_1 = \mathfrak{g}$,
    then
    $\partial^\gamma r(\zeta)
    \ll \h^{N} \langle \zeta  \rangle^{-N}
    \langle \omega_1 \rangle^{M} \langle \omega_2 \rangle^{-N}$.
  \item
    If $\mathfrak{g}_2 = \mathfrak{g}$,
    then
    $\partial^\gamma r(\zeta)
    \ll \h^{N} \langle \zeta  \rangle^{-N}
    \langle \omega_1 \rangle^{-N} \langle \omega_2 \rangle^{M}$.
  \item
    If $\mathfrak{g}_1 = \mathfrak{g}_2 = \mathfrak{g}$,
    then
    $\partial^\gamma r(\zeta)
    \ll \h^{N} \langle \zeta  \rangle^{-N}
    \langle \omega_1 \rangle^{-N} \langle \omega_2 \rangle^{-N}$.
  \end{enumerate}
\end{proposition*}
\begin{proof}
  To simplify the presentation,
  we focus on the case
  $\gamma = 0$;
  the general case
  follows by the same arguments
  applied
  with $a,b,\Omega(x,y,\zeta)$ and $F(\xi,\eta,\zeta)$
  replaced by
  some fixed derivatives (with respect to $\zeta$,
  in the latter two cases);
  note that our inputs (\S\ref{sec:taylor}, \S\ref{sec:IBP})
  apply to such derivatives.
  
  We define $Q \in \mathbb{R}_{\geq 1}$
  in the various cases as follows:
  \begin{enumerate}[(a)]
  \item $Q := \h^{-1} \left\langle \zeta  \right\rangle$
  \item $Q := \h^{-1} \left\langle \zeta  \right\rangle \langle \omega_2 \rangle$
  \item $Q := \h^{-1} \left\langle \zeta  \right\rangle \langle \omega_1 \rangle$
  \item $Q := \h^{-1} \left\langle \zeta  \right\rangle \langle \omega_1 \rangle \langle \omega_2 \rangle$
  \end{enumerate}

  We note first that, for fixed $j$,
  the element $\h^j a \star^j b
  \in \h^{(1 - \delta_1-\delta_2) j} S^{- j}_{\delta}
  \subseteq S^0_{\delta}$
  is localized
  at both $\omega_1$ and $\omega_2$.
  Thus
  $\h^j a \star^j b(\zeta) = 0$
  unless $\langle \omega_1 \rangle \asymp \langle \omega_2
  \rangle \asymp \langle \zeta  \rangle$,
  in which case
  $\h^j a \star^j b(\zeta) \ll 1$.

  Next,
  set $A := \h^{-1+\delta_1} \left\langle \omega_1 \right\rangle$
  and $B := \h^{-1+\delta_2} \left\langle \omega_2 \right\rangle$,
  so that by \eqref{eqn:fourier-bound-localized-integral-monomial},
  we have
  \begin{equation}\label{eqn:a-b-fourier-A-B}
    \int_{x,y}
    |a^\vee_{\h}(x) b^\vee_{\h}(y)| \,
    |x|^m
    |y|^n
    \ll
    A^{-m}
    B^{-n}
  \end{equation}
  for fixed $m,n \geq 0$.  By specializing this estimate to the
  case $m = n = 0$, and recalling that
  $|\Omega(x,y,\zeta)| \ll 1$ and
  $\h^j a \star^j b(\zeta) \ll 1$ for fixed $j$, we deduce in
  particular that $r(\zeta) \ll 1$.  This gives an adequate
  estimate for $r(\zeta)$ in the special case $Q \ll 1$.
  We may
  and shall thus assume that $Q$ is sufficiently large.

  We now fix $\eps > 0$ small in terms of $\delta_1+\delta_2$,
  assume that $M$ is chosen large in terms of $(N,\eps)$,
  and treat the various cases separately:
  \begin{enumerate}[(a)]
  \item
    The required estimate
    is trivial unless $| \omega_1 | \leq Q^\eps$
    and $| \omega_2 | \leq Q^\eps$,
    as we henceforth assume.

    Suppose $|\zeta| \geq Q^{2 \eps}$.
    In that case,
    $\text{$a \star^j b(\zeta) = 0$ for all $j$,}$
    so it will suffice to show that
    $a \star_{\h} b(\zeta) \ll Q^{-N}$.
    To that end, recall the function
    $F : \mathfrak{g}_1^\wedge \times \mathfrak{g}_2^\wedge
    \times \mathfrak{g}^\wedge \rightarrow \mathbb{C}$ defined
    in \S\ref{sec:IBP};
    we have
    \begin{align}
      a \star_{\h} b(\zeta)          &=
                                       \label{eqn:a-star-h-b-2}
                                       \h^{-2 \dim(\mathfrak{g})} \int_{\xi,\eta}
                                       a(\xi)
                                       b(\eta)
                                       F (\frac{\zeta_1 - \xi}{\h}, \frac{\zeta_2 - \eta}{\h},
                                       \frac{\zeta}{\h} ).
    \end{align}
    For $\xi,\eta$ with $a(\xi) b(\eta) \neq 0$,
    we have
    $|\xi| \asymp \langle \omega_1 \rangle \ll Q^\eps$
    and
    $|\eta| \asymp \langle \omega_2 \rangle \ll Q^\eps$,
    while $|\zeta| = \sqrt{|\zeta_1|^2 + |\zeta_2|^2} \gg Q^{2
      \eps}$,
    hence
    \begin{equation}\label{eq:key-lower-bound-for-t-used-in-pf}
      t :=
      \sqrt{
        \left\lvert
          \frac{\zeta_1 - \xi}{\h}
        \right\rvert^2
        + 
        \left\lvert
          \frac{\zeta_2 - \eta}{\h}
        \right\rvert^2
      }
      \geq
      \frac{1}{2}
      \left\lvert \frac{\zeta}{\h} \right\rvert
    \end{equation}
    and $t \gg Q^{2 \eps}$.  The required estimate thus follows
    from \S\ref{sec:IBP}, together with the trivial estimate
    $\O(Q^{\O(1)})$ for the $L^1$-norms of $a$ and $b$.

    We have reduced to the case that
    $|\omega_1|, |\omega_2|, |\zeta| \leq Q^{2 \eps}$.
    We then verify readily,
    using that $\delta_1 + \delta_2 < 1$ and that $\eps$ is small enough
    in terms of $\delta_1+\delta_2$, that
    \begin{equation}\label{eq:key-inequality-A-B-zeta-h-A-B}
      \max (\frac{1}{A}, \frac{1}{B}, \frac{|\zeta|}{\h A B})
      \ll Q^{-\eps}.
    \end{equation}
    Informally,
    the key point here is that
    we have reduced to a case in which
    \begin{equation}\label{eq:}
      \langle \omega_1 \rangle \approx \langle \omega_2 \rangle
      \approx \langle \zeta  \rangle,
    \end{equation}
    so that
    \[
      \frac{|\zeta|}{\h A B}
      \ll
      \h^{1 - \delta_1-\delta_2}
      \frac{\langle \zeta  \rangle}{ \langle \omega_1 \rangle
        \langle \omega_2\rangle}
      \lessapprox
      \h^{1 - \delta_1 - \delta_2}
      \langle \zeta \rangle^{-1}
      \lessapprox
      Q^{-\eps}.
    \]
    
    We now split
    $r(\zeta) =
    r'(\zeta) + r''(\zeta)$,
    where
    \begin{equation}\label{eq:discard-trunc-from-r}
      r'(\zeta) := \int_{x,y}
      a_{\h}^\vee(x) b_{\h}^\vee(y)
      (\Omega(x,y,\frac{\zeta}{\h}) - \sum_{0 \leq j < J}
      \Omega_j(x,y,\frac{\zeta}{\h}))
    \end{equation}
    and
    \[
      r''(\zeta) :=
      \int_{x,y}
      a_{\h}^\vee(x) b_{\h}^\vee(y) (\chi_1(x) \chi_2(y) - 1)
      \Omega(x,y,\frac{\zeta}{\h}).
    \]
    Since $\Omega(x,y,\zeta) \ll 1$,
    we obtain using \eqref{eqn:a-b-fourier-A-B}
    the estimate
    $r''(\zeta) \ll (A B)^{-n}$ for any fixed $n$,
    which is adequate thanks to
    \eqref{eq:key-inequality-A-B-zeta-h-A-B}.
    We estimate $r'(\zeta)$
    using
    \S\ref{sec:taylor}, \eqref{eqn:a-b-fourier-A-B}
    and
    \eqref{eq:key-inequality-A-B-zeta-h-A-B},
    giving
    the adequate estimate
    $r'(\zeta) \ll Q^{-\eps J}$.
  \item
    We may assume that $| \omega_1 | \leq Q^\eps$,
    since the required estimate
    is otherwise trivial.
    We may assume that
    $| \zeta  | \leq Q^{2 \eps}$,
    since otherwise \eqref{eq:key-lower-bound-for-t-used-in-pf}
    holds with $t \gg Q^{2 \eps}$,
    and we may conclude as above;
    in particular, $|\zeta _2| \leq Q^{2 \eps}$.
    We may assume that $| \omega_2| \leq Q^{3 \eps}$,
    since otherwise
    \eqref{eq:key-lower-bound-for-t-used-in-pf}
    holds with $t \gg Q^{3 \eps}$.
    We then verify \eqref{eq:key-inequality-A-B-zeta-h-A-B}
    and conclude as before.
  \item By the same argument,
    but
    with the roles of $\omega_1$ and $\omega_2$ reversed.
  \item If for some $k=1,2$ we have either
    \begin{itemize}
    \item  $|\zeta| \leq Q^{\eps}$ and $|\omega_k| \geq Q^{2
        \eps}$, or
    \item $|\zeta| \geq Q^{\eps}$ and $|\omega_k|
      \notin [Q^{-\eps^2} |\zeta|, Q^{\eps^2} |\zeta|]$,
    \end{itemize}
    then
    \eqref{eq:key-lower-bound-for-t-used-in-pf}
    holds with $t \gg Q^{\eps^2}$,
    so we may conclude as above.
    In the remaining cases,
    we have either
    \begin{itemize}
    \item $|\zeta| \leq Q^\eps$ and
      $|\omega_1|, |\omega_2| \leq Q^{2 \eps}$, or
    \item $|\zeta| \geq Q^{\eps}$ and 
      $Q^{-\eps^2} |\zeta| \leq |\omega_1|, |\omega_2| \leq Q^{\eps^2} |\zeta|$.
    \end{itemize}
    In either case,
    we verify \eqref{eq:key-inequality-A-B-zeta-h-A-B}
    and conclude as before.
  \end{enumerate}
\end{proof}

\subsection{Asymptotic expansions for certain convolutions}
\label{sec:appl-tayl-theor}
Here we record a miscellaneous estimate
to be applied
occasionally.
Fix $\delta \in [0,1)$,
$m  < \infty$, 
and $\psi \in C_c^\infty(\mathfrak{g})$.
For
$a \in S^m_{\delta}$, we may define
$b : \mathfrak{g}^\wedge \rightarrow \mathbb{C}$
by requiring that
\[
  b_{\h}^\vee = \psi a_{\h}^\vee.
\]
It is the Fourier transform
of a compactly-supported distribution, hence is
smooth.
\begin{lemma*}
  $b \in S^m_{\delta}$.
  Moreover,
  for each fixed $J \in \mathbb{Z}_{\geq 0}$,
  \[
    b \equiv  \sum_{0 \leq j < J}
    (-\h)^j
    \sum_{|\alpha| = j}
    \frac{\partial^\alpha \psi(0)}{\alpha !}
    \partial^\alpha a
    \mod{ \h^{(1 - \delta)
        J} S^{m- J}_{\delta}},
  \]
  with remainder depending
  continuously upon $a$.
\end{lemma*}
The proof is similar
to but much simpler
than that of theorem \ref{thm:star-prod-asymp-general},
hence left to the reader.

\section{Proofs
  concerning the operator
  assignment}\label{sec:oper-attach-symb}
The main aim of this section is to supply the proofs of theorems
\ref{thm:main-properties-opp-symbols} and
\ref{thm:rescaled-operator-memb},
as well as some of the miscellaneous
results stated in \S\ref{sec:oper-attach-symb-1}.
We retain their notation and setup.
We also establish
a generalization (theorem \ref{thm:comp-gen-subsp})
that will be very useful in applications.

\subsection{Polynomial symbols: proofs}
\label{sec:quantize-polynomials-pf}
We now prove the lemma of \S\ref{sec:quantize-polynomials}.
We first recall a characterization
of $\sym$.
Each $p \in \Sym(\mathfrak{g}_\mathbb{C})$
defines a translation-invariant differential operator on
$C^\infty(\mathfrak{g})$ that we denote by $\partial_p$: if $p = y_1 \dotsb y_n$, then
\[
  \partial_{p} \phi(x) =
  \partial_{t_1=0} \dotsb \partial_{t_n=0} \phi(x + t_1 y_1 +
  \dotsb + t_n y_n).
\]
On the other hand, each $r \in \mathfrak{U}$ defines a
left-invariant differential operator on $C^\infty(G)$ that we
denote simply by $r$: if $r = y_1 \dotsb y_n$,
then
\[
  r f(g) = \partial_{t_1=0} \dotsb \partial_{t_n=0} f(g \exp(t_1
  y_1) \dotsb \exp(t_n y_n)).
\]
As one verifies readily using Taylor's theorem,
the symmetrization map
intertwines the two actions near the origin:
if $f(\exp(x)) = \phi(x)$,
then
\begin{equation}\label{eqn:symmetrization-map-analytic-characterization}
  \partial_p \phi(0) = \sym(p) f(1).
\end{equation}

Now fix $u, v \in \pi^\infty$.
After the change of variables $\xi \mapsto -\xi$
in the definition,
we must verify that
\begin{equation}\label{eqn:opp-of-polynomials}
  \langle \pi(\sym(p)) u, v \rangle
  = 
  \int_{\xi \in \mathfrak{g}^\wedge} p(-\xi)
  (\int_{x \in \mathcal{G} }
  e^{x \xi}
  \langle \pi(\exp(x)) u, v \rangle \, d x) \, d \xi.
\end{equation}
Define
$\phi \in C_c^\infty(\mathcal{G})$
and $f \in C_c^\infty(\exp(\mathcal{G}))$
by
\[
  f(\exp(x)) := \phi(x) := \chi(x) \langle \pi(\exp(x)) u, v
  \rangle.
\]
Since $\chi \equiv 1$ in a neighborhood of the origin, we
have
$r f(1) = \langle \pi(r) u, v \rangle$
for $r \in \mathfrak{U}$.
The LHS of \eqref{eqn:opp-of-polynomials}
is thus
$\sym(p) f(1)$,
while the RHS is $(p_- \phi^\wedge)^\vee(0) = \partial_p
\phi(0)$, $p_-(\xi) := p(-\xi)$.
We conclude by
\eqref{eqn:symmetrization-map-analytic-characterization}.

\subsection{Smoothness away from the origin}
\label{sec:smooth-away-from-origin}
The following simple estimates,
suggested in \S\ref{sec:variation-of-op-wrt-chi},
will be employed
occasionally.
\begin{lemma*}
  \hfill
  \begin{enumerate}[(i)]
  \item For any $a \in S^\infty$,
    the distributional Fourier transform
    $a^\vee$
    is represented
    away from the origin by a smooth function.
  \item Fix integers $m,N$
    and a multi-index $\alpha$
    with $|\alpha| + m - N \leq - \dim(\mathfrak{g}) - 1$.
    Then for $a \in S^m$
    and $x \in \mathfrak{g} - \{0\}$,
    we have
    \[  \partial^\alpha a^\vee(x) \ll |x|^{-N},  \]
    where the implied constant depends continuously
    upon $a$.
    More generally,
    for $a \in S^m_{\delta}$,
    \begin{equation}\label{eqn:estimate-fourier-a-h-by-partial-int}
      \partial^\alpha a_{\h}^\vee(x)
      \ll
      \h^{-\dim(\mathfrak{g})}
      |x/\h^{1-\delta}|^{-N}.
    \end{equation}
  \item
    Let $n \in \mathbb{Z}_{\geq 0}$.
    If $a \in S^m$ with
    $m \leq - \dim(\mathfrak{g}) - 1 - n$,
    then $a^\vee$ is represented
    near the origin by an $n$-fold differentiable function.
  \end{enumerate}
\end{lemma*}
\begin{proof}
  We
  integrate by parts repeatedly
  in the integral defining
  $a^\vee$,
\end{proof}

\subsection{Variation with respect to the cutoff: proofs}
\label{sec:variation-of-op-wrt-chi-pf}
We now prove the lemma of \S\ref{sec:variation-of-op-wrt-chi}.
It suffices to prove assertion (ii).  Define the $\h$-dependent element
$f \in C_c^\infty(G)$ by
$a_{\h}^\vee \chi_1 = a_{\h}^\vee \chi_2 + f$.  Fix
$\eps > 0$ small enough that
$\chi_1(x) = \chi_2(x)$ whenever $|x| \leq \eps$.  Then
$f(x) \neq 0$ only if $|x| > \eps$; in that case, we may apply
\eqref{eqn:estimate-fourier-a-h-by-partial-int} to see that
for
any fixed $N \geq 0$, the $\h$-dependent elements $\h^{-N} f$ belong to a fixed
bounded subset of $C_c^\infty(G)$.  As discussed in
\S\ref{sec:smoothing-ops}, the map
$C_c^\infty(G) \rightarrow \Psi^{-\infty}$ is continuous.
The conclusion follows.

\subsection{Equivariance: proofs}\label{sec:equivariance-pf}
We now prove
the lemma of \S\ref{sec:equivariance}.
We may assume that $\eps \in (0,1)$, say,
so that $\delta + \eps/2 \in (0,1)$.
Recall that $g \in G$
is assumed to satisfy the condition
\eqref{eqn:adjoint-bound-for-G-equivariance},
which we copy here for convenience:
\begin{equation}
  \|\Ad(g)\| \ll \h^{-1 + \delta  + \eps}.
  \tag{\ref{eqn:adjoint-bound-for-G-equivariance}}
\end{equation}

We fix $\h_0 > 0$ sufficiently small in terms of $\eps$,
$\delta$, and the cutoff $\chi$ implicit in the definitino of
$\Opp$.
We treat separately the cases $\h \geq \h_0$ and $\h < \h_0$.
In the range $\h \geq \h_0$, the rescaled symbol $a_{\h}$ lies in a
bounded subset $S^{\infty}$, and we see from
\eqref{eqn:adjoint-bound-for-G-equivariance} that lies in a
fixed compact subset of $G$ modulo the center.  The required conclusion thus
follows from the same argument as in \S\ref{sec:equivariance}.
It remains to range the range $\h < \h_0$.

Define
$b  \in S^\infty_{\delta}$
so that $b^\vee$
is a smooth truncation of $a^\vee$ to
$A := \{x : |x| \leq \h^{-\delta-\eps/2}\}$.
Using \S\ref{sec:smooth-away-from-origin},
we see that
\begin{equation}\label{eqn:egorov-h-1}
  a \equiv b \mod{\h^\infty S^{-\infty}},
  \quad 
  g \cdot a \equiv g \cdot b \mod{\h^\infty S^{-\infty}},
\end{equation}
hence
\begin{equation}\label{eqn:egorov-h-2a}
  \Opp_{\h}(a) \equiv \Opp_{\h}(b)
  \mod {\h^{\infty}
    \Psi^{-\infty}},
\end{equation}
\begin{equation}\label{eqn:egorov-h-2b}
  \Opp_{\h}(g \cdot a) \equiv \Opp_{\h}(g \cdot b)
  \mod {\h^{\infty}
    \Psi^{-\infty}}.
\end{equation}
We may also verify directly,
using the identity $\pi(g) [\pi(x),T] \pi(g)^{-1}
= [\pi(\Ad(g) x), T]$
for $x \in \mathfrak{g}$
and its $n$-fold iterate,
that
\begin{equation}\label{eqn:T-vs-pi-g-etc}
  T \in \h^\infty \Psi^{-\infty}
  \implies
  \pi(g) T \pi(g)^{-1} \in \h^\infty \Psi^{-\infty}.
\end{equation}
For $x \in A$, we have
$|g \cdot x| \leq \|\Ad(g)\| \h^{ - \delta - \eps/2} \leq
\h^{-1+ \eps/2}$.
Since $\h < \h_0$,
it follows (having chosen $\h_0$ suitably)
that the cutoff $\chi$ implicit in the
definition of $\Opp_{\h}$ satisfies
$\chi(\h x) = \chi(\h(g \cdot x)) = 1$, and so the identity
\begin{equation}\label{eqn:egorov-h-3}
  \Opp_{\h}(g \cdot b)
  = \pi(g) \Opp_{\h}(b) \pi(g)^{-1}
\end{equation}
holds exactly.
We conclude by combining
\eqref{eqn:egorov-h-2b},
\eqref{eqn:egorov-h-3},
\eqref{eqn:egorov-h-2a}
and \eqref{eqn:T-vs-pi-g-etc}.

\subsection{Membership criteria for operator classes}
\label{sec:membership-criteria-technical-stuff}
We establish
a basic criterion
for membership
in the operator classes
$\Psi^m$
defined in
\S\ref{sec:sort-of-obvious-operator-class-memberships}.
This will be applied
to
establish the lemma of
\S\ref{sec:sort-of-obvious-operator-class-memberships}
and then further in \S\ref{sec-1-6-4}.
\begin{proposition}\label{prop:operator-classes-membership}
  For $m \in \mathbb{Z}$,
  an operator $T$ on $\pi$
  belongs to $\Psi^m$ if and only if
  the following holds for each
  $u \in \mathfrak{U}$:
  \begin{equation}\label{eq:crit-Psi-plus}
    \text{ if $m \geq 0$, then }
    \sup_{0 \neq v \in \pi^\infty }
    \frac{
      \|\theta_u(T) v\|^2
    }{
      \langle \Delta^m v, v \rangle
    }
    < \infty;
  \end{equation}
  \begin{equation}\label{eq:crit-Psi-minus}
    \text{ if $m \leq 0$, then }
    \sup_{0 \neq v \in \pi^\infty }
    \sup_{x_1,\dotsc,x_{-m} \in \{1\} \cup \mathcal{B}}
    \frac{
      \|\theta_u(T) x_1 \dotsb x_{-m} v\|
    }{
      \|v\|
    }
    < \infty.
  \end{equation}
  Moreover, the seminorms defining $\Psi^m$
  may be bounded
  in terms of such quantities.
\end{proposition}

The proofs occupy the remainder  of
\S\ref{sec:membership-criteria-technical-stuff}.

We extend $\pi : \mathfrak{U} \rightarrow \End(\pi^\infty)$
to
$\pi : \mathfrak{U}[\Delta^{-1}] \rightarrow \End(\pi^\infty)$.
(Here $\mathfrak{U}[\Delta^{-1}]$ is the
localization of $\mathfrak{U}$ at $\Delta$:
it is the universal ring equipped with a morphism from $\mathfrak{U}$
in which $\Delta$ becomes invertible. Although this 
ring is difficult to describe precisely, we will use it in a
rather formal fashion.)
We define on
$\mathfrak{U}[\Delta^{-1}]$ a $\mathbb{Z}$-filtration by assigning
weight $1$ to elements of $\mathfrak{g}$ and
weight $-2$ to $\Delta^{-1}$:
\begin{definition*}
  For
  $m \in \mathbb{Z}$,
  we say that $t \in \mathfrak{U}[\Delta^{-1}]$ has \emph{order
    $\leq m$} if it may be expressed as a linear combination of
  products $w_1 \dotsb w_n$ ($n \in \mathbb{Z}_{\geq 0}$)
  for which:
  \begin{enumerate}[(i)]
  \item For each $i \in \{1..n\}$,
    either
    \begin{enumerate}
    \item $w_i \in \mathfrak{g}$, or
    \item $w_i = \Delta^{-1}$.
    \end{enumerate}
  \item
    If cases (a) and (b) occur $n_1$ and $n_2$ times,
    respectively,
    then
    $n_1 - 2 n_2 \leq m$.
  \end{enumerate}
\end{definition*}
We denote by $\mathfrak{U} \ni u \mapsto \theta_u
\in \End(\mathfrak{U}[\Delta^{-1}])$
the algebra morphism extending
$\theta_x(t) := [x,t]$ for $x \in \mathfrak{g}$,
so that $\theta_{x_1 \dotsb x_n}(t) = [x_1,\dotsc,[x_n,t]]$
for $x_1,\dotsc,x_n \in \mathfrak{g}$.
Observe that
\begin{equation}\label{eqn:derivative-of-Delta-inverse}
  \theta_x(\Delta^{-1}) = [x,\Delta^{-1}] = - \Delta^{-1} [x,\Delta] \Delta^{-1}
  \text{ for $x \in \mathfrak{g}$.}
\end{equation}
\begin{lemma}\label{lem:grading-preserved-by-adjoint-stfuf}
  Let $t \in \mathfrak{U}[\Delta^{-1}]$, $u \in \mathfrak{U}$
  and $m \in \mathbb{Z}$.
  If $t$ has order $\leq m$,
  then $\theta_u(t)$ has order $\leq m$.
\end{lemma}
\begin{proof}
  By repeated application of \eqref{eqn:derivative-of-Delta-inverse}.
\end{proof}
\begin{lemma}\label{lem:annoying-boundedness-lemma}
  If $t \in \mathfrak{U}[\Delta^{-1}]$
  has order $\leq 0$, then $\pi(t)$ induces a bounded operator $\pi \rightarrow \pi$.
\end{lemma}
\begin{proof}
  By repeated application
  of \eqref{eqn:derivative-of-Delta-inverse},
  we may write $t$ as a linear combination
  of products
  of factors of the form
  $x y \Delta^{-1}$
  with $x,y \in \{1\} \cup \mathcal{B}$.
  To each such factor
  we apply the case $s=2$ of \eqref{eq:pi-s-concretized},
  giving
  $\|\pi(x y \Delta^{-1}) v \|^2
  \ll 
  \langle \Delta^2 \Delta^{-1} v, \Delta^{-1} v \rangle
  = \|v\|^2$.
\end{proof}

\begin{proof}[Proof of Proposition
  \ref{prop:operator-classes-membership}]
  The forward implications are straightforward:
  If $T \in \Psi^m$, then $\theta_u(T) : \pi^m \rightarrow \pi^{0}$ is
  bounded, so \eqref{eq:crit-Psi-plus} holds, while if $T \in \Psi^{-m}$,
  then $\theta_u(T) : \pi^{0} \rightarrow \pi^{-m}$ is bounded,
  so \eqref{eq:crit-Psi-minus} follows from 
  \eqref{eq:pi-s-concretized}.
  We turn to the converse implications,
  which we treat now in a unified manner.
  For quantities $A$ and $B$
  depending upon an element $v \in \pi^{\infty}$,
  we write $A \ll B$
  to denote that $|A| \leq c |B|$
  for some $c \geq 0$ not depending upon $v$,
  and write $A \asymp B$
  if $A \ll B \ll A$.
  Let $m \in \mathbb{Z}$.
  Assume that
  for each $u \in \mathfrak{U}$,
  \begin{equation}\label{eqn:crit-hypo-1}
    \text{ if $m \geq 0$, then }
    \|\theta_u(T) v\| \ll
    \sup_{x_1,\dotsc,x_m \in \{1\} \cup \mathcal{B}}
    \|x_1 \dotsb x_m v \|,
  \end{equation}
  \begin{equation}\label{eqn:crit-hypo-2}
    \text{ if $m \leq 0$, then }
    \sup_{x_1,\dotsc,x_{-m} \in \{1\} \cup \mathcal{B}}
    \|\theta_u(T) x_1 \dotsb x_{-m} v\| \ll  \|v\|.
  \end{equation}
  We must show then for each $u \in \mathfrak{U}$
  and $s \in \mathbb{Z}$
  that
  \begin{equation}\label{eqn:crit-goal-1}
    \langle \Delta^{s-m} \theta_u(T) v, \theta_u(T) v  \rangle
    \ll \langle \Delta ^s v, v \rangle.
  \end{equation}
  To that end,
  choose $k \in \mathbb{Z}_{\geq 0}$
  sufficiently large in terms of $m$ and $s$;
  it will suffice to assume that $2k \geq \max(m-s, -s)$.
  Then
  \[
  \langle \Delta ^{s}  \Delta^k v, \Delta^k v \rangle
  =
  \langle \Delta ^{s + 2k } v,  v \rangle
  \asymp
  \sup_{x_1,\dotsc,x_{s + 2k} \in \{1\} \cup \mathcal{B}}
  \|x_1 \dotsb x_{s + 2 k} v\|^2,
  \]
  so by the invertibility of $\Delta$,
  our task \eqref{eqn:crit-goal-1} reduces to showing that
  \begin{equation}\label{eqn:crit-goal-2}
    \langle \Delta^{s-m} \theta_{u}(T) \Delta^{k} v, \theta_{u}(T)  \Delta^k v  \rangle
    \ll
    \sup_{x_1,\dotsc,x_{s + 2k} \in \{1\} \cup \mathcal{B}}
    \|x_1 \dotsb x_{s + 2 k} v\|^2.
  \end{equation}
  In the case $s \geq m$,
  we expand the definitions of
  $\Delta^k$ and $\Delta^{s-m}$ and
  use identities such as
  $\theta_{u}(T) \pi(x) =\pi(x) \theta_{u}(T) - \theta_{x u}(T)$
  for $x \in \mathfrak{g}$
  to write the LHS  of \eqref{eqn:crit-goal-2}
  as
  a linear combination of expressions
  \begin{equation}\label{eqn:crit-rearr-1}
    \langle  \theta_{u'}(T)  x_1 \dotsb x_{s-m + 2 k} v,
    \theta_{u''}(T) y_{1} \dotsb y_{s-m + 2k} v  \rangle
  \end{equation}
  where
  $x_i,y_i \in \{1\} \cup \mathcal{B}$
  and $u', u'' \in \mathfrak{U}$.
  We then apply Cauchy--Schwarz to each such expression
  and  invoke the assumed bound for $T$
  to conclude.
  We argue similarly
  in the case
  $m \geq s$,
  but only expand $\Delta^k$.
  We arrive then at expressions of the form
  \begin{equation}\label{eqn:crit-rearr-2}
    \langle
    A
    \theta_{u'}(T)
    x_{1} \dotsb x_{s - m + 2 k}
    v,
    \theta_{u''}(T) y_{1} \dotsb y_{s - m + 2 k}  \rangle
  \end{equation}
  with
  \begin{equation}
    A
    =
    z_1 \dotsb z_{m-s} \Delta^{s-m}
    w_{1} \dotsb w_{m-s},
  \end{equation}
  where
  $x_i,y_i,z_i,w_i \in \{1\} \cup \mathcal{B}$
  and $u', u'' \in \mathfrak{U}$
  are as above.
  By lemma \ref{lem:annoying-boundedness-lemma},
  each such operator $A$ is bounded on $\pi^0$.
  We may thus apply Cauchy--Schwarz to \eqref{eqn:crit-rearr-2}
  and argue as before
  to conclude.
\end{proof}

We now establish
the result of \S\ref{sec:sort-of-obvious-operator-class-memberships}
in a general form:
\begin{proposition}\label{prop:operator-classes-order-at-most-m-generalized} 
  Let $m \in \mathbb{Z}$.
  If $t \in \mathfrak{U}[\Delta^{-1}]$ has order $\leq m$,
  then
  $\pi(t)$ has order $\leq m$,
  i.e., $\pi(t) \in \Psi^m$.
\end{proposition}
\begin{proof}
  By the lemma of \S\ref{sec:operator-classes-composition},
  it suffices to consider the following special cases:
  \begin{enumerate}[(a)]
  \item $t \in \mathfrak{g}$ and $m = 1$.
  \item $t = \Delta^{-1}$ and $m=-2$.
  \end{enumerate}
  We appeal to the criterion of proposition
  \ref{prop:operator-classes-membership}.
  Let
  $u \in \mathfrak{U}$.  In case (a), we have
  $\theta_u(t) \in \mathfrak{g}$, so
  the required
  estimate \eqref{eq:crit-Psi-plus}
  reduces to the case $s = 1$ of
  \eqref{eq:pi-s-concretized}.
  In case (b),
  we have by lemma
  \ref{lem:grading-preserved-by-adjoint-stfuf}
  that
  $\theta_u(t)$ has order $\leq -2$,
  hence that $\theta_u(t) x_1 x_2$
  has order $\leq 0$
  for $x_1,x_2 \in \{1\} \cup \mathcal{B}$,
  so the required estimate
  \eqref{eq:crit-Psi-minus}
  reduces to
  lemma \ref{lem:annoying-boundedness-lemma}.
\end{proof}

\subsection{Operator norm
  bounds}\label{sec:operator-norm-bounds}
We record some estimates to be applied
below in the proofs of theorems 
\ref{thm:main-properties-opp-symbols} and
\ref{thm:rescaled-operator-memb}.
 
\begin{proposition*}~
  \begin{enumerate}[(i)]
  \item If $a \in S^0$, then
    $\Opp(a)$ defines a bounded operator on $\pi$,
    with operator norm bounded continuously in terms of $a$.
  \item If $a \in S^0_{\delta}$
    for some fixed $\delta \in [0,1/2)$,
    then $\Opp_{\h}(a)$ defines an $\h$-dependent bounded operator on $\pi$,
    with operator
    norm bounded uniformly in $\h$ and continuously in terms of $a$.
  \end{enumerate}
\end{proposition*}
The proof occupies the remainder
of this section.
It suffices to establish assertion (ii),
which recovers assertion (i) upon taking
$\delta = 0$ and restricting to $\h$-independent symbols.

We begin with some preliminaries.
Let $\mathcal{N}$
denote the norm on
$\mathcal{S}(\mathfrak{g}^\wedge)$
given by
$\mathcal{N}(a) := \|a^\vee \|_{L^1(\mathfrak{g})}$.
It is dilation-invariant:
$\mathcal{N}(a) = \mathcal{N}(a_{\h})$.

Denote by $\|.\|$ the operator norm
on $\End(\pi)$.
We have the following
trivial estimate:
\begin{lemma}\label{lem:triv-opp-bound}
  For $a \in \mathcal{S}(\mathfrak{g}^\wedge)$,
  we have
  $\|\Opp_{\h}(a)\| \leq \mathcal{N}(a)$.
\end{lemma}
We note the following consequence of
lemma \ref{lem:localized-fn-fourier} of \S\ref{sec:localized-functions}:
\begin{lemma}\label{lem:L1-bound-localized}
  Let $a \in S^0_{\delta}$
  be localized at some element $\omega \in \mathfrak{g}^\wedge$.
  Then $\mathcal{N}(a) \ll 1$;
  the implied constant may depend
  upon $\delta$,
  and continuously upon $a$,
  but not upon $\omega$.
\end{lemma}
We recall 
the Cotlar--Stein
lemma (see \cite[Lem 18.6.5]{MR2304165}):
\begin{lemma}\label{lem:cotlar-stein}
  Let $V_1, V_2$ be Hilbert spaces.
  Let $T_j  : V_1 \rightarrow V_2$ 
  be a sequence of bounded linear operators.
  Assume that
  \begin{equation}\label{eq:cotlar-stein-hypothesis}
    \sup_j \sum_k
    \|T_j^* T_k\|^{1/2}
    \leq C,
    \quad 
    \sup_j \sum_k
    \|T_j T_k^*\|^{1/2} \leq C,
  \end{equation}
  Then the series
  $T := \sum T_j$ converges in the Banach space
  of bounded linear operators
  from $V_1$ to $V_2$,
  and has operator norm
  $\|T\| \leq C$.
\end{lemma}
We now prove assertion (ii) of the proposition.
Let $a \in S^0_{\delta}$.
As in \S\ref{sec:localized-functions}, we may write
$a = \sum_{\omega \in \Omega} a_\omega$, where $a_\omega \in
S^0_{\delta}$ is
localized at $\omega$.
and depends continuously upon $a$.
Since
$\sum a_\omega$ converges to $a$ distributionally,
we have
$\Opp_{\h}(a) = \sum_\omega \Opp_{\h}(a_\omega)$
as maps $\pi^{\infty} \rightarrow \pi^{-\infty}$.
As noted in \S\ref{sec:self-adj-of-op-schw},
we have
$\Opp(a_\omega)^* = \Opp(\overline{a_\omega})$,
so
$\Opp(a_{\omega_1})^* \Opp(a_{\omega_2})
=\Opp(\overline{a_{\omega_1}} \star_{\h} a_{\omega_2},\chi')$.
By lemmas
\ref{lem:triv-opp-bound} and \ref{lem:cotlar-stein},
it will thus suffice to show that
\begin{equation}\label{eqn:required-N-bound-after-Cotlar-Stein}
  \sup_{\omega_1 \in \Omega}
  \sum_{\omega_2 \in \Omega}
  \mathcal{N}(\overline{a_{\omega_1}} \star_{\h}
  a_{\omega_2})^{1/2}
  \ll 1,
\end{equation}
with continuous dependence upon $a$.  To that end, we fix
$N \in \mathbb{Z}_{\geq 0}$ large enough, then fix
$J \in \mathbb{Z}_{\geq 0}$ large enough in terms of $N$,
and write
\[
\overline{a_{\omega_1}} \star_{\h} a_{\omega_2} = \underbrace{
  \sum_{0 \leq j < J} \h^j \overline{a_{\omega_1}} \star^j
  a_{\omega_2} }_{ =: b_{\omega_1,\omega_2} } +
r_{\omega_1,\omega_2}.
\]
Using the proposition of \S\ref{sec:star-prod-pf-red},
we see that
\[
\mathcal{N}(r_{\omega_1,\omega_2})
\ll
\h^N
\langle \omega_1 \rangle^{-N}
\langle \omega_2 \rangle^{-N}.
\]
In particular,
$r_{\omega_1,\omega_2}$
gives an acceptable contribution to \eqref{eqn:required-N-bound-after-Cotlar-Stein}.
On the other hand,
the symbol $b_{\omega_1,\omega_2}$
is localized
at $\omega_1$
and
depends continuously upon $a$;
moreover,
for given $\omega_1$,
we have
$b_{\omega_1,\omega_2} = 0$
for $\omega_2$ outside a set of cardinality
$\O(1)$.
By lemma \ref{lem:L1-bound-localized},
we deduce that
$b_{\omega_1,\omega_2}$
gives
an acceptable
contribution to \eqref{eqn:required-N-bound-after-Cotlar-Stein}.
The proof is complete.

\subsection{Proofs of operator class memberships: without rescaling}
\label{sec-1-6-4}
We now prove
theorem
\ref{thm:main-properties-opp-symbols}.
We must verify for $m \in \mathbb{Z}$
(hence for $m \in \mathbb{Z} \cup \{\pm \infty \}$)
that
\begin{equation}\label{eq:Sm-maps-to-Psim}
  \Opp(S^m) \subseteq \Psi^m,
\end{equation}
with the induced map continuous,
and that
\begin{equation}\label{eq:composition-for-symbols-basic}
  \Opp(a) \Opp(b) = \Opp(a \star b, \chi ')
\end{equation}
for $a, b \in S^\infty$.
Recall that \eqref{eq:Sm-maps-to-Psim}
implies that
\begin{equation}\label{eq:ops-preserve-smooth-vectors}
  \Opp(S^\infty) \pi^\infty \subseteq \pi^\infty,
\end{equation}
so that the composition
in \eqref{eq:composition-for-symbols-basic}
makes sense.
In fact, we will 
prove \eqref{eq:composition-for-symbols-basic}
and \eqref{eq:ops-preserve-smooth-vectors} simultaneously,
and then combine these with
\S\ref{sec:operator-norm-bounds} to deduce
\eqref{eq:Sm-maps-to-Psim}.

\emph{We observe first that for a polynomial symbol $p$, the
corresponding operator $\Opp(p)$ acts both on $\pi^\infty$ and
on $\pi^{-\infty}$.}
This observation is a consequence of \S\ref{sec:equivariance}.
It applies in particular when
$p(\xi) = \langle \xi \rangle^{ 2N}$, in which case
$\Opp(p) = \Delta^N$.

\emph{We observe next that the composition law
  \eqref{eq:composition-for-symbols-basic} holds if
  both $a$ and $b$
  are compactly-supported,} as is clear from the preliminary
discussion of \S\ref{sec:comp-prelim}.

\emph{We observe next that the composition law
  \eqref{eq:composition-for-symbols-basic} holds under the
  assumption that either $a$ or $b$ is a polynomial.}  In view
of the previous observation, this assumption permits us to
\emph{define} the composition in
\eqref{eq:composition-for-symbols-basic}, following
\S\ref{sec:weak-defin-oper}.

Suppose first that $b$ is a polynomial.
Let $u,v \in \pi^\infty$.
We must check that
the quantities
\begin{equation}\label{eqn:check-opp-composition-0}
  \langle \Opp(a) \Opp(b) u , v  \rangle =   
    \int _{\xi \in \mathfrak{g}^\wedge }
  a(\xi)
  \left(
    \int _{x \in \mathfrak{g} }
    e^{- x \xi}
    \chi(x)
    \langle \pi(\exp(x)) \Opp(b) u, v \rangle
    \, d x \right) \, d \xi
\end{equation}
and
\begin{equation}\label{eqn:check-opp-composition-1}
  \langle \Opp(a \star b) u , v  \rangle =   
  \int _{\xi \in \mathfrak{g}^\wedge }
  (a \star b)(\xi)
  \left(
    \int _{x \in \mathfrak{g} }
    e^{- x \xi}
    \chi(x)
    \langle \pi(\exp(x)) u, v \rangle
    \, d x \right) \, d \xi
\end{equation}
are equal.
Note that in either expression, the parenthetical integral over $x$
defines a Schwartz function of $\xi$.

To verify the required equality, we note first from the star product
asymptotics that for given $b$,
the tempered distributions $a$ and $a \star b$
depend continuously upon $a \in S^{\infty}$.
Using that $C_c^\infty$ has dense image in $S^\infty$
(cf. the beginning of \S\ref{sec:star-prod-pf-red}),
we may thus reduce to the case
that $a$ lies in $C_c^\infty$.
Then $\Opp(a)$ and its adjoint $\Opp(a)^* = \Opp(\bar{a})$
act on $\pi^{\infty}$.
We may rewrite \eqref{eqn:check-opp-composition-0}
as
\begin{equation}\label{eqn:check-opp-composition-2}
  \langle \Opp(b) u, \Opp(a)^* v \rangle
  \int _{\xi \in \mathfrak{g}^\wedge }
  b(\xi)
  \left(
    \int _{x \in \mathfrak{g} }
    e^{- x \xi}
    \chi(x)
    \langle \pi(\exp(x)) u, \Opp(a)^* v \rangle
    \, d x \right) \, d \xi,
\end{equation}
where, since $\Opp(a)^* v \in \pi^\infty$,
the parenthetical over $x$ again defines a Schwartz function
of $\xi$.
Using now that for given $a$, the tempered distributions
$a \star b$ and $b$ depend continuously upon $b \in S^{\infty}$,
we may reduce the comparison of
\eqref{eqn:check-opp-composition-1} and
\eqref{eqn:check-opp-composition-2} to the case that also
$b \in C_c^\infty$.
As noted above, the conclusion is known in that case.

If instead $a$ is a polynomial, then we must check
that \eqref{eqn:check-opp-composition-1} and
\eqref{eqn:check-opp-composition-2} are equal.
We argue as before, reducing first to the case that $b$ lies in
$C_c^\infty$.

% (We apply this observation below to
% $\Opp(\xi \mapsto \langle \xi \rangle^{2 N}) = \Delta^N$ for
% $N \in \mathbb{Z}_{\geq 0}$.)

% \emph{We observe first
%   that the composition law
%   \eqref{eq:composition-for-symbols-basic}
%   is valid
%   under either of the following
%   assumptions, each of which is implied by
%   \eqref{eq:ops-preserve-smooth-vectors}:}
% \begin{itemize}
% \item $\Opp(b)$ preserves $\pi^\infty$, or
% \item $\Opp(a)$ preserves $\pi^{-\infty}$,
%   so that $\Opp(a)^*$ preserves $\pi^\infty$.
% \end{itemize}
% To see this,
% we decompose
% $a = \sum a_{\omega_1}$ and $b = \sum b_{\omega_2}$ as in \S\ref{sec:localized-functions},
% so that
% $\Opp(a_{\omega_1}) \Opp(b_{\omega_2})
% = \Opp(a_{\omega_1} \star b_{\omega_2}, \chi')$, by
% \S\ref{sec:comp-prelim}.
% By estimates as in \S\ref{sec:appl-tayl-theor},
% we see that $\sum \sum a_{\omega_1} \star b_{\omega_2}$ converges
% to $a \star b$ distributionally.
% The required conclusion follows.
% (The point is that our hypotheses
% permit us to \emph{define}
% the composition in  \eqref{eq:composition-for-symbols-basic}.)

% In particular, we deduce that
% \eqref{eq:composition-for-symbols-basic} holds if either $a$
% or $b$ is a polynomial.  We apply this observation below to
% $\Opp(\xi \mapsto \langle \xi \rangle^{2 N}) = \Delta^N$ for
% $N \in \mathbb{Z}_{\geq 0}$.

\emph{We next verify \eqref{eq:ops-preserve-smooth-vectors}.}
It will suffice
to show that \[\Delta^N \Opp(S^\infty) \pi^\infty \subseteq \pi^0\]
for each $N \in \mathbb{Z}_{\geq 0}$.
Let us say that $T \in \Opp(S^\infty)$
is \emph{good}
if $T \pi^\infty \subseteq \pi^0$.
We have
$\Delta^N \Opp(S^\infty) \subseteq \Opp(S^\infty) +
\Psi^{-\infty}$
by the special case of  \eqref{eq:composition-for-symbols-basic}
already established,
so it will suffice to show
that every element  of $\Opp(S^\infty)$ is good.

Let $m \in \mathbb{Z}$, $a \in S^m$,
$T := \Opp(a)$.
If $m \leq 0$, so that $T \in \Opp(S^m) \subseteq \Opp(S^0)$,
then $T \pi^0 \subseteq \pi^0$ by (i.a),
so $T$ is good.
For the case $m \geq 0$,
we show now by induction on $m$
that $T$ is good.
Define
$p \in S^2$ by $p(\xi) := \langle \xi \rangle^2$.  Then
$a/p \in S^{m-2}$.
By the case of \eqref{eq:composition-for-symbols-basic}
established above
and the result of \S\ref{sec:variation-of-op-wrt-chi},
we see
that
$\Opp(a/p) \Opp(p) = \Opp(a/p \star p, \chi ')
=   \Opp(a/p
\star p) + R$, where $R \in \Psi^{-\infty}$.
By theorem \ref{thm:star-prod-basic},
we have $a/p \star p = a + r$ with $r \in S^{m-1}$.
Thus $T = \Opp(a/p) \Delta - \Opp(r)  - R$; in particular,
\begin{equation}\label{eq:sweet-inductive-identity}
  T \in \Opp(S^{m-2})
  \Delta
  + \Opp(S^{m-1}) + \Psi^{-\infty}.
\end{equation}
By our inductive hypotheses, we conclude that $T$ is good.

\emph{We next verify the general case of
  \eqref{eq:composition-for-symbols-basic}.}  Let
$u,v \in \pi^\infty$.  Having established
\eqref{eq:ops-preserve-smooth-vectors}, we know that in each of
\eqref{eqn:check-opp-composition-0},
\eqref{eqn:check-opp-composition-1} and
\eqref{eqn:check-opp-composition-2}, the parenthetical integral
over $x$ defines a Schwartz function of $\xi$.  By a continuity
argument as above, we may reduce first to the case that $a$ is
compactly-supported, then to the case that $b$ is
compactly-supported, in which the conclusion is known.

\emph{In summary, we have established
  \eqref{eq:ops-preserve-smooth-vectors} and the general case of
  \eqref{eq:composition-for-symbols-basic}.  We now
  establish \eqref{eq:Sm-maps-to-Psim}.}
(Continuity will be clear from the proof.)  We may assume that
$m \in \mathbb{Z}$.  We appeal to the criterion of
\S\ref{sec:membership-criteria-technical-stuff}.
By iterated application of
\eqref{eq:composition-for-symbols-basic}
(with one symbol a polynomial),
it will suffice to verify this
criterion in the special case $u = 1$.
Our task is then to show for $T \in \Opp(S^m)$ that
there exists $C \geq 0$ so that for all $v$,
\begin{equation}\label{eq:crit-for-m-pos}
  \|T v\|^2 \leq C \langle \Delta^m v, v \rangle
  \quad 
  \text{ if } m \geq 0,
\end{equation}
\begin{equation}\label{eq:crit-for-m-neg}
  \sum_{k=0}^{-m}
  \sup_{x_1,\dotsc,x_k \in \mathcal{B}(\mathfrak{g})}
  \|T x_1 \dotsb x_k v\| \leq C \|v\|
  \quad 
  \text{ if } m \leq 0.
\end{equation}
When $m = 0$,
either assertion
follows from \S\ref{sec:operator-norm-bounds}.
The case of \eqref{eq:crit-for-m-neg}
in which $m < 0$
reduces to the case $m=0$
by the composition law \eqref{eq:composition-for-symbols-basic}:
for $k \leq -m$,
we have
$T x_1 \dotsb x_k
\in \Opp(S^m) \Opp(S^1)^k \subseteq
\Opp(S^{m+k}) + \Psi^{-\infty}
\subseteq \Opp(S^0) + \Psi^{-\infty}$.
For $m > 0$,
we may then
use the decomposition \eqref{eq:sweet-inductive-identity}
to establish \eqref{eq:crit-for-m-pos}
inductively.

\subsection{Proofs of operator class memberships: with
  rescaling}\label{sec:oper-class-memb-rescl}
We now prove theorem \ref{thm:rescaled-operator-memb}.

\subsubsection{Composition formula}
Fix $\delta \in [0,1/2)$.
From the $\h=1$ case treated
above and \S\ref{sec:variation-of-op-wrt-chi},
we have
for
$a,b \in S^{\infty}_{\delta}$
that
\begin{equation}\label{eqn:h-rescaled-comp}
  \Opp_{\h}(a)
  \Opp_{\h}(b)
  =
  \Opp_{\h}(a \star_{\h} b, \chi ')
  \equiv
  \Opp_{\h}(a \star_{\h} b)
  \mod{\h^\infty \Psi^{-\infty}}.
\end{equation}

\subsubsection{Operator bounds: overview}
It remains
to verify
for each
$m \in \mathbb{Z}$
that
\begin{equation}\label{eqn:rescaled-operator-norm-bound}
  \Opp_{\h}
  (S^m_{{\delta}}) \subseteq  \h^{\min(0,m)} \Psi_\delta^m.
\end{equation}
We know by theorem
\ref{thm:main-properties-opp-symbols} that for each
$a \in S^m_\delta$, the operator $\Opp_{\h}(a)$ defines an
$\h$-dependent element of $\Psi^m$.  Our task is to estimate
suitably the variation of its seminorms with respect to $\h$.
We carry this out
in the remainder of \S\ref{sec:oper-class-memb-rescl}.

Before proceeding, it may be instructive to give a plausibility
argument for why \eqref{eqn:rescaled-operator-norm-bound} is the
natural bound to expect.
We focus on the following
consequence
of
\eqref{eqn:rescaled-operator-norm-bound}:
for $a \in S^m_\delta$
and fixed $s \in \mathbb{Z}$,
we have
\[
  \|\Opp_{\h}(a)\|_{\pi^s \rightarrow \pi^{s-m}}
  \ll \h^{\min(0,m)}.
\]
(Conversely,
in \S\ref{sec:reduction-case-u=1},
we will apply the composition formula
\eqref{eqn:h-rescaled-comp}
to reduce the proof of
\eqref{eqn:rescaled-operator-norm-bound}
to that of this consequence.)  We consider the case that $G$ is
the vector space $\mathbb{R}^n$, regarded as an abelian Lie
group, and that $\pi = L^2(\mathbb{R}^n)$ is the regular representation.  In this case, we may omit the cutoff
$\chi$ in our definition of $\Opp$.  The assignment $\Opp$ then
attaches to each symbol the corresponding Fourier multiplier.
The
Sobolev norms $\|.\|_{\pi^s}$ may be given in terms of the
Fourier
transform $v \mapsto v^\wedge$
on $L^2(\mathbb{R}^n)$
by $\|v\|_{\pi^s}^2 = \int _{\xi } |v^\wedge(\xi)|^2 \langle \xi
\rangle^m \, d \xi$,
where as usual $\langle \xi  \rangle = (1 + |\xi|^2)^{1/2}$.
It follows that for $b \in S^m$,
\[
  \|\Opp(b)\|_{\pi^s \rightarrow \pi^{s-m}} = \sup_{\xi \in
    \mathfrak{g}^\wedge} \frac{|b(\xi)|}{\langle \xi  \rangle^m}.
\]
Recalling the definition
$a_{\h}(\xi) = a(\h \xi)$ of our rescaling
and the estimate
$a(\xi) \ll \langle \xi  \rangle^{m}$
contained in the definition of $S^\delta_m$,
it follows that
\[
  \|\Opp_{\h}(a)\|_{\pi^s \rightarrow \pi^{s-m}} = \sup_{\xi \in
    \mathfrak{g}^\wedge} \frac{|a(\xi)|}{\langle \xi/\h \rangle^m}
  \ll
  \sup_{\xi \in \mathfrak{g}^\wedge}
  \frac{
    \langle \xi  \rangle^m
  }
  {
    \langle \xi/\h \rangle^m
  }.
\]
It is not hard to see that this last supremum is comparable to
$\h^{\min(0,m)}$ (consider
separately the cases in which $\xi = 0$
and in which $\xi$ is large).

% A similar argument may be given in the setting of the
% pseudodifferential calculus on $\mathbb{R}^n$, corresponding in
% our setup to the case that $G$ is a Heisenberg group.
The
case of general $G$ is more involved due to the absence of an analogue
for general $\pi$ of the Fourier transform on
$L^2(\mathbb{R}^n)$.  In the special case $\delta = 0$, we will
reduce readily to the $\h$-independent estimates established in
\S\ref{sec-1-6-4} with the aid of the membership criteria of
\S\ref{sec:membership-criteria-technical-stuff}.
We argue similarly
in the case of general $\delta$,
but using the
composition formula \eqref{eqn:h-rescaled-comp}
as a substitute for the
commutator-theoretic arguments of
\S\ref{sec:membership-criteria-technical-stuff},
which are less efficient
when $\delta > 0$.

\subsubsection{The case $\delta = 0$}
We treat first the simplest (and most important)
case $\delta = 0$.
Let $a \in S^m_0$,
so that
$\partial^\alpha a(\h \xi) \ll \langle \h \xi  \rangle^{m -
  |\alpha|}$.
Writing $b(\xi) := a_{\h}(\xi) = a(\h \xi)$
for its rescaling,
we see by the chain rule that $\partial^\alpha b(\xi)
= \h^{|\alpha|} \partial^\alpha a(\h \xi)$.
Using the estimates
\[
  \langle \h \xi  \rangle^m
  \ll \h^{\min(0,m)} \langle \xi  \rangle^m,
  \quad
  \h \langle \xi  \rangle \ll \langle \h \xi  \rangle,
\]
it follows that
\[
  \partial^\alpha b(\xi)
  \ll \h^{\min(0,m)} \langle \xi  \rangle^{m-|\alpha|}.
\]
In other words, $b \in \h^{\min(0,m)} S^m_0$.
By the continuity
of $\Opp = \Opp_{1} : S^m \rightarrow \Psi^m$ established in theorem
\ref{thm:main-properties-opp-symbols}, we deduce that
\[
  \Opp_{\h}(a) = \Opp(b) \in \h^{\min(0,m)} \Psi^m_0,
\]
as required.

\subsubsection{Reduction to symbols supported away from the origin}\label{sec:reduct-symb-supp}
Turning to the case of general $\delta \in [0,1)$,
we treat first the subcase
in which $a \in S^m_\delta$ is supported on elements $\xi$ of
size $\O(\h^\delta)$.  Set
$b(\xi) := a(\h^\delta \xi)$, so that
$a_{\h} = b_{\h^{1-\delta}}$.  By the chain rule,
$\partial^\alpha b(\xi) = \h^{\delta |\alpha| } \partial^\alpha
a(\h^\delta \xi) \ll 1$.
Since $b$ is supported on elements
of size $\O(1)$,
we have $b \in S^m_0$ (in fact, 
$b \in S^{-\infty}_0$).
By the construction of $b$ and the case
$\delta = 0$ treated above (applied with $\h$ replaced by
$\h^{1- \delta}$), we see that
\[
  \Opp_{\h}(a) =
  \Opp_{\h^{1-\delta}}(b)
  \in
  (\h^{1-\delta})^{\min(0,m)} \Psi_0^m
  \subseteq 
  % \h^{\min(0,m)} \Psi_0^m
  % \subseteq 
  \h^{\min(0,m)} \Psi_\delta^m,
\]
as required.  By smoothly decomposing a general symbol, it will
suffice from now on to treat the case of symbols supported on
$\{\xi : |\xi| \geq \h^{\delta}\}$, say.

We henceforth fix $\delta \in [0,1)$
and take $a \in S_\delta^m$
supported on elements $\xi$ satisfying $|\xi| \geq \h^{\delta}$.
By smoothly decomposing $a$
inside $S^m_\delta$, we may assume that there exists a basis
element $z \in \mathcal{B} \subseteq \mathfrak{g}$ such that
$|z(\xi)| \asymp |\xi|$ for all $\xi \in \supp(a)$.
We define the $\h$-dependent operator
\[
  T := \Opp_{\h}(a)
\]
and
must verify
that $T \in \Psi^m_\delta$,
i.e., that
for each fixed
$u \in \mathfrak{U}$ and $s \in \mathbb{Z}$,
the operator
norm $\|\theta_u^\delta(T)\|_{\pi^s \rightarrow \pi^{s-m}}$ is
$\O(\h^{\min(0,m)})$, uniformly in $\h$
(recall from \S\ref{sec:h-dependence-Psi-m}
the definition of $\theta_u^\delta$).

\subsubsection{Reduction to the case $u=1$}\label{sec:reduction-case-u=1}
It will be convenient to reduce to the case $u=1$.
By linearity,
we may suppose that $u = x_1 \dotsb x_n$
for some $x_1,\dotsc,x_n \in \mathfrak{g}$.
Regarding $x_j$
as a linear function
on $\mathfrak{g}^\wedge$,
we have $\pi(x_j) = \h^{-1} \Opp_{\h}(x_j)$.
Using the composition formula \eqref{eqn:h-rescaled-comp},
we see that
\[
  \theta_u^\delta(T)
  \equiv
  \Opp_{\h}(  \h^{-n}\theta_u^\delta(a))
  \mod{\h^\infty \Psi^{-\infty}},
\]
where $\theta_u^{\delta}(a)$
is defined inductively by
\[
  \theta_{u}^{\delta}(a)
  = \h^{\delta}
  \theta_{x_1 \dotsb x_{n-1}}^\delta (x_n \star_{\h} a - a
  \star_{\h} x_n)
  \text{ for } n \geq 1
\]
and
$\theta_1^{\delta}(a) := a$ for $n=0$.
The star product asymptotics (theorem
\ref{thm:star-prod-asymp-general}) imply that for each fixed
$y \in \mathfrak{g}$ and $b \in S^m_\delta$, both
$y \star_{\h} b$ and $b \star_{\h} y$ are congruent to the
pointwise product $y b$ modulo
$\h^{1 - \delta} S^{m}_{\delta}$.
By induction, it follows that
\[
  \h^{-n} \theta_u^\delta(a)
  \in
  S^{m}_{\delta}.
\]
Thus the required
operator norm
bounds
$\pi^s \rightarrow \pi^{s-m}$
for each element of
$\Opp_{\h}(S_\delta^m)$
imply the same for their images under
$\theta_u^\delta$.
The claimed reduction to the case $u=1$
follows.

It remains to verify that $T : \pi^s \rightarrow \pi^{s-m}$ has
operator norm $\O(\h^{\min(0,m)})$ for each fixed $s \in \mathbb{Z}$.
As in
\S\ref{sec:membership-criteria-technical-stuff},
we fix $k \in \mathbb{Z}_{\geq 0}$ large enough in terms of
$s$ and $m$, and reduce to verifying that for each $v \in \pi$,
\begin{equation}\label{eqn:Delta-s-m-T-Delta-k-v-etc}
    \langle \Delta^{s-m} T \Delta^k v,
  T \Delta^k
  \rangle \ll
  \left(   \h^{\min(0,m)}
  \|v\|_{\pi^{s + 2 k}} \right)^2.
\end{equation}

\subsubsection{The case $s -m \geq 0$}\label{sec:case-s-m-positive}
We argue separately
according to the sign of $s-m$,
supposing first that $s-m \geq 0$.

Recall that the symbol $a$ is assumed supported on elements
$\xi$ with $|\xi| \geq \h^\delta$.
Let us assume for the moment
the stronger support condition
$|\xi| \geq 1$.

By expanding the definition of
$\Delta^{s-m}$ and appealing to Cauchy--Schwarz,
we reduce to verifying
for all $x_1,\dotsc,x_{s-m} \in \{1\} \cup \mathcal{B}
\hookrightarrow
\mathfrak{U}$
that
\[
  \|x_1 \dotsb x_{s-m} T \Delta^k v\|
  \ll
  \h^{\min(0,m)}
  \|v\|_{\pi^{s + 2 k}}.
\]
To see this, we first evaluate $x_1 \dotsb x_{s-m} T \Delta ^k$
using the composition formula.  Since
$\pi(x) = \h^{-1} \Opp_{\h}(x)$ for $x \in \mathfrak{g}$, we
obtain
\[
  x_1 \dotsb x_{s-m} T \Delta ^k
  \equiv  \h^{-(s-m+2k)}
  \Opp_{\h}(x_1 \star_{\h} \dotsb \star_{\h} x_{s-m}
  \star_{\h} a \star_{\h} p^k),
\]
where $p(\xi) := \h^2 + |\xi|^2$ and $\equiv$ denotes congruence
modulo $\h^\infty \Psi^{-\infty}$.  (Strictly speaking,
$\star_{\h}$ is not associative due to the cutoff $\chi$, so we
should specify that this iterated star product is evaluated
(say) left-to-right.  Changing the order of evaluation
introduces negligible errors lying in
$\h^\infty \Psi^{-\infty}$, so we do not belabor this point.)

By theorem
\ref{thm:main-properties-opp-symbols},
we may find a continuous seminorm
$\nu$ on $S^{s+2 k}$
so that for all $c \in S^{s+2k}$
and $v \in \pi$,
we have
\begin{equation}\label{eqn:Opp-c-v-nu-c-v-pi-s-2k}
  \|\Opp(c) v\| \leq \nu(c) \|v\|_{\pi^{s+2 k}}.
\end{equation}
Since
the evaluation at $c_{\h}$
of any $S^{s+2k}$-seminorm
is bounded polynomially in $\h^{-1}$,
we see that if $N \in \mathbb{Z}_{\geq 0}$
is fixed large enough (in terms of $m,s,k,\delta$),
then
\begin{equation}\label{eqn:h-s-m-2k-nu-c-h}
    \h^{-(s-m+2k)}
  \nu(c_{\h}) \ll
  \h^{\min(0,m)}
  \quad
  \text{ for all } c \in \h^N S^{s+2k}_\delta.
\end{equation}
Fix
$J \in \mathbb{Z}_{\geq 0}$ large enough in terms of $m,s,k,N$.
Let $b$ denote the approximation to
$x_1 \star_{\h} \dotsb \star_{\h} x_{s-m} \star_{\h} a
\star_{\h} p^k$ obtained by replacing each star product
$\star_{\h}$ with the finite part of its asymptotic expansion
obtained by summing over $0 \leq j < J$, as in the statement of
theorem \ref{thm:star-prod-basic}.
Then
\[
  b \in  S^{s+2k}_{\delta},
  \quad
  \supp(b) \subseteq \supp(a).
\]
Since $J$ was chosen large enough,
we have
\[
  x_1 \star_{\h} \dotsb \star_{\h} x_{s-m} \star_{\h} a
  \star_{\h} p^k
  = b + c,
  \quad
  c \in \h^N S_\delta^{-N}.
\]
We deduce from
\eqref{eqn:h-s-m-2k-nu-c-h}
and \eqref{eqn:Opp-c-v-nu-c-v-pi-s-2k}
the acceptable bound
\[
  \h^{-(s-m+2k)}
  \|\Opp_{\h}(c) v\|
  \ll
  \h^{\min(0,m)}
  \|v\|_{\pi^{s+2k}}.
\]
Our task thereby reduces
to verifying that
\begin{equation}\label{eqn:goal-h-s-m-k-b-v}
  \h^{-(s-m+2k)}
  \|\Opp_{\h}(b) v \|
  \ll
  \h^{\min(0,m)}
  \|v\|_{\pi^{s+2k}}.
\end{equation}

Recall from \S\ref{sec:reduct-symb-supp} that we are given an
element $z \in \mathcal{B}$ such that $|z(\xi)| \asymp |\xi|$
for all $\xi$ in the support of $a$, hence also for $\xi$ in the
support of $b$.
We now ``approximately divide $b$ on the right
by $z^{s+2k}$''
with respect to the star product;
precisely,
we construct $q \in S_\delta^0$
for which $b \approx q \star_{\h} z^{s + 2 k}$
in the sense that
\[
  b \equiv  q \star_{\h} z^{s + 2 k} \mod{\h^N S_\delta^{-N}}
\]
for fixed large enough $N$.
To that end, we take $q = \sum_{0 \leq j < M} \h^j q_j$,
where $M$ is fixed large enough in terms of $N$
and the $q_j$ are chosen
so that the required approximation holds
in a formal sense;
explicitly,
\[
  q_0 := \frac{b}{z^{s + 2 k}},
  \quad
  q_1 :=
  \frac{- b \star^1 q_0}{z^{s + 2k}},
  \quad 
  q_2 :=
  \frac{- b \star^2 q_0 + b \star^1 q_1}{z^{s + 2k}},
\]
and so on.
We see by induction on $j$
that $q_j \in \h^{- \delta j} S_\delta^{-(s+2 k) j}$.
It follows
from the composition formula  \eqref{eqn:h-rescaled-comp}
that $q$ has the indicated properties.

By another application of the composition formula,
we see that the LHS of \eqref{eqn:goal-h-s-m-k-b-v}
is given up to negligible error
by
\[
  \h^{-(s-m+2k)}
  \|\Opp_{\h}(q) \Opp_{\h}(z^{s + 2k}) v \|
  =
  \h^{m}
  \|\Opp_{\h}(q) z^{s + 2k} v \|.
\]
Our task thereby reduces to showing that
\[
  \h^m \|\Opp_{\h}(q) z^{s + 2 k} v \|
  \ll
  \h^{\min(0,m)}
  \|v\|_{\pi^{s+2k}}.
\]
To see this, we appeal to part (ii) of the proposition of
\S\ref{sec:operator-norm-bounds}, which tells us that
$\Opp_{\h}(q)$ defines a bounded operator $\pi \rightarrow \pi$
with operator norm $\O(1)$.  Thus
$\|\Opp_{\h}(q) z^{s + 2 k} v \| \ll \|z ^{s + 2 k} v\|$.  By
appeal to the lower bound in \eqref{eq:pi-s-concretized} for
$\|v\|_{\pi^{s+2k}}$ and the obvious inequality
$\h^m \leq \h^{\min(0,m)}$, we obtain the required estimate.

This completes our treatment under the assumption that $a$ is
supported on $|\xi| \geq 1$.  By smooth decomposition inside
$S^m_\delta$, it suffices now to treat the
complementary case in which $a$ is
supported on $\h^\delta \leq |\xi| \leq 2$, say.

We consider first the subcase in which
$a$ is supported on $R \leq |\xi| \leq 2 R$
for some $\h$-dependent positive real
$R$ with $\h^\delta \ll R \ll 1$.
The general strategy is then as above,
but taking into account
that the polynomial symbols
$x_j$ and $p$ have respective sizes
$\ll R$ and $\ll R^2$ on the support of $a$.
The symbol $b$ constructed as above now lies in
$R^{s-m+2k} S_\delta^{-\infty}$.
By choosing $N$ and $J$ suitably,
the remainder term $c$ remains acceptable.
The ``quotient''
$q$ obtained by approximately dividing $b$
on the right
by $z^{s + 2 k}$
now lies in
$R^{-m} S_\delta^0$.
Since $\h \leq \h^{\delta} \ll R \ll 1$, we have
$\h^m R^{-m} \ll \h^{\min(0,m)}$.
We may thus conclude as before.

In the general case that $a$ is supported on
$\h^\delta \leq |\xi| \leq 2$, we take a smooth dyadic
decomposition $a = \sum_R a^{(R)}$, where $R$ runs over powers
of two satisfying $\h^\delta \ll R \ll 1$ and $a^{(R)}$ is as in
the previous paragraph.  The quotient $q$ obtained as before has
the form $\sum_R q^{(R)}$, where
$q^{(R)} \in R^{-m} S_\delta^0$.  For each $R$, the number
of $R'$
for which the supports of $q^{(R)}$ and $q^{(R')}$ overlap is
$\O(1)$.  Thus $q \in \max(1,\h^{- \delta m}) S_\delta^0$, and
we may conclude as in the previous paragraph.

\subsubsection{The case $s - m \leq 0$}
The argument is similar in structure to that in the case
$s - m \geq 0$ considered above: we treat separately the subcases
in which $a$ is supported on $|\xi| \geq 1$ or on
$\h^{\delta} \leq |\xi| \leq 2$, and dyadically decompose in the
latter case.  The only difference is that we arrange the
composition and division arguments slightly differently.

We begin
by using the composition
formula to write the LHS
of \eqref{eqn:Delta-s-m-T-Delta-k-v-etc},
up to negligible error,
as
\[
  \h^{-4 k}
  \langle \Delta^{s - m}
  \Opp_{\h}(b) v,
  \Opp_{\h}(b) v
  \rangle
\]
where $b$ is a truncation of the asymptotic expansion of
$a \star_{\h} p^k$, with $p(\xi) = \h^2 + |\xi|^2$.
We then approximately divide the symbol $b$ on the left by
$z^{m - s}$
and on the right by $z^{s + 2 k}$, giving a symbol $q$ for which
$b \approx z^{m-s} \star_{\h} q \star_{\h} z^{s + 2 k}$
in the same sense as before.
In the case that $a$ is supported
on $|\xi| \geq 1$,
the symbol $q$ lies in $S_\delta^0$,
while in the case that $a$ is supported
on $|\xi| \asymp R$,
we have $q \in R^{-m}  S_\delta^0$.
We reduce
in either case
to verifying that
\[
  \h^{2 m}
  \langle
  \pi(z)^{s-m}
  \Delta^{s - m}
  \pi(z)^{s-m}
  \Opp_{\h}(q) v,
  \Opp_{\h}(q) v
  \rangle
  \ll
  \left(   \h^{\min(0,m)}
  \|v\|_{\pi^{s + 2 k}} \right)^2.
\]
To that end,
we see by
proposition
\ref{prop:operator-classes-order-at-most-m-generalized}
of
\S\ref{sec:membership-criteria-technical-stuff}
that
$\pi(z)^{s-m}
\Delta^{s - m}
\pi(z)^{s-m}$
has operator norm $\O(1)$.
By Cauchy--Schwarz,
we reduce further
to showing that
\[
  \h^{m}
  \|
  \Opp_{\h}(q) v \|
  \ll
  \h^{\min(0,m)}
  \|v\|_{\pi^{s + 2 k}}.
\]
We now conclude
exactly as in the case $s-m \geq 0$.

% We appeal to \S\ref{sec:operator-norm-bounds}, and argue
% as in the final steps of \S\ref{sec-1-6-4}.
% The case $m \leq 0$
% may be addressed as before, using now that
% $\pi(x) = \h^{-1} \Opp_{\h}(x)$
% for $x \in \mathfrak{g}$.  For the case
% $m > 0$, we smoothly decompose $a \in S^m_{\delta}$ as
% $a' + a''$, where $a' \in S^0_{\delta}$ is supported
% near the origin and $a'' \in S^m_{\delta}$ away from it.
% We then write
% $a'' \equiv a''/p \star_{\h} p \mod{S^{m-1}_{\delta}}$
% with $p(\xi) := |\xi|^2$ and argue inductively
% via \eqref{eqn:h-rescaled-comp}.

\subsection{Generalization to proper subspaces}
\label{sec:gener-prop-subsp-2}
Let $\mathfrak{g}_1, \mathfrak{g}_2 \leq \mathfrak{g}$
and accompanying notation be as in
\S\ref{sec:setup-star-prod-asymp};
in particular, $\mathfrak{g}_1 + \mathfrak{g}_2 =
\mathfrak{g}$.
We assume, for convenience, that the cutoffs $\chi_1, \chi_2$
have support taken small enough in terms of the cutoff
$\chi$ using to define the operator map on
$\mathcal{S}(\mathfrak{g}^\wedge)$;
this assumption matters little in practice
(cf. \S\ref{sec:variation-of-op-wrt-chi}).

Let $a \in \mathcal{S}(\mathfrak{g}_1^\wedge)$
and
$b \in \mathcal{S}(\mathfrak{g}_2^\wedge)$.
As mentioned in \S\ref{sec:setup-star-prod-asymp},
one can then define
operators
$\Opp_{\h}(a)$ and
$\Opp_{\h}(b)$ on $\pi$,
preserving $\pi^{\infty}$,
and acting via the restrictions of $\pi$ to the subgroups $G_1$ and $G_2$.
These operators typically do not belong
to any of the operator classes
we have defined,
but their \emph{composition} belongs to $\Psi^{-\infty}$;
indeed,
$\Opp(a) \Opp(b) = \Opp(a \star b)$
with
$a \star b \in \mathcal{S}(\mathfrak{g}^\wedge)$.
Some analogues of the above results hold in this setting.
For instance:
\begin{theorem}\label{thm:comp-gen-subsp}
  Fix $m_1,m_2 \in \mathbb{Z}$
  with $(m_1,m_2)$ admissible (\S\ref{sec:when-should-star}).
  Fix $\delta \in [0,1/2)$.
  Let $a \in S^{m_1}_{\delta}(\mathfrak{g}_1^\wedge)$
  and
  $b \in S^{m_2}_{\delta}(\mathfrak{g}_2^\wedge)$.
  Then
  $\Opp_{\h}(a) \Opp_{\h}(b) = \Opp_{\h}(a \star_{\h} b)$.
  Moreover,
  for each fixed $M,N \in \mathbb{Z}_{\geq 0}$
  there is a fixed $J \in \mathbb{Z}_{\geq 0}$
  so that
  \begin{equation}\label{eqn:comp-with-remainder-J-diff-subspaces}
    \Opp_{\h}(a) \Opp_{\h}(b)
    \equiv 
    \sum_{0 \leq j < J}
    \h^j \Opp_{\h}(a \star^j b)
    \mod{\h^N \Psi_\delta^{-M}}
  \end{equation}
\end{theorem}
\begin{proof}
  We apply the general star product
  asymptotics (theorem
  \ref{thm:star-prod-asymp-general}) and argue as in the proof
  of corollary \ref{cor:epic-comp-law} of
  \S\ref{sec:oper-class-memb-0}.
\end{proof}

\subsection{Disjoint supports}\label{sec:disjoint-supports}
We retain the notation of the previous subsection,
and record a simple consequence.
\begin{lemma*}
  Let $(\delta,m_1,m_2,a,b)$
  be as in the hypotheses of theorem \ref{thm:comp-gen-subsp}.
  Assume
  that the preimages in $\mathfrak{g}^\wedge$
  of the supports of $a$ and $b$
  are disjoint.
  Then $\Opp_{\h}(a) \Opp_{\h}(b) \in \h^\infty \Psi^{-\infty}$,
  with continuous dependence upon $a$ and $b$.
\end{lemma*}
  The precise meaning of ``continuous dependence'' is
  that for any seminorm $\ell$ defining the topology on
  $\h^{\infty} \Psi^{-\infty}$
  and any $N \geq 0$,
  we have $\ell(\Opp_{\h}(a) \Opp_{\h}(b)) \leq \nu_1(a) \nu_2(b) \h^N$
  for some continuous seminorms $\nu_j$ on
  $S^{m_j}_{\delta}(\mathfrak{g}_j^\wedge)$.
\begin{proof}
  We apply
  \eqref{eqn:comp-with-remainder-J-diff-subspaces}
  with large $J$
  and use that
  $a \star^j b = 0$
  for all $j \in \mathbb{Z}_{\geq 0}$
  and
  $\h^\infty \Psi^{-\infty}
  = \cap_{M,N \in \mathbb{Z}_{\geq 0}}
  \h^N \Psi_\delta^{-M}$.
\end{proof}

\section{Infinitesimal characters}\label{sec:infin-char}

\subsection{Overview}\label{sec:inf-char-overview}
 
Let $\mathbf{G}$ be a reductive group over $\mathbb{R}$. 
(Our discussion applies, by restriction of scalars,
also to groups over $\mathbb{C}$.)
We denote by $G$ the Lie group of real points of $\mathbf{G}$, by
$\mathfrak{g}$ the Lie algebra, by
$\mathfrak{g}_\mathbb{C} = \mathfrak{g} \otimes_{\mathbb{R}}
\mathbb{C}$
the complexification, by $\mathfrak{g}_\mathbb{C}^*$ its
complex dual, and by
$i \mathfrak{g}^*$
the imaginary dual of $\mathfrak{g}$, which we may identify
with the
Pontryagin dual $\mathfrak{g}^\wedge$
(cf. \S\ref{sec:measures-et-al-G-g-g-star}).
We regard $i \mathfrak{g}^*$ as a real form of
the complex vector space $\mathfrak{g}_{\mathbb{C}}^*$.

We denote by
\[
[\mathfrak{g}_\mathbb{C}^*] = \mathfrak{g}_\mathbb{C}^* \giit \mathbf{G}
\]
the GIT quotient, i.e.,
the spectrum of the ring of $G$-invariant polynomials on
$\mathfrak{g}_\mathbb{C}^*$.  This variety has a natural real form,
denoted $[i \mathfrak{g}^*]$, corresponding to the polynomials
taking real values on $i \mathfrak{g}^*$.
We will identify $[\mathfrak{g}_\mathbb{C}^*]$ with its set
of complex points
and likewise $[i \mathfrak{g}^*]$ with its set of real points inside
$[ \mathfrak{g}_\mathbb{C}^*]$.
\index{GIT quotient $[ \mathfrak{g}_\mathbb{C}^*]$}
\index{GIT quotient $[i \mathfrak{g}^*]$}
As we recall below,
these varieties
are affine spaces:
the inclusion
$[i \mathfrak{g}^*] \hookrightarrow [\mathfrak{g}_\mathbb{C}^*]$
looks like $\mathbb{R}^n \hookrightarrow
\mathbb{C}^n$.
\begin{example*}
  Suppose $G = \GL_n(\mathbb{R})$.
  Using the trace pairing,
  we may identify $\mathfrak{g}_\mathbb{C}^*$ with
  the space of $n \times n$ complex matrices $\xi$.
  The map sending $\xi$ to the characteristic polynomial of $\xi/i$
  induces isomorphisms
  \[
  [\mathfrak{g}_\mathbb{C}^*] \stackrel{\sim}{\longrightarrow}
  \text{ monic polynomials $X^n +
    p_1 X^{n-1} +
    \dotsb + p_{n-1} X + p_n \in \mathbb{C}[X]$ }
  \]
  \[
  [i \mathfrak{g}^*] \stackrel{\sim}{\longrightarrow}
  \text{ monic polynomials $X^n +
    p_1 X^{n-1} +
    \dotsb + p_{n-1} X + p_n \in \mathbb{R}[X]$. }
  \]
\end{example*}

Let $\mathfrak{U}$ denote the the universal enveloping
algebra of $\mathfrak{g}_\mathbb{C}$, and $\mathfrak{Z}$ its
center;
the latter acts by scalars on any irreducible representation of $G$.
Harish--Chandra defines an \emph{algebra} isomorphism
\begin{equation}\label{eqn:hc-hom-0}
  \gamma : \mathfrak{Z} \simeq \{ \text{regular functions on $[\mathfrak{g}_\mathbb{C}^*$]} \},
\end{equation}
to be recalled below.
Each irreducible representation $\pi$ of $G$
thus gives rise to a point $\lambda_\pi \in [\mathfrak{g}_{\mathbb{C}}^*]$,
which we refer to as the
\emph{infinitesimal character} of $\pi$.\index{infinitesimal character $\lambda_\pi$}

The first aim of this section is to record
some preliminaries concerning the assignment
$\pi \mapsto \lambda_\pi$.
We then aim to prove, using our operator calculus,
some basic estimates
involving the $\lambda_\pi$.

\subsection{Basics concerning the
  quotient}\label{sec:quotient-affine}
We denote temporarily
by $R$ and $R_{\mathbb{C}}$
the rings of regular functions
of the varieties  $[i \mathfrak{g}^*]$ and $[\mathfrak{g}_{\mathbb{C}}^*]$ 
defined above.
By definition, $R_{\mathbb{C}} \cong
\Sym(\mathfrak{g}_{\mathbb{C}})^G$
is the ring of $G$-invariant polynomials on 
$\mathfrak{g}_{\mathbb{C}}^*$;
its real form $R  \cong \Sym(i \mathfrak{g})^G
\subseteq R_{\mathbb{C}}$ consists of the polynomials
taking real values on $i \mathfrak{g}^*$.

The $\mathbb{R}$-algebra $R$ admits a finite set
$p_1,\dotsc,p_n$ of algebraically independent homogeneous
generators.
Indeed, the corresponding assertion for $R_\mathbb{C}$
is a theorem of Chevalley
\cite{MR0072877},
and remains true for $R$
thanks to the following
fact whose proof we leave to the reader:
if $K \subset L$ are fields and $R = \bigoplus R_m$ is a graded $K$-algebra
with the property that $R \otimes_K L$ is polynomial on
homogeneous generators,
then the same
is true for $R$.

Thus $R = \mathbb{R}[p_1,\dotsc,p_n]$
and
$R_\mathbb{C}  = \mathbb{C}[p_1,\dotsc,p_n]$
are polynomial rings.

Recall that $[\mathfrak{g}_{\mathbb{C}}^*]$ denotes the set of 
complex
points of the spectrum of $R_{\mathbb{C}}$
and $[i \mathfrak{g}^*]$ the set of real points of the spectrum
of $R$;
by definition, $[\mathfrak{g}_{\mathbb{C}}^*]$
consists of $\mathbb{C}$-algebra
maps $R_\mathbb{C} \rightarrow \mathbb{C}$
and
$[i \mathfrak{g}^*]$
consists of $\mathbb{R}$-algebra
maps $R \rightarrow \mathbb{R}$.
There is a natural inclusion
$[i \mathfrak{g}^*] \hookrightarrow
[\mathfrak{g}_\mathbb{C}^*]$.
Fixing generators as above, we may identify
\[
[i \mathfrak{g}^*]
\cong \mathbb{R}^n
\subseteq 
[\mathfrak{g}_\mathbb{C}^*]
\cong \mathbb{C}^n.
\]
In particular, we may speak of the Euclidean distance
between elements of
$[\mathfrak{g}_\mathbb{C}^*]$.
This depends on the choice of generators $p_i$ above;
however, on any fixed compact subset, the notions of distance arising from different choices of generators
are comparable, i.e., bounded from above and below in terms of one another. 

We denote by
$\xi \mapsto [\xi]$
the natural maps
$\mathfrak{g}_\mathbb{C}^* \rightarrow
[\mathfrak{g}_\mathbb{C}^*]$
and
$i \mathfrak{g}^* \rightarrow [i \mathfrak{g}^*]$.
The first of these maps is surjective.
More precisely,
let  $\mathfrak{t}_\mathbb{C}$
be
a Cartan subalgebra
of $\mathfrak{g}_\mathbb{C}$,
with corresponding Weyl group $W_\mathbb{C}$.
Then $\mathfrak{g}_\mathbb{C}$ splits as
$\mathfrak{n}^-_{\C} \oplus \mathfrak{t}_{\C}
\oplus \mathfrak{n}_{\C}$,
and we may identify
$\mathfrak{t}_\mathbb{C}^*$
with the
subspace of
$\mathfrak{g}_\mathbb{C}^*$
orthogonal 
to $\mathfrak{n}^-_{\C} \oplus \mathfrak{n}_{\C}$.
By Chevalley's theorem,
the natural composition
$\mathfrak{t}_\mathbb{C}^* \rightarrow \mathfrak{g}_\mathbb{C}^*
\rightarrow [\mathfrak{g}_\mathbb{C}^*]$
induces an isomorphism of complex varieties
\begin{equation}\label{eqn:t-mod-W-vs-g-mod-G}
  \mathfrak{t}_\mathbb{C}^*/W_\mathbb{C}
  \cong [\mathfrak{g}_\mathbb{C}^*].
\end{equation}
The map $i \mathfrak{g}^* \rightarrow [i \mathfrak{g}^*]$
is not in general surjective (e.g., for $G = \SO(3)$),
but its image is readily verified
to be Zariski dense.

The complex conjugation
$\xi \mapsto \overline{\xi}$ on $\mathfrak{g}_\mathbb{C}$
descends to an involution on $[\mathfrak{g}_\mathbb{C}^*]$
that we continue to denote by $\lambda \mapsto \overline{\lambda
}$.
Similarly,
the scaling action of $t \in \mathbb{C}^*$
on $\mathfrak{g}_\mathbb{C}^*$
given by $\xi \mapsto t \xi$
descends to a scaling action $\lambda \mapsto t \lambda$
on $[\mathfrak{g}_\mathbb{C}^*]$.
The unique fixed point $[0]$
of the scaling action  gives an origin on $[\mathfrak{g}_\mathbb{C}^*]$.

For $p \in R$ and $\xi \in
\mathfrak{g}_\mathbb{C}^*$, one has
$\overline{p(\xi)} = p(-\overline{\xi })$;
it follows
readily
that
\begin{equation}\label{eqn:real-form-of-git-quot-via-conjugation}
  [i \mathfrak{g}^*] = \{\lambda \in [\mathfrak{g}_\mathbb{C}^*] : \lambda = - \overline{\lambda }\}.
\end{equation}

\subsection{The regular set; description by Cartan
  subalgebras}\label{sec:real-regular-set-cartan}
 
We recall that an element of
$[\ggC^*]$ is \emph{regular} if it identifies
with a regular point of
$\mathfrak{t}_\mathbb{C}^*/W_{\mathbb{C}}$,
i.e,. a point having $|W_{\mathbb{C}}|$ distinct
preimages in $\mathfrak{t}_{\mathbb{C}}^*$;
equivalently, $\lambda$ is regular
if its preimages in $\mathfrak{g}_\mathbb{C}^*$
are regular semisimple.
We note that $\xi \in \mathfrak{g}^*_\mathbb{C}$
is regular whenever $[\xi]$ is regular (cf. \S\ref{sec:general-conventions}),
but not conversely.
As usual, we use a subscripted ``$\reg$''
to denote the subset of regular elements.

The subset $[i \mathfrak{g}^*]_{\reg}$ of
$[i \mathfrak{g}^*]$ is dense and open.  We recall how to
parametrize
$[i \mathfrak{g}^*]_{\reg} \cap \image(i \mathfrak{g}^*)$ in
terms of (real) Cartan subgroups $T$ of $G$ (compare with, e.g., \cite[Thm
5.22]{MR855239}).  For each such $T$, the complexified Lie
algebra $\mathfrak{t}_{\C}$ is a Cartan subalgebra of
$\ggC$.
As $T$ varies over a finite set of conjugacy representatives,
the images of the maps $(i \mathfrak{t}^*)_{\reg} \rightarrow [i \mathfrak{g}^*]_{\reg}$
  partition
$[i \mathfrak{g}^*]_{\reg} \cap \image(i \mathfrak{g}^*)$.

\subsection{Harish--Chandra isomorphism}
We recall the construction of the map $\gamma$
as in \eqref{eqn:hc-hom-0}.
(see, e.g., \cite[p.220]{MR855239} for further details).
Fix a Cartan subalgebra
$\mathfrak{t} \subseteq \mathfrak{g}$ and a corresponding
decomposition
$\mathfrak{g}_{\C} = \mathfrak{n}^-_{\C} \oplus \mathfrak{t}_{\C}
\oplus \mathfrak{n}_{\C}$.
Let $\mathcal{H}$ denote the universal enveloping algebra
(equivalently, symmetric algebra) of $\mathfrak{t}_\mathbb{C}$.
One has the decomposition
\begin{equation}\label{eqn:decomp-for-hc-hom}
  \mathfrak{U}= \left(
    \mathfrak{n}_{\C}^{-} \mathfrak{U} +
    \mathfrak{U} 
    \mathfrak{n}_{\C} \right) \oplus \mathcal{H}.
\end{equation}
Let $\rho \in \mathfrak{t}^*$ denote the half-sum of positive
roots
and
$\sigma : \mathcal{H} \rightarrow \mathcal{H}$
the algebra automorphism
extending
$\mathfrak{t}_\mathbb{C} \ni t \mapsto t - \rho(t) 1_{\mathcal{H}}$.
Given $z \in \mathcal{Z}$ with component $z_T \in \mathcal{H}$
relative to the decomposition \eqref{eqn:decomp-for-hc-hom},
Harish--Chandra
defines
the element
$\gamma(z)
:= \sigma(z_T) \in \mathcal{H}$.
This element turns out to be
Weyl-invariant,
and thus identifies,
via
\eqref{eqn:t-mod-W-vs-g-mod-G},
with a regular function on
$[\ggC^*]$.

\subsection{Basics on infinitesimal
  characters}\label{sec:infin-char-unit}
 
Let $\pi$ be a $\mathfrak{U}$-module on which $\mathfrak{Z}$ acts
by scalars.
For instance,
this happens when $\pi$ comes from an irreducible representation
of $G$.
The \emph{infinitesimal character} $\lambda_\pi \in [\mathfrak{g}_\mathbb{C}^*]$
is
then defined by the property:
each $z \in \mathfrak{Z}$
acts on $\pi$ by the scalar $\gamma(z)(\lambda_\pi)$.

For any $\mathfrak{U}$-module $\pi$, we may define the dual
module $\pi^*$ and the complex conjugate module
$\overline{\pi }$.
We write $\pi^+ := \overline{\pi^*}$
for the conjugate dual.
We note that if $\pi$ comes from a unitary representation
of $G$, then $\pi \cong \pi^+$.
If $\mathfrak{Z}$ acts on $\pi$ by scalars,
then it also acts by scalars on the modules $\pi^*,
\overline{\pi }$ and $\pi^+$,
whose infinitesimal characters may be described as follows.
\begin{lemma*}
  $\lambda_{\pi^*} = - \lambda_\pi$
  and
  $\lambda_{\overline{\pi}} = \overline{\lambda_\pi}$ and
  $\lambda_{\pi^+} = - \overline{\lambda_\pi}$.
\end{lemma*}
\begin{proof}
  This is presumably well-known, but we were unable to locate a
  reference.  It suffices to prove for each
  $\lambda \in [\mathfrak{g}_\mathbb{C}^*]$ that the required
  identities hold for \emph{some} $\mathfrak{U}$-module $\pi$
  with $\lambda_\pi = \lambda$.
  Let
  $\mathfrak{g}_s \subseteq \mathfrak{g}_\mathbb{C}$ be a split
  real form, with split Cartan
  $\mathfrak{t} \leq \mathfrak{g}_s$, and fix a reductive group
  $G_s$ with Lie algebra $\mathfrak{g}_s$.  Lift $\lambda$ to a
  representative $\lambda \in \mathfrak{t}_\mathbb{C}^*$, and
  form the corresponding normalized principal series
  representation $I(\lambda)$ of $G_s$ via induction from some
  Borel containing $\exp(\mathfrak{t})$.
  Then $I(\lambda)$ is a $\mathfrak{U}$-module
  whose infinitesimal character
  is the image of $\lambda$
  (by, e.g., \cite[Prop 8.22]{MR855239}),
  while $I(\lambda)$
  has dual $I(-\lambda)$
  (by, e.g., calculations as in \cite[p.170]{MR855239})
  and complex conjugate $I(\overline{\lambda})$ (by construction).
\end{proof}

In particular, by
\eqref{eqn:real-form-of-git-quot-via-conjugation}, we
obtain:
\begin{corollary*}
  If $\pi$ is an irreducible unitary representation
  of $G$,
  then $\lambda_\pi \in [i \mathfrak{g}^*]$.
\end{corollary*}

\subsection{Langlands classification} \label{ss:LC}
Recall that a representation of $G$ is
\emph{tempered} if it is
unitarizable and weakly contained in $L^2(G)$,
or equivalently,
if its central character is unitary
and the matrix coefficients of its $K$-finite vectors
(with $K$ some maximal compact subgroup of $G$)
belong to  $L^{2+\eps}$ modulo the center (see \cite{MR946351}).
The Langlands classification (see \cite[Thm 8.54]{MR855239}) asserts
that
for each irreducible representation\footnote{in the sense of $(\mathfrak{g}, K)$-modules, but this includes unitary representations, see \cite[Cor 9.2]{MR855239}.}
$\pi$ of $G$,
there is a unique $G$-conjugacy class of
pairs $(P,\sigma)$,
consisting of a parabolic subgroup $P$ of $G$
and a
representation $\sigma$ of the Levi quotient $M$,
so that
\begin{itemize}
\item $\pi$ is the unique irreducible quotient of the induced representation
  ${i}_P^G \sigma $,
\item $\sigma$ is tempered on the derived group of $M$, and
\item the absolute value  of the central character of $\sigma$ is  strictly dominant. 
\end{itemize}

The infinitesimal characters of $\sigma$ and $\pi$ coincide with reference to the natural map
$[\mathfrak{m}_\mathbb{C}^*] =  \mathfrak{m}_{\C}^*\git M_{\C} \rightarrow [\ggC^*]$
(see \cite[Prop 8.22]{MR855239}).

\subsection{Infinitesimal criterion for temperedness}
\label{sec:crit-for-temperedness}
The infinitesimal characters
of non-tempered representations are located near irregular
elements;
for lack of a reference, we record the proof.

\begin{lemma*}
  For each compact subset $\Omega$  of $[i\mathfrak{g}^*]_{\reg}$
  there exists $\h_0 > 0$
  so that
  for each $\h
  \in (0,\h_0)$,
  every irreducible unitary representation $\pi$
  with $\h \lambda_\pi \in \Omega$
  is tempered.
\end{lemma*}
\begin{proof}
  First, fix a Cartan subalgebra $\mathfrak{t}_\mathbb{C}$ of
  $\mathfrak{g}_\mathbb{C}$ and let $\lambda \in \h^{-1}
  \Omega$.  
  Since
  $\lambda$ is regular, any preimage
  $\tilde{\lambda} \in \mathfrak{t}_\mathbb{C}$ under
  \eqref{eqn:t-mod-W-vs-g-mod-G} satisfies
  $w \cdot \tilde{\lambda} \neq \tilde{\lambda}$ for all
  nontrivial elements $w$ of the Weyl group for
  $\mathfrak{t}_\mathbb{C}$.
  Since $\Omega$ is compact,
  it follows that
  \begin{equation}\label{eqn:lower-bound-difference-with-weyl-translate}
    |w \cdot \tilde{\lambda} - \tilde{\lambda}|
    \geq c \h^{-1}
  \end{equation}
  for some $c > 0$ depending only upon
  $\mathfrak{t}_\mathbb{C}$ and $\Omega$.

  Next, let $\pi$ be an irreducible unitary representation
  with $\h \lambda_\pi \in \Omega$.
  We may realize $\pi$ as the Langlands quotient of $i_P^G \sigma$
  for some $(P,\sigma)$ as in \S\ref{ss:LC}.
  Let $|\omega_\sigma|$ denote the (dominant) absolute value
  of the central character of $\sigma$.
  Assume that $\pi$ is non-tempered.
  Then $P \neq G$ and $|\omega_\sigma|$ is strictly dominant,
  and in particular nontrivial.
  
  Unitarity also imposes a constraint on $|\omega_{\sigma}|$: Let $\rho_P$ be
  the half-sum of positive roots for $\mathfrak{a}$ on the
  unipotent radical of $P$.  It follows from boundedness of matrix
  coefficients that $\rho_P |\omega_{\sigma}|^{-1}$, considered as a character
  $\mathfrak{a} \rightarrow \C^*$, is bounded above on the
  dominant cone \cite[Chapter XVI, \S 5, problems
  6-7]{MR855239}. This, together with the condition that $|\omega_{\sigma}|$
  is dominant, confines $|\omega_{\sigma}|$ to a {\em compact} subset of
  $\mathfrak{a}^*$.
   (Compare with \cite[Prop 7.18]{MR1670073}.)
  
  Now, passing to a smaller parabolic if necessary, we may
  assume that $\pi$ is a quotient
  of $i_P^G \sigma$, where
  \begin{itemize}
  \item
    the restriction of $\sigma$
    to the derived group of the Levi $M$ of $P$
    belongs to the discrete series, and
  \item the absolute central character $|\omega_{\sigma}|$
    is
    nontrivial and confined to a compact subset of $\mathfrak{a}^*$.
  \end{itemize}

  Let $\mathfrak{a}$ denote the center of $\mathfrak{m}$,
  and $\mathfrak{m}^0$ the derived subalgebra.
  We then have the splitting $\mathfrak{m} = \mathfrak{a} \oplus
  \mathfrak{m}^0$,
  which induces
  a bijection
  \[ [\mathfrak{m}_{\C}^*] \simeq  \mathfrak{a}_{\C}^* \times
  [\mathfrak{m}^{0*}_{\C}].\]
  By our assumptions
  on $\sigma$,
  its infinitesimal
  character decomposes as
  \[\lambda_\sigma = (\kappa + \mu,\nu),
  \]
  where
  $\kappa \in \mathfrak{a}^*$,
  $\mu \in i \mathfrak{a}^*$,
  and
  $\nu \in [i \mathfrak{m}^{0*}]$;
  moreover,
  $\kappa$ is nonzero and confined
  to a compact set.

  By the classification of discrete series
  \cite[Thm 9.20, Thm 12.21]{MR855239}), 
  the parameter $\nu$ comes from an imaginary
  parameter for some (compact) Cartan subgroup of $M^0$.
  We may thus find a Cartan subgroup $T$ of $G$,
  contained in $M$ and containing its center,
  so that $\lambda_\pi$ is the image
  of $\lambda + \kappa$
  for some $\lambda \in i \mathfrak{t}^*$;
  here we identify $\kappa \in \mathfrak{a}^*$
  with its pullback
  to $\mathfrak{t}_\mathbb{C}^*$.
  The unitarity of $\pi$ implies
  that $\lambda_\pi = - \overline{\lambda_\pi}$;
  since $\overline{\lambda} = -\lambda$
  and $\overline{\kappa} = \kappa$,
  it follows that
  \[\lambda + \kappa = w \cdot (\lambda - \kappa)
  \]
  for some $w$ in the Weyl group of $\mathfrak{t}_\mathbb{C}$.
  Since $\kappa \neq 0$,
  the element $w$ is nontrivial.

  In summary,
  we have shown that there exists
  a Cartan subgroup $T$ of $G$
  (which we may assume taken in a finite set of
  conjugacy representatives),
  a lift of $\tilde{\lambda}_\pi$
  of $\lambda_\pi$
  to $\mathfrak{t}_\mathbb{C}^*$,
  and a nontrivial Weyl group element
  $w$ for $\mathfrak{t}_\mathbb{C}$
  so that
  \[
  w \cdot \tilde{\lambda}_\pi
  \in  \tilde{\lambda}_\pi
  + C
  \]
  for some compact $C \subseteq \mathfrak{t}^*_\mathbb{C}$.
  But if $\h$ is small enough,
  this contradicts
  \eqref{eqn:lower-bound-difference-with-weyl-translate}.
\end{proof}

\subsection{Norms} \label{ss:norms}
Let us introduce a norm $|\cdot|$ on $[\ggC]^*$, i.e. a continuous
nonnegative function such that 
 $|t \lambda| = |t| |\lambda|$ for $t \in \C^*$, 
and that $|. |$ vanishes only at the origin. 
Any two such choices are bounded above and below in terms of each other:
\[ | \cdot |_1 \asymp |\cdot |_2\]
and so the explicit choice will not matter.

For example, 
choosing coordinates
$\lambda = (p_1,\dotsc,p_n)$
on $[\mathfrak{g}_{\C}^*]$
as in \S\ref{sec:quotient-affine},
where $p_j$ has degree $d_j \geq 1$, the function 
\[|\lambda| :=
\max_j |p_j|^{1/d_j}\]
gives a norm. Alternately, identifying $[\mathfrak{g}_{\C}^*]$
with the quotient of a Cartan subalgebra by the Weyl group,
any Weyl-invariant norm on the Cartan subalgebra gives
a norm on $[\mathfrak{g}_{\C}^*]$.

Now let $\pi$ be an irreducible unitary representation of $G$.
Recall the positive-definite densely-defined self-adjoint
operator $\Delta := \pi(\Delta)$ on $\pi$ as defined in
\S\ref{sec:the-operator-delta}.
Then the norm $|\lambda_\pi|$ of the infinitesimal character
$\lambda_\pi$ of $\pi$ gives a reasonable notion
of a ``norm of $\pi$:''
\begin{lemma}
  There
  is an eigenvalue of $\Delta$ on $\pi$
  of size $\O(1 +  |\lambda_{\pi}|^2)$.
  The implied constant depends at most upon $\mathbf{G}$
  and the choice of norm.
\end{lemma}
\begin{proof} 
  The assertion does not depend on the basis $\mathcal{B}$
  of $\mathfrak{g}$ used to define $\Delta$.  Choose a
  Cartan decomposition
  $\mathfrak{g} = \mathfrak{k} \oplus \mathfrak{p}$ and let
  $\{y_i\}$ and $\{z_j\}$ be orthonormal bases for $\mathfrak{k}$
  and $\mathfrak{p}$, respectively, as defined with respect to the
  (possibly-negated) Killing form.  Take
  $\mathcal{B} := \{y_i\} \cup \{z_j\}$, so that
  $\Delta = 2 \Delta_K - 1 - \mathcal{C}$ with
  $\Delta_K := 1 - \sum y_i^2$ and
  $\mathcal{C} := -\sum y_i^2 + \sum z_j^2$.
  Recall
  that
  $\mathcal{C}$ (the Casimir operator) defines a quadratic element
  of $\mathfrak{Z}$ whose eigenvalue on $\pi$ is thus
  $\O(1 + | \lambda_\pi|^2)$.
  We reduce to verifying
  that
  $\Delta_K$ has an eigenvalue of size $\O(1 +
  |\lambda_{\pi}|^2)$.

  For this we use Vogan's theory of minimal $K$-types
  (see \cite{MR519352} or \cite[\S XV]{MR855239}). 
  Let $\mu$ be a minimal $K$-type of $\pi$. Then  Vogan shows that 
  there exists  a parabolic subgroup $P$ and a
  (possibly not unitary)
  relative discrete series representation $\sigma$
  of the Levi quotient such that:
  \begin{itemize}
  \item $\pi$ is a subquotient of an induced representation $I_P^G(\sigma)$, {\em and}
  \item  $\mu$ is a minimal $K$-type of $I_P^G(\sigma)$.
  \end{itemize}

  This permits us to reduce the question where $\pi$ is induced from a discrete series,
  and in turn to the case of discrete series. In that case, the
  desired result follows from Blattner's formula
  (see \cite{MR0396855}
  or \cite[Thm 12.26c]{MR855239}).
 \end{proof}

\subsection{Harish--Chandra versus
  symmetrization}\label{sec:harish-chandra-vers}
\label{sec:interest-in-harish-sym}
The symmetrization map
(cf. \S\ref{sec:quantize-polynomials})
is the linear isomorphism
$\Sym(\mathfrak{g}_\mathbb{C})^G \stackrel{\mathrm{sym}}{\longrightarrow} 
\mathfrak{Z}$
given by averaging over permutations.
We consider the composition
\begin{equation}\label{eq:compose-HC-sym}
  \Sym(\mathfrak{g}_\mathbb{C})^G \stackrel{\mathrm{sym}}{\longrightarrow} 
  \mathfrak{Z} \stackrel{\gamma}{\rightarrow}
  \Sym(\mathfrak{g}_\mathbb{C})^G.
\end{equation}
\begin{lemma*}
  Let $n \in \mathbb{Z}_{\geq 1}$,
  and let $p \in   \Sym(\mathfrak{g}_\mathbb{C})^G$
  have order $\leq n$.
  Then
  \begin{equation}
    \label{HC versus sym}
    \gamma \circ \sym(p) - p \mbox{ has order $\leq
      n-1$}.
  \end{equation}
\end{lemma*}
\begin{proof}
  We use that the component
  $\sym(p)_T \in \mathcal{H}$ of $\sym(p)$ with reference to
  \eqref{eqn:decomp-for-hc-hom} coincides modulo terms of degree
  $\leq n-1$ with the restriction of $p$ via
  $\mathfrak{t}_{\C}^* \hookrightarrow \mathfrak{g}_{\C}^*$.
\end{proof}

\section{Localizing}
  Here we
record some results of the following theme: if a vector $v$ is
``microlocalized'' at a point $\xi \in \mathfrak{g}^\wedge$
(\S\ref{sec:intro-op-calc}), and a symbol $a$ is supported away
from that point, then $\Opp(a) v$ is negligibly small.
The method of proof
-- to approximately divide one symbol
by another -- is ubiquitous in microlocal analysis.
The
main result is the lemma of \S\ref{sec:some-decay-2},
which we
apply below in a few places
(\S\ref{sec:priori-estimates-rel-char},
\S\ref{sec:main-result-inv-branch-arch},
\S\ref{sec:arch-weyl-counting},
\S\ref{sec:trunc-spectr-decomp-overview})
as an \emph{a priori} estimate.

In
\S\ref{sec:division} and \S\ref{sec:some-decay}, $\pi$ is an
$\h$-dependent unitary representation of a unimodular Lie group
$G$; in \S\ref{sec:some-decay-2}, we specialize further.  
\subsection{Division}\label{sec:division}
\begin{lemma*}
  Fix $0 \leq \delta_1, \delta_2 < 1/2$
  and $M, N \in \mathbb{Z}_{\geq 0}$.
  Set $\delta := \max(\delta_1,\delta_2)$.
  For each $a \in S^{-\infty}_{\delta_1}$
  and $q \in S^{\infty}_{0}$
  for which
  \[
  a(\xi) \neq 0 \implies
  |q(\xi)| \geq \h^{\delta_2}
  \]
  there
  exist $b, b' \in \h^{-\delta_2} S^{-\infty}_{\delta}$
  so that
  \[
  \Opp_{\h}(a)
  \equiv \Opp_{\h}(q) \Opp_{\h}(b)
  \mod{
    \h^N \Psi_\delta^{-M},
  }
  \]
  \[
  \Opp_{\h}(a)
  \equiv \Opp_{\h}(b')\Opp_{\h}(q)
  \mod{
    \h^N \Psi_\delta^{-M}
  }.
  \]
\end{lemma*}
\begin{remark*}
  One can extend this result to $M = N = \infty$
  via ``Borel summation''
  as in \cite[Prop 18.1.3]{MR2304165},
  but we have no need to do so.
\end{remark*}
\begin{proof}
  (The idea of the proof
  was applied already in \S\ref{sec:case-s-m-positive}.)
  We construct $b$;
  one may similarly construct $b'$.
  Set $b_0 := a/q$.
  By the quotient rule,
  we see that
  $b_0 \in \h^{-\delta_2} S^{-\infty}_{\delta}$.
  We now inductively construct
  $b_j$ for $j \geq 1$
  so that
  \[
  q \star_{\h} \sum_{j \geq 0} \h^j b_j
  \sim a
  \]
  in the formal sense (i.e., comparing coefficients of powers of $\h$).
  Explicitly,
  \[
  b_0 := \frac{a}{q},
  \quad b_1 := \frac{- q \star^1 b_0}{q}, \quad
  b_2 := \frac{- q \star^2 b_0 + q \star^1 b_1}{q},
  \]
  and so on.
  We see by induction that $\supp(b_j) \subseteq \supp(a)$ and that
  \begin{equation}
    b_j \in \h^{-\delta_2 - 2 \delta j}
    S^{-\infty}_{{\delta}}.
  \end{equation}
  We now take
  $b: = \sum_{j \leq 0 < J} \h^j b_j$,
  with $J \in \mathbb{Z}_{\geq 0}$
  large but
  fixed,
  and appeal to \eqref{eqn:comp-with-remainder-J}.
\end{proof}

\subsection{Localizing near the locus of a symbol}
\label{sec:some-decay}
Suppose now given an $\h$-dependent $\h$-uniformly continuous map
\[
\mathcal{H} : \Psi^{-\infty} \rightarrow \mathbb{C}
\]
and a symbol $p \in S^\infty_{0}$
with the property that
there is an ($\h$-dependent) scalar $\lambda \in \mathbb{C}$ so that
either of the following conditions hold:
\begin{equation}\label{eqn:H-p-eigen-1}
  \mathcal{H}(\Opp_{\h}(p) T) = \lambda \mathcal{H}(T) \quad \text{for all $T \in
    \Psi^{-\infty}$},
\end{equation}
\begin{equation}\label{eqn:H-p-eigen-2}
  \mathcal{H}(T \Opp_{\h}(p)) = \lambda \mathcal{H}(T) \quad \text{for all $T \in
    \Psi^{-\infty}$}.
\end{equation}
We will verify that $\mathcal{H}(\Opp_{\h}(a))$
is small if $a$ is supported
away from the vanishing locus of $p$.
Let us first
choose (as we may)
an $\h$-uniformly continuous seminorm
$\nu$ on $\Psi^{-\infty}$
so that $|\mathcal{H}| \leq \nu$.
% $M \in \mathbb{Z}_{\geq 0}$
% large enough that $\mathcal{H}$
% extends to an $\h$-dependent $\h$-uniformly continuous map
% \[
% \mathcal{H} : \Psi^{-M} \rightarrow \mathbb{C}.
% \]
\begin{lemma*}
  Fix $\delta \in [0,1/2)$ and $N \in \mathbb{Z}_{\geq 0}$.
  Let $a \in S^{-\infty}_{\delta}$.
  Then
  \begin{equation}\label{eq:some-decay-1}
    \mathcal{H}(\Opp_{\h}(a))
    \ll
    \h^{-M} \langle \lambda  \rangle^{-N},
  \end{equation}
  where $M$ is fixed large enough in terms of
  the seminorm $\nu$.
  If moreover
  \[
  a(\xi) \neq 0
  \implies
  |p(\xi) - \lambda| \geq \h^{\delta},
  \]
  then
  \begin{equation}\label{eq:some-decay-2}
    \mathcal{H}(\Opp_{\h}(a))
    \ll \h^N
    \langle \lambda  \rangle^{-N}.
  \end{equation}
\end{lemma*}
\begin{proof}
  We assume \eqref{eqn:H-p-eigen-1};
  a similar proof applies
  under \eqref{eqn:H-p-eigen-2}.
  By the lemma of \S\ref{sec:h-dependence-Psi-m},
  we may choose $M$ so that
  the restriction of $\nu$ to $\Psi^{-M}_\delta$
  is continuous.
  We have
  \[
  \mathcal{H}(\Opp_{\h}(a))
  =
  \lambda^{-N}
  \mathcal{H}(\Opp_{\h}(p)^N \Opp_{\h}(a)),
  \]
  \[\Opp_{\h}(p)^N \Opp_{\h}(a)
  \subseteq 
  \Opp_{\h}(S^{-\infty}_{{\delta}}) + \h^{\infty}
  \Psi^{-\infty},
  \]
  \[
  \Opp_{\h}(S^{-\infty}_{{\delta}})
  \subseteq \Opp_{\h}(S^{-M}_{{\delta}})
  \subseteq \h^{-M} \Psi^{-M}_\delta.
  \]
  By applying these both with the given value of $N$ and with $N = 0$,
  we obtain \eqref{eq:some-decay-1}.
  We turn to \eqref{eq:some-decay-2}.
  Fix $\eps > 0$ with $\delta + \eps < 1/2$.
  By applying    \eqref{eq:some-decay-1}
  with $N$ with sufficiently large,
  we reduce to the case that $|\lambda| \leq \h^{-\eps}$.
  We must verify then that
  \begin{equation}\label{eqn:required-estimate-for-last-part-of-proof-of-symbol-vanishing-concentration}
    \mathcal{H}(\Opp_{\h}(a)) \ll \h^N.
  \end{equation}
  Set $q := \h^\eps(p - \lambda) \in S^{\infty}_{0}$.
  We have $|q| \geq \h^{\delta+\eps}$ on the support of $a$.
  By \S\ref{sec:division},
  we may write
  $\Opp_{\h}(a) \equiv \Opp_{\h}(q) \Opp_{\h}(b) \mod{ \h^N
    \Psi_\delta^{-M} }$
  for some $b$.
  Since
  $\mathcal{H}(\Opp_{\h}(q) \Opp_{\h}(b)) = 0$
  and $\mathcal{H}$ induces
  a continuous map $\Psi_\delta^{-M} \rightarrow \mathbb{C}$,
  the required estimate
  \eqref{eqn:required-estimate-for-last-part-of-proof-of-symbol-vanishing-concentration}
  follows.
\end{proof}

\subsection{Localizing near an infinitesimal character}\label{sec:some-decay-2}
We now establish an analogue of \S\ref{sec:some-decay}
which will be very useful in applications.
Fix an inclusion
$\mathbf{H} \hookrightarrow  \mathbf{G}$
of reductive groups over $\mathbb{R}$.
Let $\pi$ and $\sigma$ be $\h$-dependent
irreducible unitary representation of $G$ and $H$, respectively.
Let
\[
  \mathcal{H} : \Psi^{-\infty}(\pi) \rightarrow \mathbb{C}
\]
be an $\h$-dependent $\h$-uniformly continuous map
which factors as
an $(H \times H)$-equivariant
$\h$-uniformly continuous composition
\[
\Psi^{-\infty}(\pi)
\rightarrow
\Psi^{-\infty}(\sigma) \rightarrow \mathbb{C}.
\]
We quantify
the rescaled frequencies of $\pi$ and $\sigma$
via their infinitesimal characters:
\[
  \langle \h \lambda_\pi  \rangle := (1 +
  |\h \lambda_\pi|^2)^{1/2},
  \quad 
  \langle \h \lambda_\sigma  \rangle := (1 +
  |\h \lambda_\sigma|^2)^{1/2}.
\]
Define $\Opp : S^m := S^m(\mathfrak{g}^\wedge)
\rightarrow \Psi^m := \Psi^m(\pi)$ as usual.
Let $a \in S^{-\infty}$.
We will verify
that $\mathcal{H}(\Opp_{\h}(a))$ is small unless
$a$ is supported on
elements $\xi \in \mathfrak{g}^\wedge$
for which
$([\xi], [\xi|_{\mathfrak{h}}]) \approx (\h \lambda_\pi, \h
\lambda_\sigma)
\in [\mathfrak{g}^\wedge] \times [\mathfrak{h}^\wedge]$.
More precisely,
let us choose an $\h$-uniformly
continuous
seminorm $\nu$ on $\Psi^{-\infty}(\pi)$
so that $|\mathcal{H}| \leq \nu$.
% $M \in \mathbb{Z}_{\geq 0}$
% so that $\mathcal{H}$ factors $\h$-uniformly continuously
% through $\Psi^{-M}(\pi)$.
Then:
\begin{lemma*}
  Fix $\delta \in [0,1/2)$
  and $N \in \mathbb{Z}_{\geq 0}$.
  Fix $M \geq 0$
  large enough in terms of $\nu$.
  Let $a \in S^{-\infty}_{{\delta}}$.
  Then $\mathcal{H}(\Opp_{\h}(a))$
  satisfies the ``a priori estimate''
  \begin{equation}\label{eq:some-decay-sym-1}
    \mathcal{H}(\Opp_{\h}(a)) \ll
    \h^{-M} \langle  \h \lambda_\pi \rangle^{-N} \langle  \h \lambda_\sigma \rangle^{-N}.
  \end{equation}
  Suppose now that $a$ is supported on the complement
  of
  \begin{equation}\label{eq:dist-delta-2-from-h-lambda-pi}
    \{\xi \in \mathfrak{g}^\wedge :
    \dist([\xi],\h \lambda_\pi)
    \leq \h^{\delta}
    \text{ and }
    \dist([\xi|_{\mathfrak{h}}],\h \lambda_\sigma)
    \leq \h^{\delta}\},
  \end{equation}
  with $\dist$ the Euclidean distance
  function defined by the coordinates fixed in \S\ref{sec:infin-char}.

  Then $\mathcal{H}(\Opp_{\h}(a))$
  is negligible:
  \begin{equation}\label{eq:some-decay-sym-2}
    \mathcal{H}(\Opp_{\h}(a))
    \ll \h^N
    \langle \h \lambda_\pi  \rangle^{-N}
    \langle \h \lambda_\sigma  \rangle^{-N}.
  \end{equation}
\end{lemma*}
\begin{proof}
  We
  denote as in \S\ref{sec:quotient-affine} by  $R_{\mathbb{C}} \cong
  \Sym(\mathfrak{g}_{\mathbb{C}})^G$
  and
  $R = \Sym(i \mathfrak{g})^G$
  the rings of $G$-invariant polynomials
  on $i \mathfrak{g}^*$
  taking complex and real values,
  respectively.
  Recall
  (from \S\ref{sec:quantize-polynomials}, \eqref{eqn:hc-hom-0})
  that
  for
  $p \in R_\mathbb{C}$,
  \begin{equation}\label{eq:action-opp-p-HC-sym}
    \Opp(p)
    \text{
      acts
      by 
      the scalar }
    (\gamma \circ \sym(p))(\lambda_{\pi}).
  \end{equation}
  For $p \in R$,
  we denote by $p' \in R$ the element
  for which
  \begin{equation}\label{eqn:sym-HC-p-vs-p-prime}
    \gamma \circ \sym(p_{\h}') = p_{\h}.
  \end{equation}
  Regarding 
  $p'$ as a polynomial symbol on
  $\mathfrak{g}^\wedge$, we see that
  \begin{equation}\label{eqn:how-opp-of-s-prime-acts}
    \Opp_{\h}(p')
    \text{ acts by the scalar }
    p_{\h}(\lambda_\sigma ) = p(\h \lambda_\sigma).
  \end{equation}
  On the other hand,
  if we fix $p$ and regard $p'$ as an $\h$-dependent
  polynomial symbol,
  then the lemma of \S\ref{sec:harish-chandra-vers}
  gives
  \begin{equation}\label{eq:p-and-p-prime-are-close}
    p' = \O(1)\text{ and }  p' = p + \O(\h);
  \end{equation}
  more precisely,
  \eqref{eq:p-and-p-prime-are-close}
  says that
  $p ' \in \h^0 S^{\infty}_{0}$ and
  $p ' - p  \in \h^1 S^{\infty}_{0}$.

  Fix a set $\{p\} = \{p_G\} \sqcup \{p_H\}$,
  where
  \begin{itemize}
  \item $p_G$ runs over a system of generators
    for the ring $\Sym(i \mathfrak{g})^G$, corresponding to the coordinate
    functions defining the distance function on
    $[\mathfrak{g}^\wedge]$, and
  \item $p_H$ runs over a
    similar system for  $\Sym(i \mathfrak{h})^H$.
  \end{itemize}
  Let
  the assignment $p \mapsto p'$ be as above,
  applied either to $H$ or
  to $G$.  For each such $p$, our assumptions concerning
  $\mathcal{H}$ imply that
  \begin{equation}\label{eq:actual-sym-equiv}
    \mathcal{H}(\Opp_{\h}(p') T)
    = p(\h \lambda_\sigma)
    \mathcal{H}(T)
  \end{equation}
  for all $T \in \Psi^{-\infty}$.
  We have
  $|\h \lambda_\pi | + | \h \lambda_\sigma| \asymp \max_p |p(\h \lambda_\sigma)|$,
  so the first estimate \eqref{eq:some-decay-sym-1}
  follows immediately from \eqref{eq:actual-sym-equiv}
  and
  \S\ref{sec:some-decay}.
  In verifying
  \eqref{eq:some-decay-sym-2},
  we may thus suppose that $|\h \lambda_\pi | + |\h \lambda_\sigma | \leq
  \h^{-\eps}$
  for any fixed $\eps > 0$.
  By \eqref{eq:p-and-p-prime-are-close},
  we have
  \[
    |p(\xi) - p'(\xi)| \ll \h^{1-\eps},
  \]
  so that for small enough $\h$,
  \begin{align*}
    \dist([\xi|_{\mathfrak{h}}],\h \lambda_\sigma)
    > \h^{\delta}
    &\implies
      \max_p
      \left\lvert p(\xi) - p(\h \lambda_\sigma) \right\rvert
      \gg \h^{\delta}
    \\
    &\implies
      \max_p
      \left\lvert p'(\xi) - p(\h \lambda_\sigma) \right\rvert
      > \h^{\delta + \eps},
  \end{align*}
  say,
  and similarly
  if
  $\dist([\xi],\h \lambda_\pi) > \h^{\delta}$.
  Having fixed $\eps$ small enough,
  we may suppose that $\delta + \eps < 1/2$.
  If $a$ is supported on the complement of
  \eqref{eq:dist-delta-2-from-h-lambda-pi},
  then we may
  decompose it into pieces indexed by $p$ supported on the sets
  $\{\xi : | p'(\xi) - p(\h \lambda_\sigma) | > \h^{\delta + \eps} \}$.
  We conclude by applying
  \S\ref{sec:some-decay} to each such piece.
\end{proof}

\section{Regular coadjoint
  multiorbits}\label{sec:canonical-symplectic-form-2}
$G$ is the set of real points
of a real reductive group $\mathbf{G}$,
with notation as above.

\subsection{Notation and terminology}
For each infinitesimal character $\lambda \in [\mathfrak{g}^\wedge]$, we may form the
fiber \index{$\mathcal{O}^\lambda$}
$\mathcal{O}^\lambda := \{\xi \in \mathfrak{g}^\wedge : [\xi] =
\lambda \}$.
For example, if $\lambda = [0]$, then $\mathcal{O}^\lambda$ is
the nilcone $\mathcal{N}$.\index{nilcone $\mathcal{N}$}
As noted already in \S\ref{sec:canonical-symplectic-form},
each such fiber $\mathcal{O}^{\lambda}$
consists of a uniformly
bounded finite number of $G$-orbits.
We recall that a \emph{coadjoint multiorbit}
$\mathcal{O} \subseteq \mathfrak{g}^\wedge$ 
is a
 $G$-invariant set
contained in $\mathcal{O}^{\lambda}$
for some $\lambda$,
and that $\mathcal{O}$ is \emph{regular}
if it consists of regular elements;
then
\[
  \text{$\mathcal{O}$ is regular}
  \iff
  \text{$\mathcal{O}$ is relatively open} 
  \iff
  \text{$\mathcal{O}$ has maximal dimension},
\]
where relatively open is with respect to $\mathcal{O}^\lambda$,
and maximal means with respect to all coadjoint orbits.

Recall that to each coadjoint orbit $\mathcal{O}$ we may attach
its normalized symplectic measure $\omega := \omega_{\mathcal{O}}$ on
$\mathcal{O}$.
Let $a \in C_c(\mathfrak{g}^\wedge)$.
By a result of Rao \cite{MR0320232},
the integral $\int_{\mathcal{O}} a \, d \omega$
converges.
Hence $\omega$ may be regarded as a measure on
$\mathfrak{g}^\wedge$.

\subsection{Topology}\label{sec:topology-multiorbits}
We temporarily denote by $\mathfrak{R}$ 
the
set
of regular coadjoint multiorbits
$\mathcal{O} \subseteq \mathfrak{g}^\wedge_{\reg}$.
We equip $\mathfrak{R}$
with the
topology induced from the inclusion
\[
  \mathcal{O} \mapsto \omega_\mathcal{O} \in \{\text{locally
    finite Radon measures on } \mathfrak{g}^\wedge\},
\]
where we
endow the target with the weak-* topology.
Thus a sequence
$\mathcal{O}_j$ of regular coadjoint multiorbits converges to
$\mathcal{O}$ in $\mathfrak{R}$ if the corresponding symplectic measures
$\omega_{\mathcal{O}_j}$ tend to $\omega_\mathcal{O}$.
We note that the infinitesimal character
map $\mathfrak{R} - \{\emptyset\} \rightarrow [\mathfrak{g}^\wedge]$
is continuous,
and has finite fibers.

We note that the topology on $\mathfrak{R}$ is Hausdorff:
  the topology on the target of the above inclusion is
  Hausdorff, and the map
  $\mathcal{O} \mapsto \omega_{\mathcal{O}}$ is injective in
  view of the regularity of $\mathcal{O}$.

It is a nontrivial fact
that the topology on $\mathfrak{R}$
may be described
more simply:
\begin{theorem}\label{thm:describe-topology-on-R}
  Let $\mathcal{O}_j \in \mathfrak{R}$ be a sequence
  of regular coadjoint multiorbits.
  Set
  \[
  \mathcal{O} := \{\xi \in \mathfrak{g}^\wedge_{\reg}:
  \xi_j \rightarrow \xi \text{ for some } \xi_j \in
  \mathcal{O}_j\}.
  \]
   \begin{enumerate}[(i)]
  \item If $\mathcal{O}_j$
    has a nonempty limit $\mathcal{O}_{\lim} \in \mathfrak{R}$,
    then $\mathcal{O}_{\lim} = \mathcal{O}$.
  \item If $\mathcal{O}$ is nonempty
    and contains all regular subsequential limits
    of the sequence $\mathcal{O}_j$,
    then $\mathcal{O} \in \mathfrak{R}$
    and
    $\mathcal{O}_j \rightarrow \mathcal{O}$.
  \end{enumerate}
  \end{theorem}
This follows from arguments of Rossmann
\cite{MR587333,MR650378}, who also gives in \cite{MR650378}
some characterizations of when $\mathcal{O}$ is nonempty.
Since it does not appear to have been stated explicitly in the
above form, we outline the proof.
We consider first the
special case involving full preimages under the map
$\mathfrak{g}^\wedge_{\reg} \rightarrow
[\mathfrak{g}^\wedge]$.  To that end, for
$\lambda \in [\mathfrak{g}^\wedge]$, set
\[
  \omega_{\lambda} := \omega_{\mathcal{O}_{\reg}}^\lambda,
\]
where by convention $\omega_\lambda := 0$
if $\mathcal{O}_{\reg}^\lambda = \emptyset$.
\begin{lemma}\label{lem:cts-variation-omega-lambda}
  The measures
  $\omega_{\lambda}$
  vary continuously
  with respect to $\lambda \in [\mathfrak{g}^\wedge]$.
\end{lemma}
\begin{proof}
  Suppose that $\lambda_j \rightarrow \lambda$.  We must verify
  that $\omega_{\lambda_j} \rightarrow \omega_{\lambda}$.  By a
  diagonalization argument, we may assume that the $\lambda_j$
  are regular.  After passing to a subsequence, we may assume
  that there is a Cartan subalgebra $\mathfrak{h}$ of
  $\mathfrak{g}$ so that $\lambda_j$ is the image of
  $t_j \in \mathfrak{h}_{\reg}$, with the $t_j$ lying in the
  same connected component $C \subseteq \mathfrak{h}_{\reg}$ and
  having a limit $t \in \overline{C}
  \subseteq \mathfrak{h}$.
  The required conclusion in that case is stated explicitly in
  the second paragraph of \cite[p.59]{MR587333} and follows from
  \cite[Lem.  D]{MR587333} and the arguments of
  \cite[p.59-62]{MR587333}, which rely in turn upon results of
  Harish--Chandra.
\end{proof}

\begin{lemma}\label{lem:easy-case-near-regular-element}
  Fix $\xi_0 \in \mathfrak{g}^\wedge_{\reg}$,
  and let
  $\xi_0 \in U \subseteq \mathfrak{g}^\wedge_{\reg}$ be a sufficiently small open neighborhood.
  Let $\mathcal{O}$ be any regular coadjoint multiorbit.
 Clearly  $\omega_{\mathcal{O}} |_U$ is nonzero
  if and only if  
  $\mathcal{O} \cap U \neq \emptyset$. 
In that case, 
  $\omega_{\mathcal{O}}|_{U} = \omega_{G \cdot \xi}|_U$
  for any $\xi \in \mathcal{O} \cap U$.
  The family of measures
  $\omega_{G \cdot \xi}|_U$ varies continuously
  with respect to $\xi \in U$.
\end{lemma}
\begin{proof}
 
  We have for each $\lambda \in [\mathfrak{g}^\wedge]$
  that the intersection $U \cap \mathcal{O}^\lambda = U \cap
  \mathcal{O}_{\reg}^{\lambda}$
  is nonempty precisely when
  there is some
  $\xi \in U$ with
  $[\xi] = \lambda$.
  Since 
  the differential of the map
  $\mathfrak{g}^\wedge_{\reg} \rightarrow [\mathfrak{g}^\wedge]$
  is surjective at every point  \cite[Thm 0.1]{MR0158024},
  we have in that case  -- for sufficiently small $U$ -- that
  $U \cap \mathcal{O}^{\lambda} = U \cap (G \cdot \xi)$.
  The claims then follow from the prior lemma.
\end{proof}

\begin{proof}[Proof of theorem \ref{thm:describe-topology-on-R}]
  Let $\lambda_j = [\mathcal{O}_j] \in [\mathfrak{g}^\wedge]$
  denote the infinitesimal character of $\mathcal{O}_j$.

  Suppose first that the $\mathcal{O}_j$
  have a nonempty limit $\mathcal{O}_{\lim} \in \mathfrak{R}$,
  with infinitesimal character $\lambda := [ \mathcal{O}_{\lim}]$.
  Then $\lambda_j \rightarrow \lambda$;
  it follows that $\mathcal{O}$ is nonempty.
  It is therefore a nonempty regular
  coadjoint multiorbit with infinitesimal character $\lambda$.
  Using lemma \ref{lem:easy-case-near-regular-element},
  we see
  that
  \begin{equation}\label{eqn:limit-of-restrictions}
    (\text{restriction to $\mathfrak{g}^\wedge_{\reg}$ of
      $\omega_{\mathcal{O}_j}$})
    \rightarrow 
    (\text{restriction to $\mathfrak{g}^\wedge_{\reg}$ of
      $\omega_\mathcal{O}$}).
  \end{equation}
  Since $\mathcal{O}_j \rightarrow \mathcal{O}_{\lim}$,
  this forces $\mathcal{O}_{\lim} = \mathcal{O}$.

  Conversely,
  suppose that $\mathcal{O}$ is nonempty
  and contains the regular subsequential limits of the $\mathcal{O}_j$.
  We can find a sequence $\xi_j \in \mathcal{O}_j$
  converging to some $\xi \in \mathcal{O}$,
  so that $\lambda_j = [\xi_j]$ likewise
  tends to $\lambda := [\xi]$.
  Since every element of $\mathcal{O}$ arises
  in this way,
  we see that $\mathcal{O}$ is a
  nonempty regular
  coadjoint multiorbit
  with infinitesimal character $\lambda$.
  We aim to verify that $\mathcal{O}_j \rightarrow \mathcal{O}$.
  Using lemma \ref{lem:easy-case-near-regular-element},
  we 
  see that
  \eqref{eqn:limit-of-restrictions}
  holds.
  The remaining point is to understand what happens
  on $\mathfrak{g}^\wedge - \mathfrak{g}^\wedge_{\reg}$.
  Since
  $\omega_{\mathcal{O}_j} \leq \omega_{\lambda_j}$,
  we see from
  lemma
  \ref{lem:cts-variation-omega-lambda}
  that $\omega_{\mathcal{O}_j}$
  admits (after passing to a subsequence) a limit measure
  $\omega'$
  which is $G$-invariant and bounded
  by $\omega_{\lambda}$,
  hence is a nonnegative linear combination
  of the measures $\omega_{\mathcal{O} '}$
  attached to $G$-orbits
  $\mathcal{O}' \subseteq \mathcal{O}^{\lambda}_{\reg}$.
  By \eqref{eqn:limit-of-restrictions},
  we deduce that $\omega ' = \omega_{\mathcal{O}}$.
  Thus $\mathcal{O}_j \rightarrow \mathcal{O}$.
\end{proof}

\subsection{Bounds for symplectic measures}
\label{sec:bounds-symplectic-measures}

In what follows
we denote by
$\sup_{\mathcal{O}}$ a supremum
taken over
all nonempty regular coadjoint multiorbits.
\begin{lemma}
  For each $a \in C_c(\mathfrak{g}^\wedge)$,
  \begin{equation}\label{eqn:not-avni-aizenbud}
    \sup_{\mathcal{O}}
    | \int_{\mathcal{O}} a \, d \omega | < \infty.
  \end{equation}
\end{lemma}
\begin{proof}
  It will suffice to show that
  $\sup_{\lambda \in [\mathfrak{g}^\wedge]} |\omega_\lambda(a)|
  < \infty$,
  with $\omega_\lambda$ attached to
  $\mathcal{O}_{\reg}^{\lambda}$
  as in \S\ref{sec:topology-multiorbits}.
  Recall that $|\omega_\lambda(a)| < \infty$.
  The image in $[\mathfrak{g}^\wedge]$ of the support of $a$ is
  compact,
  so we may conclude
  via the continuity
  noted in lemma \ref{lem:cts-variation-omega-lambda}
  of \S\ref{sec:topology-multiorbits}.
\end{proof}

Combining this
with the homogeneity
property \eqref{homogeneity of
  symplectic measure}
of the symplectic measures,
we obtain a basic estimate
valid uniformly over $\mathfrak{g}^*$.
Recall the discussion of norms from \S\ref{ss:norms}
and the abbreviation $\langle \lambda  \rangle := (1 +
|\lambda|^2)^{1/2}$,
which applies in particular with $\lambda = [\mathcal{O}]$.
\begin{lemma}
  For each $\eps > 0$,
  \[
    \sup_{\mathcal{O}}
    \langle [\mathcal{O}]  \rangle^\eps 
    \int_{\xi \in \mathcal{O}}
    \langle \xi  \rangle^{-d(\mathcal{O})-\eps}
    \, d \omega_{\mathcal{O}}(\xi) < \infty.
  \]
\end{lemma}
\begin{proof}
  The contribution from $|\xi| \leq 1$ may be estimated using
  the prior lemma, noting that this contribution vanishes
  identically if $\langle [\mathcal{O}] \rangle > 1$.  For the remaining
  contribution, we split the $\xi$-integral into dyadic shells
  $A \leq |\xi| \leq 2 A$, with $A = 2^n$ for
  $n \in \mathbb{Z}_{\geq 0}$.  By \eqref{homogeneity of
    symplectic measure} and \eqref{eqn:not-avni-aizenbud}, the contribution
  from each such shell is bounded by a constant multiple of
  $A^{-\eps}$, while the smallest $A$ giving a nonzero
  contribution has size
  $\gg \langle [\mathcal{O}] \rangle$.  We conclude by
  summing dyadically over $A$.
\end{proof}

\subsection{Limit orbits} \label{sec:limit-coadjoint}
Let $\pi$ be a tempered irreducible
unitary representation of $G$,
the real points of a reductive group
$\mathbf{G}$ over $\mathbb{R}$.
We now allow $\pi$ to vary
with a positive parameter $\h \rightarrow 0$
traversing some sequence $\{h\} \subseteq (0,1)$.
We emphasize again that the dependence of $\pi$ and $\sigma$
upon $\h$
is not assumed to be (e.g.) continuous or measurable.

\subsubsection{Definition}
Let $\mathcal{O}$ be a regular coadjoint multiorbit.
\index{limit orbit}
We say that $\mathcal{O}$ is the \emph{limit orbit of
  $\pi$}
(or, pedantically speaking, \emph{limit multiorbit})
if
$\lim_{\h \rightarrow 0} \h \mathcal{O}_\pi =  \mathcal{O}$
in the sense of \S\ref{sec:topology-multiorbits}.
We then often abbreviate $\omega_{\mathcal{O}}$ to simply $\omega$.

We note that $\pi$ admits at most one limit orbit, since the
topology on $\mathfrak{R}$ is Hausdorff (see
\S\ref{sec:topology-multiorbits}).
We note also that if
$\pi$ admits the limit orbit $\mathcal{O}$, then
$\h \lambda_\pi$ converges to the infinitesimal character
$\lambda := [\mathcal{O}]$ of the limit orbit.

\subsubsection{The $\h$-independent case}
\label{sec:limit-coadj-orbits-h-indep}
Recall that $\pi$ is \emph{generic}
if
\begin{itemize}
\item   $G$ is quasi-split,
  i.e.,
contains a Borel subgroup defined over $\mathbb{R}$,
and
\item  $\pi$ admits a Whittaker model (with respect to some
nondegenerate character of the unipotent radical
of that Borel).
\end{itemize}
We refer to \cite{GKdim}
for
definitions concerning Gelfand--Kirillov dimension.

\begin{lemma*}
  Suppose that $\pi$ is $\h$-independent.
  Then $\pi$ has a limit orbit
  $\mathcal{O}$,
  where $\mathcal{O}$ is contained
  in the regular subset $\mathcal{N}_{\reg}$ of the nilcone.
  The following are equivalent:
  \begin{enumerate}[(i)]
  \item $\mathcal{O}$ is nonempty.
  \item $\pi$ has maximal Gelfand--Kirillov dimension.
  \item $\pi$  is generic.
  \end{enumerate}
\end{lemma*}
\begin{proof}
  The initial assertions and the
  equivalence of (i) and (ii) follow from \cite{MR576644}
  and \cite[Theorem C, D]{MR1353309}.
  (The initial assertions
  also follow
  readily
  from
  theorem \ref{thm:describe-topology-on-R}.)
  The equivalence of (ii) and (iii) follows from
  a result of Kostant \cite[Thm 6.7.2]{MR507800}.\footnote{
    Kostant proves what is required here modulo the possibility
    of replacing $\pi$ by another Hilbert space representation $\pi'$; Kostant's Hilbert space
    representations do not preserve the inner product, so
    it is not obvious
    that $\pi$ and $\pi'$
are ``the same.''
But Casselman \cite{MR1013462}
shows that the spaces of smooth vectors in the two representations
are isomorphic.
}
\end{proof}

\begin{remark*}
  Condition (i) is what is relevant for the purposes of this
  paper; we have invoked its equivalence with (ii) and
  (iii) only to simplify the statement of our result.
\end{remark*}

\section{Trace estimates\label{sec:kirillov-frmula}}
We now prove (a sharper form of) Theorem \ref{thm:trace-estimates-i}.
We retain the setup of \S\ref{sec:limit-coadjoint},
and define
$\Opp : S^m := S^m(\mathfrak{g}^\wedge) \rightarrow \Psi^m :=
\Psi^m(\pi)$
as usual.
Recall that
we write
\[
  \langle \xi \rangle := (1 + |\xi|^2)^{1/2}
\]
for an element $\xi$ of a normed space.

\subsection{Spaces of operators
  associated with various norms\label{sec:more-spaces-ops}}
For $T : \pi \rightarrow \pi$
and $p = 1,2$ or $\infty$,
we
let $\|T\|_p$ denote respectively
the trace, Hilbert--Schmidt or operator norm
of $T$.
We let $\mathcal{T}_p := \mathcal{T}_p(\pi)$
denote
the space of operators $T$ on $\pi$
with the property that
for each $u \in \mathfrak{U}$,
the operator
$\theta_u(T)$ (defined as in \S\ref{sec:operator-spaces-defn-etc})
induces a bounded map $\pi \rightarrow \pi$
for which
$\|\theta_u(T)\|_p < \infty$.
We equip $\mathcal{T}_p$ with its evident topology
(\S\ref{sec:prim-topol-vect}).
Note (by \S\ref{sec:membership-criteria-technical-stuff})
that $\mathcal{T}_{\infty} = \Psi^0$
is not a new space.

Given an $\h$-dependent positive real $c$, we denote as in
\S\ref{sec:asympt-param-h} and \S\ref{sec:h-dependence-Psi-m} by
$c \mathcal{T}_p$ the space of $\h$-dependent
$T \in \mathcal{T}_{p}$ for which the seminorms of $c^{-1} T$
are bounded uniformly in $\h$.  As usual, expressions involving
$\h^\infty$ are to be understood as holding whenever $\infty$ is
replaced by an arbitrary fixed number $N$.

\subsection{Approximate inverses for $\Delta$}
\label{sec:appr-delta-1}
Set $\Delta_{\h} := 1 - \h^2 \sum_{x \in
  \mathcal{B}(\mathfrak{g})} x^2$.
By the spectral theory
of the self-adjoint operator $\Delta$,
the operator
$\pi(\Delta_{\h})$
is invertible,
and its inverse has operator norm $\leq 1$.
The following lemma
allows one to control its inverse via
integral operators:
\begin{lemma*}
  For each $N \in \mathbb{Z}_{\geq 0}$
  there exist positive reals $\h_0 $ and $C$,
  depending upon $G$ and $\chi$ but not upon $\pi$,
  so that
  the following holds
  for $0 < \h \leq \h_0$:
  
  Define $b \in S^{0}$ by $b(\xi) := \langle \xi \rangle^{-N}$.
  Then $A(\h) := \Delta_{\h}^{N} \Opp_{\h}(b)^2$ defines an
  invertible operator on $\pi$.
  Moreover,
  the operator norms of $A(\h)$
  and $A(\h)^{-1}$ are bounded by $C$.
\end{lemma*}
\begin{proof}
  By applying the composition formula,
  we see that the $\h$-dependent
  operator $A(\h)$
  belongs to $1 + \h \Psi^{-1} \subseteq 1 + \h \Psi^0$.  For small
  enough $\h$, it follows by the Neumann lemma that $A(\h)$ is
  invertible with inverse of operator norm $\O(1)$, as required.
\end{proof}

\subsection{Results}
\label{sec:appl-kir-type-formulas}
\label{sec:results-appl-kir}
Let $\pi$ be an $\h$-dependent irreducible tempered
representation of $G$.
We adopt here the convention that implied constants
in any asymptotic notation
are independent of $\pi$ and $\h$,
and must depend continuously upon any symbols under
consideration.
Recall from \S\ref{sec:canonical-symplectic-form}
the definition of $d \in \mathbb{Z}_{\geq 0}$.

%We assume that the cutoffs $\chi, \chi'$
%as in \S\ref{sec:comp-prelim}
%have small enough support.
 
%
%
%
%
%
%
%
%
%
%
%
%
%
%
%
%
%
%
%
%
%
%
%
%
%
%
%

%
%
%
%
%
%
%
%
%
%
%
%
%
%
%
%
%
%

\begin{theorem} \label{trace bounds for Op}
  Let assumptions and conventions be as above.
  Let $N \in 2 \mathbb{Z}_{\geq 0}$
  satisfy $N > d$.
  \begin{enumerate}[(i)]
  \item
    For $\delta \in [0,1)$
    and $a \in S^{-N}_{\delta}$,
    we have
    \begin{equation}\label{eqn:basic-trace-bound}
      \trace(\Opp_{\h}(a)) \ll \h^{-d} \langle \h \lambda_\pi
      \rangle^{d-N}
    \end{equation}
    and
    \begin{equation}\label{eqn:kirillov-approximate}
            \h^d \trace(\Opp_{\h}(a)) = \int_{\h \mathcal{O}_{\pi}}
      a \, d \omega_{\h \mathcal{O}_\pi}
      + \O(\h^{1-\delta} \langle \h \lambda_\pi  \rangle^{d-N-1}).
    \end{equation}
    In particular, if $a$ is $\h$-independent and $\pi$ admits
    the limit orbit $(\mathcal{O},\pi)$, then
    \begin{equation}\label{eqn:kirillov-appl-to-lim-orb}
          \lim_{\h \rightarrow 0} \h^d \trace(\Opp_{\h}(a)) =
    \int_{\mathcal{O}} a \, d \omega.
    \end{equation}
    
    For each $j \in \mathbb{Z}_{\geq 0}$,
    there is a constant coefficient
    differential operator $\mathcal{D}_j$
    on $\mathfrak{g}^\wedge$ of pure degree $j$,
    so that for
    $a \in S^{-\infty}_{\delta}$
    and
    fixed $J, N' \in \mathbb{Z}_{\geq 0}$,
    \begin{equation}\label{eqn:kirillov-expanded-in-diff-ops}
      \h^d \trace(\Opp_{\h}(a)) =
      \sum_{0 \leq j < J}
      \h^{j}
      \int_{\h \mathcal{O}_\pi} \mathcal{D}_j a \, d
      \omega_{\h \mathcal{O}_\pi}
      + \O(\h^{(1-\delta) J} \langle \h \lambda_\pi  \rangle^{-N'}).
    \end{equation}
  \item
    $\Delta_{\h}^{-N/2}$ (cf. \S\ref{sec:appr-delta-1}) is trace class,
    with $\|\Delta_{\h}^{-N/2}\|_{1} = \trace(\Delta_{\h}^{-N/2}) \ll \h^{-d} \langle \h \lambda_\pi  \rangle^{d-N}$.
  \item
    For any operator
    $T$ on $\pi$,
    the trace norm is majorized
    as follows:
    % controll
    % we have
    \begin{equation}\label{eqn:control-trace-norm-via-specific-operator-norm}
      \|T\|_1 \leq C \langle \lambda_\pi  \rangle^{d-N} \|T\|_{\pi^0 \rightarrow \pi^{N}},
    \end{equation}
    where $C$ depends only upon $G$ and $N$.
    For $p=1,2$, we have
    $\Psi^{-N} \subseteq \mathcal{T}_p$; the normalized 
    map
    \begin{equation}\label{eq:Psi-bounds-trace-norms}
      \Psi^{-N} \rightarrow \mathcal{T}_p,
      \quad
      T \mapsto \langle \lambda_\pi
      \rangle^{N-d} T
    \end{equation}
    is continuous,
    uniformly in $\pi$.
    These conclusions remain valid also for non-tempered $\pi$,
    possibly with a larger value of $N$.
  \item
    For fixed $x,y \in \mathfrak{U}$ and any
    $T \in \h^0 \Psi^{-\infty}$,
    we have $\|x T y\|_2 \ll 1$,
    with continuous dependence upon $T$.
  \item
    We have
    % For $0 \leq \delta < 1/2$,
    \[
      \Opp_{\h}(S^{-N}_{0})
      \subseteq
      \h^{-d} \langle \h \lambda_\pi  \rangle^{d-N}
      \mathcal{T}_1,
    \]
    \[
      \Opp_{\h}(S^{-N}_{0})
      \subseteq
      \h^{-d/2} \langle \h \lambda_\pi  \rangle^{(d-N)/2}
      \mathcal{T}_2.
    \]
  \item For $k \in \mathbb{Z}_{\geq 1}$ and $a_1,\dotsc,a_k \in S^{-N}_{0}$,
    \begin{equation}\label{eq:final-op-est-1}
      \Opp_{\h}(a_1)
      \dotsb 
      \Opp_{\h}(a_k)
      \in \h^{-d} \langle \h \lambda_\pi \rangle^{d-N} \mathcal{T}_1,
    \end{equation}
    \begin{equation}\label{eq:final-op-est-2}
      \h^d \Opp_{\h}(a_1)
      \dotsb 
      \Opp_{\h}(a_k)
      \equiv 
      \h^d   \Opp_{\h}(a_1 \dotsb a_k)
      \mod{ \h \langle \h \lambda_\pi \rangle^{d-N} \mathcal{T}_1.}
    \end{equation}
  \item
    % Fix $0 \leq \delta < 1$, $\eps > 0$.
    Fix $\eps > 0$.
    Let $g \in G$ be an $\h$-dependent element with
    % $\|\Ad(g)\| \leq \h^{-1+\delta+\eps}$.
    $\|\Ad(g)\| \leq \h^{-1+\eps}$.
    Let $a \in S^{\infty}_{0}$.
    Then
    \begin{equation}\label{eq:final-op-est-3}
      \h^d \Opp_{\h}(g \cdot a)
      \equiv
      \h^d \pi(g) \Opp_{\h}(a) \pi(g)^{-1}
      \mod{ \h^\infty  \langle \h \lambda_\pi \rangle^{d-N} \mathcal{T}_1}.
    \end{equation}
  \end{enumerate}
\end{theorem}
\begin{proof}
  We will frequently
  apply
  theorems \ref{thm:star-prod-basic},
  \ref{thm:main-properties-opp-symbols},
  \ref{thm:rescaled-operator-memb}.
  \begin{enumerate}[(i)]
  \item The first assertion reduces
    to the estimate
\begin{equation}\label{eqn:unif-coadj-orb-bound}
  \sup_{\h \in (0,1]}
  \sup_{\mathcal{O} : d(\mathcal{O}) = d}
  \langle \h [\mathcal{O}]  \rangle^\eps 
  \h^{d}
  \int_{\xi \in \mathcal{O}}
  \langle \h \xi  \rangle^{-d - \eps}
  \, d \omega_{\mathcal{O}}(\xi)
  < \infty,
\end{equation}
which follows in term from \eqref{homogeneity of symplectic
  measure}
and the results of \S\ref{sec:bounds-symplectic-measures}.
The remaining assertions
    follow by expanding $(j^{-1/2} \chi a^\vee)^\wedge$ using \S\ref{sec:appl-tayl-theor}
    and recalling that $j(0) = 1$.
  \item
    By spectral theory,
    we may assume that $\h$ is sufficiently small.
    Set $b(\xi) := \langle \xi  \rangle^{-N/2}$.
    By
    \S\ref{sec:appr-delta-1},
    we have
    $\trace(\Delta_{\h}^{-N/2}) \ll \trace(\Opp_{\h}(b)^2)$.
    By applying the proof of (i)
    to
    $\trace(\Opp_{\h}(b)^2)
    = \trace(\Opp_{\h}(b \star_{\h} b, \chi '))$,
    we obtain
    an adequate estimate for
    $\trace(\Opp_{\h}(b)^2)$.
  \item
    Let $T \in \Psi^{-N}$.
    By (ii) and the inequality
    \begin{equation}\label{eqn:holder-involving-Delta-neg-N}
      \|A\|_1
      \leq
      \|\Delta_{\h}^{-N/2}\|_1 \|\Delta_{\h}^{N/2} A\|_\infty,
    \end{equation}
    applied with $A := \theta_u(T)$ and $\h := 1$,
    we obtain the estimate
    \eqref{eqn:control-trace-norm-via-specific-operator-norm}
    as well as
    the required inclusion
    for $p = 1$.
    We deduce the case $p = 2$
    via $\|A\|_{2} \leq \|A\|_1^{1/2} \|A\|_{\infty}^{1/2}$.

    The necessary input in this argument
    was the uniform trace
    class property of $\Delta^{-N/2}$.
    This is presumably well-known,
    and holds also non-tempered $\pi$,
    as follows, e.g,. from the proof
    of part (i) of the lemma in
    \S\ref{sec:arch-case-coarse-control-on-tempered-dual}.
  \item By a similar argument as in (iii).
  \item   We must estimate
    $\|\theta_u(\Opp_{\h}(a))\|_p$
    for fixed $u \in \mathfrak{U}$ and $p=1,2$.
    By differentiating the composition formula,
    we have
    $\theta_u(\Opp_{\h}(a))
    \equiv  
    \Opp_{\h}(\theta_u(a))
    \mod{ \h^\infty \Psi^{-\infty}}$.
    By (iii),
    we thereby reduce to the case $u = 1$.
    Using the identity
    $\|\Opp_{\h}(a)\|_2^2
    = \trace(\Opp_{\h}(a)\Opp_{\h}(a)^*)= \trace(\Opp_{\h}(a)\Opp_{\h}(\overline{a}))$,
    the composition formula,
    and (iii),
    we reduce further to the case $p=1$,
    in which it remains
    to show that
    $\|\Opp_{\h}(a)\|_1 \ll \h^{-d}
    \langle \h \lambda_\pi  \rangle^{d-N}$.
    For this
    we apply
    \eqref{eqn:holder-involving-Delta-neg-N}
    with $A := \Opp_{\h}(a)$
    and appeal to (ii)
    and the consequence
    $\|\Delta_{\h}^{N/2} \Opp_{\h}(a)\|_{\infty} \ll 1$
    of the composition formula.

    The remaining results (vi) and (vii)
    may be proved similarly
    (for (vii), cf. \S\ref{sec:equivariance-pf}).
  \end{enumerate}
\end{proof}

%
%
%
%
%
%
%

%
%
%
%
%
%
%
%
%
%
%
%
%
%
%
%

%
%
%
%
%
%
%
%
%
%
%
%
%
%
%
%
%
%
%
%
%
%
%
%
%
%
%
%
%
%
%
%
%
%
%
%
%
%
%
%
%
%
%
%
%
%
%
%
%
%
%
%
%
%
%

%
%
%
%
%
%
%
%
%
%
%
%
%

%
%
%
%
%
%
%
%
%
%
%
%
%
%
%
%
%
%
%
%
%
%
%
%
%
%
%
%
%
%
%
%
%
%
%
%
  
%
%
%
%
%
%
%

%
%
%
%
%
%
%
%
%
%
%
%
%
%
%
%
%
%
%
%
%
%
%
%
%
%
%
%
%
%
%
%
%

%
%
%
%
%
%
%

%

%
%
%
%
%
%
%
%
%
%
%
%
%
%
%
%
%
%
%
%
%

%
%
%
%
%
%
%
%
%
%
%
%
%
%
%
%
%
%
%
%
%
%
%
%
%
%
%
%
%
%
%
%
%

%

%

%
%
%
%
%
%
%
%
%
%
%
%
%
%
%
%
%
%
%
%
%
%
%
%
%
%
%

%
%
%
%
%
%
%
%
%
%
%
%
%
%
%
%
%
%

%
%
%
%

%
%
%
%
%
%
%
%
%
%
%
%
%
%
%
%
%
%
%
%

%
%
%
%
%
%
%
%
%
%
%
%
%
%
%
%
%
%
%
%
%
%
%
%
%
%
%
%
%
%
%
%
%
%
%
%
%
%

%
%
%
%
%
%
%
%
%
%
%
%
%
%
%
%
%
%
%
%
%
%
%
%
%
%
%
%
%
%
%
%
%
%
%
%
%
%
%
%
%
%
%
%
%
%
%
%
%
%
%

%
%
%
%
%
%
%

%

%

% \begin{comment}
  
\iftoggle{cleanpart}
{
  \newpage
}
\part{Gan--Gross--Prasad pairs: geometry and asymptotics}\label{part:gross-prasad-pairs}
Let $k$ be a field of characteristic zero.
For our purposes, a {\em Gan--Gross--Prasad pair}
(henceforth
``GGP pair'')
over $k$
is a pair
$(\mathbf{G},\mathbf{H})$ of algebraic $k$-groups equipped with
\begin{itemize}
\item an inclusion $\mathbf{H} \hookrightarrow \mathbf{G}$ of
  algebraic $k$-groups, isomorphic to one of the standard
  inclusions
  \begin{equation}\label{eq:std-incl}
    \mathbf{SO}_{n} \hookrightarrow \mathbf{SO}_{n+1},
    \quad
    \mathbf{U}_n \hookrightarrow \mathbf{U}_{n+1},
    \quad 
    \mathbf{GL}_n \hookrightarrow \mathbf{GL}_{n+1},
  \end{equation}
  and
\item An action (called the ``standard action'') of $\mathbf{G}$ on some vector space
  $V$,
  i.e., an embedding
  $\mathbf{G} \hookrightarrow \boldGL(V)$.
   See below for complete details. \end{itemize}
We make this definition precise in \S\ref{sec: standard
  inclusions}.  The general linear example may be understood as
a special case of the unitary example, as we will find very
convenient in our proofs.
We refer to the first case as the orthgonal case,
and the latter two as unitary cases.

The study of such pairs, locally and globally, was initated
(in the special orthogonal case) by Gross and Prasad \cite{GP}.
A broader formalism was developed in the paper of Gan, Gross and Prasad \cite{GGP}. 
The cases of $(\mathrm{SO}_2, \mathrm{SO}_3)$ and $(\mathrm{SO}_3, \mathrm{SO}_4)$
have played an
important role in the analytic theory of $L$-functions for $\GL_2$.
It is therefore very natural to consider the analytic theory of
Gan--Gross--Prasad periods in higher rank.

Part \ref{part:gross-prasad-pairs} contains several algebraic
and analytic preliminaries concerning GGP pairs, many
of which may be of independent interest.

The main aim of \S\ref{sec:gross-prasad-pairs-inv-theory}
and \S\ref{sec:stability}
is to study in detail certain
algebraic properties of the restriction to $\mathbf{H}$ of the
adjoint (equivalently, coadjoint) representation of
$\mathbf{G}$.
In the language of \S\ref{sec:intro-local-issues},
we
are studying how $H$ acts
on the system of solutions
to
\[
  \xi|_{\mathfrak{h}} = \eta \quad (\xi \in \mathfrak{g}^\wedge, \eta \in \mathfrak{h}^\wedge);
\]
this is relevant for us because (for $k$ a local field) it
models the ``asymptotic decomposition'' of the restriction to
$H$ of a unitary representation of $G$.  The ``nice'' solutions
will turn out to form a smoothly-varying family of
$\mathbf{H}$-torsors.  The definition of ``nice'' is formulated
in \S\ref{sec:stability} in terms of the GIT notion of stability
and then related in \S\ref{Satake parameters} to the absence of
``conductor dropping'' for the associated Rankin--Selberg
$L$-function.  In \S\ref{sec:volume-forms}, we study, e.g., how
integrals over $G$-orbits in $\mathfrak{g}^\wedge$ can be
disintegrated in terms of integrals over the $H$-orbits
discussed previously.  In
\S\ref{sec:gross-prasad-pairs-over-R-continuity-etc}, we apply
our results to archimedean GGP pairs.  The main output is that
we can write the integral over a coadjoint orbit $\mathcal{O}$
for $G$ explicitly in terms of integrals over the $H$-orbits on
$\mathcal{O}$, with control over how everything varies in
families.
 The main aim of
\S\ref{sec:spher-char-disint}
and \S\ref{sec:sph-char-2}
is then to prove a ``quantum
analogue'' of such integral formulas,
involving the asymptotic decomposition of the restriction
to $H$ of a tempered irreducible representation of $G$.

The pictures in \S\ref{sec:intro-op-calc}
and \S\ref{sec:intro-local-issues}
may usefully
illustrate
the discussion below.

\section{Basic definitions and invariant theory}\label{sec:gross-prasad-pairs-inv-theory}

\subsection{Orthogonal groups and unitary groups}
\label{sec:classical-groups-1}
As above, let $k$ be a field of characteristic zero. 
Let $k_1$ be either a
quadratic {\'e}tale $k$-algebra or $k$ itself,
thus either
\begin{equation}
  k_1 = k,
  \quad 
  k_1/k \text{ is a quadratic field extension, or }
  \quad 
  k_1 = k \times k.
\end{equation}
The three
cases
indicated in
\eqref{eq:std-incl} will arise accordingly.
  In the third case, $k$ is embedded diagonally in the product.

Let
$\iota$ denote the involution of $k_1$ fixing
$k$, and let $(V, \langle \rangle)$ be a $k_1$-vector
space  equipped with a nondegenerate $\iota$-linear symmetric
bilinear form.  We then denote by 
$\dim(V)$, $\End(V)$ and $\GL(V)$ the dimension,
endomorphisms and automorphisms of $V$
\emph{as a $k_1$-vector space}.
We may define the connected automorphism group
\[
  \mathbf{G} = \boldAut(V/k_1, \langle  \rangle)^0.
\]
It is a $k$-algebraic group which comes with a standard
representation
$\mathbf{G} \hookrightarrow \boldGL(V)$.

We denote as usual by
$\mathfrak{g}
\hookrightarrow \End(V)$
the
Lie algebra of $\mathbf{G}$.
% \av{My own feeling : we don't need to say $k$-points. When given a linear algebraic group over a field $K$,
% I think its Lie algebra is understood to be a $K$-vector space, and not the schematic version of that object.}
We write
\[ [\mathfrak{g}] =\mathfrak{g} \git \mathbf{G}\]
for the
\index{GIT quotient $[\mathfrak{g}]$}
set of $k$-points of the
GIT quotient
(cf. \S\ref{sec:inf-char-overview} for an example).
For $x \in \mathfrak{g}$,
we write $[x]$ for its image in $[\mathfrak{g}]$.

For $x \in \End(V)$, we denote by $x^* \in \End(V)$ the conjugate-adjoint, defined by means of the rule
$\langle u, x v \rangle = \langle x^* u, v \rangle$ for all
$u, v \in V$. 
Then  $x \mapsto x^*$ is $\iota$-linear,
and
$(x y)^* = y^* x^*$.
The Lie algebra $\mathfrak{g}$ 
consists of those
$x \in \End(V)$
that  are skew-adjoint: $x^* = -x$.

\subsection{Standard inclusions} \label{sec: standard inclusions}
Retaining the setting
of \S\ref{sec:classical-groups-1},
fix $e \in V$ for which $k_1.e$ is a free
rank one $k_1$-module on which the form $\langle \rangle$ is
nondegenerate,
and set $V_H =  (k_1 .e)^{\perp} \subseteq V$.
One then has the splitting
$V = V_H \oplus k_1 e$,
which induces an inclusion
\[\mathbf{H} := 
\boldAut(V_H/k_1, \langle \rangle)^0
\hookrightarrow
\mathbf{G} := \boldAut(V/k_1, \langle \rangle)^0.\]
A GGP pair \index{GGP pair} over $k$ is defined to be such an inclusion,
together with the accompanying standard representations.
%
%
%
%

% Let $\mathbf{G}$ be 
% %$G = \Aut(V/k_1, \langle  \rangle)^0$
% be as above.

\subsection{Extension to an algebraic closure} \label{GGP kbar}
 
The action of $G$ on $V$
is $k_1$-linear, and so preserves the $k$-linear decomposition of $V$ into isotypic
subspaces for $k_1$.
We explicate this decomposition first
in the special case that 
$k$ is algebraically closed,
so that either $k_1 = k \times k$ or $k_1 = k$:
\begin{itemize}
\item
  (Unitary case)
  If $k_1 = k \times k$,
  then we have a decomposition of $k$-vector spaces \index{$V^+, V^-$}
  \[V = V^+ \oplus V^-,\]
  where
  $k_1$ acts on $V^{\pm}$ via the two
  projections to $k$.
  The form $\langle  \rangle$
  on $V^+ \times V^-$ is valued in the first factor $k$ of $k_1$,
  and thus induces a $k$-valued perfect pairing
  $[-,-]$ between $V^+$ and $V^-$; the hermitian
  form is described in terms of $[-,-]$ as 
  \[ \langle v_1^+ + v_1^-, v_2^+ + v_2^-\rangle = ([v_1^+, v_2^-] , [v_2^{+}, v_1^-]) \in k \times k.\] 
  
Moreover, 
  $G \cong \GL(V^+)$
  identifies
  with the set of pairs $(g, {}^t g^{-1})$ in
  $\GL(V^+) \times \GL(V^-)$.
  The vector $e$ is
  given by $e^+ + e^{-}$, with  $e^+ \in V^+, e^{-} \in V^-$
  satisfying  $\langle e^+, e^{-} \rangle \neq 0$. 

  There is a corresponding splitting $V_H = V_H^+ \oplus V_H^-$.
  \index{$V_H^+$} \index{$V_H^-$}
\item 
  (Orthogonal case)
  If $k_1 = k$,
  so that
  $G = \SO(V)$,
  then we set
  \[
  V^+ := V^- := V.
  \]
\end{itemize}
In general, we fix an algebraic closure $\overline{k}$
of $k$
and apply the above considerations
with $(V \otimes_k \overline{k}, \overline{k}, k_1 \otimes_k
\overline{k})$
playing the role of ``$(V, k, k_1)$''.
We obtain in this way $G$-invariant $\overline{k}$-subspaces
$V^{\pm}$ of $V \otimes_k \overline{k}$.

\subsection{Eigenvalues and
  eigenvectors}\label{sec:eigenv-git-quot}
% In this subsection, we assume unless otherwise indicated that
% $k$ is algebraically closed.

\subsubsection{Invariants in terms of
  eigenvalues}\label{sec:invar-terms-eigenv}
We recall the description of the ring of $G$-invariant polynomial
functions on $\mathfrak{g}$ (``invariant functions'' for short).
Each $x \in \mathfrak{g}$ gives a $k_1$-linear endomorphism of
$V$
and thus has a characteristic polynomial in $k_1[t]$.
Since $x^* = -x$,
the $n$th coefficient of the characteristic polynomial
of $x$ belongs to
$k_1^{\pm}=  \{x \in k_1: \iota(x) =  \pm x\}$,
where $\pm$ depends upon the parity of $n$.
The spaces $k_1^{\pm}$ are at most one-dimensional
over $k$;
fixing bases,
we obtain invariant functions $\mathfrak{g} \rightarrow k$.
These freely generate
the ring of invariant functions in all cases except when $G$ is
even orthogonal, in which the Pfaffian $\pf(x)$
gives another invariant satisfying $\pf(x)^2 =\det(x)$.

We denote by $\ev(x)$  the multiset of roots\index{non-obvious eigenvalue multiset $\ev(x)$}
in $\overline{k}$ of the
characteristic polynomial for the $\overline{k}$-linear action
of $x \in \mathfrak{g}$ on $V^+$, except that in the odd orthogonal case, we
subtract the ``obvious root'' $0$ with multiplicity one, noting
that it always occurs.  
Thus $\# \ev(x) = n$, where
$n = \dim(V)$ in the unitary, linear and even orthogonal cases
and $n = \dim(V) - 1$ in the odd orthogonal case.
We may descend $\ev$ to a function
on
$[\mathfrak{g}]$.
The map
\[
  [\mathfrak{g}] \ni [x] \mapsto  \begin{cases}
    \text{the ordered pair }
    (\ev(x),\pf(x)) & \text{in the even orthogonal case,} \\
    \ev(x) &  \text{otherwise}
\end{cases}
\]
is then well-defined and injective,
with readily characterized image.

% When $k$ is not necessarily algebraically closed,
% we fix an algebraic closure $\overline{k}$
% and use the notation
% $\ev(x)$
% to denote the associated Galois-invariant multiset.
% inside ``the'' algebraic closure $\kbar$. 

% By a slight abuse of notation, we will also allow ourselves
% to use the notation $\ev(x)$ when $k$ is not algebraically closed. It is then
% a Galois-invariant multiset inside ``the'' algebraic closure
% $\kbar$.
% \pn{I'll come back to this issue}

\subsubsection{Geometric characterization}\label{sec:geom-char-ev}
We may describe
the set underlying the multiset $\ev(x)$
more geometrically as follows:

 \begin{lemma*}
  Let
  $x \in \mathfrak{g}$
  and $c \in k$.
  The following are equivalent:
  \begin{enumerate}[(i)]
  \item $c \in \ev(x)$.
  \item $x$ has an isotropic eigenvector in $V^+$ with eigenvalue $c$.
  \item $x$ has an isotropic eigenvector in $V^-$ with eigenvalue $-c$.
  \end{enumerate}
\end{lemma*}
\begin{proof}
%  \pn{I can't remember: why did we get rid of the original
%    argument,
%    which said that the various conditions on $x$ are Zariski
%    closed,
%    and equivalent generically?
%    Pasted below:
%  }
%  
%  We may assume that $k$ is algebraically closed.  Since
%  $x^* = -x$, we see that (ii) and (iii) are equivalent.  Since
%  $x^* = - x$, we see that (ii) and (iii) are equivalent.  It
%  remains to show that the conditions (i) and (ii) on $x$ are
%  equivalent.  Each condition is closed (being given by the
%  image of a closed set under a proper map), so it suffices to
%  consider the generic case in which $x$ has distinct
%  eigenvalues which are moreover nonzero if we are not in the
%  odd orthogonal case.  (These conditions on $x$ are indeed
%  generic, as one may see by inspecting the Cartan subalgebras
%  in each example.)  Under these assumptions, $\ev(x)$
%  consists of the nonzero eigenvalues
%  $\lambda \in \overline{k}$ for $x$ on $V^+$.  The
%  corresponding eigenvectors $v$ are necessarily isotropic: one
%  has
%  $\lambda \langle v, v \rangle = \langle x v, v \rangle = -
%  \langle x^* v, v \rangle = - \langle v, x v \rangle = -
%  \lambda \langle v, v \rangle$, but $2 \lambda \neq 0$.
%
  The equivalence of (ii) and (iii) follows
  readily from the definitions of $V^+$ and $V^-$.

  For the equivalence between (i) and (ii),
  note that any eigenvector with
  nonzero eigenvalue is necessarily isotropic.
  The required equivalence
  is thus clear if $c \neq 0$.
  Suppose henceforth that $c = 0$.
  In the unitary case,
  the space $V^+$ is
  itself isotropic,
  so (ii) just says that $x$ has $0$ as an eigenvalue
  on $V^+$,
  which is equivalent to (i).

  It remains to show that (i) and (ii) are equivalent
  in the orthogonal case with $c = 0$.
  Let $W \subseteq V^+$ denote the generalized $0$
  eigenspace for $x$.  Because the generalized
  eigenspaces with eigenvalues $\lambda$ and $-\lambda$ are in
  perfect pairing with one another, we have
  $\dim(W) \equiv \dim(V^+)$ modulo $2$; moreover, the
  restriction of $\langle -, - \rangle$ to $W$ is nondegenerate.
  Recall that $0 \in \ev(x)$
  in the even orthogonal case
  when $\dim(W) \geq 1$
  and in the odd orthogonal case
  when $\dim(W) \geq 2$.
  We see in either case
  that if $0 \in \ev(x)$,
  then $\dim(W) \geq 2$;
  since $W$ is nondegenerate,
  it thus contains a nonzero isotropic vector.
  Thus (i) implies (ii).
  For the converse implication,
  note that if $x$ has an isotropic eigenvector with
  eigenvalue $0$,
  then $W$ contains a nonzero isotropic subspace,
  whence
  $\dim(W) \geq 2$;
  to complete the proof, it is thus enough to check:
  \begin{quote}
    If $(W, \langle -, - \rangle)$
    is a nondegenerate quadratic space of dimension at least $2$, and $x \in \End(W)$ is skew-symmetric and nilpotent,
    then there is a nonzero isotropic vector in the kernel of $x$. 
  \end{quote}
  This is clear if $x = 0$. Otherwise there is a nonzero vector $v \in \ker(x)$
  which belongs to the image of $x$, say $v =xu$, and then $v$ does the trick:
  \[ \langle v, v \rangle = - \langle u, xv \rangle = 0.\]
  (Said differently, if $x$ in the quoted statement is nonzero,
  then it arises from the action of $\begin{pmatrix}
    0 & 1 \\
    0 & 0
  \end{pmatrix}$
  under some homomorphism $\mathfrak{sl}_2 \rightarrow \mathfrak{so}_W$, 
  and we choose $v$ to be a highest weight vector.)
\end{proof}

\subsubsection{Regular elements}\label{sec:characterize-regular-elements-isotropic-ev}
We note also a related characterization of regular elements:
\begin{lemma*}
Let $x \in \mathfrak{g}$ and suppose that
every isotropic subspace of $V^+$ on which $x$ acts by a
scalar is at most one-dimensional.  
In the orthogonal case, assume moreover that the kernel of
$x$
is at most $2$-dimensional.
Then $x$ is regular. 
\end{lemma*}
\begin{proof}
  % The regular elements are those that belong to finitely many
  % Borel subalgebras \cite[Theorem 1]{Steinberg}.  The Borel subalgebras are the
  % stabilizers of maximal isotropic flags in $V^+$.  If $x$ has an
  % isotropic eigenspace $W \subseteq V^+$, then any flag in $W$ extends to a
  % maximal isotropic flag stabilized by $x$.  If $\dim(W) > 1$,
  % then $W$ contains infinitely many flags, and so $x$ is
  % irregular.  Thus (i) implies (ii).  
  % We will show that   $x$ stabilizes only finitely many isotropic flags in $V^+$.

  In the unitary case, this is well-known:
  the word ``isotropic'' can be ignored since
  the form vanishes on $V^+$, and an $n \times n$ matrix is regular if and only if
  each eigenvalue corresponds to a single Jordan block. 
  Accordingly, we focus on the orthogonal case.

  Write 
  \begin{equation}\label{eqn:decomp-V-plus-gen-eigen}
    V^+
    =
    \bigoplus_{\lambda} V^+_{\lambda}
    = 
    (\bigoplus_{\{\lambda \neq 0\} / \pm} V^+_{\lambda} \oplus
    V^+_{-\lambda})
    \oplus V^+_0
  \end{equation}
  for the decomposition into generalized eigenspaces under $x$. 
  To prove regularity, we must show that the  dimension of the centralizer of $x$  
  is minimal.
  Since any element of this centralizer preserves
  each
  summand in \eqref{eqn:decomp-V-plus-gen-eigen},
  the minimal centralizer dimension for $\SO_{2n}$ or
  $\SO_{2n+1}$ is $n$,
  and
  the
  summands $V^+_{\lambda} \oplus
  V^+_{-\lambda}$
  are even-dimensional,
  we reduce to showing the following:
  \begin{itemize}
  \item $x$ induces a regular element of the orthogonal group of $V^+_{\lambda} + V^+_{-\lambda}$
    for each $\lambda \neq 0$:
    
    Our condition implies that $x-\lambda$ is a regular nilpotent on $V^+_{\lambda}$,
    so its centralizer in $\GL(V^+_{\lambda})$ has minimal dimension, which implies 
    the above claim.
    
  \item $x$ induces a regular element of the orthogonal group of $V^+_0$.
    In what follows we denote this restriction simply by $x$.  
    
    By the Jacobson-Morozov theorem, $x$ is the image of a nonzero nilpotent
    in $\mathfrak{sl}_2$ under a homomorphism $\mathfrak{sl}_2 \rightarrow \mathfrak{so}(V^+_0)$. 
    The assumption on $\ker(x)$ implies $V_0^+$ decomposes into a
    sum of at most two irreducible representations $U \oplus U'$.  
    If $U$ and $U'$ are each of dimension $\geq 2$,
    then we see by considering weights that $\ker(x)$ is contained in
    $\image(x)$; the skew-symmetry of $x$ then implies
    that $\ker(x)$ is isotropic and of
    dimension $\geq 2$,
    contrary to hypothesis.
    We may thus suppose that $\dim(U') \leq 1$.
    We may also suppose that $\dim(U) \geq 2$; if not,  
    then $x$ is the zero element of $\mathfrak{so}_1$ or $\mathfrak{so}_2$ and is regular.
    
    Now $\dim(U)$ is necessarily odd-dimensional, otherwise $\mathfrak{sl}_2$ does
    not
    preserve an orthogonal form on it.  It is now a routine computation with $\mathfrak{sl}_2$-representations
    to compute that
    the centralizer of $x$ in $\so(V_0^+)$ has the correct dimension; explicitly, 
     this centralizer is given by
         \[
    \begin{cases}
      x, x^3, \dotsc, x^{\dim(U) - 2} & U' = \{0\} \\
      x, x^3, \dotsc, x^{\dim(U)-2}, y & \dim(U')=1,
    \end{cases}
    \]
    where $y$ is a skew-symmetric transformation sending a generator of $U'$ to a
    nonzero highest weight vector in $U$ and mapping $U$ to $U'$.
  \end{itemize}
\end{proof}

\subsection{The spherical property}\label{sec:proof-spher-prop}
Given a GGP pair $(\mathbf{G},\mathbf{H})$,
we will have occasion to consider the diagonal inclusion
\begin{equation} \label{large GP inclusion} \mathbf{H}
  \longrightarrow \mathbf{G} \times \mathbf{H}\end{equation}
This is a {\em spherical pair}, in the following
well-known sense
(see, e.g.,
\cite[\S6.4, p.142, (1)]{2015arXiv150601452B}).
\begin{lemma*} \label{GP spherical}
  $\mathbf{H}$ has an open orbit, with trivial stabilizer,
  on the flag variety of $\mathbf{G} \times \mathbf{H}$. 
\end{lemma*}
 
\section{Stability}\label{sec:stability}
Recall that $k$ is a field of characteristic zero.

\subsection{Preliminaries}
\label{sec:stability-general}
Let $\mathbf{H}$ be a reductive algebraic $k$-group
which acts
(algebraically)
on an affine $k$-variety $\mathbf{M}$.
Recall,
\index{stable}
following Mumford, that an element $x \in \mathbf{M}(k)$ is called \emph{$\mathbf{H}$-stable}
if
\begin{enumerate}[(i)]
\item the stabilizer of $x$ in $\mathbf{H}$ is finite,
  and
\item the orbit $\mathbf{H} \cdot x$ is closed.
\end{enumerate}
 
We summarize some background
from geometric invariant theory:
\begin{lemma*}
Suppose we are in the setting just described with $k$ algebraically closed.
Let $M \giit H$ denote the spectrum 
of the ring of $H$-invariant regular functions
on $M$,
and $\phi : M \rightarrow M \giit H$
the canonical map.
Then $M \giit H$ is an affine variety,
and $\phi$ is surjective.
Each $H$-invariant morphism with
domain $M$ factors uniquely through $\phi$.
If $M$ is irreducible, then so is $M \giit H$.

Let $M^s \subseteq M$ denote the subset of $\mathbf{H}$-stable
elements.  Then $M^s$ and $(M \giit H)^s := \phi(M^s)$ are open
(but possibly empty).
If the isotropy group of every point in $M^s$ inside $H$ is trivial, then 
the induced map $\phi^s: M^s \rightarrow
(M \giit H)^s$ is a principal $H$-bundle (indeed, it is locally
trivial in the {\'e}tale topology).
 \end{lemma*}

For the last statement, we refer in particular to \cite[Prop
0.9]{Mumford}
and
\cite[p120]{milneLEC}.

% which shows that the induced bundle is trivial in
% the smooth topology.
% Triviality in the {\'e}tale topology then
% follows: if a map admits a section after pullback by a smooth map,
% then the same is true after pullback by an {\'e}tale map, 
% because a smooth map itself admits a section after pullback by an {\'e}tale one.

\subsubsection{Moment map
  interpretation}\label{sec:moment-map-stability}
If $\mathbf{M}$ is smooth,
then the action of
$\mathbf{H}$ on $\mathbf{M}$ induces an action on the
cotangent bundle $T^* \mathbf{M}$
and, by duality, an
$H$-equivariant moment map
\[
  \Phi : T^* M \rightarrow \mathfrak{h}^*.
\]
We then verify readily that condition (i)
in the above definition is equivalent
to:
\begin{enumerate}
\item[(i')] The moment map $\Phi$
  induces, by differentiation, a \emph{surjective} map
  $T_x^* M \rightarrow \mathfrak{h}^*$.
\end{enumerate}

\subsection{Characterization for GGP
  pairs}\label{sec:stab-terms-spectra}
Let $\mathbf{H} \hookrightarrow \mathbf{G}
\hookrightarrow \boldGL(V)$ be a GGP
pair over $k$.
We retain the accompanying notation
of
\S\ref{sec: standard inclusions}.
We denote by $E \in \End(V)$ the orthogonal projection with
image $V_H$.
For
$x \in \End(V)$, we denote by $x_H \in \End(V_H)$ the
restriction of $E x E$ to $V_H$. 
We then have a commutative $H$-equivariant diagram
\begin{equation*}
  \begin{CD}         
\mathfrak{g} @>\cong>> \mathfrak{g}^*\\
    @V x \mapsto x_H VV  @VV\text{restriction}V \\
 \mathfrak{h}@>>\cong> \mathfrak{h}^*
  \end{CD}
\end{equation*}
in which the horizontal isomorphisms are induced by the trace
pairing on $\End(V)$.  Via this diagram,
the definitions and results below
admit equivalent formulations
in terms of 
the coadjoint representations.
We sometimes also write $\xi_H$ for the restriction of $\xi \in \mathfrak{g}^*$ to $\mathfrak{h}^*$.

\begin{theorem}\label{thm:stab-terms-spectra-1}
  Let $x \in \mathfrak{g}$.  The following are equivalent:
  \begin{enumerate}[(i)]
  \item $x$ is $\mathbf{H}$-stable. \label{item:stable}
  \item $\ev(x) \cap \ev(x_H) = \emptyset$. \label{item:no-match}
  \end{enumerate}
\end{theorem}
\index{stability in $[\mathfrak{g}] \times [\mathfrak{h}]$}
\begin{definition*}
  We say that a pair $(\lambda,\mu) \in [\mathfrak{g}] \times [\mathfrak{h}]$ is
  \emph{stable} if $\ev(\lambda) \cap \ev(\mu) = \emptyset$.
\end{definition*}
 This condition is equivalent to asking that the multiset sum
$\ev(\lambda) + \ev(-\mu)$ not contain zero.  We will later
(\S\ref{Satake  parameters})
interpret that sum in terms of the Satake parameters of an
associated $L$-function. In this way, the failure of stability will be related to the situation
where the conductor of this $L$-function drops. 

We note that the set
$\{(\lambda,\mu): (\lambda,\mu) \text{ is stable}\}$ is dense
and open in $[\mathfrak{g}] \times [\mathfrak{h}]$.
Moreover, for each $\lambda \in [\mathfrak{g}]$,
the set $\{\mu : (\lambda,\mu) \text{ is stable}\}$ is dense and
open in $[\mathfrak{h}]$.

We turn to the proof
of theorem \ref{thm:stab-terms-spectra-1}.
\emph{We may and shall assume that $k$ is algebraically closed},
% so that
% the discussion of \S\ref{GGP kbar}
% applies.
so that either
$k_1 = k \times k$ (the \emph{unitary case})
or
$k_1 = k$ (the \emph{orthogonal case}).
A key ingredient
is the following geometric
characterization of
when $x$ and $x_H$ share an eigenvalue:
\begin{lemma*}
  Let  $x \in \mathfrak{g}$
  and $c \in k$. 
  The following are equivalent:
  \begin{enumerate}[(i)]
  \item $c \in \ev(x) \cap \ev(x_H)$.
  \item $x$ has either  
    \begin{itemize}
    \item an isotropic eigenvector in $V_H^+$ with eigenvalue
      $c$, or
    \item  an isotropic eigenvector in $V_H^-$ with eigenvalue $-c$.
    \end{itemize}
  \end{enumerate}
\end{lemma*}
To see the significance of the two cases, consider
the simple example of $G = \GL_2(k)$: if a $2 \times 2$ matrix $(a_{ij}) \in \mathfrak{g}$ has an eigenvalue that coincides
with the upper-left entry $a_{11}$, then either $a_{12}$ or $a_{21}$ must vanish,
corresponding to the two cases above. 

\begin{proof}
  Recall from \S\ref{sec:geom-char-ev}
  that $c$ belongs to $\ev(x)$ exactly when there
  is an isotropic vector in $V^{\pm}$ with eigenvalue $\pm c$.
  Thus (ii) implies (i).

  Conversely, suppose  $c \in \ev(x) \cap \ev(x_H)$.
  There is then
  \begin{itemize}
  \item  an isotropic eigenvector
    $v \in V_H^+$ for $x_H$ with eigenvalue $c$, and also
  \item an isotropic eigenvector
    $w \in V^-$
    for $x$ with eigenvalue $-c$.
  \end{itemize}
  We have
  \begin{equation}\label{eq:adjointness-v-w}
    \langle (x - c) v, w \rangle
    = - \langle v, (x + c) w \rangle = 0.
  \end{equation}
  and thus either:
  \begin{itemize} 
  \item $xv =c v$, so that $v$ is an isotropic eigenvector for $x$ in $V_H^+$, or
  \item $xv \neq c v$, so that 
    $e^+$ (the $V^+$ component of $e$, as in \S\ref{sec:
      standard inclusions})
    is a multiple of $(x - c) v$.
    By \eqref{eq:adjointness-v-w},
    it follows that
    $\langle e^+, w \rangle  =0$,
    hence $w$ is an isotropic eigenvector for $x$ in $V_H^-$.
  \end{itemize}
\end{proof}

Recall the Hilbert-Mumford criterion: $x$ is not $\mathbf{H}$-stable if and only if there 
there is a nontrivial $1$-parameter subgroup
$\gamma : \mathbb{G}_m \rightarrow \mathbf{H}$ with respect to
which $x$ has only nonnegative weights;
we say then that \emph{$\gamma$ witnesses the  failure of stability  for $x$}.
Equivalently,
consider the decomposition of $V$
into weight spaces for $\gamma$:
\[
  V = \bigoplus_{i \in \mathbb{Z}} V_i,
  \quad
  \text{$\gamma(t)$ acts on $V_i$ by the scalar $t^i$.}
\]
% , so that 
Then
\begin{equation}\label{eq:stability-iff}
  \text{ $x$ is not stable }
  \iff
  \text{ there exists $\gamma \neq 1$
    so that }
  x V_i \subseteq  \bigoplus_{j \geq i} V_j
  \text{ for all $i$.}
\end{equation}

To prove theorem \ref{thm:stab-terms-spectra-1},
it is enough to show that
the following are equivalent:
\begin{enumerate}[(a)]
\item $x$ has an isotropic eigenvector in $V_H^+ \cup
  V_H^-$.
\item Some $\gamma$ as above witnesses the failure of stability for $x$.
\end{enumerate}
\emph{(a) implies (b).}
We assume that $x$ has an isotropic eigenvector $v_1 \in V_H^+$; the other case is handled
analogously.
By a standard lemma,
we may choose an isotropic vector $v_2 \in V_H^-$
for which $\langle v_1, v_2 \rangle = 1$;
we refer to \cite[p.29]{MR0344216} in the orthogonal case $k_1 =
k$,
while
in the unitary case $k_1 = k \times k$, we use that the spaces
$V_H^{\pm}$ are themselves totally isotropic and in duality.
The subspace $k v_1 \oplus k v_2$ of $V_H$
is then nondegenerate,
and so
with $W := (k v_1 \oplus k v_2)^\perp$,
\begin{equation}\label{eq:V-kv1-W-kv2}
  V = k v_1 \oplus W \oplus k v_2.
\end{equation}
Write
$x v_1 = c v_1$.
Then $\langle (x + c) W, v_1 \rangle = 0$,
so
$x W \subseteq W + v_1^\perp = v_1 \oplus W$,
and thus the matrix of $x$
with respect to \eqref{eq:V-kv1-W-kv2}
is upper-triangular.
The one-parameter subgroup $\gamma$  of $H$ given by
$\gamma(t) v_1 = t v_1$, $\gamma(t)|_{W} = 1$,
$\gamma(t) v_2 = t^{-1}$
then 
witnesses the failure of stability for $x$.

\emph{(b) implies (a).}  Suppose $\gamma$ witnesses the
failure of stability for $x$.
Since $\gamma$ is nontrivial,
some $V_i$ with $i \neq 0$ is nontrivial.
The spaces
$V_{>0} := \oplus_{i > 0} V_i$ and
$V_{<0} := \oplus_{i < 0} V_i$ are totally isotropic,
contained in $V_H$,
and
in duality,
hence both nonzero.
The nonzero $x$-stable isotropic subspace
$V_{>0}$
of $V_H$ thus contains an isotropic eigenvector $v$ for $x$.
In the unitary case,
we further split $v = v^+ + v^-$ to get an isotropic
eigenvector in $V_H^+ \cup V_H^-$.

\subsection{Fibers of $x \mapsto ([x],[x_H])$}
\label{sec:fibers-x-mapsto}
We retain
the setting of \S\ref{sec:stab-terms-spectra},
and
assume that $k$ is algebraically closed.
\begin{theorem}\label{thm:stab-terms-spectra-2}
  The morphism of varieties
  \begin{equation}\label{eqn:princ-bundle}
    \{\text{$H$-stable $x \in \mathfrak{g}$}\}
    \rightarrow 
    \{\text{stable $(\lambda,\mu) \in [\mathfrak{g}] \times [\mathfrak{h}]$}\}
  \end{equation}
  \[
  x \mapsto ([x],[x_H])
  \]
  defines a principal $H$-bundle.
  In particular, for stable
  $(\lambda,\mu) \in [\mathfrak{g}] \times [\mathfrak{h}]$, the
  fiber
  \begin{equation} \label{Olambdamudef} 
    \mathcal{O}^{\lambda,\mu} :=  \{x \in \mathfrak{g} : [x] = \lambda, [x_H]
    = \mu \}    \end{equation}
-- which consists entirely of $H$-stable elements, by the prior theorem -- is an $H$-torsor.
\end{theorem}
As before, it suffices
to consider the first map \eqref{eqn:princ-bundle}.
\emph{We may and shall assume that $k$ is algebraically closed.}
We require several lemmas.  
\begin{lemma}\label{lem:surj-inv-map}
  The map $\mathfrak{g} \rightarrow [\mathfrak{g}] \times
  [\mathfrak{h}]$
  given by $x \mapsto ([x], [x_H])$
  is surjective.
\end{lemma}
\begin{proof}
Recall from  \S\ref{sec:proof-spher-prop} that the quotient $\mathbf{X}$  of $\GG \times \HH$
by the diagonally embedded $\HH$ is a spherical variety, i.e.,
each Borel subgroup  of $\GG \times \HH$  has an open orbit on $\mathbf{X}$.
In particular, fixing such a Borel $\mathbf{B}$,     the associated moment map
\[ T^* X \rightarrow \mathfrak{g}^* \oplus \mathfrak{h}^*\]
has image which surjects onto $\mathfrak{b}^*$, because $\mathbf{B}$ acts simply transitively  on an open subset of $\mathbf{X}$. 
 
 This readily implies that the composition
 $ T^* X \rightarrow [\mathfrak{g}^*] \times [\mathfrak{h}^*]$
 is surjective. Indeed, if $\mathfrak{t}$ is the torus quotient of $\mathfrak{b}$,
and $\lambda \in \mathfrak{t}^* \hookrightarrow \mathfrak{b}^*$, 
all extensions of $\lambda$
  from $\mathfrak{b}^*$ to
to $\mathfrak{g}^* \oplus \mathfrak{h}^*$ have the same image in $[\mathfrak{g}^*] \times [\mathfrak{h}^*]$;
this common   image corresponds to the class of $\lambda$ in $\mathfrak{t}^*$ modulo the Weyl group.

The map $T_{x_0}^* X \rightarrow [\mathfrak{g}^*] \times
[\mathfrak{h}^*]$ is thus also surjective, where $x_0 \in X$
belongs to the open orbit.  The image of the moment map there is
 \[ \mbox{ orthogonal complement of $\diag \mathfrak{h}$ inside $\mathfrak{g}^* \oplus \mathfrak{h}^*$,}\]
which is precisely the set $\{ (\xi, -\xi_H): \xi \in \mathfrak{g}^* \}$. 
Thus every element of $[\mathfrak{g}^*] \times [\mathfrak{h}^*]$
is of the form $[\xi] \times [-\xi_H]$; negating the second coordinate gives the result. 
\end{proof}

\begin{lemma}\label{lem:stab-impl-reg}
  Let $x \in \mathfrak{g}$ be $H$-stable.
  Then $x$  and $x_H$ are regular.
\end{lemma}
\begin{proof}
  Suppose, to the contrary, that either $x$ or $x_H$ is irregular.
  We divide into cases
  using  \S\ref{sec:characterize-regular-elements-isotropic-ev}  
  and, in each case, produce an isotropic eigenvector $v \in V_H^+$,
  contradicting stability by the lemma of \S\ref{sec:stab-terms-spectra}:
  \begin{itemize}
  \item[(a)] $x$ has an isotropic eigenspace $W \subset V^+$ of dimension
  $\geq 2$. Take $v$ to be any nonzero element of $W \cap V^+_H$. 
  \item[(a)'] $x_H$ has an isotropic eigenspace $W \subset V_H^+$ of dimension
  $\geq 2$. 
Take $v$ to be any nonzero element
  of the kernel of $W \stackrel{x}{\rightarrow} V^+/V^+_H$. 

  \item[(b)] there is a subspace $W$ of  $V^+$, of dimension $\geq 3$, on which $x$ 
  is identically zero. Take $v$ to be a nonzero  isotropic element
  of $W \cap V^+_H$; it is possible because this space is at least $2$ dimensional.
  
  \item[(b)']
   there is a subspace $W$ of $V^+_H$, of dimension $\geq 3$, on which $x_H$ 
  is identically zero.  Take $v$ to be a nonzero isotropic element of 
the kernel of $W \stackrel{x}{\rightarrow} V^+/V^+_H$; it is possible because this space
is at least $2$ dimensional.

  \end{itemize}
  
\end{proof}

The proof of theorem \ref{thm:stab-terms-spectra-2}
requires also a further stability characterization:
\begin{lemma}\label{lem:stab-char-3}
  The equivalent conditions (i) and (ii) of theorem \ref{thm:stab-terms-spectra-1}
  are also equivalent to:
  \begin{enumerate}
  \item[(iii)]
    The $k_1[x]$-module
    \[k_1[x]  e \subset V\]
    generated by $e$ (see \S\ref{sec: standard inclusions})
    is:
    \begin{itemize}
    \item[-] all of $V$
      in the unitary cases; 
    \item[-] a   nondegenerate subspace of codimension $\leq 1$
      in the orthogonal cases. 
    \end{itemize}
  \end{enumerate}
\end{lemma}
 \begin{remark*}
  Rallis and Schiffmann \cite[Thm 6.1, Thm
  17.1]{2007arXiv0705.2168R} obtained a related\footnote{ We note that in the orthogonal
    case,
    our results do not exactly agree with theirs,
    due to a slight inaccuracy in the latter: the first paragraph of the
    proof of \cite[Thm 17.1]{2007arXiv0705.2168R}
    suggests that the Lie algebra of the orthogonal
    group of a nontrivial quadratic space is nontrivial, which
    fails when the latter is one-dimensional.  }
  equivalence.
  See also \cite{MR3495795} and \cite{MR3245011}.
\end{remark*}
\begin{proof}
  Observe that
  \[ U =  (k_1[x] e)^{\perp} \subset V.\] 
  is a $k_1[x]$-stable
  subspace of $V_H$ which is nondegenerate precisely
  when $k_1[x]e$ is.
  In view of the lemma of \S\ref{sec:stab-terms-spectra},
  it is enough to show that
  the following are equivalent:
  \begin{enumerate}[(a)]
  \item 
      $x$ has no isotropic  eigenvector
      in $V_H^+ \cup V_H^-$.
    \item 
      $U$ is trivial in unitary cases,
      and nondegenerate of dimension $\leq 1$ in orthogonal cases.
    \end{enumerate}
    
    If (b) holds,
    then $U$ is anisotropic;
    since any eigenvector
    in $V_H^+ \cup V_H^-$ belongs to $U$,
    it follows that (a) holds.
    
    Conversely, assuming (a),
    we proceed separately
    in the unitary and orthogonal cases:
    \begin{itemize}
    \item In the unitary case,
      the spaces $V^{\pm}$ are isotropic.  Since
      $x$ is $k_1$-linear, our hypothesis implies that the spaces
      $U^{\pm}$ contain no eigenvectors for $x$.  It follows that
      $U^{\pm} = \{0\}$ and thus $U = \{0\}$.
    \item In the orthogonal case,
      there is a maximal isotropic subspace
      $X$ of $U$ stabilized by $x|_U$. 
      If $X \neq \{0\}$, then $x|_U$ has an eigenvector in $X$, hence
      an isotropic eigenvector, contrary to our hypothesis.  Thus
      $X = \{0\}$.  Therefore $\dim(U) \leq 1$
      and, if $U$ is nonzero, then the quadratic form must be nonzero on it.
    \end{itemize}
  \end{proof}

We finally prove theorem \ref{thm:stab-terms-spectra-2}.
Consider the unique morphism $j$
fitting into the commutative diagram
\begin{displaymath}
  \xymatrix{
    \mathfrak{g}^s \ar[d] \ar[dr]^{\, \, x \mapsto ([x],[x_H])}   &   \\
    \mathfrak{g}^s \giit H \ar[r]_j
    & ([\mathfrak{g}] \times [\mathfrak{h}])^s
  }
\end{displaymath}
where a superscripted $s$ denotes the subset of stable elements
in the sense of either theorem \ref{thm:stab-terms-spectra-1}
or the definition that follows it. 

By lemma \ref{lem:surj-inv-map} (and again theorem \ref{thm:stab-terms-spectra-1}) $j$ is surjective.
We will show in addition, with notation as in
\eqref{Olambdamudef}, that if $(\lambda, \mu) \in  ([\mathfrak{g}] \times [\mathfrak{h}])^s$, then
\begin{quote}
 if $x, y \in \mathcal{O}^{\lambda,\mu}$,
there is a unique $h \in H$ with $h \cdot x=y$. 
\end{quote}
This implies that $j$ is injective;   then (by Zariski's main theorem, using characteristic zero)
$j$ will be an isomorphism.  It also implies (see the lemma
of \S\ref{sec:stability-general})
  that the map $\mathfrak{g}^s \rightarrow \mathfrak{g}^s \giit H$
  is a principal $H$-bundle.

By lemma \ref{lem:stab-impl-reg},
$x$ and $y$ 
belong to the unique regular (open) $G$-orbit
$\mathcal{O}$ contained in the fiber $\mathcal{O}^\lambda$ above $\lambda$.
Therefore there exists $g_0 \in G$ so that $g_0 \cdot y = x$.
Set
$e' := g_0 e$.  From $[x_H] = [y_H]$, we deduce
that
\begin{equation}\label{eq:x-E-vs-x-E-prime}
  [x_H] = [x_{H'}],
\end{equation}
where we define $x_{H'} := E' x E'$,
with $E'$  the orthogonal projection onto the orthocomplement of $e'$.
Since $H$
is the $G$-stabilizer of $e$, it will suffice to show the
following:
\begin{claim*}
  There is a unique $g$ in the $G$-stabilizer
  of $x$ for
  which $g e' = e$.
\end{claim*}
Indeed, the element $h := g g_0$
then belongs to $H$ and satisfies $h \cdot y = x$,
and any such $h$ arises from some $g$ as in the claim. 

We show first
that
for all $n \in \mathbb{Z}_{\geq 0}$,
\begin{equation}\label{eq:inner-products-with-powers-of-x-match-up-bam}
  \langle x^n e, e \rangle
  = \langle x^n e', e' \rangle.
\end{equation}
 For this, it suffices to
show that the formal power series
\[
  \sum_{n \in \mathbb{Z}_{\geq 0}}
t^n
\langle x^n e, e \rangle
=
\langle e, (1 - t x)^{-1} e \rangle
\]
is unchanged by replacing $(x,e)$ with $(x,e')$.
To that end, it suffices in view of \eqref{eq:x-E-vs-x-E-prime}
to establish the identity
\begin{equation}\label{eq:power-series-magic}
  % \sum_{n \in \mathbb{Z}_{\geq 0}}
  % t^n
  % \langle x^n e, e \rangle
  % =
  \frac{
    \langle e, (1 - t x)^{-1} e \rangle
  }
  {
    \langle e, e \rangle
  }
  =
  \frac{\det(1 - t x_H)}{\det(1 - t x)}.
\end{equation}
% since the RHS
% doesn't change if we replace $(x,e)$ with $(x, e')$.
For this,
we extend $e$ to an orthogonal $k_1$-basis $e=e_1, e_2, \dots, e_n$
of $V$.
The LHS of \eqref{eq:power-series-magic}
is then the upper-left matrix entry
of $(1 - t x)^{-1}$.
The identity
\eqref{eq:power-series-magic}
follows from Cramer's rule,
noting
that
the matrix of
$(1-tx_H)$
with respect to the basis $e_2, \dots, e_n$ of $V_H$
is
the lower-right
$(n-1) \times (n-1)$ submatrix
of the matrix of $(1-tx)$ with respect 
to $e_1,\dotsc,e_n$.

% where we applied Cramer's rule and \av{added:} the equality is one of formal power series in $k_1[[t]]$. 
% The final expressions
% are the same for $(x,e)$ and for $(x, e')$, proving \eqref{eq:inner-products-with-powers-of-x-match-up-bam}.

% \av{ I don't know, it still seems kinda confusing. How about we spell out a bit like this?
% Extend $e$ to a $k_1$-basis  $e=e_1, e_2, \dots, e_n$
% satisfying $\langle e_i, e_j \rangle = 0$. 
% We extend scalars from $k$ to $k[[t]]$, so that $(1-tx)$  
%  is understood as a $k_1[[t]]$-linear endomorphism of $V \otimes_{k} k_1[[t]]$.
%  We will apply Cramer's rule to this endomorphism. 
% The top $(1,1)$ matrix entry 
% of $(1-tx)^{-1}$ is given by  the ratio 
%  $\frac{\langle e, (1- tx)^{-1} e \rangle}{\langle e, e \rangle}$.
%  Moreover, when we take the lower $(n-1) \times (n-1)$ submatrix
%  for the matrix of $(1-tx)$, we obtain precisely the matrix of $(1-tx_H)$
%  with respect to the basis $e_2, \dots, e_n$. }
% \pn{
%   I'm down for spelling things out in terms of a basis as you
%   indicate.
%   I think the identity of formal power series makes perfect
%   sense as it's written,
%   and we lose little by regarding $t$ as an arbitrary element
%   of the (infinite) field $k$, or as a formal variable.
% }

%
%
%
%
%
%
%
%
%
%
%
%
%
%
%
%
%
%
%
%
%
%
%
%
%
%
%
%
%
%
%
%
%
%
%
%
%
%
%

Consider the submodules $W := k_1[x] e$
and $W' := k_1[x] e'$ of $V$.
Let $U, U'$ denote the orthogonal complements
in $V$ of $W,W'$.
By lemma \ref{lem:stab-char-3},
the spaces $U, U'$ are nondegenerate,
and so $V = W \oplus U = W' \oplus U'$.
Since $x$ is skew-adjoint, it preserves these decompositions.
By
\eqref{eq:inner-products-with-powers-of-x-match-up-bam}
there
is a unique isometric $k_1[x]$-equivariant isomorphism
$g : W' \rightarrow W$ for which $g e' = e$.  In particular,
$\dim(W) = \dim(W')$,
so that $\delta := \dim(U) = \dim(U')$.

%\av{Added this to spell out:} Suppose that $g$ is as in the {\em Claim.} Necessarily 
%it carries $W'$ to $W$, and so also $U'$ to $U$. 
% Now $g|_{W'}$ is uniquely specified by the fact that it carries
% $x^n e'$ to $x^n e$. Moreover,  $U'$ is nonzero only in the orthogonal case,
% where its dimension is one, and there $g|_{U'}$ is 
% uniquely specified by its determinant being equal to $1$. This verifies the uniqueness
% assertion of the {\em Claim.} We now check existence.
% 
If $\delta \neq 0$, then we are in the orthogonal case with $\delta = 1$.
By Witt's theorem, there is
an extension of $g$ to an isometric isomorphism
$V \rightarrow V$.  There are two such extensions,
which may be obtained from one another by composing with the
nontrivial element of the orthogonal group of the line $U$.
There is thus a unique extension which belongs to $G$ (the
\emph{connected component} of the orthogonal group of $V$).
This extension remains $x$-equivariant -- indeed,
$gU = U'$, and $x$ acts on the one-dimensional spaces $U$ and $U'$
by zero.   
 The proof of the claim
is thus complete in this case.

Suppose now that $\delta = 0$, so that $g : V \rightarrow V$
is the unique isometric $x$-equivariant morphism for which
$g e' = e$.  If we are not in the orthogonal case, then $g$
defines an element of $G$, and so the proof of the claim is
likewise complete.  What remains to be checked is that in the
orthogonal case, $g$ belongs to $G$,
or equivalently,
either that $\det(h) = 1$ or $\det(g) = 1$. 
To do this we must use the one piece
of information not used to date, namely, that not merely
eigenvalues but also Pfaffians match.

Consider first the case that $\dim(V)$ is odd.  Then
\[
  \pf(x_H) = \pf(h \cdot y_H) = \det(h) \pf(y_H)
  = \det(h) \pf(x_H),
\]
so if $\pf(x_H) \neq 0$,
then $\det(h) = 1$. 
Otherwise, $\pf(x_H) = 0$.
Then $0 \in \ev(x_H)$,
so the kernel of $x_H$ is nonzero, necessarily even-dimensional,
and so
of dimension $\geq 2$.
Thus the kernel of $x$ contains
a nonzero element $v \in V_H$  (not necessarily isotropic).
Clearly $x^n v$ is orthogonal
to $e$ for all $n \in \mathbb{Z}_{\geq 0}$,
hence $v \in U$,
contrary to our assumption
that $\delta = 0$.

If $\dim(V)$ is even,
then we may argue
as above
that $\det(g) = 1$ (since $g \cdot  x  = x$)
unless $\pf(x) = 0$.
In the case $\pf(x) = 0$,  the kernel of $x$ is again even-dimensional,
and thus contains a nonzero element $v \in V_H$;
as above,  this contradicts the assumption that $\delta =0$. 
We thereby deduce as required that
$\det(g) = 1$.

This completes the proof of the claim, hence of theorem \ref{thm:stab-terms-spectra-2}.

\section{Satake parameters and $L$-functions}

\label{Satake  parameters}
In this section we show that the notion of stability
is closely related to the analytic notion of conductor dropping.
This is not used in the proof of our main result but, of course, is helpful in interpreting it.

\subsection{Local Langlands and infinitesimal character}
Let us first recall the relationship between the Langlands
parameters and infinitesimal characters
(see, e.g.,
\cite[\S6]{2012arXiv1212.2192A} or \cite[\S11]{MR546608}).
Let $\mathbf{G}$ be a reductive group over $\mathbb{R}$.
Let $\pi$ be an irreducible admissible representation of $G$.
The local Langlands parameterization attaches to
$\pi$ a
conjugacy class of
representations
of the real Weil group:
\[\phi_{\pi}: W_{\mathbb{R}} \longrightarrow \underbrace{ G^\vee \rtimes \Gal(\C/\mathbb{R})}_{\LG}.\]
Here $G^\vee$ is the complex dual group of $\GG$, and the right-hand side
defines the Langlands dual $\LG$. 

Now restriction to $\mathbb{C}^* = W_{\C} \subset W_{\mathbb{R}}$
gives a homomorphism $\phi_{\pi}^0: \C^* \longrightarrow
G^\vee$,
which we may express uniquely as
\[ \phi_{\pi}^0(e^t)= \exp(t \lambda_{\pi} + \bar{t} \mu_{\pi})\]
for some commuting elements $\lambda_{\pi}, \mu_{\pi}$
in
$\mathfrak{g}^\vee$, the Lie algebra of $G^\vee$.

If $\pi$ is tempered, then the image of $\phi_{\pi}$ is bounded,
which imposes the constraint
\begin{equation} \label{tempered restriction} \mu_{\pi} = - \overline{\lambda_{\pi}}, \mbox{ thus }
  \phi_{\pi}^0(e^t)= \exp(2 i \mathrm{Im} \left( t \lambda_{\pi}  \right)), \end{equation}
where we write
$t \lambda_\pi = \Re(t \lambda_\pi) + i \Im(t \lambda_\pi)$.

\begin{example*}
  If $G = \GL_n(\mathbb{R})$ and $\pi$ is a
  principal series representation with parameters
  $i \nu_1,\dotsc,i \nu_n \in i \mathbb{R}$,
  then one may
  choose $\phi_\pi$ so that
  $\lambda_\pi = \mu_\pi = \diag(i \nu_1, \dotsc, i \nu_n)$.  If
  $G = \GL_2(\mathbb{R})$ and $\pi$ factors through the
  discrete series representation of $\PGL_2(\mathbb{R})$
  of lowest weight
  $k$,
  then one may arrange that
  $\lambda_\pi = - \mu_\pi = \diag(\frac{k-1}{2},
  \frac{1-k}{2})$.
\end{example*}

We observe next that one may identify
\begin{equation} \label{dualdual-0} \mathfrak{g}^\vee  \git G^\vee \simeq [\mathfrak{g}_{\C}^*]\end{equation}
Indeed, if we fix maximal tori $T^\vee \subset G^\vee$
and $\mathbf{T}_{\C} \subset \mathbf{G}$,
then $T^\vee$ and $\mathbf{T}_{\C}$ are dual, canonically
up to the action of the Weyl group.
For complex tori $T_1$ and $T_2$, an identification of $T_1$ with the dual of $T_2$
identifies the Lie algebra of $T_1$ with the complex linear dual of the Lie algebra for $T_2$. 
So there is a canonical identification
\[   \C[(\mathfrak{t}^\vee)^*]^{W} \simeq \C[\mathfrak{t}_{\C}]^W\]
which induces \eqref{dualdual-0}.

\begin{lemma*}
  The identification \eqref{dualdual-0} carries   $\lambda_{\pi}$ to the infinitesimal character of $\pi$. 
\end{lemma*}

The lemma justifies our notational abuse of using
the same symbol $\lambda_\pi$ above as we had in
\S\ref{sec:infin-char} for the infinitesimal character.

\proof
Recall (\S\ref{ss:LC}) that the representation $\pi$ is a summand of the
unitarily
normalized parabolic induction from 
a discrete series representation $\sigma$ on the Levi factor $\mathbf{M}$ of a 
parabolic subgroup $\mathbf{P} \subset \mathbf{G}$.  
The Langlands parameters of $\sigma$ and $\pi$ are related by means
of the natural inclusion of the $L$-group of $\mathbf{M}$ into the $L$-group of $\mathbf{G}$,
and the natural map  
\[ [\mathfrak{m}^*_{\C}] \longrightarrow [\mathfrak{g}^*_{\C}]\]
carries the infinitesimal character of $\sigma$ to that of $\pi$. 
The two maps (of dual groups, and of dual Lie algebras) are compatible with reference to \eqref{dualdual-0}. 

We thereby
reduce to the case of a discrete series
representation, and then (by the characterization of the local Langlands correspondence for discrete series via infinitesimal characters, see   \cite[\S 11.2]{Borel})
we further reduce to the case of a real torus $\mathbf{S}$. 
This follows from  \cite[\S 9.3 eq. (2)]{Borel}   (note the misprint: the second occurrence of $\sigma\cdot x$ should be $\sigma \cdot \bar{x}$);
see also  \cite[\S 6]{AdamsVogan}.
\qed

 \subsection{GGP pairs}
 For simplicity, we restrict ourselves to the case $K=\R$, and leave the straightforward
extensions to $K=\C$ to the reader.

Let $(\mathbf{G}, \mathbf{H})$ now be a GGP pair,
in the sense of  \S \ref{sec:gross-prasad-pairs-inv-theory},
over $K=\R$; let $K_1/K$ be the associated $K$-algebra,
and let $n = [\dim_{K}(V)/2]$, where $V$ is the associated $K_1$-vector space. 
Thus
\[ \mathbf{G} = \left(\mbox{form of $\mathrm{SO}_{2n}$ or $\mathrm{SO}_{2n+1}$ or $\mathrm{GL}_n$ over $K$.} \right),\]
\[ \mathbf{H} = \left( \mbox{form of $\mathrm{SO}_{2n-1}$ or $\mathrm{SO}_{2n}$ or $\mathrm{GL}_{n-1}$ over $K$.} \right).\]

\subsubsection{Dual Lie algebra}  
The $K$-group $\mathbf{G}$ admits a  representation by $K_1$-linear automorphisms of $V$. The form
\[ x,y \in \mathfrak{g} \mapsto \mathrm{tr}_{V}(xy)\]
is actually $K$-valued and nondegenerate; it identifies
 $ \mathfrak{g} \simeq \mathfrak{g}^*$
(duality of real vector spaces).

We may assign to each $x \in \mathfrak{g}$
a multi set $\ev(x)$ of complex numbers -- namely,
the multiset of eigenvalues of $x$ in the standard representation, 
where we remove zero with multiplicity one in the odd orthogonal case.
By means of the identification above,
we may also make sense of $\ev(x)$ for $x \in \mathfrak{g}^*$;
this is a set of size $2n, 2n, n$ in the three cases above. Similarly $\ev(y)$ for $y \in \mathfrak{h}^*$
is a set of size  $2n-2, 2n, (n-1)$ in the three cases above.

\subsection{Rankin-Selberg  $L$-function for GGP pairs}
The tensor product of standard representations
on ${H}^\vee  \times {G}^\vee $ extends to a homomorphism
% \pn{
%   Probably ${}^L ({H}  \times {G})$
%   for the LHS?
% }
\[ {}^L ({H}^\vee  \times {G}^\vee ) \rightarrow \mathrm{Sp}_{2n-2} \times \mathrm{O}_{2n} \mbox{ or } \mathrm{O}_{2n} \times \mathrm{Sp}_{2n} \mbox{ or }
(\GL_{n-1} \times \GL_n(\C)) \rtimes \{ \pm 1\}.\]
We define the Rankin-Selberg representations $\rho$ to be the
natural representations
of the right hand side of dimensions
$(2n-2) \cdot (2n), (2n)^2, 2 n(n-1)$ respectively; in the 
last case we induce the standard representation of $(\GL_{n-1} \times \GL_n(\C)) $ to the disconnected group.

\begin{lemma}
The correspondence  \eqref{dualdual-0}
\[ {A}^\vee  \times {B}^\vee  \in \mathfrak{g}^\vee  \git G^\vee  \times \mathfrak{h}^\vee  \git H^\vee  \leftrightarrow 
A' \times B' \in [\mathfrak{g}_{\C}^*] \times [\mathfrak{h}_{\C}^*]\]
has the property that
\begin{equation}\label{eqn:eigenvalue-for-RS-L-factor}
  \mbox{eigenvalues of $\rho(A^\vee  \times
    B^\vee )$}= 
( \ev(A') + \ev(B') )  \times \begin{cases} \{1,-1\},
  \mbox{unitary case} \\ \{1\}, \mbox{orthogonal
    cases} \end{cases},
\end{equation}
On the right hand side we interpret the sum of two multisets as all pairwise sums of elements from
the individual multisets. 
\end{lemma}
 
 \proof
  Let $\Omega_G$ be the multiset of weights arising from the standard representation for $\mathfrak{g}_{\C}$
where we remove all zero weights. Define $\Omega_H$ similarly and let $\Omega = \Omega_G \times \Omega_H$;
this $\Omega$ is a Weyl-invariant multi set inside $\mathfrak{t}_{G,\C}^* \oplus \mathfrak{t}_{H,\C}^*$.  (We add subscripts
to clarify whether dealing with a torus for $G$ or a torus for $H$.)

Let $\Omega^\vee$ be the multi set of weights for the Rankin--Selberg representation of $\mathfrak{g}^\vee \oplus \mathfrak{h}^\vee$,
as just described above; it is a Weyl-invariant multi set in $(\mathfrak{t}_G^\vee)^* \oplus (\mathfrak{t}_H^\vee)^*$.
 
Therefore what we must show is that
$\Omega$ corresponds to $\Omega^\vee$ 
under \eqref{dualdual-0} and the trace duality, i.e., under the sequence of identifications 
\begin{equation} \label{compo} \mathfrak{t}_{G,\C}^* \stackrel{\text{trace pairing}}{\longrightarrow} \mathfrak{t}_{G,\C} \longrightarrow (\mathfrak{t}_G^\vee)^*\end{equation}
and its analogue for $H$. 
The final identification is well-defined only up to the Weyl group (this ambiguity makes no difference
for comparing Weyl-invariant multisets).

 There is a standard basis for roots
for $\mathrm{SO}_{2n}, \mathrm{SO}_{2n+1}, \mathrm{GL}_n$,
labelled as \[\{ \pm e_i \pm e_j\}_{[n]}, \{ \pm e_i \pm e_j, \pm e_j\}_{[n]}, \{e_i - e_j\}_{[n]}\] respectively,
where we use $[n]$ as a shorthand for $1 \leq i,j \leq n$, and we always omit zero roots. 
Label similarly the roots for $H$ as \[\{ \pm f_i \pm f_j, \pm f_i\}_{[n-1]}, 
\{ \pm f_i \pm f_j\}_{[n]}, \{f_i - f_j\}_{[n-1]}\]
in the three respective cases.  
The composite  \eqref{compo} 
carries $e_i$ to $e_i^\vee$ (the corresponding standard basis for the space of cocharacters); 
with these identifications we readily compute
\[ \Omega = \{ \pm e_i \pm f_j \} \mbox{ and }\Omega^\vee  = \{\pm e_i^\vee \pm f_j^\vee  \}\] 
\qed

 \subsection{Satake parameters, conductor drop, and stability}
\label{sec:stab-regul-terms}
 
 The set appearing in  
 \eqref{eqn:eigenvalue-for-RS-L-factor}
 almost determines the Rankin-Selberg $L$-function: 
 
Any irreducible admissible representation of the real Weil group 
is at most $2$-dimensional. If it has bounded image  then its restriction to $\mathbb{C}^*$
is  of the form 
\[ z \mapsto |z|^{it} \quad \mbox{ or } \quad z \mapsto  z^{n/2+it} \bar{z}^{-n/2+it} \oplus  z^{-n/2+it} \bar{z}^{n/2+it} \ \ (t \in \mathbb{R}, n \in \mathbb{Z}).\]
For a complex number $z = x+iy$, write $z^+ = |x| + iy$.
The associated $L$-factor is given respectively by
\[ \Gamma_{\R}(s + \eps + it) \quad \mbox{ or } \quad   \prod_{i=1}^{2} \Gamma_{\R}\Bigl(s+\left(\frac{n}{2}+it \right)^+ + \eps_i\Bigr)\]
where $\eps_i \in \{0,1\}$ and where
$\Gamma_{\mathbb{R}}(s) = \pi^{-s/2} \Gamma(s/2)$ is the real $\Gamma$
function.

From this it readily follows that if $\pi$ is a  tempered representation of $G$
with $\lambda_{\pi} = A^\vee$ (equivalently, in the notation of the prior lemma, with infinitesimal
character $A'$), and similarly $\sigma$ a tempered representation of $H$
with $\lambda_{\sigma} =B^\vee$ (equivalently, with infinitesimal character $B'$), 
then we have
\[  L_{\mathbb{R}}(\pi \times \sigma,  \rho, s) =  \prod_{i} \Gamma_{\mathbb{R}}(s+   \lambda_i^+ + \eps_{i}),\]
where the $\lambda_i$  range through the multiset appearing in
\eqref{eqn:eigenvalue-for-RS-L-factor}.
   For this reason, we will refer to the multiset on the LHS of
\eqref{eqn:eigenvalue-for-RS-L-factor} as the \emph{multiset of Satake parameters for the Rankin-Selberg $L$-function.}

We may now reinterpret  theorem \ref{thm:stab-terms-spectra-1}.
Here we denote by $\sigma^\vee$
the contragredient of $\sigma$;
its Satake parameters
are the negatives of those of $\sigma$. 
\begin{lemma*}
  \label{lem:stable regular}
  Let $\pi$ and $\sigma$
  be irreducible representations of $G$ and $H$, respectively,
  with
  infinitesimal characters $\lambda \in [\mathfrak{g}^*_{\C}]$ and $\mu \in [\mathfrak{h}^*_{\C}]$. 

  The following are equivalent:
  \begin{enumerate}
  \item[(a)] $(\lambda, \mu)$ is stable.
  \item[(b)] No Satake parameter for the local $L$-factor $L(\pi \times \sigma^{\vee}, \rho, s)$
    is equal to zero.    \end{enumerate}
\end{lemma*}

The significance of this reinterpretation is that (b) is related to an important analytic phenomenon --  dropping of the analytical conductor. 
It would  be interesting to see if this relation between
stability and conductor drop extends to other integral
representations.

\section{Volume forms}
\label{sec:volume-forms}

Let\index{volume form}
$k$
be a field of characteristic zero.
A \emph{volume form} on a smooth $k$-variety $\mathbf{Y}$
is simply
an everywhere nonvanishing global section of the bundle
of top-degree algebraic differential forms.
When $k = \mathbb{R}$,
volume forms
give rise to measures (\S\ref{sec:appl-vari-over}).

The purpose of this section is to describe the various volume
forms that exist on a Lie algebra, its dual, and its coadjoint orbits,
as well as the relationships between these forms that arise in the
context of a GGP pair.  As we explain in
\S\ref{sec:scal-limit-planch} and \S\ref{sec:sph-char-2}, the
results obtained here model the asymptotic representation theory
of $G$ and $H$.

We note that, in order to evaluate the constant in the main theorem of
this paper, we really need the exact relationships between these
volume forms (rather than, say, their relationship up to an
unspecified proportionality constant).

Some special cases of the foregoing results have been
established for certain compact groups $G$ and $H$ (see, e.g.,
\cite[Prop 4.2]{MR1818248} and \cite[Prop 3.1]{MR3185209}).

Throughout this section,
we work over an algebraically closed
field $k$ of characteristic zero.
In the following section,
we
deduce results over $k = \mathbb{R}$.

\subsection{Fibral volume forms}\label{sec:fibral-volume-forms}
Recall that a short exact sequence
$X \rightarrow Y \rightarrow Z$
of vector spaces
induces a natural isomorphism
$\det(Y^*) \cong \det(X^*) \otimes \det(Z^*)$,
where $\det$ denotes the top exterior power.

More generally,
given a smooth morphism of varieties,
a volume form on source and target induces a volume form on
each
fiber.
To be precise,
let
$f : Y \rightarrow Z$
be a morphism of smooth irreducible varieties,
and fix a regular value $z \in Z$, i.e.,
for each $y \in f^{-1}(z)$,
the induced map $T_y Y \rightarrow T_{z} Z$
of tangent spaces is surjective.
Then the fiber $X = f^{-1}(z)$ is smooth,
and we have a sequence of maps
\begin{equation*}
  \begin{CD}         
    X
    @>\text{inclusion}>>
    Y
    \\
    @. @VV f V \\
    @. Z.
  \end{CD}
\end{equation*}
We obtain
for each $x \in X$
a
short exact sequence
$T_x X \rightarrow T_x Y \rightarrow T_z Z$
and hence an identification
\begin{equation}\label{eq:identify-canonical-stalks}
  \det(T_x^* Y)
  \cong \det(T_x^* X) \otimes \det(T_z^* Z).
\end{equation}

Let $\beta$ and $\gamma$ be nowhere vanishing volume forms on
$Y$ and $Z$, respectively.   There is then a unique volume form
$\alpha$ on $X$ so that $\beta = \alpha \otimes \gamma$ under
\eqref{eq:identify-canonical-stalks} at each point of $X$.  We refer to $\alpha$
as
the \emph{fibral volume form} with respect to $\beta$ and
$\gamma$, and express this symbolically
by $\alpha = \beta/\gamma$. We shall also say, in this situation, that the sequence $X \rightarrow Y \rightarrow Z$
is compatible with the volume forms.
\index{volume form, fibral}

\subsection{Haar forms}
Let $V$ be a finite-dimensional vector space over $k$.
A \emph{Haar form}
  \index{volume form, Haar}
$\beta$ on $V$ is defined to be a translation-invariant volume
form.
It induces a \emph{dual form} $\beta^*$
on the dual space $V^*$.
This terminology will be applied
most frequently when $V = \mathfrak{h}^*$
for a reductive $k$-group $\mathbf{H}$.
\begin{example*}
  Suppose $\mathbf{H} = \boldGL_2(k)$.
  We may identify
  $\mathfrak{h}^*$
  with the space of $2 \times 2$ matrices.
  The Haar forms on $\mathfrak{h}^*$
  are the nonzero multiples of $d \xi_{1 1} \wedge d \xi_{1 2} \wedge d \xi_{2 1} \wedge d \xi_{2 2}$.
\end{example*}

\subsection{Symplectic volume forms}
\label{sec:algebraic-sympl-volume-forms}
Let $\mathbf{G}$ be a reductive $k$-group.
Recall (from \S\ref{sec:general-conventions})
that regular elements of $\mathfrak{g}^*$ are
those
whose stabilizer has minimal dimension, and that
we denote subsets of
regular elements by a subscripted $\reg$, as in
$\mathfrak{g}^*_{\reg}$.
We have:
\begin{lemma*}
  Every
  fiber $\mathcal{O}^\lambda$ of $\mathfrak{g} \rightarrow [\mathfrak{g}]$ contains a
  unique open orbit $\mathcal{O}^{\lambda}_{\reg}$.
  \index{$\mathcal{O}^\lambda$}
  \index{$\mathcal{O}^\lambda_{\reg}$}
  This orbit
  consists  precisely of the regular elements of that fiber.
  An element $\xi \in \mathfrak{g}^*$
  is regular if and only if
  the
  linear map $\mathfrak{g}^* = T_\xi \mathfrak{g}^* \rightarrow
  T_{[\xi]} [\mathfrak{g}^*]$
  obtained by differentiating the projection
  $\mathfrak{g} \rightarrow [\mathfrak{g}]$
  is surjective.
\end{lemma*}
\proof
See Theorem 0.1 and Theorem 3
of
\cite{MR0158024}.
\qed

The orbit $\mathcal{O}^{\lambda}_{\reg}$ carries a canonical $G$-invariant symplectic form $\sigma$,
and so also a $G$-invariant symplectic volume form
\index{volume form, symplectic}
$\frac{1}{d!} \sigma^d$,
by the algebraic version of the discussion of
\S\ref{sec:canonical-symplectic-form}:
the symplectic pairing
on $T_{\xi}(\mathcal{O}_{\reg}^\lambda) = \{\ad_x^* \xi : x \in
\mathfrak{g} \}$
is given by
\begin{equation}\label{eqn:defn-alg-symp-pairing}
  (\ad^*_x \zeta,
\ad^*_y \zeta)
\mapsto
\langle \zeta, [x,y] \rangle.
\end{equation}
The same discussion applies
to any coadjoint orbit,
but we require here only
the regular case.

\subsection{Affine volume
  forms}\label{sec:affine-measures}
Let $\mathbf{H}$ be a reductive $k$-group.
The quotient $[\mathfrak{h}^*]$
is an affine space (cf. \S\ref{sec:quotient-affine}).
Choose an isomorphism
$[\mathfrak{h}^*] \rightarrow \mathbb{A}^r$, or equivalently,
generators $p_1,\dotsc,p_r$ for the ring of $G$-invariant
regular functions on $\mathfrak{h}^*$.  The volume form
$d p_1 \wedge \dotsb \wedge d p_r$ on $[\mathfrak{h}^*]$ is
independent, up to scaling, of the choice of generators $p_i$;
indeed, any two choices will differ by an invertible
element of $\C[p_1, \dots, p_r]$.
An \emph{affine volume form} on $[\mathfrak{h}^*]$
is then defined
to be a
nonzero multiple of $d p_1 \wedge \dotsb \wedge d p_r$.
  
\begin{example*}
  Suppose $\mathbf{H} = \boldGL_r(k)$.
  We may identify
  $\mathfrak{h}^*$
  with the space of $r \times r$ matrices.
  By sending a matrix to its characteristic
  polynomial,
  we obtain an isomorphism
  \[[\mathfrak{h}^*]
  \stackrel{\sim}{\longrightarrow} \mbox{ monic
    polynomials $x^r + \sum_{1}^r a_i x^{r-i}$.}\]
An affine volume form is then given by
  $d a_1 \wedge \dotsb \wedge \ d a_r$.
\end{example*}

\begin{lemma*}
  For each Haar form $\beta_H$ on $\mathfrak{h}^*$
  there is a unique affine form
  $\gamma_H$
  on
  $[\mathfrak{h}^*]$
  so that for each $\mu \in [\mathfrak{h}^*]$,
  the fibral volume form for the sequence
  \[\mathcal{O}^{\mu}_{\reg} \rightarrow \mathfrak{h}^* \rightarrow [\mathfrak{h}^*]\]
  is the symplectic volume form.
\end{lemma*}
\begin{proof}
  See \cite[Lem C]{MR587333}.
\end{proof}

 Later it will be useful to have an explicit formula available for the
form on $[\mathfrak{h}]$. Choose 
 as usual a Chevalley basis $H_{i}, X_{\alpha}, X_{-\alpha}$  for $\mathfrak{h}$,
where $i$ ranges over simple roots and $\alpha$ over all positive roots;
in particular $\alpha_i(H_{i})=2$ and $[X_{\alpha_i}, X_{-\alpha_i}] = H_i$. 
Wedging these together gives a volume form on $\mathfrak{h}^*$.
Write $\mathfrak{t}$ for the Cartan subalgebra spanned by the $H_i$. 
Now for $\mu \in \mathfrak{t}^*$ regular
the natural projection gives an identification 
$\mathfrak{t}^* \simeq T_{\mu}[\mathfrak{h}^*]$. 
A short computation shows that the pullback to
$\mathfrak{h}^*$
of the affine
volume form is given by \begin{equation} \label{affine form}
   \prod_{\alpha > 0}
\langle \mu, \alpha^{\vee} \rangle \cdot  \ \bigwedge_{i} H_{i};
 \end{equation}
note that the Weyl group acts by the sign character
on both factors of \eqref{affine form}, so the product is invariant. 

\iftoggle{easyproof}
{
\warning{These are informal proof notes and may not be verified.}

\proof 
 Let $\mathfrak{t} = \langle H_{\alpha} \rangle$. Note that $H_{\alpha} = \alpha^{\vee}$
in usual notation. We compute that 
$ X_{\pm \alpha} \xi  = \pm \mu(H_{\alpha}) X'_{\mp \alpha}.$
 The only nonzero symplectic pairings between such vectors are
 \[ \omega(X_{\alpha}', X_{-\alpha}') = \pm \mu(H_{\alpha})^{-1}.\]
 The top wedge power of $\omega$
 thus assigns 
 to $\bigwedge_{\alpha > 0} (X_{\alpha}' \wedge X_{-\alpha}')$ the number
 \begin{equation} \label{oooo} \prod_{\alpha > 0} \mu(H_{\alpha})^{-1}.\end{equation}
 Fix on $\mathfrak{g}^*$ the volume form  defined by the basis $H_{i}, X_{\pm \alpha}$ for the dual space. 
 It follows that the  volume forms in the sequence 
 $\mathfrak{g} \cdot  \lambda \hookrightarrow \mathfrak{g}^* \rightarrow [\mathfrak{g}^*]$
 are  
 compatible so long as the 
 volume form on  $[\mathfrak{g}^*] \simeq   \mathfrak{t}^* / W$
 is given by
 \begin{equation} \label{affine form}
   (\prod_{\alpha > 0}
   \xi(H_{\alpha}) ) \, d\nu,
 \end{equation}
 where $\nu$ is the invariant volume form on $\mathfrak{t}^*$ associated to the basis $H_{i}$. 

 But the  pullback of the affine differential form from $[\mathfrak{g}^*] =\mathfrak{t}^*/W$
 to $\mathfrak{t}^*$ is given by  a multiple of \eqref{affine form}. 
 Indeed,
 this pullback is
 certainly of the form $P(\lambda) d\nu(\lambda)$ for some polynomial.
 Moreover, $P$ vanishes along each line $\lambda(H_{\alpha}) =
 0$,
 since
 the map $\mathfrak{t} \rightarrow \mathfrak{t}^* \git W$ is
 branched there.
 By computing degrees,
 we deduce finally
 that $P$ must be a scalar multiple of $\prod \lambda(H_{\alpha})$. 
}

\begin{definition*}
  Given a Haar form $\beta_H$,
  the \emph{normalized affine form}
  $\gamma_H$ on $[\mathfrak{h}^\wedge]$
  is the one associated
  \index{volume form, normalized affine}
  by the lemma.
\end{definition*}
\begin{remark*}
  When working over
  a local field,
  the normalized affine form
  is closely related to the scaling limit of the Plancherel
  measure; see \S\ref{sec:scal-limit-planch}.
\end{remark*}

\subsection{Orbital volume forms for a GGP pair}\label{sec:orbital-volume-forms}
Let $\mathbf{H} \hookrightarrow \mathbf{G}$
be a GGP pair over $k$.
Recall, from theorem \ref{thm:stab-terms-spectra-2},
that for each stable element $(\lambda, \mu) \in
[\mathfrak{g}^*] \times [\mathfrak{h}^*]$,
the corresponding fiber  $\mathcal{O}^{\lambda,\mu}$ of the map
\[\mathfrak{g}^* \times \mathfrak{h}^* \rightarrow [\mathfrak{g}^*] \times [\mathfrak{h}^*]\]
is an $H$-torsor.
Fixing a basepoint $\xi \in \mathcal{O}^{\lambda, \mu}$,
the orbit map gives an identification
\[ H \stackrel{\sim}{\longrightarrow} \mathcal{O}^{\lambda,
  \mu}.\]
Fix a Haar form $\beta_H$ on $\mathfrak{h}^*$.
  \index{volume form, orbital}
We define the \emph{orbital volume form} $\alpha$
to be the volume form on $\mathcal{O}^{\lambda, \mu}$
transferred, via the orbit map,
from the volume form on $\mathfrak{h}$
dual to $\beta_H$.

\begin{theorem} \label{volume form theorem}
  Fix  Haar volume forms $\beta_G, \beta_H$ on $\mathfrak{g}^*,
  \mathfrak{h}^*$.
  Equip $[\mathfrak{g}^*], [\mathfrak{h}^*]$
  with the corresponding normalized affine volume forms $\gamma_G, \gamma_H$.
  Let $(\lambda,\mu)$ be stable.
  Equip $\mathcal{O}^{\lambda}_{\reg},
  \mathcal{O}^{\mu}_{\reg}$ with their symplectic
  volume forms.
  Let $\alpha$ denote the orbital volume form
  on $\mathcal{O}^{\lambda,\mu}$.
  Then in each of the following three sequences,
  either $\alpha$ or $-\alpha$ is the fibral volume form.
  \begin{equation}\label{eqn:sequences-compatible-fibral}
    \xymatrix{
      & \mathcal{O}^{\lambda}_{\reg} \times \mathcal{O}^{-\mu}_{\reg}
      \ar[rr]^
      {
        \qquad (\xi,\eta) \mapsto \xi|_{\mathfrak{h}} + \eta
      }  &&  \mathfrak{h}^* \\
      \mathcal{O}^{\lambda,\mu}
      \ar[ru]^
      {
        \xi \mapsto
        (\xi,-\xi|_{\mathfrak{h}})
      }
      \ar[r] \ar[rd] & \mathcal{O}^{\lambda}_{\reg} \ar[r] & [\mathfrak{h}^*]   \\
      &  \mathfrak{g}^*_{\reg}  \ar[r] & [\mathfrak{g}^*] \times [\mathfrak{h}^*]
    }
  \end{equation}
\end{theorem}
 
Note
that for the top sequence, $0 \in \mathfrak{h}^*$
is a regular value
thanks to the ``trivial stabilizer''
consequence of stability;
for the middle sequence,
$\mu \in [\mathfrak{h}^*]$ is a regular value
because ``stable implies regular'';
for the bottom sequence, $(\lambda,\mu) \in [\mathfrak{g}^*]
\times [\mathfrak{h}^*]$
is a regular value
because of the ``principle bundle'' consequence
of stability (see \S\ref{sec:fibers-x-mapsto}).

The proof for the upper exact sequence is given in  \S \ref{sec:good fibrations},
and for the bottom sequence in \S \ref{OHAA}.
The claims for the two lower sequences are readily seen to be equivalent, 
so this will conclude the proof.

\subsubsection{Proof for the top sequence} 
\label{sec:good fibrations}
We examine the top sequence of theorem \ref{volume form
  theorem}.
We equip $\mathcal{O}^{\lambda}_{\reg} \times
\mathcal{O}^{-\mu}_{\reg}$
with its symplectic volume form $\Omega$,
and must show that $\Omega / \beta_H =  \pm \alpha$. 
The differential at $\tau \in \mathcal{O}_{\lambda,\mu}$
of the sequence
in question
fits into a commutative diagram
\begin{equation}
  \xymatrix{
    T_\tau(\mathcal{O}^{\lambda,\mu}) 
    \ar[r]
    &
    T_{\zeta}(    \mathcal{O}^{\lambda}_{\reg} \times
    \mathcal{O}^{-\mu}_{\reg}   )
    \ar[r]^{
      \quad \quad \quad \res
    }  &
    \mathfrak{h}^* \\
    \mathfrak{h}
    \ar[u]^{o_\tau}_{\cong}
    \ar[ru]_{o_\zeta \circ \delta}
    &  &  \\
 }
\end{equation}
in which
$\zeta := (\tau,-\tau|_{\mathfrak{h}})$,
$o_\tau$ and $o_\zeta$
denote differentials of orbit maps,
and
$\delta : \mathfrak{h} \rightarrow \mathfrak{g} \oplus
\mathfrak{h}$ and $\res : \mathfrak{g}^* \oplus \mathfrak{h}^*
\rightarrow \mathfrak{h}^*$
are given by $\delta(x) := (x, x)$
and $\res(\xi,\eta) := \xi|_{\mathfrak{h}} + \eta$.
We claim that this is a \emph{Lagrangian fibration},
i.e., that $o_\zeta \circ \delta(\mathfrak{h})$ is Lagrangian and
that the duality between $\mathfrak{h}$ and $\mathfrak{h}^*$
is induced by the symplectic structure on the middle term.
We may then conclude via the definition of symplectic volume
forms.  

To verify the Lagrangian fibration property, it is enough to
check
for each
$x \in \mathfrak{h}$ and
$\eta \in T_{\zeta}(\mathcal{O}^{\lambda}_{\reg}
\times \mathcal{O}^{-\mu}_{\reg} ) \subset \mathfrak{g}^* \times
\mathfrak{h}^*$
that
$\sigma(o_\zeta \circ \delta(x), \eta) = \langle x, \res(\eta) \rangle$,
where
$\sigma$ denotes the symplectic pairing
on
$T_\zeta(    \mathcal{O}^{\lambda}_{\reg}
\times \mathcal{O}^{-\mu}_{\reg} )$.
Indeed,
we may write $\eta = o_\zeta(y)$
for some $y \in \mathfrak{g} \oplus \mathfrak{h}$.
By the definition of the symplectic pairing
(see  \eqref{eqn:defn-alg-symp-pairing}),
we then have
\[
  \sigma(o_\zeta \circ \delta(x), \eta)
  = \langle \zeta, [\delta(x),y] \rangle
  = \langle \delta(x), o_\zeta(y) \rangle
  = \langle \delta(x), \eta  \rangle
  = \langle x, \res(\eta) \rangle.
\]
This concludes the proof.

\subsubsection{Proof for the bottom sequence}\label{sec:spher-vari-diff}
\label{OHAA}

It will be convenient to deduce the assertion from one 
that is more explicitly phrased in the context of spherical varieties.
The proof that follows borrows ideas that are well-known in that context
(resolving the cotangent bundle, degenerating to the boundary).

Thus, let $\mathbf{M}$ be any reductive group over $k$ with Lie
algebra $\mathfrak{m}$.  
Suppose that $\mathfrak{s}$ is a spherical Lie
subalgebra of $\mathfrak{m}$, that is to say, that there is a 
Borel
subgroup $\mathbf{B}$ of $\mathbf{M}$
whose Lie algebra $\mathfrak{b}$ is complementary to $\mathfrak{s}$;
in particular,    $\mathfrak{s}$ is
 of dimension $\dim(M/B)$.  Let
$\zeta \in \mathfrak{s}^{\perp} \subset \mathfrak{m}^*$.
We denote
by $T_{[\zeta]}$ the tangent space to $[\mathfrak{m}^*]$ at
$[\zeta]$,
and assume that the sequence
\begin{equation} \label{qqu} 
  \mathfrak{s}   \xrightarrow{x \mapsto \ad_x^* \zeta}
  \mathfrak{s}^{\perp}  \xrightarrow{\text{project}}  T_{[\zeta]}.
\end{equation}
is short exact;
here the final map
is
\[
  \mathfrak{s}^\perp
  \hookrightarrow 
  \mathfrak{m}^*
  \cong T_\zeta \mathfrak{m}^*
  \rightarrow T_{[\zeta]}.
\]
Note that the composition
of the sequence \eqref{qqu}
is always zero.

\begin{claim*}
  This sequence is compatible, up to signs,  with volume forms, where:
  \begin{itemize}
  \item We fix Haar forms on $\mathfrak{s}$ and $\mathfrak{m}$, and give $\mathfrak{s}^{\perp}$ the induced form via 
    $\det(\mathfrak{s}^{\perp}) \simeq \det(\mathfrak{m})^*  \otimes  \det(\mathfrak{s})$.
  \item
    The given Haar form on $\mathfrak{m}$
    defines a normalized affine form on $[\mathfrak{m}^*]$,
    hence a volume form on
    $T_{[\zeta]}$.
  \end{itemize}
\end{claim*}
Before proving the claim,
let us see how it implies  the desired result, namely, that the bottom sequence of the theorem is compatible with volume forms. 
We apply the claim
with $(\mathbf{M},\mathbf{S}) := (\mathbf{G} \times \mathbf{H}, \diag \mathbf{H})$.
Then $\mathfrak{s} = \diag \mathfrak{h} \hookrightarrow \mathfrak{g}
\times \mathfrak{h = \mathfrak{m}}$,
and we have an isomorphism
\[
\iota : \mathfrak{s}^\perp \stackrel{\sim}{\longrightarrow} \mathfrak{g}^*,
\]
\[
(\xi, -\xi|_{\mathfrak{h}}) \mapsto \xi,
\]
compatible with the adjoint actions of $S \cong H$.
Let $(\lambda,\mu)$ be stable.
Fix a basepoint
$\xi \in \mathcal{O}^{\lambda,\mu} \subset \mathfrak{g}^*$.
Then the sequence \eqref{qqu}
with
$\zeta := \iota^{-1}(\xi)$
is isomorphic
to the differential of bottom sequence of the theorem,
namely:
\begin{equation*}
  \begin{CD}         
    \mathfrak{s}  @> x \mapsto \ad_x^* (\xi,-\xi_H) >> \mathfrak{s}^\perp
    @>[ \cdot ] >>
    T_{(\lambda,-\mu)}([\mathfrak{g}^*] \times [\mathfrak{h}^*])
    \\
    @V \cong V x \mapsto \ad_x^* \xi  V
    @V \iota V \cong V
    @V (\tau_1, \tau_2) \mapsto (\tau_1,-\tau_2) V \cong V
    \\
    T_\xi(\mathcal{O}^{\lambda,\mu})
    @>>> T_\xi(\mathfrak{g}^*_{\stab})
    @>\xi \mapsto ([\xi], [\xi_H])>>
    T_{(\lambda,\mu)}([\mathfrak{g}^*] \times [\mathfrak{h}^*]).
  \end{CD} 
\end{equation*}
Here we identify
$T_\xi(\mathfrak{g}^*_{\stab}) = \mathfrak{g}^*$,
with
the subscripted $\stab$
denoting the subset of $\mathbf{H}$-stable elements.
\index{subscripted $\stab$}

The rightmost vertical arrow preserves volume forms, up to sign. Since the claim
shows that the upper sequence is compatible with volume forms, the lower sequence
is also compatible with the volume forms that are transferred
from the upper sequence.
The desired result follows.

\begin{proof}[Proof of the claim]

Consider the set of $\zeta \in \mathfrak{s}^{\perp}$ for which  there exists a Borel subalgebra $\mathfrak{b} \leqslant \mathfrak{m}$ such that:
    \begin{itemize}
   \item[(i)]    $\mathfrak{m}  = \mathfrak{b} \oplus \mathfrak{s}$; 
   \item[(ii)] $\zeta$ is the composition of the projection $\mathfrak{b} \oplus \mathfrak{s} \rightarrow \mathfrak{b}/[\mathfrak{b}, \mathfrak{b}]$
     with a regular character $\bar{\zeta}$ of $\mathfrak{b}/[\mathfrak{b},\mathfrak{b}]$. Here  ``regular'' means ``regular when identified
     with a character of the torus quotient of $\mathfrak{b}$.''
   \end{itemize}
   We claim that it is sufficient
   to prove the claim for such $\zeta$.
   Indeed, condition (ii) implies that $\zeta$ is regular
semisimple
   when considered as an element of $\mathfrak{m}^*$. 
    For such a $\zeta$, the set   of Borel subgroups $\mathfrak{b}$
 that satisfy the polarization condition $\zeta([\mathfrak{b}, \mathfrak{b}]) = 0$ is actually finite.
 It follows from this that:
 \begin{itemize}
 \item  For  $\zeta$ arising as in (i), (ii),
   the stabilizer of $\zeta$
     in $\mathfrak{m}$ is contained
     in $\mathfrak{b}$,
     hence the stabilizer
   in $\mathfrak{s}$ is trivial.  Therefore the first map of sequence
   \eqref{qqu}  is injective.
   
   \item For $\zeta$ arising as in (i), (ii), the second map of \eqref{qqu}
   is surjective (and so, counting dimensions, sequence
   \eqref{qqu} is exact).
   Indeed, writing $\mathfrak{t}$ for the torus quotient of $\mathfrak{b}$,
   there is a natural map $P: \mathfrak{t}^* \rightarrow [\mathfrak{m}^*]$,
   and the second map of \eqref{qqu} sends $\zeta$ to the image of $\bar{\zeta}$
   under that map.  But $P$ is submersive at the regular element $\bar{\zeta}$. 
     
  \item The dimension of the set of $\zeta$ that arise as in (i), (ii) equals $\dim(M/B)  + \dim(T) = \dim(\mathfrak{s}^{\perp})$.
 \end{itemize}
 
 Let $Y$ be the set of all $\zeta \in \mathfrak{s}^{\perp}$ for which \eqref{qqu} is short exact. Then $Y$ is a Zariski-open subset of $\mathfrak{s}^{\perp}$, 
 and it contains the  constructible set $X$ of $\zeta$ arising as in (i), (ii).
 Since the dimension of $X$ coincides with the dimension of $Y$,
 $X$ is Zariski-dense inside $Y$.  It follows
 that it is sufficient to prove the desired assertion
 for $\zeta \in X$.
    
 We will do this by degeneration. 
 Let $L$ be the variety of pairs
 \[ (\mathfrak{s}', \zeta'),\]
 where $\mathfrak{s}'$ is a subalgebra of $\mathfrak{m}$ of the same dimension
 as $\mathfrak{s}$, and such that \eqref{qqu},
 with $\mathfrak{s}'$ replacing $\mathfrak{s}$, 
   is short exact.  
   Comparing the
  volume forms on the various factors of \eqref{qqu} describes an
  $M$-invariant regular function $f : L \rightarrow \mathbb{G}_m$, where $f=1$
  at any point $(\mathfrak{s}', \zeta')$ where the sequence is compatible with volume forms.
 Fixing $\zeta,  \bar{\zeta}, \mathfrak{b}$ as in (i), (ii), we will show
that $f(\mathfrak{s}, \zeta)=1$. 

 Let $B$ be the associated Borel subgroup,
 and fix a maximal torus $T \subset B$ with Lie algebra $\mathfrak{t}$.
 We get a splitting $\mathfrak{m} = \overline{\mathfrak{n}} \oplus  \mathfrak{t} \oplus \mathfrak{n}$, where $\mathfrak{n} =  [\mathfrak{b}, \mathfrak{b}]$. 
   Choose a regular one-parameter subgroup $\gamma: \mathbb{G}_m \rightarrow T$
   with $\langle \gamma, \alpha \rangle < 0$ for every
   positive root $\alpha$,
   so that $\gamma(t)$ contracts  $\mathfrak{n}$  as $t \rightarrow \infty$. 
   
Let $\zeta_{\mathfrak{t}}$ denote the restriction of $\zeta$ to $\mathfrak{t}$;
we regard $\zeta_{\mathfrak{t}}$ as an element of $\mathfrak{m}^*$ by extending
trivially on $\mathfrak{n} \oplus \overline{\mathfrak{n}}$. 
 Then it is easy to see that
$(\bar{\mathfrak{n}}, \zeta_{\mathfrak{t}}) \in L$; moreover, 
\begin{equation} \label{limit} \lim_{t \rightarrow \infty} \Ad \gamma(t) \cdot (\mathfrak{s},\zeta) = (\overline{\mathfrak{n}}, \zeta_{\mathfrak{t}}).\end{equation}

Indeed,
by the assumption $\mathfrak{s} \cap \mathfrak{b} = \{0\}$
and a dimension computation,
we see that
$\mathfrak{s}$ projects onto $\overline{\mathfrak{n}}$
with respect to the splitting
$\mathfrak{m} = \overline{\mathfrak{n}} \oplus \mathfrak{b}$.
It follows readily from this that 
$\lim_{t} \Ad \gamma(t) \mathfrak{s} = \overline{\mathfrak{n}}$.
Furthermore, the character $\Ad(\gamma(t))  \zeta$ is trivial on
the subspace
 $ \Ad(\gamma(t))  \mathfrak{s} \oplus \mathfrak{n}$,
 which converges in turn
 to $\overline{\mathfrak{n}} \oplus \mathfrak{n}$. This implies that  
 $\Ad(\gamma(t))  \zeta$  converges in $\mathfrak{m}^*$
 to the character $\zeta_{\mathfrak{t}}$, extended trivially on
 $\overline{\mathfrak{n}} \oplus \mathfrak{n}$,
 and
 so concludes the proof of \eqref{limit}.

  Since $f$ is $T$-invariant and regular, it follows that
  $f(\mathfrak{s},\zeta) = f(\overline{\mathfrak{n}},\zeta_{\mathfrak{t}})$.
  It remains to show that $ f(\overline{\mathfrak{n}},\zeta_{\mathfrak{t}})=1$. 
This is a routine computation with \eqref{affine form}.

\iftoggle{easyproof}
{
\warning{These are informal proof notes for the benefit of the author,
and may not be verified.}
Since $\zeta_{\mathfrak{t}}$ is regular,
  the map $\mathfrak{t}^* \rightarrow T_{[\zeta_{\mathfrak{t}}]}$
  obtained by differentiating the projection $\mathfrak{t}^*
  \rightarrow [\mathfrak{m}^*]$
  at $\theta$ is an isomorphism.
  The sequence \eqref{qqu}
  thus identifies with
  \begin{equation} \label{explicit}
    \overline{\mathfrak{n}} \xrightarrow{x \mapsto \ad^*(x) \theta} \mathfrak{t}^* \oplus
    \mathfrak{n}^*
    \xrightarrow{\text{project}} \mathfrak{t}^*.
  \end{equation}
  Fix a Chevalley basis $H_{\alpha}, X_{\pm \alpha}$ for $\mathfrak{m}$,
  with
  $\alpha(H_{\alpha}) = 2$ and $[X_{\alpha}, X_{-\alpha}] =  H_{\alpha}$,
  and a dual basis
  $H_{\alpha}', X_{\pm \alpha}'$ for $\mathfrak{m}^*$,
  The first  map   sends
  $X_{-\alpha}$ to $\pm \theta(H_{\alpha}) X_{\alpha}'$.
  The volume forms may be chosen up to sign as follows: on $\overline{\mathfrak{n}}$ we wedge together $X_{-\alpha}'$;
  on $\mathfrak{t}^* \oplus \mathfrak{n}^*$ we wedge  $H_{\alpha}, X_{\alpha}$; 
  on the last factor $\mathfrak{t}^*$ we multiply
  the wedge of all $H_{\alpha}$
  by the factor $\prod_{\alpha > 0} \theta(H_{\alpha})$.   Observe that the first map
  carries $\overline{\mathfrak{n}}$, with this volume form, to $\mathfrak{n}^*$ with volume form
  $ \left( \prod_{\alpha} \theta(H_{\alpha}) \right)^{-1}  \wedge \bigwedge_{\alpha} X_{\alpha}$;
  we see that this sequence is compatible with volume forms. 
}
 \end{proof}

\section{Measures and
  integrals}\label{sec:gross-prasad-pairs-over-R-continuity-etc}
We now apply
the preceding considerations to define and compare measures
on spaces associated
to a GGP pair over an
archimedean
local field.

\subsection{Real varieties}\label{sec:appl-vari-over}
We denote by $X = \mathbf{X}(\mathbb{R})$
the set of real points
of a smooth real algebraic variety $\mathbf{X}$.
For each Haar measure $\lambda$ on $\mathbb{R}$
there is an assignment (see \cite[\S 2.2]{MR670072})
\[
\left\{
  \begin{aligned}
  &\text{$\mathbb{R}$-rational top degree differential forms} \\
  &\quad\quad\quad\quad\quad\quad \text{$\omega$ on $\mathbf{X}$}
\end{aligned}
\right\}
  \rightarrow 
  \left\{
    \begin{aligned}
      &\text{measures}\\
      &\text{$|\omega|$ on $X$}
\end{aligned}
\right\},
\]
which is functorial
under pullback by 
{\'e}tale maps, compatible with products, satisfies
$|f \omega| = |f| \cdot |\omega|$, and
is normalized by
$|dx| = \lambda$ when $\mathbf{X}$ is the affine line $\mathbb{A}^1$.

We henceforth take for $\lambda$ the measure
$\frac{\mathrm{Lebesgue}}{\sqrt{2 \pi}}$.  On
the affine space
$\mathbb{A}^n$, we then have
\begin{equation} \label{measure example} |dx_1 \wedge \dots \wedge  dx_n| = \frac{\mbox{Lebesgue}}{(2 \pi)^{n/2}}.\end{equation}

We have normalized $\lambda$ to be Fourier self-dual for the
character $\psi(x) := e^{ix}$.
This normalization has the following consequence:
Let $V$ be a real vector
space.  Using $\psi$, we may identify the real dual $V^*$ with
the Pontryagin dual $V^{\wedge} := \Hom(V, \C^{(1)})$.  Then 
dual algebraic volume forms on $V$ and $V^*$ correspond to
(Fourier-)dual measures on $V$ and $V^{\wedge}$.

\subsection{Groups} \label{sec:groups-over-R}
Let $\mathbf{G}$ be a reductive group over
$\mathbb{R}$.
Recall (from \S\ref{sec:general-conventions})
that  we denote by
$\mathfrak{g}$ the real Lie algebra, by $\mathfrak{g}^*$ its
real dual,
and by
$\mathfrak{g}^{\wedge}$ the Pontryagin dual
$\Hom(\mathfrak{g}, \C^{(1)})$.
Sending $\xi \in i \mathfrak{g}^*$ to $x \mapsto
e^{x \xi} \in \mathbb{C}^{(1)}$ gives
an identification
$i \mathfrak{g}^* \simeq \mathfrak{g}^{\wedge}$.

 We
suppose given a Haar measure $d g$ on the
Lie group $G$.  There is then a compatible Haar measure $d x$ on
$\mathfrak{g}$, normalized as in \S\ref{sec:measures-et-al-G-g-g-star} by requiring
that $U \subseteq \mathfrak{g}$ and $\exp(U) \subseteq G$ have
similar volumes when $U$ is a small neighborhood of the origin.
We obtain also a Fourier-dual measure $d \xi$ on
$\mathfrak{g}^\wedge$.  We choose an $\mathbb{R}$-rational Haar
form $\beta$ on the real vector space $\mathfrak{g}^\wedge$ by
requiring that $|\beta| = d \xi$;
it normalizes a dual form $\beta^\vee$
on $\mathfrak{g}$.  Explicitly, if we choose
coordinates $x = (x_1,\dotsc,x_n)$ and
$\xi = (\xi_1,\dotsc,\xi_n)$ as in \S\ref{sec-1-3-1} so that
$x \xi = \sum x_j \xi_j$ and $d x = d x_1 \dotsb d x_n$, then
$d \xi = (2 \pi)^{-n} d \xi_1 \dotsb d \xi_n$,
$\beta = \pm (2 \pi)^{-n/2} d \xi_1 \wedge \dotsb \wedge d
\xi_n$
and
$\beta^\vee  = \pm (2 \pi)^{n/2} d x_1 \wedge \dotsb \wedge d
x_n$.

Recall that for $\lambda \in [\mathfrak{g}^\wedge]$,
we set\index{$\mathcal{O}^\lambda$}
\[
  \mathcal{O}^{\lambda} = \{\xi \in \mathfrak{g}^\wedge :
  [\xi] = \lambda \}.
\]
On the regular subset $\mathcal{O}^{\lambda}_{\reg}$
we have both an algebraic symplectic volume form
$\omega_{\alg} = \frac{1}{d!} \sigma^d$
as in \S\ref{sec:algebraic-sympl-volume-forms}
and a normalized symplectic measure $\omega
= \frac{1}{d!} (\frac{\sigma }{2 \pi })^d$
as in \S\ref{sec:canonical-symplectic-form}.
We verify readily that
the algebraic form is $\mathbb{R}$-rational,
and satisfies $|\omega_{\alg}| = \omega$.

By the recipe of \S\ref{sec:affine-measures},
$\beta$ induces a normalized affine volume form
$\gamma$ 
on $[\mathfrak{g}^\wedge]$.
(We use that $\mathfrak{g}^\wedge$ and $[\mathfrak{g}^\wedge]$
are real forms of $\mathfrak{g}_{\mathbb{C}}$ and $[\mathfrak{g}_\mathbb{C}]$.)
We verify readily that $\gamma$ is $\mathbb{R}$-rational,
hence induces a \emph{normalized affine measure}
$|\gamma|$
on $[\mathfrak{g}^\wedge]$.
By the construction and compatibilities noted previously,
we then have
\begin{equation}
  \int_{\mathfrak{g}^\wedge}
  a = \int_{\lambda \in [\mathfrak{g}^\wedge]}
  (\int_{\mathcal{O}_{\lambda}^\reg}
  a \, d \omega )
\end{equation}
for each $a \in C_c(\mathfrak{g}^\wedge)$.

\subsection{GGP pairs}
Let $(\mathbf{G},\mathbf{H})$ be a GGP pair
over an archimedean local field.
By restriction of
scalars,
we may regard $\mathbf{G}$ and $\mathbf{H}$ as reductive
groups over $\mathbb{R}$;
the discussion and notation of \S\ref{sec:groups-over-R}
thus applies.
\subsubsection{Stability consequences}\label{sec:stab-cons}
For $\lambda \in [\mathfrak{g}^{\wedge}]$ and $\mu \in
[\mathfrak{h}^{\wedge}]$,
we set
\index{$\mathcal{O}^{\lambda,\mu}$}
\[
  \mathcal{O}^{\lambda, \mu}
  := \{\xi \in \mathfrak{g}^\wedge : [\xi] = \lambda,
  [\xi|_{\mathfrak{h}}] = \mu\}.
\]
As before, a subscripted $\stab$ denotes
the subset of $\mathbf{H}$-stable elements.
\index{subscripted $\stab$}
\begin{theorem}\label{thm:stability-consequences-over-R}
  The map
  $\mathfrak{g}^\wedge_{\stab} \rightarrow \{\text{stable }
  (\lambda,\mu)\}$
  is a principal $H$-bundle
  over its image,
  with fibers $\mathcal{O}^{\lambda,\mu}$.
  In particular,
  if $(\lambda,\mu) \in [\mathfrak{g}^\wedge] \times
  [\mathfrak{h}^\wedge]$
  is stable,
  then either
  \begin{itemize}
  \item $\mathcal{O}^{\lambda,\mu} = \emptyset$, or
  \item $\mathcal{O}^{\lambda,\mu}$ is an $H$-torsor,
    i.e., a closed $H$-invariant subset of $\mathfrak{g}^\wedge$
    on which $H$ acts simply transitively;
    moreover, $\mathcal{O}^{\lambda,\mu}$ consists
    of $\mathbf{H}$-stable regular elements.
  \end{itemize}
\end{theorem}
\begin{proof}
  This follows readily
  from the corresponding
  properties (\S\ref{sec:fibers-x-mapsto})
  established over the algebraic closure
  $\mathbb{C}$.
\end{proof}
More generally,
for each regular coadjoint multiorbit
$\mathcal{O} \subseteq \mathfrak{g}^\wedge$
and $\mu \in [\mathfrak{h}^\wedge]$,
we set
\index{$\mathcal{O}(\mu)$}
\[
  \mathcal{O}(\mu) :=
  \{\xi \in \mathcal{O} : [\xi|_{\mathfrak{h}}] = \mu \}.
\]
For example,
$\mathcal{O}^{\lambda,\mu} = \mathcal{O}^{\lambda}(\mu)$.  Then
\[
  \mathcal{O}_{\stab}
  \rightarrow [\mathfrak{h}^\wedge]
\cap \image(\mathcal{O}_{\stab})
\]
is a principal $H$-bundle,
with fibers $\mathcal{O}(\mu)$.

Let $\pi$ and $\sigma$ be tempered irreducible
unitary representations of $G$ and $H$,
respectively.
We set
\index{$\mathcal{O}_{\pi,\sigma}$}
\[
\mathcal{O}_{\pi,\sigma} :=
\{\xi \in \mathcal{O}_\pi : \xi|_{\mathfrak{h}}
\in \mathcal{O}_\sigma \}.
\]
We note that
\[
\mathcal{O}_{\pi,\sigma}
\subseteq
\mathcal{O}_\pi(\lambda_\sigma)
\subseteq \mathcal{O}^{\lambda_\pi, \lambda_\sigma}.
\]
The pictures in \S\ref{sec:intro-op-calc}
and \S\ref{sec:intro-local-issues}
give some examples
to which these notations apply.
Theorem \ref{thm:stability-consequences-over-R}
implies that
if
$(\lambda_\pi,\lambda_\sigma)$
is stable
and
$\mathcal{O}_{\pi,\sigma}$
is nonempty,
then
$\mathcal{O}_{\pi,\sigma}$ is an $H$-torsor
consisting of $\mathbf{H}$-stable elements,
hence
\begin{equation}\label{eq:}
  \mathcal{O}_{\pi,\sigma} =
  \mathcal{O}_\pi(\lambda_\sigma)
  = \mathcal{O}^{\lambda_\pi, \lambda_\sigma}.
\end{equation}

\subsubsection{Integral transforms and identities}\label{sec:integr-transf-ident}
We assume given a Haar measure
on $H$.
As in \S\ref{sec:groups-over-R},
this choice defines a measure
$[\mathfrak{h}^{\wedge}]$.
By the discussion of \S\ref{sec:orbital-volume-forms}
and \S\ref{sec:appl-vari-over},
we obtain also -- for
stable $(\lambda,\mu) \in [\mathfrak{g}^\wedge] \times
[\mathfrak{h}^\wedge]$ -- a measure on
$\mathcal{O}^{\lambda,\mu}$,
which we will see below
induces a measure on $\mathfrak{g}^\wedge$.
More generally, we adopt the convention that
an integral
over $\mathcal{O}^{\lambda,\mu}$
is defined to be zero
unless $(\lambda,\mu)$ is stable;
in that case, the measure 
is given explicitly by the pushforward of the Haar measure from $H$,
i.e.,
\[
  \int_{\mathcal{O}^{\lambda,\mu}}
  a :=
  \int_{s \in H} a(s \cdot \xi_{\lambda,\mu})
\]
for any basepoint $\xi_{\lambda,\mu} \in \mathcal{O}^{\lambda,\mu}$.
We note (cf. \S\ref{sec:stab-terms-spectra})
that for given $\lambda$,
the pair $(\lambda,\mu)$ is stable
for $\mu$ outside a measure zero subset.

For a regular coadjoint multiorbit $\mathcal{O} \subseteq
[\mathfrak{g}^\wedge]$
and $\mu \in [\mathfrak{h}^\wedge]$,
the set $\mathcal{O}(\mu)$
is either empty or of the form $\mathcal{O}^{\lambda,\mu}$
with $\lambda = [ \mathcal{O}]$.
Thus integration over $\mathcal{O}(\mu)$ is defined.

\begin{theorem}\label{thm:integr-transf-ident}
  Integration defines a continuous map
  \[
    \{\text{stable } (\lambda,\mu) \in [\mathfrak{g}^\wedge] \times
    [\mathfrak{h}^\wedge] \}
    \times \mathcal{S}(\mathfrak{g}^\wedge)
    \rightarrow \mathbb{C}
  \]
  \[
    (\lambda,\mu,a)
    \mapsto \int_{\mathcal{O}^{\lambda,\mu}} a.
  \]
  We have
  \begin{equation}\label{eqn:consequence-of-vol-form-thm}
    \int_{\mathcal{O}^{\lambda}_{\reg}}
    a \, d \omega 
    =
    \int_{\mu \in [\mathfrak{h}^\wedge]}
    \int_{\mathcal{O}^{\lambda,\mu}} a.
  \end{equation}
  More generally,
  for any regular coadjoint multiorbit $\mathcal{O} \subseteq
  \mathfrak{g}^\wedge$,
  \begin{equation}\label{eqn:disintegration-along-H-for-individual-orbit}
    \int_{\mathcal{O}}
    a \, d \omega 
    =
    \int_{\mu \in [\mathfrak{h}^\wedge]}
    \int_{\mathcal{O}(\mu)} a.
  \end{equation}
\end{theorem}
\begin{proof}
  The convergence follows from the inequalities recorded below in
  \S\ref{sec:appl-loja},
  the continuity from theorem \ref{thm:stab-terms-spectra-2},
  and the integral formulas
  from 
  theorem \ref{volume form theorem}.
\end{proof}
This result will  be applied
in \S\ref{sec:main-result-inv-branch-arch}.

\subsection{A Lojasiewicz-type inequality}
\label{sec:appl-loja}
We pause to record a technical lemma
justifying the convergence and continuity of
the integral transforms defined above.
% This estimate is in the spirit of the Lojasiewicz inequalities

% % 

% Recall that a subset $S$ of a real vector space $V$ is
% \emph{semialgebraic} if it may be defined in terms of polynomial
% identities and inequality, and that then a function
% $f : S \rightarrow W$ mapping to a vector space $W$ is
% \emph{semialgebraic} if its graph is semialgebraic.
% We record a standard Lojasiewicz type inequality:
% \todo{REF}
% \begin{lemma}
%   Let $V, V', W$ be normed vector spaces over the reals, let
%   $S \subseteq V, S' \subseteq V'$ be semialgebraic
%   subsets, and let $f : S \times S' \rightarrow W$ be a
%   semialgebraic function.
%   Let $y_0 \in S'$.  Suppose that the
%   map $f(\cdot,y_0)$ is proper in the topological sense,
%   i.e.,
%   that $\{x : f(x,y_0) \in A\}$ is
%   compact whenever $A \subseteq W$ is compact,
%   or equivalently
%   that $|f(x,y_0)| \rightarrow \infty$
%   as $|x| \rightarrow \infty$.
%   % 
%   % 
%   % 
%   Then there are
%   positive reals $c,\eps,r$ and a neighborhood $U$ of $y_0$ in
%   $S'$ so that $|f(x,y)| \geq c |x|^\eps$ for all
%   $x \in S$ and $y \in U$
%   with $|x| \geq r$.
% \end{lemma}

We fix a norm $|.|$ on $\mathfrak{g}^\wedge$
and a faithful finite-dimensional representation of $\mathbf{H}$.
We may use the latter to define an algebraic norm $|.|$ on
$H$: denoting the faithful representation by $R$,
we set $|h| := \max(\|R(h)\|, \|R(h^{-1})\|)$, where $\|.\|$
denotes
the operator norm on the space of $R$.
\begin{lemma*}
  Let $\xi \in \mathfrak{g}^\wedge$ be $\mathbf{H}$-stable.
  There are then positive reals
  $c_1, c_2$
  and a (topological) neighborhood $U$
  of $\xi$ so that
  for all $\eta \in U$ and $s \in H$,
  \begin{equation}\label{eq:stable-implies-growth}
    |s \cdot \eta| > c_1 |s|^{c_2}.
  \end{equation}
\end{lemma*}

\begin{proof}
  The required estimate is in the spirit of the Lojasiewicz
  inequality, but for lack of a convenient reference
  it is simplest for us to argue directly:
  
  Let $Z$ denote the set of $\mathbf{H}$-stable elements of
  $\mathfrak{g}^\wedge$.  By Theorem
  \ref{thm:stab-terms-spectra-1},
  the set $Z$ may be described explicitly
  as the nonvanishing locus of a certain resultant $\rho :
  \mathfrak{g}^\wedge \rightarrow \mathbb{R}$.
  % in the algebra $\mathbb{R}[\mathfrak{g}^\wedge]$ of polynomial
  % functions $\mathfrak{g}^\wedge \rightarrow \mathbb{R}$.
  We
  may thus regard $Z$ as the set of real points of the real
  affine variety
  $\mathbf{Z} := \Spec \mathbb{R}[\mathfrak{g}^\wedge, 1/\rho]$.

  The group $\mathbf{H}$ acts on $\mathbf{Z}$.
  The orbit map $(h,z) \mapsto (h z,z)$, regarded
  as an algebraic map
  \[\mathbf{H} \times \mathbf{Z} \rightarrow \mathbf{Z} \times
  \mathbf{Z},\] is a \emph{closed immersion}.  Indeed, this
  property may be checked on complex points, and Theorem
  \ref{thm:stab-terms-spectra-2} says that
  $\mathbf{Z}(\mathbb{C}) \rightarrow \mathbf{Z}(\mathbb{C})
  \git \mathbf{H}(\mathbb{C})$ is a principal
  $\mathbf{H}(\mathbb{C})$-bundle.  In particular, the induced
  map on coordinate rings is surjective.
  
  Choose generators $f_1, \dots, f_k$ for the
  $\mathbb{R}$-algebra of
  regular functions on
  $\mathbf{Z}$.
  % Each
  % generator $f_j$ defines a function
  % $f_j : Z \rightarrow \mathbb{R}$.
  The closed immersion property noted above implies that each
  regular function $P(h)$ on $\mathbf{H}$
  may be written as a polynomial in
  $f_i(z)$ and $f_j(h \cdot z)$.
  Writing $\|z\| = \max_{i} |f_i(z)|$,
  we may thus find positive constants $C$ and $K$
  so that
  for all $h \in H$ and $z \in Z$, we have
  \[P(h) \leq   C (\|z\| + \|h\cdot z\|)^K,\]
  or equivalently,
  \[\|h \cdot z\| \geq C^{-1/K} P(h)^{1/K} - \|z\|.\]

  The norm $|s|$ defined for $s \in H$ by a
  finite-dimensional faithful representation is
  comparable to $\max_{i \in I} |P_i(s)|$ for some finite
  collection $(P_i(h))_{i \in I}$
  of regular functions on $\mathbf{Z}$.
  Restricting $\eta$ to a compact subset of $Z$,
  we obtain
  \begin{equation}\label{eqn:s-cdot-eta-aint-too-small}
    |s \cdot \eta| \geq c_1 |s|^{c_2} - c_3
  \end{equation}
  for suitable
  constants $c_1, c_2, c_3 > 0$ (depending only upon the given
  compact  set).
  When $|s|$ is large enough,
  the required estimate follows from
  \eqref{eqn:s-cdot-eta-aint-too-small}.
  In the remaining range,
  $s$ and $\eta$ both lie in fixed compact sets
  and $s \cdot \eta \neq 0$,
  giving the adequate estimate $|s \cdot \eta| \gg 1 \gg |s|$.

\end{proof}

% \begin{proof}

%   Set
%   $\lambda := [\xi], \mu := [\xi|_{\mathfrak{h}}]$.  By the main
%   result of \S\ref{sec:fibers-x-mapsto}, we may find an
%   {\'e}tale neighborhood
%   $(j : S' \rightarrow [\mathfrak{g}^\wedge] \times
%   [\mathfrak{h}^\wedge], \widetilde{(\lambda,\mu)} \in S')$
%   of $(\lambda,\mu)$ and an $H$-equivariant section
%   $f : H \times S' \rightarrow \mathfrak{g}^\wedge$ 
%   for which $f(1,\widetilde{(\lambda,\mu)}) = \xi$.
%   The orbit maps $f(\cdot,y_0)$ are
%   closed embeddings, hence topologically proper,
%   so we may conclude by applying the previous lemma with $S := H$.
% \end{proof}

% 
% 
% 
% 
% 
% 
% 
% 
% 
% 
% 
% 
% 
% 
% 
% 
% 
% 
% 
% 
% 
% 
% 
% 
% 
% 
% 
% 
% 
% 
% 
% 
% 
% 
% 
% 
% 
% 

\subsection{The scaling limit of Plancherel
  measure}\label{sec:scal-limit-planch}
It is instructive to note that the normalized affine volume
measure
on $[\mathfrak{h}^\wedge]$ is
closely related to the
Plancherel measure $\mu$ on $\hat{H}$
(cf. \S\ref{sec:plancherel-formula}).
We do not use this comparison directly,
and so will be brief
and sketchy.  For an open set
$U \subset [\mathfrak{h}^\wedge]$, let
$\widetilde{U} := \{\sigma \in \hat{H}_{\temp} : \lambda_\sigma \in U\}$
denote the set of isomorphism classes of tempered irreducible
unitary representations having infinitesimal character in $U$.
We assume for simplicity
that we are working
over a complex group,
so that for each $\sigma \in \hat{H}_{\temp}$,
we have $\mathcal{O}_\sigma = \mathcal{O}^{\mu}_{\reg}$
with $\mu := \lambda_\sigma$.
\begin{lemma*}
  Fix a nonempty bounded open subset $U$ of $[\mathfrak{h}^\wedge]$
  whose boundary has measure zero (with respect
  to any measure in the class of smooth measures).
  Then
  \[
    \lim_{\h \rightarrow 0}
    \frac{
      \mbox{Plancherel measure of $\widetilde{\h^{-1} U}$}
    }
    {
      \mbox{normalized affine measure of $\h^{-1} U$}
    } = 1.
  \]
\end{lemma*}
While one can prove this simply by examining the explicit
form of Plancherel measure, it would then be tedious to check
carefully the normalization of constants.  We
sketch a different argument,
presumably well-known,
which makes the normalization
clear.
Define $\Opp_{\h} : C_c^\infty(\mathfrak{h}^\wedge)
\rightarrow \End(\sigma)$ as
usual.
Fix $a \in C_c^\infty(\mathfrak{h}^\wedge)$,
with dilates $a_{\h}(\xi) := a(\h \xi)$ as usual.
By the definition
of normalized
affine measure,
$\int_{\mu \in [\mathfrak{h}^\wedge]}
(\int_{\mathcal{O}^\mu} a_{\h})
=
\int_{\mathfrak{h}^\wedge}
a_{\h}$.
Recall that $\Opp_{\h}(a) = \sigma(f)$
for some $f \in C_c^\infty(H)$
supported near $1$
and given by
$f(\exp(x)) := a_h^\vee(x) \chi(x) j(x)^{-1}$;
since $\chi(0) = j(0) = 1$,
we have in particular
$\int_{\mathfrak{h}^\wedge} a_{\h} = f(1)$.
By the
Plancherel formula (\S\ref{sec:plancherel-formula}),
$f(1)
= \int_{\sigma \in \hat{H}_{\temp}} \trace(\sigma(f)) \, d \mu(\sigma)$.
By the Kirillov formula (see \S\ref{sec:results-appl-kir}),
$\trace(\sigma(f))
=
\trace(\Opp_{\h}(a)) \sim \int_{\mathcal{O}_{\sigma}}
a_{\h} \, d \omega_{\mathcal{O}_\sigma}$.
(Here and henceforth the precise meaning of $\sim$
may
be inferred
from the precise statements of \S\ref{sec:results-appl-kir}.)
It follows that
\begin{equation}\label{eqn:planch-asymp-big-eqn}
  \int_{\mu \in [\mathfrak{h}^\wedge]}
  \Bigl(\int_{\mathcal{O}^\mu} a_{\h} \Bigr)
  \sim
  \int_{\sigma \in \hat{H}_{\temp}}
  \Bigl(\int_{\mathcal{O}_\sigma} a_{\h} \, d \omega_{\mathcal{O}_\sigma} \Bigr) \, d \mu(\sigma).
\end{equation}
We now choose $a$
so that
$\int_{\mathcal{O}^\mu }
a$
approximates the characteristic function $1_U(\lambda)$.
The left hand side of \eqref{eqn:planch-asymp-big-eqn}
then approximates the affine measure of $\h^{-1} U$,
while the right hand side approximates
the Plancherel measure of the set $\widetilde{U}$ of representations $\sigma$ for which
$\lambda_{\sigma} \in h^{-1} U$.

\section{Relative characters: disintegration}
\label{sec:spher-char-disint}
\label{sec:local-disintegration}
Let $(\mathbf{G},\mathbf{H})$ be a GGP pair over a local field
$F$.
We equip $G$ and $H$ with some Haar measures.  To simplify
notation, we do not display these Haar measures in our integration
notation.  Thus $\int_{s \in H} f(s)$ denotes the integral of
$f \in L^1(H)$.

By the proofs of
\cite[Prop 1.1]{MR2585578}
and \cite[\S2]{MR3159075},
the corresponding Harish--Chandra functions
(cf. \S\ref{sec:matrix-coeff-bounds-for-tempered-reps})
satisfy
\begin{equation}\label{eqn:convergence-assumption-integral-Xi}
  \int_{H} \Xi_G|_H \cdot \Xi_H
  < \infty.
\end{equation}

Let $\pi$ and $\sigma$
be tempered irreducible unitary representations
of $G$ and $H$, respectively.
More precisely,
we denote in this subsection
by $\pi$ and $\sigma$ the spaces of smooth
vectors in the underlying Hilbert spaces.
Choose an orthonormal basis
$\mathcal{B}(\sigma)$
consisting of
isotypic vectors
for the action of some fixed maximal
compact subgroup of $H$. Similarly choose an orthonormal basis $\mathcal{B}(\pi)$ for $\pi$. 
\begin{lemma*}~\label{prop:sph-chr-dis}
  \begin{enumerate}[(i)]
  \item
    For $v_1, v_2 \in \pi$,
    the formula
    \index{$\mathcal{H}_\sigma$}
    \begin{equation} \label{II  Hermitian form}
      \mathcal{H}_{\sigma}(v_1 \otimes \overline{v_2}) :=
      \sum_{u \in \mathcal{B}(\sigma)}
      \int_{s \in H}
      \langle s v_1, v_2 \rangle \langle u, s u \rangle
    \end{equation}
    converges and defines an $H$-invariant hermitian
    form
    \[
    \mathcal{H}_{\sigma} : \pi \otimes \overline{\pi }
    \rightarrow \mathbb{C}.
    \]
  \item
    For $T \in \pi \otimes \overline{\pi }$, one has
    \begin{equation}\label{eqn:trace-T-equals-integral-H-T}
      \trace(T) = \int_{\sigma \in \hat{H}_{\temp}}
      \mathcal{H}_\sigma(T),
    \end{equation}
    where
    $\trace : \pi \otimes \overline{\pi } \rightarrow
    \mathbb{C}$
    is the linear map
    for which
    % denotes the linear extension of
    $\tr(v_1 \otimes \overline{v_2}) := \langle v_1, v_2
    \rangle$
    and the integral is taken with respect to the Plancherel
    measure on $\hat{H}_{\temp}$
    dual to the chosen Haar measure on $H$.
    In particular,
    \begin{equation}\label{eqn:H-disintegration-explicated-v1-v2}
      \langle v_1, v_2 \rangle
      =
      \int_{\sigma \in \hat{H}_{\temp}}
      \sum_{u \in \mathcal{B}(\sigma)}
      \int_{s \in H}
      \langle s v_1, v_2 \rangle \langle u, s u \rangle.
    \end{equation}
  \item \label{item:arch-case-rel-char-disint}
    Suppose that $F$ is archimedean,
    so that the definitions of Part \ref{part:micr-analys-lie}
    apply,
    and let $N \geq 0$ be sufficiently large
    in terms of $G$.
    Then $\mathcal{H}_\sigma$ extends
    to a map \[\mathcal{H}_\sigma : \Psi^{-N}(\pi)
    \rightarrow \mathbb{C}\]
    which is continuous, uniformly
    in $\pi$ and $\sigma$,
    and given by
    \begin{equation}\label{eqn:formula-H-sigma-arch-case}
      \mathcal{H}_{\sigma}(T)
      =
      \sum_{u \in \mathcal{B}(\sigma)}
      \int_{s \in H}
      \trace(s T)
      \langle u, s u  \rangle
      =
      \sum_{\substack{
          v \in \mathcal{B}(\pi) \\ u \in \mathcal{B}(\sigma) 
        }
      }
      \int_{s \in H}
      \langle s T v, v \rangle
      \langle u, s u  \rangle.
    \end{equation}
  \end{enumerate}
\end{lemma*}
\begin{proof}
  The details of the proof are technical and not particularly interesting, so we
  have relegated most of them to
  \S\ref{sec:prel-repr-reduct}.  Let $v_1, v_2 \in \pi$, and
  define $f : H \rightarrow \mathbb{C}$ by
  $f(s) := \langle s v_1, v_2 \rangle$.
  Then
  \begin{equation}\label{eq:defn-of-f-for-relative-char-disint}
    f(1) = \langle v_1, v_2 \rangle,
    \quad 
    \mathcal{H}_\sigma(v_1 \otimes \overline{v_2})
    = \trace(\sigma(f)).
  \end{equation}
  Observe now that assertions (i) and (ii)
  are formal consequences
  of the Plancherel formula.
  That formula does not directly apply, because $f$ is typically
  not compactly-supported,
  but an approximation argument gives
  what is needed.
  We postpone the details
  to \S\ref{sec:disintegration-details}.
\end{proof}

\begin{remark*}
   It is expected, and known for $F$ non-archimedean (see
  \cite{SV}), that $\mathcal{H}_\sigma$ satisfies the positivity
  condition
  \begin{equation}\label{eq:H-sig-pos}
    \text{$\mathcal{H}_\sigma(v \otimes \overline{v}) \geq 0$
      for all $v \in \pi$.}
  \end{equation}
  Assume this.  Since $\dim \Hom_H(\pi,\sigma) \leq 1$, we may
  then write
  \begin{equation} \label{lsigmadef} \int_{s \in H} \langle s v,
    v \rangle \langle u, u s \rangle = \left| \langle
      \ell_{\sigma}(v), u \rangle \right|^2\end{equation} for
  some $H$-invariant functional
  $\ell_{\sigma} : \pi \rightarrow \sigma$, determined up to
  phase.  One has
  $\mathcal{H}_{\sigma}(v) = \| \ell_{\sigma}(v) \|^2$.  The
  continuity of $\ell_{\sigma}$ follows from that of
  $\mathcal{H}_{\sigma}$.
\end{remark*}

\section{Relative characters: asymptotics in the stable
  case\label{sec:sph-char-2}}
We assume
that $(\mathbf{G},\mathbf{H})$ is a GGP pair
over an \emph{archimedean} local field,
and
retain
the notation and conventions
of
\S\ref{sec:gross-prasad-pairs-over-R-continuity-etc}
and \S\ref{sec:spher-char-disint}.
In particular,
by restriction of scalars,
we
may regard $\mathbf{G}$ and $\mathbf{H}$
as real reductive groups,
and the corresponding point sets $G$ and $H$ as real Lie groups
with Lie algebras $\mathfrak{g}$ and $\mathfrak{h}$.

\subsection{Motivation}\label{sec:asympt-motiv}
Let $\pi \in \hat{G}_{\temp}$ and
$\sigma \in \hat{H}_{\temp}$ be tempered irreducible unitary
representations.  We allow $\pi$ and $\sigma$ to vary arbitrarily with a
scale parameter $\h \rightarrow 0$, which we normalize so that
the rescaled infinitesimal characters
$\h \lambda_\pi$ and $\h \lambda_\sigma$ remain bounded.

Fix $a \in S^{-\infty}(\mathfrak{g}^\wedge)$,
and set
\[
  \Opp_{\h}(a) :=
  \Opp_{\h}(a:\pi) \in \Psi^{-\infty}(\pi).
\]
We wish to understand the $\h \rightarrow 0$
asymptotics of
the quantity
$\mathcal{H}_{\sigma}(\Opp_{\h}(a))$
defined via \S\ref{sec:spher-char-disint}.

To see what to expect,
observe the following:
\begin{align*}
  \int_{\sigma \in \hat{H}_{\temp}}
  \mathcal{H}_\sigma(\Opp_{\h}(a))
  &\stackrel{\eqref{eqn:trace-T-equals-integral-H-T}}{=}
    \trace(\Opp_{\h}(a)) \\
  &\stackrel{\eqref{eqn:kirillov-approximate}}{\approx}
    \int_{\mathcal{O}_{\pi}} a_{\h} \, d \omega
  \\
  &\stackrel{\eqref{eqn:disintegration-along-H-for-individual-orbit}}{=}
    \int_{\mu \in [\mathfrak{h}^\wedge]}
    \int_{\mathcal{O}_\pi(\mu)}
    a_{\h}.
\end{align*}
Consideration of the action
of the universal enveloping algebra of $H$
suggests
that the above sequence
localizes
to individual $\sigma$,
i.e., that
\begin{equation}\label{eqn:spher-char-expected-after-heuristic}
  \mathcal{H}_\sigma(\Opp_{\h}(a))
  \approx 
  \int_{\mathcal{O}_{\pi,\sigma}}
  a_{\h},
\end{equation}
at least if
$(\h \lambda_\pi,\h \lambda_\sigma)$ stays away from
the boundary of the stable locus.

Strictly speaking, the most one can deduce immediately from such
reasoning is an equality like
\eqref{eqn:spher-char-expected-after-heuristic}, but summed
over all $\sigma$ of fixed infinitesimal character.
However, there is another way to motivate
\eqref{eqn:spher-char-expected-after-heuristic},
again ignoring issues of rigour.
If we were to pretend that the
exponential map were an isomorphism
with trivial Jacobian
and to ignore the cutoff in $\Opp_{\h}(a)$,
we would obtain
\begin{align} \label{second heuristic} 
  \mathcal{H}_{\sigma}(\Opp_{\h}(a))
  &=
    \int_{h \in H}
    \trace(\pi(h) \Opp_{\h}(a))
    \overline{\chi_\sigma(h)}
  \\
  &\stackrel{\cdot}{=}
    \int_{x \in \mathfrak{g}}
    a_{\h}^\vee(x)
    \int_{y \in \mathfrak{h}}
    \chi_\pi(e^{y+x})
    \chi_\sigma(e^{-y}),
\end{align}
which leads formally
to \eqref{eqn:spher-char-expected-after-heuristic}
via
the Kirillov  formula.

The expectation
\eqref{eqn:spher-char-expected-after-heuristic}
belongs to
the general philosophy of the orbit method,
whereby restricting a
representation of $G$ to the subgroup $H$ corresponds to
disintegrating its coadjoint orbit along the projection
$\mathfrak{g}^\wedge \twoheadrightarrow \mathfrak{h}^\wedge$
(cf. \S\ref{sec:intro-branch-stab}).
The main result of \S\ref{sec:sph-char-2},
stated below in \S\ref{sec:sph-chr-statement-result},
confirms this
expectation in a sharper form.
Our proof will be along the lines of the second
argument discussed above; we will chop the $H$-integral
up into ranges depending on how far the group element $h$ is from the identity.

\subsection{A priori
  estimates}\label{sec:priori-estimates-rel-char}
For orientation
and later applications,
we record some crude bounds.
We abbreviate
$S^m_{\delta} := S^m_{\delta}(\mathfrak{g}^\wedge)$,
and write $a$ for an element of
$S^{-\infty}_{\delta}$
for some fixed $0 \leq \delta < 1/2$.

\begin{lemma}
  We have the very weak bound
  \begin{equation}\label{eqn:a-priori-for-rel-char}
  \mathcal{H}_{\sigma}(\Opp_{\h}(a))
  % \in \mathcal{H}_\sigma(\Opp_{\h}(S^{-N}_{\delta}))
  % \subseteq \h^{-N} \mathcal{H}_\sigma(\Psi^{-N}(\pi))
  \ll \h^{-N}.
\end{equation}
\end{lemma}
\begin{proof}
  Recall from \S\ref{sec:spher-char-disint}
  that for $N$ chosen sufficiently large (relative to
  $G$), the map
  $\mathcal{H}_\sigma : \Psi^{-N}(\pi) \rightarrow \mathbb{C}$
  is continuous, uniformly in $\pi$ and $\sigma$.
  By enlarging  $N$ suitably
  and appealing to the lemma of \S\ref{sec:h-dependence-Psi-m},
  we deduce that $\mathcal{H}_\sigma$
  induces a continuous map
  $\mathcal{H}_\sigma : \Psi^{-N}_\delta(\pi) \rightarrow
  \mathbb{C}$,
  uniformly in the same sense.
  Since
  $a \in S_\delta^{-N}$,
  we know by
  theorem
  \ref{thm:rescaled-operator-memb}
  that $\Opp_{\h}(a)
  \in \h^{-N} \Psi_\delta^{-N}$.
  The required estimate follows.
%   Choose a
%   seminorm $\nu$ on $\Psi^{-N}(\pi)$, uniformly continuous in
%   $\pi$ and $\sigma$, so that
%   $|\mathcal{H}_\sigma(T)| \leq \nu(T)$ for all
%   $T \in \Psi^{-N}(\pi)$.  By the lemma of
%   \S\ref{sec:h-dependence-Psi-m}, we may fix $M \geq 0$ so that
%   $\nu(T) \ll \h^{-M}$ for all $T \in \Psi^{-N}_\delta(\pi)$.
%   Since
%   $\Opp_{\h}(S_\delta^{-N}) \subseteq \h^{-N}
%   \Psi_\delta^{-N}(\pi)$ (theorem
%   \ref{thm:rescaled-operator-memb}), it follows that
%   $\mathcal{H}_\sigma(\Opp_{\h}(S_\delta^{-N}) ) \ll \h^{-N-M}$.
%   We conclude by replacing $N$ by $N+M$.
% % From this
% % and the estimate
% % \eqref{eqn:rescaled-operator-norm-bound},
% % we conclude.
\end{proof}
The technique employed in this last proof for passing from
$\Psi^{-N}$ to $\Psi^{-N}_\delta$ (after possibly enlarging $N$)
will be applied in the remainder of \S\ref{sec:sph-char-2}
without explicit mention.

Recall from \S\ref{sec:infin-char} that we identify
$[\mathfrak{g}^\wedge]$ and $[\mathfrak{h}^\wedge]$ with
Euclidean spaces, equipped with distance functions.
\begin{lemma}
  Assume that $a$ is supported in some fixed
  subset $U \subseteq \mathfrak{g}^\wedge$ (independent of $\h$, but otherwise arbitrary).
  Then the very strong bound
  \begin{equation}\label{eqn:inf-chars-concentrated-near-image-of-U-otherwise}
    \mathcal{H}_{\sigma}(\Opp_{h}(a)) \ll \h^N \langle \h
    \lambda_\pi  \rangle^{-N}
    \langle \h \lambda_\sigma  \rangle^{-N}
  \end{equation}
  holds unless
  \begin{equation}\label{eqn:inf-chars-concentrated-near-image-of-U}
    \text{$(\h \lambda_\pi, \h \lambda_\sigma)$
      is within distance $o_{\h \rightarrow 0}(1)$
      of the image of $U$ in $[\mathfrak{g}^\wedge] \times [\mathfrak{h}^\wedge]$}.
  \end{equation}
\end{lemma}
\begin{proof}
  We observe first that $\mathcal{H}_{\sigma}$
  factors as an $(H \times H)$-equivariant sequence
  \[
    \Psi^{-\infty}(\pi) \rightarrow \Psi^{-\infty}(\sigma) \xrightarrow{\trace} \mathbb{C}
  \]
  where the first arrow sends $T$ to the operator on $\sigma$
  given by $\int_{s \in H} \trace(s T) \sigma(s^{-1})$.  This
  sequence is continuous, uniformly in $\pi$ and $\sigma$, by
  the same argument as in the proof of part
  \eqref{item:arch-case-rel-char-disint} of the lemma of
  \S\ref{sec:local-disintegration}.  These observations allow us
  to deduce the required implication (in sharper form) from the
  results of \S\ref{sec:some-decay-2}.
\end{proof}

\subsection{Main
  result}\label{sec:sph-chr-statement-result}
For convenience, we recall  some notation and conventions from
\S\ref{sec:gross-prasad-pairs-over-R-continuity-etc}:
\begin{itemize}
\item We write
$\mathcal{O}_{\pi,\sigma}$
for the intersection of $\mathcal{O}_\pi$
with the preimage of $\mathcal{O}_\sigma$
under the projection $\mathfrak{g}^\wedge \rightarrow
\mathfrak{h}^\wedge$.
\item Integration over the set $\mathcal{O}_{\pi,\sigma}$,
  or over its rescaling
  $\h \mathcal{O}_{\pi,\sigma}$,
  is defined to be zero unless
  that set is nonempty and $\mathbf{H}$-stable;
  in that case,    it is
  an
  $H$-torsor,
  i.e., a closed subset of $\mathfrak{g}^\wedge$
  on which $H$ acts simply transitively,
  and we equip it with the transport of the Haar measure
  from $H$.
\end{itemize}

The simplest case $\delta = 0$
of the following result is the relevant one for our applications,
but we will pass to the general case $\delta > 0$ in the course
of the proof.
\begin{theorem}\label{thm:sph-char-main-export}
  Fix a compact subset $U \subseteq \mathfrak{g}^\wedge$
  consisting of $\mathbf{H}$-stable elements.  Let $\h$ traverse
  a sequence of positive reals tending to zero.  Fix
  $0 \leq \delta < 1/2$, and let
  $a \in S^{-\infty}_{\delta}(\mathfrak{g}^\wedge)$ with
  $\supp(a) \subseteq U$.  Let $\pi$ and $\sigma$ be
  $\h$-dependent tempered irreducible unitary representations of
  $G$ and $H$.
  Then
  \begin{equation}\label{eqn:sph-char-main-export}
    \mathcal{H}_{\sigma}(\Opp_{\h}(a))
    = \int_{\h \mathcal{O}_{\pi,\sigma}} a + O(\h^{1-2\delta}),
  \end{equation}
  More precisely, 
  there are
  differential
  operators $\mathcal{D}_j$ on $\mathfrak{g}^\wedge$
  with the following properties:
  \begin{itemize}
  \item $\mathcal{D}_0 a = a$.
  \item $\mathcal{D}_j$ has
    order $\leq 2j$ and
    has
    % is homogeneous of degree $j$:
    homogeneous degree $j$:
    $\mathcal{D}_j (a_{\h}) = \h^j (\mathcal{D}_j a)_{\h}$.
  \item For each fixed $J \in \mathbb{Z}_{\geq 0}$,
    \begin{equation}\label{eqn:sph-char-main-export-2}
      \mathcal{H}_{\sigma}(\Opp_{\h}(a))
      =
      \sum_{0 \leq j < J}
      \h^j
      \int_{\h \mathcal{O}_{\pi,\sigma}}
      \mathcal{D}_j a + O(\h^{(1-2\delta) J}).
    \end{equation}
  \end{itemize}
  We may take the implied constant in
  \eqref{eqn:sph-char-main-export-2}
  to be $C \nu(a)$,
  where \begin{itemize}
  \item $C \geq 0$ is a scalar depending at most upon $U$ and $\delta$,
    and
  \item $\nu$
    is
    a continuous seminorm on
    $S^{-\infty}_{\delta}(\mathfrak{g}^\wedge)$
    depending at most upon $J$.
  \end{itemize}
\end{theorem}

We note
that the maps
\[
a \mapsto \int_{\h \mathcal{O}_{\pi,\sigma}} a,
\]
whose domain
we take
to be
the class of
symbols $a$ arising in theorem
\ref{thm:sph-char-main-export},
are $\h$-uniformly continuous.
Indeed, by the support assumption on $a$,
the integral on the RHS
vanishes identically
unless the pair $(\h \lambda_\pi, \h \lambda_\sigma)$
belongs to a fixed compact subset of the set of stable pairs;
the claim thus follows from the discussion of
\S\ref{sec:integr-transf-ident}.
Since $\mathcal{D}_j$ has order $\leq 2 j$,
we deduce in particular that
\begin{equation}\label{eqn:obvious-bound-for-RHS-of-rel-char}
  \h^j \int_{\h \mathcal{O}_{\pi,\sigma}} \mathcal{D}_j a \ll
  \h^{(1 - 2 \delta) j},
\end{equation}
which explains
why \eqref{eqn:sph-char-main-export-2}
remains consistent as $J$ varies.

Theorem \ref{thm:sph-char-main-export} applies readily to
$\mathcal{H}_{\sigma}(\Opp_{\h}(a_1) \dotsb \Opp_{\h}(a_k))$ for
fixed $k$ and $a_1,\dotsc,a_k$ as in the hypothesis: just expand
$\Opp_{\h}(a_1) \dotsb \Opp_{\h}(a_k)$ using the composition
formula \eqref{eqn:comp-with-remainder-J}, then apply the
uniform continuity of
$\mathcal{H}_\sigma : \Psi^{-N}(\pi) \rightarrow \mathbb{C}$ to
the remainder and the main formula
\eqref{eqn:sph-char-main-export-2} to the other terms.
Taking the resulting estimate
to leading order gives in particular that
\begin{equation}\label{eqn:sph-char-main-export-multiple}
  \mathcal{H}_{\sigma}(
  \Opp_{\h}(a_1 ) \dotsb \Opp_{\h}(a_j))
  = \int_{\h \mathcal{O}_{\pi,\sigma}} a_1 \dotsb a_k + O(\h^{1-2\delta}).
\end{equation}
Recall from \S\ref{sec:stab-cons}
the definition of the notation (e.g.) $\mathcal{O}(\h \lambda_\sigma)$.
Using theorem \ref{thm:integr-transf-ident},
it follows also that if
$\pi$ admits a limit orbit
$(\mathcal{O},\omega)$,
then
\begin{equation}\label{eqn:sph-char-main-export-limiting}
  \mathcal{H}_{\sigma}(
  \Opp_{\h}(a)
  )
  \simeq
  \begin{cases}
    \int_{\mathcal{O}(\h \lambda_\sigma)} a 
    & \text{if } \mathcal{O}_{\pi,\sigma} \neq \emptyset, \\
    0 & \text{otherwise},
  \end{cases}
\end{equation}
where $\pr : \mathfrak{g}^\wedge \rightarrow
\mathfrak{h}^\wedge$
is the natural projection
and $A \simeq B$ means
$A = B + o_{\h}(1)$.
One has also the analogue
of \eqref{eqn:sph-char-main-export-limiting}
for
multiple symbols,
as in \eqref{eqn:sph-char-main-export-multiple}.

For the proof of theorem
\ref{thm:sph-char-main-export},
we may assume that \eqref{eqn:inf-chars-concentrated-near-image-of-U}
is satisfied,
since otherwise the integrals
$\int_{\h \mathcal{O}_{\pi,\sigma}} \mathcal{D}_j a$
are eventually identically zero,
and so the claim \eqref{eqn:sph-char-main-export-2}
follows from
the \emph{a priori} estimate
\eqref{eqn:inf-chars-concentrated-near-image-of-U-otherwise}.

\subsection{Reduction to symbols on the product}
We perform here an important technical
reduction
for the proof of theorem \ref{thm:sph-char-main-export}.
We introduce the notation
\[
\mathbf{M} := \mathbf{G} \times \mathbf{H},
\quad
\mathbf{S} := \text{diagonal embedding of 
$\mathbf{H}$ in $\mathbf{M}$},
\]
and set
\[\tau := \pi  \boxtimes \overline{\sigma },
\]
so that $\tau$ is
a tempered irreducible representation of the reductive
group $M$;
conversely, every such $\tau$ arises in this way.
We equip $S$ with the transport of Haar from $H$.
Recall from \eqref{eqn:formula-H-sigma-arch-case}
that
\[
\mathcal{H}_{\sigma}(\Opp_{\h}(a))
= \sum_{\substack{
    v \in \mathcal{B}(\pi) \\
    u \in \mathcal{B}(\sigma) 
  }}
\int_{s \in H}
\langle s \Opp_{\h}(a) v, v \rangle
\langle u, s u  \rangle.
\]
We may rewrite the integrand as
\[
\langle s \Opp_{\h}(a) v, v \rangle
\langle u, s u  \rangle
= 
\langle s \Opp_{\h}(a) (v \otimes \overline{u}),
v \otimes \overline{u} \rangle,
\]
where here $s$ acts on $\tau$ diagonally
while $\Opp_{\h}(a)$
acts
via the restriction of $\tau$ to $G$.
Thus
\[
\mathcal{H}_{\sigma}(\Opp_{\h}(a))
= \sum_{v \in \mathcal{B}(\tau)}
\int_{s \in S}
\langle s \Opp_{\h}(a) v, v \rangle.
\]

We now exploit the $S$-invariance
to ``fatten up'' $\Opp_{\h}(a)$;
this will have the effect
of replacing the symbol $a$
on $\mathfrak{g}^\wedge$
by a symbol on $\mathfrak{m}^\wedge$
supported close to $\mathfrak{s}^\perp
:= \{\xi \in \mathfrak{m}^\wedge : \xi|_{\mathfrak{s}} = 0\}$.
To that end,
fix $b \in C_c^\infty(\mathfrak{s}^\wedge)$
supported in a small neighborhood of the origin
and identically $1$ in a smaller neighborhood.
We may then form $\Opp_{\h}(b)$, which acts
on $\tau$ via its restriction to $H$.
By invariance of Haar measure, we have
$\int_{s \in S}
\langle s \Opp_{\h}(b) \Opp_{\h}(a) v, v \rangle
=
c_0
\int_{s \in S}
\langle s \Opp_{\h}(a) v, v \rangle$,
where $c_0 := \int_{\mathfrak{s}} \chi b_{\h}^\vee
= 1 + \O(\h^\infty)$.
Combining this with the \emph{a priori} bound
$\mathcal{H}_{\sigma}(\Opp_{\h}(a)) \ll \h^{-\O(1)}$ (see \eqref{eqn:a-priori-for-rel-char}),
we obtain
\begin{equation}\label{eqn:formula-h-sig-opp-a-after-inserting-opp-b}
    \mathcal{H}_{\sigma}(\Opp_{\h}(a))
=
\sum_{v \in \mathcal{B}(\tau)}
\int_{s \in S}
\langle s \Opp_{\h}(b) \Opp_{\h}(a) v, v \rangle
+ \O(\h^\infty).
\end{equation}

The composition $\Opp_{\h}(b) \Opp_{\h}(a)$
of operators on $\tau$ 
is a bit subtle because the symbols $a,b$
are defined using different Lie algebras.
We may nevertheless
compose them using
\eqref{eqn:comp-with-remainder-J-diff-subspaces};
what's crucial here is that $a$ and $b$
both have order $-\infty$,
and
$\mathfrak{m}$ is spanned
by $\mathfrak{s}$ and $\mathfrak{g}$.
We obtain in this way -- for any fixed $N_1, N_2 \geq 0$,
and large enough fixed $J \geq 0$ --
an expansion
\begin{equation}\label{eqn:b-a-expanded-in-prelim-rel-char-red}
  \Opp_{\h}(b) \Opp_{\h}(a)
  \equiv 
  \Opp_{\h}(a')
  \mod{\h^{N_1} \Psi^{-N_2}(\tau)}
\end{equation}
where
\[
a' :=
\sum_{0 \leq j < J}
\h^j \Opp_{\h}(b \star^j a)
\in C_c^\infty(\mathfrak{m}^\wedge).
\]

Arguing as in \S\ref{sec:spher-char-disint}
--
using now
that
$\int_{S} \Xi_M < \infty$ --
we see that the formula
\begin{equation}\label{eqn:defn-of-H-on-tau}
  \mathcal{H}(T)
  =
  \int_{s \in S}
  \trace(s T)
  =
  \sum_{v \in \mathcal{B}(\tau)}
  \int_{s \in S}
  \langle s T v, v \rangle
\end{equation}
defined initially
by the first equality
for smooth finite-rank tensors
$T \in \tau \otimes \overline{\tau }$,
extends continuously
to
\[
\mathcal{H} : \Psi^{-N}(\tau) \rightarrow \mathbb{C},
\]
uniformly in $\pi$ and $\sigma$,
for $N$ sufficiently large but fixed.
These observations give an adequate
estimate for the contribution
to \eqref{eqn:formula-h-sig-opp-a-after-inserting-opp-b}
from the remainder term
implicit
in \eqref{eqn:b-a-expanded-in-prelim-rel-char-red}.
Thus
$\mathcal{H}_{\sigma}(\Opp_{\h}(a))$
is given up to acceptable error
by
$\mathcal{H}(\Opp_{\h}(a'))$.

Recall that $a$ was assumed supported
on a fixed compact collection of $\mathbf{H}$-stable
elements of $\mathfrak{g}^\wedge$.
We claim that if the support of $b$
is chosen small enough,
then $a'$
will be supported on a compact
collection of $\mathbf{S}$-stable elements
of $\mathfrak{m}^\wedge$.
Indeed,
if the support of $b$ is small,
then
the symbol $a'$ will be supported
close to
$\mathfrak{s}^\perp$.
But we have an identification
\begin{equation}\label{eqn:S-perp-vs-G-wedge}
  \mathfrak{s}^\perp = \{(\xi, - \xi|_{\mathfrak{h}}) : \xi \in \mathfrak{g}^\wedge \} \cong \mathfrak{g}^\wedge
\end{equation}
which
intertwines the coadjoint actions of $\mathbf{S}$ and
$\mathbf{H}$,
hence identifies
$\mathbf{S}$-stable elements
of $\mathfrak{s}^\perp$
with $\mathbf{H}$-stable elements
of $\mathfrak{g}^\wedge$.
Since the $\mathbf{S}$-stable locus
in $\mathfrak{m}^\wedge$
is open,
the claim follows.

The coadjoint multiorbit
of $\tau$ is given by
\[
\mathcal{O}_\tau = \mathcal{O}_\pi \times
\mathcal{O}_{\overline{\sigma}}
= \{(\xi, -\eta) : \xi \in \mathcal{O}_\pi,  \eta \in \mathcal{O}_\sigma\},
\]
so the identification
\eqref{eqn:S-perp-vs-G-wedge}
induces
\[
\mathcal{O}_\tau \cap \mathfrak{s}^\perp
\cong
\mathcal{O}_{\pi,\sigma},
\]
intertwining $S$ and $H$.
In particular,
$\mathcal{O}_\tau \cap \mathfrak{s}^\perp$
is an $S$-torsor;
we equip it
with the transport of Haar from $S$,
which is then compatible
with the above identification.
We equip $\h \mathcal{O}_\tau \cap \mathfrak{s}^\perp$
similarly.

Since $b \equiv 1$ in a neighborhood of the origin
in $\mathfrak{s}^\wedge$,
we have $\mathfrak{a} '(\xi) = a(\xi|_{\mathfrak{g}})$
for all $\xi \in \mathfrak{m}^\wedge$
close to $\mathfrak{s}^\perp$.
In particular,
\[
\int_{\h \mathcal{O}_\tau \cap \mathfrak{s}^\perp}
a' =
\int_{\h \mathcal{O}_{\pi,\sigma}} a.
\]
More generally,
any homogeneous differential operator $\mathcal{D} '$
on $\mathfrak{m}^\wedge$
induces a homogeneous
differential operator $\mathcal{D}$ on
$\mathfrak{g}^\wedge$,
of the same homogeneity degree and of no larger order,
so that  $\mathcal{D}' a' (\xi) = \mathcal{D} a(\xi|_{\mathfrak{g}})$
for $\xi \in \mathfrak{s}^{\perp}$; in particular
\[
\int_{\h \mathcal{O}_\tau \cap \mathfrak{s}^\perp}
\mathcal{D}' a' =
\int_{\h \mathcal{O}_{\pi,\sigma}} \mathcal{D} a.
\]
The proof of theorem \ref{thm:sph-char-main-export}
thereby reduces
to that of the following
(in which we
have relabeled $(a',\tau)$ to $(a,\pi)$):
\begin{theorem}\label{thm:sph-chr-gen}
  Fix a compact subset $U \subseteq \mathfrak{m}^\wedge$
  consisting of $\mathbf{S}$-stable elements.  Let $\h$ traverse
  a sequence of positive reals tending to zero.  Fix
  $0 \leq \delta < 1/2$, and let
  $a \in S^{-\infty}_{\delta}(\mathfrak{m}^\wedge)$ with
  $\supp(a) \subseteq U$.  Let $\pi$ be
  an
  $\h$-dependent tempered irreducible unitary representation of
  $M$, and $\Opp : S^m(\mathfrak{m}^\wedge)
  \rightarrow \Psi^m(\pi)$ as usual.
  There are
  differential
  operators $\mathcal{D}_j$ on $\mathfrak{m}^\wedge$,
  satisfying properties analogous
  to those enunciated
  in the statement of theorem \ref{thm:sph-char-main-export},
  so that for each fixed $J \geq 0$,
  \begin{equation}\label{eqn:sph-char-gen-refined}
    \mathcal{H}(\Opp_{\h}(a))
    =
    \sum_{0 \leq j < J}
    \int_{\h \mathcal{O}_\pi \cap \mathfrak{s}^\perp}
    \mathcal{D}_j a + O(\h^{(1-2\delta) J}).
  \end{equation}
\end{theorem}

The proof of theorem \ref{thm:sph-chr-gen} occupies
the remainder of this section.
The discussion
of \S\ref{sec:intro-local-issues},
phrased in terms of microlocalized vectors,
might serve as a useful guide
to the following arguments.

For the same reasons as explained in
\S\ref{sec:priori-estimates-rel-char}, we may reduce to the case
that $\h \lambda_\pi$ is within $o(1)$ of the image of $U$ in
$[\mathfrak{m}^\wedge]$, so that
$a \mapsto \int_{\h \mathcal{O}_\pi \cap \mathfrak{h}^\perp}$ is
$\h$-uniformly continuous, and
\begin{equation}\label{eqn:rel-char-easy-RHS-bound-product-case}
  \h^j
\int_{\h \mathcal{O}_\pi \cap \mathfrak{s}^\perp}
\mathcal{D}_j a
\ll
\h^{(1 - 2 \delta ) j}.
\end{equation}

We may assume that $0 \leq \delta < 1/2$ is sufficiently large:
the problem becomes more general as $\delta$ increases, and the
asymptotic expansion \eqref{eqn:sph-char-gen-refined} for a
given $\delta$ (taken with $J$ sufficiently large) implies it
for all smaller values.

Recall (from \eqref{eqn:defn-of-H-on-tau})
that
\[
\mathcal{H}(\Opp_{\h}(a))
= \int_{s \in H}
\trace(\pi(s) \Opp_{\h}(a)).
\]
We will analyze below the contribution to the latter from
various ranges of $\|s - 1\|$, where $\|.\|$ denotes the
operator norm on $\End(\mathfrak{m}^\wedge)$.  We note that $S$
contains no nontrivial central elements of $M$, so that the
coadjoint representation of $M$, restricted to $S$, is a
faithful representation; thus $\|s - 1\|$ may be regarded as
quantifying the distance from $s$ to the identity element $1$ of
$S$.

\subsection{Small elements give the expected main term}
\label{sph-chr-sec-3}
Let $j_S$ be attached as in \S\ref{sec:measures-et-al-G-g-g-star}
to the group $S$.
\begin{lemma*}
  Fix $\Theta \in C_c^\infty(\mathfrak{s})$,
  with $\Theta \equiv 1$ near $0$.
  Fix $J \in \mathbb{Z}_{\geq 0}$ and
  $1/2> \delta'  >\delta$.
  Then
  \begin{equation}\label{eqn:sphr-chr-exp-heuristic-applied}
    \int_{y \in \mathfrak{s}}
    \Theta (y/\h^{\delta'})
    j_S(y)
    \trace(\pi(\exp(y)) \Opp_{\h}(a))
    =
    \sum_{0 \leq j < J} \h^j
    \int_{\h \mathcal{O}_\pi \cap \mathfrak{s}^\perp}
    \mathcal{D}_j a
    + \O(\h^{(1-2\delta) J}),
  \end{equation}
  with $\mathcal{D}_j$ as in
  the statement of theorem \ref{thm:sph-chr-gen}.
\end{lemma*}

The basic idea of the proof is as follows.  The LHS involves
traces of group elements close to the identity, which may be
evaluated with the Kirillov formula.
The conclusion then follows essentially as in
the formal sketch \eqref{second heuristic}.

In practice, we implement the proof by using
the multiplication law for symbols; the  $\Theta(y/h^{\delta'})$
factor is accounted for, in the proof below, by the symbol $s$.  
The simplest case  for the argument is when $\delta =0$. 
In that case,  both $a$ and the symbol $s$ belong to 
the ``easy'' symbol class $S^0$ defined in 
\S\ref{sec:defin-basic-symb}, although their supports are on quite different scales.

\begin{proof}
  Note first of all that it is permissible to prove the statement with $\h^{(1-2 \delta)J}$ replaced
  by $\h^{J'}$ so long as $J' \rightarrow \infty$ as $J
  \rightarrow \infty$; 
  one just applies it with a larger value of $J$
  to obtain the version above, noting that (by
  \eqref{eqn:rel-char-easy-RHS-bound-product-case})
  the contribution of $\mathcal{D}_j a$ has size $\O(\h^{(1-2 \delta)j})$.

  We will first establish the modified form of
  \eqref{eqn:sphr-chr-exp-heuristic-applied} obtained by
  omitting the factor $j_S(y)$ from the integrand;
  we will later explain why including
  this factor does not affect the required conclusion.
  With this modification, we may write the
  left-hand side as
  \begin{equation}\label{eq:rwerite-small-elts}
    \int_{y \in \mathfrak{s}}
    \Theta (y/\h^{\delta'})
    \trace(\pi(\exp(y)) \Opp_{\h}(a))
    = \trace(\Opp_{\h}(b) \Opp_{\h}(a)),
  \end{equation}
  where $b \in \mathcal{S}(\mathfrak{s}^\wedge)$
  is defined by requiring that
  $b_{\h}^{\wedge}(y) = \Theta(y/\h^{\delta'})$,
  i.e.,
  that
  $b(\h \eta) = \h^{\delta ' \dim(\mathfrak{s})} \Theta^\vee(\h^{\delta '} \eta)$.
  Using our assumptions on $\Theta$, 
  we 
  check readily that
   \begin{equation}\label{eq:symbol-class-for-delta-like-b}
    b \in
    \h^{\delta' (\dim \mathfrak{s})} 
    S^{-\infty}_{1-\delta'}(\mathfrak{s}^{\wedge})
  \end{equation}
  and
  \begin{equation} \label{convenient property of FT}
    \int_{\eta \in \mathfrak{s}^{\wedge}}
    \eta^{\alpha} b(\eta) =
    1_{\alpha = 0},
  \end{equation}
  for all multi-indices $\alpha$,
  where $1_X$ denotes the indicator function
  for the condition $X$.
  The property \eqref{convenient property of FT}
  remains valid up to an additive error $\O(\h^{\infty})$
  if we replace $\mathfrak{s}^{\wedge}$ by a ball of fixed
  radius about the origin
  (or indeed, by a ball of radius $\O(\h^{1-\delta ''})$,
  $\delta '' > \delta '$).
  We should thus think of $b$ as a very strong approximation to the delta function on $\mathfrak{s}^{\wedge}$,
  with thickness at scale $\h^{1-\delta'}$.

  Although $b$ is defined using a smaller Lie algebra than $a$,
  we can compose $\Opp_{\h}(a)$ and $\Opp_{\h}(b)$ using
  theorem \ref{thm:star-prod-asymp-general}.
  % the results of  
  % \S\ref{sec:star-prod-expn} and 
  % \S\ref{sec:gener-prop-subsp-2}.
  Thanks to the crucial assumption
  $\delta' > \delta$ (which may be rewritten
  $\delta + (1 - \delta') < 1$),
   we obtain in this way
  an asymptotic expansion
  \begin{equation*}
    \Opp_{\h}(b) \Opp_{\h}(a)
    =
    \sum_{0 \leq j < J}
    \h^j \Opp_{\h}(b \star^j a)
    + \Opp_{\h}(r),
  \end{equation*}
  involving
  star products
  \[
    \h^j b \star^j a \in
    \h^{(\delta ' - \delta) j}
    S^{-\infty}_{\delta'}(\mathfrak{m}^{\wedge})
  \]
  and remainder
  $r \in \h^{(\delta ' - \delta) J}
  S^{-\infty}_{\delta'}(\mathfrak{m}^{\wedge})$.
  By the consequence \eqref{eqn:basic-trace-bound} of
  the Kirillov formula,
  we obtain a satisfactory estimate for $\trace(\Opp_{\h}(r))$
  provided that $J$ is taken sufficiently large.
  To the remaining terms we apply
  the Kirillov formula expanded in
  terms of differential operators
  (see \eqref{eqn:kirillov-expanded-in-diff-ops}).
  We obtain an asymptotic
  expansion
  for the RHS of \eqref{eq:rwerite-small-elts}
  as a linear combination
  taken over small $j_1, j_2 \geq 0$
  of integrals
  \begin{equation}\label{eq:expn-after-star-and-Kirillov}
    \h^{j_1 + j_2}
    \int_{\zeta \in \h \mathcal{O}_\pi}
    \partial^{\alpha '}
    (\zeta^{\gamma}
    \cdot 
    \partial^\alpha a(\zeta) \cdot
    \partial^\beta b(\zeta))
  \end{equation}
  involving multi-indices satisfying
  $|\alpha| + |\beta| - |\gamma| = j_1$,
  $|\alpha|, |\beta| \leq j_1$
  and
  $|\alpha '| = j_2$.
  By the product
  rule applied to $\partial^{\alpha '}$,  
  followed by partial integration,
  we may rewrite the above as
  \[
    \h^j
    \int_{\h \mathcal{O}_\pi}
    b
    \mathcal{D}_j a,
  \]
  where $j = j_1 + j_2$
  and $\mathcal{D}_j$ has the form indicated
  in the statement of the theorem.

  The map $\h \mathcal{O}_\pi \rightarrow \mathfrak{s}^\wedge$
  has full rank in a neighborhood of $\omega$, so we may fix a
  small open neighborhood $N_\omega \subset \h \mathcal{O}_\pi$
  of $\omega$
  and local coordinates
  \[
    N_\omega \ni 
    \xi = (\xi_1,\xi_2)
    \in \mathfrak{s}^\wedge \times 
    (\h \mathcal{O}_\pi \cap \mathfrak{s}^\perp),
    \quad
    \xi_1 \approx 0,
    \quad
    \xi_2 \approx \omega
  \]
  so that
  \begin{itemize}
  \item the coordinate $\xi_1$ defines
    the projection to $\mathfrak{s}^\wedge$, and
  \item $\xi = (0,\xi)$ for $\xi \in \h \mathcal{O}_\pi \cap \mathfrak{s}^\perp$.
  \end{itemize}
  The
  integral of a function $f$ on $\h \mathcal{O}_\pi$
  supported on $N_\omega$ may be expressed in such coordinates as
  \[
    \int_{\h \mathcal{O}_\pi} f
    = \int_{\xi_1 \in \mathfrak{s}^\wedge}
    \int_{\xi_2 \in h \mathcal{O}_\pi \cap \mathfrak{s}^\perp}
    f(\xi_1,\xi_2)
    w(\xi_1,\xi_2)
    \, d \xi_1 \, d \xi_2,
  \]
  where $d \xi_1$ denotes the given Haar measure,
  $d \xi_2$ the
  transport of Haar from $S$, and $w$ is a smooth Jacobian
  factor.
  We have $b(\xi) = b(\xi_1)$,
  and $a$ is supported
  in $N_\omega$ for small enough $\h$,
  so
  \[
    \int_{\h \mathcal{O}_\pi}
    b
    \mathcal{D}_j a
    =
    \int_{(\xi_1,\xi_2) \in N_\omega}
    b(\xi_1)
    w(\xi_1,\xi_2)
    \mathcal{D}_j a(\xi_1,\xi_2) \,  d \xi_1 \, d \xi_2.
  \]
  Using Taylor's theorem,
  we may write
  \begin{equation}\label{eq:taylor-of-a-for-rel-char-exp}
    w(\xi_1,\xi_2)
    \mathcal{D}_j a(\xi_1,\xi_2)
    =
    w(0,\xi_2)
    \mathcal{D}_j a(0,\xi_2)
    +
    \sum_{1 \leq |\alpha| < A}
    c_\alpha(\xi_2) \xi_1^{\alpha}
    +
    \O(
    |\xi_1|^A \h^{- \delta A}
    )
  \end{equation}
  for any fixed $A$.
  The compatibility of the top sequence in
  \eqref{eqn:sequences-compatible-fibral}
  tells us that
  $w(0, \xi_2) = 1$,
  so the contribution to $\int_{\h \mathcal{O}_\pi} b
  \mathcal{D}_ja$
  from the first term on the RHS of
  \eqref{eq:taylor-of-a-for-rel-char-exp}
  is
  \begin{align*}
    \int_{(\xi_1,\xi_2) \in N_\omega}
    b(\xi_1)
    \mathcal{D}_j a(0,\xi_2)
    \, d \xi_1 \, d \xi_2
    &=
    ( \int_{\xi_1} b(\xi_1) \, d \xi_1)
    ( \int_{\xi_2} \mathcal{D}_j a(0,\xi_2) \, d \xi_2)
      + \O(\h^\infty)
      \\
    &= 
    \int_{\h \mathcal{O}_\pi \cap \mathfrak{s}^\perp} \mathcal{D}_j a + \O(\h^\infty).
  \end{align*}
  The remaining Taylor monomials
  contribute $\O(\h^\infty)$,
  by \eqref{convenient property of FT}
  and the remark thereafter.
  The contribution from the remainder
  term
  is dominated by  
  \[
    \h^{-\delta A}
    \int_{\xi_1} |\xi_1|^A |b(\xi_1)| \, d \xi_1
    \ll
    \h^{\delta ' \dim(\mathfrak{s})
      + (1 - \delta - \delta ') A
    },
  \]
  thanks to
  \eqref{eq:symbol-class-for-delta-like-b}
  and the definition of $b$;
  informally,
  $b(\eta)$ is concentrated on $|\eta| \ll \h^{1 - \delta '}$.
  Since $\delta, \delta ' < 1/2$,
  we have $1 - \delta - \delta ' > 0$,
  so this last estimate is adequate
  for $A$ sufficiently large.

  This completes the proof of the modified
  assertion obtained
  by omitting $j_S$.
  To incorporate that factor,
  we define
  the symbol
  $c$ by requiring
  that $c_{\h}^\vee(y) = \Theta(y/\h^{\delta '}) j_S(y)$,
  and then follow the previous argument
  up to
  \eqref{eq:expn-after-star-and-Kirillov},
  leading us to consider
  \begin{equation}\label{eq:expn-after-star-and-Kirillov-2}
    \h^{j_1 + j_2}
    \int_{\zeta \in \h \mathcal{O}_\pi}
    \partial^{\alpha '}
    (\zeta^{\gamma}
    \cdot 
    \partial^\alpha a(\zeta) \cdot
    \partial^\beta c(\zeta))
  \end{equation}
  We apply \S\ref{sec:appl-tayl-theor}
  to obtain an asymptotic expansion
  for $c$
  given up
  to acceptable error
  by a sum
  over finitely many multi-indices
  $\alpha$
  of the quantities
  $\frac{\partial ^\alpha j_S(0) }{ \alpha !}
  (-\h)^{|\alpha|}
  \partial^{\alpha} b(\zeta)$;
  inserting these into \eqref{eq:expn-after-star-and-Kirillov-2}
  yields terms of the form
  \eqref{eq:expn-after-star-and-Kirillov}, which we treat as
  before.
\end{proof}

\subsection{Huge elements contribute negligibly}
\label{sph-chr-sec-4}
\begin{lemma*}
  For each fixed $N \geq 0$ there is a fixed $N' \geq 0$
  so that
  \begin{equation}\label{eqn:sph-char-super-tail}
    \int_{s \in S: \|s - 1\| \geq \h^{-N'}}
    \trace(\pi(s) \Opp_{\h}(a))
    \ll \h^N.
  \end{equation}
\end{lemma*}
\begin{proof}
  We note that,
  thanks to
  the finitude
  $\int_{S} \Xi_M < \infty$
  and a Lojasiewicz-type inequality,
  the estimate
  \[
  \int_{s \in S : \|s-1\| \geq X}
  \Xi_M(s) \ll X^{-\eta}\]
  holds
  for some fixed $\eta > 0$.
  By the matrix coefficient bounds for tempered representations (see
  \S \ref{sec:prel-repr-reduct})
  and the trace norm estimate
  \eqref{eqn:control-trace-norm-via-specific-operator-norm},
  we deduce that the LHS of \eqref{eqn:sph-char-super-tail}
  is dominated for some fixed $L \geq 0$
  (independent of $N'$) by
  \[
  \|\Delta^L \Opp_{\h}(a) \Delta^L \|_{1}
  \int_{s \in S : \|s-1\| \geq \h^{-N'}}
  \Xi(s)
  \ll
  \h^{-4 L + \eta N'}.
  \]
  Taking $N'$ large enough gives an adequate estimate.
\end{proof}
\subsection{Medium-sized elements contribute negligibly}
\label{sph-chr-sec-5}
This section contains
the most delicate arguments.

\begin{lemma}\label{lem:large-s-1-param-trick}
  Let $\Omega \subset \mathfrak{m}^\wedge$ be a compact
  collection of $\mathbf{S}$-stable elements.
  For each
  $\omega \in \Omega$
  and
  $s \in S$
  there exists $u \in \mathfrak{m}$
  with $|u| = 1$
  so that
  \[
  |s \cdot u| \asymp
  |(s \cdot u) \omega|
  \gg \|s\|^\eps,
  \]
  where the implied constants and
  the positive quantity $\eps$
  depend only upon $\Omega$.
  (Recall from \S\ref{sec:measures-et-al-G-g-g-star}
  that the natural pairing between $\mathfrak{m}$ and
  $\mathfrak{m}^\wedge$
  is denoted by juxtaposition.)
\end{lemma}
\begin{proof}
  Fix a maximal compact
  subgroup $K$ of $S$ and maximal split Cartan subalgebra
  $\mathfrak{a}$ of $S$, so that $S = K \exp(\mathfrak{a}) K$.
  We may reduce readily to verifying that the conclusion holds when
  $s \in \exp(\mathfrak{a})$, say $s = \exp(r z)$ with $r > 0$,
  $z \in \mathfrak{a}$ and $|z| = 1$.  Consider the weight
  decompositions for the adjoint and coadjoint actions of $z$:
  \[
  \mathfrak{m} = \oplus_{t \in
    \mathbb{R}} \mathfrak{m}_t,
  \quad 
  \mathfrak{m}^\wedge = \oplus_{t \in \mathbb{R}}
  \mathfrak{m}_t^\wedge.
  \]
  Each $\omega \in \Omega$ is $\mathbf{S}$-stable,
  and so has both
  positive and negative weights
  (e.g., by applying Hilbert--Mumford
  to a dense set of one-parameter subgroups,
  or by noting that the $S$-orbit of $\omega$
  must be topologically closed).
  By compactness, each
  $\omega \in \Omega$ has projection onto
  $\oplus_{t \leq -\eps} \mathfrak{m}_t^\wedge$ of norm $\gg 1$.
  We may choose a weight vector
  $u \in \oplus_{t \geq \eps}\mathfrak{m}_t$ with $|u| = 1$ and
  $|u \omega| \asymp 1$.  Then $s \cdot u$ is a multiple
  of $u$ with $|s \cdot u| \gg \|s\|^\eps$.
  The required estimates follow for $\omega$.
  The same choice of $u$ works for all $\omega '$
  in a small neighborhood of $\omega$,
  so we may conclude by the compactness of $\Omega$.
\end{proof}

\begin{lemma}
  Suppose $s \in S$ satisfies $\|s - 1\| \geq \h^{1/2 - \eta}$
  for some fixed $\eta > 0$.
  Then
  \[
  \trace(\pi(s) \Opp_{\h}(a)) \ll \h^{\infty}.
  \]
\end{lemma}

The informal idea is to write $a = \sum a_i$, where
each $a_i$ has very small microlocal support; the trace of
$\pi(s) \Opp_{\h}(a_i)$ is then small, because the stability
condition implies that $s \cdot \mathrm{supp}(a_i)$ and
$\mathrm{supp}(a_i)$ are disjoint.  It is worth noting that the
result is essentially optimal: if $\|s-1\| \asymp \h^{1/2}$,
that is, if $s$ is just a bit closer to the identity than the
scale prescribed by the lemma, then $s$ does not move
Planck-scale balls significantly.

\begin{proof}
  The problem becomes more general as $\delta$ increases,
  so we may and shall assume that
  \begin{equation}\label{eqn:delta-ineq-wrt-1/2-eta}
    \delta \geq 1/2 - \eta/2.
  \end{equation}
  We fix $\eps > 0$ sufficiently small in terms of $\delta, \eta$ and $U$.

  By decomposing $a$ into
  $\h^{-\O(1)}$ many pieces,
  we may assume
  that it is supported on a ball $B(\omega, \h^\delta)
  := \{\xi \in \mathfrak{m}^\wedge : |\xi - \omega| \leq
  \h^{\delta}\}$
  centered at some $\omega \in U$.
  We may choose a compactly-supported
  ``envelope'' $\psi \in S^{-\infty}_{\delta}$
  with
  \begin{itemize}
  \item $0 \leq \psi \leq 1$,
  \item $\psi \equiv 1$ on $B(\omega,2 \h^\delta)$, and
  \item $\psi \equiv 0$ on $B(\omega,3 \h^\delta)$.
  \end{itemize}
  We may write
  $\trace(\pi(s) \Opp_{\h}(a))
  = E_1 + E_2$,
  where
  \[
  E_1 := 
  \trace(\pi(s) \Opp_{\h}(a) \Opp_{\h}(1-\psi)),
  \]
  \[
  E_2 :=
  \trace(\Opp_{\h}(\psi) \pi(s) \Opp_{\h}(a)).
  \]
  Since $a$ and $1 - \psi$ have disjoint supports,
  we see
  (by \S\ref{sec:disjoint-supports}
  and \S\ref{sec:appl-kir-type-formulas})
  that
  $E_1 \ll \h^{\infty}$.

  We turn now to $E_2$.
  The idea is that the translation by $s$ of the support
  of the symbol $a$
  is disjoint from the support of $\psi$.
  This idea can be implemented
  rigorously using the operator calculus when $\|s\|$ is not too
  large.
  Indeed, suppose first that $\|s-1\| \leq \h^{-\eps}$, 
  which means
  by
  \eqref{eqn:delta-ineq-wrt-1/2-eta})
  that 
  \[
  \h^{\delta-\eta/2} \leq \|s-1\| \leq \h^{-\eps}.
  \]
  The upper bound implies in particular that
  $\|s\| \ll \h^{-\eps} \leq \h^{-1+\delta+\eps}$,
  so the hypothesis
  \eqref{eqn:adjoint-bound-for-G-equivariance} of
  \S\ref{sec:equivariance} is satisfied;
  by the conclusion of that section,
  the operator norm
  of $\pi(s) \Opp_{\h}(\psi) \pi(s)^{-1} - \Opp_{\h}(s \cdot
  \psi)$
  is negligible,
  so we reduce to showing that
  \[\|\Opp_{\h}(s \cdot \psi) \Opp_{\h}(a)\|_{1}\] is
  negligible.
  Since $s$ distorts Lie algebra elements
  by at most $\|s\| \ll \h^{-\eps}$,
  we have $s \cdot \psi \in S^{-\infty}_{{\delta + \eps}}$,
  and may assume that $\delta + \eps < 1/2$.
  It will thus suffice to verify
  that $s \cdot \psi$ and $a$ have disjoint supports.
  To that end, we need only verify for $\xi \in B(\omega, 3 \h^\delta)$
  that
  \begin{equation}\label{eqn:s-moves-stuff}
    \|s-1\| \geq  \h^{\delta - \eta/2}
    \implies |s \cdot \xi - \xi| \gg \h^{\delta-\eps},
  \end{equation}
  say.
  Note that the union of the sets $B(\omega, 3 \h^\delta)$,
  as $\omega$ varies over $U$ 
  and $\h$ over sufficiently small positive reals,
  is contained in a fixed compact collection
  of stable elements.
  The estimate \eqref{eqn:s-moves-stuff} follows
  in the range $\|s-1\| \geq  \h^{\eps}$ from Hilbert--Mumford, as in the proof of 
  lemma \ref{lem:large-s-1-param-trick},
  and in the remaining range $\h^{\eps} \geq \|s-1\| \geq \h^{\delta - \eta/2}$
  from Lie algebra considerations,
  using that $\xi$ has trivial $\mathfrak{s}$-centralizer.

  It remains to handle the case that
  $\|s - 1\| \geq \h^{-\eps}$, hence $\|s\| \gg \h^{-\eps}$.
  Direct application of the symbol calculus does not work as
  well here, because $s \cdot  \psi$ is overly distorted.  We instead
  construct convolution operators along well-chosen lines
  inside $\mathfrak{m}$, corresponding
  to one-parameter subgroups
  in $G$, so that there is no issue of distortion.
  
  Fix a Fourier transform
  between $\mathbb{R}$
  and its Pontryagin dual $\mathbb{R}^\wedge = i \mathbb{R}$,
  and fix
  $\Theta \in C_c^\infty(\mathbb{R}^\wedge)$
  and
  $\chi \in C_c^\infty(\mathbb{R})$,
  each identically $1$ in neighborhoods of the respective origins.
  Choose $u \in \mathfrak{m}$
  as in lemma \ref{lem:large-s-1-param-trick},
  so that $|u| = 1$ and $|s \cdot u| \gg \h^{-\eps^2}$
  and
  \begin{equation}\label{eqn:alignment-s-cdot-u-with-omega}
    \frac{s \cdot u}{|s \cdot u|} \omega
    \asymp 1.
  \end{equation}
  Set
  $r := \h^{1 - \eps^3}/|s \cdot u|$,
  so that $r \ll \h^{1 + \eps^2 - \eps^3}$,
  and
  \[
\mathcal{C}_1 := \int_{t \in \mathbb{R}}
  \frac{\Theta^\vee(t/r)}{r}
  \chi(\h^{\eps^3} t/r) \pi(\exp(t u)),
  \]
  \[
\mathcal{C}_2 := 
  \pi(s) \mathcal{C}_1 \pi(s)^{-1} = \int_{t \in \mathbb{R}}
  \frac{\Theta^\vee(t/\h^{1-\eps^3})}{\h^{1-\eps^3}}
  \chi(t/\h^{1-2 \eps^3}) \pi(\exp(t \frac{s \cdot u}{|s \cdot u|})).
  \]
  
  Informally, we should think of $\mathcal{C}_1$ as a convolution operator
  in the $u$ direction utilizing a bump function of width
  substantially smaller than $\h$,
  whereas $\mathcal{C}_2$ is a convolution operator in the $s \cdot u$ direction utilizing a bump function of width slightly greater than $\h$. 
   
  It will suffice to verify that (with $\|\cdot\|_{\infty}$ the operator norm)
  \begin{equation}\label{eqn:E1-nearly-preserves-envelope}
    \|\mathcal{C}_1 \Opp_{\h}(a)  - \Opp_{\h}(a)\|_{\infty} \ll \h^{\infty},
  \end{equation}
  \begin{equation}\label{eqn:E2-kills-Opp-a}
    \|\Opp_{\h}(\psi) \mathcal{C}_2 \|_{\infty} \ll \h^{\infty},
  \end{equation}
  because then,
   writing $\equiv$ to denote agreement up to $\O(\h^\infty)$
  and applying
  \S\ref{sec:disjoint-supports}
  and \S\ref{sec:appl-kir-type-formulas},
  we have
  \begin{align*}
    E_2 &\equiv
             \trace(\Opp_{\h}(\psi) \pi(s) \mathcal{C}_1  \Opp_{\h}(a)) \\
           &=
             \trace(\Opp_{\h}(\psi)  \mathcal{C}_2 \pi(s)
             \Opp_{\h}(a))
             \equiv 0.
  \end{align*}
  To establish
  \eqref{eqn:E1-nearly-preserves-envelope} and
  \eqref{eqn:E2-kills-Opp-a},
  let $\mathfrak{l}_1, \mathfrak{l}_2 \subseteq \mathfrak{m}$
  denote the lines spanned by $u$ and $s \cdot u$, respectively,
  and observe that we may write $\mathcal{C}_j = \Opp_{\h}(b_j)$,
  with $b_1 \in S^0_{0}(\mathfrak{l}_1)$,
  $b_2 \in S^{-\infty}_{\eps^3}(\mathfrak{l}_2)$
  satisfying
  \begin{itemize}
  \item $b_1(\xi) = 1 + \O(\h^{\infty})$ for $|\xi| \leq \h^{-\eps^4}$, and
  \item $b_2(\xi) = \O(\h^{\infty})$ for $|\xi| \geq \h^{\eps^4}$,
  \end{itemize}
  and similarly for derivatives.
  In particular, $b_1$ is approximately $1$ on the image of the support of $a$,
  while (by
  \eqref{eqn:alignment-s-cdot-u-with-omega}) $b_2$ is approximately $0$
  on the image of the support of $\psi$.
  The required estimates \eqref{eqn:E1-nearly-preserves-envelope} and \eqref{eqn:E2-kills-Opp-a}
  follow from the asymptotic expansion
  \eqref{eqn:comp-with-remainder-J-diff-subspaces}.
\end{proof}
\subsection{Completion of the proof}
\label{sph-chr-sec-6}
We note for $X \geq 1$ that
\begin{equation}\label{eqn:not-too-many-s}
  \vol(\{s \in S : \|s-1\| \leq X\})
  \ll X^{\O(1)}
\end{equation}
(cf. \cite[Lem 2.A.2.4]{MR929683}).
The definition
\eqref{eqn:defn-of-H-on-tau}
and the results of \S\ref{sph-chr-sec-4}
and \S\ref{sph-chr-sec-5}
show that for any fixed $\eta > 0$,
\[
\mathcal{H}(\Opp_{\h}(a))
= 
\int_{s \in S : \|s - 1\| \geq \h^{1/2-\eta}}
\trace(\pi(s) \Opp_{\h}(a)) + \O(\h^{\infty}).
\]
We now write $s = \exp(y)$, pull the integral back to the Lie
algebra,
and combine \S\ref{sph-chr-sec-5}
and \S\ref{sph-chr-sec-3}
to derive the required asymptotic expansion.

\begin{appendices}
  \section{Some technicalities
    related to the Plancherel
    formula}\label{sec:prel-repr-reduct}
  The aim of this section,
  which the reader is encouraged to skip,
  is to supply
  the unsurprising details
  required by the proofs of \S\ref{sec:spher-char-disint}.
  
  Let $F$ be a
  local field, either archimedean or non-archimedean.
  Let $\mathbf{G}$ be a reductive group over $F$.
  We denote as usual by $\hat{G}_{\temp} \subseteq \hat{G}$ the tempered dual of $G$,
  thus each $\pi \in \hat{G}_{\temp}$ is a tempered irreducible unitary
  representation of $G$.

  We always choose a Haar measure $d g$ on $G$
  and a maximal compact subgroup $K := K_G$ of $G$.
  For a unitary representation $\pi$ of $G$,
  we denote by $\mathcal{B}(\pi)$
  an orthonormal basis consisting of $K$-isotypic vectors.
  For $f \in L^1(G)$
  we define
  $\pi(f) := \int_{g \in G} \pi(f) f(g) \, d g$, as usual.

  When $F$ is archimedean,
  we retain the notation of Part \ref{part:micr-analys-lie}
  ($\mathfrak{U}, \Delta, \dotsc$),
 applied to the real Lie group underlying $G$.

  \subsection{Uniform bounds for $K$-types}\label{sec:uniform-bounds-k}
  Assume that $F$ is archimedean.
  The proof of \cite[Lem 10.4]{MR855239}
  shows that there is
  an element $\kappa$
  of the universal enveloping algebra
  of $K$
  with the following properties:
\begin{enumerate}[(i)]
\item
  $\kappa$ acts on each irreducible representation
  $\tau$ of $K$ by a scalar $\kappa_\tau$.\footnote{
    We use the notation
    $\kappa_\tau$
    for what Knapp denotes
    $d_\lambda^2 (1 + \|\lambda|_{Z_{\mathfrak{k}}}\|^2)$.
  }
\item $\dim(\tau) \leq \kappa_\tau^{1/2}$.
\item $\sum_{\tau \in \hat{K}} \kappa_\tau^{-1}$ is finite.
\end{enumerate}
(Explicitly,
one may take $\kappa = -c \sum_{x \in \mathcal{B}(\Lie(K))}
x^2$
for large enough $c > 0$.)

\begin{lemma*}
  Let $\pi$ be an irreducible admissible representation
  of $G$.
  \begin{enumerate}[(i)]
  \item Let
    $v \in \pi$ be $\tau$-isotypic.
    Then
    $\dim(K v)^{1/2} \|v\| \leq
    \|\kappa v\|$.
  \item 
    $\pi(\kappa)$ is positive and invertible.
    $\trace(\pi(\kappa)^{-2}) \leq C$,
    where $C$ depends only upon $G$.
  \end{enumerate}
\end{lemma*}
\begin{proof}~
  \begin{enumerate}[(i)]
  \item By \cite[Thm 8.1]{MR855239},
    we have
    $n_\tau  := \Hom_K(\tau,\pi) \leq \dim(\tau)$,
    so that
    $\dim(K v) \leq n_\tau \dim(\tau) \leq \dim(\tau)^2$.
    The conclusion follows from the enunciated properties of $\kappa$.
  \item
    $\trace(\pi(\kappa)^{-2})
    \leq
    \sum_{\tau  \in \hat{K}}
    n_\tau \dim(\tau) \kappa_\tau^{-2}
    \leq 
    \sum_{\tau  \in \hat{K}}
    \kappa_\tau^{-1}
    < \infty$.
  \end{enumerate}
\end{proof}

\subsection{Bounds for matrix coefficients\label{sec:matrix-coeff-bounds-for-tempered-reps}}

Assume that $\pi$ is tempered.

\subsubsection{}
By
\cite{MR946351},
there is a function $\Xi := \Xi_G : G \rightarrow \mathbb{R}_{>0}$
(depending also upon $K$),
called the \emph{Harish--Chandra spherical function},
with the following property:
for any $\pi \in \hat{G}_{\temp}$
and any $K$-finite $u,v \in \pi$,
one has
\begin{equation}\label{eqn:matrix-coeff-bound-basic}
  |\langle g u, v \rangle|
  \leq \Xi(g)
  (\dim K u)^{1/2} (\dim K v)^{1/2}
  \|u\| \|v\|
\end{equation}
for all $g \in G$.
The function $\Xi$ descends
to $G/Z$, where $Z$ denotes the center of $G$,
and tends to zero at infinity
on $G/Z$.

\subsubsection{}
Assume now
that $F$ is archimedean.
We may then readily translate the bound
\eqref{eqn:matrix-coeff-bound-basic}
in terms of the Sobolev norms
defined in \S\ref{sec:sobolev-spaces}:
\begin{lemma*}
  For $\pi \in \hat{G}_{\temp}$,
  $g \in G$
  and $u, v \in \pi^s$,
  \begin{equation}\label{eqn:matrix-coeff-bound-wrt-sobolev-norms}
    \left\lvert \langle g u, v \rangle \right\rvert
    \leq
    c
    \Xi(g)
    \|u\|_{\pi^s} \|v\|_{\pi^s},
  \end{equation}
  where $c \geq 0$ and $s \in \mathbb{Z}_{\geq 0}$
  depend only upon $G$.
\end{lemma*}
\begin{proof}
  Let $\kappa$
  be as in \S\ref{sec:uniform-bounds-k}.
  Let $v \in \pi$ be $K$-finite;
  write its isotypic decomposition
  as $v = \sum v_\tau$.
  By part (i) of the lemma of \S\ref{sec:uniform-bounds-k},
  \begin{equation}\label{eqn:knapp-K-type-summed-bound}
    \sum
    \dim(K v_\tau)^{1/2}
    \|v_\tau \|
    \leq
    \sum
    \kappa_\tau^{-1/2}
    \|\kappa v_\tau\|
    \leq
    (\sum \kappa_\tau^{-1})^{1/2}
    \|\kappa v\|.
  \end{equation}

  To prove \eqref{eqn:matrix-coeff-bound-wrt-sobolev-norms}, we
  may assume by continuity that the vectors $u, v$ are
  $K$-finite.  We decompose such vectors into their
  $K$-isotypic components, apply
  \eqref{eqn:matrix-coeff-bound-basic}
  to the inner product arising from each pair of components,
  and then apply \eqref{eqn:knapp-K-type-summed-bound},
  giving
  \[
    \left\lvert \langle g u, v \rangle \right\rvert
    \leq 
    (\sum \kappa_\tau^{-1})
    \Xi(g)
    \|\kappa u\| \|\kappa v\|.
  \]
  We conclude
  by
  appeal to the simple consequence \eqref{eq:pi-s-concretized}
  of the definition of $\|.\|_{\pi^s}$.
\end{proof}

\subsection{Plancherel formula\label{sec:plancherel-formula}}
Let $\pi$ be an irreducible unitary
representation
of $G$.
We denote by $\chi_\pi$ its distributional character,
as in \S\ref{ss: Kiriillov}.
The Plancherel formula asserts that, 
for $f \in C^{\infty}_c(G)$, we have the identity
\begin{equation}\label{eq:plancherel-formula-in-general}
  f(1) = \int_{\pi \in \hat{G}_{\temp}} \chi_{\pi}(f),
\end{equation}
with the latter integral taken
with respect
to a certain measure on $\hat{G}_{\temp}$,
called the {\em Plancherel measure}
dual to $d g$.
For $n \in \mathbb{Z}_{\geq 0}$
large enough in terms of $G$,
the formula extends
by continuity to the class of $n$-fold differentiable
compactly-supported functions.

\subsection{Some crude growth bounds\label{sec:trace-class-property-uniform-and-integrated}}

\subsubsection{}\label{sec:non-arch-case-coarse-control-on-tempered-dual}
Assume that $F$ is non-archimedean.
Let $U$ be a compact open subgroup of $G$.
For any admissible
representation $\pi$
of
$G$ (e.g., any irreducible unitary representation),
the space $\pi^U$ 
of $U$-fixed vectors
is finite-dimensional.
By applying
the Plancherel formula to the normalized
characteristic function of $U$,
we see moreover that
$\int_{\pi \in \hat{G}_{\temp}} \dim(\pi^U)
< \infty$.

\subsubsection{}\label{sec:arch-case-coarse-control-on-tempered-dual}
Assume that $F$ is archimedean.
\begin{lemma*}
  There is $N \in \mathbb{Z}_{\geq 0}$, depending only upon $G$,
  so that:
  \begin{enumerate}[(i)]
  \item $\sup_{\pi \in \hat{G}}
    \trace(\pi(\Delta^{-N})) < \infty$
  \item $\int_{\pi \in \hat{G}_{\temp}} \trace(\pi(\Delta^{-N}))
    < \infty$.
  \end{enumerate}
\end{lemma*}
This is likely well-known;
we record a  proof for completeness.
\begin{proof}
  Here we require implied constants
  to be uniform in $\pi$.
  \begin{enumerate}[(i)]
  \item Let $\kappa \in \mathfrak{U}$ be as in
    \S\ref{sec:uniform-bounds-k}.
    Assume that $N$ exceeds twice the degree of $\kappa$.
    By lemma
    \ref{lem:annoying-boundedness-lemma} of \S\ref{sec:membership-criteria-technical-stuff},
    the operator $A := \pi(\kappa)^2 \pi(\Delta^{-N})$
    then has uniformly bounded operator norm.
    Since $\pi(\kappa)^{-2}$ is positive,
    it follows
    that
    $\trace(\pi(\Delta^{-N}))  \ll
    \trace(\pi(\kappa)^{-2})$.
    We conclude by part (ii) of the lemma of \S\ref{sec:uniform-bounds-k}.
  \item
    By spectral theory, it suffices to establish the
    modified conclusion obtained by replacing $\Delta$ with the
    rescaled variant $\Delta_{\h}$
    (cf. \S\ref{sec:appr-delta-1}) for some fixed
    $\h \in (0,1]$.  Set $b(\xi) := \langle \xi \rangle^{-N}$.
    Let $\pi \in \hat{G}_{\temp}$.
    Then $T_\pi := \Opp_{\h}(b:\pi)^2$ is positive definite.
    By
    \S\ref{sec:appr-delta-1}, we have
    $\trace(\pi(\Delta_{\h}^{-2 N})) \asymp
    \trace(T_\pi)$
    for $\h$ small enough.
    Let
    $n \in \mathbb{Z}_{\geq 0}$ large enough that the Plancherel
    formula holds for $n$-fold differentiable functions
    $f \in C_c(G)$, and assume that $N$ is large enough in terms
    of $n$.  By the composition formula, combined with
    \S\ref{sec:smooth-away-from-origin},
    we then have $T_\pi = \pi(f)$ where $f \in C_c(G)$ is
    $n$-fold differentiable.
    Thus $\int_{\pi \in \hat{G}_{\temp}} \tr(\pi(\Delta_{\h}^{-2 N}))
    \ll \int_{\pi \in \hat{G}_{\temp}} \trace(\pi(f)) = f(1) < \infty$.
  \end{enumerate}
\end{proof}

\subsection{Plancherel formula,
  II\label{sec:plancherel-formula-2}}
We record an extension of the Plancherel
formula \eqref{eq:plancherel-formula-in-general}
to
a larger space of functions that we denote by $\mathcal{F} :=
\mathcal{F}_G$;
in brief, it consists of functions whose $G \times G$
derivatives
lie in $L^1(G, \Xi)$:
\begin{itemize}
\item In the
non-archimedean case,
we define $\mathcal{F}^U$,
for each compact open subgroup $U$ of $G$,
to be the space of
bi-$U$-invariant functions
$f : G \rightarrow \mathbb{C}$
satisfying
\begin{equation}\label{eq:defn-of-cal-F-space-for-plancherel-formula}
  \int_{g \in G} \Xi(g) |f(g)| \, d g < \infty,
\end{equation}
equipped with the evident topology (\S\ref{sec:prim-topol-vect});
we then set $\mathcal{F} := \cup \mathcal{F}^U$,
equipped with the direct limit topology.
\item 
In
the archimedean case
(so that $G$ is regarded as a real Lie group), we take for $\mathcal{F}$
the space of smooth
functions $f : G \rightarrow \mathbb{C}$ each of whose
$G \times G$-derivatives (i.e., allowing applications of both left- and
right-invariant differential operators)
satisfies the analogue
of \eqref{eq:defn-of-cal-F-space-for-plancherel-formula};
we equip $\mathcal{F}$ with
its evident topology (\S\ref{sec:prim-topol-vect}). 
\end{itemize}
In either case, observe  that
$C_c^\infty(G)$ is dense in $\mathcal{F}$
and (by the Sobolev lemma) that point
evaluations on $\mathcal{F}$ are continuous.

Let $\pi \in \hat{G}_{\temp}$.
For $f \in \mathcal{F}$,
we wish to define and study an operator $\pi(f)$ on
the space $\pi^\infty$ of smooth vectors in $\pi$.
It is natural to ask that this operator satisfies
\begin{equation}\label{eqn:defn-of-pi-f-for-f-in-F}
  \langle \pi(f) u, v \rangle
  =
  \int_{g \in G} f(g) \langle g u, v \rangle \, d g
\end{equation}
for $u, v \in \pi^\infty$;
note that the RHS of \eqref{eqn:defn-of-pi-f-for-f-in-F}
converges absolutely, thanks to
\eqref{eqn:matrix-coeff-bound-wrt-sobolev-norms} and
\eqref{eq:defn-of-cal-F-space-for-plancherel-formula}.
\begin{lemma*}
  Let $f \in \mathcal{F}$ and $\pi \in \hat{G}_{\temp}$.
  \begin{enumerate}[(i)]
  \item There is a unique continuous linear map
    $\pi(f) : \pi^\infty  \rightarrow \pi^\infty$ for which
    \eqref{eqn:defn-of-pi-f-for-f-in-F} holds.
  \item 
    The map $f \mapsto \pi(f)$ is
    continuous for the trace norm $\|.\|_{1}$ on the target.
  \item
    The map
    \[f \mapsto \int_{\pi \in \hat{G}_{\temp}}
    \|\pi(f)\|_{1}\]
    is finite-valued and continuous.
  \item
    The Plancherel formula
    \eqref{eq:plancherel-formula-in-general} remains valid.
  \end{enumerate}
\end{lemma*}
\begin{proof}
  We may initially define
  $\pi(f) u$
  as the anti-linear functional on $\pi^\infty$
  for which \eqref{eqn:defn-of-pi-f-for-f-in-F}
  holds.
  We aim then then to verify that $\pi(f) u$ is represented
  by a smooth vector
  and that the resulting map has the required properties.

  In the non-archimedean case, $f$ is $U \times U$-invariant for
  some open subgroup $U$ of $G$.  We verify readily that the
  functional $\pi(f) u$ is then $U$-invariant, hence represented
  by a unique element of the finite-dimensional space $\pi^U$.
  The remaining assertions may be verified
  by simpler analogues of the arguments to follow
  (using
  \S\ref{sec:non-arch-case-coarse-control-on-tempered-dual}
  instead of
  \S\ref{sec:arch-case-coarse-control-on-tempered-dual}).  We turn
  henceforth to the details of archimedean case.
  \begin{enumerate}[(i)]
  \item
    Choose $N \in \mathbb{Z}_{\geq 0}$ large enough
    that
    we have the bound
    (cf. \S\ref{sec:matrix-coeff-bounds-for-tempered-reps})
    $\langle g u, v \rangle \ll \Xi(g)  \ \| \Delta^N u\| \  \|\Delta^N v\|$
    for $u, v \in \pi^\infty$.
    Since
    \[
    \int_{g \in G} f(g) \langle g u, v \rangle \, d g
    =
    \int_{g \in G}
    (\Delta^{2N} \ast  f \ast \Delta^{2N})(g)
    \langle g \Delta^{- 2N} u,  \Delta^{-2N} v \rangle \, d g,
    \]
    we then have
    \begin{equation}\label{eqn:basic-estimate-for-pi-f-with-f-in-F}
      |\langle \pi(f) u, v  \rangle|
      \ll
      \nu(f)
      \|\Delta^{-N} u\|
      \|\Delta^{-N} v\|.
    \end{equation}
    for some continuous norm
    $\nu(f) := \int_{g \in G} \Xi(g) \, |\Delta^{2N} \ast f \ast
    \Delta^{2N}(g)|$
    on $\mathcal{F}$.  By summing over $v$ in an orthonormal basis 
    and appealing to \S\ref{sec:arch-case-coarse-control-on-tempered-dual},
    we deduce that $\pi(f)u$ is represented by an element of the
    Hilbert space $\pi^0$.  By a similar argument applied to
    $\Delta^n \pi(f) u$ for each $n \in \mathbb{Z}_{\geq 0}$, we
    deduce that $\pi(f) u \in \pi^{\infty}$ and that the induced
    map $\pi(f) : \pi^{\infty} \rightarrow \pi^{\infty}$ is
    continuous.
  \item
    The trace norm of $\pi(f)$
    is bounded
    by $\sum_{u,v \in \mathcal{B}(\pi)} | \langle \pi(f) u, v
    \rangle |$,
    so we conclude by summing
    \eqref{eqn:basic-estimate-for-pi-f-with-f-in-F}
    and appealing to \S\ref{sec:arch-case-coarse-control-on-tempered-dual}.
  \item
    We argue similarly,
    using now
    also part (ii) of the lemma of
    \S\ref{sec:arch-case-coarse-control-on-tempered-dual}.
  \item We appeal to continuity and the density of $C_c^\infty(G)$ in $\mathcal{F}$.
  \end{enumerate}
\end{proof}

  \subsection{Proof of the lemma of
    \S\ref{sec:local-disintegration}}\label{sec:disintegration-details}
  By \S\ref{sec:matrix-coeff-bounds-for-tempered-reps} and
  \eqref{eqn:convergence-assumption-integral-Xi},
  the function $f$ defined in
  \eqref{eq:defn-of-f-for-relative-char-disint}
  belongs to
  the space $\mathcal{F}_H$ from \S\ref{sec:plancherel-formula-2}.
  Assertions (i) and (ii)
  thus follow from \S\ref{sec:plancherel-formula-2}.
  
  To establish (iii),
  we show first that
  the map
  $\Phi : \pi \otimes \overline{\pi} \rightarrow \mathcal{F}_H$
  given by $v_1 \otimes \overline{v_2} \mapsto f$ extends
  continuously to
  $\Phi : \Psi^{-\infty}(\pi) \rightarrow \mathcal{F}_H$.
  To
  that end, observe that for each $t_1, t_2 \in \mathfrak{U}$,
  we
  may write
  $t_1 \ast \Phi(v_1 \otimes \overline{v_2}) \ast t_2 = \Phi(t_1
  v_1 \otimes \overline{t_2^{\iota} v_2})$
  with $\iota$ the standard involution on $\mathfrak{U}$.  Thus
  for each continuous seminorm $\nu$ on $\mathcal{F}_H$ there
  are $C \geq 0$ and $N \in \mathbb{Z}_{\geq 0}$ so that
  $\nu(\Phi(v_1 \otimes \overline{v_2})) \leq C \|\Delta^N v_1\|
  \|\Delta^N v_2\|$.
  By summing over orthonormal bases and appealing
  to \S\ref{sec:arch-case-coarse-control-on-tempered-dual}, we deduce that
  $\nu(\Phi(T)) \ll \|\Delta^N T \Delta^N\|_2$ for all
  $T \in \pi \otimes \overline{\pi }$, where $\|.\|_2$ denotes
  the Hilbert--Schmidt norm.  By part (iv) of the theorem of
  \S\ref{sec:results-appl-kir}, we deduce that $\Phi$
  extends continuously to $\Psi^{-\infty}(\pi)$
  with the required uniformity.
  By the definition of the topology
  on $\Psi^{-\infty}(\pi)$,
  we may pass to $\Psi^{-N}(\pi)$
  for some fixed $N$.

  % \pn{new argument written here,
  %   replacing an earlier argument
  %   that was not quite correct
  % }
  The same argument gives,
  for $T \in \Psi^{-N}(\pi)$
  and with notation as in \eqref{eqn:formula-H-sigma-arch-case},
  that
  each of the quantities
  \[
    \sum_{u \in \mathcal{B}(\sigma)}
    \left\lvert 
    \int_{s \in H}
      \trace(s T)
      \langle u, s u  \rangle
    \right\rvert,
    \quad 
      \sum_{\substack{
          v \in \mathcal{B}(\pi) \\ u \in \mathcal{B}(\sigma) 
        }
      }
    \left\lvert 
      \int_{s \in H}
      \langle s T v, v \rangle
      \langle u, s u  \rangle
    \right\rvert
  \]
  is finite and depends continuosly
  upon $T$.
  The required identity
  \eqref{eqn:formula-H-sigma-arch-case}
  thus follows by continuous extension
  from the finite rank case.
  % \pn{
  %   OLD and inaccurate:
  %   The same argument gives the absolute convergence of the triple
  %   sum/integral on the RHS of
  %   \eqref{eqn:formula-H-sigma-arch-case}; the formula
  %   \eqref{eqn:formula-H-sigma-arch-case} then follows by swapping
  %   orders of summation/integration.
  % }
\end{appendices}

%
%

%
%
%
%
%
%
%
%
%
%
%
%
%
%

%
%
%
%
%
%
%
%
%
%
%
%
%
%
%
%
%
%
%
%

%
%
%
%
%
%
%
%
%
%
%
%
%
%
%
%
%
%
%
%
%
%
%
%
%
%
%
%
%
%
%
%
%
%
%
%
%
%
%
%
%
%
%
%
%
%
%
%
%
%
%
%
%
%
%
%
%
%
%
%
%
%
%
%
%
%

%

%
%
%
%
%
%
%

%

%
%
%
%
%
%

%
%
%
%
%

%
%
%
%
%
%

%
%
%
%
%
%
%
%
%
%
%
%
%
%
%
%
%
%

%
%
%
%
%
%
%
%
%
%
%
%
%
%
%
%
%
%
%
%
%
%
%
%
%
%
%
%
%
%
%
%
%
%
%
%
%
%

%
%
%
%
%
%
%
%
%
%
%
%
%
%
%
%
%
%
%
%
%
%
%
%
%
%
%
%
%
%
%
%
%
%
%
%
%
%
%
%
%
%
%
%
%
%
%
%
%
%
%
%
%
%
%
%
%
%
%
%
%
%
%
%
%
%
%
%
%
%
%
%
%
%
%
%

%
%

%

%

% \end{comment}

\iftoggle{cleanpart}
{
  \newpage
}
\part{Inverse
  branching}\label{sec:inv-branch}
\section{Overview}\label{sec:inv-branch-overview}
Let $(\mathbf{G},\mathbf{H})$ be a GGP pair over a local field
$F$ of characteristic zero.
Fix a tempered irreducible
representation $\pi$ of $G$.
More precisely,
we abuse notation
in what follows, as in \S\ref{sec:spher-char-disint},
by working implicitly with underlying spaces
of smooth vectors.

\index{$\hat{H}$, $\hat{H}_{\temp}$, $\hat{H}_{\temp}^{\pi}$}
Let $\hat{H}$ denote the unitary dual of $H$,
$\hat{H}_{\temp} \subseteq \hat{H}$
the tempered dual,
and $\hat{H}_{\temp}^{\pi}$ the $\pi$-distinguished
subset,
i.e.,
\[
\hat{H}_{\temp}^{\pi} := \left\{\sigma \in \hat{H}_{\temp} :
  \begin{gathered} \text{
      there is a nonzero}
    \\
    \text{$H$-equivariant
      map}
    \\
    \ell_\sigma : \pi \rightarrow \sigma
  \end{gathered}
\right\}.
\]

For each $\sigma \in \hat{H}_{\temp}$,
the discussion of \S\ref{sec:spher-char-disint}
gives us a map
$\mathcal{H}_\sigma : \pi \otimes \overline{\pi } \rightarrow
\mathbb{C}$.  It is known at least for $F$ non-archimedean
(see \cite[Theorem 5]{2015arXiv150601452B} in the unitary case
and \cite[Proposition 5.7]{Wald4} in the special orthogonal
case)
that $\mathcal{H}_{\sigma}$ 
is nonzero precisely when $\sigma \in \hat{H}_{\temp}^{\pi}$.

Recall that $\mathcal{H}_\sigma$ is expected to satisfy the
positivity condition \eqref{eq:H-sig-pos}, and that this
expectation is a theorem in the non-archimedean case \cite{SV}.

% and that this Recall the expected positivity
% condition \eqref{eq:H-sig-pos}

% \pn{added} We impose from here on the assumption that all
% hermitian forms $\mathcal{H}_\sigma$ under consideration satisfy
% the expected positivity condition \eqref{eq:H-sig-pos}.
% As noted earlier,
% .  Imposing this assumption simplifies our
% presentation; we explain at the end (see !!!)  why it is not
% ultimately necessary for the proofs of our main results.

% \pn{tweaked}

If $\sigma$ is tempered
and
$\mathcal{H}_\sigma$ is nonzero
and satisfies the expected positivity condition,
then we may write
\begin{equation}\label{eq:H-vs-l-sigma}
  \mathcal{H}_\sigma(v_1 \otimes \overline{v_2})
  =
  \sum_{u \in \mathcal{B}(\sigma)} \int _{s \in H}
  \langle s v_1, v_2 \rangle
  \langle u , s u \rangle
  =
  \langle \ell_\sigma(v_1), \ell_\sigma(v_2) \rangle
\end{equation}
for some $\ell_\sigma$ as above, uniquely defined up to a scalar
of magnitude one.  If $\mathcal{H}_\sigma$ vanishes, then we
take $\ell_\sigma = 0$.

  In the unexpected case that the positivity condition
\eqref{eq:H-sig-pos} is violated (necessarily for $F$
archimedean), we define $\ell_\sigma$ by requiring that
\eqref{eq:H-vs-l-sigma} hold up to some scalar of magnitude one.

We crudely extend the definition of $\ell_\sigma$ to
non-tempered $\sigma \in \hat{H}$ by choosing an $H$-equivariant
map $\ell_\sigma : \pi \rightarrow \sigma$, possibly zero but
nonzero if possible (and in that case, unique up to a scalar),
and requiring that \eqref{eq:H-vs-l-sigma} hold.  Thus
$\mathcal{H}_\sigma$, for non-tempered $\sigma$, is defined only
up to multiplication by a positive real.

We have defined an
association
\[
\pi \otimes \overline{\pi } \rightarrow
\{\text{functions } \hat{H} \rightarrow \mathbb{C} \}
\]
\[
T \mapsto [\sigma \mapsto \mathcal{H}_\sigma(T)].
\]
Let us pause to speak informally about the relevance of
this association
to our aims.  Recall, from \S\ref{sec:intro-op-calc},
that we
may think of self-adjoint elements
$\sum_j v_j \otimes \overline{v_j} \in \pi \otimes
\overline{\pi}$
as weighted families of vectors in $\pi$,
hence the above association
as an assignment
\[
\left\{
  \begin{gathered}
    \text{weighted families of} \\
    \text{vectors $v$ in $\pi$}
  \end{gathered}
\right\}
\rightarrow
\left\{
  \begin{gathered}
    \text{weighted families of} \\
    \text{representations $\sigma$ of $H$}
  \end{gathered}
\right\}.
\]
To implement the basic strategy of this paper
(cf. \S\ref{sec:rough-idea-proof},
\S\ref{sec:intro-local-issues}), we would like to know that we
can approximate any reasonable family of representations
in this way, while
retaining some control over the family of vectors achieving the
approximation.  This is the ``inverse branching problem''
alluded to in the title; by comparison, the classical
\emph{branching problem} concerns how a representation of a
group decomposes upon restriction to a subgroup, or perhaps how individual vectors decompose.

In the global setting (Part
\ref{part:appl-aver-gan}),
the pairs
$(\pi,\sigma)$ as above will arise as the local components
of a pair $(\Pi,\Sigma)$ of automorphic forms over a number field,
taken unramified outside some fixed set $R$ of places containing 
the archimedean places. 
We will single out an individual archimedean place
$\mathfrak{q} \in R$ as the ``interesting'' one, assume that the
relevant groups are compact at all other archimedean places, and
aim to study families with ``increasing
frequency at $\mathfrak{q}$'' and ``fixed level at
$\mathfrak{p}$'' for all
$\mathfrak{p} \in R - \{\mathfrak{q}\}$, with some fairly
flexible definition of ``fixed level.''  Motivated by this aim,
we consider here in Part \ref{sec:inv-branch}
the ``inverse branching problem'' indicated above in the following
aspects:
\begin{itemize}
\item For \emph{varying} families of representations $\sigma$,
  taken over a suitable scaling limit,
  and with $F$ an archimedean
  local field (\S\ref{sec:inverse-branching-dist-place}).
\item For \emph{fixed} families of $\sigma$
  in either of the following cases:
  \begin{itemize}
  \item (the trivial case in which) $H$ is
    compact (\S\ref{sec:compact-groups}).
  \item $F$ is non-archimedean (\S\ref{localBernstein}),
    after some general preliminaries
    (\S\ref{sec-padic-prelims}).
  \end{itemize}
\end{itemize}

An important subtlety is that the families of interest to us
will \emph{not} in general be ``microlocally separated'' from
the complementary series, e.g., via their infinitesimal
character.
We must nevertheless exclude the
latter from our final formula, due to the absence of a general
conjecture along the lines of Ichino--Ikeda in the non-tempered
case.
 These considerations motivate the estimate
\eqref{eqn:estimate-for-discarding-NT-stuff-in-arch-case}
and are responsible for the main difficulties of \S\ref{localBernstein}.

\section{The case of compact groups}\label{sec:compact-groups}
Suppose that $H$ is compact.  Then
$\hat{H}_{\temp}^{\pi}$ is a discrete
countable set.  The hermitian forms $\mathcal{H}_\sigma$
describe the canonical decomposition
\[
\pi|_{H} \cong \oplus_{\sigma \in \hat{H}_{\temp}^{\pi}} \sigma.
\]
Thus for any finitely-supported
function $k : \hat{H}_{\temp}^{\pi} \rightarrow \mathbb{C}$ there
exists $T \in \pi \otimes \overline{\pi }$ so that
\[
\mathcal{H}_\sigma(T) = k(\sigma)
\]
for all $\sigma \in \hat{H}_{\temp}^{\pi}$.
If $k$ is valued in the nonnegative reals, then we may take $T$
to be positive definite.

\section{The distinguished archimedean
  place}\label{sec:inverse-branching-dist-place}
We assume here that $F$ is archimedean.
 By restriction of scalars,
we may regard $\mathbf{G}$ and $\mathbf{H}$ as real reductive groups.

\subsection{Setup}\label{sec:inv-branch-arch-setup}
We allow the tempered irreducible representation $\pi$ of $G$
to vary with a
positive parameter $\h \rightarrow 0$.  We assume that $\pi$ has
a   limit  orbit
(see
\S\ref{sec:limit-coadjoint}) \[(\mathcal{O},\omega) = \lim_{\h
    \rightarrow 0} (\h \mathcal{O}_{\pi}, \omega_{\h
    \mathcal{O}_{\pi}}).\]
As in \S\ref{sec:gross-prasad-pairs-over-R-continuity-etc},
we write $\mathcal{O}_{\stab} \subseteq \mathcal{O}$
for the subset of $\mathbf{H}$-stable elements.
We recall that for
each
$\mu \in [\mathfrak{h}^\wedge] \cap
\image(\mathcal{O}_{\stab})$,
the preimage $\mathcal{O}(\mu)$ of $\{\mu\}$ in
$\mathcal{O}$ is an $H$-torsor. 
The map
$\mathcal{O}_{\stab}
\rightarrow
[\mathfrak{h}^\wedge] \cap
\image(\mathcal{O}_{\stab})$
is a principal $H$-bundle, with fibers $\mathcal{O}(\mu)$.

We assume given a Haar measure on $H$;
as explained in
\S\ref{sec:gross-prasad-pairs-over-R-continuity-etc},
this choice
defines measures
on $\mathfrak{h}$, $\mathfrak{h}^\wedge$,
$[\mathfrak{h}^\wedge]$,
and on the sets $\mathcal{O}(\mu)$ as above.

\subsection{Orbit-distinction}
\label{sec:asymp-distinction}
 Let $\sigma \in \hat{H}_{\temp}$.
Recall
(from \S\ref{sec:inv-branch-overview}) that $\sigma$ is
\emph{distinguished}
by $\pi$ 
if $\Hom_H(\pi,\sigma) \neq 0$.
We say that $\sigma$ is {\em
  orbit-distinguished} by $\pi$ if
\index{distinguished vs. orbit-distinguished representations}
$\mathcal{O}_{\pi,\sigma}$ --
the intersection of $\mathcal{O}_\pi$
with the preimage of $\mathcal{O}_\sigma$ --
is nonempty.

\begin{remark*}
  Our asymptotic expansion of relative characters (theorem
  \ref{thm:sph-char-main-export}) implies that if
  \begin{enumerate}[(i)]
  \item
    $\h \lambda_\sigma$ belongs to a fixed compact subset
    $E$ of
    $[\mathfrak{h}^\wedge] \cap \image(\mathcal{O}_{\stab})$,
    and if
  \item $\h > 0$ is small enough in terms of $E$,
  \end{enumerate}
  then
  orbit-distinction implies distinction.
  One expects also the converse
  implication, that distinction implies
  orbit-distinction under the stated hypotheses.  This would
  follow from the following:
  \begin{itemize}
  \item Strong multiplicity one for archimedean $L$-packets.
    This is
    addressed in unitary cases by the preprint
    \cite{2015arXiv150601452B}
    and in orthogonal cases by the recent preprint \cite{2020arXiv200913947L}.
    % , and likely provable in orthogonal
    % cases by existing techniques, but we are not aware of a
    % published reference.
  \item
    That distinction implies nonvanishing of the
    matrix coefficient integral, known in $p$-adic cases
    (cf. \S\ref{sec:inv-branch-overview})
    and in the unitary archimedean case \cite[Theorem 5]{2015arXiv150601452B},
    and likely provable in the orthogonal archimedean case.
  \end{itemize}
  In any event, orbit-distinction seems easier to check than
  distinction, so we are content to formulate our main
  results in terms of the former notion.
\end{remark*}

\subsection{Main result}\label{sec:main-result-inv-branch-arch}
Let
\begin{equation}\label{eqn:let-k-blah}
  k \in C_c^\infty([\mathfrak{h}^\wedge] \cap
\image(\mathcal{O}_{\stab})).
\end{equation}
For each $\h > 0$
we define a function
$k_{\h} : \hat{H} \rightarrow \mathbb{C}$
by setting
\[
  k_{\h}(\sigma) := k(\h \lambda_\sigma)
\]
if $\sigma$ is tempered and
$\mathcal{O}_{\pi,\sigma} \neq \emptyset$; otherwise,
we set
$k_{\h}(\sigma) := 0$.

We may find
\begin{equation}\label{eqn:U-V-1}
  \text{  precompact open subsets }
  U \subset [\mathfrak{h}^\wedge]
  \text{ and }
  V \subset \mathfrak{g}^\wedge,
\end{equation}
with
\begin{equation}\label{eqn:U-V-2}
  \text{$\overline{V}$ consisting
of $\mathbf{H}$-stable elements,}
\end{equation}
so that
\begin{equation}\label{eqn:containments-support-k-U-V}
  \supp(k) \subseteq \image(V),
\quad
\image(\overline{V}) \subseteq U,
\quad
\overline{U}
\subset \image(\mathcal{O}_{\stab}).
\end{equation}
Since $\mathcal{O}_{\stab}$ is a submanifold
of $\mathfrak{g}^\wedge$
and the map $\mathcal{O}_{\stab} \rightarrow [\mathfrak{h}^\wedge]$
is a principal $H$-bundle over its image,
we may readily find $a \in C_c^\infty(V)$
so that for each $\mu \in [\mathfrak{h}^\wedge]$,
\begin{equation}\label{eqn:surjectivity-of-integral-transform-1}
  \int_{\mathcal{O}(\mu)} a = k(\mu).
\end{equation}
From
\eqref{eqn:surjectivity-of-integral-transform-1}
and the asymptotic formulas
of \S\ref{sec:sph-chr-statement-result}
 it follows that
\begin{equation}\label{eqn:surjectivity-of-integral-transform-2}
\mathcal{H}_\sigma (\Opp_{\h}(a))
= k_{\h}(\sigma) + o_{\h \rightarrow 0}(1)
\text{ for all 
$\sigma \in \hat{H}_{\temp}$
with $\h \lambda_\sigma \in U$}.
\end{equation}
If $k$ is real-valued, then we may arrange that $a$ is
real-valued.
In the language of \S\ref{sec:intro-local-issues}
and \S\ref{sec:inv-branch-overview}, we have
achieved our goal of producing a weighted
family of vectors -- that obtained by writing
$\Opp_{\h}(a) = \sum_j v_j \otimes \overline{v_j}$ --
that picks off the weighted family of representations
described by $k_{\h}$.
We note
in passing also
that by \eqref{eqn:disintegration-along-H-for-individual-orbit},
we have
\begin{equation}\label{eqn:surjectivity-of-integral-transform-integrated}
  \int_{\mathcal{O}} a \, d \omega
  =
  \int_{\mu \in [\mathfrak{h}^\wedge]}
  \int_{\mathcal{O}(\mu)} a
  = \int_{[\mathfrak{h}^\wedge]} k,
\end{equation}
with integration over $[\mathfrak{h}^\wedge]$
defined by the normalized affine measure.

We aim now to elaborate upon this observation
in somewhat  technical ways
that will turn out to be convenient for our global applications.
It will be useful to work with ``positive-definite families,''
such as those attached to $\Opp_{\h}(a)^2$ for real-valued $a$,
and to bound the error
in \eqref{eqn:surjectivity-of-integral-transform-2}
in terms of another such family.

We will also need to say something
about non-tempered $\sigma$.
In that case,
we have only thus far (cf. \S\ref{sec:inv-branch-overview})
normalized $\mathcal{H}_\sigma$ up to a scalar.
It will be convenient
now to impose the following
more precise normalization, again motivated by global
considerations
(cf. \S\ref{sec:trunc-H-expn-overview}):
we suppose given an $\h$-dependent family of maps
$\mathcal{H}_{\sigma}$
that factors as a composition
\begin{equation}\label{eqn:NT-unif-cont}
  \Psi^{-\infty}(\pi) \rightarrow \Psi^{-\infty}(\sigma)
  \xrightarrow{\trace} \mathbb{C}
\end{equation}
with the first arrow $\h$-uniformly continuous.
In practice, this is a fairly weak requirement.
We note that the analogous continuity 
holds in the tempered case by the discussion of
\S\ref{sec:priori-estimates-rel-char}.
\begin{theorem}\label{thm:arch-inv-branch}
  Let $k, U, V$ be
  as in \eqref{eqn:let-k-blah} and \eqref{eqn:U-V-1},
  satisfying the assumptions \eqref{eqn:U-V-2}
  and \eqref{eqn:containments-support-k-U-V}.
  Assume that $k \geq 0$. 
  Then for each $\eps > 0$ and $N \in \mathbb{Z}_{\geq 0}$,
  there exist nonnegative $a, a_1, a_2, a_{\nt} \in
  C_c^\infty(V)$
  with the following properties:
  \begin{enumerate}[(i)]
  \item
    \label{item:a-and-a-plus-small-integrals}
    $\int_{\mathcal{O}} a_1^2 \, d \omega$
    is bounded by a constant depending only
    upon $k$ and $V$,
    while
    $\int_{\mathcal{O}} a_2^2 \, d \omega$
    and
    $\int_{\mathcal{O}} a_{\nt}^2 \, d \omega$
    are bounded by $\eps$.
  \item
    \label{item:approximated-k-main-term}
    $|\int_{[\mathfrak{h}^\wedge]} k - \int_{\mathcal{O}} a^2 \, d \omega |
    \leq \eps$.
  \item
    \label{item:approximated-k-siga-opp}
    Assume that $\h > 0$ is sufficiently small.
    \begin{itemize}
    \item
      Let
      $\sigma$ be a tempered irreducible unitary
      representation of $H$ for which $\h \lambda_\sigma \in
      U$.
      If 
      $\mathcal{O}_{\pi,
          \sigma}$ is nonempty,
        then
      \begin{equation}\label{eqn:estimate-for-approx-rel-char-arch-temp-0}
        \left\lvert k_{\h}(\sigma) 
        \right\rvert
        \leq
        |\mathcal{H}_{\sigma}(\Opp_{\h}(a_1)^2)|
      \end{equation}
      and
      \begin{equation}\label{eqn:estimate-for-approx-rel-char-arch-temp}
        \left\lvert k_{\h}(\sigma) -
          \mathcal{H}_{\sigma}(\Opp_{\h}(a)^2)
        \right\rvert
        \leq
        |\mathcal{H}_{\sigma}(\Opp_{\h}(a_2)^2)|.
      \end{equation}
    \item
      Let
      $\sigma$ be a non-tempered irreducible unitary
      representation of $H$ for which $\h \lambda_\sigma \in
      U$.
      Then
      \begin{equation}\label{eqn:estimate-for-discarding-NT-stuff-in-arch-case}
        \mathcal{H}_{\sigma}(\Opp_{\h}(a)^2)
        = 
        \mathcal{H}_{\sigma}(\Opp_{\h}(a_{\nt})^2)
        + \O(\h^N).
      \end{equation}
      The implied constant
      may depend upon $(N,k,a,\eps)$,
      but not upon $(\pi,\sigma,\h)$.
    \end{itemize}
  \end{enumerate}
\end{theorem}
\begin{remark*}\label{rmk:positivity-assumption-annoyance-abs-valz}
 We note that, since
$ \Opp_{\h}(a_j)^2$ is positive-definite,
the absolute values on the RHS of
\eqref{eqn:estimate-for-approx-rel-char-arch-temp-0} and
\eqref{eqn:estimate-for-approx-rel-char-arch-temp} should not be
necessary; in any event, they will disappear from our analysis
when we pass to the global setting, in which the product of
local hermitian forms as above is manifestly positive.
\end{remark*}

\begin{proof}
  The informal idea for
  \eqref{eqn:estimate-for-approx-rel-char-arch-temp-0} and
  \eqref{eqn:estimate-for-approx-rel-char-arch-temp} is as in the
  arguments leading to
  \eqref{eqn:surjectivity-of-integral-transform-1} and
  \eqref{eqn:surjectivity-of-integral-transform-2}: for instance,
  to get \eqref{eqn:estimate-for-approx-rel-char-arch-temp}, we
  can just choose $a$ so that
  $\int_{\mathcal{O}(\mu)} a^2 \, d \omega \approx k(\mu)$, with
  the difference thus majorized by
  $\int_{\mathcal{O}(\mu)} a_1^2 \, d \omega$ for some small
  $a_1$; the asymptotics for $\mathcal{H}_{\sigma}(\dotsb)$ then
  give the required estimates.
  
  The informal idea for
  \eqref{eqn:estimate-for-discarding-NT-stuff-in-arch-case} is
  that the infinitesimal characters of non-tempered
  representations are close to irregular elements, which form a set of measure zero.
  We may thus construct $a_{\nt}$ from $a$ by shrinking its support
  to be concentrated near the inverse images of irregular
  elements.

  Turning to details,
  choose $a_1$
  as indicated, depending only upon $k$ and $V$, so that
  $\int_{\mathcal{O}(\mu)} a_1^2 \geq k(\mu) + 1$
  for $\mu \in \supp(k)$.
  Fix an open subset $U_0 \subset [\mathfrak{h}^\wedge]$
  with $\supp(k) \subseteq U_0$ and $\overline{U_0} \subseteq \image(V)$.
  Choose $\eps_1 > 0$ small enough in terms of $k$ and $\eps$,
  then choose
  $a_2$  as indicated
  so that
  $\int_{\mathcal{O}(\mu)} a_2^2 \geq 2 \eps_1$
  for $\mu \in U_0$
  and
  $\int_{\mathcal{O}} a_2^2 \, d \omega \leq \eps$.
  Choose $a$ as indicated so that
  $| k (\mu) - \int_{\mathcal{O}(\mu)} a^2 |
  \leq \eps_1$
  for $\mu \in U$.
  The set
  \[
  W
    :=
    \{\xi \in \mathfrak{h}^\wedge \cap \image(\overline{V}) :
    [\xi] \notin  [\mathfrak{h}^\wedge]_{\reg}
    \}
  \]
  is compact
  and
  has measure zero.
  We may thus find $a_{\nt}$ as indicated,
  with $\int_{\mathcal{O}} a_{\nt}^2 \, d \omega \leq \eps$,
  so that
  $a = a_{\nt}$ in a small neighborhood of $W$.

  Assertion \eqref{item:a-and-a-plus-small-integrals}
  is clear by construction.
  Assertion \eqref{item:approximated-k-main-term}
  follows as in \eqref{eqn:surjectivity-of-integral-transform-integrated}
  if $\eps_1$ is sufficiently small.
  Turning to assertion \eqref{item:approximated-k-siga-opp},
  let $\sigma \in \hat{H}$ with $\h \lambda_\sigma \in U$.
  
  Suppose first that $\sigma$ is tempered and
  $\mathcal{O}_{\pi,\sigma} \neq \emptyset$.
   Then
  $k_{\h}(\sigma) = k(\h \lambda_\sigma)$,
  while the asymptotic formulas of \S\ref{sec:sph-chr-statement-result}
  give
  \[
  \mathcal{H}_{\sigma}(\Opp_{\h}(a)^2)
  =
  \int_{
    \mathcal{O}(\h \lambda_\sigma)
  }
  a^2
  + o_{\h \rightarrow 0}(1),
  \]
  and similarly for $a_1,a_2$.
  Thus \eqref{eqn:estimate-for-approx-rel-char-arch-temp-0}
  and \eqref{eqn:estimate-for-approx-rel-char-arch-temp}
  hold for $\h$ small enough in terms of $\eps_1$.

  Suppose next that $\sigma$ is non-tempered.
  By the composition formula
  \eqref{eqn:comp-with-remainder-J-diff-subspaces},
  we may write
  \[\Opp_{\h}(a)^2
    = \Opp_{\h}(a_{\nt})^2
    + \Opp_{\h}(c) + \mathcal{E},
  \]
  where:
  \begin{itemize}
  \item The $\h$-dependent element $c \in C_c^\infty(V)$
    is bounded with respect to $\h$
    and
    vanishes identically on
    a small neighborhood of $W$.
    Since non-tempered representations
    have infinitesimal characters
    close to irregular elements
    (cf. \S\ref{sec:crit-for-temperedness}),
    it follows that $c = 0$
    on
    $\{\xi \in \mathfrak{h}^\wedge : \dist([\xi], \h \lambda_\sigma) \leq
    \eps_2\}$
    for some small but fixed $\eps_2 > 0$.
    By \S\ref{sec:some-decay-2},
    we deduce that $\mathcal{H}_\sigma(\Opp_{\h}(c)) \ll \h^N$.
  \item
    $\mathcal{E} \in \h^{N'} \Psi^{-N'}$,
    where $N' \rightarrow \infty$ as $J \rightarrow \infty$,
    so that,
    by
    the assumed uniform continuity of \eqref{eqn:NT-unif-cont},
    we have $\mathcal{H}_{\sigma}(\mathcal{E}) \ll \h^N$.
  \end{itemize}
  The required estimate
  \eqref{eqn:estimate-for-discarding-NT-stuff-in-arch-case}
  follows.
\end{proof}

\subsection{Auxiliary estimates relevant
  for Weyl's law}\label{sec:arch-weyl-counting}
Recall that our main result concerns
the average of an $L$-function over a family.
We record here, for completeness, a technical estimate
relevant for computing the \emph{cardinality} of that family (cf. \S\ref{analytic-conductors} for its application).

Let $\mathcal{H}$ denote the Hecke algebra
of smooth compactly-supported complex measures on
$H$.
Since we have fixed a Haar measure $d h$ on $H$,
we may identify $\mathcal{H}$ with $C_c^\infty(H)$;
in particular, we may define the evaluation
$f(1)$ at the identity element $1 \in H$
of any $f \in \mathcal{H}$.

\begin{lemma*}
  Let $k, U$ be as in \S\ref{sec:main-result-inv-branch-arch}, with $k \geq 0$.
  Fix $\eps > 0$ and $N \in \mathbb{Z}_{\geq 0}$,
  and let $\h > 0$ be sufficiently small.
  There are positive-definite elements
  $f, f_1 \in \mathcal{H}$,
  supported on $1 + o_{h \rightarrow 0}(1)$,
  so that
  \begin{equation}\label{eqn:arch-weyl-integral-k-f-1}
    |\int_{[\mathfrak{h}^\wedge]} k -  \h^d f(1)| \leq \eps
  \end{equation}
  and
  \begin{equation}\label{eqn:f-1-bound-arch-weyl-easy}
    f_1(1) \leq \eps \h^{-d},
  \end{equation}
  and for each $\sigma \in \hat{H}$,
  \begin{equation}\label{eq:required-claim-for-arch-weyl-law}
    |k_{\h}(\sigma) - \chi_\sigma(f)| \leq \chi_\sigma(f_1)
    + \O(\h^N \langle \h \lambda_\sigma  \rangle^{-N}),
  \end{equation}
  where $\chi_{\sigma} : \mathcal{H} \rightarrow \mathbb{C}$ denotes the character.
\end{lemma*}
Before giving the proof,
we record the basic idea.
If we argue formally -- ignoring convergence, truncations,
etc. --
then
for each $a \in S^{-\infty}(\mathfrak{g}^\wedge)$,
the function $f : H \rightarrow \mathbb{C}$ defined by
\[
f(s) := \trace(\pi(s^{-1}) \Opp_{\h}(a)),
\]
where
$\Opp_{\h}(a) := \Opp_{\h}(a:\pi)$,
satisfies
\[
\chi_\sigma(f)
= \sum_{\substack{
    v \in \mathcal{B}(\pi) \\
    u \in \mathcal{B}(\sigma) 
  }}
\int_{s \in H}
\langle s \Opp_{\h}(a) v, v \rangle
\langle u,  s u  \rangle
=
\mathcal{H}_{\sigma}(\Opp_{\h}(a)),
\]
\[
f(1) = \trace(\Opp_{\h}(a)).
\]
To make $f$ positive definite, we can argue instead with
$\Opp_{\h}(a)^2$.  The lemma should thus be plausible in view of
the analogous passage from
\eqref{eqn:surjectivity-of-integral-transform-1},
\eqref{eqn:surjectivity-of-integral-transform-2} and
\eqref{eqn:surjectivity-of-integral-transform-integrated} to
Theorem \ref{thm:arch-inv-branch}.  The subtlety is that $f$ as
defined above is not compactly-supported, so does not belong to
$\mathcal{H}$ as we have defined it.  To make matters worse, the
integral defining $\chi_\sigma(f)$ need not converge when
$\sigma$ is non-tempered.  To get around these issues, we truncate $f$,
taking care to do so in a manner that preserves
positive definiteness.  For $\sigma$ tempered and stable
relative to $\pi$, we have seen already that the main
contribution to the integral over $s \in H$ defining
$\mathcal{H}_{\sigma}(\Opp_{\h}(a))$ comes from $s$ fairly
small, so the truncation has negligible impact in that case.  We
can use the operator calculus and the same argument as in
Theorem \ref{thm:arch-inv-branch} to control the contributions
from the remaining $\sigma$.

\begin{proof}
  For $b \in S^{-\infty}(\mathfrak{h}^\wedge)$, let us denote by
  $\widetilde{\Opp_{\h}}(b) \in \mathcal{H}$ the element
  implicit, for a unitary representation $\sigma$ of $H$, in the
  definition of
  $\Opp_{\h}(b : \sigma) = \sigma(\widetilde{\Opp_{\h}}(b))$.

  Choose $V$ as in \S\ref{sec:main-result-inv-branch-arch},
  choose $\eps_1 > 0$ small enough in terms of $k$ and $\eps$,
  assume that $\h > 0$ is small enough in terms of $\eps_1$,
  and choose $a, a_2 \in C_c^\infty(V)$
  as in the proof of theorem \ref{thm:arch-inv-branch}.
  Define
  $f_0 \in \mathcal{H}$
  by the formula
  \[
    f_0 :=
    \int_{s_1,s_2 \in H}
    \chi(s_1) \chi(s_2)
    \left\langle
      \pi(s_1) \Opp_{\h}(a),
      \pi(s_2) \Opp_{\h}(a)
    \right\rangle
    \delta_{s_2^{-1} s_1},
  \]
  where:
  \begin{itemize}
  \item  $\chi \in C_c^\infty(H)$ denotes a suitable normalized cutoff:
    supported near the identity, nonnegative-valued,
    invariant under inversion, constant near the identity,
    and satisfying $\int_H |\chi|^2 = 1$.
  \item $\Opp_{\h}(a) := \Opp_{\h}(a:\pi)$.
  \item $\delta_{s_2^{-1} s_1 }$ is the $\delta$-mass;
    its integral as above will define a smooth measure.
  \item We employ the Hilbert--Schmidt inner product on $\End(\pi)$.
  \end{itemize}
  Then $f_0$ is positive definite.
  Choose an open $W \subset \mathfrak{h}^\wedge$
  so that
  $W \supseteq \image(\overline{V})$ and
  $U \supseteq \image(\overline{W})$.
  Choose $b \in C_c^\infty(\mathfrak{h}^\wedge)$
  supported in the preimage of $U$ and with $b = 1$ on $W$.
  We note,
  by \eqref{eqn:comp-with-remainder-J-diff-subspaces}, that
  \[
  \Opp_{\h}(b:\pi) \Opp_{\h}(a)
  \equiv  \Opp_{\h}(a) \mod{\h^\infty \Psi^{-\infty}(\pi)}.
  \]
  Set
  \[
    f := \widetilde{\Opp_{\h}}(b) \, f_0 \,  \widetilde{\Opp_{\h}}(b).
  \]
  (Here and below the only relevant product structure on
  $\mathcal{H}$ is given by convolution.)
  In the same way that $f$ was defined
  in terms of $a$,
  let $f_+$ be defined
  in terms of $a_2$.
  Arguing as in the proof of theorem \ref{thm:arch-inv-branch}, let $b_{\nt}$ be obtained from
  $b$ by smoothly truncating to a sufficiently small neighborhood of those
  elements in the support of $b$ whose image in
  $[\mathfrak{h}^\wedge]$ is irregular,
  and set $f_{\nt} := 
  \widetilde{\Opp_{\h}}(b_{\nt}) \, f_0 \, \widetilde{\Opp_{\h}}(b_{\nt})$.
  Finally, set
  \[
    f_1 :=
    2 \eps_1 f_+
    +
    f_{\nt}.
  \]
  Then $f, f_1$ are positive definite.
  By slowly shrinking the support of $\chi$,
  we may arrange that they are supported on $1 + o_{\h
    \rightarrow 0}(1)$,
  or indeed on $1 + \O(\h^{\eta})$ for any fixed $\eta \in (0,1)$;
  this has no effect on the arguments to follow.

  We now verify that these constructions
  lead to the required estimates.
  We start with \eqref{eqn:arch-weyl-integral-k-f-1}.
  By unwinding the definitions, we see that
  \[
  f_0(1)
  =
  (\int_H \chi^2)
  \trace(\Opp_{\h}(a)^2)
  = \h^{-d}(\int_{\mathcal{O}} a^2 \, d \omega + o_{\h \rightarrow 0}(1)).
  \]
  As in the proof of theorem \ref{thm:arch-inv-branch},
  it follows that $f_0(1) = \h^{-d} (\int_{[\mathfrak{h}^\wedge]} k +
  o_{\eps_1 \rightarrow 0}(1))$ for small enough $\h$,
  which gives the modified form of
  \eqref{eqn:arch-weyl-integral-k-f-1}
  obtained by replacing $f$ with $f_0$.
  To obtain the required assertion concerning $f$,
  we calculate first that
  \[
  f(1)
  = \int_{s \in H}
  \trace(T_s \Opp_{\h}(a)^2 T_s),
  \]
  where
  \begin{equation}\label{eqn:defn-T-s-for-silly-proof}
      T_s := \int_{t \in H}
  \chi(s t)
  \widetilde{\Opp}_{\h}(b)(t) \pi(t).
  \end{equation}
  Note that $T_s = 0$ unless
  $s$ is small,
  and the integrand in \eqref{eqn:defn-T-s-for-silly-proof}
  vanishes
  unless $t$ is small.
  By trivially estimating
  the $L^1$-norm of $\widetilde{\Opp}_{\h}(b)$,
  we see that
  the operator norm of $T_s$ is $\O(1)$
  and the trace norm of $\Opp_{\h}(a)^2$ is $\O(\h^{-d})$.
  Thus to compute $f(1)$ to accuracy $o(\h^{-d})$,
  it suffices to do so
  after replacing $T_s$ by any modification
  $T_s'$
  differing in operator norm by $o(1)$.
  To that end, let us pull the integral
  \eqref{eqn:defn-T-s-for-silly-proof}
  back to the Lie algebra,
  writing $t = e^y$ with $y \in \mathfrak{h}$.
  The integrand is concentrated on $|y| = \O(\h)$,
  so we may truncate it to $|y| \leq \h^{1-\eta}$ for some fixed
  $\eta >0$
  and then Taylor expand $\chi(s e^y) = \chi(s) +
  \O(\h^{1-\eta})$.
  The modification
  \[
  T_s' := \chi(s) \int_{t \in H}  \widetilde{\Opp}_{\h}(b)(t)
  \pi(t)
  = \chi(s) \Opp_{\h}(b:\pi)
  \]
  is thus acceptable for our purposes,
  and we obtain
  \begin{equation}\label{eqn:f-evaluated-at-1-expanded}
      f(1) = (\int_{H} \chi^2)
      \trace(\Opp_{\h}(b) \Opp_{\h}(a)^2 \Opp_{\h}(b)) + o_{\h \rightarrow 0}(\h^{-d}),
    \end{equation}
    say.
  We appeal now to the composition formula
  \eqref{eqn:comp-with-remainder-J-diff-subspaces}
  to replace
  $\Opp_{\h}(b) \Opp_{\h}(a)^2 \Opp_{\h}(b)$
  with
  $\Opp_{\h}(a)^2$,
  and argue as before.
  This completes the verification of \eqref{eqn:arch-weyl-integral-k-f-1}.

  The same arguments applied to $f_1$
  lead to the estimate \eqref{eqn:f-1-bound-arch-weyl-easy}.

  We turn finally to
  \eqref{eq:required-claim-for-arch-weyl-law}.
  Thus $\sigma$ be a unitary representation of $H$.
  We consider first the case that
  $\h \lambda_\sigma \notin U$.
  Then $k_{\h}(\sigma) = 0$.  On the
  other hand,
  the results of \S\ref{sec:some-decay-2}
  --
  applied with $(G,\pi)$ playing the duplicate role of ``$(H,\sigma)$'',
  and using the continuity of \eqref{eq:Psi-bounds-trace-norms}
  as an \emph{a priori} estimate
  to clean up remainders --
  give that
  the trace norm of $\Opp_{\h}(b:\sigma)$ is
  $\O(\h^N \langle \h \lambda_{\sigma} \rangle^{-N})$ for any
  fixed $N$.  Since the operator norms of $\sigma(f_0)$,
  $\Opp_{\h}(b:\sigma)$ and $\Opp_{\h}(b_{\nt}:\sigma)$ are
  readily bounded by $\h^{-\O(1)}$, the claim
  \eqref{eq:required-claim-for-arch-weyl-law} follows in this
  case.

  It remains to consider the case that $\h \lambda_\sigma \in U$.
  If $\sigma$ is tempered,
  then we may verify  -- using arguments
  similar to those leading to \eqref{eqn:f-evaluated-at-1-expanded} --
  that
  $\chi_\sigma(f)
  = \mathcal{H}_{\sigma}(\Opp_{\h}(a)^2)
  +
  \O(\h)$,
  and similarly for $\chi_\sigma(f_1)$.
  If $\sigma$ is
  non-tempered, then we see as before that
  $\Opp_{\h}(b - b_{\nt}:\sigma)$ has trace norm
  $\O(\h^\infty)$.  In either case, we may conclude as in the
  proof of theorem \ref{thm:arch-inv-branch}.
\end{proof}

\section{Preliminaries on representations of $p$-adic groups}
\label{sec-padic-prelims}
In this section,
we recall how
tempered representations of a $p$-adic reductive group
fit into families indexed by quotients of certain tori,
and explain the relationship of this picture to the Bernstein center. 

The considerations of this section apply to
any reductive group $\mathbf{H}$ over a non-archimedean local
field $F$ of characteristic zero; the group $\mathbf{G}$ plays
no role.

\subsection{Standard parabolic and Levi subgroups}
\label{sec-padic-1-1}
We fix a minimal parabolic subgroup $P_0$ of $H$ and Levi
subgroup $M_0 < P_0$.
A \emph{standard} parabolic subgroup is
one that contains $P_0$.  A \emph{standard} Levi subgroup $M$ is
one that contains $M_0$ and arises a Levi component of a
standard parabolic $P = M N$; thus $P = M P_0$.  Each parabolic
or Levi subgroup is conjugate to a standard one, so there is
little loss of generality in restricting to the latter.

For each standard Levi subgroup $M$ , we have an
induction functor ${i}_M^G = \Ind_P^G$ from smooth
representations of $M$ to smooth representations of $G$,
normalized to take unitary representations to unitary representations;
here and henceforth ``unitary'' and ``unitarizable'' are used interchangeably.

What matters most for our purposes is the set of
irreducible subquotients
of $i_M^G \tau$.
This set is independent of the $H$-conjugacy class
of $(M,\tau)$,
and  
makes sense
for any Levi subgroup $M$, not necessarily standard.
If moreover
$\tau$ is unitary, then so is $i_M^G \tau$,
hence subquotients of $i_M^G \tau$
are the same as submodules.

\subsection{Good compact open subgroups}
\label{sec-padic-1-3} \index{good open subgroup}
Let $J$ be a compact open subgroup of $H$.
Recall that $J$ admits an \emph{Iwahori factorization}
with respect to a parabolic subgroup $P$ if there is a Levi decomposition $P = M N$,
with associated opposite parabolic $P^{-} = M N^{-}$,
so that
\[
  J = J_{N_-} J_M J_{N}
  \text{ with }
  J_{N_-} := J \cap N_-,
  J_M := J \cap M,
  J_N := J \cap N.
\]
We may and shall
fix a maximal compact
subgroup $K$ of $G$
for which $K P_0 = H$
(see, e.g.,  \cite[Cor 9.12]{MR2602034}).

Let us call a compact open subgroup $J$ of $H$  \emph{good}
(relative to the choice of $K$ and $P_0$)
if
\begin{enumerate}[(i)]
\item $J$ is a normal subgroup of $K$, and
\item $J$ admits an Iwahori factorization
  with respect to each standard parabolic subgroup.
\end{enumerate}
\begin{lemma*}
  For suitable choice of $K$ as above, there are good compact open subgroups
  $(J_n)_{n \geq 0}$ of $H$
  that form a neighborhood basis of the identity.
\end{lemma*}
\begin{proof}
  We apply Proposition 4.2 of \cite{MR1371680} with $x$ a special vertex.
\end{proof}

\subsection{Invariant vectors and induction}
\label{sec-padic-1-4}
The following result is standard:
\begin{lemma}
  Let $J$ be a good compact open subgroup of $H$,
  let $P = M N$ be a parabolic subgroup of $H$,
  and $\tau$ a
  supercuspidal representation of $M$.
  The following are equivalent:
  \begin{itemize}
  \item[(a)] $\tau$ admits a nonzero $J_M$-fixed vector;
  \item[(b1)] Some subquotient
    of ${i}_P^H \tau$ admits a nonzero $J$-fixed vector.
  \item[(b2)] ${i}_P^H \tau$ admits a nonzero $J$-fixed vector. 
  \item[(c)] Every subquotient of ${i}_P^H \tau$ admits a nonzero $J$-fixed vector.
  \end{itemize}
\end{lemma}

\proof

The dimension of the space of $J$-fixed
vectors may be expressed
as the trace of an averaging operator,
and so may be computed in terms of a composition series.
Conditions (b1) and (b2) are thus equivalent;
henceforth we refer to them together as (b).
Obviously (c) implies (b1).

We now show that (a) and (b) are equivalent.
We may assume that $P$ and $M$ are standard.
The $J$-fixed vectors in
${i}_M^H \tau $
are described by pairs
$(x,v)$,
where $x \in P \backslash H / J$
and $v \in \tau$
satisfy
\begin{equation}\label{eq:desc-J-fixed-induct}
  \delta_P^{1/2} \tau(g) v = v
  \text{ for all } g \in P \cap x J x^{-1},
\end{equation}
where as usual $\tau$ acts
via the projection $P \twoheadrightarrow M$.
Since $H = P K$,
we may assume that $x \in K$;
then $x J x^{-1} = J$,
so \eqref{eq:desc-J-fixed-induct}
just says that $v$ is fixed by
the projection to $M$ of $P \cap J = J_M J_N$,
i.e.,
by $J_M$.
Thus (a) and (b) are equivalent.

  It remains to see that (a) implies (c).
Suppose thus that $\tau^{J_M} \neq 0$.
Let $\pi$ be any subquotient of $i_P^H \tau$.
In fact, we may assume that $\pi$ is a submodule, by a standard argument:
there exists a parabolic $Q = NU$
such that the Jacquet module $\pi_U$ is supercuspidal, up to twist;
by the computations of Bernstein and Zelevinsky \cite[2.13]{BZ}
there is an element of $g$ carrying $(M, \tau)$
to a constituent of $(N, \pi_U)$.  We may therefore
suppose that $N=M$ and that $\pi_U$ contains a conjugate $w\tau$  of $\tau$
by the normalizer of $M$; thus $\pi \hookrightarrow i_{MU}^G (w \tau)$. 
Since $w$ has a representative belonging to $K$, it is equivalent
whether $(w\tau)$ or $\tau$ has a $J_M$-fixed vector.

 By Frobenius reciprocity,
the inclusion $\pi \rightarrow {i}_P^G \tau$ gives 
rise to a nonzero map
$\pi_N \rightarrow \tau$.  Since $\tau$ is irreducible,
this map is surjective.
Taking $J_M$-invariants
gives a surjective map $(\pi_N)^{J_M} \rightarrow \tau^{J_M}$.
But a basic theorem
\cite[Theorem 3.3.3]{Casselman}
asserts that the map
$\pi^J \longrightarrow (\pi_N )^{J_M}$
is surjective.
Thus our assumption $\tau^{J_M} \neq 0$  implies that
$\pi^J \neq 0$, as required.
\qed

\subsection{Classifications}

\subsubsection{Terminology}\label{sec:padic-class-terminology}
Let $\pi$ be an irreducible representation of $H$.
Recall that $\pi$ is \emph{tempered} if it
is unitary and its matrix coefficients lie in $L^{2+\eps}$
modulo the center.
Recall that $\pi$ is
\emph{square-integrable} if its central character is unitary and
its matrix-coefficients are square-integrable modulo the center;
in particular, $\pi$ is unitary and tempered.
Recall that
$\pi$ is \emph{supercuspidal} if its matrix coefficients are
compactly-supported modulo the center; then $\pi$ is unitary if
and only if its central character is unitary,
in which case it is square-integrable.
In particular, any supercuspidal representation
has an unramified twist
which is unitary.

\subsubsection{Bernstein--Zelevinsky;
  infinitesimal
  characters}\label{sec:bernstein-zelevinsky-classification}
By the results of \cite[\S2]{BZ}:
\begin{lemma*}
  For each
  irreducible representation $\pi$ of $H$
  there is a unique $H$-conjugacy class $[(M,\tau)]$
  of pairs $(M,\tau)$, where
  $M$ is a Levi subgroup
  and $\sigma$ is a supercuspidal representation
  of $M$,
  so that $\pi$ is a subquotient of $i_M^G \tau$.
\end{lemma*}
\index{infinitesimal character ($p$-adic group)} 
The \emph{infinitesimal character} of an irreducible
representation $\pi$ of $H$ is the class
$\lambda_\pi := [(M,\tau)]$ arising in the lemma.
By an \emph{infinitesimal character} for $H$
we will mean any such class $[(M,\tau)]$.

\subsubsection{Langlands, from square-integrable to tempered}
\label{sec:langlands-classification-tempered}
By \cite[Prop III.4.1]{MR1989693}:  
\begin{lemma*}
  For each
  tempered irreducible representation $\pi$ of $H$
  there is a unique $H$-conjugacy class $[(M,\sigma)]$
  of pairs $(M,\sigma)$, where
  $M$ is a Levi subgroup
  and $\sigma$ is a square-integrable representation
  of $M$,
  so that $\pi$ is a subquotient (equivalently, submodule) of $i_M^G \sigma$.
\end{lemma*}

\subsection{Bernstein Components}
\label{sec-padic-1-2-2}
In this and the following subsections
we recall
some facts from \cite{MR771671}
(cf. \cite[\S2]{MR874050} for a summary).

For a Levi subgroup $M$ of $H$,
let $\mathfrak{X}_M$ denote the group of unramified characters
of $M$,
i.e., homomorphisms $\chi : M \rightarrow \mathbb{C}^\times$
that are trivial on the subgroup $M^0$
on which all algebraic characters
have valuation zero.
The group $\mathfrak{X}_M$ is a complex torus, i.e., isomorphic 
to $(\mathbb{C}^\times)^r$,
while the subgroup $\mathfrak{X}_M^0$ of unitary
characters identifies with the compact subtorus $(\mathbb{C}^{(1)})^r$.

For each supercuspidal representation $\tau$ of $M$,
the set
\begin{equation*}
  \Theta =
  \{
  [(M,\tau \otimes \chi)] :
  \chi \in \mathfrak{X}_M
  \}
\end{equation*}
is a \emph{Bernstein component}, or simply a \emph{component}
for short,
of the set of infinitesimal characters.
Each component  may be identified with the quotient of
$\mathfrak{X}_M$
by a finite group (see for instance the discussion in \cite[\S3.3.1]{Haines}).

\subsection{Bernstein center}
\label{sec-padic-1-2-3}
Varying $(M,\tau)$, the set of infinitesimal characters 
identifies with a disjoint union of finite quotients of complex tori,
giving it the structure of a complex algebraic variety,
typically with infinitely many components.
The \emph{Bernstein center}
is the algebra $\mathfrak{Z}(H)$ 
of regular functions
on the variety of infinitesimal characters for $H$;
it is the direct
product over the set of components $\Theta$
of the algebra of regular functions on $\Theta$,
and we have
\[
\text{ spectrum of $\mathfrak{Z}(H)$ }
=
\sqcup \Theta,
\]
the union taken over all Bernstein components.
By \cite{MR771671} (cf. \cite[\S2.2]{MR874050}),
we have:
  \begin{lemma*}
    There is a natural action of $\mathfrak{Z}(H)$ on the category
  of representations of $H$: for each $z \in \mathfrak{Z}(H)$
  and each representation $\pi$ of $H$, there is an associated
  $H$-equivariant endomorphism $z : \pi \rightarrow \pi$,
  such that for each
  $H$-equivariant morphism of representations
  $j : \pi \rightarrow \pi '$, we have $z \circ j = j \circ z$.
  If $\pi$ is irreducible, then $z : \pi \rightarrow \pi$ is
  scalar multiplication by $z(\lambda_\pi)$.
\end{lemma*}
To apply this in practice, let $\mathcal{H}$ denote the Hecke \index{$\mathcal{H}$}
algebra of locally constant compactly-supported
measures on $H$, under convolution, regarded as a
representation of $H$ under
the action defined by left translation.  For a vector $v$
in a representation $\pi$ of $H$, the action map
$\mathcal{H} \rightarrow \pi$ given by $f \mapsto f \ast v$ is
then $H$-equivariant, so for each $z \in \mathfrak{Z}(H)$, we
have $(z \cdot f) \ast v = z \cdot (f \ast v)$.
Let $J$ be a compact open subgroup of $H$ that fixes $v$,
and take $f := e_J$, the corresponding averaging operator.
Then $z \cdot v = h_z \ast v$, where
$h_z := z \cdot e_J$ is a central element of the
bi-$J$-invariant subalgebra \index{$\mathcal{H}_J$}
$\mathcal{H}_J \subseteq \mathcal{H}$.
In particular,
if $\pi$ is irreducible,
then $h_z$ acts on $\pi^J$
by the scalar $z(\lambda_\pi)$.

\subsection{Components arising
from Langlands classification of the tempered dual}\label{sec:comp-aris-from-langl}

\subsubsection{}
By \S\ref{sec:langlands-classification-tempered},
there is a natural map
\begin{equation}\label{eq:langalnds-class-p}
  \mathfrak{l} : \hat{H}_{\temp}
  \rightarrow
  \left\{[(M,\tau)] :
    \begin{gathered}
      \text{$\tau$ is a square-integrable representation}
      \\
      \text{of the Levi subgroup $M$ of $H$}
    \end{gathered}
  \right\}
\end{equation}
assigning to
$\sigma$ the $H$-conjugacy class $[(M,\tau)]$
of pairs
as indicated
for which
$\sigma \hookrightarrow i_M^H \tau$.

As in \S\ref{sec-padic-1-2-2},
we may fix $M$ and vary $\tau$ in a family
of unramified unitary twists $\{\tau \otimes \chi\}_{\chi \in
  \mathfrak{X}_M^0}$
to write the RHS of \eqref{eq:langalnds-class-p} as a disjoint union
of subsets $\mathcal{D}$
parametrized by the compact tori $\mathfrak{X}_M^0$
and identified with quotients $\mathfrak{X}_M^0/\Gamma$
for some finite subgroups $\Gamma$ 
of $\mathfrak{X}_M^0 \rtimes N(M)/M$.
To differentiate
from the Bernstein components (\S\ref{sec-padic-1-2-2}),
we refer to these subsets $\mathcal{D}$
as \emph{$\mathfrak{l}$-components};
we are unaware of any standard terminology.
We rewrite
\eqref{eq:langalnds-class-p}
as
\[
\mathfrak{l} : \hat{H}_{\temp} \rightarrow \sqcup \mathcal{D}.
\]

\subsubsection{}
We note in passing, for the sake of orientation, that by generic
irreducibility
 \cite[Theorem 6.6.1]{Casselman}, each
$\mathfrak{l}$-component $\mathcal{D}$ contains a nonempty
Zariski open subset $U$ so that the map
$\mathfrak{l}^{-1}(U) \rightarrow U$ is injective.

\subsubsection{}\label{sec:p-comp-vs-Bern-comp}
If $\mathfrak{l}(\sigma) = [(M,\tau)]$,
then
the infinitesimal character
$\lambda_\sigma$
is the image of $\lambda_\tau$
in the space of infinitesimal characters for $H$ under the evident map
from infinitesimal characters for $M$.  

In particular,
\[
  \mathfrak{l}(\sigma_1) = \mathfrak{l}(\sigma_2)
  \implies
  \lambda_{\sigma_1} = \lambda_{\sigma_2}.
\]
For each $\mathfrak{l}$-component $\mathcal{D}$
there
is thus a (unique) Bernstein
component $\Theta$
for which $\sigma \mapsto \lambda_\sigma$
descends to
a map
\[
\mathcal{D} \rightarrow \Theta.
\]
This map is continuous; indeed, as noted in
\S\ref{sec-padic-1-2-2}, both spaces are locally (for the
analytic topology) identified with character tori for Levi
sugroups, and locally the map is given by a homomorphism of
these character tori arising from an inclusion of Levi
subgroups.
  
Two distinct $\mathfrak{l}$-components
$\mathcal{D}_1, \mathcal{D}_2$ may map to the same Bernstein
component $\Theta$, and their images may overlap.  For instance,
there do exist (for general $H$)
non-isomorphic square-integrable representations
of $H$ having the same infinitesimal character.

\subsection{Finiteness}
The Bernstein components
or
$\mathfrak{l}$-components form countable sets.  More
precisely,
it follows
from \S\ref{sec-padic-1-4}
that for any good compact open subgroup $J$ of $H$, a
representation $\sigma$ of $H$ with
$\mathfrak{l}(\sigma) = [(M,\tau)]$ has a nonzero $J$-fixed
vector if and only if the representation $\tau$ of the Levi $M$,
taking $M$ standard without loss of generality, has a nonzero
$J_M$-fixed vector.  By \cite[\S2.3]{MR874050} and the
Plancherel formula (or see \cite[Thm VIII.1.2]{MR1989693}), only
finitely many Bernstein components or $\mathfrak{l}$-components
contain some $[(M,\tau)]$ with this property.

\section{The case of auxiliary $p$-adic places}
\label{localBernstein}
We consider now the non-archimedean
case of the setup of \S\ref{sec:inv-branch-overview},
thus $(\mathbf{G},\mathbf{H})$ is a GGP pair over a non-archimedean local field
$F$ of characteristic zero,
and $\pi$ is a smooth tempered irreducible
unitary representation of $G$.
We introduce the abbreviation \index{$\Omega$, the $\pi$-distinguished tempered dual}
\[
\Omega := \hat{H}_{\temp}^{\pi}
\]
for the $\pi$-distinguished tempered dual of $H$. 
The notation
and terminology of \S\ref{sec-padic-prelims}
will be employed freely.

\subsection{The structure of $\Omega$} \label{Omegatorus}
Recall from \S\ref{sec:comp-aris-from-langl} the map
$\mathfrak{l} : \hat{H}_{\temp} \rightarrow \sqcup \mathcal{D}$
arising from the Langlands classification.
\begin{lemma}[Strong multiplicity one]
  The induced map
  \[
  \mathfrak{l} : \Omega \rightarrow \sqcup \mathcal{D}
  \]
  is injective.
\end{lemma}
\begin{proof}
  Indeed, strong multiplicity one
  \cite{MR3058848, MR3371496}
  implies
  that each $L$-packet of tempered representations of $H$
  contains at most one $\pi$-distinguished element.
  Since each fiber of $\mathfrak{l}$
  is contained in a single $L$-packet,
  the conclusion follows.

  We note that the results of \cite{MR3058848, MR3371496} are
  formulated as conditional on certain expected properties of
  $L$-packets for classical groups.
  It is not straightforward
  for us to extract these properties from the literature, so we
  observe also that the required conclusion -- multiplicity one for the full
  induction of a 
  square-integrable
  representation -- can be verified directly (see  \cite[\S 14.2,14.3]{BPU}).
\end{proof}

 For $t = [(M,\tau)] \in \mathfrak{l}(\Omega)$,
write $\sigma_t := \mathfrak{l}^{-1}(t) \in \Omega$.
Then $i_M^H \tau$ decomposes as a finite direct sum of tempered
irreducible representations $\sigma$ of $H$;
one of these summands is the given $\pi$-distinguished
representation $\sigma_t$,
while strong multiplicity one
implies that the remaining summands are not $\pi$-distinguished.
It follows that for any $v_1,v_2 \in \pi$,
\begin{equation}\label{eq:formula-for-herm-form-p-components}
  \mathcal{H}_{\sigma_t}(v_1 \otimes \overline{v_2})
  =
  \sum _{u \in \mathcal{B}(i_M^H \tau)}
  \int_{h \in H}
  \langle h v_1, v_2 \rangle \langle u , h u  \rangle.
\end{equation}
Here the sum over $u$ is really a finite sum, since $v_1, v_2$
are smooth (under $H$).  The RHS of
\eqref{eq:formula-for-herm-form-p-components} is manifestly
continuous as  $t$ varies within a given $\mathfrak{l}$-component
$\mathcal{D}$, and so defines a continuously-varying
family of hermitian forms on $\pi \otimes \overline{\pi}$.
\begin{lemma}
  The image of $\Omega$ under $\mathfrak{l}$ is
  a union of $\mathfrak{l}$-components.
\end{lemma}
\begin{proof}
  We need only verify that
  $\mathfrak{l}(\Omega)$ is both  closed and open:
  it is closed, by the upper semicontinuity
  of multiplicity \cite[Lemma D.1]{MR2930996}, and open,
  by the continuity of $t \mapsto \mathcal{H}_{\sigma_t}$
  noted above.
\end{proof}
Thus $\mathfrak{l}$ identifies $\Omega$ with a disjoint union of
(typically infinitely many) finite quotients of compact tori,
hence equips $\Omega$ with a natural topology
with respect to which the hermitian forms
$\mathcal{H}_{\sigma}$ vary continuously.
We may also
speak of the space $C_c(\Omega)$ of compactly-supported
continuous functions.
We henceforth refer to $\Omega$ and $\mathfrak{l}(\Omega)$
interchangeably.

It seems conceivable to us that the topology just defined on
$\Omega$ coincides with the topology induced by the Fell
topology, but we have not verified this.
In any event, the more explicit topology just defined
is the relevant one for our purposes.

The map
\begin{equation}\label{eq:inf-char-on-Omega}
  \Omega \longrightarrow
  \sqcup \Theta
  =
  \mbox{spectrum of the Bernstein center}
\end{equation}
is continuous, for the topology just defined on $\Omega$.
This continuity follows from the fact the map from
$\mathfrak{l}$-components
to Bernstein components, discussed in
\S\ref{sec:p-comp-vs-Bern-comp}, is continuous.

\subsection{Main
  results}\label{sec:padic-inverse-branching-main-results}
Let $J$ be a good (see \S\ref{sec-padic-1-3}) compact open
subgroup of $H$.  We denote by \index{$\hat{H}^J, \hat{H}^J_{\temp}, \Omega^J$}
$\hat{H}^J, \hat{H}^J_{\temp}$ and $\Omega^J$ the spaces
corresponding to representations of $H$ having a nonzero
$J$-fixed vector.  The space $\Omega^J$ is a finite union of
$\mathfrak{l}$-components, and as
$J$ traverses a
neighborhood
basis of the identity, we
have $\Omega = \cup_J \Omega^J$ and
$C_c(\Omega) = \cup_J C(\Omega^J)$.

The following notion is motivated  by our
global applications, and related to the class of functions that
appear in Sauvageot's ``principe de densit{\'e}''
\cite{MR1468833}.\footnote{
  There were some points in the original paper \cite{MR1468833}
that we
do not understand: specifically
the usage of Lemma 2.1 on page 181. The Lemma 2.1 assumes that the algebra in question separates points.
This is related to the distinction between
$\mathfrak{l}$-components
and Bernstein components. 
}
We henceforth adopt the
convention
that
$k(\sigma) := 0$ for
$k \in C_c(\Omega)$ and $\sigma \in \hat{H} - \Omega$.

\begin{definition}\label{defn:allowable}
   We say that an $\mathfrak{l}$-component
  $\mathcal{D}$ of $\Omega$ is {\em allowable} if
  there is a good compact open subgroup $J$ as above, with $\mathcal{D} \subseteq
  \Omega^J$, 
  so that every
  nonnegative $k \in C(\mathcal{D}) \subseteq C_c(\Omega)$
  may be approximated in the following senses:
 
  \begin{enumerate}[(i)]
  \item \label{item:sauv-rel-2}
    For each $\eps > 0$ there are
    positive-definite
    $T, T_+ \in \pi^J \otimes \overline{\pi ^J}$,
    with $\trace(T_+) \leq \eps$,
    so that for each $\sigma \in \hat{H}$,
    \begin{equation}\label{eq:allowable-main-defn-new}
      \left\lvert
        k(\sigma)
        -
        \mathcal{H}_\sigma(T)
      \right\rvert
      \leq
      \mathcal{H}_{\sigma}(T_+)
    \end{equation}
    (We emphasize that, per
    the general conventions
    of \S\ref{sec:inv-branch},
    all tensors
    considered here
    such as $T, T_+$
    are \emph{smooth}.)
  \item  \label{item:sauv-2}
    For each $\eps > 0$
    there are $\phi, \phi _+ \in \mathcal{H}_J$,
    with $\phi_+$ positive definite and $\phi _+(1) \leq \eps$,
    so that for each $\sigma \in \hat{H}$,
    \begin{equation}\label{eq:allowable-main-defn-new-2}
      \left\lvert
        k(\sigma)
        -
        \trace(\sigma(\phi))
      \right\rvert
      \leq
      \trace(\sigma(\phi _+)).
    \end{equation}
  \end{enumerate}
  
  We say that a function $k: \Omega  \rightarrow \C$ is allowable if 
  its support lies in a finite union of allowable
  $\mathfrak{l}$-components.
\end{definition}
  
  \begin{remark*}
    Allowability (applied with $\eps = 1$, say)
    implies that there is a positive-definite $T_+
    \in \pi^J \otimes \overline{\pi ^J}$
    so that
    \begin{equation} \label{item:sauv-rel-1} |k(\sigma)| \leq \mathcal{H}_{\sigma}(T_+) \end{equation}
    for all $\sigma \in \hat{H}$.
    Similarly     there is a positive-definite element
    $\phi _+$
    of the bi-$J$-invariant Hecke algebra
    $\mathcal{H}_J \cong C_c^\infty(J \backslash H / J)$
    so that
\begin{equation} \label{item:sauv-1} |k(\sigma)| \leq \trace(\sigma(\phi_+))\end{equation}
    for all $\sigma \in \hat{H}$.
The upper bounds \eqref{item:sauv-rel-1} and \eqref{item:sauv-1}
will be useful in applications
involving products
of several groups such as $G$.

Also note
that
the LHS
of
either 
\eqref{eq:allowable-main-defn-new}
or
\eqref{eq:allowable-main-defn-new-2}
vanishes identically
unless $\sigma$ belongs to $\hat{H}^J$,
whose Plancherel measure is finite.
By the Plancherel formula
(cf. \S\ref{sec:spher-char-disint}, \S\ref{sec:plancherel-formula})
it follows
that for $T, \phi$ as in the conclusion,
we have
\begin{equation}\label{eq:remaerk-after-aux-padic-thm}
    \trace(T)
  + o_{\eps \rightarrow 0}(1)
  = 
\int_{\hat{H}_{\temp}} k
=
\phi(1)
  + o_{\eps \rightarrow 0}(1).
\end{equation}
  
  \end{remark*}
  
\subsection{Cuspidal type components are allowable}
\label{sec:analys-regul-comp}
There is a a class of Bernstein components
which are  particularly straightforward to analyze:

We say that an $\mathfrak{l}$-component
is {\em of cuspidal type} if the inducing data
is not merely discrete series, but
supercuspidal. For example, 
the $\mathfrak{l}$-component of unramified principal series is of cuspidal type. 

\begin{theorem}\label{thm:consequences-of-allowability}
  Any $\mathfrak{l}$-component of cuspidal type is allowable.
\end{theorem}

It seems reasonable to expect that all
$\mathfrak{l}$-components are
allowable. 

\subsection{Outline of the proof.} \label{outlineoftheproofsec}
In this subsection, we
give an overview of
the argument. Details
of the steps are given in the following subsections.  The
subtlety is in controlling the contribution of the non-tempered
spectrum, which meets the tempered spectrum in different ways
(for example, the complementary series can approach both the
Steinberg representation and a tempered principal series).

In the argument that follows, integrals and volumes are always
computed with respect to Plancherel measure.

We must verify conditions (i) and (ii) of Definition
\ref{defn:allowable} of
\S\ref{sec:padic-inverse-branching-main-results}.
We first prove (i) and then deduce (ii).

For the proof of (i),
it suffices to obtain the
required approximation
\eqref{eq:allowable-main-defn-new}
for each
continuous function
\[
  k: \mathcal{D} \rightarrow
  [0,1],
\]
where $\mathcal{D}$
is
an $\mathfrak{l}$-component of cuspidal type.
Recall that we extend $k$ by zero to $\hat{H}$.
\begin{itemize}
\item[Step 1:]   Let $\Theta$ denote the union of
  Bernstein components with $J$-fixed vectors, where $J$ is chosen
  small enough so that $\mathcal{D}$ maps into $\Theta$.
  The complex variety $\Theta$ has a natural
  real form (see \S\ref{sec:real-form-bernstein} for details);
  we write $\mathbb{R}[\Theta]$ for the set of regular functions
  on that real form,
  and regard it as a subring of the Bernstein center
  (\S\ref{sec-padic-1-2-3}).
  Given
  $z \in \mathbb{R}[\Theta]$
  and
  $T = \sum_i u_i \otimes \overline{v_i} \in \pi^J \otimes
  \overline{\pi^J}$,
  we set
  \[
    z T z :=
    \sum_i z u_i \otimes \overline{z v_i} \in \pi^J \otimes
    \overline{\pi^J}.
  \]
  Then, for each $\sigma \in \hat{H}$, we have
  $z(\lambda_\sigma) \in \mathbb{R}$ (\S\ref{sec-padic-1-2-3}),
  and
  \begin{equation}\label{eq:H-sigma-z-equivari}
    \mathcal{H}_{\sigma}(z T z)
    =
    z(\lambda_\sigma)^2 \mathcal{H}_\sigma(T).
  \end{equation}
  In particular,
  $\mathcal{H}_{\sigma}(z T z) = 0$ unless
  $\lambda_\sigma  \in \Theta$.

\item[Step 2:] 
  There are finitely many $\pi$-distinguished $\mathfrak{l}$-components besides $\mathcal{D}$ that map into $\Theta$; let 
  $\mathcal{D}' $ be their union.
  Since
  the $\mathfrak{l}$-component $\mathcal{D}$ is of cuspidal
  type,
  we have:
  \begin{equation} \label{DD'}
    \mbox{$\mathcal{D} \rightarrow \Theta$ is injective and its image is disjoint
      from the image of $\mathcal{D}'$.}
  \end{equation} 

 \item[Step 3:]    We may find a
  positive-definite $T_0$ such that $\mathcal{H}_{\sigma}(T_0) \geq 1$ for $\sigma \in \mathcal{D} \cup \mathcal{D}'$
  (see  the lemma of \S \ref{sec:uniform-distinction-lemma}).

\item[Step 4:]  Fix $\eps <1$. 

  Then, by an application of Stone-Weierstrass
  and the assumed disjointness (see \S\ref{SSW} for details),
  we may find  $z_0 \in \mathbb{R}[\Theta]$
  which approximates
  $\sqrt{k(\sigma)/\mathcal{H}_{\sigma}(T_0)}$
  on $\mathcal{D}$
  and approximates zero on $\mathcal{D}'$: 
  \begin{equation} \label{kapprox} | k(\sigma) -
    z_0(\lambda_\sigma)^2 \mathcal{H}_{\sigma}(T_0)| < \eps 
    \ \qquad  (\sigma \in \mathcal{D} \cup \mathcal{D}'). \end{equation}
  
  Set $T := z_0 T_0 z_0$.  Then \eqref{eq:H-sigma-z-equivari}
  and \eqref{kapprox} 
  imply that
  \begin{equation} \label{Conq} |k(\sigma) - \mathcal{H}_{\sigma}(T)| < \eps \ \ (\sigma \in \hat{H}^J_{\temp}),\end{equation}
  since the left hand side vanishes for $\sigma \in \hat{H}^J_{\temp} - (\mathcal{D} \cup \mathcal{D}').$
Using our assumptions on $k$ and $\eps$,  we get
\begin{equation} \label{Conq2}  \int_{\mathcal{D}'}
  |\mathcal{H}_{\sigma}(T)| 
  \ll \eps, 
  \quad \sup_{\mathcal{D}} |\mathcal{H}_{\sigma}(T)| \ll 1. \end{equation}
% \begin{equation} \label{Conq2}  \int_{\mathcal{D}'} |\mathcal{H}_{\sigma}(T)|  < \eps \cdot \vol(\mathcal{D}'), 
%   \quad \sup_{\mathcal{D}} |\mathcal{H}_{\sigma}(T)| <
%   2. \end{equation}
Here we adopt the convention that implied
constants may depend upon
$(\mathcal{D},J,k)$
(hence possibly upon $(\mathcal{D} ', \Theta, T_0)$),
but must be independent of $\eps$.
We have also abused notation
  and written integrals over $\mathcal{D}, \mathcal{D}'$
  where the domain of the integral is, more precisely, $\mathfrak{l}^{-1} \mathcal{D}$
  and $\mathfrak{l}^{-1} \mathcal{D}'$. 
  
\item[Step 5:]

  Now take $ T_+ = z T z + \eps T_0$,
  where
  $z \in \mathbb{R}[\Theta]$
   has the following properties, again
  achieved by Stone-Weierstrass:
  \begin{itemize}
  \item  $z(\lambda_\sigma) \geq 1$
    for each non-tempered $\sigma \in \hat{H}$.
  \item $z(\lambda_\sigma) \in [-2,2]$ for each $\sigma \in \hat{H}$. 
  \item  $z(\lambda_\sigma)$ is small on average over $\sigma
    \in \mathcal{D}$:
the 
    integral of $z(\lambda_{\sigma})^2$, taken with respect to Plancherel measure,
    is bounded by $\eps$. 
  \end{itemize}
  See \S\ref{SSW} for the construction. 
  Note that 
  the first and third requirements ``pull in opposite directions'' because 
the intersection
  \begin{equation} \label{ntemp} \overline{\mbox{infinitesimal characters of nontempered representations
      }} \cap \lambda(\mathcal{D})
  \end{equation}
  need not be empty; we can nevertheless simultaneously satisfy
  them because the measure of the set \eqref{ntemp} is zero.
  Note also that we are again using our assumption that
  $\mathcal{D}$ is of cuspidal type -- the same intersection but
  with $\mathcal{D}$ replaced by $\mathcal{D}'$ does not
  necessarily have measure zero.  For example, if
  $G = \PGL_2(F)$ and $\mathcal{D}$ is the
  $\mathfrak{l}$-component of cuspidal type consisting of the
  unitary principal series representations, then the other
  $\mathfrak{l}$-component in $\mathcal{D} '$
  is the singleton consisting of the
  Steinberg representation, whose infinitesimal character is a
  limit of infinitesimal characters of complementary series
  representations.

  We claim now that
  \begin{equation}\label{eq:claim-k-sigma-H-sigma-T-bla}
    |k(\sigma) - \mathcal{H}_{\sigma}(T)| \leq
    \mathcal{H}_{\sigma}(T_+)
    \text{
      for all $\sigma \in \hat{H}$}
  \end{equation}
  and
   \begin{equation}\label{eq:claim-trace-T-plus}
     \trace(T_+)
     \ll \eps.
    % \leq \eps \left( \tr(T_0)  + 2 + 4  \cdot
    %   \vol(\mathcal{D}')\right).
  \end{equation}
  Assuming
  the claim,
  we may replace $\eps$ by a smaller constant as needed to obtain
  the desired pair $(T, T_+)$.
  We
  verify \eqref{eq:claim-k-sigma-H-sigma-T-bla} separately in the following cases:
  \begin{itemize}
  \item for $\sigma \in \mathcal{D} \cup \mathcal{D} '$,
    we have $|k(\sigma) - \mathcal{H}_\sigma(T)| < \eps
    \mathcal{H}_\sigma(T_0)
    \leq \mathcal{H}_\sigma(T_+)$;
 
  \item for $\sigma \in \hat{H}_{\temp}$
    but $\sigma \notin \mathcal{D} \cup \mathcal{D}'$, both sides are zero;
  \item for non-tempered $\sigma$,
    we have
    $\mathcal{H}_\sigma(T) \geq 0$
    and
    $z(\lambda_\sigma) \geq 1$
    and $k(\sigma) = 0$,
    hence
    \[
      |k(\sigma) - \mathcal{H}_\sigma(T)|
      = \mathcal{H}_\sigma(T) 
      \leq \mathcal{H}_\sigma(z T z)
      \leq 
      \mathcal{H}_\sigma(T_+).
    \]
  \end{itemize}
  We verify \eqref{eq:claim-trace-T-plus}
  using
  that $\trace(T_+)
  = \eps \trace(T_0)
  + \int_{\sigma} z(\lambda_\sigma)^2 \mathcal{H}_\sigma(T)$,
  and the following estimates,
  deduced using \eqref{Conq2}  and the construction of $z$:
  \[
    \int_{\sigma \in \mathcal{D}}
    z(\lambda_\sigma)^2 \mathcal{H}_\sigma(T)
    \leq
    (\max_{\sigma \in \mathcal{D}} \mathcal{H}_\sigma(T))
    \int_{\sigma \in \mathcal{D}}
    z(\lambda_\sigma)^2
    \ll \eps,
    % \leq 2 \eps,
  \]
  and 
  \[
    \int_{\sigma \in \mathcal{D}'}
    z(\lambda_\sigma)^2 \mathcal{H}_\sigma(T)
    \ll \eps.
 %    \leq
 % 4 \eps \vol(\mathcal{D} ').
  \]
  
%  For later use, we observe that in fact $\mathcal{H}_{\sigma}(T_+)$ is bounded {\em pointwise} 
%  by a constant multiple of $\eps$ for $\sigma \in \mathcal{D}'$:
%  \begin{equation} \label{Hsbound}
%    \sup_{\sigma \in \mathcal{D}'} \mathcal{H}_{\sigma}(T_+) \leq \eps \cdot (1 + \sup_{\sigma} \mathcal{H}_{\sigma}(T_0)).
%  \end{equation}
%  The constant in parentheses depends only on $\mathcal{D}$, not on the function $k$.

\item[Step 6:] We now prove (ii). One could
  argue in parallel with the prior argument; however,
  for (ii), one encounters issues of reducibility that
  do not occur in (i) -- in the context of (i) such issues are effectively eliminated
  by strong multiplicity one.  We have therefore found it more convenient, although perhaps slightly unnatural, 
 to deduce (ii) from (i). 

  Given $k$, take $T \in \pi^J \otimes \overline{\pi^J}$
  as in (i). 
  Define the bi-$J$-invariant function $\phi : H \rightarrow
  \mathbb{C}$
  by $\phi_0(h) := \trace(\pi(h) T)$.
  %$ and similarly define $\phi_{0,+}
  %= \trace(\pi(h) T_+)$; note that $\phi_0(h) = \overline{\phi_0(h^{-1})}$. 
  Write $e_J \in \mathcal{H}_J$ for the normalized
  characteristic function of $J$.
  The dimension of $\sigma^J$ is uniformly bounded for $\sigma \in \hat{H}$;
  let $M$ be an upper bound for this dimension. 

  In a formal sense,
  we have
  $\mathcal{H}_\sigma(T) = \trace(\sigma(\phi_0))$
  (see  \S\ref{sec:spher-char-disint}).
  The basic idea of the argument is to use this formal identity
  to construct functions from the $T, T_+$ previously constructed.
  This is not entirely straightforward because
  $\phi_0$ is not compactly supported.

  However, at least if $\sigma$ is tempered,
  the integral defining the operator $\sigma(\phi_0)$
  converges
  (see \S\ref{sec:plancherel-formula-2}),
  and defines a nonnegative operator because $T
  \geq 0$.  (To verify the non-negativity, it suffices to treat the case when $T$
  is of rank one, and then it comes from the positivity of the matrix coefficient integral, see \cite{SV}). 
 
  Take
  \[ \phi_1 = \mbox{ (bi-$J$-invariant) truncation of $\phi_0$} +  \mbox{small multiple of } e_J,\]
  with a large enough
  truncation;
  by truncating symmetrically, we arrange 
  that
  $\phi_1(x^{-1}) = \overline{\phi_1(x)}$ for all $x \in H$.
  Using the absolute convergence
  of the matrix coefficient integral defining
  $\mathcal{H}_{\sigma}$,
  we obtain the following:
  \begin{itemize}
  \item $|\trace(\sigma(\phi_1))- \mathcal{H}_{\sigma}(T)| \leq \eps$ 
  for all $\sigma \in  \hat{H}^J_{\temp}$.
  \item
  $\sigma(\phi_1)$ is positive definite for each such $\sigma$. 
  \item $\sigma(\phi_1)$ is zero if $\sigma$ does not have a $J$-fixed vector, i.e., $\lambda_\sigma \notin \Theta$. 
  \end{itemize}

  For $\sigma \in \hat{H}^{J}_{\temp}$,
  we have 
  \begin{align} \label{kboundRRR}
    |\trace(\sigma(\phi_1))  - k(\sigma)| &\leq
                                            |\trace(\sigma(\phi_1)) - \mathcal{H}_{\sigma}(T)| + |\mathcal{H}_{\sigma}(T) - k(\sigma)|    \\
                                          &\leq
                                            2 \eps ,
  \end{align}
  using \eqref{Conq} at the second step.

  We now construct a function $\phi_3$
  which controls $\phi_1$ 
  on the non-tempered set, in a sense to be made precise.
  Fix $\eps' > 0$, 
  and choose $z$ as in {\em Step 5} 
  but now with $\int_{\sigma \in \mathcal{D}}
  z(\lambda_\sigma)^2 < \eps'$.
  Put
  $ \phi_3 = \eps^{-1}(z \cdot \phi_1) \ast (z \cdot  \phi_1)^{\vee}$,
  where as usual $f^{\vee}(g) = \overline{f(g^{-1})}$. 
  Then $\phi_3$ is positive definite and symmetric. 
%   \av{Is it clear that $z
%    \cdot \phi_1$ inherits the symmetry property?}
%  \pn{
%    This seems to me to follow from the following observations
%    (not quite as clearly as I'd like,
%    but let's see):
%    \begin{itemize}
%    \item In defining $z \cdot \phi_1$,
%      it doesn't matter whether we regard the Hecke algebra
%      as a representation of $H$ under left translation or under right
%      translation.
%      The simplest way I can think of to see this is
%      to decompose
%      $\phi_1$ in $L^2(H)$ as a Plancherel integral
%      with respect to $H \times H$,
%      and then arguing as in \S\ref{sec-padic-1-2-3}
%      using the characterizing properties
%      of the action of the Bernstein 
%      center.
%      Perhaps there's some more direct way.
%    \item The anti-automorphism $\iota$ of the Hecke algebra
%      obtained by replacing $\phi(h)$ with $\phi(h^{-1})$
%      is $H$-equivariant,
%      where $H$ acts by left translation on the source
%      and right translation on the target.
%    \item
%      Using \S\ref{sec-padic-1-2-3}, we deduce that $z$ commutes with $\iota$.
%    \item
%      By similar arguments, and using that $z$ is real,
%      we see that $z$ commutes with replacing $\phi(h)$ by
%      $\overline{\phi(h)}$.
%    \item The symmetry property of $z \cdot \phi_1$ should then follow.
%    \end{itemize}
%    Should we explain this?
%    Is there a simpler way?
%    Perhaps Bernstein's original paper
%    contains such foundational stuff?
%    (I've never looked at it; never managed to find it online.)
%  }
  We claim that for any non-tempered $\sigma \in \hat{H}$,
  \begin{equation}\label{eqn:nt-est-phi1-phi3-M-eps}
    |\trace(\sigma(\phi_1))| \leq \trace(\sigma(\phi_3)) + M \eps.
  \end{equation}
  Indeed, choose an orthonormal basis
  for $\sigma^J$
  consisting of eigenvectors for $\sigma(\phi_1)$.
  The basis has cardinality at most $M$.
  Since $z \cdot \phi_1$ acts self-adjointly on $\sigma^J$,
  its action coincides with that of $(z \cdot \phi_1)^{\vee}$. 
  Thus if $v$ belongs to the chosen basis and has  eigenvalue
  $c$ under $\sigma(\phi_1)$, 
  then
  $\sigma(\phi_3) v = c ' v$,
  where
  $c ' := z(\lambda_\sigma)^2 \eps^{-1} |c|^2 \geq
  \eps^{-1} |c|^2$.
  Thus $c' \geq |c|$ whenever $|c| \geq \eps$.
  We obtain
  \eqref{eqn:nt-est-phi1-phi3-M-eps}
  by
  summing over $v$,
  considering separately
  the cases $|c| \geq \eps$ and $|c| \leq \eps$.
  % For each such eigenvector,
  % with eigenvalue
  % $\lambda \in \mathbb{R}$, the eigenvalue of $\sigma(\phi_3)$ on $v$
  % is given by $z(\lambda_{\sigma})^2 \eps^{-1} \lambda^2$; this is visibly
  % $\geq \lambda$ if $\lambda \geq \eps$.   Summing over $\lambda$
  % gives the above inequality. 

  Moreover, 
  \begin{align}  \phi_3(1) &= \int_{\sigma  \in \hat{H}^J_{\temp}}  \eps^{-1}  z(\lambda_{\sigma})^2
  \tr(\sigma(\phi_1 * \phi_1)) \label{eq:phi-3-1-via-planch}
   \end{align}
  We bound the integrand on the RHS as follows:
  \begin{itemize}
  \item  Suppose that $\sigma \in \Omega \cap \mathfrak{l}^{-1}( \mathcal{D})$, i.e., $\sigma$ has $\mathfrak{l}$-parameter in $\mathcal{D}$
    {\em and} is distinguished. 
    The trace of $\sigma(\phi_1 * \phi_1)$
    is bounded by $M \| \phi_1\|_{L^1}^2$.
    The contribution of such $\sigma$ to the integral above is therefore  $\ll   \frac{\eps'}{\eps} M \cdot \|\phi_1\|_{L^1}^2$. 
    
  \item   If $\sigma$ is
    not as just described, then $k(\sigma) = 0$,
    so  \eqref{kboundRRR} implies that
    $\trace(\sigma(\phi_1)) \leq 2 \eps$. 
    But  positivity of $\sigma(\phi_1)$ 
    implies that $  \tr(\sigma(\phi_1 * \phi_1))  \leq   \trace(\sigma(\phi_1))^2$.
    The integrand on the RHS of \eqref{eq:phi-3-1-via-planch}
    is thus bounded by $16 \eps$. 
  \end{itemize}
  
  Taken together, we get
\begin{equation} \label{phi3identitybound} \phi_3(1) \ll  \frac{\eps'}{\eps} M \|\phi_1\|_{L^1}^2 +  \eps.\end{equation}
  % \underbrace{ 2\eps \cdot \mathrm{meas}(\mathcal{D}') + 2\phi_2(1)}_{\ll \eps}.
   Now $\phi_1$ depends on $\eps$. 
However, choosing first $\eps$ and then reducing $\eps'$ as appropriate, 
equations \eqref{eqn:nt-est-phi1-phi3-M-eps} and \eqref{phi3identitybound}
show that
 $(\phi_1,   \phi_3 +   2M \eps e_J)$ give the desired pair of functions (up to a final rescaling of $\eps$).

\end{itemize}

\subsection{Proofs for steps 3,4,5}
\label{sec:prelims-allowability}

\subsubsection{Uniform
  distinction}\label{sec:uniform-distinction-lemma}
\begin{lemma*}
  Let $J$ be a good compact open subgroup of $H$.
  There is a positive-definite (smooth) tensor
  $T \in \pi^J \otimes \overline{\pi^J}$
  so that $\mathcal{H}_{\sigma}(T) \geq 1$
  for all $\sigma \in \Omega^J$.
\end{lemma*}
\begin{proof}
  Each $\sigma \in \Omega^J$
  is $\pi$-distinguished
  and contains nonzero $J$-invariant vectors,
  so we may find
  $x \in \pi^J$ with $\mathcal{H}_\sigma(x \otimes
  \overline{x}) 
  \geq 2$.
  By continuity -- using the formula
  \eqref{eq:formula-for-herm-form-p-components}
  for $\mathcal{H}_\sigma$ in terms
  of matrix coefficients --
  we then have
  $\mathcal{H}_{\sigma'}(x \otimes
  \overline{x}) 
  \geq 1$
  for all $\sigma ' \in \Omega^J$ in some neighborhood of $\sigma$.
  Thus by the compactness of $\Omega^J$,
  we may find a finite collection of vectors
  $x_1, \dotsc, x_n \in \pi^J$ and corresponding
  finite rank tensor
  $T = \sum_j x_j \otimes \overline{x_j}$
  so that $\mathcal{H}_\sigma(T_0) \geq 1$
  for all $\sigma \in \mathcal{D}$.
\end{proof}

\subsubsection{The real form of a Bernstein component}\label{sec:real-form-bernstein}
Let $\Theta$ be any Bernstein component.
We denote by
$\Theta^{\unit}$
the image in
$\Theta$ of $\hat{H}$,
i.e.,
the set of infinitesimal characters in $\Theta$ arising
from some \emph{unitary} representation.
Since unitary representations are isomorphic to their
conjugate-dual,
$\Theta^{\unit}$
is pointwise fixed by the anti-holomorphic involution
$\Theta \ni [(M,\tau)] \mapsto [(M,\tau^+)]$, where as usual
$\tau^+$ denotes conjugate dual.
That involution 
defines a real form of $\Theta$
whose real points contain $\Theta^{\unit}$.
We henceforth abuse notation
slightly by writing $\mathbb{R}[\Theta]$
for the set of regular functions on that real form,
with real coefficients;
any such function is real-valued on $\Theta^{\unit}$.

We write
$\Theta^0 \subseteq \Theta^{\unit}$ for the subset
consisting of $[(M,\tau)]$
with $\tau$ unitary (i.e., $\tau \cong \tau^+$) and set
$\Theta^{\nt} := \Theta^{\unit} - \Theta^0$.
 The set $\Theta^0$ is
in general
a finite quotient of a compact torus;
we equip it with the pushforward
of an arbitrary Haar measure on the latter.

\begin{example*}
  Suppose $G = \PGL_2(F)$
  and that
  $\Theta = \mathfrak{X} / W$
  is the principal series component as considered above, so that
  we may identify $\mathfrak{X} \cong \mathbb{C}^\times$ by
  sending $\chi$ to its value $\alpha$ on a uniformizer and
  $\Theta$ with the quotient of $\mathbb{C}^\times$ by the
  equivalence relation $\sim$ defined by
  the inversion map $\alpha \mapsto \alpha^{-1}$.
  We then have the following identifications
  (here $q$ denotes
  the cardinality of the residue field of $F$):
  \begin{itemize}
  \item $\{\text{real points of } \Theta\} \cong \mathbb{C}^{(1)} \cup
    \mathbb{R}^\times / \sim$
  \item $\Theta^{\unit} \cong \mathbb{C}^{(1)} \cup
    [q^{-1/2},q^{1/2}]/\sim$
  \item $\Theta^0 \cong   \mathbb{C}^{(1)}/\sim$
  \item  $\Theta^{\nt} \cong   (q^{-1/2},1) \cup (1,q^{1/2}) / \sim$
  \end{itemize}
  We may identify $\mathbb{C}[\Theta]$
  with the ring of Laurent polynomials in $\alpha$ that are
  invariant under $\alpha \mapsto \alpha^{-1}$, and
  $\mathbb{R}[\Theta]$ with the subring consisting of those
  having real coefficients.
\end{example*}

\begin{lemma} \label{zero intersection}
 Then the closure $\overline{\Theta ^{\nt}}$
  of $\Theta^{\nt}$
  intersects $\Theta^0$
  in a set of measure  $0$.
\end{lemma}

\begin{proof}
  Let $t = [(M,\tau)] \in \Theta^{\unit}$.
  Then
  there is a unitary representation $\sigma \in \hat{H}$
  with infinitesimal character
  $\lambda_\sigma = t$.
  Since $\sigma$ is
  isomorphic to its own conjugate dual $\sigma^+$,
  we have $\lambda_{\sigma} = \lambda_{\sigma^+}$,
  so that
  $\tau^+ \cong w \tau$ for some
  $w \in N(M)/M$. If $w$ is trivial, then 
$\tau$ has unitary central character; thus  $\sigma$ is tempered  
  and thus $t \in \Theta^{0}$.

It follows that each
$t \in \overline{\Theta ^{\nt}} \cap \Theta^0$
  is contained in the set $\{[(M,\tau)] \in \Theta^0 : w \tau \cong \tau
  \}$
  for some $1 \neq w \in N(M)/M$.
  Each of these sets has measure zero.
\end{proof}
\subsubsection{Applications of Stone-Weierstrass} \label{SSW}

\begin{lemma*}
Let $\mathcal{D}, \mathcal{D}', \Theta$ be as in 
\S \ref{outlineoftheproofsec}. For any $k \in C_c(\mathcal{D})$ and $\eps  >0$
there is a regular function $f \in \R[\Theta]$ such that:
\begin{itemize}
\item $|k - f| < \eps$ on $\mathcal{D}$;
\item $|f| < \eps$ on $\mathcal{D}'$.
\end{itemize}
\end{lemma*}

\begin{proof}
  This follows from a variant of Stone-Weierstrass.
  We spell out some details:

  Since $\mathcal{D}$ is compact,
  infinitesimal character induces a homeomorphism
  between it and its image $\lambda(\mathcal{D})$ in $\Theta$. 
  Therefore, the continuous function $k$ is pulled back from a continuous function
  (also denoted $k$) on $\lambda(\mathcal{D})$. 

  Moreover, $\lambda(\mathcal{D})$ is disjoint from
  $\lambda(\mathcal{D}')$ by assumption,
  and each of these sets is closed.
  By 
  Tietze's extension theorem, we may find a continuous function on $\Theta$
  which induces $k$ on $\lambda(\mathcal{D})$, and is zero on $\lambda(\mathcal{D}')$. 
  The union of these sets is contained in a compact subset of the real points of $\Theta$,
  and then we apply Stone-Weierstrass as usual. 
\end{proof}

\begin{lemma*} \label{Bcent2}
  For each $\eps > 0$
  there is a regular function $f \in \mathbb{R}[\Theta]$
  with the following properties:
  \begin{itemize}
  \item $f$ is valued in $[-2,2]$ on $\Theta^{\unit}$.
  \item $f \geq 1$ on $\Theta^{\nt}$.
  \item $\int_{\Theta^0} f^2 \leq \eps$. 
  \end{itemize}
\end{lemma*}
For instance,
in the above $\PGL_2(F)$ example,
the conclusion of the lemma holds
with
$f := (\alpha^{-n} + \alpha^{-n+2} + \dotsb + \alpha^{n})/n$
for $n \in \mathbb{Z}_{\geq 1}$ taken large
enough in terms of $\eps$.
Note also that, in the third part, we take the measure on $\Theta^0$ to be that induced from Haar measure,
but that implies a similar statement for Plancherel measure,
which is a continuous multiple of the Haar measure
(see, e.g., \cite{MR1989693}).
% as
% used in the earlier applications.

\begin{proof}
 Write $\overline{\Theta^{\unit}}$ for the closure  of
 $\Theta^{\unit}$.
 Then $\overline{\Theta^{\unit}}$ is compact.  By
  Stone-Weierstrass, $\mathbb{R}[\Theta]$ is dense in the space
of
continuous real-valued functions
$\phi$ on $\overline{\Theta^{\unit}}$.
  Choose $\eps_1 > 0$ sufficiently
  small in terms of $\eps$.
  It then suffices to find
  such a $\phi$ for which
  \begin{itemize}
\item $|\phi| \leq 2 - \eps_1$ on $\overline{\Theta^{\unit}}$. 
  \item $\phi \geq (1+\eps_1)$ on $\Theta^{\nt}$, and
  \item $\int_{\Theta^0} \phi^2 \leq \eps_1$,
  \end{itemize}
  because then we may find $f \in \mathbb{R}[\Theta]$ with
  $\|f-\phi\| \leq \eps_1$
  on $\overline{\Theta^{\unit}}$, and this satisfies the required conditions.
   
  The existence of $\phi$ follows from Urysohn's lemma, using lemma
  \ref{zero intersection}.  
  \iftoggle{easyproof}
  {
  In more detail: Write $B=  \overline{\Theta^{\nt} }$ and $A = \Theta^0$. It's enough to make the function on $A \cup B$,
  then we use Tietze's extension theorem to 
  extend to the closure of $\Theta^{\unit}$,  without increasing supremum norm. 
Choose a open neighbourhood $U_A \subset A$ of $A \cap B$ whose measure, in $A$, is at most $\eps$. 
Now $A-U_A$ and $B$ are disjoint closed sets in a normal space. 
By Urysohn's lemma,
we can find a function $\phi: A \cup B \rightarrow [0,1]$
with the property that $\phi \equiv 0$ on $A-U_A$ and $\phi  = 1$ on $B$. 
Clearly $\int_{A} \phi^2 \leq \eps$, and $\phi$ equals $1$ on the non-tempered locus and is globally bounded by $1$. 
}
  \end{proof}

\iftoggle{cleanpart}
{
  \newpage
}
\part{Application to the averaged Gan--Gross--Prasad
  period}\label{part:appl-aver-gan}
We aim now to formulate and prove our main
result (theorem \ref{thm:main-subconvex},
stated at the very end).

\section{Setting}\label{sec:setting-for-main-result}
\subsection{Basic setup}
\label{sec-18-1}
Let $F$ be a number field;
denote by $\mathbb{Z}_F$ its ring of integers (we are using the letter $\mathcal{O}$ for a coadjoint orbit)
and by $\mathbb{A}$ its adele ring.   

Let $(\mathbf{G},\mathbf{H})$ be a GGP pair over $F$,
in the sense of \S\ref{sec:gross-prasad-pairs-inv-theory}.
\index{typical place $\mathfrak{p}$}
\index{distinguished place $\mathfrak{q}$}
We denote by $\mathfrak{p}$ a typical place of $F$ (possibly
archimedean!)
and
by $F_\mathfrak{p}$ the corresponding completion.
When  $\mathfrak{p}$ is non-archimedean,
we denote by $\mathbb{Z}_{\mathfrak{p}} \subseteq F_\mathfrak{p}$ the ring
of integers.
We set $G_\mathfrak{p} := \mathbf{G}(F_\mathfrak{p})$,
$H_\mathfrak{p} := \mathbf{H}(F_\mathfrak{p})$.

We fix a finite set $R$ of places of $F$
which   is sufficiently large in the following
senses:
\begin{itemize}
\item $R$ contains every archimedean place.
\item The groups $\mathbf{G}$ and $\mathbf{H}$
  admit smooth models
  over $\mathbb{Z}_F[1/R]$,
  which we continue to denote by  $\mathbf{G}$ and $\mathbf{H}$.
  This implies that for each $\mathfrak{p} \notin R$,
  the subgroups
  \[K_{\mathfrak{p}} := \mathbf{G}(\mathbb{Z}_{\mathfrak{p}})
    \leq G_{\mathfrak{p}}\text{ and }
    J_{\mathfrak{p}} := \mathbf{H}(\mathbb{Z}_{\mathfrak{p}})
    \leq H_{\mathfrak{p}}.
  \]
  are hyperspecial maximal compact subgroups.
\item  The inclusion $\mathbf{H} \hookrightarrow \mathbf{G}$
  extends to a closed immersion of the smooth models over $\mathbb{Z}_F[1/R]$,
  so that $K_{\mathfrak{p}}$ contains
  $J_{\mathfrak{p}}$.
\item
  Set
  \[
    G_R := \prod_{\mathfrak{p} \in R} G_\mathfrak{p},
    \quad K := \prod_{\mathfrak{p} \notin R} K_\mathfrak{p},
    \quad 
    H_R := \prod_{\mathfrak{p} \in R} H_\mathfrak{p},
    \quad J := \prod_{\mathfrak{p} \notin R} J_\mathfrak{p}.
  \]
  Then
  $\mathbf{G}(F) \cdot G_R  \cdot K = \mathbf{G}(\mathbb{A})$, and similarly for
  $\mathbf{H}$ in place of $\mathbf{G}$.
\end{itemize}
\subsection{Measures}
\label{sec-18-2}
We equip the quotients
$[\mathbf{G}] := \mathbf{G}(F) \backslash
\mathbf{G}(\mathbb{A})$
and
$[\mathbf{H}] := \mathbf{H}(F) \backslash
\mathbf{H}(\mathbb{A})$
with Tamagawa measures
and denote by $\tau(\mathbf{G})$ and $\tau(\mathbf{H})$
their volumes.

We fix a factorization
of the associated measures
on $\mathbf{G}(\mathbb{A}) = \prod' G_\mathfrak{p}$,
$\mathbf{H}(\mathbb{A}) = \prod' H_\mathfrak{p}$
in such a way that $K$ and $J$ have volume one.
We always equip products, such as $G_R$ and $H_R$,
with the product of the Haar measures on
the corresponding
components
$G_\mathfrak{p}$ and $H_\mathfrak{p}$.

\subsection{Automorphic forms}
\label{sec-18-3}
For the rest of this paper,
the letters $\Pi$ and $\Sigma$
denote irreducible square-integrable
automorphic representations $\Pi \subseteq L^2([\mathbf{G}])$ and $\Sigma \subseteq
L^2(\mathbf{H})$ 
that are unramified outside $R$,
i.e., that
admit vectors invariant by $K$ and $J$, respectively.
More precisely,
we write $\Pi$ and $\Sigma$ for
\emph{the subspaces
  spanned by the smooth 
  factorizable vectors}
in the corresponding Hilbert space representations,
so that we may identify $\Pi \cong \otimes_\mathfrak{p}
\Pi_\mathfrak{p}$
and
$\Sigma  \cong \otimes_\mathfrak{p} \Sigma_\mathfrak{p}$,
where $\Pi_\mathfrak{p}$ and $\Sigma_\mathfrak{p}$
are smooth irreducible unitarizable representations of $G_\mathfrak{p}$ and $H_\mathfrak{p}$.

For $\mathfrak{p} \notin R$,
the spaces $\Pi_\mathfrak{p}^{K_\mathfrak{p}}$
and $\Sigma_\mathfrak{p}^{J_\mathfrak{p}}$
are one-dimensional,
so the fixed subspaces  $\Pi^K$ and $\Sigma^J$ define irreducible
representations of $G_R$ and $H_R$,
respectively.
We may identify
these fixed subspaces
with the products of local components at $R$:
\[
  \Pi^K \cong \Pi_R := \otimes_{\mathfrak{p} \in R}
  \Pi_\mathfrak{p},
  \quad
  \Sigma^J \cong \Sigma_R := \otimes_{\mathfrak{p} \in R}
  \Sigma_\mathfrak{p}.
\]
We fix unitary structures on $\Pi_R$ and $\Sigma_R$
so that the above identifications are isometric.

Here, and later, when we (e.g.) sum over $\Sigma$, we always
have in mind that we sum over a set of representatives of $\Sigma$
as above, whose Hilbert space direct sum is $L^2([\H])$. (The   choice of representatives is ambiguous only in the event of global multiplicity larger than $1$.)
 \subsection{Branching coefficients\label{sec:branch-coefs}}
\label{sec-18-4}
Assume that $(\Pi_R,\Sigma_R)$
is \emph{distinguished}
in that the space of $H_R$-invariant linear
forms on $\Pi_R \otimes \Sigma_R^\vee$ is nonzero.
That space is then one-dimensional.
The
space
$\mathcal{I}$ 
consisting of all
$H_R$-invariant hermitian forms on
\[
  \Pi_R \otimes {\Pi^\vee_R}
  \otimes {\Sigma^\vee_R}
  \otimes \Sigma_R
  \rightarrow \mathbb{C} 
\]
is likewise one-dimensional.
We may define $\mathcal{P} \in \mathcal{I}$
by the formula
\begin{equation}
  \mathcal{P}(v_1 \otimes v_2 \otimes u_1  \otimes {u_2})
  :=
  (\int_{[\mathbf{H}]}  v_1 \overline{u_1})
  (\int_{[\mathbf{H}]}  \overline{v_2} u_2).
\end{equation}
If $\Pi_R$ and $\Sigma_R$ are tempered,
then we may define
$\mathcal{H} \in \mathcal{I}$
by
\begin{equation}
  \mathcal{H}(v_1 \otimes v_2 \otimes u_1 \otimes u_2)
  :=
  \int_{h \in H_R}
  \langle h v_1, v_2 \rangle
  \langle u_1, h u_2 \rangle,
\end{equation}
using the temperedness assumption to justify convergence
(cf. \S\ref{sec:local-disintegration}).

If
$\mathcal{H}$ is nonzero -- as is expected 
(cf. \S\ref{sec:inv-branch-overview})
-- then it spans the one-dimensional space $\mathcal{I}$, so we may define a branching coefficient $\mathcal{L}(\Pi,\Sigma) \in \mathbb{R}_{\geq 0}$ by
requiring that
\begin{equation}\label{eqn:branching-coef-defn}
  \mathcal{P}
  =
  \mathcal{L}(\Pi,\Sigma)
  \cdot \mathcal{H}
  \quad 
  (\text{on }
  \Pi_R
  \otimes \Pi_R^\vee 
  \otimes {\Sigma^\vee_R}
  \otimes \Sigma_R).
\end{equation}
We have suppressed from our notation the dependence of $\mathcal{L}(\Pi,\Sigma)$ 
upon the fixed set $R$ of places
at which everything is assuming unramified.
\subsection{The conjectures of Ichino--Ikeda and N. Harris}
\label{sec-18-5}
 
See \cite{MR2585578, MR3159075}.

\begin{conjecture*}
  If $(\Pi_R,\Sigma_R)$ is distinguished and $\Pi_R, \Sigma_R$
  are tempered, then $\mathcal{H}$, as defined in
  \S\ref{sec-18-4}, is nonzero, so $\mathcal{L}(\Pi,\Sigma)$ is
  defined; it is given by
  \begin{equation}
    \mathcal{L}(\Pi,\Sigma)
    =
    2^{-\beta} \rsL
    \Delta_G^{(R)},
  \end{equation}
  where $2^{\beta}$ is the order of the component group of the
  Arthur parameter for $\Pi \boxtimes \Sigma$ and $\Delta_G^{(R)}$ is, as in the introduction, 
  the partial $L$-factor of   the $L$-function whose local factor at almost every prime $p$
  equals $\frac{p^{\dim(G)}}{\# \mathbf{G}(\mathbf{F}_p)}$. 
\end{conjecture*}

\begin{remark*}
  We expect (but have not attempted to verify rigorously)\index{$2^\beta$}
  that 
  \[2^\beta = \tau(\mathbf{G}) \tau(\mathbf{H})\]
  holds generically, that is to say, for ``typical'' $\Pi$ and $\Sigma$. 
Let us explain where this comes from.
  
The left hand side is, by definition, the 
number of components
of the
centralizer (in $G^{\vee} \times  H^{\vee}$) of the  Arthur parameter for $\Pi \boxtimes \Sigma$. 
Now, for ``typical'' $\Pi, \Sigma$, one expects that the 
image of its Arthur parameter meets $G^{\vee} \times H^{\vee}$
in a Zariski dense set.
 Then the centralizer in question is simply 
the set of Galois-invariants in the center $Z(G^{\vee} \times
H^{\vee})$.
The cardinality of  the latter is directly related to Tamagawa numbers:
A result of Kottwitz
\cite[(5.1.1)]{KST},
building on work of Sansuc, shows that
\[ \tau(\mathbf{G}) \tau ( \mathbf{H}) = 
 \frac{\# Z(G^{\vee} \times H^{\vee})^{\mathrm{Gal}}}{ h},\]
where $h$ is the order of the Tate-Shafarevich group
for $Z(G^{\vee} \times H^{\vee})$. In the cases at hand,
$Z(G^{\vee})$ and $Z(H^{\vee})$ are either $\{\pm 1\}$
or the torus $\mathbb{G}_m$, and in the latter
case the Galois action is through a quadratic character;
in all cases, $h=1$. 
\end{remark*}

\subsection{Some unconventional notation}
\label{sec-18-7}

\subsubsection{}
\label{sec-18-7-1}
We fix once and for all an archimedean place $\mathfrak{q}$ of
$F$.  This place plays a privileged role
in both our results and their proofs,
so we introduce the otherwise unconventional notation
\[
  G := G_\mathfrak{q},
  \quad
  G' := \prod_{\mathfrak{p} \in R - \{\mathfrak{q}\}} G_\mathfrak{p},
\]
\[
  \Gamma := \mathbf{G}(\mathbb{Z}_F[1/R]) = \mathbf{G}(F) \cap K
  \hookrightarrow G \times G',
\]
\[
  [G] := \Gamma \backslash (G \times G'),
\]
so that
\index{$G, G', \Gamma, [G]; H, H', \Gamma_H, [H]$}
$[G]
\cong [\mathbf{G}]
/ K$.
We analogously
define $H, H', \Gamma_H$, and $[H] := \Gamma_H \backslash (H \times H') \cong [\mathbf{H}]/J$.
By our choice of factorization of Haar measures,
the quotients $[G]$ and $[H]$ have volumes $\tau(\mathbf{G})$ and $\tau(\mathbf{H})$.

\subsubsection{}
\label{sec-18-7-2}

We denote as in \S\ref{sec:measures-et-al-G-g-g-star} by $\mathfrak{g}$ and $\mathfrak{h}$ the Lie algebras
of $G$ and $H$, respectively,
and
by $\mathfrak{g}^\wedge
\cong i \mathfrak{g}^*$
and
$\mathfrak{h}^\wedge
\cong i \mathfrak{h}^*$
their the Pontryagin duals.
We have normalized Haar measures both on $H$ and $H'$, giving
rise to Plancherel measures on the unitary duals of both.
The Haar measure on $H$
normalizes Haar measures on $\mathfrak{h}$ and
$\mathfrak{h}^\wedge$,
hence a normalized affine measure
on the GIT quotient $[\mathfrak{h}^\wedge]$
(see \S\ref{sec:infin-char},
\S\ref{sec:groups-over-R}).

\subsubsection{}
\label{sec-18-7-3}
\label{sec:unconv-not-3}
Recall that for any $\Pi$ and $\Sigma$,
we may isometrically identify and embed
\[
  \Pi_R \cong \Pi^K \hookrightarrow L^2([G]),
  \quad
  \Sigma_R \cong \Sigma^J \hookrightarrow L^2([H]).
\]
We may then unitarily factor
\[
  \Pi_R = \pi \otimes \pi ',
  \quad
  \Sigma_R
  = \sigma \otimes \sigma ',
\]
where
\[\pi \cong \Pi_\mathfrak{q}, \quad
  \sigma = \Sigma_\mathfrak{q},
\]
\[
  \pi ' \cong \prod_{\mathfrak{p} \in R - \{\mathfrak{q} \}}
  \Pi_\mathfrak{p},
  \quad
  \sigma ' \cong \prod_{\mathfrak{p} \in R - \{\mathfrak{q} \}}
  \Sigma_\mathfrak{p}
\]
are smooth irreducible unitary representations of $G,H,G',H'$,
respectively.

\subsection{Assumptions\label{sec:assumptions-for-main-result}}
For the remainder of the paper, we
fix an individual $\Pi$ as above.
We now impose several assumptions
concerning $\mathbf{G}, \mathbf{H}, R, \mathfrak{q}, \Pi$: %

\begin{enumerate}
\item The representations $\pi$ of $G$ and $\pi '$ of $G'$ are tempered.
\item
  $\mathbf{G}$ and  $\mathbf{H}$ are
  anisotropic and  
  non-trivial.
  This has the following consequences:
  \begin{itemize}
  \item $[\mathbf{G}],[G], [\mathbf{H}]$ and
    $[H]$ are compact.
  \item
    The pair $(\mathbf{G}, \mathbf{H})$ is isomorphic
    either to
    \begin{itemize}
    \item  $(\mathrm{U}_{n+1}, \mathrm{U}_n)$ with $n \geq 1$,
      or to
    \item 
    $(\mathrm{SO}_{n+1}, \mathrm{SO}_n)$ with $n \geq 2$,
    \end{itemize}
    and not to $(\GL_{n+1}, \GL_n)$. 
    In particular, $\tau(\mathbf{G}) = \tau(\mathbf{H})=2$ (see,
    e.g., \cite{MR670072} and 
    \cite{MR0156851}).
  \end{itemize}
\item $G_\mathfrak{p}$ is compact for every archimedean place $\mathfrak{p} \neq \mathfrak{q}$.
\item
  $\mathbf{G}$ is quasi-split at
  $\mathfrak{q}$, so that $G$ is quasi-split,
  and the representation
  $\pi$ is generic, hence satisfies the equivalent conditions of
  the lemma of \S\ref{sec:limit-coadj-orbits-h-indep}.  (The
  assumption concerning $\pi$ holds if, for instance,
  it belongs to the principal series.)
  In particular, the
  limit orbit $\mathcal{O}$ of
  $\pi$ is a \emph{nonempty} union of open
  $G$-orbits on the regular subset $\mathcal{N}_{\reg}$ of the nilcone in
  $\mathfrak{g}^\wedge$:
  \[
    \emptyset \neq \mathcal{O} \subseteq \mathcal{N}_{\reg}.
  \]
\end{enumerate}

\begin{example*}
  Take $F = \mathbb{Q}(\alpha)$, $\alpha^2 =
  5$.
  Let $n \in \mathbb{Z}_{\geq 4}$.
  Take for
  $\mathbf{G}$ the special orthogonal group of the
  $F$-quadratic form
  \[
    x_1^2 + \dotsb + x_m^2 + (1 - \alpha) (x_{m+1}^2
    + \dotsb + x_n^2),
  \]
  where $n = 2 m$ or $2 m + 1$.
 
   Similarly,
  let $\mathbf{H}$
  be the special orthogonal group of the
  quadratic form
  \[
    x_1^2 + \dotsb + x_m^2 + (1 - \alpha) (x_{m+1}^2
    + \dotsb + x_{n-1}^2),
  \]
  embedded in $\mathbf{G}$ as usual.
  Let
  $\mathfrak{q}$ be the archimedean place of
  $F$ sending $\alpha$ to $\sqrt{5} =
  2.23\dotsc$, and $\mathfrak{q}
  '$ the other archimedean place.  Then $G \cong
  \SO(m,m)$ or $\SO(m,m+1)$ is split, while $G_{\mathfrak{q} '}
  \cong \SO(n)$ is compact;
  similarly, $H \cong \SO(m,m-1)$ or $\SO(m,m)$
  and $H_{\mathfrak{q} '} \cong \SO(n-1)$.
\end{example*}

\begin{remark*}
  Assumption (1), or some strong bound in that
  direction, is required by the formulation of the Ichino--Ikeda/N. Harris
  conjectures; we also exploit it through our use
  of the Kirillov formula. Assumption (2)
  is essential for the measure classification step.
   Assumption (3) is primarily for convenience.
  Assumption (4) ensures that  
  regular nilpotent elements exist,
  or equivalently,
  that the limit coadjoint orbits
  considered here be nonempty
  (cf. \S\ref{sec:limit-coadj-orbits-h-indep}).
\end{remark*}

\section{Construction of limit states}\label{sec:limit-stat-attach-symb}

\subsection{Setting}\label{sec:constr-limit-stat-setting}
Recall that we have defined the quotient
\[[G] = \Gamma \backslash (G  \times G'),\]
where
\begin{itemize}
\item $G$ is a reductive group over an archimedean local field,
\item $G'$ is an $S$-arithmetic group, and
\item $\Gamma$ is a cocompact lattice in $G \times G'$.
\end{itemize}
The groups $G$ and $G'$ arose from ``half''
of a GGP pair,
but the properties just enunciated are what matter here.
We denote by $\mu_{[G]}$  the Haar measure
that we have normalized on $[G]$.
From an automorphic representation $\Pi$ of $\mathbf{G}$,
we obtained a unitary
$G \times G'$ subrepresentation
\[
\pi \otimes \pi ' \hookrightarrow L^2([G]).
\]
We will assume starting in \S\ref{sec:from-symb-funct}
that $\pi$
varies with an infinitesimal parameter $\h \rightarrow 0$
and
admits a regular limit coadjoint orbit
(cf. \S\ref{sec:limit-coadjoint})
\[
(\mathcal{O},\omega)
= 
\lim_{\h \rightarrow 0}
(\h \mathcal{O}_\pi, 
\omega_{\h\mathcal{O}_\pi}).
\]
In applications,
$\mathcal{O}$ will be a subset
of $\mathcal{N}_{\reg}$,
but this feature plays no role in
\S\ref{sec:limit-stat-attach-symb}.

\subsection{Overview}
The main aim of this section is to construct,
after passing to a subsequence of $\{\h\}$,
a natural $G$-equivariant assignment
\begin{equation}\label{eq:microlocal-lift-overview}
\mathcal{O}
\rightarrow 
\{\text{probability measures on $[G]$}\}
\end{equation}
which
captures the average limiting behavior of
the $L^2$-masses $|v|^2 \, d \mu_{[G]}$ of
vectors $v \in \pi$
microlocalized at a point
$\xi \in \mathcal{O}$.
The construction shares many common features
with standard constructions
in the pseudodifferential calculus
(referred
to variously as the quantum limits, semiclassical limits,
microlocal defect measures, ...).

This construction will be achieved in three stages:
\begin{enumerate}
\item
    {\bf From operators to
  functions.}
  Recall, from \S\ref{sec:operator-classes},
  the space
  $\Psi^{-\infty} := \Psi^{-\infty}(\pi)$ of ``smoothing operators''
  on $\pi$;
  it contains $\pi \otimes \overline{\pi }$ as the subspace
  of finite-rank operators.
  Recall also,
  from \S\ref{sec:more-spaces-ops},
  the space
  \[\mathcal{T}_1 := \mathcal{T}_1(\pi) = \{\text{``smoothly trace class'' operators on $\pi$}\}.\]
  The first stage, achieved in \S\ref{sec:reduction-of-pfs-quantization}, is then
  to construct a 
  $G$-equivariant positivity-preserving map
  \[
  \mathcal{T}_1 \rightarrow C^\infty([G]),
  \]
  \[
  T \mapsto [T]
  \]
  such that the trace of $T$ is the integral of $[T]$.
  (The map will depend
  upon the choice of a fixed ``family of vectors''
  in the auxiliary representation $\pi '$.)
\item
    {\bf From symbols to
  functions.}
  The second stage, achieved in \S\ref{sec:from-symb-funct},
  is to compose the map $T \mapsto [T]$
  with the operator calculus
  \[
  \Opp_{\h} : S^m := S^m(\mathfrak{g}^\wedge) \rightarrow \Psi^m
  := \Psi^m(\pi),
  \]
  specialized here to the
  Schwartz space
  $S^{-\infty} = \mathcal{S}(\mathfrak{g}^\wedge)$.  This gives
  a family of maps $\mathcal{S}(\mathfrak{g}^\wedge) \rightarrow C^{\infty}([G])$,
  depending upon $\h$.
  We will show that the leading order
  asymptotics of this family of maps
  are described, after passing to a subsequence, by a limit map
  \[
  [ \cdot ] : \mathcal{S}(\mathfrak{g}^\wedge) \rightarrow C^\infty([G]),
  \]
  with several natural properties.
  
  Informally (cf. \S\ref{sec:intro-op-calc}), fix a small
  nonempty open set $E \subseteq \mathcal{O}$, and suppose that
  $a|_{\mathcal{O}}$ defines a smooth approximation to the
  normalized characteristic function
  $\vol(E, d \omega)^{-1} 1_E$ of $E$.  Then $[a]$ roughly describes
  the limiting \emph{average} of $|v|^2$, taken over all vectors
  $v \in \pi$ microlocalized in $E$.
\item {\bf From  points to measures.}  The third stage,
  achieved in \S\ref{sec:constr-mu}, is to  
  describe the map
  $[\cdot]$
  in terms of measures.
  The description may be understood as an effective
  implementation of measure disintegration.  It will allow us to
  analyze our limit states using measure-theoretic techniques,
  notably Ratner's theorem.
\end{enumerate}

\subsection{Stage 1: from operators to
  functions}\label{sec:reduction-of-pfs-quantization}
We fix a (smooth, finite-rank) nonzero element
$T' \in \pi' \otimes \overline{\pi '}$ with the following
properties:
\begin{itemize}
\item $T'$ is positive definite.
\item \emph{Unless explicitly stated otherwise}, $\trace(T') = 1$.  
\end{itemize}
Equivalently, $T'$ is a finite sum
\begin{equation}\label{eqn:formula-for-T-prime}
  T' = \sum_i c_i u_i \otimes \overline{u_i},
\end{equation}
where the $u_i \in \pi '$ are smooth unit vectors
and the $c_i$ are nonnegative reals summing to $1$.
The normalization
$\trace(T') = 1$ serves
to simplify notation;
in practice,
we may reduce to the case in which it is satisfied by
multiplying $T'$ by a suitable positive scalar.

For each
$T \in \pi \otimes \overline{\pi }$, we obtain an element
$T \otimes T' \in L^2([G] \times [G])$.
We denote by
$[T \otimes T'] \in L^1([G])$ its diagonal restriction.
Since the interesting variation
will happen in the $T$ variable,
we will often  drop the $T'$
from the notation by abbreviating
\index{$[T]$}
\[
[T] := [T \otimes T']
\in L^1([G]).
\]
\begin{example*}
  Suppose that $T = v_1 \otimes \overline{v_2}$
  and $T' = w_1 \otimes \overline{w_2}$
  with $v_i \in \pi, w_i \in \pi '$,
  so that $u_i := v_i \otimes w_i$
  defines an element of $\pi \otimes \pi ' \hookrightarrow L^2([G])$.
  Then $(T \otimes T')(x,y) = u_1(x) \overline{u_2(y)}$,
  while $[T] (x)= u_1(x) \overline{u_2(x)}$.
\end{example*}
\begin{lemma*}
  Fix $T'$ as above.
  Abbreviate ``continuous, uniformly in $\pi$''
  to
  ``continuous.''
  \begin{enumerate}[(i)]
  \item $\|[T]\|_{L^1([G])} \leq \|T\|_1$,
    where $\|.\|_1$ denotes the trace norm.
  \item The map $T \mapsto T \otimes T'$ extends
    uniquely
    to a continuous $G \times G$-equivariant map
    \[
    \Psi^{-\infty} \rightarrow C^\infty([G] \times [G]).
    \]
  \item The map $T \mapsto [T]$ extends
    uniquely
    to a continuous $G$-equivariant map
    \[
    \mathcal{T}_1 \rightarrow C^\infty([G]).
    \]
  \item
    For $T \in \mathcal{T}_1$,
    we have
    \[\int_{[G]} [T] = \trace(T).\]
  \item If $T$ is a nonnegative self-adjoint operator,
    then $[T]$ is a nonnegative function:  
    \[ T \geq 0 \implies [T] \geq 0. \]
  \item $[T]$ is invariant by the action of the center of
    $\mathbf{G}(\mathbb{A})$ on $[G]$.
  \end{enumerate}
\end{lemma*}
\begin{proof}
  For (i),
  let $\mathcal{B}(\pi)$
  be an orthonormal basis for $\pi$ consisting of eigenvectors
  $v$ 
  for the nonnegative self-adjoint finite-rank operator
  $\sqrt{T^* T}$, with eigenvalues $c_v \geq 0$.
  Writing $T'$ as in \eqref{eqn:formula-for-T-prime},
  we then have
  \begin{equation}
    [T] =
    \sum_{v \in \mathcal{B}(\pi)}
    \sum_i
    c_i
    (T v \otimes u_i) \cdot \overline{v \otimes u_i},
  \end{equation}
  where each sum is really a finite sum.
  By Cauchy--Schwarz,
  \begin{align*}
    \int_{[G]}
    |(T v \otimes u_i) \cdot \overline{v \otimes u_i}|
    &\leq
      (
      \int_{[G]}
      |T v \otimes u_i|^2)^{1/2}
    \\
    &=
      \langle T v \otimes u_i, T v \otimes u_i \rangle^{1/2}
    \\
    &= 
      \langle T^* T v \otimes u_i, v \otimes u_i \rangle^{1/2}
    \\
    &= 
      c_v.
  \end{align*}
  Thus by the triangle inequality,
  \[
  \|[T]\|_{L^1([G])}
  \leq (\sum_{v \in \mathcal{B}(\pi)} c_v) (\sum_i c_i)
  = \sum_{v \in \mathcal{B}(\pi)} c_v
  =
  \|T\|_1
  \]

  For (ii),
  we note that the map in question is isometric
  up to a constant factor (given by the Hilbert--Schmidt norm of $T'$);
  the conclusion follows from the Sobolev lemma, applied
  on the homogeneous space $[G] \times [G]$.
  We analogously deduce (iii)
  from (i) and the Sobolev lemma on
  $[G]$.

  For (iv) and (v),
  the proof reduces by continuity to the finite-rank case,
  and then by linearity
  to the rank one case as in the above example.
  In that case,
  \[
  \int_{[G]}
  [T]
  = \langle u_1, u_2 \rangle = \langle v_1, v_2 \rangle
  \langle w_1, w_2 \rangle = \trace(T) \trace(T')
  = \trace(T).
  \]
  If $T \geq 0$,
  then we may assume moreover that $v_1 = v_2$ and $w_1 = w_2$,
  hence that $u_1 = u_2 =: u$,
  and so
  \[
  [T](x) = |u(x)|^2 \geq 0.
  \]

  For (vi),
  let $\omega$ denote the central character of $\Pi$.
  Then the function $T \otimes T'$ on $[G] \times [G]$
  has central character $(\omega,\omega^{-1})$,
  and so its restriction $[T]$ to $[G]$ has trivial central
  character.
\end{proof}

\subsection{Stage 2: from symbols to functions}\label{sec:from-symb-funct}
We now allow $\pi$ to vary with a positive parameter $\h
\rightarrow 0$
in such a way that
we obtain a regular  limit coadjoint orbit
$(\mathcal{O},\pi)$.
We assume that $\h$ traverses some sequence $\{\h\}$
having $0$ as its unique accumulation point.

For $k \in \mathbb{Z}_{\geq 1}$
and
$(a_1,\dotsc,a_k) \in \mathcal{S}(\mathfrak{g}^\wedge)^k$,
the operator $\Opp_{\h}(a_1) \dotsb \Opp_{\h}(a_k)$
belongs to $\mathcal{T}_1$,
and thus yields a smooth function
\index{$[a]_{\h}, [a], [a_1,\dotsc,a_k]_{\h}, [a_1,\dotsc,a_k]$}
\[
[a_1,\dotsc,a_k]_{\h} : [G] \rightarrow
\mathbb{C}
\]
\begin{equation}\label{eqn:defn-brack-sub-h}
  [a_1,\dotsc,a_k]_{\h} :=
\h^d [\Opp_{\h}(a_1)\dotsb \Opp_{\h}(a_k)],
\end{equation}
where
the exponent $d$
in the normalizing factor $\h^d$ denotes half the real dimension
of $\mathcal{O}$, as usual.
\begin{lemma*}
  The linear map $[\cdot]_{\h} : \mathcal{S}(\mathfrak{g}^\wedge)^k \rightarrow
  C^\infty([G])$
  defined by \eqref{eqn:defn-brack-sub-h}
  is $\h$-uniformly continuous.
  That is to say:
  for each continuous seminorm $\nu$
  on $C^\infty([G])$
  and each $k \geq 1$,
  there is a continuous seminorm
  $\mu$ on $\mathcal{S}(\mathfrak{g}^\wedge)$
  so that for all $\h \in (0,1]$
  and all $a_1,\dotsc,a_k \in \mathcal{S}(\mathfrak{g}^\wedge)$,
  we have
  \[
    \nu([a_1,\dotsc,a_k]_{\h})
    \leq
    \mu(a_1) \dotsb \mu(a_k).    
  \]
\end{lemma*}
\begin{proof}
  Indeed,
  by
  \eqref{eq:final-op-est-1}
  and
  part (iii) of the lemma of
  \S\ref{sec:reduction-of-pfs-quantization},
  $[\cdot]_{\h}$ is a composition
  \[
    [\cdot]_{\h} 
    :
    \mathcal{S}(\mathfrak{g}^\wedge)^k
    \xrightarrow{\h^d \Opp_{\h}^{\otimes k}}
    \mathcal{T}_1
    \xrightarrow{[\cdot]} C^\infty([G])
  \]
  of $\h$-uniformly continuous maps.
\end{proof}

We aim now to show that this family of maps has a limit
satisfying
several natural properties.

\begin{theorem} \label{limit state theorem}
  Let notation be as in \S\ref{sec:constr-limit-stat-setting}
  and \eqref{eqn:defn-brack-sub-h}.
  After passing to a
  subsequence of $\{\h\}$,
  there exist continuous maps
  (depending upon $\pi$ and the choice of subsequence)
  \[[\cdot] : \mathcal{S}(\mathfrak{g}^\wedge)^k \rightarrow C^\infty([G])\]
  \[
    (a_1,\dotsc,a_k) \mapsto [a_1,\dotsc,a_k],
  \]
  indexed by $k \in \mathbb{Z}_{\geq 1}$
  with the following properties:
  \begin{enumerate}[(i)]
  \item For fixed
    $k \in \mathbb{Z}_{\geq 1}$,
    the function
    $[a_1,\dotsc,a_k]_{\h}$ converges to $[a_1,\dotsc,a_k]$
    in the $C^\infty$-topology,
    with continuous dependence upon $(a_1,\dotsc,a_k)$.

  \item $[a_1,\dotsc,a_k] = [a_1 \dotsb a_k]$. In particular,
      \begin{equation}\label{eq:characterizing-bracket-map-recap}
  [a_1 \dotsb a_k]
  = \lim_{h \rightarrow 0}
  \h^d [\Opp_{\h}(a_1) \dotsb \Opp_{\h}(a_k)]
  \text{ in $C^\infty([G])$}
\end{equation}

  \item $[a, \overline{a}] \geq 0$.
  \item If $a \geq 0$, then $[a] \geq 0$.
  \item $[\cdot]$ is $G$-equivariant (for any $k$).
  \item
    \begin{equation}\label{eqn:limit-state-construction-preserves-volumes}
      \int_{[G]} [a] =    
      \int_{\mathcal{O}} a \, d \omega.
    \end{equation}
  \end{enumerate}
\end{theorem}
\begin{proof}~
  \begin{enumerate}[(i)]
  \item The existence and first property
    of $[\cdot]$ follows by an Arzela--Ascoli
    type argument together with the noted uniform continuity
    of the maps $[\cdot]_{\h}$.  We record the details for
    convenience:

    We observe first that for any
    $(a_1,\dotsc,a_k) \in \mathcal{S}(\mathfrak{g}^\wedge)^k$,
    the subset
    \[
      \{ [a_1,\dotsc,a_k]_{\h} : \h \in (0,1] \}
    \]
    of $C^\infty([G])$
    is \emph{bounded},
    i.e., each seminorm has bounded image.
    This feature is a consequence of the $\h$-uniform continuity
    of $[\cdot]_{\h}$
    noted in the above lemma.
    % , noted following
    % \eqref{eqn:defn-brack-sub-h}.

    We observe next that each bounded subset of
    $C^{\infty}([G])$ is precompact.  Indeed,
    boundedness forces equicontinuity of the subset, and
    therefore the existence of a subsequence that converges in
    the space $C([G])$ of continuous functions equipped with the
    supremum norm.  One similarly extracts a further subsequence
    that converges along with its first derivatives in $C([G])$,
    and -- proceeding in this way and using a diagonal argument
    -- a subsequence that converges in $C^{\infty}([G])$.

    We observe finally that $\mathcal{S}(\mathfrak{g}^\wedge)^k$ is
    a separable topological vector space.  A
    countable dense subset may be given, for instance, by finite rational linear
    combinations of Hermite functions.

    We now choose a countable dense $\mathbb{Q}$-subspace $S_0$
    of $\mathcal{S}([\mathfrak{g}])^k$.  By the second
    observation above, we may find a subsequence of $\{\h\}$
    along which $[a_1, \dotsc, a_k]_{\h}$ converges in
    $C^{\infty}([G])$ for each $(a_1, \dots, a_k)$ in $S_0$.
    Call this limit $L(a_1, \dots, a_k)$.  The noted
    $\h$-uniform continuity of the maps $[\dotsb]_{\h}$ then
    implies that the map $L: S_0 \rightarrow C^{\infty}([G])$ is
    continuous, hence extends uniquely to the desired continuous
    map
    $\mathcal{S}([\mathfrak{g}])^k \rightarrow C^{\infty}([G])$.
    We verify readily that this extension is
    $\mathbb{R}$-linear.
    
  \item
    We apply $[\cdot] : \mathcal{T}_1 \rightarrow
    C^\infty([G])$
    to \eqref{eq:final-op-est-2}.
  \item For every $\h$, the operator
    $\Opp_{\h}(a) \Opp_{\h}(\overline{a})= \Opp_{\h}(a) \Opp_{\h}(a)^*$
    is nonnegative-definite,
    and so $[a, \overline{a}]_{\h} \geq 0$ by
    the assumed properties of $T \mapsto [T]$.
    We conclude by taking limits.
  \item By continuity, it suffices to show that $[a] \geq 0$
    for every $a$ of the form $a = b \overline{b}$, which is the
    content of the previous assertion.
  \item
    We may assume $k=1$.
    We
    then
    apply $[\cdot] : \mathcal{T}_1 \rightarrow
    C^\infty([G])$
    to \eqref{eq:final-op-est-3}.
  \item We have
    $\int_{[G]} [a] = \lim_{h \rightarrow 0} \int_{[G]}
    [a]_{\h}$
    and
    $\int_{[G]} [a]_{\h} = \trace(\h^d \Opp_{\h}(a))$
    and
    (by \eqref{eqn:kirillov-appl-to-lim-orb})
    $\lim_{\h \rightarrow 0} \trace(\h^d \Opp_{\h}(a))
    = \int_{\mathcal{O}} a \, d \omega$.
  \end{enumerate}
\end{proof}

\subsection{Stage 3: from points to measures\label{sec:constr-mu}}
% We aim now to extend
% $[\cdot] : \mathcal{S}(\mathfrak{g}^\wedge) \rightarrow C^{\infty}([G])$ to a map between
% spaces of measures.
% We introduce some notation.
% Given a locally compact space $X$, we denote by $\mathcal{M}(X)$
% the space of finite signed Radon measures on $X$ and by
% $\mathcal{M}_c(X)$ the set of all $\mu \in \mathcal{M}(X)$ with
% compact support.  The space $\mathcal{M}(X)$ comes with two
% topologies: the weak topology, and the topology induced by the
% total variation norm.   
% We equip $\mathcal{M}_c(X)$
% with the direct limit
% of the weak topologies on the subspaces
% $\mathcal{M}(C)$, for
% $C \subseteq X$ compact.
% (We never work with the total variation topology on $\mathcal{M}_c(X)$.)
\begin{theorem} \label{thm:quantization-of-measures}
  Let notation and assumptions be as in Theorem \ref{limit state
    theorem};
  assume also that we have passed to a subsequence of $\{\h\}$
  for which that the conclusions of that theorem hold.
  There is then a $G$-equivariant
  linear map
  \begin{equation}\label{eqn:quantizing-delta-masses}
    \mathcal{O}
    \rightarrow \{\text{probability measures on } [G]\},
    \qquad
    \xi \mapsto [\delta_\xi]
  \end{equation}
  so that
  for all $a \in C_c^\infty(\mathfrak{g}^\wedge)$
  and
  $\Psi  \in C_c([G])$,
  \begin{equation}\label{eqn:disintegration-super-explicit}
    \int_{[G]}
    [a] \Psi
    = \int_{\xi \in \mathcal{O}}
    a(\xi)  (\int [\delta_\xi] \Psi ) \, d \omega(\xi).
  \end{equation}
\end{theorem}
An importance
consequence
of this theorem
is
that
\begin{equation}\label{eqn:sweet-sweet-centralizer-invariance-muahaha}
  \text{$[\delta_\xi]$ is
    invariant
    by the $G$-centralizer of $\xi$.}
\end{equation}
The measures $[\delta_\xi]$ and the invariance property
\eqref{eqn:sweet-sweet-centralizer-invariance-muahaha}
generalize the ``microlocal lift'' of semiclassical analysis and
its invariance under geodesic flow.
Informally speaking, $[\delta_{\xi}]$
describes the {\em average} behavior of the measures
$|v|^2 \mu_{[G]}$, where $v \in \pi \subset L^2([G])$ is
localized near $\xi$.

\begin{proof}[Proof of theorem \ref{thm:quantization-of-measures}]
  Consider the distribution $\eta$ on
  $\mathfrak{g}^\wedge \times [G]$
  given
  for $a \in C_c^\infty(\mathfrak{g}^\wedge)$ and $\Psi \in
  C^\infty([G])$ by
  \begin{equation}
    \eta(a \otimes \Psi)
    :=
    \int_{[G]}
    [a] \cdot \Psi.
  \end{equation}
  (The Schwartz kernel theorem
  provides the unique extension of this definition
  from the algebraic tensor product
  $C_c^\infty(\mathfrak{g}^\wedge) \otimes C^\infty([G])$
  of test function spaces
  to the space
  $C_c^\infty(\mathfrak{g}^\wedge \times [G])$
  of test functions on the product space.)
  Part (iv) of theorem \ref{limit state theorem}
  implies that $\eta$ is positive
  on positive functions.
  Those results and the Riesz representation theorem
  then imply that
  $\eta$ defines a $G$-invariant positive measure
  on $\mathfrak{g}^{\wedge} \times [G]$,
  pushing forward in the first coordinate to the measure $\omega$ on 
  $\mathcal{O}$. 

  Since $\eta$ is positive,
  we have for $a \in C_c(\mathfrak{g}^{\wedge})$   and $\Psi \in C_c([G])$
  that
\begin{equation} \label{eta bound}
    |\eta(a \otimes \Psi)| \leq \eta(|a| \otimes |\Psi|)
    \leq \|\Psi \|_{\infty}
    \int_{\mathcal{O}} |a| \, d \omega.
\end{equation}
  
Suppose, for a moment, that $\Psi \geq 0$. 
The rule $a \mapsto \eta(a \otimes \Psi)$
defines then a positive functional on $C_c(\mathfrak{g}^{\wedge})$, 
thus a measure. This measure is absolutely continuous
with respect to $\omega$, and thus (by Radon--Nikodym) may be written as 
$f_{\Psi} \omega$ for some measurable nonnegative function $f_{\Psi}$  on $\mathfrak{g}^{\wedge}$. 
The bound \eqref{eta bound} shows that in fact $f_{\Psi} \leq \|\Psi\|_{\infty}$
almost everywhere with respect to $d\omega$.  
The rule $\Psi \mapsto f_{\Psi}$ defines then a function
\[C_c([G])_{\geq 0}  \longrightarrow L^{\infty}(\mathcal{O})\]
and this extends to  a bounded linear map on all of $C_c([G])$ by splitting into positive and negative parts.

  The image of the map $\Psi \mapsto f_\Psi$
  in fact lies inside the space $C(\mathcal{O})$
  of continuous functions,
  since there is a dense subspace of the source (all convolutions with elements in $C_c(G)$)
  for which this is true. This uses the following   consequence
  of the regularity of the limit multiorbit $\mathcal{O}$:
  \begin{equation}\label{eqn:key-property-for-limit-measures}
    \text{Every $G$-orbit in $\mathcal{O}$ is 
      open in $\mathcal{O}$.}
  \end{equation}
Finally, the composition
  \[C_c([G])
  \stackrel{\Psi \mapsto f_{\Psi}}{\longrightarrow} C(\mathcal{O}) \stackrel{\text{eval. at $\xi$}}{\longrightarrow} \R\]
  gives the   desired measure $\delta_{\xi}$. 
\end{proof}

\section{Equidistribution of limit states}
\label{sec:appl-meas-class}

\subsection{Overview and statement of result}
We  aim now to apply Ratner's
classification of invariant measures on a homogeneous space to
obtain strong control over the limit states constructed in the
prior section.

Recall, thus, that we are considering a real reductive group $G$,
an $S$-arithmetic group $G'$,
a cocompact lattice $\Gamma < G \times G'$,
the corresponding compact quotient
$[G] := \Gamma \backslash (G \times G')$,
and a
unitary
$G \times G'$-subrepresentation $\pi \otimes \pi'$
of $L^2([G])$;
this data arose from the automorphic representation
$\Pi$ 
of the reductive group $\mathbf{G}$ over the number field $F$.

We have constructed, in the prior section,
some
$G$-equivariant
``limit state''
assignments
\[
  [ \cdot ] : \mathcal{S}(\mathfrak{g}^\wedge) \rightarrow C^\infty([G])
\]
and
\[
  \mathcal{O}
  \stackrel{\xi \mapsto [\delta_{\xi}]}\longrightarrow 
  \{\text{measures on $[G]$}\}, \]
capturing the limiting behavior of the functions on $[G]$
obtained from high-frequency vectors in $\pi$.
Thanks to the
disintegration \eqref{eqn:disintegration-super-explicit}, which
we abbreviate here to
\[
  [a] \, d \mu_{[G]} = \int_{\xi \in \mathcal{O}}
  a(\xi) [\delta _\xi] \, d \omega(\xi),
\]
the assignments $[\cdot]$
are determined
completely
by the probability measures
$[\delta_\xi]$ on $[G]$.
Each measure $[\delta_\xi]$
is invariant by the centralizer $G_\xi \leq G$ of $\xi$.

We employ now crucially our assumption that $\pi \otimes \pi '$ is
\emph{fixed}, independent of the scaling parameter $\h
\rightarrow 0$.
The limit
coadjoint orbit $\mathcal{O}$ for $\pi$ is then
contained in
the regular nilpotent set $\mathcal{N}_{\reg} \subseteq
\mathfrak{g}^\wedge$,
so the centralizers
$G_\xi$
contain unipotent elements.
We might thus hope to apply Ratner's
theorem to deduce that $[\delta_\xi]$ is simply the Haar
probability measure on $[G]$,
hence for any
symbol $a \in \mathcal{S}(\mathfrak{g}^\wedge)$ that the limit state
$[a]: [G] \rightarrow \mathbb{C}$ is simply the constant
function satisfying the normalization condition
$\int_{[G]} [a] = \int_{\mathcal{O}} a \, d \omega$.
Unwinding the definition of the limit states,
this may be understood informally as a form
of quantum unique ergodicity
for high-frequency vectors in the fixed representation $\pi$.

The actual picture is mildly more complicated due to ``multiple
component'' issues arising from the presence of the auxiliary
group $G'$; the $G$-action on $[G]/J$,
for the compact open $J < G'$
under which our measures have known invariance,
is not in general
transitive, and so we cannot hope to control the behavior of
limit states across all of $[G]$.
However, it turns out
that the
control obtained is enough to understand \emph{integrals of limit
  states over $[H]$}, in the following sense sufficient for our purposes:
\begin{theorem}\label{thm:Phi-same-integral-H-G}
  For any $a \in \mathcal{S}(\mathfrak{g}^\wedge)$,
  \begin{equation}\label{eqn:Phi-same-integral}
    \frac{1}{\tau(\mathbf{H})}
    \int_{[H]}
    [a]
    = 
    \frac{1}{\tau(\mathbf{G})}
    \int_{[G]}
    [a].
  \end{equation}
\end{theorem}
The proof occupies the remainder of this section.

\begin{remark*}
  The proof makes use of our assumption
  (\S\ref{sec:assumptions-for-main-result}) that $G$ is
  quasi-split.  On the other hand, without this assumption, we
  have $\mathcal{O} = \emptyset$ by
  \S\ref{sec:limit-coadj-orbits-h-indep}, hence $[a] = 0$ by
  Theorem \ref{thm:quantization-of-measures}, so the conclusion
  \eqref{eqn:Phi-same-integral} is uninteresting.
\end{remark*}

\subsection{Reduction to $G_{\infty}^+$-invariance\label{sec:reduce-to-G-plus-invariance}}
\label{sec-22-2}
Set
$F_\infty := F \otimes \mathbb{R}$,
and let $G_{\infty}^+$ denote the
(topological)
connected component
of the group $\mathbf{G}(F_\infty)$ of archimedean
points
of $\mathbf{G}$. Recall that we have extended the rule $a \mapsto [a]$ from 
$\mathcal{S}(\mathfrak{g}^\wedge)$ to measures on $\mathcal{O}$ in  \S\ref{sec:constr-mu}. 
We have the following reduction:

\begin{lemma}\label{21_reduction_lemma}
  Assume that for each $\xi \in \mathcal{O}_{\reg}$,
  the measure $[\delta_\xi]$ is $G_{\infty}^+$-invariant.
  Then the conclusion of theorem
  \ref{thm:Phi-same-integral-H-G} holds.
\end{lemma}

Before we give the proof, we give some algebraic preliminaries. 
Let $\widetilde{\GG}$ denote the simply connected covering
group of the derived group $[\GG, \GG] \subset \GG$.
Because  $\G(F_{\mathfrak{q}})$ is
quasi-split,
each $F$-simple
factor of $\widetilde{\GG}$ has noncompact
$F_{\mathfrak{q}}$-points.
By strong approximation  \cite[Theorem 7.12]{PR},
it follows that
the group $\widetilde{\GG}(F_\infty)$
has dense image in $[\widetilde{\GG}]$.  
The image  $\GG(\adele)^+$ of $\widetilde{\GG}(\adele) \rightarrow \GG(\adele)$
is readily seen to contain the commutator subgroup of $\GG(\adele)$,
and thereby $\GG(\adele)^+$ is a normal subgroup with abelian
quotient. 

Let $U$ be an open compact subgroup of
$\mathbf{G}(\mathbb{A})$ under which $[a]$ is invariant.
By the noted density of
$\widetilde{\GG}(F_\infty)$ in $[\widetilde{\GG}]$,
we see that
\begin{equation}
  \mathbf{G}(F) G_{\infty}^+ U \supseteq \mathbf{G}(\mathbb{A})^+.
\end{equation}

% The image of
% $\widetilde{\GG}(F_\infty)$ in $\GG(F_\infty)$
% is contained in  $G_{\infty}^+$.   We conclude that $\mathbf{G}(F) G_{\infty}^+ U$ contains the image of $\widetilde{\GG}(\adele)$;
% this image is a normal subgroup of $\GG(\adele)$, which we
% call $\GG(\adele)^+$.

\begin{lemma}
  Each $G_{\infty}^+$-orbit on $[\mathbf{G}]/U$ has the same measure.
\end{lemma}

\begin{proof}
  Each such orbit  is of the form
  $\mathbf{G} (F) \backslash \mathbf{G} (F) G_{\infty}^+  g U$,
  for some
  $g \in \G(\adele)$.
  We claim that the various open sets $\mathbf{G} (F) G_{\infty}^+ g U$
  are thus
  right translates of one another, and their  quotients by $\mathbf{G} (F)$ have the same measure. 
  To see this, we note that    $\mathbf{G} (F) G_{\infty}^+ g U$ is
  the right translate by $g$ of  
  $\mathbf{G} (F) G_{\infty}^+ (g U g^{-1})$. 
  This latter set contains $\G(\adele)^+$,  and we deduce that
  \[\mathbf{G} (F) G_{\infty}^+ (g U g^{-1}) =   \mathbf{G} (F)  \mathbf{G} (\mathbb{A})^+ G_{\infty}^+ ( g U g^{-1}) \]  
  which is independent of $g$ 
  because the quotient $\mathbf{G} (\mathbb{A})/\mathbf{G} (\mathbb{A})^+$  is abelian.  
\end{proof} 

We may now give the postponed proof of Lemma \ref{21_reduction_lemma}.

\begin{proof}
  The results of \S\ref{sec:constr-mu},
  imply that, if 
  $a \in \mathcal{S}(\mathfrak{g}^\wedge)$, 
  then 
  $[a]$ is $G_{\infty}^+$-invariant.

  The
  quotient $[\mathbf{G}]/U$ is a finite union of
  $G_{\infty}^+$-orbits, so it is enough to verify
  \eqref{eqn:Phi-same-integral} when the smooth function $[a]$
  is replaced by the characteristic function of one of these
  $G_{\infty}^+$ orbits.
  To that end, it suffices to check
  that
  the intersection of $[H]$ with each one of these $G_{\infty}^+$ orbits
  has equal measure;
  once that is shown, both sides of the
  equality are equal to $1/N$, with $N$ the number of
  $G_{\infty}^+$-orbits.
  % and it is enough to check that the
  % intersection of $[\mathbf{H}]$ with each of these orbits has
  % equal measure.

  % Claim: G(F)/image of \tilde{G}(F) is an abelia group
  % Proof:  We show the image contains the commutator subgroup. Let Z be the center of G, so we have an isogeny Z \times \tilde{G} --> G.  Given g1, g2 in G, we can lift to (z1, g1*) and (z2, g2*) in Z \times \tilde{G}, except z1, z2, g1*, g2* may lie in a larger field L, wlog Galois over F. 
  % Then [g1, g2] = image of [g1*, g2*].
  % But the action of Galois sends g1* and g2* to central multiples of themselves, and so [g1*, g2*] in \tilde{G}(L) is Galois-fixed, so it belongs to \tilde{G}(K); 
  % so [g1, g2] is the image of an element of \tilde{G}(K) as required.
  
  The quotient $\GG(\adele)/\GG(\adele)^+$ is an abelian group.
  Thus the further quotient \[\eps := \mathbf{G}(F) G_{\infty}^+ U \backslash \GG(\adele), \]
  which indexes the  $G_{\infty}^+$-orbits on $[\mathbf{G}]/U$, 
  has the structure of finite abelian group.
  The induced map 
  $\HH(\adele) \rightarrow \eps$ is a group homomorphism.
  The subset of $[H]$ meeting a specified component of $[G]/U$
  corresponds to a fiber of the map $\HH(F) \backslash \HH(\adele) \rightarrow \eps$.
  To see that each such fiber has the same measure,
  it is enough to show that this map $\HH(\adele) \rightarrow \eps$ is surjective.
  
  This surjectivity follows if we can verify  the surjectivity of
  \[\GG(F) \H(\adele) \twoheadrightarrow
  \GG(\adele)/\GG(\adele)^+.\]  Now, defining $\GG(F_v)^+$
  analogously to its adelic counterpart,  split into cases as follows:
  \begin{itemize}
  \item $\GG = \mathrm{SO}_n$: in this case the spinor norm injects $\GG(F_v)/\GG(F_v)^+$
    into $F_v^\times$ modulo squares, and our assertion follows from the surjectivity of the spinor norm on $\mathrm{SO}_{n-1}$
    (for $n-1 \geq 2$). 

  \item  $\GG = \mathrm{U}_n$: in this case the determinant
    injects $\GG(F_v)/\GG(F_v)^+$ into the norm one elements of $(E \otimes F_v)^\times$, 
    where $E$ is the quadratic extension defining the unitary group. 
    Again this determinant is surjective on $\mathrm{U}_{n-1}$,
    whence the conclusion.
  \end{itemize}
\end{proof}

\subsection{Application of Ratner's theorem}
\label{sec-22-4}

We now invoke the full force of our
assumptions (see \S\ref{sec:assumptions-for-main-result}).
Recall, thus, that
\begin{enumerate}
\item[(i)] $\mathbf{G}$ is a reductive group over a
  number field $F$ with adele ring $\mathbb{A}$, 
\item[(ii)]
  $\mathbf{G}$ is anisotropic,
  so that the quotient $[\mathbf{G}] := \mathbf{G}(F)
  \backslash \mathbf{G}(\mathbb{A})$ is compact, and
\item[(iii)]
  $\mathfrak{q}$ is an archimedean place
  for which the local component
  $\mathbf{G}_\mathfrak{q} := \mathbf{G} \times_F
  F_{\mathfrak{q}}$
  is a quasi-split reductive
  group over the archimedean local field $F_\mathfrak{q}$.
  We write
  $G := \mathbf{G}(F_\mathfrak{q})
  = \mathbf{G}_\mathfrak{q}(F_\mathfrak{q})$,
  as usual.
\end{enumerate}
To simplify terminology
in what follows,
we regard $G$ as a {\em real} reductive group;
thus, if $F_{\mathfrak{q}}$ is complex rather than real,
we regard $G$ as the real points
of the real algebraic group $\mathrm{Res}_{F_{\mathfrak{q}}/\R}
(\mathbf{G} \times_{\mathbb{Q}} \mathbb{R})$. 

Let $\xi \in
\mathcal{O} \subseteq \mathcal{N}_{\reg} \subseteq \mathfrak{g}^\wedge$.
As noted above,
our goal reduces to showing
(for each such $\xi$)
that
the measure $[\delta_\xi]$
on $[G]$ is $G_{\infty}^+$-invariant.
 The element $\xi$ is \emph{regular nilpotent}:
its $G$-orbit is an open subset of the nilcone.
Such elements
may be characterized
(with
respect to any $G$-equivariant isomorphism $\mathfrak{g} \cong
\mathfrak{g}^\wedge$)
by the following result of Kostant \cite[Thm 5.3]{MR0114875}:
\begin{lemma}\label{lem:kostant-reg-nil}
  Let $x \in \mathfrak{g}$ be nilpotent.  Let
  $\mathfrak{b} = \mathfrak{t} \oplus \mathfrak{n}$ be a Borel
  subalgebra whose unipotent radical $\mathfrak{n}$ contains
  $x$.  Then $x$ is regular nilpotent if and only if its
  component with respect to each simple root is nonzero.
\end{lemma}
For instance, if $G$ is a general linear group,
then the regular nilpotent elements
are the conjugates of the nilpotent Jordan blocks.

Set
\begin{itemize}
\item[(iv)] $\mathfrak{u} :=$ centralizer
  in $\mathfrak{g} := \Lie(G)$ 
  of the regular nilpotent element $\xi$,
\item[(v)]  $U \leq G$ the connected Lie subgroup
  generated by $\exp(\mathfrak{u})$, and
\item[(vi)] $\mathbf{Z} :=$ center of $\mathbf{G}$.
\end{itemize} 
\begin{example*} 
  Suppose $G = \GL_3(\mathbb{R})$.
  Then we may take
  $\xi$ to correspond
  under the trace pairing
  to an upper-triangular Jordan block,
  in which case
  \begin{equation}\label{eqn:upper-triang-Jordan-block-gl3}
      U = \left\{ \begin{pmatrix}
      a & b & c \\
       & a & b \\
       &  & a
     \end{pmatrix} :
     (a,b,c) \in \mathbb{R}^\times_+ \times \mathbb{R} \times \mathbb{R}
   \right\}.
  \end{equation}
\end{example*}

Under the bijection
$[G] \cong [\mathbf{G}] / K$
following
from
strong approximation and our assumptions on
the set of places $R$,
we may identify $[\delta_\xi]$
with $\mu$, where
\begin{itemize}
\item[(vii)] $\mu$ is a $\mathbf{Z}(\mathbb{A}) \cdot U$-invariant
  probability measure on $[\mathbf{G}]$.
\end{itemize}
In what follows,
 we write $G^+$ and $G_{\infty}^+$
for the topologically connected components of $G$ and $G_{\infty}$.
Our task reduces to establishing the following:
\begin{theorem} \label{invt measure theorem}
  Let notation and assumptions be as in (i)--(vii).
  Then
  $\mu$ is
  $G_{\infty}^+$-invariant.
\end{theorem}
The proof occupies the remainder of this section.

Using the map $[\G]/\mathbf{Z}(\adele) \rightarrow [\G/\mathbf{Z}]$,  
we may replace $\mathbf{G}$ by
$\mathbf{G}/\mathbf{Z}$ to reduce to the case that $\mathbf{G}$ is
semisimple.

It
suffices to establish
the $G_{\infty}^+$-invariance of $\mu$ after
pushforward to the quotient of $[\G]$ by an
open compact subgroup of $\mathbf{G}(\mathbb{A}_f)$.
Any
such quotient is a finite union of spaces of the form
$\Gamma \backslash G_{\infty}$,
where
\begin{itemize}
\item
  $G_{\infty}$ is the Lie group given by
  \begin{equation}\label{eqn:G-inf-q-T}
    G_{\infty}
    := \prod_{\mathfrak{p} \mid \infty }
    G_\mathfrak{p}
    = G \times G_T,
  \end{equation}
  where $\mathfrak{p}$ runs over the archimedean places of $F$
  and
  $T$ denotes the set of archimedean places other than
  $\mathfrak{q}$.
  In other words,
  $G_{\infty} = \mathbf{G}'(\mathbb{R})=\mathbf{G}(F_{\infty})$
  with
  $\mathbf{G}' := \mathrm{Res}_{F/\Q} \mathbf{G}$.
  Note that we may regard
  $U$ as a subgroup of
  $G_{\infty}$.
\item
  $\Gamma \leqslant G_{\infty} $ is an
  arithmetic lattice.
  (The notation
  $\Gamma$ has been used
  differently in other sections,
  but this notational overload
  should introduce no confusion.)
\end{itemize}
It will thus suffice to show that any $U$-invariant probability
measure $\nu$ on $\Gamma \backslash G_\infty$ is in fact
$G_{\infty}^+$-invariant.  In verifying this, we may apply
ergodic decomposition (see \cite[\S 8.7]{Einsiedler-Ward}) to
reduce to the case that $\nu$ is ergodic.  Ratner has proven
\cite{RatnerMeasureClassification} that any ergodic
$U$-invariant $\nu$ is the
$S$-invariant measure on a closed $S$-orbit $xS$, for a closed
connected subgroup $S \leq G_{\infty}$ with $S \supseteq U$.  We
will show eventually that $S=G_{\infty}^{+}$.

The basic idea of the proof is
as follows: the compactness assumption on $[\mathbf{G}]$
will be seen to imply that $S$ contains no nontrivial normal
unipotent
subgroups,
but any connected subgroup $S$ with this property
that contains $U$
must also contain $G^+$.
It may be instructive to consider
the example \eqref{eqn:upper-triang-Jordan-block-gl3}.

Translating $\nu$ by $x^{-1}$, and replacing $U$ by $x U x^{-1}$, we may suppose that $x=e$. 
Then  $\Gamma_S := S \cap \Gamma$ is a lattice inside $S$.
Define the $\mathbb{Q}$-algebraic group
\[
\text{$\mathbf{S} :=$ Zariski-closure of $\Gamma_S$ inside
  $\mathbf{G}'$.}
\]
Recall our convention (\S\ref{sec:general-conventions})
that ``reductive''
is short for ``connected reductive algebraic.''
\begin{lemma}
  $\mathbf{S}$ is reductive,
  and $\mathbf{S}(\mathbb{R})  \geq U$.
\end{lemma}
\begin{proof}
  We show first that
  $\mathbf{S}(\mathbb{R})$
  contains $S$.
  To see this,
  recall
  that the ergodicity of the $U$-action on $\Gamma_S \backslash S$,
  with respect to $\nu$, implies that almost every orbit is
  equidistributed, and in particular dense.
  We may thus find $s \in S$   so that
  $\Gamma_S s U$ is dense in $S$.
  Then
  \begin{equation}\label{eqn:before-Borel-density}
    \text{$\Gamma_S (s U s^{-1}) $ is dense in $S$.}
  \end{equation}
  On the other hand,
  Borel's density theorem \cite[4.7.1]{MR2158954}
  implies that $\mathbf{S}(\mathbb{R})$
  contains all unipotent elements of $S$;
  in particular, 
  $\mathbf{S}(\mathbb{R}) \supseteq s U s^{-1}$.
  By \eqref{eqn:before-Borel-density},
  we deduce that $\mathbf{S}(\mathbb{R}) $ contains a
  dense set of elements of $S$
  and hence $S$ itself, as required.
  
  Since $\mathbf{S}$ contains
  the topologically connected set $S$ as a Zariski-dense subset,
  it follows in particular that $\mathbf{S}$ is connected.

  The unipotent radical of $\mathbf{S}$
  is trivial.  Otherwise
  $\mathbf{S}(\Q)$ would contain nontrivial
  unipotent elements, as would
  $\mathbf{G}'(\Q)=\mathbf{G}(F)$,
  contradicting the
  standard compactness criterion for the compact quotient
  $[\mathbf{G}]$.
\end{proof}

\begin{lemma}\label{lem:S-q-equals-G-prime-q}
  $\mathbf{S}(\mathbb{R}) \geq G$.
\end{lemma}
\begin{proof}
  The splitting
  \eqref{eqn:G-inf-q-T}
  comes from a splitting
  of real algebraic groups
  \begin{equation} \label{Above} \mathbf{G}' \times_{\Q} \R=
    \mathbf{G}'_\mathfrak{q} \times
    \mathbf{G}'_{T},
  \end{equation}
  where $\mathbf{G}'_{\mathfrak{q}}
  = \Res_{F_\mathfrak{q}/\mathbb{R}}(\mathbf{G}_{\mathfrak{q}})$;
  in particular,
  $\mathbf{G}'_{\mathfrak{q}}$ is quasi-split.
  Let
  $\mathbf{S}_\mathfrak{q}$
  denote the kernel of $\mathbf{S} \rightarrow \mathbf{G}_T'$;
  it defines a reductive subgroup of $\mathbf{G}'_{\mathfrak{q}}$.
  It will suffice to verify that $\mathbf{S}_\mathfrak{q} = \mathbf{G}'_\mathfrak{q}$.
  We prove this below,
  in lemma \ref{lem:classify-S-centralizing-reg-unipo},
  in its natural generality.
\end{proof}

\begin{lemma}\label{lem:S-contains-G-plus}
  $S \geq G^+$.
\end{lemma}
\begin{proof}
  Let $\mathfrak{s} \leq \Lie(G_\infty)$ denote the Lie algebra of
  the Lie group $S \leq G_\infty$.
  We may  regard $\mathfrak{g}$ as a subalgebra
  of $\Lie(G_\infty)$.
  We must show that $\mathfrak{s}$ contains
  $\mathfrak{g}$.

  Observe first that $\mathfrak{s}$ is invariant
  under the
  adjoint action of $S$, hence that of $\Gamma_S$.
  Since $\mathbf{S}$ is the Zariski closure of $\Gamma_S$, it follows that $\mathfrak{s}$ is
  invariant under $\mathbf{S}(\mathbb{R})$.  By lemma
  \ref{lem:S-q-equals-G-prime-q}, we deduce that $\mathfrak{s}$ is
  invariant by $\mathbf{G}_\mathfrak{q} '(\mathbb{R}) = G$.
  
  Since $\mathfrak{s}$ is $G$-invariant
  and contains $\mathfrak{u}$,
  we
  deduce
  that
  $\mathfrak{h} := \mathfrak{s} \cap
  \mathfrak{g}$
  is a normal Lie subalgebra of   of $\mathfrak{g}$,
  containing $\mathfrak{u}$. Splitting 
  the semisimple real
  Lie algebra $\mathfrak{g}$
  as a sum $\bigoplus_{i \in I} \mathfrak{g}_i$
  of quasi-split simple real Lie algebras $\mathfrak{g}_i$,
  we must have by normality $\mathfrak{h}  = \bigoplus_{j \in J \subset I} \mathfrak{g}_j$;
  but the projection of $\mathfrak{u}$ to each factor $\mathfrak{g}_i$ is nontrivial,
  so in order
  that $\mathfrak{h}$ contains $\mathfrak{u}$ we must
  have  $J=I$, i.e., $\mathfrak{h} = \mathfrak{g}$. 
  This implies $\mathfrak{s} \leqslant \mathfrak{g}$ as desired.
\end{proof}

%  
%  satisfies the hypotheses of
%  lemma \ref{lem:normal-subalgebras-containing-reg-unips},
%  given below,
%  whose conclusion then gives the required containment.
%\end{proof}
%
%\begin{lemma}\label{lem:normal-subalgebras-containing-reg-unips}
%  Let $\mathfrak{h}$ be a normal Lie
%  subalgebra
%  of $\mathfrak{g}$,
%  with $\mathfrak{h} \geq \mathfrak{u}$.
%  Then $\mathfrak{h}  = \mathfrak{g}$.
%\end{lemma}
%\begin{proof}
%  We may split the semisimple real
%  Lie algebra $\mathfrak{g}$
%  as a sum $\oplus \mathfrak{g}_i$
%  of simple real Lie algebras $\mathfrak{g}_i$,
%  each quasi-split.
%  We accordingly decompose
%  $\mathfrak{u} = \oplus \mathfrak{u}_i$
%  and
%  $\mathfrak{h} = \oplus \mathfrak{h}_i$,
%  with $\mathfrak{h}_i$ normal in $\mathfrak{g}_i$
%  and $\mathfrak{h}_i \geq \mathfrak{u}_i$.
%  The regularity of $\xi$ forces
%  each $\mathfrak{u}_i$, hence each
%  $\mathfrak{h}_i$,
%  to be nonzero.
%  The normality of $\mathfrak{h}$
%  then gives $\mathfrak{h}_i = \mathfrak{g}_i$,
%  whence $\mathfrak{h} = \mathfrak{g}$.
%\end{proof}

\begin{lemma}\label{sec:strong-approx-orbits-dense}
  Any orbit of $G^+$ on any
  connected component
  of $\Gamma \backslash G_{\infty} $ is dense.
\end{lemma}
\begin{proof}
  We apply strong approximation.  Recall that we have reduced to
  the case that $\G$ is semisimple.  Since $G^+$
  is normal in $G_{\infty}$, we reduce further
  to verifying that the $G^+$ orbit of the identity coset in
  $\Gamma \backslash G_{\infty}$ is dense in its connected
  component.
  
  Let $\G^{\sc}$ be the simply connected covering group of
  $\G$.
  By definition, $\G^{\sc}$ is an $F$-algebraic group.
  Every $F$-simple
  factor of $\G^{\sc}$ must have noncompact
  $F_{\mathfrak{q}}$-points, because  $\G(F_{\mathfrak{q}})$ is
  quasi-split.  
  By strong approximation  \cite[Theorem 7.12]{PR},
  the orbits of $\G^{\sc}(F_{\mathfrak{q}})$ 
  on $[\G^{\sc}]$ are dense; in particular,
  the closure of each
  $\G^{\sc}(F_{\mathfrak{q}})$-orbit
  is actually stable under
  $\G^{\sc}(F_\infty)$.
  By a theorem of Cartan,
  the groups $\G^{\sc}(F_{\mathfrak{q}})$
  and
  $\G^{\sc}(F_\infty)$
  are topologically connected.
  Their images in $\mathbf{G}(\mathbb{A})$
  are thus the connected components $G^+$ and
  and $G_{\infty}^+$, and the desired conclusion readily follows. 
\end{proof}

By lemmas \ref{lem:S-contains-G-plus} and
\ref{sec:strong-approx-orbits-dense}, we see that the closed
orbit $\Gamma_S \backslash S$ coincides with a component of
$\Gamma \backslash G_{\infty}$, hence that $S = G_{\infty}^{+}$.
This completes the proof of theorem \ref{invt measure theorem}
modulo the following lemma, postponed above, which we have
found
convenient to formulate using notation independent from that in
the rest of this section.

\begin{lemma}\label{lem:classify-S-centralizing-reg-unipo}
  Let $\mathbf{G}$ be
  a quasi-split  real reductive group, 
  and let
  $\mathbf{S} \subset \mathbf{G}$
  be
  a real reductive algebraic subgroup.
  (By our usual
  conventions,  $\mathbf{G}$ and $\mathbf{S}$ are connected.)
  Let
  $\mathfrak{g} = \Lie(\mathbf{G}(F))$
  and  $\mathfrak{s} := \Lie(\mathbf{S}(F))$
  denote the corresponding Lie algebras.
  Let $\xi$ be a regular nilpotent element of $\mathfrak{g}^*$,
  with centralizer $\mathfrak{g}_\xi \leq \mathfrak{g}$.
  Suppose that $\mathfrak{s} \geq \mathfrak{g}_\xi$.
  Then $\mathbf{S}=\mathbf{G}$.
\end{lemma}

\proof
We may without loss replace $\mathbf{G}, \mathbf{S}, \mathfrak{g}, \mathfrak{s}$ 
by their complexifications.  
Since $\mathbf{S}$ is connected,
it is enough to show that $\mathfrak{s} = \mathfrak{g}$.

We may identify the coadjoint action
on $\mathfrak{g}^*$ with
the adjoint action on $\mathfrak{g}$,
via the trace pairing with respect
to a faithful linear representation.
Thus $\xi \in \mathfrak{g}$.
Our hypothesis implies
also that $\xi \in \mathfrak{s}$,
so it makes sense to speak of the regularity
of $\xi$ both with respect to $\mathbf{S}$ and $\mathbf{G}$.

The regularity of $\xi$ is equivalent  (by Springer, Kurtzke, see \cite{Kurtzke})
to its Lie algebra centralizer
being abelian.
Since $\xi$ is regular for $\mathbf{G}$,
it follows that it is likewise regular for $\mathbf{S}$.
Thus  
\[
\mathrm{rank}(\mathbf{S}) = \dim(\mathfrak{g}_\xi)
= \mathrm{rank}(\mathbf{G}),  
\]
that is to say, $\mathbf{S}$
is an equal rank subgroup of $\mathbf{G}$.

Since $\xi \in \mathfrak{s}$ is nilpotent,
we may find
\begin{itemize}
\item a maximal torus $\mathbf{T} \leq \mathbf{S}$,
  with Lie algebra $\mathfrak{t} \subseteq \mathfrak{s}$, and
\item a Borel subalgebra $\mathfrak{t} \oplus \mathfrak{n}$
  of $\mathfrak{s}$
\end{itemize}
so that $\xi \in \mathfrak{n}$.
Since $\mathbf{S}$ and $\mathbf{G}$ have equal rank,
$\mathbf{T}$ is likewise a maximal torus for $\mathbf{G}$.
Let $\Phi$ denote the set of roots for $\mathbf{T}$
acting on $\mathfrak{g}$,
so that
$\mathfrak{g} = \mathfrak{t} \oplus (\oplus_{\alpha \in \Phi}
\mathfrak{g}_\alpha)$,
with each $\mathfrak{g}_\alpha$ one-dimensional.
Since $\mathfrak{s}$ is a $\mathfrak{t}$-stable
subspace of $\mathfrak{g}$ that contains
$\mathfrak{t}$,
we have $\mathfrak{s} = \mathfrak{t} \oplus (\oplus_{\alpha \in
  \Phi '} \mathfrak{g}_\alpha)$
for some $\Phi ' \subseteq \Phi$.
We must show that $\Phi' = \Phi$.

We may find a positive system
$\Phi^+ \subseteq \Phi$
for which $\mathfrak{n} \subseteq \sum_{\alpha \in \Phi^+}
\mathfrak{g}_\alpha$
(e.g.,
by considering a generic $1$-parameter subgroup
of $\mathbf{T}$ having positive
eigenvalues on $\mathfrak{n}$).
In particular, $\xi = \sum_{\alpha \in \Phi^+} \xi_\alpha$,
with $\xi_\alpha \in  \mathfrak{g}_\alpha$.
Let $\Delta \subseteq \Phi^+$
denote the corresponding simple system.
Since $\xi$ is regular nilpotent for $\mathfrak{g}$,
we have by lemma \ref{lem:kostant-reg-nil}
that
$\xi_\alpha \neq 0$ for all $\alpha \in \Delta$.
Since $\xi \in \mathfrak{s}$,
we deduce that $\Phi '$ contains $\Delta$.
Since $\mathbf{S}$ is reductive,
we conclude that $\Phi ' = \Phi$.

\qed

\section{Recap and overview of the proof}\label{sec:recap-overview-proof}
 
We pause to recall
(from \S\ref{sec:setting-for-main-result})
some aspects of the basic setup, as well as what
we have shown thus far. We will then discuss how the remainder of the proof proceeds from this point. 

\subsection{Recap }
We are considering an inclusion of compact quotients
\[
[H] = \Gamma_H \backslash (H \times H')
\hookrightarrow [G] = \Gamma \backslash (G \times G')
\]
arising from a Gross--Prasad pair; $H < G$ and $H' < G'$ are
inclusions of real reductive and $S$-arithmetic groups,
respectively.  From the fixed automorphic representation $\Pi$
of $\mathbf{G}$ we obtained a $G \times G'$-subrepresentation
$\pi \otimes \pi ' \hookrightarrow L^2([G])$.  From a fixed
positive-definite tensor $T' \in \pi ' \otimes \overline{\pi '}$
of trace $1$, we constructed in
\S\ref{sec:limit-stat-attach-symb} an assignment from Schwartz
functions
$a \in \mathcal{S}(\mathfrak{g}^\wedge)$ to
``limit states'' $[a] \in C^\infty([G])$, describing the
limiting behavior of $L^2$-masses of certain families of vectors
in $\pi \otimes \pi '$.

Each limit state $[a]$, for
(say) real-valued $a$, is (informally) a
limiting weighted average of
$L^2$-masses of vectors $v \otimes v' \in \pi \otimes \pi '$
for which $v$ is microlocalized within the support of
$a|_{\mathcal{O}}$.  Thus $\int_{[H]} [a]$ is an average of
integrals $\int_{[H]} |v \otimes v'|^2$, each of which may be
decomposed spectrally as a sum of contributions
$\mathcal{P}_{\Sigma}(v \otimes v')$ from automorphic
representations $\Sigma$ as above;
more generally and precisely,
we denote by\index{$\mathcal{P}_{\Sigma}$}
\begin{equation}
  \mathcal{P}_{\Sigma} : \Psi^{-\infty} \rightarrow \mathbb{C}
\end{equation}
the (uniformly in $\Sigma$) continuous maps
obtained by the following composition:
\begin{equation}
  \Psi^{-\infty} \xrightarrow{T \mapsto T \otimes T'}
  C^\infty([G]^2)
  \xrightarrow{\text{restrict}}
  C^\infty([H]^2)
  \xrightarrow{\text{project}}
  (
  \Sigma_R
  \hat{\otimes}
  {\Sigma^\vee_R }
  )^\infty
  \xrightarrow{\int_{[{H}]}}
  \mathbb{C}.
\end{equation}
We retain the convention
from \S\ref{sec:reduction-of-pfs-quantization}
of dropping $T'$ from the notation;
when we wish to indicate it explicitly,
we write
$\mathcal{P}_{\Sigma}(T \otimes T')$. 
For example,
if
$T = v \otimes \overline{v}$
and $T' = w \otimes \overline{w}$,
so that $u := v \otimes w \in \pi \otimes \pi ' \hookrightarrow
C^\infty([G])$,
then
\[
\mathcal{P}_{\Sigma}(T)
= \|\text{projection to $\Sigma$ of the restriction $u|_{[H]}$}\|^2.
\]
We will verify
below that for any $T \in \Psi^{-\infty}$,
\begin{equation}\label{eqn:interchange-H-period-with-spectral-decomp}
  \int_{[H]}
  [T]
  = \sum_{\Sigma}
  \mathcal{P}_{\Sigma}(T),
\end{equation}
where
we sum
over
$\Sigma$ as in \S\ref{sec-18-3}.
In particular,
exchanging the limit
$[a] = \lim_{\h \rightarrow 0} [a]_{\h}$
taken in $C^\infty([G])$
with integration
over the compact set $[H]$
gives
\begin{equation}\label{eqn:interchange-H-period-with-spectral-decomp-2}
  \int_{[H]}
  [a]
  =
  \lim_{\h \rightarrow 0}
  \h^d
  \sum_{\Sigma}
  \mathcal{P}_{\Sigma}(\Opp_{\h}(a)).
\end{equation}
The period formula
\eqref{eqn:I-I-explicated-for-P-Sigma}
says that,
under assumptions on $\Pi$ and $\Sigma$
to be recalled below,
\begin{equation}\label{eqn:I-I-explicated-for-P-Sigma}
  \mathcal{P}_{\Sigma}(T)
  = \mathcal{L}(\Pi,\Sigma) \mathcal{H}_{\sigma}(T) \mathcal{H}_{\sigma '}(T'),
\end{equation}
where the maps
$\mathcal{H}_{\sigma} : \Psi^{-\infty} \rightarrow \mathbb{C},
\quad
\mathcal{H}_{\sigma '} : \pi ' \otimes \overline{\pi '}
\rightarrow \mathbb{C}$ 
are as in  \S\ref{sec:local-disintegration}.
To summarize, then
\begin{equation}\label{eqn:I-I-explicated-for-P-Sigma2}
  \lim_{\h \rightarrow 0}
  \quad
  \h^{d}
\sum_{\Sigma: \text{\eqref{eqn:I-I-explicated-for-P-Sigma}
    holds}} \mathcal{L}(\Pi,\Sigma) 
\mathcal{H}_{\sigma}(\Opp_{\h}(a)) \mathcal{H}_{\sigma '}(T')  + \sum_{\text{remaining } \Sigma}= \int_{[H]} [a].
\end{equation}

Moreover, we showed already in \S\ref{sec:appl-meas-class}
-- using Ratner's theorem -- that the
limit state construction $a \mapsto [a]$ is \emph{essentially trivial}
in the examples of interest: modulo
``connectedness issues,'' it produces constant functions on
$[G]$, and the integral $\int_{[H]} a$ is proportional to the trace of $T \otimes T'$. 
Because of this, 
 \eqref{eqn:I-I-explicated-for-P-Sigma2} gives an asymptotic of the desired nature, but
it still requires some cleanup.

 \subsection{Cleanup}

 We now outline what is required to  massage
 \eqref{eqn:I-I-explicated-for-P-Sigma2} into the required shape
(cf.  \S\ref{sec:rough-idea-proof},
\S\ref{sec:intro-local-issues}). The most important issue is, of course, to 
  choose $T$ and $T'$ so that $\sigma \mapsto
  \mathcal{H}_{\sigma}(T)$
  and $T' \mapsto \mathcal{H}_{\sigma'}(T')$
approximate desired test functions on the unitary duals of $H$ and $H'$. This problem has already been solved, at least enough for our purposes, in Part IV of this paper
addressing inverse branching.    We focus on the other issues that arise (although we will recall a part of this analysis below):

The  second summand of \eqref{eqn:I-I-explicated-for-P-Sigma2}
arises from situations where  $\mathcal{L}(\Pi,\Sigma)$ is undefined, in particular:
\begin{itemize}
\item[(i)] either $\sigma$ or $\sigma'$ is non-tempered.  
\item[(ii)] $\mathcal{H}_{\sigma} =0$
  (even though $(\pi,\sigma)$ is distinguished).
\end{itemize}
The first of these represents an actual possibility that does
occur in practice (e.g., when $\Sigma$ is the trivial
representation!); the second is not expected to occur, but has
not yet been ruled out in the literature.
Since we aim in this paper to prove an unconditional
theorem,
we must show that these bad cases
yield negligible  contributions
to \eqref{eqn:interchange-H-period-with-spectral-decomp-2}.

The machinery to handle (i) was already set up in the course of our analysis of inverse branching. Informally speaking, 
we may construct $\widetilde{T}$ and $\widetilde{T}'$ with the property that
$\mathcal{P}_{\Sigma}(\widetilde{T} \otimes \widetilde{T}')$
majorizes
$\mathcal{P}_{\Sigma}(T \otimes T')$
on the non-tempered spectrum, but such that the trace of
$\widetilde{T} \otimes \widetilde{T'}$
is small. In the archimedean case, for example, 
this was proved in 
theorem \ref{thm:arch-inv-branch} by cutting off the symbol $a$ to a small neighbourhood of the locus of irregular infinitesimal characters.
In any case, once $\widetilde{T}$ and $\widetilde{T}'$ have been constructed, we just apply 
\eqref{eqn:interchange-H-period-with-spectral-decomp-2} to control the total contribution of the non-tempered spectrum.

Part (ii) is dealt with by means of symbol calculus: such $(\pi, \sigma)$ cannot be orbit-distinguished (in the sense of
\S \ref{sec:asymp-distinction}). Now orbit-distinction is the
semiclassical analogue of
distinction,
and if $(\pi, \sigma)$ fails to be orbit-distinguished,
then microlocal analysis shows that any $H$-invariant
functional on $\pi$ must
at least be negligibly small (in a suitable sense).
To formalize this intuition with symbol calculus,
we construct a suitable symbol on $\mathfrak{h}^{\wedge}$ with the property that  it misses
entirely the orbit for $\sigma$, but on the other hand is identically $1$ on the support of $a$ (or rather its projection
under $\mathfrak{g}^{\wedge} \rightarrow \mathfrak{h}^{\wedge}$); see
theorem \ref{thm:trunc-H-expn} for further details.

There remain other contributions to clean up, e.g., from when
\begin{itemize}
\item[(iii)] $\h \lambda_\sigma$ is very large, or
\item[(iv)] the pair $(\lambda_\pi,\lambda_\sigma)$ is
  not $\mathbf{H}$-stable (cf. \S\ref{sec:stability}),
  in which case the asymptotic formulas
  of \S\ref{sec:sph-char-2}
  for $\mathcal{H}_\sigma$ do not apply.
\end{itemize}
We do not discuss these in detail here, but just observe that
much of the difficulty of (iv) is avoided  by fiat: we consider only those test functions
on the unitary dual of $H$
which are supported
above
the stable locus.

As should be clear from this outline, all of this analysis makes heavy use of the microlocal calculus from Parts
\ref{part:micr-analys-lie} and \ref{part:micr-analys-lie2}. 
\section{Spectral expansion and truncation
  of the $H$-period\label{sec:trunc-H-expn-overview}}

\subsection{Spectral
  decomposition\label{sec:spec-decomp-initial-H-period}}
Here we verify the Parseval-type identity
\eqref{eqn:interchange-H-period-with-spectral-decomp}.
  More generally,
  let $f$ be a smooth function on $[H] \times [H]$.
  Then
  $f = \sum_{\Sigma_1,\Sigma_2} f_{{\Sigma_1} \otimes {\Sigma^\vee_2}}$,
  where
  $f_{{\Sigma_1} \otimes {\Sigma^\vee_2}}$ 
  belongs to the $J \times J$-fixed subspace of 
  the smooth completion of ${\Sigma_1} \otimes {\Sigma^\vee_2}$,
  and the sum converges in the $C^\infty([H] \times [H])$-topology,
  hence commutes with integration over the diagonal
  copy of $[H]$ inside $[H] \times [H]$:
  \begin{equation}\label{eqn:spectral-decomp-H-times-H-integrated}
    \int_{[H]} f
    = \sum_{\Sigma} \int_{[H]} f_{{\Sigma_1} \otimes {\Sigma^\vee_2}}.
  \end{equation}
  Taking for $f$ the restriction of $T \otimes T'$ to $[H] \times [H]$,
  one obtains $f|_{[H]} = [T]|_{[H]}$
  and $\int_{[H]} f_{\Sigma \otimes \Sigma^\vee} = \mathcal{P}_{\Sigma}(T)$,
  so \eqref{eqn:spectral-decomp-H-times-H-integrated}
  specializes to \eqref{eqn:interchange-H-period-with-spectral-decomp}.

\subsection{Weyl law upper bound\label{sec:weyl-upper-bnd}}
\label{sec-20-5}
The map $\mathcal{P}_{\Sigma}$ is identically zero
unless $\sigma '$
is \emph{$T'$-distinguished}
in the sense that there is a nonzero
equivariant map
$\pi \otimes \overline{\pi '} \rightarrow \sigma
\otimes  \overline{\sigma '}$
and it does not vanish on $T'$.
By our 
assumption
that $G_\mathfrak{p}$ compact for all archimedean $\mathfrak{p} \neq \mathfrak{q}$,
the latter condition forces $\sigma'$
to belong to some compact subset of the unitary dual of
$H'$ depending only upon $T'$.
A weak form of the Weyl law
then reads: for $x \geq 1$,
\begin{equation}\label{eq:weak-weyl-law}
  \# \{\Sigma : \mathcal{P}_{{\Sigma}} \neq 0, |\lambda_\sigma| \leq
  x   \} \ll x^{ \O(1)}.
\end{equation}
This follows from the {\em usual} Weyl law on $[H]$,
using \S\ref{ss:norms}.

\subsection{Truncated spectral
  decomposition}\label{sec:trunc-spectr-decomp-overview}
The main result of this section is the following.
As explained above,
it allows us to discard all terms in the spectral expansion
\eqref{eqn:interchange-H-period-with-spectral-decomp-2} with large
eigenvalue,
and also those microlocally separated from the symbol $a$.

Let $U$ be an open subset of $[\mathfrak{h}^\wedge]$.
Let $W \subseteq \mathfrak{h}^\wedge$ be an open subset
of the preimage of $U$.
We say that $\sigma$ is \emph{bad}
(relative to the scaling parameter $\h$
and the choice of $U$ and $W$)
if
either
\begin{itemize}
\item $\h \lambda_\sigma \notin U$, or
\item $\sigma$ is tempered and
  $\h \mathcal{O}_\sigma \cap W
  =\emptyset$.
\end{itemize}
We say otherwise that $\sigma$ is \emph{good}.

\begin{theorem}\label{thm:trunc-H-expn}
  Let $a \in C_c^\infty(\mathfrak{g}^\wedge)$
  be supported in the preimage
  of $W$.
  Then
  \begin{equation}\label{eq:trunc-spectr-decomp-k-equals-1}
    \int_{[H]} [a]
    =
    \lim_{\h \rightarrow 0}
    \h^d
    \sum _{
        \Sigma : \text{$\sigma$ is good}
    }
    \mathcal{P}_{\Sigma}(\Opp_{\h}(a)).
  \end{equation}
  More generally,
  for $a_1,\dotsc,a_k$ satisfying
  the same assumptions as $a$,
  \begin{equation}\label{eq:trunc-spectr-decomp-general-k}
    \int_{[H]} [a_1 \dotsb a_k]
    =
    \lim_{\h \rightarrow 0}
    \h^d
    \sum _{
      \Sigma : \text{$\sigma$ is good}
    }
    \mathcal{P}_{\Sigma}(\Opp_{\h}(a_1) \dotsb \Opp_{\h}(a_k)).
  \end{equation}
\end{theorem}
\begin{proof}
  We
  note that
  $\mathcal{P}_{\Sigma}$
  factors as an $(H \times H)$-equivariant composition
  of $\h$-uniformly continuous maps  
  \begin{equation}\label{eq:continuity-of-P-Sigma}
      \Psi^{-\infty} = \Psi^{-\infty}(\pi)
      \rightarrow \Psi^{-\infty}(\sigma) \xrightarrow{\trace_\sigma} \mathbf{C}.
  \end{equation}
  We may choose $\eps > 0$,
  depending only upon the support of $a$,
  so that
  \[
    \xi \in \supp(a),
    \quad
    \h \lambda_\sigma \notin U
    \implies
    \dist([\xi|_{\mathfrak{h}}], \h \lambda_\sigma) \geq \eps.
  \]
  By \eqref{eq:some-decay-sym-2},
  it follows
  that
  for
  each fixed $N \in \mathbb{Z}_{\geq 0}$,
  \begin{equation}\label{eq:negligible-outside-U}
    \h \lambda_\sigma \notin U
        \implies
    \mathcal{P}_{\Sigma}(\Opp_{\h}(a))
    \ll
    \h^N \langle \h \lambda_\sigma \rangle^{-N}.
  \end{equation}
  From this and \eqref{eq:weak-weyl-law} we see that the
  contribution to
  \eqref{eqn:interchange-H-period-with-spectral-decomp-2} from
  those $\Sigma$ with $\h \lambda_\sigma \notin U$ is
  negligible.  
  
  It remains to estimate the contribution from when
  $\sigma$ is tempered and
  $\h \lambda_\sigma \in U$ but
  $\h \mathcal{O}_\sigma \cap W = \emptyset$.  By
  \eqref{eq:weak-weyl-law}, the number of such $\Sigma$ is
  $\h^{-\O(1)}$, so it will suffice to show for each such
  $\Sigma$ that
  $\mathcal{P}_{\Sigma}(\Opp_{\h}(a)) = \O(\h^{\infty})$.

  The image in $\mathfrak{h}^\wedge$ of the support of $a$ is a
  compact subset of $W$.  By the smooth version of Urysohn's
  lemma, we may thus choose a real-valued $b \in C_c^\infty(W)$
  so that $b \equiv 1$ on the image of the support of $a$.  By
  the composition formula
  \eqref{eqn:comp-with-remainder-J-diff-subspaces}, we
  then have
  \[
  \Opp_{\h}(a) \equiv
  \Opp_{\h}(b:\pi) \Opp_{\h}(a)
  \mod{\h^{\infty} \Psi^{-\infty}}.
  \]
  By the continuity of
  \eqref{eq:continuity-of-P-Sigma}, it follows that
  \[
  \mathcal{P}_{\Sigma}(\Opp_{\h}(a)) =
  \mathcal{P}_{\Sigma}(\Opp_{\h}(b:\pi) \Opp_{\h}(a)) +
  \O(\h^\infty).
  \]
  By the equivariance of
  \eqref{eq:continuity-of-P-Sigma}, we have
  $\mathcal{P}_{\Sigma}(\Opp_{\h}(b:\pi) \Opp_{\h}(a)) =
  \trace(T_1 T_2)$, where
  $T_1 := \Opp_{\h}(b:\sigma)$
  and
  $T_2 \in \Psi^{-\infty}(\sigma)$ denotes the image of
  $\Opp_{\h}(a)$.  By Cauchy--Schwarz for the Hilbert--Schmidt
  inner product, we have
  $|\trace(T_1 T_2)|^2 \leq
  \trace(T_1^* T_2) \trace(T_2^* T_2)$.  Using
  the composition formula, the Kirillov
  formula for $\sigma$ and the assumption on the support of $b$,
  we see that $\trace(T_1^* T_1) \ll \h^{\infty}$.
  On the other hand, we obtain using
  theorem
  \ref{thm:rescaled-operator-memb}
  and
  \eqref{eq:Psi-bounds-trace-norms} and the continuity of
  \eqref{eq:continuity-of-P-Sigma} that
  $\trace(T_2^* T_2) \ll \h^{-\O(1)}$.

  This completes the proof of
  \eqref{eq:trunc-spectr-decomp-k-equals-1}.
  
  For \eqref{eq:trunc-spectr-decomp-general-k},
  note first by \eqref{eq:characterizing-bracket-map-recap}
  and
  \eqref{eqn:interchange-H-period-with-spectral-decomp-2}
  that
  \[
  \int_{[H]} [a_1 \dotsb a_k]
  = \lim_{\h \rightarrow 0}
  \h^d \sum_{\Sigma} \mathcal{P}_{\Sigma}(\Opp_{\h}(a_1) \dotsb \Opp_{\h}(a_k)).
  \]
  We fix
  $N \geq 0$ large enough and apply the composition formula
  \eqref{eqn:comp-with-remainder-J} to write
  \[\Opp_{\h}(a_1) \dotsb \Opp_{\h}(a_k) \equiv \sum_{0 \leq j <
    J} \h^j \Opp_{\h}(b_j) \mod{\h^N \Psi^{-N}},\]
  where $J$ is large but fixed and the $b_j$ satisfy the same
  assumptions as the $a_i$.  We apply
  \eqref{eq:some-decay-sym-2} as before to the contribution from
  the $\Opp_{\h}(b_j)$.  We clean up the remainder contribution
  using \eqref{eq:some-decay-sym-1}.
\end{proof}

\section{The smoothly weighted asymptotic formula}\label{sec-18}
\subsection{Overview}
We retain the notation and setup of
\S\ref{sec:setting-for-main-result}.
This section contains the   main automorphic result in this paper:
an asymptotic formula
for the averaged GGP branching coefficient
$\mathcal{L}(\Pi,\Sigma)$,
with the automorphic representation $\Pi$
of $\mathbf{G}$ fixed and the automorphic representation
$\Sigma$
of $\mathbf{H}$ traversing
a smoothly-weighted family.
(We will refine this in minor ways
in \S\ref{sec:norm-asympt-form},
by extending to unweighted families
and then dividing through by their cardinalities.)
The proof
involves three main inputs developed hitherto:
\begin{itemize}
\item The
  ``inverse branching'' results of
  \S\ref{sec:inv-branch}
  (especially \S\ref{sec:main-result-inv-branch-arch}
  and \S\ref{sec:padic-inverse-branching-main-results}),
  which allow us to pick off any reasonable
  family of representations $\Sigma$,
  say defined by weights $w(\Sigma)$, using
  a suitable family of vectors $v \in \Pi$:
  \begin{equation}\label{eqn:summation-v-smooth-weight-premable}
      \int_{[H]} |v|^2 \approx \sum_{\Sigma}
  \mathcal{L}(\Pi,\Sigma) w(\Sigma)
  \text{ on average over } v.
  \end{equation}
\item The ``truncated spectral formula'' of
  \S\ref{sec:trunc-H-expn-overview}, which allows us to
  discard some of the contributions of
  undesirable automorphic forms $\Sigma$
  implicit in
  \eqref{eqn:summation-v-smooth-weight-premable}.
\item The ``equidistribution'' result, proved in
  \S\ref{sec:appl-meas-class}
  using Ratner's theorem,
  by which
  we deduce that (on average)
  \[
  \frac{1}{\vol([H])} \int_{[H]} |v|^2
  \approx 
  \frac{1}{\vol([G])} \int_{[G]} |v|^2,
  \]
  leading to the required asymptotic formula
  for $\sum_{\Sigma} \mathcal{L}(\Pi,\Sigma) w(\Sigma)$.
\end{itemize}

\subsection{Function spaces}\label{sec:function-spaces}
We now define the precise sets of weights $w(\Sigma)$
to be summed against.

\subsubsection{Spaces of representations of $H$}
We may apply the notation of
\S\ref{sec:inv-branch-overview}
to each place $\mathfrak{p}$ of $F$,
thus the sets
\[
\hat{H}_\mathfrak{p}
\supseteq 
(\hat{H}_{\mathfrak{p}})_{\temp}
\supseteq
(\hat{H}_{\mathfrak{p}})^{\pi_\mathfrak{p}}_{\temp}
\]
are respectively
the unitary dual of $\hat{H}_\mathfrak{p}$,
the tempered dual
and
its $\pi_\mathfrak{p}$-distinguished subset.
We omit the index when $\mathfrak{p} = \mathfrak{q}$,
and use a superscripted prime to denote a product over all
$\mathfrak{p} \in R - \{\mathfrak{q} \}$.

\subsubsection{The distinguished archimedean place}
Recall (\S\ref{sec:quotient-affine}) that the geometric
quotient $[\mathfrak{h}^\wedge]$ of
$\mathfrak{h}^\wedge = i \mathfrak{h}^*$ is isomorphic to
an affine space.
We
denote as in \S\ref{sec:inverse-branching-dist-place} by
$\mathcal{O}_{\stab} \subseteq \mathcal{O} \subseteq
\mathfrak{g}^\wedge$
the subset of
$\mathbf{H}$-stable elements
of
the
limit coadjoint orbit $\mathcal{O}$ of $\pi$,
and set
\[
\mathcal{K} := C_c^\infty([\mathfrak{h}^\wedge] \cap
\image(\mathcal{O}_{\stab})).
\]
In view of the stability characterization given in
\S\ref{sec:stab-terms-spectra}, an element of $\mathcal{K}$ is
just a smooth compactly-supported function $k$ on the space
\[ \{ \mu \in [\mathfrak{h}^{\wedge}]: \mbox{no eigenvalue of $\mu$ equals zero} \}.\]
(We used the fact that $\mathcal{O}$ is contained in the nilcone.)
%
%$[\mathfrak{h}^\wedge]$ of infinitesimal characters $\mu$ for
%$H$ which is supported away from the locus where some eigenvalue
%(cf. \S\ref{sec:eigenv-git-quot})
%of $\mu$ coincides with some eigenvalue of the image
%$\lambda_{\mathcal{O}} \in [\mathfrak{g}^\wedge]$ of 
%$\mathcal{O}$.
%Since $\mathcal{O}$ is contained in the nilcone,
%we have $\lambda_{\mathcal{O}} = 0$,
%so the support condition on $k$
%may be rephrased:
%\begin{center}
%  On the support of $k$, every eigenvalue is nonzero.
%\end{center}

As in \S\ref{sec:main-result-inv-branch-arch},
we assign to
each fixed $k \in \mathcal{K}$ and all
sufficiently small $\h > 0$ a function
\index{rescaled test function $k_{\h}$}
\[
k_{\h} : \hat{H} \rightarrow \mathbb{C},
\]
as follows:
 
\begin{itemize}
\item If $\sigma$ is tempered and
  $\mathcal{O}_{\pi,\sigma} \neq \emptyset$
  (i.e.,
  $\sigma$ is ``orbit-distinguished'' by $\pi$;
  see \S\ref{sec:asymp-distinction}),
  then we evaluate on the rescaled infinitesimal character:
  \[
  k_{\h}(\sigma) := k(\h \lambda_\sigma).
  \]
\item Otherwise, we set $k_{\h}(\sigma) := 0$.
\end{itemize}

\subsubsection{The auxiliary archimedean places}
Let $\mathfrak{p}$ be an archimedean place other than $\mathfrak{q}$.
Since the groups $(G_\mathfrak{p},
H_\mathfrak{p})$
are assumed compact,
the set $(\hat{H}_\mathfrak{p})_{\temp}^{\pi_\mathfrak{p}}$
is finite.
We denote by $\mathcal{K}_\mathfrak{p}$
the set of complex-valued functions
$k_\mathfrak{p} :
(\hat{H}_\mathfrak{p})_{\temp}^{\pi_\mathfrak{p}} \rightarrow
\mathbb{C}$.
(Compare with
\S\ref{sec:compact-groups}.)

\subsubsection{The auxiliary $p$-adic places}
Let $\mathfrak{p}$ be a finite place in $R$.  We denote by
$\mathcal{K}_\mathfrak{p}$ the space of allowable functions
$k_\mathfrak{p} :
(\hat{H}_\mathfrak{p})_{\temp}^{\pi_\mathfrak{p}} \rightarrow
\mathbb{C}$ (see
\S\ref{sec:padic-inverse-branching-main-results} for the
definition of ``allowable,'' and recall that Theorem
\ref{thm:consequences-of-allowability},
stated in
\S\ref{sec:analys-regul-comp}, provides a large supply of
allowable functions).

\subsubsection{The auxiliary places, grouped together}
We denote by $\mathcal{K} '$
the space of functions
\[
  k' :
  (\hat{H'})_{\temp}^{\pi'}
  \rightarrow
  \mathbb{C}
\]
spanned by
the pure tensors
$k'(x) := \prod _{\mathfrak{p} \in R - \{\mathfrak{q} \}}
k_\mathfrak{p}(x_\mathfrak{p})$
for $k_\mathfrak{p} \in \mathcal{K}_\mathfrak{p}$.
We extend each such $k'$ by zero
to a function on the unitary dual $\hat{H'}$ of $H'$.

\subsection{Main result}\label{sec:fundamental-main-result}
We denote by $d \in \mathbb{Z}_{\geq 0}$
half the real dimension of $\mathcal{O}$,
as usual (see \S\ref{analytic-conductors} for numerics).
We fix $k \in \mathcal{K}, k' \in \mathcal{K} '$,
and set
\[
  \ell :=
  \h^d
  \sum_{
    \substack{
      \sigma, \sigma ' \text{ tempered} \\
      \mathcal{O}_{\pi,\sigma} \neq \emptyset 
    }
  }
  \mathcal{L}(\Pi,\Sigma)
  k_{\h}(\sigma)
  k'(\sigma').
\]
Here and henceforth we sum over automorphic representation
$\Sigma$ of $\mathbf{H}$, as in
\S\ref{sec:setting-for-main-result}, whose components
$\sigma, \sigma'$ satisfy the displayed conditions,
so that $\mathcal{L}(\Pi,\Sigma)$ is defined.

We write $\int k$ and $\int k'$ for integrals
taken with respect to the normalized affine measure on
$[\mathfrak{h}^\wedge]$
and the Plancherel measure on $\hat{H'}_{\temp}$,
respectively.
We write $A \simeq B$
for $A = B + o_{\h \rightarrow  0}(1)$.  
\begin{theorem}\label{thm:fundamental}
  \[
    \ell \simeq \frac{\tau(\mathbf{H})}{\tau(\mathbf{G})}
    \int k \int k'.
  \]
\end{theorem}

The proof occupies the remainder of this section.
We may and shall assume that $k, k'$ are
\emph{nonnegative}.

It will be convenient
to introduce the following otherwise unusual notation.
\begin{definition*}
  Let $p : X \rightarrow Y$ be a continuous
  map between topological spaces.
  Let $U \subseteq X$ and $V \subseteq Y$ be subsets.
  We write $U \prec V$
  if
  $\overline{p(U)} \subseteq V^0$,
  and
  similarly $V \prec U$
  if
  $\overline{V} \subseteq p(U)^0$.
  (Here $\overline{V}$ and $V^0$
  denote closure and interior.)
\end{definition*}
The precise choice of $p$ (typically a ``projection'')
to which one should apply this
notation should be clear by context in what follows.
We caution that $\prec$ is not
in any sense ``transitive.''
 \begin{lemma}
  There are precompact open subsets
  $U \subset [\mathfrak{h}^\wedge]$,
  $W \subseteq \mathfrak{h}^\wedge$,
  $V \subset \mathfrak{g}^\wedge$
  so that
  for small enough $\h > 0$,
  \begin{enumerate}[(i)]
  \item $\overline{V}$ consists of $\mathbf{H}$-stable
    elements,
  \item
      $\supp(k)
    \prec
    V
    \prec
    W
    \prec
    U
    \prec
    \mathcal{O}_{\stab}$, and
  \item \label{item:conditions-involving-cal-C}
    if $\sigma \in \hat{H}_{\temp}$
    satisfies $\mathcal{O}_{\pi,\sigma} = \emptyset$,
    then
    $\h \mathcal{O}_\sigma \cap W = \emptyset$.
  \end{enumerate}
\end{lemma}
Informally,
this says that
\begin{itemize}
\item $V$ is large enough
  to support symbols $a$
  suitable for approximating
  $k$ via
  $k(\mu) \approx \int_{\mathcal{O}(\mu)} a$,
  but not much larger;
\item $W$ is large enough to majorize
  the image of
  the support of any such symbol $a$,
  but small enough to avoid
  any multiorbit
  $\mathcal{O}_{\sigma}$
  for which $\mathcal{O}_{\pi,\sigma} = \emptyset$; and
\item $U$ is large enough to majorize all of the above,
  but not too large.
\end{itemize}
\begin{proof}
  We start by choosing $U$ as indicated with
  $\supp(k) \prec U \prec \mathcal{O}_{\stab}$.
  We then find a compact subset $K$ of $\mathcal{O}_{\stab}$
  so that $\supp(k) \subseteq \image(K) \subset U$.
  Fix $\xi \in K$.
  We may find $\xi_{\h} \in \h \mathcal{O}_\pi$
  tending to $\xi$.
  Since $\xi$ is $\mathbf{H}$-stable,
  its $\mathfrak{h}$-stabilizer is trivial.
  It follows readily that
  the map
  $G \rightarrow G \cdot \xi \rightarrow \mathfrak{h}^\wedge$
  has surjective differential
  $\mathfrak{g} \rightarrow T_{\xi}(G \cdot \xi) \rightarrow \mathfrak{h}^\wedge$.
  The same holds true
  for all $\xi'$ in a small neighborhood of $\xi$;
  in particular, for $\h$ small enough,
  it holds for $\xi_{\h}$.
  We may thus find a small precompact open neighborhood
  $\tilde{W}_{\xi} \subseteq \mathfrak{g}^\wedge$ of $\xi$
  whose image $W_{\xi} \subseteq \mathfrak{h}^\wedge$
  satisfies $W_{\xi} \prec \h \mathcal{O}_{\pi}$ for all small
  $\h$.
  In particular,
  $W_\xi \cap \h \mathcal{O}_\sigma = \emptyset$
  whenever $\mathcal{O}_{\pi,\sigma} = \emptyset$.
  We may assume moreover,
  having chosen $\tilde{W}_\xi$ small enough,
  that
  $W_\xi \prec U$.
  Choose a small precompact open neighborhood
  $V_\xi$ of $\xi$,
  with $\overline{V_\xi}$ consisting
  of $\mathbf{H}$-stable
  elements,
  so that $V_\xi \prec W_\xi$.
  Since $K$ is compact,
  we may find $\xi_{1},\dotsc, \xi_n \in K$
  so that $K \subseteq V := V_{\xi_1} \cup \dotsb \cup V_{\xi_n}$.
  We then take $W := W_{\xi_1} \cup \dotsb  \cup W_{\xi_n}$.
\end{proof}
We henceforth fix such $U,W,V$.
We note
that
$U$ and $V$
satisfy the conditions
enunciated in \S\ref{sec:main-result-inv-branch-arch}.

\begin{lemma} \label{limit lemma}
  Let $a_1, \dotsc, a_k \in C_c^\infty(\mathfrak{g}^\wedge)$.
  Let $T' \in \pi' \otimes \overline{\pi'}$
  be positive definite, but not necessarily
  of trace $1$.
  Then
  \begin{equation} \label{Asymp} \h^d \sum_{\Sigma}
    \mathcal{P}_{\Sigma}(\Opp_h(a_1) \dotsb \Opp_h(a_k) \otimes T')
    \simeq
    \frac{\tau(\mathbf{H})}{\tau(\mathbf{G})} \cdot
    (\int_{\mathcal{O}} a_1 \dotsb a_k \, d \omega ) \cdot \trace(T').
  \end{equation}
  If the $a_i$ are supported
  in the preimage of $W$,
  then the same holds
  after restricting to
  $\Sigma$ for which
  $\sigma$
  is good
  in the sense of
  \S\ref{sec:trunc-spectr-decomp-overview}.
\end{lemma}
\proof
We may normalize $T'$ to have trace $1$.
We then construct ``limit states'' $[a_1 \dotsb a_k]$ as in
\S \ref{sec:limit-stat-attach-symb}.
(This involves
passing to subsequences of $\{\h\}$,
which we may 
do after having assumed for the sake of contradiction
that the estimate fails
for some infinite sequence of $\h$
tending to zero.)
The LHS tends to $\int_{[H]} [a_1 \dotsb a_k]$ ,
by \eqref{eq:characterizing-bracket-map-recap}
and \eqref{eqn:interchange-H-period-with-spectral-decomp}.
We then use the equidistribution statement
\eqref{eqn:Phi-same-integral}, together with
\eqref{eqn:limit-state-construction-preserves-volumes},
to get to the
right-hand side.

For the last statement we use theorem \ref{thm:trunc-H-expn}.
\qed

\begin{lemma}\label{lem:final-approx}
  For each $\eps > 0$ and $N \in \mathbb{Z}_{\geq 0}$
  there exist nonnegative $a, a_1, a_2, a_{\nt}  \in
  C_c^\infty(V)$ and (smooth, finite-rank) positive-definite tensors $T',
  T_1', T_2' \in \pi ' \otimes \overline{\pi '}$
  so that
  \begin{equation}\label{eqn:main-term-constantcompare}
    |\int k \int k' -  (\int_{\mathcal{O}} a^2 \, d \omega) \trace(T')| \leq   \eps
  \end{equation}
  and
  \begin{equation}\label{eq:B-small}
    \sum_{j=1,2} (\int_{\mathcal{O}} a_j^2 \, d \omega) \trace(T_j') \leq \eps
  \end{equation}
  and
  \begin{equation}\label{eq:a-nt-small-final}
    (\int_{\mathcal{O}} a_{\nt}^2 \, d \omega) \trace(T') \leq \eps
  \end{equation}
  and, for $\sigma$ tempered
  with $\h \lambda_\sigma \in U$,
   \begin{equation}\label{eqn:approx-arg-global-tempered}
    \left\lvert
      k_{\h}(\sigma) k'(\sigma ')
      -
      \mathcal{H}_{\sigma}(\Opp_{\h}(a)^2)
      \mathcal{H}_{\sigma '}( T')
    \right\rvert
    \leq
    \sum_{j=1,2}
    |\mathcal{H}_{\sigma}(\Opp_{\h}(a_j)^2)|
    \mathcal{H}_{\sigma '}(T_j')
  \end{equation}
  and, for $\sigma$ non-tempered
  with $\h \lambda_\sigma \in U$,
  \begin{equation}\label{eqn:estimate-for-discarding-NT-stuff-in-arch-case-copy}
    \mathcal{P}_{\Sigma}(\Opp_{\h}(a)^2 \otimes T')
    =
    \mathcal{P}_{\Sigma}(\Opp_{\h}(a_{\nt})^2 \otimes T')
    + \O(\h^N).
  \end{equation}
\end{lemma}
\begin{proof}
  We combine together the analogous approximation results
  obtained in \S\ref{sec:compact-groups},
  \S\ref{sec:main-result-inv-branch-arch} and
  \S\ref{sec:padic-inverse-branching-main-results}.
  We record
  details below.

  In what follows, we write (e.g.) $C(x,y,z)$ for some
  constant $\geq 1$ depending only upon the quantities $x,y,z$.
  We allow the precise definitions
  of such constants to vary from one invocation to the next.
  For instance, we have $\int k \leq C(k)$
  and $\int k' \leq C(k')$
  and $\int k_\mathfrak{p} \leq C(k_\mathfrak{p})$.
  We
  fix $\eta \in (0,1)$; at the end of the proof, $\eta$ will be chosen
  small enough in terms of $\eps$ and the intervening constants.
 
  We may assume that $k'$ is a pure tensor
  $\prod_{\mathfrak{p} \neq \mathfrak{q} } k_\mathfrak{p}$
  with each factor $k_\mathfrak{p}$ nonnegative;
  here and henceforth
  $\mathfrak{p}$
  is restricted to the set $R$ of relevant places.
  By \S\ref{sec:compact-groups}
  and \S\ref{sec:padic-inverse-branching-main-results} (see
  \eqref{eq:allowable-main-defn-new} and \eqref{item:sauv-rel-1}),
  we may find for each auxiliary place
  $\mathfrak{p} \in R - \{\mathfrak{q} \}$
  some positive-definite tensors
  $T_\mathfrak{p} ', T_{\mathfrak{p}}^{(1)},
  T_{\mathfrak{p}}^{(2)}$,
  with
  \[\trace(T_{\mathfrak{p}}^{(1)}) \leq C (k_\mathfrak{p}), \quad \left\lvert k_\mathfrak{p}(\Sigma_{\mathfrak{p}})
    \right\rvert \leq
    \mathcal{H}_{\Sigma_\mathfrak{p}}(T_\mathfrak{p}^{(1)})
  \]
  and
  \[
    \trace(T_{\mathfrak{p}}^{(2)}) \leq \eta,
    \quad
    \left\lvert k_\mathfrak{p}(\Sigma_{\mathfrak{p}})
      - \mathcal{H}_{\Sigma_\mathfrak{p}}(T_\mathfrak{p}')
    \right\rvert \leq \mathcal{H}_{\Sigma_\mathfrak{p}}(T_\mathfrak{p}^{(2)})
  \]
  By \eqref{eq:remaerk-after-aux-padic-thm},
  we may assume
  also that
  $|\int k_\mathfrak{p} - \trace(T_\mathfrak{p}')| \leq \eta$,
  hence in particular that
  $\trace(T_\mathfrak{p}') \leq C(k_\mathfrak{p})$.
  We set
  \[
    T' := \otimes_{\mathfrak{p} \neq \mathfrak{q}}
    T_\mathfrak{p} '.
  \]
  
  By Theorem \ref{thm:arch-inv-branch},
  we may find $a, a_1, a_2, a_{\nt}$
  of the required form
  for which
  \begin{itemize}
  \item  $\int_{\mathcal{O}} a_1^2 \leq C(k,V)$,
  \item $\int_{\mathcal{O}} a_2^2$
    and
    $\int_{\mathcal{O}} a_{\nt}^2$
    and
    $\left\lvert \int k -
      \int_{\mathcal{O}} a^2 \right\rvert$
    are bounded by $\eta$
    (thus $\int_{\mathcal{O}} a^2 \leq C(k)$),
  \item  
    if $\sigma$ is tempered and  $\mathcal{O}^{\pi,
      \sigma}$ is nonempty,
    then
     \begin{equation*}
      \left\lvert k_{\h}(\sigma) 
      \right\rvert
      \leq
      |\mathcal{H}_{\sigma}(\Opp_{\h}(a_1)^2)|
    \end{equation*}
    and
    \begin{equation*}
      \left\lvert k_{\h}(\sigma) -
        \mathcal{H}_{\sigma}(\Opp_{\h}(a)^2)
      \right\rvert
      \leq
      |\mathcal{H}_{\sigma}(\Opp_{\h}(a_2)^2)|,
    \end{equation*}
    and
  \item if $\sigma$ is non-tempered,
    then
    \eqref{eqn:estimate-for-discarding-NT-stuff-in-arch-case-copy}
    holds.
    (Note that
    the hermitian form $\mathcal{P}_{\Sigma}(- \otimes T')$
    satisfies the hypotheses
    indicated around \eqref{eqn:NT-unif-cont}.)
  \end{itemize}

  For tempered $\sigma$ with $\h \lambda_\sigma \in U$,
  the LHS of \eqref{eqn:approx-arg-global-tempered}
  is bounded by
  \begin{equation}\label{eq:sum-over-j-sharp}
    \sum_{j = (j_\mathfrak{p})_{\mathfrak{p} \in R}}^\sharp
    \mathcal{H}_{\sigma}(\Opp_{\h}(    a_{j_\mathfrak{q}})^2      )
    \prod_{\mathfrak{p} \neq \mathfrak{q} }
    \mathcal{H}_{\Sigma_\mathfrak{p}}(T_\mathfrak{p}^{(j_\mathfrak{p})}),
  \end{equation}
  where the sum is taken over all tuples
  $j$ as indicated
  for which
  \begin{itemize}
  \item $j_\mathfrak{p} \in \{1,2\}$ for all $\mathfrak{p} \in
    R$, and
  \item $j_\mathfrak{p} = 2$ for at least one $\mathfrak{p} \in R$.
  \end{itemize}

  For $r=1,2$,
  define
  \[
    T_r' := 
    \sum_{j = (j_\mathfrak{p})_{\mathfrak{p} \in R} :
      j_\mathfrak{q} = r}^\sharp
    \otimes_{\mathfrak{p} \neq \mathfrak{q} }
    T_\mathfrak{p}^{(j_\mathfrak{p})}.
  \]
  Then \eqref{eq:sum-over-j-sharp} equals the RHS of
  \eqref{eqn:approx-arg-global-tempered}, so that
  \eqref{eqn:approx-arg-global-tempered} holds.

  It remains to verify the estimates \eqref{eqn:main-term-constantcompare} and
  \eqref{eq:B-small} and \eqref{eq:a-nt-small-final}.
  Combining together the estimates noted above,
  we obtain the inequalities
  \begin{equation*}
    |\int k \int k' -  (\int_{\mathcal{O}} a^2 \, d \omega)
    \trace(T')| \leq
    \eta 
    2^{\# R}
    C(k)
    \prod_{\mathfrak{p} \neq \mathfrak{q}}
    C(k_\mathfrak{p}),
  \end{equation*}
  \begin{equation}\label{eqn:}
    (\int_{\mathcal{O}} a_1^2 \, d \omega)
    \trace(T_1') 
    \leq
    \eta 
    2^{\# R-1}
    C(k,V)
    \prod_{\mathfrak{p} \neq \mathfrak{q}} C(k_\mathfrak{p}),
  \end{equation}
  \begin{equation*}
    (\int_{\mathcal{O}} a_2^2 \, d \omega)
    \trace(T_2') 
    \leq
    \eta 
    2^{\# R-1} \prod_{\mathfrak{p} \neq \mathfrak{q}} C(k_\mathfrak{p})
    % \sum_{r=1,2}
    % \sum_{j = (j_\mathfrak{p})_{\mathfrak{p} \in R} :
    %   j_\mathfrak{q} = r}^\sharp
    % \left(
    %     \begin{cases}
    %       C(k,V) & \text{ if } r = 1, \\
    %       \eta & \text{ if } r = 2,
    %     \end{cases}
    %   \right)
    %   \prod_{\mathfrak{p} \neq \mathfrak{q} }
    %   \left( 
    % \begin{cases}
    %   C(k_\mathfrak{p}) & \text{ if } j_\mathfrak{p} = 1, \\
    %   \eta  & \text{ if } j_\mathfrak{p} = 2,
    % \end{cases}
    %  \right)
  \end{equation*}
  and
  \begin{equation*}
    (\int_{\mathcal{O}} a_{\nt}^2 \, d \omega) \trace(T')
    \leq
    \eta \prod_{\mathfrak{p} \neq \mathfrak{q}} C(k_\mathfrak{p}).
  \end{equation*}
  We conclude by choosing $\eta$ small enough that the RHS of
  each of the above inequalities is bounded by $\eps$.
\end{proof} 

We turn now to the proof of the theorem.
We
retain the
definition of ``good''
(relative to $\h,U,W$)
from \S\ref{sec:trunc-spectr-decomp-overview};
recall that this excises all $\sigma$ for which
$\h \lambda_{\sigma} \notin U$, as well as those which are not
orbit-distinguished.
Note also that every $\sigma$ in the support
of $k_{\h}$ is good.
Let $\eps > 0$ be small;
we eventually let it tend to zero
sufficiently slowly with respect to $\h$.

\begin{lemma}
  We have
\begin{equation}\label{eqn:ell-M-E}
  |\ell  - (\mathcal{M} - \mathcal{M}_{\nt})| \leq \mathcal{E},
\end{equation}
where $\ell$ is as defined in
\S\ref{sec:fundamental-main-result}
and
\[
  \mathcal{M} :=\h^d
  \sum_{\Sigma: \sigma  \text{ good}}
  \mathcal{P}_{\Sigma }(\Opp_{\h}(a)^2 \otimes T').
\]
\[
  \mathcal{M}_{\nt} :=\h^d
  \sum_{\Sigma: \sigma  \text{ non-tempered, good}}
  \mathcal{P}_{\Sigma}(\Opp_{\h}(a)^2 \otimes T').
\]
\[
  \mathcal{\mathcal{E}} :=\h^d
  \sum_{\Sigma: \sigma  \text{ good}}
  \sum_{j=1,2}
  \mathcal{P}_{\Sigma}(\Opp_{\h}(a_j)^2 \otimes T_j').
\]
\end{lemma}
\begin{proof}
  By the definitions and
  the
  period formula \eqref{eqn:I-I-explicated-for-P-Sigma},
  we  have
  \[
    \ell  - (\mathcal{M} - \mathcal{M}_{\nt})
    =
    \h^d
  \sum_{\Sigma: \sigma \text{ good, tempered}}
  \mathcal{L}(\Pi,\Sigma)
  (
  k_{\h}(\sigma)
  k'(\sigma')
  -
  \mathcal{H}_{\sigma}(\Opp_{\h}(a)^2)
  \mathcal{H}_{\sigma '}(T')
  ).
\]
By \eqref{eqn:approx-arg-global-tempered},
the above is bounded in magnitude by
\[
  \sum_{j=1,2}
  \h^d
  \sum_{\Sigma: \sigma \text{ good, tempered}}
  \underbrace{
    \mathcal{L}(\Pi,\Sigma)
    |\mathcal{H}_{\sigma}(\Opp_{\h}(a_j)^2)|
    \mathcal{H}_{\sigma '}(T_j')
  }_{\displaystyle
    =  \mathcal{P}_{\Sigma}(\Opp_{\h}(a_j)^2 \otimes T_j') } \leq \mathcal{E}.
\]
  Note that the
  absolute values
  surrounding $\mathcal{H}_{\sigma}$
  on the RHS of \eqref{eqn:approx-arg-global-tempered},
  which are expected to be extraneous
  (see the remark following Theorem \ref{thm:arch-inv-branch}),
  have disappeared thanks to the manifest
  positivity of $\mathcal{P}_{\Sigma}$.
\end{proof}

By lemma \ref{limit lemma} and
\eqref{eqn:main-term-constantcompare}, we have
(with $\simeq$ as defined in
\S\ref{sec:fundamental-main-result})
\begin{equation}\label{eqn:cal-M-bounds-nearly-done}
    \mathcal{M}
  \simeq \frac{\tau(\mathbf{H})}{\tau(\mathbf{G})}
  (\int_{\mathcal{O}} a^2 \, d \omega ) \trace(T')
  = \frac{\tau(\mathbf{H})}{\tau(\mathbf{G})}
  \int k \int k'
  + o_{\eps \rightarrow 0}(1).
\end{equation}

The estimate
\begin{equation}\label{eqn:M-nt-o-eps-1}
  \mathcal{M}_{\nt}
  =
  o_{\eps \rightarrow 0}(1).
\end{equation}
follows from lemma \ref{limit lemma},
% \eqref{eq:B-small} and
%\eqref{eqn:estimate-for-discarding-NT-stuff-in-arch-case-copy},
%taking $N$ large enough and 
using the weak Weyl law \eqref{eq:weak-weyl-law} to discard
error terms:
%\eqref{eqn:estimate-for-discarding-NT-stuff-in-arch-case-copy}:
\begin{align*}
  \mathcal{M}_{\nt }
  &\stackrel{\eqref{eqn:estimate-for-discarding-NT-stuff-in-arch-case-copy}}{=}
  \h^d
  \sum_{\sigma  \text{ non-tempered, good}}
  (
  \mathcal{P}_{\Sigma}(\Opp_{\h}(a_{\nt})^2 \otimes T')
  + \O(\h^N)
    )
  \\
  &\stackrel{\eqref{eq:weak-weyl-law}}{\leq}
    \h^d
    \sum_{\Sigma}
    \mathcal{P}_{\Sigma}(\Opp_{\h}(a_{\nt})^2 \otimes T')
    +  \O(\h^{N-\O(1)})
  \\
  &\stackrel{\textrm{Lem. 2}}{=}
    \frac{\tau(\mathbf{H})}{\tau(\mathbf{G})}
    \underbrace{(\int_{\mathcal{O}} a_{\nt}^2 \, d \omega)    \tr(T') }_{\leq \eps, \textrm{ by \eqref{eq:a-nt-small-final}}}
+ o_{\h \rightarrow 0}(1) + \O(\h^{N-\O(1)})
  \\
  &= o_{\eps \rightarrow 0}(1).
\end{align*}
We note that in the second step,
we used the positivity of $\mathcal{P}_{\Sigma}$
to drop the condition ``$\sigma$ non-tempered,''
and that in the final step,
the implied constant $\O(1)$ in the exponent is independent of
$N$, and so by taking $N$ large enough
we may arrange that the  difference $N - \O(1)$ is positive.

The estimate
\begin{equation}\label{eqn:E-o-eps-1}
  \mathcal{E} = o_{\eps \rightarrow 0}(1)
\end{equation}
follows similarly
from lemma \ref{limit lemma} and \eqref{eq:B-small}:
\[
  \mathcal{E}
  =
    \frac{\tau(\mathbf{H})}{\tau(\mathbf{G})}
    \underbrace{
    \sum_{j=1,2} (\int_{\mathcal{O}} a_j^2 \, d \omega)
    \trace(T_j')
    }_{\leq \eps}
    + o_{\h \rightarrow 0}(1)
    =
    o_{\eps \rightarrow 0}(1).
\]

By combining these last estimates
\eqref{eqn:cal-M-bounds-nearly-done},
\eqref{eqn:M-nt-o-eps-1}
and
\eqref{eqn:E-o-eps-1}
with \eqref{eqn:ell-M-E}, we conclude the proof of Theorem \ref{thm:fundamental}.

\section{The normalized asymptotic formula}\label{sec:norm-asympt-form}
The hard work having been completed, we explain here how theorem
\ref{thm:fundamental} may be applied to sharply-truncated
sums over families. By dividing out the cardinalities of
those families, we then obtain the normalized asymptotic
formulas promised in \S\ref{sec:introduction}.

\subsection{Approximating nice sets by continuous functions\label{sec:nice-and-approximable}}
\label{sec-991-1}
Let $(X,\mu)$ be a normal topological space
equipped with a Borel measure.
\begin{definition}
  We say that a subset $U \subseteq X$
  is \emph{nice}
  if it is open, precompact,
  and has measure zero boundary.
\end{definition}
\begin{definition}
  Given a  class $\mathcal{C}$  of integrable functions on $X$,
  we say that an integrable function
  $w$ on $X$ is \emph{approximable by
    $\mathcal{C}$}
  if
  \begin{enumerate}[(i)]
  \item  there exists $k_0 \in \mathcal{C}$
    so that $|w| \leq k_0$, and
  \item for each $\eps > 0$
    there exist $k, k_+ \in \mathcal{C}$
    so that
    \[
      |w - k| \leq k_+,
      \quad
      \int k_+ \leq \eps.
    \]
  \end{enumerate}
\end{definition}
By an exercise in applying Urysohn's lemma,
we have:
\begin{lemma*}
  The characteristic function of any nice subset $U \subseteq X$
  is approximable by the class $C_c(X)$
  of continuous compactly-supported functions.
\end{lemma*}
The notion of approximability is
well-behaved with respect to products:
Let
$X_1,\dotsc,X_n$ be spaces as above, each equipped with a Borel
measure.
For $j = 1..n$,
let $w_j$ be an integrable function on $X_j$
that is
approximable by some class $C_j$
of integrable functions.
Assume also that $k_1 + k_2 \in \mathcal{C}_j$
whenever $k_1,k_2 \in \mathcal{C}_j$.
Then the function
$X_1 \times \dotsb \times X_n \ni x \mapsto
w_1(x_1) \dotsb w_n(x_n)$
is approximable by
the class $\mathcal{C}$ consisting
of all sums of functions
of the form $x \mapsto k_1(x_1)\dotsb k_n(x_n)$
with $k_j \in \mathcal{C}_j$.

\subsection{Summing against approximable weights}
\label{sec-991-2}
The spaces
$[\mathfrak{h}^\wedge] \cap \image(\mathcal{O}_{\stab})$ and
$(\hat{H'})_{\temp}^{\pi'}$ come with natural topologies, given in the former case by identifying
$[\mathfrak{h}^\wedge]$ with an affine space, in the latter
by the discussion of \S\ref{sec:compact-groups} and
\S\ref{Omegatorus}.
They also come with natural measures: normalized affine measure and the
restriction of Plancherel measure, respectively.
The terminology of \S\ref{sec:nice-and-approximable} thus
applies,
and we readily
derive from theorem
\ref{thm:fundamental}
the following:
\begin{corollary*}
  Let
  $k : [\mathfrak{h}^\wedge] \rightarrow \mathbb{C}$
  and
  $k' : (\hat{H'})_{\temp}^{\pi '} \rightarrow \mathbb{C}$
  be any functions approximable by
  the function spaces $\mathcal{K}$ and
  $\mathcal{K}'$
  defined in \S\ref{sec:function-spaces}.
  Then with notation as in \S\ref{sec:fundamental-main-result},
  \[
  \h^d
  \sum_{
    \substack{
      \sigma, \sigma ' \text{ tempered} \\
      \mathcal{O}_{\pi,\sigma} \neq \emptyset 
    }
  }
  \mathcal{L}(\Pi,\Sigma)
  k_{\h}(\sigma)
  k'(\sigma')
  \simeq
   \frac{\tau(\mathbf{H})}{\tau(\mathbf{G})}
  \int k \int k'.
  \]
\end{corollary*}

\subsection{Summing over unweighted families}
\label{sec-991-4}
\begin{definition*}
We say that a subset
$U \subset [\mathfrak{h}^\wedge] \cap
\image(\mathcal{O}_{\stab})$
is \emph{admissible} if it is nice
(cf. \S\ref{sec:nice-and-approximable}),
and that a subset
$U' \subset (\hat{H'})_{\temp}^{\pi'}$ is \emph{admissible} if is
nice and if its projection
onto $(\hat{H}_\mathfrak{p})_{\temp}^{\pi_\mathfrak{p}}$,
for a finite place $\mathfrak{p} \in R$,
is contained in the union of
the allowable components (see \S\ref{localBernstein}).
\end{definition*}
The   (smooth version of the) lemma of \S\ref{sec:nice-and-approximable} implies:
\begin{lemma*}
  If $U$ (resp. $U'$) as above is admissible, then its
  characteristic function is approximable
  by $\mathcal{K}$ (resp. $\mathcal{K} '$).
\end{lemma*}
We henceforth fix some nonempty admissible sets $U, U'$ above, and set  
\begin{equation} \label{Fhdef}
  \mathcal{F}_{\h}
  :=
  \left\{ \Sigma :
    \begin{gathered}
      \sigma, \sigma ' \text{ are tempered}, \\
    \sigma \mbox{ is orbit-distinguished, i.e. $\mathcal{O}^{\pi, \sigma} \neq \emptyset$}, \\
    \h \lambda_{\sigma} \in U, 
    \sigma ' \in U ' \end{gathered}
  \right\}.
\end{equation}
The previous corollary specializes as follows:
\begin{corollary*}
  \begin{equation}\label{eqn:moment-for-family-unnormalized}
      \h^d
  \sum_{
    \Sigma \in \mathcal{F}_{\h}
  }
  \mathcal{L}(\Pi,\Sigma)
  \simeq
  \frac{\tau(\mathbf{H})}{\tau(\mathbf{G})}
  \vol(U) \vol(U').
  \end{equation}
\end{corollary*}

\subsection{Family size}
\label{analytic-conductors}
 
To interpret the LHS of \eqref{eqn:moment-for-family-unnormalized}
as a normalized average of $\mathcal{L}(\Pi,\Sigma)$,
we need to know the approximate cardinality of the family
$\mathcal{F}_{\h}$.
\begin{lemma*}
  One has
  \begin{equation}\label{eqn:weyl-law-for-F-h-in-stmt-main}
    \h^d
    |\mathcal{F}_{\h}|
    =
    \tau(\mathbf{H})
    \vol(U )
    \vol(U' ) + o_{h \rightarrow 0}(1),
  \end{equation}
\end{lemma*} 
\begin{proof}

  This can be deduced from the methods of Duistermaat, Kolk and
  Varadarajan \cite{MR532745} and F. Sauvageot
  \cite{MR1468833}, but it is simpler for us to give a more
  direct argument.
  We apply
  the trace formula for the compact quotient $[H]$
  in a standard way,
  using the approximation arguments
  given by
   \begin{itemize}
  \item the lemma of \S\ref{sec:arch-weyl-counting}, and
  \item
    theorem \ref{thm:consequences-of-allowability},
    \S\ref{sec:padic-inverse-branching-main-results},
    via the consequences \eqref{eq:allowable-main-defn-new-2},
    \eqref{item:sauv-2}.
  \end{itemize}
  For small $\h$,
  the support condition
  $1 + o_{\h \rightarrow 0}(1)$ on the test function at
  $\mathfrak{q}$ implies that the only nonzero contribution on
  the geometric side comes from the identity element.  We
  readily obtain the smoothly-weighted variant of
  \eqref{eqn:weyl-law-for-F-h-in-stmt-main} from which
  \eqref{eqn:weyl-law-for-F-h-in-stmt-main} itself then follows
  as in the proof of \eqref{eqn:moment-for-family-unnormalized}.
\end{proof}

\subsection{Analytic conductors}
Let $\Sigma \in \mathcal{F}_{\h}$.
We denote by
$C(\Pi,\Sigma)$
the analytic conductor of
$\rsLplain_s$
at $s=1/2$,.

\begin{table}\label{tab:numerology}
  \begin{tabular}{|c|c|c|c|c|c|c|} 
    \hline \hline
    \mbox{label} & $\mathbf{G}$ & $\mathbf{H}$ &      $\dim(B_H)$   & $m_G$ & $m_H$ & $\eps$    \\ 		%
    \hline
    (i) & $\GL_{n+1}$ & $\mathrm{GL}_n$ &   $\frac{n(n+1)}{2}$ & $n+1$ & $n$ & $2$  \\									%
    (ii) & $\mathrm{U}_{n+1}$ & $\mathrm{U}_n$ &  $\frac{n(n+1)}{2}$ & $n+1$    & $n$  & $2$  \\									%
    (iii) & $\mathrm{SO}_{2n+2}$ & $\mathrm{SO}_{2n+1}$ &  $n(n+1)$ & $2n+2$ & $(2n)$ & $1$   \\											%
    (iv) & $\mathrm{SO}_{2n+1}$  & $\mathrm{SO}_{2n}$ &  $n^2$ &  $2n$ &  $2n$ & $1$ \\		%
    \hline \hline
  \end{tabular}
  \caption{Numerology: $m_G$ is the dimension of the standard representation of $\hat{G}$, similarly for $H$}
\end{table}

\begin{lemma*}
  $C(\Pi, \Sigma)
    \asymp |\mathcal{F}_{\h}|^4$.
\end{lemma*}
\begin{proof}
  Our assumptions imply that the contribution to the analytic
  conductor from places other than
  $\mathfrak{q}$ is bounded.
   By the discussion
  of \S\ref{Satake parameters} (especially
  \eqref{eqn:eigenvalue-for-RS-L-factor}),
  we have
  \[ C(\Pi, \Sigma)
  \asymp \h^{-\eps m_H  m_G}\]
  with notation as in
  the table.
  On the other hand,
  we have seen that $|\mathcal{F}_{\h}| \asymp \h^{-\dim
    B_H}$.  By inspection, $4 \dim(B_H) = \eps
  n_G n_H$ in all
  cases.  The
  required estimate
  follows. 
\end{proof}

\subsection{Main result}
\label{sec-18-11}
Dividing
\eqref{eqn:moment-for-family-unnormalized}
by
\eqref{eqn:weyl-law-for-F-h-in-stmt-main},
we conclude:
\begin{theorem}\label{thm:main-subconvex}
  Let notation and assumptions be as above.
  In particular:
  \begin{itemize}
  \item $(\G, \H)$ is a GGP pair over a number field $F$.
    We have fixed
    a large enough finite
    set of places $R$.
  \item We have fixed an
    automorphic representation  $\Pi$ on $\G$, unramified
    outside of $R$,
    and
    satisfying the conditions enunciated in
    \S \ref{sec:assumptions-for-main-result}.
  \item   We have fixed an archimedean place $\mathfrak{q} \in R$,
    and set
    $H = \H(F_{\mathfrak{q}})$ and  $H' = \prod_{\p \in R-\{\mathfrak{q}\}} \H(F_{\p})$.
    For an automorphic representation  $\Sigma$ of $\mathbf{H}$,
    unramified outside of $R$,
    we have denoted by $\sigma$ and $\sigma'$ 
    the associated representations of $H$ and $H'$. 
  \end{itemize} 
  Let
  $U, U'$ be admissible subsets of the tempered
  distinguished spectra of $H, H'$, as above,
  and let  $\mathcal{F}_h$ denote
  the the family of representations on $\H$
  associated to these subsets, as in \eqref{Fhdef}. 

  Let $\mathcal{L}(\Pi,\Sigma)$ denote the branching coefficient
  as defined in \S\ref{sec-18-4}.  
  Then
  \begin{equation}\label{eqn:moment-asymptotic-assuming-Arthur}
    \frac{1}{|\mathcal{F}_{\h}|}
    \sum_{\Sigma \in \mathcal{F}_{\h}} \mathcal{L}(\Pi,\Sigma)
    =
    \frac{1}{\tau(\mathbf{G})}
    + o_{\h \rightarrow 0}(1).
  \end{equation}
  If we assume the conjectures of Ichino--Ikeda/N. Harris
  (\S\ref{sec-18-5}),
  or restrict to cases in which those conjectures
  are known (e.g., \cite{MR3245011}), then
  \begin{equation}\label{eqn:MAGA}
    \frac{1}{|\mathcal{F}_{\h}|}
    \sum_{\Sigma \in \mathcal{F}_{\h}}
    \frac{\tau(\mathbf{G}) \tau(\mathbf{H})}{2^{\beta}}
    \Delta_G^{(R)}
    \rsL
    =
    \tau(\mathbf{H})
    + o_{\h \rightarrow 0}(1).
  \end{equation}
\end{theorem}
We may ``simplify''
the final formulas
by recalling
(\S\ref{sec:assumptions-for-main-result})
that
$\tau(\mathbf{G}) = \tau(\mathbf{H})=2$.

\printindex

% \bibliography{refs}{}

\begin{thebibliography}{100}

\bibitem[ALTV]{2012arXiv1212.2192A}
J.~{Adams}, M.~{van Leeuwen}, P.~{Trapa}, and D.~A. {Vogan}, Jr.
\newblock {Unitary representations of real reductive groups}.
\newblock {\em ArXiv e-prints}, December 2012.

\bibitem[AV]{AdamsVogan}
Jeffrey Adams and David~A. Vogan, Jr.
\newblock Contragredient representations and characterizing the local
  {L}anglands correspondence.
\newblock {\em Amer. J. Math.}, 138(3):657--682, 2016.

\bibitem[BV]{MR576644}
Dan Barbasch and David~A. Vogan, Jr.
\newblock The local structure of characters.
\newblock {\em J. Funct. Anal.}, 37(1):27--55, 1980.

\bibitem[Bar]{MR1818248}
Yu. Baryshnikov.
\newblock G{UE}s and queues.
\newblock {\em Probab. Theory Related Fields}, 119(2):256--274, 2001.

\bibitem[Be1]{MR0367730}
Richard Beals.
\newblock A general calculus of pseudodifferential operators.
\newblock {\em Duke Math. J.}, 42:1--42, 1975.

\bibitem[Be2]{MR0435933}
Richard Beals.
\newblock Characterization of pseudodifferential operators and applications.
\newblock {\em Duke Math. J.}, 44(1):45--57, 1977.

\bibitem[BZ]{BZ}
I.~N. Bernstein and A.~V. Zelevinsky.
\newblock Induced representations of reductive {${\mathfrak{p}}$}-adic groups.
  {I}.
\newblock {\em Ann. Sci. \'Ecole Norm. Sup. (4)}, 10(4):441--472, 1977.

\bibitem[BDK]{MR874050}
J.~Bernstein, P.~Deligne, and D.~Kazhdan.
\newblock Trace {P}aley-{W}iener theorem for reductive {$p$}-adic groups.
\newblock {\em J. Analyse Math.}, 47:180--192, 1986.

\bibitem[Ber]{MR771671}
J.~N. Bernstein.
\newblock Le ``centre'' de {B}ernstein.
\newblock In {\em Representations of reductive groups over a local field},
  Travaux en Cours, pages 1--32. Hermann, Paris, 1984.
\newblock Edited by P. Deligne.

\bibitem[BR1]{MR1930758}
Joseph Bernstein and Andre Reznikov.
\newblock Sobolev norms of automorphic functionals.
\newblock {\em Int. Math. Res. Not.}, (40):2155--2174, 2002.

\bibitem[BR2]{MR2726097}
Joseph Bernstein and Andre Reznikov.
\newblock Subconvexity bounds for triple {$L$}-functions and representation
  theory.
\newblock {\em Ann. of Math. (2)}, 172(3):1679--1718, 2010.

\bibitem[Beu1]{2015arXiv150601452B}
R.~{Beuzart-Plessis}.
\newblock {A local trace formula for the Gan-Gross-Prasad conjecture for
  unitary groups: the archimedean case}.
\newblock {\em ArXiv e-prints}, June 2015.

\bibitem[Beu2]{MR3371496}
R.~Beuzart-Plessis.
\newblock Endoscopie et conjecture locale raffin\'ee de {G}an-{G}ross-{P}rasad
  pour les groupes unitaires.
\newblock {\em Compos. Math.}, 151(7):1309--1371, 2015.

\bibitem[Beu3]{2016arXiv160206538B}
R.~{Beuzart-Plessis}.
\newblock {Comparison of local spherical characters and the Ichino-Ikeda
  conjecture for unitary groups}.
\newblock {\em ArXiv e-prints}, February 2016.

\bibitem[Beu4]{BPU}
Rapha\"el Beuzart-Plessis.
\newblock La conjecture locale de {G}ross-{P}rasad pour les repr\'esentations
  temp\'er\'ees des groupes unitaires.
\newblock {\em M\'em. Soc. Math. Fr. (N.S.)}, (149):vii+191, 2016.

\bibitem[BGW]{MR3495795}
Manjul Bhargava, Benedict~H. Gross, and Xiaoheng Wang.
\newblock Arithmetic invariant theory {II}: {P}ure inner forms and obstructions
  to the existence of orbits.
\newblock In {\em Representations of reductive groups}, volume 312 of {\em
  Progr. Math.}, pages 139--171. Birkh\"auser/Springer, Cham, 2015.

\bibitem[BFKMMS]{2018arXiv180401450B}
V.~{Blomer}, {\'E}.~{Fouvry}, E.~{Kowalski}, P.~{Michel},
  D.~{Mili{\'c}evi{\'c}}, and W.~{Sawin}.
\newblock {The second moment theory of families of L-functions}.
\newblock {\em ArXiv e-prints}, April 2018.

\bibitem[Blo]{MR2972602}
Valentin Blomer.
\newblock Period integrals and {R}ankin-{S}elberg {$L$}-functions on {$GL(n)$}.
\newblock {\em Geom. Funct. Anal.}, 22(3):608--620, 2012.

\bibitem[BFKMM]{MR3650231}
Valentin Blomer, \'Etienne Fouvry, Emmanuel Kowalski, Philippe Michel, and
  Djordje Mili\'cevi\'c.
\newblock On moments of twisted {$L$}-functions.
\newblock {\em Amer. J. Math.}, 139(3):707--768, 2017.

\bibitem[BM]{MR3334233}
Valentin Blomer and Djordje Mili{\'c}evi{\'c}.
\newblock The second moment of twisted modular {$L$}-functions.
\newblock {\em Geom. Funct. Anal.}, 25(2):453--516, 2015.

\bibitem[Bor1]{MR546608}
A.~Borel.
\newblock Automorphic {$L$}-functions.
\newblock In {\em Automorphic forms, representations and {$L$}-functions
  ({P}roc. {S}ympos. {P}ure {M}ath., {O}regon {S}tate {U}niv., {C}orvallis,
  {O}re., 1977), {P}art 2}, Proc. Sympos. Pure Math., XXXIII, pages 27--61.
  Amer. Math. Soc., Providence, R.I., 1979.

\bibitem[Bor2]{Borel}
A.~Borel.
\newblock Automorphic {$L$}-functions.
\newblock In {\em Automorphic forms, representations and {$L$}-functions
  ({P}roc. {S}ympos. {P}ure {M}ath., {O}regon {S}tate {U}niv., {C}orvallis,
  {O}re., 1977), {P}art 2}, Proc. Sympos. Pure Math., XXXIII, pages 27--61.
  Amer. Math. Soc., Providence, R.I., 1979.

\bibitem[Bor3]{MR1102012}
Armand Borel.
\newblock {\em Linear algebraic groups}, volume 126 of {\em Graduate Texts in
  Mathematics}.
\newblock Springer-Verlag, New York, second edition, 1991.

\bibitem[BK]{MR2396123}
H.~M. Bui and J.~P. Keating.
\newblock On the mean values of {$L$}-functions in orthogonal and symplectic
  families.
\newblock {\em Proc. Lond. Math. Soc. (3)}, 96(2):335--366, 2008.

\bibitem[Cah]{MR3994322}
Benjamin Cahen.
\newblock Invariant symbolic calculus for compact {L}ie groups.
\newblock {\em Arch. Math. (Brno)}, 55(3):139--155, 2019.

\bibitem[Cas1]{Casselman}
W.~Casselman.
\newblock Introduction to the theory of admissible representations of reductive
  $p$-adic groups.
\newblock {\em course notes, available from {\tt
  https://www.math.ubc.ca/~cass/research/pdf/p-adic-book.pdf}}.

\bibitem[Cas2]{MR1013462}
W.~Casselman.
\newblock Canonical extensions of {H}arish-{C}handra modules to representations
  of {$G$}.
\newblock {\em Canad. J. Math.}, 41(3):385--438, 1989.

\bibitem[Ch]{MR0072877}
Claude Chevalley.
\newblock Invariants of finite groups generated by reflections.
\newblock {\em Amer. J. Math.}, 77:778--782, 1955.

\bibitem[CF]{MR1784410}
J.~B. Conrey and D.~W. Farmer.
\newblock Mean values of {$L$}-functions and symmetry.
\newblock {\em Internat. Math. Res. Notices}, (17):883--908, 2000.

\bibitem[CFKRS]{MR2149530}
J.~B. Conrey, D.~W. Farmer, J.~P. Keating, M.~O. Rubinstein, and N.~C. Snaith.
\newblock Integral moments of {$L$}-functions.
\newblock {\em Proc. London Math. Soc. (3)}, 91(1):33--104, 2005.

\bibitem[CHH]{MR946351}
M.~Cowling, U.~Haagerup, and R.~Howe.
\newblock Almost {$L^2$} matrix coefficients.
\newblock {\em J. Reine Angew. Math.}, 387:97--110, 1988.

\bibitem[DN]{MR2729264}
Mladen Dimitrov and Louise Nyssen.
\newblock Test vectors for trilinear forms when at least one representation is
  not supercuspidal.
\newblock {\em Manuscripta Math.}, 133(3-4):479--504, 2010.

\bibitem[DKV]{MR532745}
J.~J. Duistermaat, J.~A.~C. Kolk, and V.~S. Varadarajan.
\newblock Spectra of compact locally symmetric manifolds of negative curvature.
\newblock {\em Invent. Math.}, 52(1):27--93, 1979.

\bibitem[EW]{Einsiedler-Ward}
Manfred Einsiedler and Thomas Ward.
\newblock {\em Ergodic theory with a view towards number theory}, volume 259 of
  {\em Graduate Texts in Mathematics}.
\newblock Springer-Verlag London, Ltd., London, 2011.

\bibitem[FLO]{MR2930996}
Brooke Feigon, Erez Lapid, and Omer Offen.
\newblock On representations distinguished by unitary groups.
\newblock {\em Publ. Math. Inst. Hautes \'Etudes Sci.}, 115:185--323, 2012.

\bibitem[FS]{2017arXiv170900935M}
Jan Frahm and Feng Su.
\newblock Upper bounds for geodesic periods over rank one locally symmetric
  spaces.
\newblock {\em Forum Math.}, 30(5):1065--1077, 2018.

\bibitem[GGP]{GGP}
Wee~Teck Gan, Benedict~H. Gross, and Dipendra Prasad.
\newblock Symplectic local root numbers, central critical {$L$} values, and
  restriction problems in the representation theory of classical groups.
\newblock {\em Ast\'erisque}, (346):1--109, 2012.
\newblock Sur les conjectures de Gross et Prasad. I.

\bibitem[Gr]{MR1474159}
Benedict~H. Gross.
\newblock On the motive of a reductive group.
\newblock {\em Invent. Math.}, 130(2):287--313, 1997.

\bibitem[GP]{GP}
Benedict~H. Gross and Dipendra Prasad.
\newblock On the decomposition of a representation of {${\rm SO}_n$} when
  restricted to {${\rm SO}_{n-1}$}.
\newblock {\em Canad. J. Math.}, 44(5):974--1002, 1992.

\bibitem[Gu]{gutt1983explicit}
Simone Gutt.
\newblock An explicit*-product on the cotangent bundle of a lie group.
\newblock {\em Letters in Mathematical Physics}, 7(3):249--258, 1983.

\bibitem[Hai]{Haines}
Thomas~J. Haines.
\newblock The stable {B}ernstein center and test functions for {S}himura
  varieties.
\newblock In {\em Automorphic forms and {G}alois representations. {V}ol. 2},
  volume 415 of {\em London Math. Soc. Lecture Note Ser.}, pages 118--186.
  Cambridge Univ. Press, Cambridge, 2014.

\bibitem[HR]{MR2602034}
Thomas~J. Haines and Sean Rostami.
\newblock The {S}atake isomorphism for special maximal parahoric {H}ecke
  algebras.
\newblock {\em Represent. Theory}, 14:264--284, 2010.

\bibitem[Har1]{MR3812114}
Benjamin Harris.
\newblock Wave front sets of reductive {L}ie group representations {II}.
\newblock {\em Trans. Amer. Math. Soc.}, 370(8):5931--5962, 2018.

\bibitem[Har2]{MR3159075}
R.~Neal Harris.
\newblock The refined {G}ross-{P}rasad conjecture for unitary groups.
\newblock {\em Int. Math. Res. Not. IMRN}, (2):303--389, 2014.

\bibitem[HS]{MR0396855}
Henryk Hecht and Wilfried Schmid.
\newblock A proof of {B}lattner's conjecture.
\newblock {\em Invent. Math.}, 31(2):129--154, 1975.

\bibitem[Hoh]{geogebra}
M.~Hohenwarter~et al.
\newblock {G}eo{G}ebra, April 2018.
\newblock \url{http://www.geogebra.org}.

\bibitem[H\"or]{MR2304165}
Lars H\"ormander.
\newblock {\em The analysis of linear partial differential operators. {III}}.
\newblock Classics in Mathematics. Springer, Berlin, 2007.
\newblock Pseudo-differential operators, Reprint of the 1994 edition.

\bibitem[HNS]{HNSminimal}
Yueke Hu, Paul Nelson, and Abhishek Saha.
\newblock Some analytic aspects of automorphic forms on gl(2) of minimal type.
\newblock {\em Commentarii Mathematici Helvetici}, 94:767--801, 12 2019.

\bibitem[II]{MR2585578}
Atsushi Ichino and Tamutsu Ikeda.
\newblock On the periods of automorphic forms on special orthogonal groups and
  the {G}ross-{P}rasad conjecture.
\newblock {\em Geom. Funct. Anal.}, 19(5):1378--1425, 2010.

\bibitem[KaSa]{KS99}
Nicholas~M. Katz and Peter Sarnak.
\newblock Zeroes of zeta functions and symmetry.
\newblock {\em Bull. Amer. Math. Soc. (N.S.)}, 36(1):1--26, 1999.

\bibitem[KeSn1]{KeSn00a}
J.~P. Keating and N.~C. Snaith.
\newblock Random matrix theory and {$\zeta(1/2+it)$}.
\newblock {\em Comm. Math. Phys.}, 214(1):57--89, 2000.

\bibitem[KeSn2]{KeSn00b}
J.~P. Keating and N.~C. Snaith.
\newblock Random matrix theory and {$L$}-functions at {$s=1/2$}.
\newblock {\em Comm. Math. Phys.}, 214(1):91--110, 2000.

\bibitem[Kir1]{Kirillovnilpotent}
A.~A. Kirillov.
\newblock Unitary representations of nilpotent {L}ie groups.
\newblock {\em Uspehi Mat. Nauk}, 17(4 (106)):57--110, 1962.

\bibitem[Kir2]{MR1701415}
A.~A. Kirillov.
\newblock Merits and demerits of the orbit method.
\newblock {\em Bull. Amer. Math. Soc. (N.S.)}, 36(4):433--488, 1999.

\bibitem[Kn]{MR855239}
Anthony~W. Knapp.
\newblock {\em Representation theory of semisimple groups}, volume~36 of {\em
  Princeton Mathematical Series}.
\newblock Princeton University Press, Princeton, NJ, 1986.
\newblock An overview based on examples.

\bibitem[Kos1]{MR0114875}
Bertram Kostant.
\newblock The principal three-dimensional subgroup and the {B}etti numbers of a
  complex simple {L}ie group.
\newblock {\em Amer. J. Math.}, 81:973--1032, 1959.

\bibitem[Kos2]{MR0158024}
Bertram Kostant.
\newblock Lie group representations on polynomial rings.
\newblock {\em Amer. J. Math.}, 85:327--404, 1963.

\bibitem[Kos3]{MR507800}
Bertram Kostant.
\newblock On {W}hittaker vectors and representation theory.
\newblock {\em Invent. Math.}, 48(2):101--184, 1978.

\bibitem[Kot]{KST}
Robert~E. Kottwitz.
\newblock Stable trace formula: cuspidal tempered terms.
\newblock {\em Duke Math. J.}, 51(3):611--650, 1984.

\bibitem[KMS]{MR3702671}
Emmanuel Kowalski, Philippe Michel, and Will Sawin.
\newblock Bilinear forms with {K}loosterman sums and applications.
\newblock {\em Ann. of Math. (2)}, 186(2):413--500, 2017.

\bibitem[Kur]{Kurtzke}
John~F. Kurtzke, Jr.
\newblock Centralizers of irregular elements in reductive algebraic groups.
\newblock {\em Pacific J. Math.}, 104(1):133--154, 1983.

\bibitem[K\"{u}s]{MR3669792}
Benjamin K\"{u}ster.
\newblock On the semiclassical functional calculus for {$h$}-dependent
  functions.
\newblock {\em Ann. Global Anal. Geom.}, 52(1):57--97, 2017.

\bibitem[Luo]{2020arXiv200913947L}
Zhilin {Luo}.
\newblock {A Local Trace Formula for the Local Gan-Gross-Prasad Conjecture for
  Special Orthogonal Groups}.
\newblock {\em arXiv e-prints}, page arXiv:2009.13947, September 2020.

\bibitem[MV]{michel-2009}
Philippe Michel and Akshay Venkatesh.
\newblock The subconvexity problem for {${\rm GL}_2$}.
\newblock {\em Publ. Math. Inst. Hautes \'Etudes Sci.}, (111):171--271, 2010.

\bibitem[Mil]{milneLEC}
James~S. Milne.
\newblock Lectures on etale cohomology (v2.21), 2013.
\newblock Available at www.jmilne.org/math/.

\bibitem[MW]{MR3058848}
C.~M\oe~glin and J.-L. Waldspurger.
\newblock {\em Sur les conjectures de {G}ross et {P}rasad. {II}}.
\newblock Soci\'et\'e Math\'ematique de France, Paris, 2012.
\newblock Ast\'erisque No. 347 (2012) (2012).

\bibitem[MS1]{2017arXiv170605167M}
J.~{M{\"o}llers} and F.~{Su}.
\newblock {The second moment of period integrals and Rankin-Selberg
  $L$-functions for ${\rm GL}(3)\times{\rm GL}(2)$}.
\newblock {\em ArXiv e-prints}, June 2017.

\bibitem[Mor]{MR2158954}
Dave~Witte Morris.
\newblock {\em Ratner's theorems on unipotent flows}.
\newblock Chicago Lectures in Mathematics. University of Chicago Press,
  Chicago, IL, 2005.

\bibitem[MP]{MR1371680}
Allen Moy and Gopal Prasad.
\newblock Jacquet functors and unrefined minimal {$K$}-types.
\newblock {\em Comment. Math. Helv.}, 71(1):98--121, 1996.

\bibitem[MFK]{Mumford}
D.~Mumford, J.~Fogarty, and F.~Kirwan.
\newblock {\em Geometric invariant theory}, volume~34 of {\em Ergebnisse der
  Mathematik und ihrer Grenzgebiete (2) [Results in Mathematics and Related
  Areas (2)]}.
\newblock Springer-Verlag, Berlin, third edition, 1994.

\bibitem[Nel1]{MR0107176}
Edward Nelson.
\newblock Analytic vectors.
\newblock {\em Ann. of Math. (2)}, 70:572--615, 1959.

\bibitem[NS]{MR0110024}
Edward Nelson and W.~Forrest Stinespring.
\newblock Representation of elliptic operators in an enveloping algebra.
\newblock {\em Amer. J. Math.}, 81:547--560, 1959.

\bibitem[Nel2]{nelson-padic-que}
Paul~D. Nelson.
\newblock Microlocal lifts and quantum unique ergodicity on {${\rm GL}\sb
  2(\Bbb{Q}_p)$}.
\newblock {\em Algebra Number Theory}, 12(9):2033--2064, 2018.

\bibitem[Nel3]{nelson-subconvex-reduction-eisenstein}
Paul~D. Nelson.
\newblock Subconvex equidistribution of cusp forms: {R}eduction to {E}isenstein
  observables.
\newblock {\em Duke Math. J.}, 168(9):1665--1722, 2019.

\bibitem[NPS]{PDN-AP-AS-que}
Paul~D. Nelson, Ameya Pitale, and Abhishek Saha.
\newblock Bounds for {R}ankin--{S}elberg integrals and quantum unique
  ergodicity for powerful levels.
\newblock {\em J. Amer. Math. Soc.}, 27(1):147--191, 2014.

\bibitem[Ols]{MR3185209}
Grigori Olshanski.
\newblock Projections of orbital measures, {G}elfand-{T}setlin polytopes, and
  splines.
\newblock {\em J. Lie Theory}, 23(4):1011--1022, 2013.

\bibitem[Ono]{MR0156851}
Takashi Ono.
\newblock On the {T}amagawa number of algebraic tori.
\newblock {\em Ann. of Math. (2)}, 78:47--73, 1963.

\bibitem[PR]{PR}
Vladimir Platonov and Andrei Rapinchuk.
\newblock {\em Algebraic groups and number theory}, volume 139 of {\em Pure and
  Applied Mathematics}.
\newblock Academic Press, Inc., Boston, MA, 1994.
\newblock Translated from the 1991 Russian original by Rachel Rowen.

\bibitem[RS]{2007arXiv0705.2168R}
S.~{Rallis} and G.~{Schiffmann}.
\newblock {Multiplicity one Conjectures}.
\newblock {\em ArXiv e-prints}, May 2007.

\bibitem[Rao]{MR0320232}
R.~Ranga~Rao.
\newblock Orbital integrals in reductive groups.
\newblock {\em Ann. of Math. (2)}, 96:505--510, 1972.

\bibitem[Rat]{RatnerMeasureClassification}
Marina Ratner.
\newblock On {R}aghunathan's measure conjecture.
\newblock {\em Ann. of Math. (2)}, 134(3):545--607, 1991.

\bibitem[Rez]{MR2373356}
Andre Reznikov.
\newblock Rankin-{S}elberg without unfolding and bounds for spherical {F}ourier
  coefficients of {M}aass forms.
\newblock {\em J. Amer. Math. Soc.}, 21(2):439--477, 2008.

\bibitem[Ri]{MR1064995}
Marc~A. Rieffel.
\newblock Lie group convolution algebras as deformation quantizations of linear
  {P}oisson structures.
\newblock {\em Amer. J. Math.}, 112(4):657--685, 1990.

\bibitem[Rob]{JR}
Justin Roberts.
\newblock Classical {$6j$}-symbols and the tetrahedron.
\newblock {\em Geom. Topol.}, 3:21--66, 1999.

\bibitem[Ros1]{MR587333}
W.~Rossmann.
\newblock Limit characters of reductive {L}ie groups.
\newblock {\em Invent. Math.}, 61(1):53--66, 1980.

\bibitem[Ros2]{MR650378}
W.~Rossmann.
\newblock Limit orbits in reductive {L}ie algebras.
\newblock {\em Duke Math. J.}, 49(1):215--229, 1982.

\bibitem[Ros3]{MR650379}
W.~Rossmann.
\newblock Tempered representations and orbits.
\newblock {\em Duke Math. J.}, 49(1):231--247, 1982.

\bibitem[Ros4]{MR1353309}
W.~Rossmann.
\newblock Picard-{L}efschetz theory and characters of a semisimple {L}ie group.
\newblock {\em Invent. Math.}, 121(3):579--611, 1995.

\bibitem[Ros5]{MR508985}
Wulf Rossmann.
\newblock Kirillov's character formula for reductive {L}ie groups.
\newblock {\em Invent. Math.}, 48(3):207--220, 1978.

\bibitem[SVe]{SV}
Yiannis Sakellaridis and Akshay Venkatesh.
\newblock Periods and harmonic analysis on spherical varieties.
\newblock {\em Ast\'{e}risque}, (396):viii+360, 2017.

\bibitem[SVo]{MR1670073}
Susana~A. Salamanca-Riba and David~A. Vogan, Jr.
\newblock On the classification of unitary representations of reductive {L}ie
  groups.
\newblock {\em Ann. of Math. (2)}, 148(3):1067--1133, 1998.

\bibitem[Sar]{MR780071}
Peter Sarnak.
\newblock Fourth moments of {G}r\"ossencharakteren zeta functions.
\newblock {\em Comm. Pure Appl. Math.}, 38(2):167--178, 1985.

\bibitem[SST]{MR3675175}
Peter Sarnak, Sug~Woo Shin, and Nicolas Templier.
\newblock Families of {$L$}-functions and their symmetry.
\newblock In {\em Families of automorphic forms and the trace formula}, Simons
  Symp., pages 531--578. Springer, [Cham], 2016.

\bibitem[Sau]{MR1468833}
Fran\c{c}ois Sauvageot.
\newblock Principe de densit\'e pour les groupes r\'eductifs.
\newblock {\em Compositio Math.}, 108(2):151--184, 1997.

\bibitem[Ser]{MR0344216}
J.-P. Serre.
\newblock {\em A course in arithmetic}.
\newblock Springer-Verlag, New York, 1973.
\newblock Translated from the French, Graduate Texts in Mathematics, No. 7.

\bibitem[ShTs]{ShalikaTschinkel}
Joseph Shalika and Yuri Tschinkel.
\newblock Height zeta functions of equivariant compactifications of unipotent
  groups.
\newblock {\em Comm. Pure Appl. Math.}, 69(4):693--733, 2016.

\bibitem[ShTe]{MR3437869}
Sug~Woo Shin and Nicolas Templier.
\newblock Sato-{T}ate theorem for families and low-lying zeros of automorphic
  {$L$}-functions.
\newblock {\em Invent. Math.}, 203(1):1--177, 2016.
\newblock Appendix A by Robert Kottwitz, and Appendix B by Raf Cluckers, Julia
  Gordon and Immanuel Halupczok.

\bibitem[SoTh]{2018arXiv180403654S}
K.~{Soundararajan} and J.~{Thorner}.
\newblock {Weak subconvexity without a Ramanujan hypothesis}.
\newblock {\em ArXiv e-prints}, April 2018.

\bibitem[Sou]{soundararajan-2008}
Kannan Soundararajan.
\newblock Weak subconvexity for central values of {$L$}-functions.
\newblock {\em Ann. of Math. (2)}, 172(2):1469--1498, 2010.

\bibitem[SZ]{MR2874638}
Binyong Sun and Chen-Bo Zhu.
\newblock Multiplicity one theorems: the {A}rchimedean case.
\newblock {\em Ann. of Math. (2)}, 175(1):23--44, 2012.

\bibitem[T]{MR764508}
Michael~E. Taylor.
\newblock Noncommutative microlocal analysis. {I}.
\newblock {\em Mem. Amer. Math. Soc.}, 52(313):iv+182, 1984.

\bibitem[Ve]{venkatesh-2005}
Akshay Venkatesh.
\newblock Sparse equidistribution problems, period bounds and subconvexity.
\newblock {\em Ann. of Math. (2)}, 172(2):989--1094, 2010.

\bibitem[Vo1]{GKdim}
David~A. Vogan, Jr.
\newblock Gel\cprime fand-{K}irillov dimension for {H}arish-{C}handra modules.
\newblock {\em Invent. Math.}, 48(1):75--98, 1978.

\bibitem[Vo2]{MR519352}
David~A. Vogan, Jr.
\newblock The algebraic structure of the representation of semisimple {L}ie
  groups. {I}.
\newblock {\em Ann. of Math. (2)}, 109(1):1--60, 1979.

\bibitem[Vo3]{MR1737729}
David~A. Vogan, Jr.
\newblock The method of coadjoint orbits for real reductive groups.
\newblock In {\em Representation theory of {L}ie groups ({P}ark {C}ity, {UT},
  1998)}, volume~8 of {\em IAS/Park City Math. Ser.}, pages 179--238. Amer.
  Math. Soc., Providence, RI, 2000.

\bibitem[Wal1]{MR1989693}
J.-L. Waldspurger.
\newblock La formule de {P}lancherel pour les groupes {$p$}-adiques (d'apr\`es
  {H}arish-{C}handra).
\newblock {\em J. Inst. Math. Jussieu}, 2(2):235--333, 2003.

\bibitem[Wal2]{Wald4}
Jean-Loup Waldspurger.
\newblock Une formule int\'egrale reli\'ee \`a la conjecture locale de
  {G}ross-{P}rasad, 2e partie: extension aux repr\'esentations temp\'er\'ees.
\newblock {\em Ast\'erisque}, (346):171--312, 2012.
\newblock Sur les conjectures de Gross et Prasad. I.

\bibitem[Wal3]{MR929683}
Nolan~R. Wallach.
\newblock {\em Real reductive groups. {I}}, volume 132 of {\em Pure and Applied
  Mathematics}.
\newblock Academic Press, Inc., Boston, MA, 1988.

\bibitem[Wat]{watson-2008}
Thomas~Crawford Watson.
\newblock {\em Rankin triple products and quantum chaos}.
\newblock ProQuest LLC, Ann Arbor, MI, 2002.
\newblock Thesis (Ph.D.)--Princeton University, \url{arXiv.org:0810.0425}.

\bibitem[We]{MR670072}
Andr{\'e} Weil.
\newblock {\em Adeles and algebraic groups}, volume~23 of {\em Progress in
  Mathematics}.
\newblock Birkh\"auser Boston, Mass., 1982.
\newblock With appendices by M. Demazure and Takashi Ono.

\bibitem[Wh]{MR0095844}
Hassler Whitney.
\newblock Elementary structure of real algebraic varieties.
\newblock {\em Ann. of Math. (2)}, 66:545--556, 1957.

\bibitem[Wo]{MR2982336}
Michael~C. Woodbury.
\newblock {\em Explicit trilinear forms and subconvexity of the triple product
  {L}-function}.
\newblock ProQuest LLC, Ann Arbor, MI, 2011.
\newblock Thesis (Ph.D.)--The University of Wisconsin - Madison.

\bibitem[Z1]{2017arXiv171208844Z}
W.~{Zhang}.
\newblock {Periods, cycles, and $L$-functions: a relative trace formula
  approach}.
\newblock {\em ArXiv e-prints}, December 2017.

\bibitem[Z2]{MR3245011}
Wei Zhang.
\newblock Fourier transform and the global {G}an-{G}ross-{P}rasad conjecture
  for unitary groups.
\newblock {\em Ann. of Math. (2)}, 180(3):971--1049, 2014.

\end{thebibliography}
% \bibliographystyle{plain}

\def\cprime{$'$} \def\cprime{$'$} \def\cprime{$'$} \def\cprime{$'$}

\end{document}